\documentclass[a4paper,11pt]{article}

\usepackage{amsmath,amssymb,graphics,epsfig,color,cite}
\usepackage{mathrsfs,bbm}
\usepackage{amsthm}
\usepackage{tikz}
\usepackage{hyperref}
 \usepackage{geometry}
\definecolor{light}{gray}{.9}

\DeclareMathAlphabet{\mathpzc}{OT1}{pzc}{m}{it}

\newcounter{savecntr}% Save footnote counter
\newcommand{\Ga}[1]{\small $\Gamma_{#1}$}
\newcommand{\field}[1]{\mathbb{#1}} 
 \newcommand{\nz}{\field{N}}

\newcommand{\mb}[1]{\mathbb{#1}}

\newcommand{\mc}[1]{\mathcal{#1}}
\newcommand{\R}{\mb{R}}

\newcommand{\mbf}[1]{\mathbf{#1}}

\DeclareMathOperator{\range}{Ran}

\newcommand{\lp}{\langle}
\newcommand{\rp}{\rangle}

\newcommand{\ve}{\varepsilon}

\DeclareMathOperator{\hess}{Hess}

\def\Hess{{\rm Hess}}% Hessian
\def\det{{\rm det}}% determinant

\def\supp{\mathop{\rm supp} \nolimits} % Support
\def\Ran{{\rm Ran}}

\def\E{{\mathbb E}}
\def\P{{\mathbb P}}

\def \Sp{{\rm  Sp\;}}

\def\and {{\rm \; and \;}}

\def\Ker {{\rm \; Ker  \;}}
\def\Im {{\rm  Im\;}}

\def\dim {{\rm \; dim  \;}}

\newcommand {\pa}{\partial}

\newtheorem{theorem}{Theorem}
\newtheorem{proposition}{Proposition}
\newtheorem{definition}[proposition]{Definition}
\newtheorem{lemma}[proposition]{Lemma}
\newtheorem{corollary}[proposition]{Corollary}
\newtheorem{remark}[proposition]{Remark}

\newcommand{\comment}[1]{ }
\usepackage{multicol}

\title{Sharp asymptotics of the first exit point density}
\author{Giacomo Di Ges\`u\setcounter{savecntr}{\value{footnote}}\thanks{Current affiliation: Institut f\"{u}r Analysis und Scientific Computing, E101-TU Wien, Wiedner Hauptstr. 8, 1040 Wien, Austria. E-mail: \{giacomo.di.gesu,boris.nectoux\}@asc.tuwien.ac.at}$\ $\footnotemark[2]\, , Tony Leli\`evre\setcounter{savecntr}{\value{footnote}}\thanks{CERMICS, \'Ecole des Ponts, Universit\'e  Paris-Est, INRIA, 77455 Champs-sur-Marne, France. E-mail:  tony.lelievre@enpc.fr}\,, Dorian Le Peutrec\setcounter{savecntr}{\value{footnote}}\thanks{Laboratoire de Math\'ematiques d'Orsay, Univ. Paris-Sud, CNRS, Universit\'e Paris-Saclay, 91405 Orsay, France. E-mail: dorian.lepeutrec@math.u-psud.fr}$\, \ $and Boris Nectoux\footnotemark[1]\, \footnotemark[2]}

\begin{document}

 \maketitle

\begin{abstract}
We consider the exit event from a metastable state for the overdamped
Langevin dynamics $dX_t = -\nabla f(X_t) dt + \sqrt{h} dB_t$. Using
tools from semiclassical analysis, we prove that, starting from the
quasi stationary distribution within the state, the exit event can be
modeled using a jump Markov process parametrized with the Eyring-Kramers
formula, in the small temperature regime $h \to 0$. We provide in
particular sharp asymptotic estimates on the exit distribution which
demonstrate the importance of the prefactors in the Eyring-Kramers
formula. Numerical experiments indicate that the geometric assumptions
we need to perform our analysis are likely to be necessary. These
results also hold starting from deterministic initial conditions within
the well which are sufficiently low in energy. From a modelling
viewpoint, this gives a rigorous justification of the transition state
theory and the Eyring-Kramers formula, which are used to relate the
overdamped Langevin dynamics (a continuous state space Markov dynamics)
to kinetic Monte Carlo or Markov state models (discrete state space
Markov dynamics). From a theoretical viewpoint, our analysis paves a new
route to study the exit event from a metastable state for a stochastic
process.
\end{abstract}

%%%FOOTNOTES%%%

%\footnotetext[1]{Institut f\"{u}r Analysis und Scientific Computing, TU Wien, Wiedner Hauptstr. 8, 1040 Wien, Austria. E-mail: giacomo.di.gesu@tuwien.ac.at}
%\footnotetext[2]{CERMICS, \'Ecole des Ponts, Universit\'e
%    Paris-Est, INRIA, 77455 Champs-sur-Marne, France. E-mail: \{tony.lelievre,boris.nectoux\}@enpc.fr}
%\footnotetext[3]{Laboratoire de Math\'ematiques d'Orsay,
  %  Univ. Paris-Sud, CNRS, Universit\'e Paris-Saclay, 91405 Orsay, France. E-mail: dorian.lepeutrec@math.u-psud.fr}

\tableofcontents

\section{Motivation and presentation of the results}\label{intro}

In materials science, biology and chemistry, atomistic models are now
used on a daily basis in order to predict the macroscopic properties  of matter 
from a microscopic description. The basic ingredient is a
potential energy function $f: \R^d \to \R$ which associates to a set
of coordinates of particles the energy of the system. In practice, $d$
is very large, since the system contains many particles (from tens
of thousands to millions).

Using this function $f$, two types of models are built: continuous
state space Markov models (stochastic differential equations), such as
the Langevin or overdamped Langevin dynamics,
and discrete state space Markov models (jump Markov processes). The
objective of the analysis presented in this work is to make a
rigorous link between these two types of approaches, and in particular
to provide a justification of the use of Eyring-Kramers laws to
parameterize jump Markov models, by studying the exit event from a metastable state for the overdamped Langevin dynamics. 

Jump Markov models are used by practitioners for many reasons. From a
modelling viewpoint, new insights can be gained by building such
coarse-grained models, that are easier to handle than a
large-dimensional stochastic differential equation. From a numerical
viewpoint, it is possible to simulate a jump Markov model over
timescales which are much larger than the original Langevin
dynamics. Moreover, there are many algorithms which use the underlying
jump Markov model in order to accelerate the sampling of the original
dynamics~\cite{sorensen-voter-00,voter-97,voter-98}.

This section is organized as follows.  First, the two models under
consideration are introduced, namely the overdamped Langevin dynamics in  Section~\ref{lang}, and the underlying jump Markov process  in Section~\ref{jump}. Next, Section~\ref{biblio} is devoted to a review of the mathematical literature dealing with metastable processes and the exit event from a metastable state. In Section~\ref{sec.qsd}, the notion of quasi stationary distribution is reviewed. This is a crucial tool in our analysis, in order to connect the overdamped Langevin dynamics to a jump Markov process. Then, in Section~\ref{statementthm}, our main result (Theorem~\ref{TBIG0})  is stated. 
In Section~\ref{sec:dis_gene}, we generalize Theorem~\ref{TBIG0} in various directions and discuss the geometric assumptions used to state Theorem~\ref{TBIG0}. Finally, in Section~\ref{sec:sketch}, we give an outline of the proof of Theorem~\ref{TBIG0}, together with the general organization of the paper.

\subsection{Overdamped Langevin dynamics and metastability}\label{lang}

The continuous state space Markov model we consider in this work is
the so-called overdamped Langevin dynamics in $\mathbb R^d$
\begin{equation}\label{eq:langevin}
d X_t = -\nabla f(X_t) d t + \sqrt{h} \, d B_t,
\end{equation}
driven by the potential function $f: \mathbb R^d \to \mathbb R$. We
assume in the following that $f$ is a $C^{\infty}$ Morse function (all
the critical points are non degenerate). The
parameter $h=2k_B T>0$ is proportional to the temperature $T$ and $(B_t)_{t\geq 0}$ is a standard $d$-dimensional Brownian motion. One
henceforth assumes that
$$\exists h_0,\, \forall h < h_0,\,\int_{\mathbb R^d} e^{-\frac{2}{h}f(x)} dx<\infty.$$
The invariant probability measure of (\ref{eq:langevin}) is
\begin{equation}\label{eq:mu}
\frac{ e^{-\frac{2}{h}  f(x)} \, dx}{ \displaystyle\int_{\mathbb R^d}
  e^{-\frac{2}{h}  f(y)} dy  }.
\end{equation}

The basic observation which motivates the use of a jump Markov model
to obtain a reduced description of the dynamics~\eqref{eq:langevin} is
the following. In many practical cases of interest in biology, physics
or chemistry, the
dynamics~\eqref{eq:langevin} is metastable, meaning that the process
$(X_t)_{t \ge 0}$ remains trapped for very long times in some regions
(called metastable states). 
If a state is metastable, the way the process leaves this state should not depend too much on the way it enters the state. 
It is thus tempting to introduce an
underlying jump process among these metastable states.

% The overdamped Langevin dynamics (\ref{eq:langevin}) exhibits two terms : the term $-\nabla f(X_t) $ sends the process towards local minima of $f$, while the noise term $\sqrt{h} \, d B_t$ drives $X_t$ from one basin of attraction of the dynamics $\dot{x}=-\nabla f(x)$ to another one. If the temperature is small ($h <<1$), the process $X_t$ remains during a very long period of time trapped around a local minimum of $f$, called a metastable state, before going to another region. Consequently the small temperature regime combined with the shape of the function $f$ can be responsible for long periods of inactivity where no transition occurs between the local minima of $f$. There is a clear separation of time scales between vibrations and transitions. \\

Let us consider a region $\Omega \subset \R^d$ and the associated exit
event from $\Omega$. More precisely, let us introduce
\begin{equation}\label{eq:tau}
\tau_{\Omega}=\inf \{ t\geq 0 | X_t \notin \Omega      \}
\end{equation}
\label{page.tau}
the first exit time from $\Omega$. The exit event from $\Omega$ is
fully characterized by the couple of random variables $\left
  (\tau_{\Omega}, X_{\tau_{\Omega}}\right)$. 
The focus of this work is the justification of a jump Markov process to
model the exit event from the region $\Omega$, in the small temperature
regime $h \to 0$.

% Let us first define the first exit time from a domain. 

% \begin{definition} \label{tau}
% Let $\Omega\subset \mathbb R ^d$ and consider the dynamics (\ref{eq:langevin}).  The first exit time of $\Omega$ is defined by 
% $$\tau_{\Omega}=\inf \{ t\geq 0 | X_t \notin \Omega      \}.$$
% \end{definition}
% \noindent

% The exit event from a domain $\Omega\subset \mathbb R^d$ is characterized by the random variables $\left(\tau_{\Omega}, X_{\tau_{\Omega}}\right)$. \\

% \textbf{Outline of this work.}\\\\
% \noindent
% Our aim in this work will be to study the precise asymptotic of the exit point distribution $X_{\tau_{\Omega}}$ in the small temperature regime ($h\to 0$), see Theorem~\ref{TBIG0}, Corollary~\ref{co.proba-sigmai} and Theorem~\ref{th.gene_sigma}. The particular interest of our result is that they explicit a prefactor.\\

% \textbf{Plan of the introduction.}\\\\
% \noindent

\subsection{From the potential function to a jump Markov
  process} \label{jump}

The potential function $f$ can also be used to build a jump Markov
process to describe the evolution of the system. Jump Markov models
are continuous-time Markov processes with values in a discrete state
space. In molecular dynamics such processes are known as kinetic Monte
Carlo models~\cite{voter-05} or Markov state models~\cite{schuette-98,schuette-sarich-13,bowman-pande-noe-14}.

 % The use of jump Markov models to go from  the microscopic dynamics  (\ref{eq:langevin}) to a jump Markov process is based on the following observation. It is observed that, for applications in biology, material sciences or chemistry, the microscopic dynamics (\ref{eq:langevin}) is metastable: transitions occur over very long period of time compared to the typical vibration times. The stochastic process  (\ref{eq:langevin}) remains trapped for a long time in some region of the configurational space, as explained in Section~\ref{lang}. The hope is that the stochastic process  (\ref{eq:langevin}) looses its memory before exiting this metastable region. This last point justifies why one can model the exit from this region by a jump Markov Process. We refer to \cite{schuette-sarich-13} and \cite{schuette-sarich-13} on how  modelling with a Markov state models in molecular dynamics. 

\paragraph{Kinetic Monte Carlo models.}
% Kinetic Monte Carlo methods aim at simulating a time continuous discrete state space Markov process. In particular, in molecular dynamics, Kinetic Monte Carlo methods are used in order to simulate a molecular system over very long times by considering only the jump process among the discrete states of the configurational space. \\\\

The basic requirement to build a kinetic Monte Carlo model is a
discrete collection of states $D\subset \mathbb N$, with associated rates $k_{i,j}\geq 0$ for
transitions from state $i$ to state $j$,  where $(i,j)\in
D\times D$ and $i \neq j$. The neighboring states of state~$i$ are
those states $j$ such that $k_{i,j}>0$.  The dynamics is then given by a jump
Markov process $(Z_t)_{t \ge 0}$ with infinitesimal generator $L \in
\R^{D \times D}$, where $L_{i,j}=k_{i,j}$ for $i \neq j$.

To be more precise, let us describe how to build the jump
process $(Z_t)_{t \ge 0}$ by defining the residence times $(T_n)_{n
  \ge 0}$ and the subordinated Markov chain $(Y_n)_{n \ge 0}$. Starting at time $0$ from a state $Y_0 \in D$, the model
consists in iterating the following two steps over $ n \ge 0$: given
$Y_n$, 
\begin{itemize}
\item first sample the residence time $T_n$ in $Y_n$ as an
exponential random variable with parameter $\sum_{j \neq Y_n}
k_{Y_n,j}$: $\forall i \in D$, $\forall t>0$,
\begin{equation}\label{eq:exit_time_kMC}
\P(T_n>t | Y_n=i) = \exp\Bigg(-\sum_{j \neq i}
k_{i,j} \, t \Bigg),
\end{equation}
\item and then sample {\em independently from $T_n$} the next visited state
$Y_{n+1}$ using the following law:
\begin{equation}\label{eq:exit_point_kMC}
\forall j \neq i, \, \P(Y_{n+1}=j | Y_n=i)=\frac{k_{i,j}}{\sum_{j \neq
    i} k_{i,j}}.
\end{equation}
\end{itemize}
The continuous time Markov process $(Z_t)_{t \ge 0}$ is then defined as:
$$\forall n \ge 0, \, \forall t \in \left[\sum_{m=0}^{n-1}
  T_m,\sum_{m=0}^{n} T_m \right), \, Z_t=Y_n$$
with the convention $\sum_{m=0}^{-1} T_m = 0$.

\paragraph{Transition rates and Eyring Kramers law.} 

Starting from the potential function $f:\R^d \to \R$, one approach to
build a kinetic Monte Carlo model is to consider a collection of
disjoint subsets $(\Omega_k)_{k \in D}$ which form a partition of the
space $\R^d$ and to set the transition
rates $k_{i,j}$ by considering transitions between these subsets, see
for example~\cite{voter-05,wales-03,cameron-14b,fan-yip-yildiz-14}.

The concept of jump rate between two states is one of the fundamental
notions in the modelling of materials. Many papers have been devoted
to the rigorous evaluation of jump rates from a full-atom
description. The most famous formula is probably the rate derived in
the harmonic transition state theory~\cite{marcelin-15,vineyard-57}, which gives
an explicit expression for the rate in terms of the underlying
potential energy function: this is the Eyring-Kramers formula, that we
now introduce.

Let us consider a subset $\Omega$ of $\R^d$, which should be thought
as one of the subsets $(\Omega_k)_{k \in D}$ introduced above, say the
state $k=0$. If
$\Omega$ is metastable (in a sense which will be made precise below),
it seems sensible to model the exit event from $\Omega$ using a jump
Markov model, as introduced in the previous paragraph. As explained above, this requires to define jump rates
$(k_{0,j})$ from the state $0$ to the neighboring states $j$. The aim
of this paper is to prove that  the  exit from $\Omega$ for 
the dynamics~\eqref{eq:langevin} can be approximated using a kinetic Monte-Carlo model with transition rates computed with  the
Eyring-Kramers formula:
\begin{equation}\label{eq:kramers}
k_{0,j}=A_{0,j} \, e^{-\frac{2}{h} \left( f(z_j) - f(x_0) \right) }
\end{equation}
where $A_{0,j} >0 $ is a prefactor, $x_0=\arg\min_{x \in \Omega} f(x)$
is the global minimum of $f$ on $\Omega$ which is assumed to be unique
and $z_j=\arg\min_{z \in \partial \Omega_j} f(z)$ where $\partial
\Omega_j$ denotes the part of the boundary $\partial \Omega$ which
connects the region $\Omega$ (numbered 0) with the neighboring region
numbered $j$ (see Figure \ref{fig:partition}  for a schematic representation when $\Omega$ has $4$ neighboring states). 

%\begin{figure}[h!]
%\begin{center}
%\begin{tikzpicture}
%\tikzstyle{vertex}=[draw,circle,fill=black,minimum size=6pt,inner sep=0pt]
%%\draw[very thick] (1.5,0) arc (0:360:1.5);
%%\tikz\foreach \i in {0,\dots,0}\foreach \j in {0,\dots,0}
%%  \draw [black, dashed,  thick] 
%%    plot [smooth cycle, tension=1, domain=0:320, samples=30] (\x:{2.2+rand/3});
%  %  \draw (-3,2) node[]{$k_{0,1}$};
%\tikzstyle{ball}=[circle, dashed, minimum size=1cm, draw]
%\tikzstyle{point}=[circle, fill, minimum size=.01cm, draw]
%
%
%
%\draw [rounded corners=10pt] (2,0.5) -- (.25,1.5) -- (0.5,2.5) -- (0.6,3.5) -- (2.75,3.75) -- (3.5,3) -- (3,2) -- (3.44,1.5) --cycle;
%\draw [rounded corners=10pt]    (2.75,3.75)  -- (3.5,3) -- (3,2) --(5.5,2.5)--(6.5,4) --(6.5,5)  -- (5,6) -- (3,5)  --cycle;
% \draw (2,1.5) node[]{$\Omega_0$};
%  \draw (5,4) node[]{$\Omega_j$};
%\draw (3.3 ,3.2) node[vertex,label=north east: {$z_j$}](v){};
%\draw (2 ,2.5) node[vertex,label=north west: {$x_0$}](v){};
%\draw (4.7 ,1.5) node[]{$\partial\Omega_j$}; 
%\draw[->, very thick] (4.4,1.7) -- (3.35 ,2.6);
%\end{tikzpicture}
%\caption{The domain $\Omega_0$ and its neighbor $\Omega_j$.}
% \label{fig:kramer}
% \end{center}
%\end{figure}
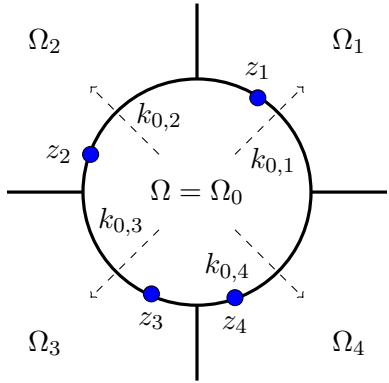
\begin{figure}[h!]
\begin{center}
\begin{tikzpicture}
\tikzstyle{vertex}=[draw,circle,fill=blue,minimum size=6pt,inner sep=0pt]
\draw (0,0) node[]{$\Omega=\Omega_{0}$};
\draw (2,2) node[]{$\Omega_{1}$};
\draw (-2,2) node[]{$\Omega_{2}$};
\draw (-2,-2) node[]{$\Omega_{3}$};
\draw (2,-2) node[]{$\Omega_{4}$};
\draw[very thick,-](0,1.5)--(0,2.5);
\draw[very thick,-](1.5,0)--(2.5,0);
\draw[very thick,-](-1.5,0)--(-2.5,0);
\draw[very thick,-](0,-1.5)--(0,-2.5);
\draw[dashed,->](0.5,0.5)--(1.4,1.4) ;
\draw[dashed,->](-0.5,-0.5)--(-1.4,-1.4) ;
\draw[dashed,->](-0.5,0.5)--(-1.4,1.4) ;
\draw[dashed,->](0.5,-0.5)--(1.4,-1.4) ;
\draw (1,0.4) node[]{$k_{0,1}$};
\draw (-1,-0.4) node[]{$k_{0,3}$};
\draw (-0.5,1) node[]{$k_{0,2}$};
\draw (0.4,-1) node[]{$k_{0,4}$};
\draw[very thick] (1.5,0) arc (0:360:1.5);

\draw (0.8 ,1.25) node[vertex,label=north: {$z_1$}](v){};
\draw (0.5 ,-1.4) node[vertex,label=south: {$z_4$}](v){};
\draw (-1.4,0.5) node[vertex,label=west: {$z_2$}](v){};
\draw (-0.6 ,-1.35) node[vertex,label=south: {$z_3$}](v){};

\end{tikzpicture}
\caption{The domain $\Omega=\Omega_0$ is associated with state $0$.
  The neighboring domains $(\Omega_i)_{i=1,\ldots,4}$ correspond to the
  four possible next visited states, with associated rates $(k_{0,i})_{i=1,\ldots,4}$. For $i=1,\ldots,4$, the point $z_i$ represent the lowest energy  point on $\partial \Omega\cap \partial \Omega_i$.
  }
 \label{fig:partition}
 \end{center}
\end{figure}
\noindent
The prefactors $A_{0,j}$ depend on the dynamics under consideration
and on the potential function $f$ around $x_0$ and $z_j$. If $\Omega$
is the basin of attraction of $x_0$ for the gradient dynamics
$\dot{x}=-\nabla f(x)$ so that the points $z_j$ are order one saddle
points, the prefactor writes for the overdamped Langevin
dynamics~\eqref{eq:langevin}
\begin{equation}\label{eq:kramers_pref}
A_{0,j}=\frac{1}{2\pi} | \lambda^-(z_j)| \sqrt{\frac{ {\rm det} ( {\rm Hess} f)(x_0)}{ | {\rm det} ( {\rm Hess} f)(z_j) | } } ,\end{equation}
 where $\lambda^-(z_j)$ is the negative eigenvalue of the Hessian of
 $f$ at $z_j$. The
 formulas~\eqref{eq:kramers}--\eqref{eq:kramers_pref} have been obtained by
 Kramers~\cite{kramers-40}, but also by many authors previously, see the
 exhaustive review of the literature reported in~\cite{hanggi-talkner-barkovec-90}. We
 also refer to~\cite{hanggi-talkner-barkovec-90} for generalizations to the Langevin dynamics.

\subsection{Review of the mathematical literature on the Eyring-Kramers formula} \label{biblio}

Let us give the main mathematical approaches to rigorously derive the Eyring-Kramers formula or to  study of the
exit event from a domain for a stochastic process. See also~\cite{berglund-13} for a nice review.

\paragraph{Global approaches. }

 Some authors adopt a global approach: they look at the
spectrum of the infinitesimal generator of the continuous space
dynamics in the small temperature regime $h \to 0$.  It can be shown that there are exactly $m$ small
eigenvalues, $m$ being the number of local minima $\{x_1,\ldots,x_m\}$ of $f$, and that each of these
eigenvalues writes asymptotically when $h\to 0$, $A_i\, e^{-\frac 2h (f(z_i)-f(x_i))}$, for $i\in \{1,\ldots,m\}$. One thus recovers  the form of the Eyring-Kramers formula~\eqref{eq:kramers}. 
Here, the
saddle point $z_i$ attached to the local minimum $x_i$ is defined by (it is here
  implicitly assumed that the inf sup value is attained at a single
  saddle point~$z_i$)
$$f(z_i)=\inf_{\gamma \in {\mathcal P}(x_i,B_i)} \sup_{t \in [0,1]}
f(\gamma(t))$$
where ${\mathcal P}(x_i,B_i)$ denotes the set of continuous paths from
$[0,1]$ to $\R^d$ such that $\gamma(0)=x_i$ and $\gamma(1) \in B_i$ with
$B_i$ the union of small balls around local minima lower in energy
than~$x_i$. For the dynamics~\eqref{eq:langevin}, we refer for example to the
work~\cite{helffer-klein-nier-04} based on semi-classical analysis
results for Witten Laplacian and the
articles~\cite{bovier-eckhoff-gayrard-klein-04,bovier-gayrard-klein-05,eckhoff-05} 
where a potential theoretic approach is adopted. In the latter
results, a connection is made between the small eigenvalues and mean
transition times between metastable states. Let us also mention the
earlier results~\cite{miclo-95,holley-kusuoka-stroock-89,davies-82a,davies-82b}.
%  Let us also
% mention~\cite{menz-schlichting-14} where the Eyring-Kramers
% formula is proved for the first eigenvalue of the infinitesimal generator, as well
% as for the logarithmic Sobolev inequality constant.
% For the dynamics~\eqref{eq:L}, similar results are obtained
% in~\cite{herau-hitrick-sjostrand-11}.
These spectral approaches give the
cascade of relevant time scales to reach from a local minimum an
other local minimum which is lower in energy. They do not give any information about the typical time scale to go from one local minimum to any other local minimum (say from the global minimum to the second lower minimum for example). 
These global approaches can be used to build jump Markov models using a
Galerkin projection of the infinitesimal generator of $(X_t)_{t \ge 0}$ onto the first $m$
eigenmodes, which gives an excellent approximation of the
infinitesimal generator. This has been extensively investigated by
Sch\"utte and his collaborators~\cite{schuette-sarich-13},
starting with the seminal work~\cite{schuette-98}. % In fact, Sch\"utte et al. look at the eigenvalues  close to 1 for the so-called transfer operator $P_t = {\rm e}^{tL^{(0)}_{f,h}}$  (for a well  chosen lag time $t > 0$), which is equivalent to looking at the small positive eigenvalues of $-L^{(0)}_{f,h}$.

% One aims now at giving a review as large as possible of what has already been done in Mathematics in order to characterize
%  Several works using different theories have been used to study such quantities, especially in order to give their asymptotic in the small temperature regime: $h\to 0$.  

\paragraph{Local approaches. }

In this work, we are interested in a local approach, namely the study
of the exit event (exit time and exit point) from a fixed given metastable state $\Omega$. The most
famous approach to study the exit event in the small temperature limit is the large deviation
theory~\cite{freidlin-wentzell-84}. It relies essentially on the study of slices of
 the process  defined with a suitable sequence of increasing stopping
times. In the theory of large deviations, the notion of rate functional
is fundamental and gives the cost of deviating from a deterministic
trajectory.

 In the small
temperature regime, large deviation results
provide the exponential rates~\eqref{eq:kramers}, but without the
prefactors and without precise error bounds. It can also be proven that the exit time is exponentially
distributed in this regime, see~\cite{day-83}. For the
dynamics~\eqref{eq:langevin}, a typical result on the exit point distribution
is the following (see~\cite[Theorem 5.1 in Section 5]{freidlin-wentzell-84}): for
all $\Omega'$ compactly embedded in $\Omega$, for any $\gamma >0$, for any $\delta > 0$, there
exists $\delta_0 \in (0,\delta]$ and $h_0 > 0$ such that for
all $h < h_0$, for all $x \in \Omega'$ such that $f(x) <
\min_{\partial \Omega }f$, and for all $y
\in \partial \Omega$,
\begin{equation}\label{eq:LD}
{\rm e}^{-\frac 2h  ( f(y)-\min_{\partial \Omega }f + \gamma)}\le \P_{x} (X_{\tau_\Omega} \in
{\mathcal V}_{\delta_0}(y))  \le
{\rm e}^{-\frac 2h (f(y)-\min_{\partial \Omega }f- \gamma)}
\end{equation}\label{page.px}
where ${\mathcal V}_{\delta_0}(y)$ is a $\delta_0$-neighborhood of $y$ in
$\partial \Omega$. Here and in the following, the subscript~$x$ indicates that the
stochastic process starts from~$x \in \R^d$: $X_0=x$. 

The strength of large deviation theory is that it is
very general: it applies to any dynamics (reversible or non
reversible) and in a very general geometric setting, even though it
may be difficult in such general cases to make explicit the rate
functional, and thus to determine the exit rates. See for
example~\cite{bouchet-reygner-16,landim2017dirichlet} for   recent contributions to the non
reversible case.

Many authors have developed partial differential approach  to the same problem. We refer
to~\cite{day-99} for a comprehensive review. In particular, formal
approaches to study the exit time and the exit point distribution have been developed by
Matkowsky, Schuss and collaborators in~\cite{matkowsky-schuss-77,matkowsky-schuss-79,naeh-klosek-matkowsky-schuss-90,schuss-09}
and by Maier and Stein in~\cite{maier-stein-93,maier-stein-97}, using formal
expansions for singularly perturbed elliptic equations. Some of the
results cited above actually consider more general dynamics
than~\eqref{eq:langevin}. Rigorous version of these derivations have
been obtained
in~\cite{devinatz-friedman-78a,devinatz-friedman-78b,kamin-78,holley-kusuoka-stroock-89, perthame-90,eizenberg-90, mathieu-95,mathieu-94}. 
%More recently, accurate asymptotics on the low lying spectra of the infinitesimal generator of the dynamics \eqref{eq:langevin} have been obtained in \cite{ helffer-nier-06,le-peutrec-10   } with techniques from semi-classical analysis. 

\paragraph{Rescaling in time and convergence to a jump process. }
Finally, some authors prove the convergence to a jump Markov process using a
rescaling in time. See for example~\cite{kipnis-newman-85} for
a one-dimensional diffusion in a double well, and~\cite{galves-olivieri-vares-87,mathieu-95} for a
similar problem in larger dimension. In \cite{sugiura-95}, a rescaled
in time diffusion process converges to a jump Markov process living on
the global minima of the potential $f$, assuming they are separated
by saddle points having the same heights. We also refer to the recent review paper~\cite{landim2018metastable} for related results.

\medskip

\noindent
In this work, we are interested in precise asymptotics of the
distribution of $X_{\tau_\Omega}$. Our approach is local, justifies the
Eyring-Kramers formula~\eqref{eq:kramers} with the prefactors and provides sharp error
estimates (see~\eqref{resultatkil}). It uses techniques developed in particular in the
previous works \cite{helffer-klein-nier-04,helffer-nier-06,le-peutrec-10,lelievre-nier-15}. Our analysis requires to combine various
tools from semiclassical analysis to address new questions: sharp
estimates on quasimodes far from the critical points for Witten
Laplacians on manifolds with boundary, a precise analysis of the normal
derivative on the boundary of the first eigenfunction of Witten
Laplacians, and fine properties of the Agmon distance on manifolds with
boundary.\\

%In this work, we are interested in precise asymptotics of the
%distribution of $X_{\tau_{\Omega}}$. Our approach is local, justifies the Eyring-Kramers formula \eqref{eq:kramers} with the prefactors and provides error estimates (see~\eqref{resultatkil}). It uses techniques developed in particular in the previous works \cite{helffer-klein-nier-04,helffer-nier-06,le-peutrec-10,lelievre-nier-15}. 
\noindent 
Let us finally mention that a summary of the results of this work appeared in~\cite{di-gesu-lelievre-le-peutrec-nectoux-17,IHPLLN}.

\subsection{Quasi stationary distribution} \label{sec.qsd}

The quasi stationary distribution  is the cornerstone of our analysis.
The quasi stationary distribution can be seen as a local equilibrium
for a metastable stochastic process when it is trapped in a metastable
region. It is actually useful in order to make precise quantitatively
what a metastable domain is. In all what follows, we focus on the
overdamped Langevin dynamics~\eqref{eq:langevin} and a domain $\Omega
\subset \R^d$. For generalizations
to other processes, we refer to~\cite{collet-martinez-san-martin-13,champagnat2017general,cattiaux2009quasi} and in particular
to~\cite{nier-14} for the spectral analysis of the  Langevin dynamics on a bounded domain.

 \subsubsection{Definition and a first property of quasi stationary distributions}\label{sec:QSD_1st_prop}
Let us first define the quasi stationary distribution.
\begin{definition} \label{defQSD}
Let $\Omega\subset \mathbb R ^d$ and consider the dynamics (\ref{eq:langevin}).  A quasi stationary distribution is a probability measure $\nu_h$ supported in $\Omega$ such that
$$\forall t \ge 0, \, \nu_h(A)=\frac{\displaystyle\int_{\Omega} \P_{x} \left[X_t \in A,   t<\tau_{\Omega}\right]  \nu_h(dx) }{\displaystyle \int_{\Omega}   \P_{x} \left[t<\tau_{\Omega} \right] \nu_h(dx)}.$$
\end{definition}
\noindent 
In  \label{page.nuh} words, if
$X_0$ is distributed according to  $\nu_h$, then $\forall t>0$,  $X_t$
is still distributed according to  $\nu_h$ conditionally on 
$X_s \in \Omega$ for all $s \in (0,t)$.
It is important to notice that $\nu_h$ is not the invariant measure \eqref{eq:mu} of
the original process restricted to $\Omega$, i.e.  $\nu_h\neq \mbf 1_\Omega(x)\, \frac{ e^{-\frac{2}{h}  f(x)} \, dx}{ \int_{\mathbb R^d}
  e^{-\frac{2}{h}  f(y)} dy  }$. 
\medskip

\noindent
In all the following, we will consider that $\Omega$ is a bounded
domain in $\R^d$. In this context, we have the following results from~\cite{le-bris-lelievre-luskin-perez-12}.
\begin{proposition}
Let $\Omega\subset \mathbb R ^d$ be a bounded domain and consider the
dynamics~\eqref{eq:langevin}. Then, there exists a probability measure
$\nu_h$ with support in $\Omega$ such that, whatever the law of the
initial condition $X_0$ with support in $\Omega$, 
\begin{equation}\label{eq:cv_qsd}
\lim_{t \to \infty} \| {\rm Law}(X_t| t < \tau_\Omega) - \nu_h \|_{TV} = 0.
\end{equation}
\end{proposition}
\noindent
Here and in the following, ${\rm Law}(X_t| t < \tau_\Omega)$ denotes
the law of $X_t$ conditional to the event $\{t<\tau_\Omega\}$.
A corollary of this proposition is that the quasi stationary distribution~$\nu_h$ exists and is unique.\\
This proposition also explains why it is relevant to consider the
quasi stationary distribution for a metastable domain. The domain
$\Omega$ is metastable if the convergence in~\eqref{eq:cv_qsd} is much
quicker than the exit from $\Omega$.
%This will be quantified below
%using eigenvalues of the infinitesimal generator with Dirichlet
%boundary conditions. 
In the following of this paper, we will assume
that $\Omega$ is a metastable domain, and we will thus consider the
exit event from $\Omega$, assuming that $X_0$ is distributed according
to the quasi stationary distribution $\nu_h$. Let us also mention that we will obtain       results
starting form deterministic initial conditions $X_0=x \in \Omega$ which are sufficiently low in energy (see Section~\ref{sec:unpoint}).

% As it will be stated, if the domain $\Omega$ is bounded in $\mathbb R^d$, the quasi stationary distribution $\nu_h$ exists and is unique. Moreover, 
% $$\lim_{t\to \infty} \mathcal L \left(X_t\big | t<\tau_{\Omega}\right)=\nu_h.$$
% The convergence to the quasi stationary distribution is exponentially fast in time (see Proposition~\ref{conve}). This justifies that it is relevant in order to study the exit event from $\Omega$ and the time spent in $\Omega$, to assume that $X_t$ is initially distributed according to the quasi stationary distribution $\nu_h$. $X_t$ will actually reach the quasi stationary distribution before leaving $\Omega$. This statement can be taken as a definition of metastability (metre ref).  \\

% This last remark makes also the link with the section~\ref{transitionrateslangevin} dedicated to introduce transition rates between metastable states for the dynamics (\ref{eq:langevin}). These transition rates will be computed assuming the process (\ref{eq:langevin}) starts from the quasi stationary distribution of $\Omega$. 

\subsubsection{An eigenvalue problem related to the quasi stationary distribution}

In this section, a connection is made between the quasi stationary
distribution and an eigenvalue problem for the infinitesimal generator
of the dynamics (\ref{eq:langevin}) 
\begin{equation}\label{eq:L0}
L^{(0)}_{f,h}=-\nabla f \cdot \nabla  + \frac{h}{2}  \ \Delta
\end{equation}
\label{page.LOFH}
with Dirichlet boundary conditions on $\partial \Omega$. In the notation $L^{(0)}_{f,h}$, the superscript $(0)$ indicates
that we consider an operator on functions, namely $0$-forms. 
Here and in the following, we assume that the domain $\Omega$ is a
connected open bounded~$C^{\infty}$ domain in~$\mathbb R^d$. 
Let us introduce the weighted $L^2$ space
$$L^2_w(\Omega)=\left\{u:\Omega \to \R, \, \int_\Omega u^2(x)
 e^{-\frac{2}{h} f(x)} \, dx < \infty\right\}$$
(the weighted Sobolev spaces $H^k_w(\Omega)$ are defined
similarly). The subscript $w$ in the notation $L^2_w(\Omega)$ and $H^{k}_w(\Omega)$ refers to the
fact that the weight  function~$x\in \Omega \mapsto e^{-\frac{2}{h} f(x)} $ appears in the inner product.  

The basic
observation to define our functional framework is that the operator
$L^{(0)}_{f,h}$ is self-adjoint on $L^2_w(\Omega)$. 
Indeed, for any smooth test functions $u$ and $v$  with compact supports in $\Omega$, one has
$$\int_\Omega (L^{(0)}_{f,h}u) v e^{-\frac{2}{h} f} = \int_\Omega
(L^{(0)}_{f,h}v) u e^{-\frac{2}{h} f} = - \frac{h}{2} \int_\Omega
 \nabla u\cdot \nabla v\, e^{-\frac{2}{h} f}.$$
This gives a proper framework to introduce the Dirichlet realization
$L^{D,(0)}_{f,h}(\Omega)$ on $\Omega$ of the operator~$L^{(0)}_{f,h}$:
\begin{proposition} \label{fried}
The Friedrichs extension associated with the quadratic form 
$$\phi \in C^{\infty}_c(\Omega)\mapsto \frac{h}{2} \int_{\Omega}\left
  \vert \nabla \phi\right\vert^2 e^{-\frac{2}{h}f(x)} dx,$$ 
on $L^2_w(\Omega)$, 
is denoted
$-L^{D,(0)}_{f,h}(\Omega)$. It is
a  non negative unbounded self adjoint operator on
$L^2_w(\Omega)$ with
domain $$D\left(L^{D,(0)}_{f,h}(\Omega)\right)=H^1_{w,0}(\Omega)\cap
H^2_w(\Omega)$$
where $H^1_{w,0}(\Omega)=\{u \in H^1_w(\Omega), \, u=0 \text{ on
} \partial \Omega\}$.
\end{proposition}
\label{page.LODFH}
%\begin{proof}
%\begin{sloppypar}
%The quadratic form $\phi \in C^{\infty}_c(\Omega)\mapsto \frac{h}{2} \int_{\Omega}\left
%  \vert \nabla \phi\right\vert^2 e^{-\frac{2}{h}f(x)} dx$  is symmetric, non negative and closable and its closure is the quadratic form $Q:\, w \in H^1_{w,0}(\Omega)\mapsto \frac{h}{2} \int_{\Omega}\left
%  \vert \nabla w\right\vert^2 e^{-\frac{2}{h}f(x)} dx$. Let $-L^{D,(0)}_{f,h}(\Omega)$ be the self adjoint operator associated with $Q$, which domains is 
%  $$D\left (-L^{D,(0)}_{f,h}(\Omega)\right)=\left \{ u\in H^1_{w,0}(\Omega), \, \exists b\in L^2_w(\Omega), \, \forall v\in H^1_{w,0}(\Omega), \, Q(u,v) = \langle b,v   \rangle_{L^2_w } \right \},$$
%  and $-L^{D,(0)}_{f,h}(\Omega) u=b$. Let $u\in D\left (-L^{D,(0)}_{f,h}(\Omega)\right)$. Then, in the sense of distribution, it holds 
%  $-\frac{h}{2} \div \left ( e^{-\frac 2h f} \nabla u\right )=b$ and from standard regularity results for elliptic operators, we get $u\in H^2_w(\Omega)$. Therefore $D\left (-L^{D,(0)}_{f,h}(\Omega)\right)=H^1_{w,0}(\Omega)\cap
%H^2_w(\Omega)$. \end{sloppypar} \end{proof}
\noindent
The compact injection  $H^1_w(\Omega)\subset L^2_w(\Omega)$ (which follows from~\cite[Theorem 1 in Section 5.7]{evans-10}  together with the fact that if a sequence $(u_n)_{n\in \mathbb N}$ is bounded in  $H^1_w(\Omega)$ then, $(e^{-\frac 1hf}\, u_n)_{n\in \mathbb N}$ is bounded in  $H^1(\Omega)$),   implies that the operator $L^{D,(0)}_{f,h}(\Omega)$ has  compact
resolvent.  
 Consequently, its spectrum is purely discrete. Let us introduce $\lambda_h >0$ the smallest eigenvalue
 of $-L^{D,(0)}_{f,h}(\Omega)$. One has the 
 following proposition, which
 follows from standard results for the first eigenfunction of an
 elliptic operator, see for example~\cite[Section 6.3 and Theorem 2 in Section 6.5]{evans-10} and~\cite[page 128]{le-bris-lelievre-luskin-perez-12}.
\begin{proposition} \label{prop:QSD_spec}
   The smallest eigenvalue $\lambda_h$
 of $-L^{D,(0)}_{f,h}(\Omega)$  is non degenerate and its associated eigenfunction $u_h$
 has a sign on~$\Omega$.  Moreover $u_h \in C^{\infty} (\overline \Omega)$.
\end{proposition}
\noindent
Without  \label{page.lambdah}
 loss of
 generality, one can assume that:
\begin{equation}\label{eq.u_norma0}
u_h > 0 \text{ on } \Omega \ \text{ and }  \int_{\Omega} u_h^2(x)\ e^{-\frac{2}{h}f(x)} dx=1.
\end{equation}
\label{page.uh}
The eigenvalue-eigenfunction couple $(\lambda_h,u_h)$ satisfies:
\begin{equation} \label{eq:u}
\left\{
\begin{aligned}
 -L^{(0)}_{f,h}\, u_h &=  \lambda_h u_h    \ {\rm on \ }  \Omega,  \\ 
u_h&= 0 \ {\rm on \ } \partial \Omega.
\end{aligned}
\right.
\end{equation}
The link between the quasi
stationary distribution $\nu_h$ (see Definition~\ref{defQSD}) and $u_h$
is given by the following proposition (see for example \cite{le-bris-lelievre-luskin-perez-12}):
\begin{proposition} \label{uniqueQSD}
The unique quasi stationary distribution $\nu_h$ associated with the
dynamics~\eqref{eq:langevin} and the domain $\Omega$ is given by:
\begin{equation} \label{eq:expQSD}
\nu_h(dx)=\frac{\displaystyle  u_h(x) e^{-\frac{2}{h}  f(x)}}{\displaystyle \int_\Omega u_h(y) e^{-\frac{2}{h}  f(y)}dy}dx,
\end{equation}
\label{page.nuh2}
where $u_h$ is the eigenfunction associated with the smallest eigenvalue of $-L^{D,(0)}_{f,h}(\Omega)$ (see Proposition~\ref{prop:QSD_spec}).
\end{proposition}

\subsubsection{Back to the jump Markov process}\label{sec:back_jump}

As explained in Section~\ref{sec:QSD_1st_prop}, if the process remains
for a sufficiently long time in the domain $\Omega$, it is natural to
consider the exit event starting from the quasi stationary
distribution attached to $\Omega$. The next proposition characterizes
the law of this exit event (see for example~\cite{le-bris-lelievre-luskin-perez-12}).
\begin{proposition}\label{indep1}
Let us consider the dynamics~\eqref{eq:langevin} and the quasi
stationary distribution $\nu_h$ associated with the domain $\Omega$. If
$X_0$ is distributed according to $\nu_h$, the random variables
$\tau_{\Omega}$ and $X_{\tau_{\Omega}}$ are independent. Furthermore, $\tau_{\Omega}$ is exponentially distributed with parameter $\lambda_h$ and the law of $X_{\tau_{\Omega}}$ has a density with respect to the Lebesgue measure on $\partial \Omega$ given by
\begin{equation}\label{eq:dens}
z\in \partial \Omega \mapsto - \frac{h}{2\lambda_h} \frac{ \partial_n
  u_h(z) e^{-\frac{2}{h} f(z)}}{\displaystyle \int_\Omega u_h(y) e^{-\frac{2}{h} f(y)}dy},
\end{equation}
where $u_h$ is the eigenfunction associated with the smallest eigenvalue $\lambda_h$  of $-L^{D,(0)}_{f,h}(\Omega)$ (see Proposition~\ref{prop:QSD_spec}). 
\end{proposition}
\noindent
Here and in the following, $\partial_{n}=n\cdot \nabla$ stands for the normal derivative and $n$ is the unit outward normal on~$\partial \Omega$.\\
This proposition shows that, starting from the quasi-stationary
distribution in the domain $\Omega$, the exit event can be modeled by
a jump Markov process {\em without any approximation}. Indeed, using
the notation of Section~\ref{jump}, let us consider that $\Omega
\subset \R^d$ is associated with the state $0$. Let us assume that
$\Omega$ is surrounded be $n$ neighbourding states, associated with
domains $(\Omega_i)_{i=1,\ldots, n}$ (see Figure~\ref{fig:partition}
  for a schematic representation when $n=4$). Let us define the
  transition rates:
  \label{page.ex}
\begin{equation}\label{eq:exact_rates}
\forall i \in \{1, \ldots n\},\, k_{0,i}=\frac{\P_{\nu_h}
  \left(X_{\tau_\Omega} \in \partial \Omega \cap
    \Omega_i\right)}{\E_{\nu_h}(\tau_\Omega)}.
\end{equation}
\label{page.koi}
Then, by Proposition~\ref{indep1}, the exit event is such that:
\begin{itemize}
\item The residence time $\tau_\Omega$ is exponentially distributed
  with parameters $\sum_{i=1}^n k_{0,i}$.
\item The next visited state is independent of the residence time and is $i$ with probability
  $\frac{k_{0,i}}{\sum_{j=1}^n k_{0,j}}$.
\end{itemize}
This is exactly the two properties~\eqref{eq:exit_time_kMC}
and~\eqref{eq:exit_point_kMC} which are required to define a
transition using a jump Markov process. The quasi stationary
distribution can thus be used to parameterize the underlying jump
Markov process if the domains are metastable.
\medskip

\noindent
The question we would like to address in this work is now the
following: what is the error introduced when one approximates the
exact rates~\eqref{eq:exact_rates} using the Eyring-Kramers
formula~\eqref{eq:kramers}--\eqref{eq:kramers_pref}. From
Proposition~\ref{indep1}, since $\E_{\nu_h}(\tau_\Omega)=1/\lambda_h$, one has the following formula for the exact
rates:
\begin{equation}\label{eq:exact_rates_formula}
k_{0,i}=- \frac{h}{2} \frac{ \displaystyle \int_{\partial
    \Omega\cap \partial\Omega_{i}}(\partial_n u_h )(z)\, e^{-\frac{2}{h} f(z)}  \, 
 \sigma (dz) }{\displaystyle \int_{\Omega} u_h(y) \, e^{-\frac{2}{h} f(y)} \, dy}
\end{equation}
where $\sigma$ denotes the Lebesgue measure on $\partial \Omega$.
We will be able to prove that in the small temperature regime $h \to
0$, the exact rates~\eqref{eq:exact_rates_formula} can indeed be 
accurately approximated by the Eyring-Kramers
formula~\eqref{eq:kramers} with explicit error bounds. The asymptotic
analysis is done directly on the rates, and not only on the logarithm
of the rates (which is the typical result obtained with the large
deviation theory for example, see Section~\ref{biblio}).

\subsection{Statement of the main result} \label{statementthm}

We state in this section the main result of this work (Theorem~\ref{TBIG0})
on the asymptotic behavior of the normal derivative $\partial_n u_h$
in the regime $h \to 0$, as well as its corollary on the exit point
density and the accuracy of the approximation of the exit rates by
the Eyring-Kramers formula. \\
This section is organized as follows. We introduce in Section~\ref{sec:def_agmon_intro} a crucial tool in our analysis, the Agmon distance. Then, in Section~\ref{sec:hypoa},  we give the set of hypotheses which will be needed throughout this work. Finally, Section~\ref{sec:main_result} is dedicated to the statement of our main result.

 \subsubsection{Agmon distance}\label{sec:def_agmon_intro}

Our results hold under some geometric assumptions which require to
introduce the so-called Agmon distance. The
objective of this section is to introduce this distance, which is
particularly useful to quantify the decay of eigenfunctions away from
critical points~\cite{simon-84,helffer-sjostrand-84}. We introduce the
Agmon distance in a general setting, namely for $\Omega$ a Riemannian
manifold, but one could think  of $ \Omega$
 as a $C^{\infty}$ connected open bounded subset of $\mathbb R^d$.

% The Agmon distance appears in the hypotheses of Theorem~\ref{TBIG0} and Theorem~\ref{th.gene_sigma}. This part aims at giving a definition of the Agmon distance before stating Theorem~\ref{TBIG0} and Theorem~\ref{th.gene_sigma}. Let us first define the length of a curve. 

\begin{definition} \label{L}
Let $\overline \Omega$ be a $C^{\infty}$ oriented connected  compact  Riemannian  manifold of dimension $d$ with boundary $\partial \Omega$ and $f: \overline \Omega\to \mathbb R$ be $C^{\infty}$. Define $g : \overline \Omega \to  \mathbb R$ by
\begin{equation}\label{eq:def_g}
\forall x \in \Omega, \ g(x)= \left\vert \nabla f(x)\right\vert
\text{ and } \ \forall x \in \partial \Omega, \ g(x)= \left\vert
  \nabla_T f(x)\right\vert,
\end{equation}
\label{page.g} 
where for any $x \in \partial \Omega$, $\nabla_Tf(x)$ denotes the
tangential gradient of the function $f$ on $\partial \Omega$, i.e.  $\nabla_Tf(x)=\nabla f(x)-(\nabla f(x)\cdot n) \, n$, where $n$ is the unit outward normal  to~$\pa \Omega$ at~$x$.   
One defines the length $L$ of a Lipschitz curve $\gamma: I\to \overline \Omega$, where~$I\subset \mathbb R $ is an interval, by
$$L(\gamma,I):= \int_I g\left(\gamma(t)\right) \left\vert \gamma'(t) \right\vert  dt \in [0+\infty].$$
\end{definition}
\noindent
Let \label{page.L}  us recall that the 
Rademacher theorem (see for example~\cite{evans-gariepy-92}) states
that every Lipschitz function admits almost everywhere a derivative (which is then bounded by the Lipschitz constant). Therefore, if $I$ is bounded, then $L(\gamma,I)<\infty$. Let us now define the Agmon distance.
\begin{definition} \label{def.agmon-intro}
Let $g$ be the function introduced in Definition~\ref{L}. The Agmon distance between $x\in \overline \Omega$ and $y \in \overline \Omega$ is defined by
\begin{equation} \label{eq:definition}
d_a\left(x,y\right)=\inf _{\gamma\in {\rm Lip}(x,y)} L\left(\gamma,(0,1)\right),
\end{equation} 
\label{page.dagmon}
where ${\rm Lip}\left(x,y\right)$ is the set of  curve $\gamma :[0,1] \to \overline \Omega$ which are Lipschitz with $\gamma(0)=x$, $\gamma(1)=y$.
\end{definition}
\noindent
The Agmon distance is obviously symmetric, non negative and satisfies
the triangular inequality. It is a distance if the critical points of
$f$ and $f_{\big | \partial \Omega}$ are isolated (see Proposition~\ref{opoo} below). Let us mention that in the case when $\Omega$ is a manifold without boundary, the Agmon distance introduced in Definition~\ref{def.agmon-intro} coincides with the Agmon distance defined in \cite[Appendix 2]{helffer-sjostrand-85}. 

We will give in Section~\ref{agmonproperty} more details about the
Agmon distance we consider. In particular, it will be shown that the
Agmon distance to the critical points of $f|_{\partial \Omega}$
coincides with the solution to the eikonal equation $\vert \nabla \Phi
\vert ^2=\vert \nabla f \vert^2$ in neighborhoods of the critical
points. This requires to use the tangential gradient of $f$ on
$\partial \Omega$ in the definition of the Agmon distance (see~\eqref{eq:def_g}).

% The properties of the Agmon distance on a domain without boundary are already known (see \cite{helffer-sjostrand-84,helffer-sjostrand-85a,helffer-sjostrand-85b,helffer-sjostrand-86a}). To the best of our knowledge, no definition of the Agmon distance on a domain with boundary, exists in the literature. In this work, it is required to have an Agmon distance which coincides with the eikonal solution near some critical points of the restriction of $f$ to the boundary (see \ref{agmonproperty}). In Section~\ref{agmonproperty}, one will show that they coincide if one puts the tangential gradient of $f$ in the definition of the Agmon distance (see Definition~\ref{def.agmon-intro}). 

\subsubsection{Notations and hypotheses}  \label{sec:hypoa}

As already stated above, we assume that $\Omega$ is a connected open
bounded $C^{\infty}$ domain of $\mathbb R^d$ and $f :  \overline{\Omega}  \to
\mathbb R$ is a $C^{\infty}$ function.\footnote{Actually, as explained
  in Section~\ref{sec:gram_schmidt}, we will perform the
analysis in a more general setting, namely when $\overline{\Omega}$ is
a $C^\infty$ oriented connected compact Riemannian manifold. In this
introductory section, we stick to a simpler presentation, with
$\Omega$ a subset of $\R^d$.}
We will need the following set of assumptions:

\begin{itemize}
\item[\textbf{[H1]}] The function $f : \overline{\Omega} \to \mathbb
R$ is a Morse function on $\Omega$ and the restriction of $f$ to the
boundary of $\Omega$ denoted by $f | _{  \partial \Omega}$, is a Morse
function.  The function $f$ does not have any critical point on
$\partial \Omega$. 

\item[\textbf{[H2]}] The function $f$ has a unique global minimum $x_0\in \Omega$ in $\overline{\Omega}$:
$$\min_{\partial \Omega}f>\min_{\overline \Omega}f= \min_{\Omega}f=f(x_0).$$ 
The point $x_0$ is the unique critical point of $f$ in
$\overline{\Omega}$.  The function
$f|_{\partial \Omega}$ has exactly $n \ge 1$ local minima denoted by
$(z_i)_{i=1,\ldots,n}$ such that $f(z_1)\leq
f(z_2)\leq \ldots \leq f(z_n)$.
 \label{page.z1zn}
\item[\textbf{[H3]}] $\partial_n f>0$ on $\partial \Omega$.
\end{itemize}
In the following, $n_0 \in \{1, \ldots, n\}$ denotes the number of points in $\arg\min f|_{\partial \Omega}$:
$$f(z_1)=\ldots
= f(z_{n_0}) < f(z_{n_0+1}) \le \ldots \leq f(z_n).$$ 
 \label{page.n0}
We will need to define the basins of attraction of the local minima
$z_i$ for the
dynamics $\dot{x}=-\nabla_T f(x)$ in $\partial \Omega$, where, we
recall, for any $x \in \partial \Omega$, $\nabla_T f(x)$ denotes the
tangential gradient of $f$ on $\partial \Omega$.

\begin{definition} \label{Bz}
Assume that \textbf{[H1]}  holds. For each local minimum $z
\in \partial \Omega$, one denotes by $B_{z}\subset \partial \Omega$
the basin of attraction of $z$ for the dynamics in $\partial \Omega$
$\dot{x}=-\nabla_T f(x)$: denoting by~$\varphi_{t}(y)$ the solution to $ \frac{d}{dt}\varphi_{t}(y)=-\nabla_T f(\varphi_{t}(y))$
with initial condition $\varphi_{0}(y)=y\in \pa \Omega$,  one has 
 $B_{z}:= \big \{y\in\pa \Omega, \, \lim_{t \to \infty} \varphi_{t}(y)
 =z   \big \}$. Notice that $B_z$ is an
 open subset of $\pa \Omega$.
Additionally, one defines $B_{z}^c:=\partial \Omega \setminus B_{z}$.
\end{definition}
 \label{page.Bzi}
From this definition, one obviously has that for  each local minimum
$z \in \partial \Omega$, for any  $x \in B_{z}$, $f(x) \ge
f(z)$.
On Figure~\ref{fig:H1H2H3H4}, one gives a schematic representation in dimension $2$ of a function~$f$ satisfying the assumptions~\textbf{[H1]},~\textbf{[H2]}, and~\textbf{[H3]}, and of its restriction to~$\pa \Omega$, in the case  $n=4$ and $n_0=2$.

\begin{figure}[h!]
\begin{center}
\begin{tikzpicture}[scale=0.5]
\tikzstyle{vertex}=[draw,circle,fill=black,minimum size=4pt,inner sep=0pt]
\tikzstyle{ball}=[circle, dashed, minimum size=1cm, draw]
\tikzstyle{point}=[circle, fill, minimum size=.01cm, draw]
\draw [thick, rounded corners=10pt] (1,0.5) -- (-0.25,2.5) -- (1,5) -- (5,6.5) -- (7.6,3.75) -- (6,1) -- (4,0) -- (2,0) --cycle;
%%%
 
\draw [ densely dashed, rounded corners=10pt] (0.9,0.5) -- (-0,2.5) -- (3,5.5) -- (5.6,4.5) --(5.5,1) -- (4,-0.1) -- (2.2,1.5) --cycle;

\draw [densely dashed] (3.4,3) circle (1);
\draw [densely dashed] (3.4,3) circle (1.5);

\draw [densely dashed] (-8.9,3) -- (-7.9,3);
     \draw  (-5.5,3) node[]{$\text{level sets of $f$}$};

     \draw  (6.7,3.5) node[]{$\Omega$};
   
\draw (1.7,5.3) node[vertex,label=north: {$z_4$}](v){};
\draw (3.4,3) node[vertex,label=south: {$x_0$}](v){};
\draw (4.2,0.1) node[vertex,label=south west: {$z_2$}](v){};
\draw (0.38,1.45) node[vertex,label=south west: {$z_1$}](v){};
\draw (6.2,5.2) node[vertex,label=north east: {$z_3$}](v){};
\end{tikzpicture}

\hspace{0.7cm}

\begin{tikzpicture}[scale=0.75]
\tikzstyle{vertex}=[draw,circle,fill=black,minimum size=4pt,inner sep=0pt]
\tikzstyle{ball}=[circle, dashed, minimum size=1cm, draw]
\tikzstyle{point}=[circle, fill, minimum size=.01cm, draw]

\draw [dashed] (-5.4,-0.6)--(6,-0.6);
\draw [dashed,->] (-5.4,-0.6)--(-5.4,3);
\draw [dashed] (6,-0.6)--(6,3);
\draw [densely dashed,<->] (-5.25,3)--(5.97,3);
 \draw (0,3.34) node[]{$\pa \Omega$};
 \draw (-5.9,2.8) node[]{$f|_{\pa \Omega}$};
 \draw[thick] (-5.4,2) ..controls  (-5.22,1.96).. (-5,1.6)  ;
\draw[thick] (-5,1.6) ..controls  (-3,-1.4).. (-2,2.2)  ;
\draw[thick] (-1.6,2.2) ..controls  (0,-1.5).. (1.8,2)  ;
\draw[thick] (-2,2.2) ..controls  (-1.8,2.81).. (-1.6,2.2)   ;
\draw[thick] (2,2) ..controls  (1.91,2.13).. (1.8,2)   ;
\draw[thick] (4.1,2.5) ..controls  (3,-0.3).. (2,2)  ;
\draw[thick]  (4.1,2.5) ..controls  (4.28,2.66).. (4.8,1.5)  ;
\draw[thick]  (4.8,1.5)  ..controls  (5.04,1).. (5.5,1.6)  ;
\draw[thick]  (5.5,1.6)   ..controls  (5.87,2).. (6,2)  ;

\draw (5.1,1.1) node[vertex,label=south: {$z_4$}](v){};
\draw (3,0.3) node[vertex,label=south: {$z_3$}](v){};
\draw (-3.17,-0.6) node[vertex,label=south: {$z_1$}](v){};
\draw (0,-0.6)  node[vertex,label=south: {$z_2$}](v){};

\draw[<->]  (-5.4,-1.4) -- (-1.92,-1.4) ;
\draw[<->]  (-1.9,-1.4) -- (1.89,-1.4) ;
\draw[<->]   (1.92,-1.4) -- (4.1,-1.4) ;
\draw[<->]   (4.1,-1.4)-- (6,-1.4) ;
 \draw (-4,-1.8) node[]{$B_{z_1}$};
  \draw (0,-1.8) node[]{$B_{z_2}$};
   \draw (3.2,-1.8) node[]{$B_{z_3}$};
  \draw (5,-1.8) node[]{$B_{z_4}$};
\end{tikzpicture}

\caption{Schematic representation  in dimension $2$ of a function $f$ satisfying the assumptions \textbf{[H1]}, \textbf{[H2]}, and \textbf{[H3]}, and of its restriction $f|_{\pa \Omega}$ to  $\pa \Omega$.  On the figure, $n=4$ and $n_0=2$. }
 \label{fig:H1H2H3H4}

\end{center}
\end{figure}
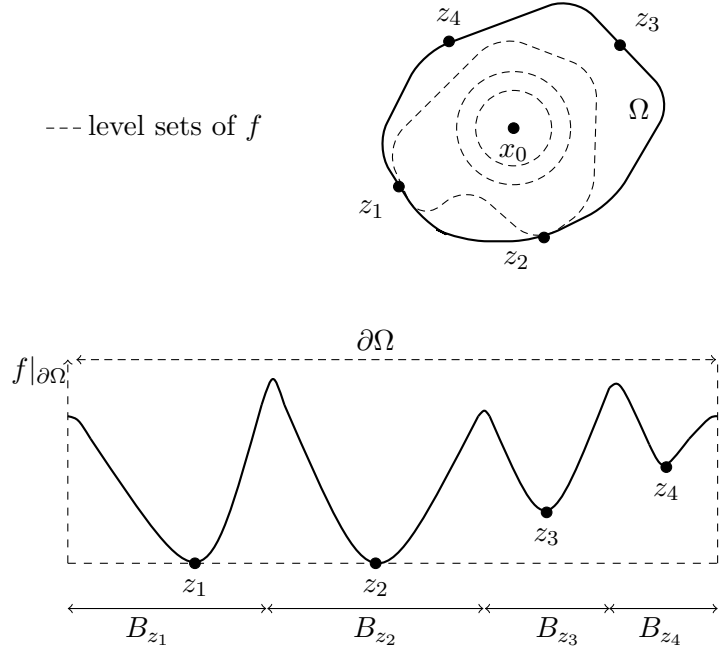

As a consequence of the assumption \textbf{[H1]}, the determinants of the
Hessians of $f$ (resp. of $f|_{\partial \Omega}$) at the critical points of
$f$ (resp. of  $f|_{\partial \Omega}$) are non zero. These
quantities appear in the prefactors of the Eyring-Kramers law (see
Equation~\eqref{resultatkil} below).

\begin{remark}\label{re:Hessian}
Let us recall how the Hessians are defined. 
Let $\phi: N \to \R$ be a $C^\infty$ function defined on a Riemannian $C^\infty$
manifold $N$ of dimension $d$. By standard results of Riemannian geometry, the Hessian $\Hess \, \phi(x)$ of $\phi$ at a point $x \in N$ is defined as a bilinear symmetric form acting on vectors in the tangent space $T_x N$ as:
\begin{equation}\label{eq:Hessian}
\forall X,Y \in \Gamma(T N),\, \Hess \, \phi(X,Y) = \nabla_X d\phi(Y)
\end{equation}
where $\nabla$ is the covariant derivative (Levi-Civita connection) and $d\phi$ is the differential of~$\phi$. Then, $\det \, \Hess \,\phi(x)$ is defined as the determinant of the bilinear form $\Hess \, \phi(x)$ in any orthonormal basis of $T_x N$.

In practice, $\det \, \Hess \, \phi(x)$ can be computed at a critical point of $\phi$  using a local chart as follows. Let us assume that $x_0$ is a critical point of  $\phi$: $d_{x_0} \phi = 0$. Let us introduce
$\psi: y\in U \mapsto \psi(y)\in V$ a local chart around~$x_0$, where $U \subset \R^d$ is a neighborhood
of $0$, $V \subset N$ is a neighborhood of $x_0$ and $\psi(0)=x_0$. Let us assume in
addition that the vectors $(e_i)_{i=1,\ldots,d}:=\left(\frac{\partial \psi}{\partial
    y_i}(0)\right)_{i=1,\ldots,d}$ are orthonormal (thus defining an
orthonormal basis of $T_{x_0} N$). Let us introduce
the symmetric matrix $H$ associated with the second order differential
of $\phi \circ \psi$ at point $0$: $\forall (u,v) \in \R^d \times \R^d$
$$D^2_0 (\phi \circ \psi) \left(\sum_{i=1}^d u_i e_i, \sum_{i=1}^d v_i e_i
\right)=u^T H v.$$
Then 
$$\det \, \Hess \, \phi(x_0)  = \det \, H.$$
 This formula is only valid at a critical point and is a direct consequence of the definition~\eqref{eq:Hessian} of the Hessian and the explicit expression of the Levi Civita connection in the local chart~$\psi$:
$$\nabla_X d\phi(Y)|_x= \sum_{i,j=1}^d \left( \frac{\partial^2 (\phi\circ \psi)}{\partial {y_i \partial y_j}  }(y) -\sum_{k=1}^d \Gamma^k_{i,j}(\psi(y))\frac{\partial (\phi\circ \psi)}{\partial{y_k}  }(y)    \right)  Y_iX_j $$
where $x=\psi(y)\in V$, $\Gamma^k_{i,j}(x)$ are the Christoffel symbols of the connection $\nabla$ associated with the basis $\left ((\partial_{y_j}\psi)(\psi^{-1}(x)\right)_{j=1,\dots,n}$ of $T_xN$ and $(X_j)_{j=1,\dots,n}$ (respectively $(Y_j)_{j=1,\dots,n}$) are the coordinates of $X$ (respectively $Y$) in this basis. \end{remark}

\subsubsection{Main result}  \label{sec:main_result}

In view of equations~\eqref{eq:dens}
and~\eqref{eq:exact_rates_formula}, we need to give an estimate of three
quantities in order to analyze the exit point density and the
asymptotic of the transition rates in the regime $h \to 0$:
$\int_{\Sigma} (\partial_n u_h) \,e^{-\frac{2}{h} f}$ for a subset $\Sigma$ of
$\partial \Omega$, $\int_\Omega u_h \, e^{-\frac{2}{h} f}$ and
$\lambda_h$, where, we recall $(\lambda_h,u_h)$ is defined by
\eqref{eq:u}. We will consider a subset $\Sigma$ such that $\Sigma
\subset B_{z_i}$ for a local minimum $z_i$ (see Definition~\ref{Bz}
for the definition of $B_{z_i}$).  

\begin{theorem} \label{TBIG0}
Assume that \textbf{[H1]}, \textbf{[H2]} and \textbf{[H3]} hold. Moreover assume that
\begin{itemize}
\item $\forall i\in \left\{1,\ldots,n\right\}$, \begin{equation} \label{hypo1}
 \inf_{z\in B_{z_i}^c} d_a(z,z_i) >\max[f(z_n)-f(z_i),f(z_i)-f(z_1)],
 \end{equation}
\item  and \begin{equation} \label{hypo2} f(z_1)-f(x_0)>f(z_n)-f(z_1).  \end{equation}
\end{itemize}
Then, for all $i\in \left\{1,\ldots,n\right\}$ and all open set
$\Sigma_i \subset \partial \Omega$ containing $z_i$ and such that
$\overline\Sigma_i \subset B_{z_i}$, in the limit $h\to 0$
 \label{page.Sigmai}
\begin{equation}\label{eq:dnuh_asymptot}
\int_{\Sigma_i}   (\partial_n u_h )\,  e^{- \frac{2}{h}  f} d\sigma=A_i(h) \,e^{-\frac{2f(z_i)-f(x_0)}{h}}\left( 1+ O(h) \right),
\end{equation}
 where $u_h$ is the eigenfunction associated with the smallest eigenvalue of $-L^{D,(0)}_{f,h}(\Omega)$ (see Proposition~\ref{prop:QSD_spec}) which satisfies \eqref{eq.u_norma0} and
 $$A_i(h)=-\frac{ ({\rm det \ Hess } f   (x_0))^{1/4}   \partial_nf(z_i)   2\pi^{\frac{d-2}{4}}    }{    \sqrt{ {\rm det \ Hess } f|_{\partial \Omega}   (z_i) }  } h^{\frac{d-6}{4}}.$$
 \end{theorem}
 
Let us mention that the importance of the  geometric assumptions  \eqref{hypo1} and ~\eqref{hypo2} will be  discussed in Section \ref{sec:numeric}.
 
 \begin{remark}
As will become clear in the proof of Theorem \ref{TBIG0}, it can actually be proven that for all $i\in \left\{1,\ldots,n\right\}$, the residual $r_i(h)=O(h)$ appearing in~\eqref{eq:dnuh_asymptot} admits a full asymptotic expansion in $h$: there exists a sequence $(b_{k,i})_{k \ge 0}\in \mathbb R^{\mathbb N}$ such that for all $N\in \mathbb N$, in the limit $h \to 0$,
 $$r_i(h) = h\, \sum_{k=0}^{N} b_{k,i}h^k+ O(h^{N+2}).$$
 We do not state our main result with this expansion since, for general domains $\Omega$, the explicit computations of the sequence $(b_{k,i})_{k \ge 0}$ is not possible in practice. This remark also holds for all the residuals $O(h)$  in the next results. 
 \end{remark}
 \begin{proposition} \label{moyenneu}
Assume that \textbf{[H1]}, \textbf{[H2]} and \textbf{[H3]} hold. Then when $h\to 0$
$$\int_{\Omega} u_h(x) \ e^{- \frac{2}{h} f(x) }dx= \frac{ \pi ^{\frac{d}{4} } }{  \left({\rm det \ Hess } f   (x_0)  \right)^{1/4}  } \  h^{\frac d4} \ e^{-\frac{1}{h}f(x_0)} (1+O(h) ),$$
where $u_h$ is the eigenfunction associated with the smallest eigenvalue of $-L^{D,(0)}_{f,h}(\Omega)$ (see Proposition~\ref{prop:QSD_spec}) which satisfies \eqref{eq.u_norma0}.
 \end{proposition}

\begin{proposition}\label{lambdah}
Assume that \textbf{[H1]}, \textbf{[H2]}, and
\textbf{[H3]} hold. Then, in the
limit $h \to 0$,
\begin{equation}\label{eq:lhh}
\lambda_h= \frac{\sqrt{ {\rm det \ Hess } f   (x_0) }  }{\sqrt{\pi h}}\sum_{i=1}^{n_0}\frac{  \partial_nf(z_i)    }{    \sqrt{ {\rm det \ Hess } f|_{ \partial \Omega}   (z_i) }  } \ e^{-\frac{2}{h}(f(z_1)-f(x_0))}\left( 1+ O(h) \right),
\end{equation}
where $\lambda_h$ is the smallest eigenvalue of $-L^{D,(0)}_{f,h}(\Omega)$ (see Proposition~\ref{prop:QSD_spec}).
\end{proposition}
\noindent
Theorem~\ref{TBIG0} is the main contribution of this
work. Actually Theorem~\ref{TBIG0} will be proven in a more general
framework: namely when $\overline \Omega$ is a $C^{\infty}$ connected compact
oriented Riemannian $d$-dimensional manifold with  boundary $\partial
\Omega$. Theorem~\ref{TBIG0}, Proposition~\ref{moyenneu} and Proposition~\ref{lambdah}  are respectively proved in Sections \ref{sec:goodquasimodes},  \ref{sec:moyeu} and \ref{sec:lambdah}.  
  For the sake of completeness, we provide a proof of Proposition~\ref{lambdah} in our specific setting, but this result actually holds under weaker geometric assumptions, see~\cite{di-gesu-lelievre-le-peutrec-nectoux-16} or \cite{helffer-nier-06}.  
\medskip

\noindent
These results have the following consequence on the first exit point
distribution and the estimate of the exact rates~$(k_{0,i})_{i=1,\ldots,n}$ using the
Eyring-Kramers formula (see Section~\ref{sec:back_jump}). We recall
that~$(X_t)_{t\geq 0}$ denotes the solution to~(\ref{eq:langevin}),
$\tau_{\Omega}$ is the exit
time from the domain~$\Omega$ and~$\nu_h$ is the quasi stationary distribution associated with~$(X_t)_{t\ge 0}$ and~$\Omega$.

\begin{corollary}\label{co.proba-sigmai}
 Under the hypotheses of Theorem~\ref{TBIG0}, for $i\in
\{1,\ldots,n\}$ and  for all
 open sets $\Sigma_i \subset \partial \Omega$ containing $z_i$ and
 such that $\overline\Sigma_i \subset B_{z_i}$, in the limit $h\to 0$:
\begin{align} 
\nonumber  
\mathbb P_{\nu_h} \left[ X_{\tau_{\Omega}} \in \Sigma_i\right]&=  \frac{\partial_nf(z_i)  }{ \sqrt{ {\rm det \ Hess } f_{|\partial \Omega }   (z_i) }}  
\left ( \sum_{k=1}^{n_0} \frac{ \partial_nf(z_k) }{\sqrt{ {\rm det \ Hess } f_{|\partial \Omega }   (z_k) } } \right)^{-1}  \\
\label{eq:limproba}
&\quad \times e^{-\frac{2}{h} (f(z_i)-f(z_1))}  (   1+   
 O(h)  ).
 \end{align}
\end{corollary}
\begin{sloppypar}
\noindent
As a simple consequence of Corollary~\ref{co.proba-sigmai}, we obtain the expected 
result that $(X_t)_{t \ge 0}$ leaves $\Omega$ around the global
minima of $f$ on $\partial \Omega$:  for any 
 collection of open sets $(\Sigma_j)_{1 \le j \le n_0}$ such that for all $j\in \{1,\dots,n_0\}$, $\overline \Sigma_j \subset B_{z_j}$ and $z_j \in \Sigma_j$, in
 the limit $h\to 0$,
$ \mathbb P_{\nu_h}\left [ X_{\tau_{\Omega}} \in \bigcup_{j=1}^{n_0}\Sigma_j\right] =1+ O (h)$. 
Actually, this latter  result can be proven with an exponentially small residual ($O (h)$ is replaced by $O
\left(e^{-c/h}\right)$ for some positive $c$) in a more general setting (see for instance
\cite{kamin-78,perthame-90,sugiura-95,freidlin-wentzell-84,day1987r,day1984a}). Let us also refer
to~\cite{di-gesu-lelievre-le-peutrec-nectoux-16} where we discuss this result in a more general
setting (for example $f$ can have several critical points in $\Omega$ and the assumptions~\eqref{hypo1} and~\eqref{hypo2} are not
needed). 

\end{sloppypar}
\begin{corollary}\label{cor:rates}
Let us assume that the hypotheses  of  Theorem~\ref{TBIG0} are satisfied. 
Let  $i\in
\{1,\ldots,n\}$ and   
 $\Sigma_i \subset \partial \Omega$ be an open set  containing $z_i$  
 such that $\overline\Sigma_i \subset B_{z_i}$. 
Using the notation of Section~\ref{sec:back_jump}, assume that
$\Sigma_i$ is the common boundary between $\Omega$ and another domain
$\Omega_i\subset \mathbb R^d$.  Under the hypotheses of Theorem~\ref{TBIG0}, the transition rate given
by~\eqref{eq:exact_rates}, to go from $\Omega$ to $\Omega_i$ satisfies, in the limit $h \to 0$,
\begin{equation} \label{resultatkil}
k_{0,i}=\frac{1}{\sqrt{\pi h}}  \partial_nf(z_i)  \frac{ \sqrt{  {\rm det \ Hess } f  (x_0) } }{  \sqrt{  {\rm det \ Hess } f_{|\partial \Omega }   (z_i) }} e^{-\frac{2}{h}(f(z_i)-f(x_0))} (1+ O (h)).
\end{equation} 
\end{corollary}
\noindent
This corollary thus gives a justification of the Eyring-Kramers formula and the Transition State Theory to build Markov models.
As stated in the assumptions, the exit rates are obtained assuming $\partial_n f >
0$ on $\partial \Omega$: the local minima $z_1, \ldots ,z_n$ of $f$ on
$\partial \Omega$ are therefore
not saddle points of $f$ but so-called {\em generalized saddle points} (see~\cite{helffer-nier-06,le-peutrec-10}). This  appellation ''generalized saddle points'' is justified by the fact that, under  \textbf{[H1]}, \textbf{[H2]}, \textbf{[H3]} and when~$f$ is extended by~$-\infty$ outside~$\overline \Omega$ (which  is consistent with the Dirichlet boundary conditions used to define $L^{D,(0)}_{f,h}$), the points $(z_i)_{i=1,\ldots,n}$ are geometrically   saddle points of $f$:~$z_i$ is a local minimum of $f|_{\pa \Omega}$ and a  local maximum of~$f|_{D_i}$, where $D_i$ is  the straight line passing through~$z_i$ and orthogonal to  $\pa \Omega$   at~$z_i$.
 In a future work, we intend to extend these results to the case where
the points $(z_i)_{1 \le i \le n}$ are saddle points of~$f$, in which
case we expect to prove the same result~\eqref{resultatkil} for the exit
rates, with a modified prefactor:
%$$A_{0,i}=\displaystyle\frac{1}{\pi}|\lambda^-(z_j)| \frac{\displaystyle \sqrt{\det
    %\Hess \, f(x_0)}}{\displaystyle\sqrt{|\det \Hess \
 %   f(z_j)|}}$$
$$A_{0,i}=\displaystyle\frac{1}{\pi}|\lambda^-(z_i)| \frac{\displaystyle \sqrt{\det\,
  \Hess \, f(x_0)}}{\displaystyle\sqrt{|\det\,  \Hess \
   f(z_i)|}}$$
(this formula can be obtained using formal expansions on
the exit time and Laplace's method). Notice that the latter formula differs from~\eqref{eq:kramers}--\eqref{eq:kramers_pref} by a
multiplicative factor $1/2$ since $\lambda_h$ is the exit rate
from $\Omega$ and not the transition rate to one of the neighboring state. Concerning this multiplicative factor $1/2$,  we refer for example  to the
remark on page 408 in~\cite{bovier-eckhoff-gayrard-klein-04},~\cite[Remark 10]{IHPLLN}, and the results on
asymptotic exit times in~\cite{maier-stein-93}. This
factor is due to the fact that once on the saddle point, the process
has a probability one half to go back to $\Omega$, and a probability one half
to effectively leave $\Omega$, in the limit $h\to0$. This
multiplicative factor does not have any influence on the law of the
next visited state which only involves ratio of the rates $k_{0,i}$,
see Section~\ref{sec:back_jump} and Equation~\eqref{eq:limproba}.

% To check if the hypothesis \eqref{hypo1} and \eqref{hypo2} are satisfied, one can refer to Section~\ref{lower bound}.

\subsection{Discussion and generalizations}\label{sec:dis_gene}

As explained above, the interest of Theorem~\ref{TBIG0} is that it
justifies the use of the Eyring-Kramers formula to model the exit
event using a jump Markov model including the prefactors. It gives in particular the relative
probability to leave $\Omega$ through each of the local minima $z_i$ of $f$ on the
boundary $\partial \Omega$. Moreover, one obtains an estimate of the
relative error  on the exit probabilities (and
not only on the logarithm of the exit probabilities as
in~\eqref{eq:LD}): it is of order $h$, see Equation~\eqref{eq:limproba}.

In Section~\ref{sec:unpoint}, we explain how this result can be
generalized to a situation where the process $(X_t)_{t \ge 0}$ is
assumed to start under another initial condition than the quasi
stationary distribution.  The importance
of the geometric assumption~\eqref{hypo1}-\eqref{hypo2} (resp. assumption~\eqref{eq:hypo2_bis}) to obtain the 
asymptotic result of Corollary~\ref{co.proba-sigmai} (resp. its generalization to deterministic initial conditions, see Corollary~\ref{cc1}) is discussed in Section~\ref{sec:numeric}. Finally, in Section~\ref{sec:gen_sigma}, we discuss extensions to less stringent conditions than~\eqref{hypo1}-\eqref{hypo2}. Moreover
the exit
through subsets of $\partial \Omega$ which do not necessarily contain
one of the local minima~$z_i$ of $f|_{\partial \Omega}$ is considered: this shows in particular
the interest of estimating the prefactors in the asymptotic approximations of the
exit rates.
 
 \subsubsection{Extension of the result to other initial conditions}
 \label{sec:unpoint}
%
%
%Let us first notice that, under the
%assumptions stated in  Corollary~\ref{co.proba-sigmai}, for any test function
%$F\in C^{\infty}(\partial \Omega)$ satisfying ${\rm supp}\, F
%\subset B_{z_i}$ and $z_i\in {\rm int} \left ({\rm supp} \, F\right)$,  when $h\to 0$,
%\begin{equation}\label{eq:EnuF}
%\E_{\nu_h} [ F(X_{\tau_{\Omega}} )]=     \frac{\partial_nf(z_i)  }{ \sqrt{ {\rm det \ Hess } f_{|\partial \Omega }   (z_i) }}  
%\left ( \sum_{k=1}^{n_0} \frac{ \partial_nf(z_k) }{\sqrt{ {\rm det \ Hess } f_{|\partial \Omega }   (z_k) } } \right)^{-1}  e^{-\frac{2}{h} (f(z_i)-f(z_1))} (   F(z_i) +   
% O(h) ).
%\end{equation}
The question we would like to address in this section is how to
generalize Corollary~\ref{co.proba-sigmai},  to a deterministic initial condition: $X_0=x$ for $x \in \Omega$.

\begin{corollary}\label{cc1}
Let us assume that all the hypotheses of Corollary~\ref{co.proba-sigmai} are
satisfied, and that in addition there exists $i_0\in \{2,\ldots,n\}$ such
that 
\begin{equation}\label{eq:hypo2_bis}
2( f(z_{i_0})-f(z_1))<f(z_1)-f(x_0).
\end{equation}
 Let $j\in \{1,\ldots,i_0\}$ and $\alpha \in \R$ be such
that  $$f(x_0) < \alpha< 2 f(z_1)-f(z_{j}).$$ Then, for $i\in
\{1,\ldots,j\}$ and for all
 open sets $\Sigma_i \subset \partial \Omega$ containing $z_i$ and
 such that $\overline\Sigma_i \subset B_{z_i}$, we have uniformly in $x\in
f^{-1}( (-\infty, \alpha ]) \cap \Omega$, in the limit $h\to 0$:
\begin{equation}\label{eq.pxx}
\P_{x} [ X_{\tau_{\Omega}}\in \Sigma_i]=      \frac{\partial_nf(z_i)  }{ \sqrt{ {\rm det \ Hess } f_{|\partial \Omega }   (z_i) }}  
\left ( \sum_{k=1}^{n_0} \frac{ \partial_nf(z_k) }{\sqrt{ {\rm det \ Hess } f_{|\partial \Omega }   (z_k) } } \right)^{-1}  e^{-\frac{2}{h} (f(z_i)-f(z_1))} (   1 +   
 O(h) ).
 \end{equation}
\end{corollary}
Let us give a simple example to illustrate this result. In a situation where $n=2$, this corollary shows that the
estimates we have obtained on the probability to exit in a
neighborhood of $z_2$ under the
assumption $X_0 \sim \nu_h$ are still valid if $X_0=x$ for $x \in
f^{-1}\left((-\infty,2f(z_1)-f(z_2))\right) \cap \Omega$ under the assumption $f(z_1) - f(x_0) > 2
(f(z_2)-f(z_1))$, which is a stronger assumption than~\eqref{hypo2}.

\subsubsection{On the geometric assumptions~\eqref{hypo1},~\eqref{hypo2} and~\eqref{eq:hypo2_bis}}\label{sec:numeric}

\underline{On the geometric assumption~\eqref{hypo1}.}
\medskip

\noindent
The question we would like to address is the following:
is  the assumption~\eqref{hypo1} necessary for the result on the exit
point density~\eqref{eq:limproba} to hold?

In order to test this assumption numerically, we consider the
following simple two-dimensional setting. The potential function is $$f(x,y)=x^2+y^2-ax,$$
with $a\in(0,1/9)$, and the domain $\Omega$ is defined by (see Figure \ref{fig:domaine}):
$$\Omega=(-1,1)^2\cup \left\{\left(x,y\right)| x^2+(y-1)^2<1\right\}\cup \left\{\left(x,y\right)| x^2+(y+1)^2<1\right\}.$$
The two local minima of $f$ on $\partial \Omega$ are $z_1=(1,0)$ and
$z_2=(-1,0)$. Notice that $f(z_2)-f(z_1)=2a>0$. The potential $f$ has
a unique critical point in $\Omega$, namely the global minimum
$x_0=(a/2,0)$. Let us check that the assumptions of
Theorem~\ref{TBIG0} are satisfied in this setting (i.e. for $a\in (0,\frac 19)$). Indeed, the inequality $f(z_1)-f(x_0)>f(z_2)-f(z_1)$ is satisfied if and only if $1-3a+\frac{a^2}{4}>0$ i.e. if and only if $a\notin  (2(3-\sqrt 8), 2(3+\sqrt 8))$. Moreover, using Proposition~\ref{agmonz1}, the inequality $d_a(z_1,B_{z_1}^c)>f(z_2)-f(z_1)$ is satisfied. Finally, to check that the inequality  $d_a(z_2,B_{z_2}^c)>f(z_2)-f(z_1)$ is satisfied we use Proposition~\ref{gradient} with $W=\{(x,y)\in \mathbb R^2, \, \vert (x,y)-z_2\vert \le\frac 13\}\cap \overline \Omega$ and $W'=\{(x,y)\in \mathbb R^2, \,\vert (x,y)-z_2\vert \le \frac 23\}\cap \overline \Omega$. In that case, one has $\alpha=\frac 13$ (where $\alpha$ is defined by \eqref{eq.alpha_agmon}) and thus  the inequality 
$$\alpha  \inf_{x\in \overline{W'\setminus  W}} g(x)=\frac 13 \min \left ( \frac 23, \Big \vert 2\big (-1+\frac 23\big)-a\Big \vert \right)=\frac 13 \min\left  ( \frac 23, \frac 23+a \right) >f(z_2)-f(z_1) =2a$$
is satisfied if and only if $a<\frac 19$. 

  \begin{figure}[h]
\begin{center}
\begin{tikzpicture}
\tikzstyle{vertex}=[draw,circle,fill=red,minimum size=6pt,inner sep=0pt]
\draw[->,dashed] (-3,0)--(3,0);
\draw[->,dashed] (0,-2.3)--(0,2.5);

\draw (1,1)--(1,-1);
\draw[|-|,ultra thick,blue]  (-1,1)--(-1,-1) node[midway,above left=0.2cm]{$\Sigma_2$};
\draw[ultra thick] (-1,0) node[below left=0.1cm]{$z_2$};
\draw (1,1) arc (0:180:1cm);
\draw (-1,-1) arc (180:360:1cm);
\draw [very thick] (-1,0) node[vertex] (v) {};
\draw [very thick] (1,0) node[vertex] (v) {};
\draw[ultra thick] (1,0) node[below right=0.1cm]{$z_1$};
\draw [very thick](0.1,0) node[vertex,label=below:{$x_0 $}] (v) {};
 \end{tikzpicture}
\caption{The domain $\Omega$.}
 \label{fig:domaine}
 \end{center}
\end{figure}
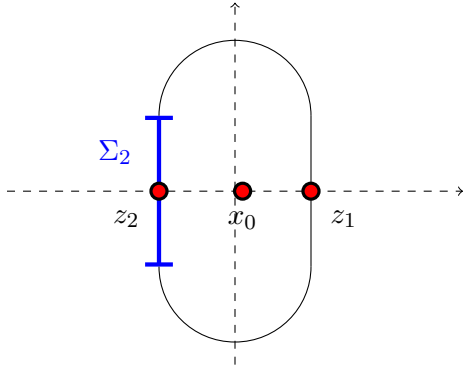

Let us consider the segment $\Sigma_2$ joining the two points
$(-1,-1)$ and $(-1,1)$. This subset of $\partial \Omega$ contains
the highest saddle point $z_2$ and is included in $B_{z_2}$. From
Theorem~\ref{TBIG0}, we expect that, in the limit $h \to 0$,
\begin{align*}
\P_{\nu_h}\left[X_{\tau_{\Omega}} \in \Sigma_2\right] = \exp \left(G\left(\frac{2}{h}\right)\right) (1+O(h))
\end{align*}
where
\begin{align*}
G\left(x\right) &=\ln\left[\frac{\partial_n f(z_2) \sqrt{ {\rm det \ Hess } f|_{\partial \Omega} 
    (z_1) } }{\partial_n f(z_1) \sqrt{ {\rm det \ Hess }
    f|_{\partial {\Omega} }   (z_2) }}\right]  -x\, (f(z_2)-f(z_1)).
\end{align*}
The function $G$ is compared for various value sod $h$  to
the numerically estimated function $F$ defined by  $F\left(\frac{2}{h}\right)=\ln
\left( \P_{\nu_h}\left[X_{\tau_{\Omega}} \in \Sigma_2\right]\right)$.  In
practice, the quasi stationary distribution~$\nu_h$ is sampled using a Fleming-Viot particle system (the convergence
diagnostics is based on  a Gelman-Rubin statistics,
see~\cite{binder-lelievre-simpson-15}) composed of $10^{5}$
particles. The probability $\mathbb P_{\nu_h}(X_{\tau_{\Omega}} \in
\Sigma_2)$ is estimated using a Monte Carlo procedure using $6\times
10^{5}$ particles distributed according to the quasi stationary
distribution $\nu_h$. The dynamics~\eqref{eq:langevin} is discretized in
time using an Euler-Maruyama scheme with a timestep $\Delta t$ which
is made precise in the captions of the figures. On
Figures~\ref{fig:res1} and~\ref{fig:res2}, we observe an excellent agreement between the theory and the numerical results.
\begin{figure}[h!]
\centering
\includegraphics[height=5.5cm]{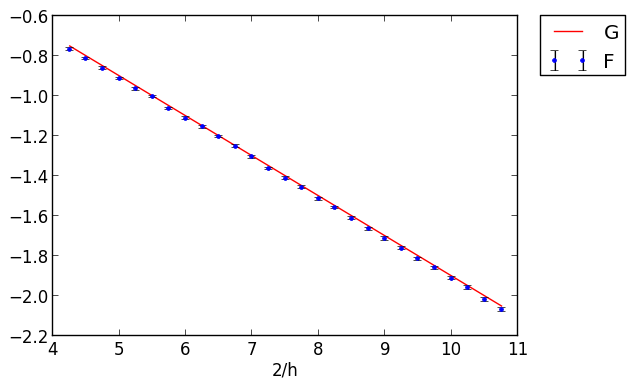}
  \caption{Logarithm of the  probability $\mathbb P_{\nu_h}\left(X_{\tau_{\Omega}} \in \Sigma_2\right)$ as a function of $\frac 2h$:
 comparison of the theoretical result function ($G$) with the numerical
    result (function $F$, $\Delta t=5.10^{-3}$); $a=1/10$.}
  \label{fig:res1}
\end{figure}

\begin{figure}[h!]
\centering
\includegraphics[height=5.5cm]{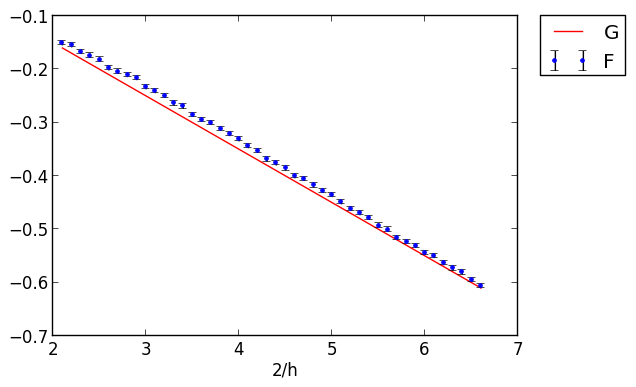}
  \caption{Logarithm of the  probability $\P_{\nu_h}(X_{\tau_{\Omega}} \in \Sigma_2)$ as function of $\frac 2h$ as a function of $\frac 2h$:
    comparison of the theoretical result function ($G$) with the numerical
    result (function $F$, $\Delta t=2.10^{-3}$);  $a=1/20$.}
  \label{fig:res2}
\end{figure}
\medskip

\noindent
Now, the potential function $f$ is modified such that the
assumption~\eqref{hypo1} is not satisfied anymore. More precisely, the potential
function is 
$$f(x,y)=\left(y^2-2 \ a(x) \right)^3,$$
 with  $a(x)=a_1x^2+b_1x+ 0.5$
where $a_1$ and $b_1$ are chosen such that 
$a(-1+\delta)=0$, $a(1)=1/4$  for $\delta=0.05$. We have $f(z_1)=-1/8$ and
$f(z_2)=-8 (a(-1))^3> 0 > f(z_1)$. Moreover, two 'corniches' (which are in the level set $f^{-1}(\{0\})$ of $f$, and on which $|\nabla f|=0$) on
the 'slopes of the hills' of the potential $f$ join the point $(-1+\delta,0)$ to
$B_{z_2}^c$ (at the points $(1,-1/\sqrt 2)\in B_{z_2}^c$ and $(1,1/\sqrt 2)\in B_{z_2}^c$) so that $\inf_{z\in B_{z_2}^c}
d_a(z,z_2) < f(z_2)-f(z_1)$. Indeed, in that case assumption~\eqref{hypo1} is not
satisfied since 
\begin{align*}
\inf_{z\in B_{z_2}^c} d_a(z,z_2)&\leq d_a\left(z_2, (1,1/\sqrt 2)\right)\\
&\leq d_a\left(z_2, (0,-1+\delta)\right)+ d_a\left((0,-1+\delta), (1,1/\sqrt 2)\right)\\
&=f(z_2)-f(0,-1+\delta)+0\\
&=f(z_2) < f(z_2)-f(z_1).
\end{align*}
%The justification of the previous computation can be found in Section
%\ref{agmonproperty}.
Notice that the Hessians $(\Hess \ f|_{\partial
  \Omega})(z_1)$ and $(\Hess \  f|_{\partial
  \Omega})(z_2)$ are nonsingular. The
functions $f_{|\Omega}$ and $f|_{\partial \Omega}$ are not Morse functions, but an arbitrarily  small perturbation (which we neglect here)  turns them into Morse functions. When
comparing the numerically estimated probability $\P_{\nu_h}(X_{\tau_{\Omega}} \in \Sigma_2)$, with
the theoretical asymptotic result in the limit $h\to 0$, we observe a discrepancy on the
prefactors, see Figure~\ref{fig:res3}.

\begin{figure}[!h]
\centering
\includegraphics[height=5.5cm]{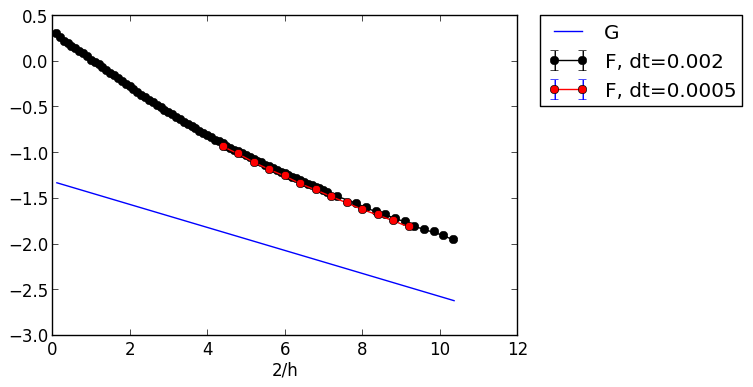}
  \caption{Logarithm of the  probability $\P_{\nu_h}(X_{\tau_{\Omega}} \in \Sigma_2)$ as a function of $\frac 2h$:
      comparison of the theoretical result function ($G$) with the numerical
    result (function $F$, $\Delta t=2.10^{-3}$ and $\Delta t=5.10^{-4}$).}
  \label{fig:res3}
\end{figure}

Therefore, it seems that assumption~\eqref{hypo1} is indeed required to get
an accurate description of the dynamics by the jump Markov process
using the Eyring-Kramers law to estimate the rates between the neighboring
states.
\medskip

\noindent
\underline{On the geometric assumptions~\eqref{hypo2} and~\eqref{eq:hypo2_bis}.}
\medskip

\noindent
To discuss the necessity of the assumptions \eqref{eq:hypo2_bis} in Corollary~\ref{cc1} and~\eqref{hypo2} in Corollary~\ref{co.proba-sigmai}, we consider a one-dimensional case, where the law of  $X_{\tau_\Omega}$ when $X_0=x$ has an explicit expression. Let $f:\mathbb R\to \mathbb R$ be $C^{\infty}$ and let $z_1,z_2\in \mathbb R$ such that $z_1<z_2$. Let us assume that $f'(z_1)<0$, $f'(z_2)>0$, $f(z_1)<f(z_2)$ and $f$ has only one critical point in $(z_1,z_2)$ denoted by $x_0$. This implies in particular that $f(x_0)=\min_{[z_1,z_2]}f<f(z_1)$. Moreover let us assume that $f''(x_0)>0$. Therefore, the hypotheses \textbf{[H1]}-\textbf{[H2]}-\textbf{[H3]} hold. For $x\in [z_1,z_2]$, let us denote by $w_h(x)=\P_{x}[X_{\tau_{(z_1,z_2)}}=z_2]$. It is standard that using a Feynman-Kac formula, $w_h$ solves the elliptic boundary value poblem
$$\frac{h}{2} w_h'' - w_h'f'=0, \ w_h(z_1)=0,\ w_h(z_2)=1.$$ 
Therefore, one has for $x\in [z_1,z_2]$: 
$$w_h(x)=\frac{ \displaystyle{\int_{z_1}^x e^{\frac 2h f}}}{\displaystyle{\int_{z_1}^{z_2} e^{\frac 2h f}}}.$$ Let $x\in [z_1,z_2]$. Using Laplace's method, if $f(x)<f(z_1)$, one obtains in the limit $h\to 0$:
$$\P_{x}[X_{\tau_{(z_1,z_2)}}=z_2]=-\frac{f'(z_2)}{f'(z_1)}e^{-\frac 2h (f(z_2)-f(z_1))}(1+O(h)),$$
if $f(x)=f(z_1)$, $x\neq z_1$,  it holds in the limit $h\to 0$:
$$\P_{x}[X_{\tau_{(z_1,z_2)}}=z_2]=f'(z_2)\left (\frac{1}{f'(x)}-\frac{1}{f'(z_1)}\right )e^{-\frac 2h (f(z_2)-f(z_1))}(1+O(h)),$$
and if $f(x)>f(z_1)$, it holds in the limit $h\to 0$:
$$\P_{x}[X_{\tau_{(z_1,z_2)}}=z_2]=\frac{f'(z_2)}{f'(x)}e^{-\frac 2h (f(z_2)-f(x))}(1+O(h)).$$
Therefore, in dimension one, the estimate~\eqref{eq.pxx} holds if and only if $x\in f^{-1}( (-\infty, f(z_1)))$. In accordance with Corollary~\ref{cc1}, the asymptotic~\eqref{eq.pxx} only holds for initial conditions which are sufficiently low in energy. However, we observe that in this simple one-dimensional setting, the assumption~\eqref{eq:hypo2_bis} is not needed. We do not know if the result of Corollary~\ref{cc1} would hold in general without the assumption~\eqref{eq:hypo2_bis}.

Let us now discuss the assumption~\eqref{hypo2} in the framework of Theorem~\ref{TBIG0} and Corollary~\ref{co.proba-sigmai}. 
From \eqref{eq:expQSD}, one has:
$$\P_{\nu_h}[X_{\tau_{(z_1,z_2)}}=z_2]=\frac{\displaystyle{ \int_{z_1}^{z_2} u_h w_h e^{-\frac 2h f}}}{\displaystyle{\int_{z_1}^{z_2} u_he^{-\frac 2h f}}}.$$
Using Lemma~\ref{nablauu}, Lemma~\ref{pi0u} and~\eqref{eq:uh}, one has  for some $c>0$, for any $\delta>0$ and for $h$ small enough: 
$$u_h(x)= \frac{\chi(x)}{\sqrt{\int_{z_1}^{z_2} \chi^2 e^{-\frac 2h f}}}(1+\alpha_h)+r(x),\ {\rm for}\ x\in \overline \Omega$$
 with $\alpha_h\in \mathbb R$, $\alpha_h=O(e^{-\frac ch})$,  $\int_{z_1}^{z_2} r^2 e^{-\frac 2h f}=O(e^{-\frac 2h (f(z_1)-f(x_0)-\delta)})$ and where $\chi\in C^{\infty}_{c}(z_1,z_2)$ is given by Lemma~\ref{nablauu}.  Therefore, one has:
 \begin{align*}
 \P_{\nu_h}[X_{\tau_{(z_1,z_2)}}=z_2]&=\frac{1}{\int_{z_1}^{z_2} u_he^{-\frac 2h f}}\left [ \frac{ \int_{z_1}^{z_2} \chi (x) \int_{z_1}^{x} e^{\frac 2h (f(y)-f(x))}dydx }{  \int_{z_1}^{z_2} e^{\frac 2h f}\sqrt{\int_{z_1}^{z_2} \chi^2 e^{-\frac 2h f}}} (1+\alpha_h) + \int_{z_1}^{z_2} rw_he^{-\frac 2h f}  \right].
 \end{align*}
Using Proposition~\ref{moyenneu} and Laplace's method, one gets for any $\delta>0$, in the limit $h\to 0$:
$$\P_{\nu_h}[X_{\tau_{(z_1,z_2)}}=z_2]=-\frac{f'(z_2)}{f'(z_1)}e^{-\frac 2h (f(z_2)-f(z_1))}(1+O(h))+O(e^{-\frac 1h (f(z_2)-f(x_0)+f(z_1)-f(x_0)-\delta)}).$$ 
Therefore, the result of Corollary~\ref{co.proba-sigmai} holds if 
\begin{equation} \label{eq.st1}
2(f(z_1)-f(x_0))>f(z_2)-f(z_1).
\end{equation}
This explicit computation in dimension one shows that the result of Corollary 1 indeed requires an assumption of the type: the height of the energy barrier to leave the well $f(z_1)-f(x_0)$ is sufficiently large compared to the largest difference in energy of the saddle points $f(z_2)-f(z_1)$. Notice that~\eqref{eq.st1} differs from (21) by a multiplicative factor~$\frac 12$. We do not know if the result of Corollary~\ref{co.proba-sigmai} would hold in general under the weaker assumption~\eqref{eq.st1}. Finally, let us mention that when $d=1$, \eqref{hypo1} is always satisfied.

\subsubsection{Extension of the results to a subset of generalized saddle points and to more general subsets of
  $\partial \Omega$}\label{sec:gen_sigma}

It is actually possible to generalize the result of Theorem~\ref{TBIG0} and
Corollary~\ref{co.proba-sigmai} to less stringent conditions than~\eqref{hypo1}-\eqref{hypo2} and to more general subsets $\Sigma \subset \partial \Omega$.

\begin{theorem} \label{th.gene_sigma}
Assume that \textbf{[H1]}, \textbf{[H2]} and \textbf{[H3]} hold. Assume that there exist $k_0\in\{1,\dots,n\}$ and $f^*\in \mathbb R$ such that $f(z_{k_0}) \le f^*\le f(z_{k_0+1})$ 
(with the convention $f(z_{k_0+1})=+\infty$ if $k_0=n$),
\label{page.fstar}
\begin{equation}
\label{eq.hyp.gene.1}
\left\{
\begin{aligned}[c]
&\forall i\in \{1,\ldots,k_0\}, \, \, \inf_{z\in B_{z_i}^c} d_a(z,z_i) >\max[f^*-f(z_i),f(z_i)-f(z_1)],  \\
&\forall i\in \{k_{0}+1,\ldots,n\},\, \, \inf_{z\in B_{z_i}^c} d_a(z,z_i) >f^*-f(z_1), 
\end{aligned}
\right. 
\end{equation}
and, 
\begin{equation} \label{eq.hyp.gene.2}
f(z_1)-f(x_0)>f^*-f(z_1).
\end{equation}
Let $u_h$ be the eigenfunction associated with the smallest eigenvalue of $-L^{D,(0)}_{f,h}(\Omega)$ (see Proposition~\ref{prop:QSD_spec}) which satisfies \eqref{eq.u_norma0}.
\begin{enumerate}
\item  For all $i\in \{1,\ldots,k_0\}$ and for all smooth open set $\Sigma_i \subset \partial \Omega$ containing $z_i$ and such that $\overline\Sigma_i \subset B_{z_i}$, the limit \eqref{eq:dnuh_asymptot} holds for $\int_{\Sigma_i}  ( \partial_n u_h)\,   e^{- \frac{2}{h}  f} d\sigma$ and the limit \eqref{eq:limproba} holds for $\P_{\nu_h} \left[ X_{\tau_{\Omega}} \in \Sigma_i\right]$. Moreover, if $f(z_{k_0+1})>f(z_{k_0})$, for all $i\in \{k_0+1,\ldots,n\}$ and for all smooth open set $\Sigma_i \subset \partial \Omega$ containing $z_i$ and such that $\overline\Sigma_i \subset B_{z_i}$, there exist $\ve >0$ and $h_0>0$ such that for all $h\in (0,h_0)$ 
\begin{equation}  \label{eq.dnuh_expch}
 \int_{\Sigma_i}   (\partial_{n}u_h)\, e^{-\frac{2}{h}f} d\sigma =\left (\int_{\Sigma_{k_0}}   \   (\partial_{n}u_h)  \,   e^{-\frac{2}{h}f} d\sigma \right) O\left ( e^{-\frac{\ve}{h}} \right),
  \end{equation}
  and \begin{equation} \label{eq.expch}
  \P_{\nu_h} \left[ X_{\tau_{\Omega}} \in \Sigma_i\right]=\P_{\nu_h} \left[ X_{\tau_{\Omega}} \in \Sigma_{k_0}\right] O \left( e^{-\frac{\ve}{h}}  \right)
    \end{equation}
\item Let $j_0 \in \{1, \ldots, k_0\}$ and $\Sigma\subset \partial \Omega$ be a smooth open set such that $\overline \Sigma\subset B_{z_{j_0}}$ and $\inf_{\Sigma} f=f^*$. Let $(B^*,p^*)\in \mathbb R_+^*\times \mathbb R$ be such that 
 \begin{equation}\label{eq:DL}
\int_{\Sigma} (\partial_nf) \, e^{-\frac{2}{h} f } d\sigma = B^* \ h^{p^*}
\ e^{-\frac{2}{h} f^*}\left( 1+ O(h) \right).
\end{equation}
\label{page.Bstar}
Then, one obtains in the limit $h\to 0$
 \begin{equation}\label{eq.gene_dnuh}
\int_{\Sigma}    (\partial_n u_h)\,   e^{- \frac{2}{h} f}  \   d\sigma=-\frac{ 2B^* \left({\rm det \ Hess }f(x_0)\right)^{\frac 14}}{\pi^{\frac d4}}h^{p^*-\frac d4-1}e^{-\frac{1}{h}\left(2f^*-f(x_0)\right)}\left( 1+ O(h) \right)
 \end{equation}
 and
 \begin{equation}\label{eq.gene_sig}
 \P_{\nu_h} [ X_{\tau_{\Omega}} \in \Sigma]=\frac{ B^*}{\pi^{\frac{d-1}{2}}} \left ( \sum_{k=1}^{n_0} \frac{ \partial_nf(z_k) }{\sqrt{ {\rm det \ Hess } f_{|\partial \Omega }   (z_k) } } \right)^{-1}h^{p^*-\frac{d-1}{2}} e^{-\frac{2}{h}\left( f^*-f(z_1)\right)}\left( 1+ O(h) \right).
 \end{equation}
 \end{enumerate}

\end{theorem}

In practice, the expansion~\eqref{eq:DL} is given by  
Laplace's method. Theorem~\ref{th.gene_sigma} is a generalization of Theorem~\ref{TBIG0}. Indeed,~\eqref{eq.hyp.gene.1}-\eqref{eq.hyp.gene.2} is weaker than \eqref{hypo1}-\eqref{hypo2} (\eqref{hypo1}-\eqref{hypo2} implies \eqref{eq.hyp.gene.1}-\eqref{eq.hyp.gene.2} for $k_0=n$ and $f^*=f(z_n)$) and item~2 gives an asymptotic result on the exit probability through $\Sigma \subset B_{z_{j_0}}$ even if $z_{j_0} \not\in \overline{\Sigma}$.

As an illustration, let us state a corollary of this theorem, which
demonstrates the importance of obtaining a precise asymptotic result including the prefactors.  Let us consider a simple
situation with only two local
minima $z_1$ and $z_2$ on the boundary, with $f(z_1) <
f(z_2)$. Let us now compare the two exit probabilities (see
Figure~\ref{fig:compa} for a schematic representation of the geometric
setting):
\begin{itemize}
\item The probability to leave
through $\Sigma_2$ such that $\overline{\Sigma_2} \subset B_{z_2}$
and  $z_2 \in \Sigma_2$;
\item The probability to leave through $\Sigma$ such that
  $\overline{\Sigma} \subset B_{z_1}$ and
$\inf_{\Sigma} f=f(z_2)$.
\end{itemize}
By classic results from the large deviation theory (see for
example~\eqref{eq:LD}) the probability to exit through $\Sigma$ and
$\Sigma_2$ both scale like a prefactor times ${\rm
  e}^{-\frac{2}{h}(f(z_2)-f(z_1))}$:  the difference can only be read from the
prefactors. Actually, using item 2 in Theorem~\ref{th.gene_sigma}, one obtains the existence of $C>0$ such that 
in the limit $h \to 0$ (see Corollary~\ref{co.gene_sigma} below),
\begin{equation}\label{eq:exit_sigma}
\frac{\P_{\nu_h}(X_{\tau_\Omega} \in \Sigma)}{\P_{\nu_h}(X_{\tau_\Omega} \in \Sigma_2)} \sim C
\sqrt{h} .
\end{equation}
The probability to leave through $\Sigma_2$ (namely through the generalized saddle point $z_2$)
is thus  larger than through $\Sigma$, even though the two regions
are at the same height.
This result explains why the local minima
of $f$ on the boundary (namely the generalized saddle points) play
such an important role when studying the exit event. Let us now state
the precise result.
\begin{corollary} \label{co.gene_sigma}
Assume \textbf{[H1]}, \textbf{[H2]},  \textbf{[H3]}. Assume that
$f|_{\partial \Omega}$ has only two local minima $z_1$ and $z_2$ such that $f(z_1)<f(z_2)$ and,   
\begin{equation} \label{hypo11_bis} 
\text{   for $j\in \{1,2\}$, }  \inf_{z\in B_{z_j}^c} d_a(z,z_j) >f(z_2)-f(z_1),
\end{equation} 
 and
 \begin{equation}\label{hypo22_bis}
 f(z_1)-f(x_0)>f(z_2)-f(z_1).
 \end{equation}
Let $\Sigma \subset \partial \Omega$ be a smooth open set such that
$\overline{\Sigma} \subset B_{z_1}$. Assume moreover that $\inf_{\Sigma}
f=f(z_2)$ and that the infimum is attained at a single point $z^*$:
$\inf_{\Sigma}f=f(z^*)$ (necessarily $z^*\in \partial
\Sigma$). Finally, let us assume that $z^*$ is a non degenerate
minimum  of $f_{|\partial \Sigma }$ and $\partial_{n(\partial
  \Sigma)} f_{|\partial \Sigma}(z^*)<0$ where $n(\partial \Sigma)$ is the unit outward
normal to $\partial \Sigma\subset \partial \Omega$. Then, one has the following asymptotic expansion of $\P_{\nu_h}
\left[ X_{\tau_{\Omega}} \in \Sigma\right]$  in the limit $h \to 0$:
\begin{align*}
         \P_{\nu_h} \left[ X_{\tau_{\Omega}} \in \Sigma\right]    \  
        & =-\frac{\sqrt{h}}{2\sqrt \pi} \frac{\partial_{n}f(z^*)}{\partial_{n(\partial \Sigma)}f(z^*)     \sqrt {{\rm det \ Hess } f_{|\partial \Sigma }   (z^*) }}  
        \left ( \sum_{k=1}^{n_0} \frac{ \partial_nf(z_k) }{\sqrt{ {\rm det \ Hess } f_{|\partial \Omega }   (z_k) } } \right)^{-1}  \\
         &\quad \times e^{-\frac{2}{h}(f(z_2)-f(z_1))} 
(  1     +      
 O(h)   ),        
\end{align*}
\label{page.zstar}
with by convention, ${\rm det \ Hess } f_{|\partial \Sigma }   (z^*) =1$ if $d=2$.
%     $$
%         \P_{\nu_h} \left[ X_{\tau_{\Omega}} \in \Sigma\right]    \  
%         =-\frac{1}{2\sqrt \pi} \frac{\partial_{n}f(z^*)}{\partial_{n}f(z_1)\partial_{n(\partial \Sigma)}f(z^*)} \sqrt{  {\rm det \ Hess } f_{|\partial \Omega }   (z_1)  }  
%         \sqrt{h} \ e^{-\frac{2}{h}(f(z_2)-f(z_1))} 
%(  1     +      
% O(h)   ).  
%    $$
   \end{corollary}
Corollaries \ref{co.proba-sigmai}, \ref{cc1} and \ref{co.gene_sigma} imply the result~\eqref{eq:exit_sigma} announced above.
\begin{remark}
By using Laplace's method, one can check that the asymptotic results obtained in Corollaries \ref{co.proba-sigmai}, \ref{cc1} and \ref{co.gene_sigma} on the law of $X_{\tau_\Omega}$  imply that the density of $X_{\tau_\Omega}$ with respect to the Lebesgue measure on $\partial \Omega$ is, in the limit $h \to 0$,
\begin{equation}\label{eq:approx_exit_density}
z \in \partial \Omega \mapsto  \frac{ \partial_nf(z) \, e^{-\frac{2}{h} f(z)} }{\int_{\partial \Omega} \partial_nf \, e^{-\frac{2}{h} f } d\sigma } (1+O(h)).
\end{equation}
This indeed yields  the same asymptotic limits on the exit distribution. This is reminiscent of previous results obtained in  \cite{kamin-78,day1987r, perthame-90}, where the authors proved that, starting from a deterministic initial condition in $\Omega$, $X_{\tau_\Omega}$ has a density with respect to the Lebesgue measure on $\partial \Omega$ which satisfies, in the limit $h \to 0$,
$z \in \partial \Omega \mapsto  \frac{ \partial_nf(z) \, e^{-\frac{2}{h} f(z)} }{\int_{\partial \Omega} \partial_nf \, e^{-\frac{2}{h} f } d\sigma } + o(1)$, which is however a less precise estimate than~\eqref{eq:approx_exit_density}.
\end{remark}

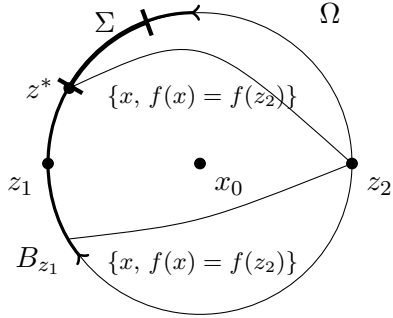
\begin{figure}
\begin{center}
\begin{tikzpicture}
\tikzstyle{vertex}=[draw,circle,fill=black,minimum size=4pt,inner sep=0pt]

%%%%
\draw (2,2) node[label=right:{$\!\!\!\!\!\!\!\!\!\!\!\!\!\!\Omega$}] (v) {};
\draw[] (2,0) arc (0:360:2cm) ;
\draw[very thick, >-<] (0,2) arc (90:220:2cm) node[label=left:{$B_{z_1}$}] (v) {};
\draw[ultra thick, |-|] (-1.72,1) arc (150:110:2.cm) node[label=left:{$\Sigma~$}] (v) {};

%\draw (0.3,8) node[below] { $x_d$};
\draw (0 ,0) node[vertex,label=south east: {$x_0$}](v){};
\draw (-2 ,0) node[vertex,label=south west: {$z_1$}](v){};
\draw (2,0 ) node[vertex,label=south east: {$z_2$}](v){};
\draw (-1.72 ,1 ) node[vertex,label=west: {$z^*$}](v){};
\draw (2,0) ..controls (0,1.8) .. (-1.72 ,1 )
node[midway,below=0.3cm]{\footnotesize $\{x, \, f(x)= f(z_2)\}$} ;
\draw (2,0) ..controls (0,-0.8) .. (-1.72 ,-1 )
node[midway,below=0.3cm]{\footnotesize $\{x, \, f(x)= f(z_2)\}$} ;
\end{tikzpicture}
\caption{Schematic representation of the geometric setting of
  Corollary~\ref{co.gene_sigma}. The subset~$\Sigma$ is such that
  $\overline{\Sigma} \subset B_{z_1}$ and $\inf_{\Sigma} f=f(z_2)$.}
 \label{fig:compa}
 \end{center}
\end{figure}

% Let us notice that exponential rate $f(z_2)-f(z_1)$ is in accordance with the Freidlin  Wentzel theory (see \cite{freidlin-wentzell-84}). However, the result stated in \cite{freidlin-wentzell-84} can actually only predict the limit  $\lim_{h\to 0} h\ln\left(\mathbb P_x\left[ X_{\tau_{\Omega}} \in \Sigma\right]\right)=2\left(f(z_2)-f(z_1)\right)$, for $x\in \Omega$. Therefore, it is not possible to determine where the process (\ref{eq:langevin}) is more likely to leave $\Omega$ at a level set of $f$. A natural question arises when one is interested in comparing the exit around $z_2$ or around $B_{z_1}\cap \{x\in \partial \Omega, f(x)= f(z_2)\}$. Why $z_2$ should be more relevant than $z^*$ (see  Corollary~\ref{co.gene_sigma}) for the first exit event from $\Omega$?\\
% This is why one stated Corollary~\ref{co.gene_sigma} which answers the question. The exit is more likely to occur, in the small temperature regime, around $z_2$ than on $\Sigma$ (see figure \ref{fig:compa}). Indeed according to Corollary~\ref{co.gene_sigma} and Theorem $\ref{TBIG0}$, for any open set $\Sigma_2\subset \partial \Omega$ such that $z_2\in \Sigma_2$,  when $h\to 0$
% $$\P_{\nu_h} \left[ X_{\tau_{\Omega}} \in \Sigma\right]=O\left( h^{1/2} \P_{\nu_h} [ X_{\tau_{\Omega}} \in \Sigma_2]\right).$$
% And this comparison can only be made using the prefactor.

\subsection{Strategy for the proof of Theorem~\ref{TBIG0} and outline of the paper} \label{sec:sketch}

The aim of this section is to give an overview of the strategy for the
proof of Theorem~\ref{TBIG0}. In view of~\eqref{eq:dnuh_asymptot}, we would like to
identify the asymptotic behavior of the normal derivative $\partial_n
u_h$ on $\partial \Omega$ in the limit $h \to 0$. We recall that $(\lambda_h,u_h)$ satisfies  the
eigenvalue problem~\eqref{eq:u}. By differentiating~\eqref{eq:u},
we observe that $\nabla u_h$
satisfies
\begin{equation}\label{eq:L1_eig}
\left\{
\begin{aligned}
L^{(1)}_{f,h} \nabla u_h &= - \lambda_h \nabla u_h \text{ on $\Omega$,}\\
\nabla_T u_h & = 0 \text{ on $\partial \Omega$,}\\
\left(\frac{h}{2} {\rm div} - \nabla f \cdot \right) \nabla u_h & = 0 \text{ on $\partial \Omega$,}\\
\end{aligned}
\right.
\end{equation}
where
\begin{equation}\label{eq:L1}
L^{(1)}_{f,h}= \frac{h}{2} \Delta - \nabla f \cdot \nabla - {\rm
  Hess} \, f
\end{equation}
is an operator acting on $1$-forms (namely on vector fields).
Therefore $\nabla u_h$ is an eigen-$1$-form 
of  the operator $-L^{(1)}_{f,h}$ with tangential Dirichlet boundary
conditions (see~\eqref{eq:L1_eig}), associated with the small eigenvalue
$\lambda_h$. 

It is known (see for
example~\cite{helffer-nier-06}) that in our geometric setting, $-L^{D,(0)}_{f,h}(\Omega)$ admits exactly one
eigenvalue smaller than $\frac{\sqrt{h}}{2}$, namely $\lambda_h$ with
associated eigenfunction $u_h$ (this is because~$f$ has only one local
minimum in $\Omega$) and that $-L^{D,(1)}_{f,h}(\Omega)$ admits exactly $n$
eigenvalues smaller than $\frac{\sqrt{h}}{2}$  (where, we recall, $n$ is the
number of local minima of~$f$ on $\partial \Omega$). Actually, all these small
eigenvalues are exponentially small in the regime $h \to 0$,
the other eigenvalues being bounded from below by a constant in this regime.
The idea is then to construct an appropriate basis (with so called
quasi-modes, which are localized on the generalized saddle points $(z_i)_{i =1, \ldots ,n}$) of the eigenspace associated with small eigenvalues for $L^{D,(1)}_{f,h}(\Omega)$, and then
to decompose $\nabla u_h$ along this basis.

The article is organized as follows. In
Section~\ref{sec:gram_schmidt}, we introduce the general setting for
the proof of our results, and the Gram-Schmidt procedure which allows,
starting from a set of quasi-modes, to compute the projection of
(an approximation of) $\nabla u_h$ along the quasi-modes. In order to  quantify the distance between the space
spanned by these quasi-modes and the eigenspace of $L^{D,(1)}_{f,h}(\Omega)$
associated with small eigenvalues, we
need to use so-called Agmon estimates. Section~\ref{agmonproperty} is
devoted to a presentation of the main properties of the Agmon distance
which intervenes in these estimates. The most technical part is the
effective construction of the quasi-modes using auxiliary simpler
eigenvalue problems associated with each of the local minima
$(z_i)_{i =1, \ldots ,n}$. This is explained in
Section~\ref{sec:quasi-modes} which concludes the proof of Theorem \ref{TBIG0}. Finally, Section~\ref{sec:proofs}
concludes the paper by providing the proofs of all the other results stated above, in particular Theorem~\ref{th.gene_sigma}.
\medskip

\noindent
For the ease of the reader, a list of the main notation used in this work is provided at the end of this work.

\section{General setting and strategy for the proof of Theorem~\ref{TBIG0}}\label{sec:gram_schmidt}

The general setting for proving the results presented in
Section~\ref{intro} will be the following: $\overline \Omega$ is a
$C^{\infty}$ oriented connected  compact  Riemannian  manifold of
dimension $d$ with boundary $\partial \Omega$ and the function $f$ is
a $C^{\infty}$ real valued function defined on $\overline \Omega$. One defines
$\Omega:=\overline{\Omega}\setminus \partial \Omega$. In particular,
Theorem~\ref{TBIG0} will actually be proven in this framework. Notice that the assumptions
\textbf{[H1]}, \textbf{[H2]} and \textbf{[H3]}  are
still meaningful in this more general setting. 

% \subsubsection{List of Hypotheses in a more general setting}\label{hypob}
% \textbf{It will be convenient to define a more general set of hypotheses since the proofs of Theorem~\ref{TBIG0}, Proposition~\ref{moyenneu}, Corollary~\ref{co.proba-sigmai}, Theorem~\ref{th.gene_sigma} and Corollary~\ref{co.gene_sigma} will be actually made for $C^{\infty}$ compact manifolds with boundary.}\\

% Let $\overline \Omega$ be a $C^{\infty}$ oriented connected  compact  Riemannian manifold of dimension $d$ with boundary $\partial \Omega$ and let $f : \overline \Omega \to \mathbb R$ be a $C^{\infty}$ function. \\ 

% One assumes throughout this work that the function $f : \overline \Omega\to \mathbb R$ and the restriction of $f$ to the boundary of $\Omega$, denoted $f|_{\partial \Omega}$ have a finite number of critical points.\\

% \textbf{[H1].} The function $f : \overline \Omega \to \mathbb R$ is a Morse function on $\Omega$ and the restriction of $f$ to the boundary of $\Omega$, denoted $f|_{\partial \Omega}$, is a Morse function.  The function $f$ does not have any critical point on $\partial \Omega$. \\

% \textbf{[H2].} The function $f$ has a unique critical point in $\Omega$. This critical point is denoted by $x_0$. The function $f|_{\partial \Omega}$ has exactly $n$ local minima denoted by $(z_i)_{i=1,\ldots,n}$ ($n\in \mathbb n^{\flat}$) such that $f(z_1)\leq f(z_2)\leq\ldots\leq f(z_n)$. Additionally $$\min_{\overline \Omega}f= \min_{\Omega}f=f(x_0).$$ 

% \comment{Les hypotheses 1b et 2b etaient exactement les memes que 1a et
%   2a !}

In order to use previous results from the literature on semi-classical
analysis, we will transform the original problem~\eqref{eq:u} on
$(\lambda_h,u_h)$ associated with weighted Hilbert space
$H^q_w(\Omega)$ to an eigenvalue problem on the standard (non-weighted) Hilbert spaces
$H^q(\Omega)$, by using a unitary transformation which relates the
operator $L^{(p)}_{f,h}$ to the Witten Laplacians $\Delta^{(p)}_{f,h}$. This is explained in Section~\ref{sec:witten}, together
with some first well-known results on the spectrum of Witten
Laplacians. Then, in Section~\ref{sec:quasi-modes_gram_schmidt}, we
explain what are the requirements on the quasi-modes we will build in
order to obtain the estimate~\eqref{eq:dnuh_asymptot}, see
Proposition~\ref{ESTIME}. Section~\ref{estime} is finally devoted to
the proof of Proposition~\ref{ESTIME}.

\subsection{Witten Laplacians} \label{sec:witten}
% \subsubsection{Motivation}
% This section is dedicated to introduce the framework used to prove Theorem~\ref{TBIG0}.  The function $u_h$ (see Proposition~\ref{prop:QSD_spec}) is an eigenfunction attached to the smallest eigenvalue $\lambda_h$ of the Dirichlet realization on $\Omega$ of the infinitesimal generator denoted $-L^{D,(0)}_{f,h}(\Omega)$ (see Proposition~\ref{fried}). Theorem~\ref{TBIG0} aims at estimating the gradient of $u$.  \\

% The vector field $\nabla u_h$ is an eigenvector of another Dirichlet realisation of an operator defined on vectors fields denoted $-L^{(1)}_{f,h}(\Omega)$ (see \ref{Lp}). In addition, one notices that $-L^{D,(0)}_{f,h}(\Omega)$ and the Dirichlet realization of the Witten Laplacian on $\Omega$ are unitary equivalent. It will be a choice to work with Witten Laplacians and this choice is justified by the recent work of Francis Nier and Bernard Helffer (see \cite{helffer-nier-06}) in which the spectrum has been investigated using tools from semi classical analysis. In \cite{helffer-nier-06}, Witten Laplacians are attached to a complex which  will turn out to be a convenient way to investigate the gradient of $u_h$. \\

% In this section 
 
\subsubsection{Notations for Sobolev spaces}
\label{notation}
For $p\in\{0,\ldots,d\}$, one denotes by
$\Lambda^pC^{\infty}(\overline{\Omega})$ the space of $C^{\infty}$ $p$-forms on
$\overline\Omega$\label{page.cinfty}. Moreover, $\Lambda^pC^{\infty}_T(\overline\Omega)$ is the set of
$C^{\infty}$ $p$-forms $v$ such that $\mbf tv=0$ on $\partial \Omega$, where $\mathbf t$ denotes the tangential trace on forms (see for instance~\cite[p. 27]{GSchw}). Likewise, the set $\Lambda^pC^{\infty}_N(\overline\Omega)$ is the set of $C^{\infty}$ $p$-forms $v$ such that $\mbf n v=0$ on $\partial \Omega$, where $\mathbf n$ denotes the normal trace on forms defined by: for all $w\in \Lambda^pC^{\infty}(\overline\Omega)$, $\mathbf n w:=w|_{\partial \Omega}-\mathbf t w$\label{page.cinftyt}. 

For $p\in \{0, \ldots,d\}$ and $q\in \mathbb N$, one denotes by
$\Lambda^pH^q_w(\Omega)$\label{page.wsobolevq} the weighted Sobolev spaces of $p$ forms with
regularity index $q$, for the weight $e^{-\frac{2}{h} f(x)}dx$ on
$\Omega$: $v\in \Lambda^pH^q_w(\Omega)$ if and only if for all multi-index $\alpha$ with $\vert \alpha \vert \le q$, the $\alpha$ derivative of $v$ is in $\Lambda^pL^2_w(\Omega)$ where $\Lambda^pL^2_w(\Omega)$ is the completion of the space $\Lambda^pC^{\infty}(\overline{\Omega})$ for the norm 
$$w\in \Lambda^pC^{\infty}(\overline{\Omega})\mapsto\sqrt{ \int_{\Omega} \vert w\vert^2e^{-\frac 2h f}}.$$
 See for example~\cite{GSchw} for an introduction to Sobolev spaces on manifolds with boundaries. For $p\in\{0,\ldots,d\}$ and $q> \frac{1}{2}$, the set $\Lambda^pH^q_{w,T}(\Omega)$ is defined by 
$$\Lambda^pH^q_{w,T}(\Omega):= \left\{v\in \Lambda^pH^q_w(\Omega)
  \,|\,  \mathbf  tv=0 \ {\rm on} \ \partial \Omega\right\}.$$
  \label{page.wsobolevqt}
Notice that $\Lambda^pL^2_w(\Omega)$ is the space
$\Lambda^pH^0_w(\Omega)$, and that $\Lambda^0H^1_{w,T}(\Omega)$ is
the space $H^{1}_{w,0}(\Omega)$ that we introduced in Proposition~\ref{fried} to
define the domain of $L^{D,(0)}_{f,h}(\Omega)$. Likewise for $p\in\{0,\ldots,d\}$ and $q> \frac{1}{2}$, the set $\Lambda^pH^q_{w,N}(\Omega)$ is defined by 
$$\Lambda^pH^q_{w,N}(\Omega):= \left\{v\in \Lambda^pH^q_w(\Omega)
  \,|\,  \mathbf  nv=0 \ {\rm on} \ \partial \Omega\right\}.$$
We will denote by $\Vert . \Vert_{H^q_w}$  the norm on the weighted space $\Lambda^pH^q_w
(\Omega)$. Moreover $\langle
\cdot , \cdot\rangle_{L^2_w}$ denotes the scalar product in $\Lambda^pL^2_w (\Omega)$.  \label{page.psl2w}\\
Finally, we will also use the same notation without the index $w$ to
denote the standard Sobolev spaces defined with respect to the
Lebesgue measure on $\Omega$ \label{page.psl2}.

\subsubsection{Witten Laplacians on a manifold with boundary}

Let us first recall some basic properties of Witten Laplacians, as
well as the link between those and the operators $L^{(p)}_{f,h}$
introduced above (see~\eqref{eq:L0} and~\eqref{eq:L1}). As already
explained above, we will actually need in this article to work only
with $0$ and $1$-forms ($p\in \{0,1\}$). For an introduction to the Hodge theory and the Hodge Laplacians on manifolds with boundary, one can refer to \cite{GSchw}. % A short appendix on Riemannian geometry is also given in \cite{lelievre-nier-15}. 
\medskip

\noindent
Denote by $d$ the exterior derivative on $\Omega$,
$$d^{(p)}: \Lambda^pC^{\infty}(\Omega )\to \Lambda^{p+1}C^{\infty}(\Omega),$$
and $(d^{(p)})^*$ its adjoint. The exterior derivative is $2$ nilpotent, 
$$d^{(p+1)}\circ d^{(p)} =0,$$
and therefore $(d^{(p)})^*\circ (d^{(p+1)})^* =0$. In all what
follows, the superscript $(p)$ may be
omitted when there is no ambiguity.

Let us now introduce the so called distorted exterior derivative 
\begin{equation}\label{eq:dfh}
d^{(p)}_{f,h}:=e^{-\frac{f}{h}}\, h\, d^{(p)}\, e^{\frac{f}{h}}=hd^{(p)}+df\wedge
\end{equation}
and its formal adjoint
\begin{equation}\label{eq:d*fh}
(d^{(p)}_{f,h})^*:=e^{\frac{f}{h}}\,h\,(d^{(p)})^* \,e^{\frac{-f}{h}}=h(d^{(p)})^*+\mathbf{i}_{\nabla f}.
\end{equation}
The distorted exterior derivative was firstly introduced by Witten in~\cite{witten-82}.
\begin{definition}
The Witten Laplacian is the non negative differential operator
$$\Delta^{(p)}_{f,h}:=\left(d^{(p)}_{f,h}+\left(d^{(p)}_{f,h}\right)^*\right)^2.$$
\end{definition}
\noindent
An \label{page.wlaplacien} equivalent  expression of the Witten Laplacians is 
$$\Delta^{(p)}_{f,h}=h^2 \Delta_H^{(p)} + \vert \nabla f\vert^2 + h\left(\mathcal L_{\nabla f} + \mathcal L^*_{\nabla f}\right),$$
where $\mathcal L$ stands for the Lie derivative, $\nabla$ is the covariant derivative associated to the metric   on $\Omega$ and $\Delta_H^{(p)}$ is the Hodge Laplacian acting on $p$-forms, defined by:
$$\Delta_H^{(p)}:= \left(d^{(p)}+\left(d^{(p)}\right)^*\right)^2.$$
We recall that $\Delta_H^{(p)}$ is a positive operator (in $\R^d$, $\Delta_H^{(0)}=-\sum_{i=1}^d \partial^2_{x_i,x_i}$).
The operator $\mathcal L_{\nabla f} + \mathcal L^*_{\nabla f}$ is an
operator of order $0$ (namely a multiplicative operator). On $0$-forms, namely on functions, the Witten Laplacian has the following expression 
\begin{equation}\label{eq:laplwitt0}
\Delta^{(0)}_{f,h}=h^2 \Delta_H^{(0)} + \vert \nabla f\vert^2 + h\,\Delta_H^{(0)}  f.
\end{equation}
Let us  now make precise the natural Dirichlet and Neumann boundary conditions  for Witten Laplacians on a manifold with boundary    (see~\cite{helffer-nier-06}).
\begin{proposition} \label{wittenquadra}
The Dirichlet realization of $\Delta^{(p)}_{f,h}$ on $\Omega$ is the
operator $\Delta^{D,(p)}_{f,h}(\Omega)$ with domain
$$D\left(\Delta^{D,(p)}_{f,h}(\Omega)\right)=\left\{  v \in
  \Lambda^pH^2(\Omega)  | \ \mathbf  t v=0, \ \mathbf  t d^*_{f,h} v=0
\right\}.$$
The Neumann realization of $\Delta^{(p)}_{f,h}$ on $\Omega$ is the
operator $\Delta^{N,(p)}_{f,h}(\Omega)$ with domain
$$D\left(\Delta^{N,(p)}_{f,h}(\Omega)\right)=\left\{  v \in
  \Lambda^pH^2(\Omega)  | \  \mathbf  n v=0, \  \mathbf n d_{f,h}v=0
\right   \}.$$
 The operators
$\Delta^{D,(p)}_{f,h}(\Omega)$ and $\Delta^{N,(p)}_{f,h}(\Omega)$ are both self adjoint operators with compact resolvent. 
\end{proposition}
\noindent
We \label{page.wlaplaciend} recall that ${\bf t}$ denotes the tangential trace on forms and
that ${\bf n}\omega=\omega - {\bf t} \omega$ is the normal trace. The
proof of Proposition~\ref{wittenquadra} can be found
in~\cite[Section 2.4]{helffer-nier-06} and in \cite[Section 4.2.3]{le-peutrec-10}. It is a generalization of what is stated in
\cite{GSchw} for the Hodge Laplacians. One can check that the operator $\Delta^{D,(p)}_{f,h}\left (\Omega\right )   $ is actually the Friedrichs extension associated to the quadratic form
\begin{equation}\label{eq:fried_Witten}
 v \in \Lambda^pH^1_T\left (\Omega\right ) \mapsto \left\Vert d^{(p)}_{f,h}v\right \Vert_{
    L^2}^2 +\left \Vert \left (d^{(p)}_{f,h}\right
   )^*v\right \Vert_{ L^2}^2.
\end{equation}
%Q_{f,h}^{(p)}\left   (\Omega\right )w:=
The following properties are easily checked for $v\in D(\Delta^{D,(p)}_{f,h}(\Omega))$ such that $d_{f,h}v\in D(\Delta^{D,(p+1)}_{f,h}(\Omega))$ and $d_{f,h}^*v\in D(\Delta^{D,(p-1)}_{f,h}(\Omega))$:
\begin{equation}\label{eq:comutation}
d_{f,h}\Delta^{D,(p)}_{f,h}(\Omega)v=\Delta^{D,(p+1)}_{f,h}(\Omega)d_{f,h}v
\text{ and } \  d_{f,h}^*\Delta^{D,(p)}_{f,h}(\Omega)v=\Delta^{D,(p-1)}_{f,h}(\Omega)d_{f,h}^*v.
\end{equation}
Similar relations hold for $\Delta^N_{f,h}(\Omega)$. 
\medskip

%The relations (\ref{eq:comutation}) allows us to define the Dirichlet complex structure (see \cite{chang-liu-95}, \cite{helffer-nier-06} and \cite{laudenbach2011morse}):
%$$\{0\} \longrightarrow D\left(\Delta^{D,(0)}_{f,h}(\Omega)\right) \xrightarrow{d_{f,h} } D\left(\Delta^{D,(1)}_{f,h}(\Omega)\right) \xrightarrow{d_{f,h} } \cdots D\left(\Delta^{D,(d)}_{f,h}(\Omega)\right) \xrightarrow{d_{f,h} }\{0\}$$
%$$\{0\} \xleftarrow{d_{f,h}^* } D\left(\Delta^{D,(0)}_{f,h}(\Omega)\right) \xleftarrow{d_{f,h}^* } D\left(\Delta^{D,(1)}_{f,h}(\Omega)\right) \cdots\xleftarrow{d_{f,h} ^*} D\left(\Delta^{D,(d)}_{f,h}(\Omega)\right) \longleftarrow \ \{0\}.$$
%One can define similarly the Neumann complex structure (see \cite{le-peutrec-10}).
\noindent
One can relate the infinitesimal generator $L^{(0)}_{f,h}$ of the diffusion
(\ref{eq:langevin}) to the Witten Laplacian $\Delta^{(0)}_{f,h}$
through the unitary transformation:
$$  \phi \in L^2_w(\Omega) \mapsto e^{-\frac{f}{h}} \phi \in L^2(\Omega).$$
Indeed, one can check that
\begin{equation}\label{eq:unitary_p0}
\Delta^{D,(0)}_{f,h}(\Omega) = -2\, h \, e^{-\frac{f}{h}} \,  L^{D,(0)}_{f,h}(\Omega)  \, e^{\frac{f}{h}}.
\end{equation}
Let us now generalize this to $p$-forms, using extensions of
$L^{(0)}_{f,h}$ to $p$-forms.
\begin{proposition} \label{Lp}
The Friedrichs extension associated with the quadratic form 
$$v \in \Lambda^pC^{\infty}_T(\Omega)\mapsto \frac{h}{2}
\left[\left\Vert d^{(p)}v\right\Vert_{  L^2_w(\Omega)}^2
  +\left\Vert
   e^{\frac{2f}{h}}  \left(d^{(p)}\right)^* e^{-\frac{2f}{h}} v\right\Vert_{ L^2_w(\Omega)}^2\right]$$ 
on $\Lambda^pL^2_w(\Omega)$, is denoted
$\left(-L^{D,(p)}_{f,h}(\Omega), \
  D\left(-L^{D,(p)}_{f,h}(\Omega)\right)\right)$. The operator $-L^{D,(p)}_{f,h}(\Omega)$ is a  positive unbounded self adjoint operator on $\Lambda^pL^2_w(\Omega)$. Besides, one has 
$$D\left(-L^{D,(p)}_{f,h}(\Omega)\right )=\left\{  v \in    \Lambda^pH^2_w(\Omega) \, | \, \mathbf  tv=0, \ \mathbf  t d^*\left(e^{-\frac{2f}{h}}v\right)=0   \right\}.$$
\end{proposition}
\label{page.generatord}
\noindent
For $p=0$, the differential operator 
$$L_{f,h}^{(0)} =  - \frac{h}{2} \Delta^{(0)}_H  -    \nabla f  \cdot \nabla $$
is the infinitesimal generator~\eqref{eq:L0} of the overdamped
Langevin dynamics~\eqref{eq:langevin}.
 For $p=1$ one gets the operator already introduced in~\eqref{eq:L1}:
\begin{equation} \label{generator1forms}  
L_{f,h}^{(1)}  =   - \frac{h}{2} \Delta^{(1)}_H  -    \nabla f  \cdot \nabla   -   \hess f  ,
\end{equation}
where we recall $\hess f $ is the Hessian of $f$, see Remark~\ref{re:Hessian}.
The generalisation of~\eqref{eq:unitary_p0} to $p$-forms is:
\begin{equation}\label{eq:unitary} 
\Delta^{D,(p)}_{f,h}(\Omega) = -2\, h \, e^{-\frac{f}{h}}\left( L^{D,(p)}_{f,h}(\Omega) \right)\, e^{\frac{f}{h}}.
\end{equation}
The intertwining relation (\ref{eq:comutation})
writes on $L_{f,h}^{D,(p)}(\Omega)$:
\begin{equation}\label{commutationL}
L_{f,h}^{D,(p+1)}(\Omega) d = d L_{f,h}^{D,(p)}(\Omega)
\text{ and } L_{f,h}^{D,(p-1)}(\Omega) d_{2f,h}^{*} = d^{*}_{2f,h} L_{f,h}^{D,(p)}(\Omega).
\end{equation}
Thanks to the relation~\eqref{eq:unitary}, the operators
$L_{f,h}^{D,(p)}(\Omega)$ and $\Delta^{D,(p)}_{f,h}(\Omega)$ have the
same spectral properties. In particular the operators
$L_{f,h}^{D,(p)}(\Omega)$ and $\Delta^{D,(p)}_{f,h}(\Omega)$ both have
compact resolvents, and thus a discrete spectrum. The generalization of
Proposition~\ref{prop:QSD_spec} is the following:   
 \begin{proposition} \label{UU} 
The smallest eigenvalue of $-L_{f,h}^{D,(0)}(\Omega)$, denoted by
 $\lambda_h$, is positive and non degenerate. The associated eigenfunction $u_h$ has sign on $\Omega$. Moreover $u_h \in C^{\infty}\left(\overline \Omega,\mathbb R\right)$.
\end{proposition}
\noindent
Without \label{page.uhlambdah} loss of generality, one can assume that $u_h$ satisfies 
\eqref{eq.u_norma0}. 
%The couple $(\lambda_h,u_h)$ satisfies the following relation
%\begin{equation} \label{eq:uu}
%\left\{
%\begin{aligned}
% -L^{(0)}_{f,h} u_h &=  \lambda_h \, u_h     \text{ on }  \Omega,  \\ 
%u_h&= 0  \text{ on  } \partial \Omega.
%\end{aligned}
%\right.
%\end{equation}
Thanks to the relation~\eqref{eq:unitary}, the couple
$\left(\mu_h,v_h\right):=\big(2h\lambda_h, u_h\,e^{-\frac{f}{h}} \big)$
is the first eigenvalue and eigenfunction of
$\Delta^{D,(0)}_{f,h}(\Omega)$. The couple $ (\mu_h,v_h)$ satisfies
\begin{equation} \label{eq:mu1}
\left\{
\begin{aligned}
\Delta^{(0)}_{f,h} \, v_h &=  \mu_h \,v_h    \ {\rm on \ }  \Omega,  \\ 
v_h&= 0 \ {\rm on \ } \partial \Omega.
\end{aligned}
\right.
\end{equation} 
Moreover, $v_h>0$ on $\Omega$ and  $$\int_{\Omega} v_h^2(x) \, dx=1.$$
The following lemma, which is a direct consequence of the spectral theorem (see for instance~\cite[Theorem 8.15]{helffer2013spectral}), 
 will be instrumental throughout this work.
\begin{lemma} \label{quadra}
Let $(A,D(A))$ be a non negative self adjoint operator on a Hilbert
Space $\left(\mc H, \Vert.\Vert\right)$ with associated quadratic form
$q_A(u)=(u,Au)$ with domain $D(q_A)$. Then  for any $u\in D(q_A)$ and $b>0$
$$\left\Vert \pi_{[b,+\infty)} (A) \, u\right\Vert^2 \leq \frac{q_A(u)}{b}$$
where, for $E \subset \R$ a Borel set, $\pi_{E}(A)$ is the spectral
projection of the operator $A$ on $E$.
\label{page.piE}
\end{lemma}
%\begin{proof}
%The inequality is easily obtained for $u\in D(A)$ by writing $A$ as an
%integral using its spectral measure~\cite[Theorem 8.15]{helffer2013spectral}: 
%\begin{align*}
%(u,Au)&=\int_{[0,\infty)} \lambda \, d (\pi_{\lambda}(A)
%u,u)\\
%&\ge  \int_{(b,\infty)} \lambda \, d (\pi_{\lambda}(A)
%u,u) \ge b (\pi_{[b,+\infty)} (A) \, u , u ) = b \left\Vert \pi_{[b,+\infty)} (A) \, u\right\Vert^2.
%\end{align*} The inequality also holds for $u\in
%D(q_A)$ thanks to the density of $D(A)$ in $D(q_A)$. 
%\end{proof}
\noindent
This lemma will be in particular applied to the non negative self adjoint operators
$\Delta^{D,(p)}_{f,h}(\Omega)$ and $- L^{D,(p)}_{f,h}(\Omega)$ and
their associated quadratic forms.

\subsubsection{Small eigenvalues of $\Delta^{D,(0)}_{f,h}(\Omega)$ and $L^{D,(0)}_{f,h}(\Omega)$ }
\begin{sloppypar}
According to \cite[Corollary 2.4.4]{helffer-nier-06}, the following relations hold for all $v\in \Lambda^pH^1_{T}(\Omega)$:
$$\pi_{[0,h^{3/2})}\left
  (\Delta^{D,(p+1)}_{f,h}(\Omega)\right)\, d_{f,h}v=d_{f,h}\, \pi_{[0,h^{3/2})}\left
  (\Delta^{D,(p)}_{f,h}(\Omega)\right) v,$$
  and
  $$\pi_{[0,h^{3/2})}\left
  (\Delta^{D,(p-1)}_{f,h}(\Omega)\right)\, d_{f,h}^*v=d_{f,h}^*\, \pi_{[0,h^{3/2})}\left
  (\Delta^{D,(p)}_{f,h}(\Omega)\right) v.$$ 
  For $p\in\{0,\ldots,n\}$, let us define $F^{(p)}_h:= \text{Ran}\left( \pi_{[0,h^{3/2})}\left
  (\Delta^{D,(p)}_{f,h}(\Omega)\right) \right)$. 
Then, according to the previous intertwining relations, one can define a  finite
dimensional  Dirichlet complex structure (see \cite{chang-liu-95}, \cite{helffer-nier-06} and \cite{laudenbach2011morse}):
$$\{0\} \longrightarrow F^{(0)}_h\xrightarrow{d_{f,h} }  F^{(1)}_h\xrightarrow{d_{f,h} } \cdots  F^{(d)}_h \xrightarrow{d_{f,h} }\{0\}$$
$$\{0\} \xleftarrow{d_{f,h}^* }  F^{(0)}_h  \xleftarrow{d_{f,h}^* }  F^{(1)}_h\cdots\xleftarrow{d_{f,h} ^*} F^{(d)}_h\longleftarrow \ \{0\}.$$
For $p\in\{0,\ldots,n\}$, the dimension of the vector space $F^{(p)}_h$   in the regime $h \to 0$ have   been studied
 in \cite[Section~3]{helffer-nier-06} when $\nabla f\neq 0$ on $\pa \Omega$ and when $f:\overline \Omega\to \mathbb R$ and $f_{|\pa \Omega}$ are Morse functions. In particular, it is proved there that  the dimension of $F^{(0)}_h$ (respectively $F^{(1)}_h$) is equal to the number of local minima of $f$ (respectively to the number of generalized critical points of index $1$). A generalized critical point of index~$1$ for $\Delta^{D,(1)}_{f,h}(\Omega)$ is  either a local minimum of $f|_{\partial \Omega}$ such that
$\partial_n f(z_i)>0$ or a saddle point of index~$1$ of~$f$ inside~$\Omega$. In our setting, thanks to assumptions \textbf{[H1]}, \textbf{[H2]} and \textbf{[H3]}, there are $n$ generalized critical points of index~$1$, which are the local minima $(z_i)_{i=1,\ldots,n}$ of $f|_{\partial \Omega}$.
\end{sloppypar}
\begin{proposition} \label{ran}
Under \textbf{[H1]}, \textbf{[H2]}, and \textbf{[H3]}, there exists $h_0>0$ such that for all $h \in (0,h_0)$,
$$\dim  F^{(0)}_h =1 \ \text{ and } \ \dim F^{(1)}_h =n.$$
\end{proposition}
\noindent
 We refer to \cite[Theorem 3.2.3]{helffer-nier-06} for the proof of this
  proposition. 
 \medskip
 
 \noindent
 From~\cite{helffer-nier-06}, each eigenvalue  $\mu$  
of $\Delta^{D,(1)}_{f,h}(\Omega)$ which is smaller than~$h^{3/2}$ is exponentially small when $h\to 0$, i.e.  
$$\limsup_{h\to 0} h\ln \mu <0.$$ 
Thanks to (\ref{eq:unitary}), similar results hold for $L^{D,(p)}_{f,h}(\Omega)$: there exists $h_0>0$ such that for all $h \in (0,h_0)$
$$
\dim  \range   \pi_{[0,\frac{\sqrt h}{2})} \left(-L^{D,(0)}_{f,h}(\Omega)\right)
=1 \text{ and }
\dim  \range   \pi_{[0,\frac{\sqrt h}{2})} \left(-L^{D,(1)}_{f,h}(\Omega)\right)  =n.
$$
The spectral projection $\pi_{[0,\frac{\sqrt h}{2})}\left(-L^{D,(0)}_{f,h}(\Omega)\right)$ is the orthogonal
projection in $L^2_w(\Omega)$ on  ${\rm span}( u_h)$ and thanks to the
intertwining property (\ref{commutationL}), we have the following
crucial property:
\begin{equation}\label{eq:uh_in_span}
\nabla u_h \in \range   \pi_{[0,\frac{\sqrt h}{2})}\left(-L^{D,(1)}_{f,h}(\Omega)\right).
\end{equation}
For the ease of notation, for $p\in \{0,1\}$, we use in the following the notation:
\begin{equation}\label{eq.pih}
\pi^{(p)}_h:=\pi_{[0,\frac{\sqrt h}{2})}\left(
 - L^{D,(p)}_{f,h}(\Omega)\right).
\end{equation}
\label{page.pihp}

\subsection{Statement of the assumptions required for the
  quasi-modes}\label{sec:quasi-modes_gram_schmidt}

\subsubsection{Assumptions on quasi-modes for $L_{f,h}^{D,(p)}$, $p\in \{0,1\}$}
\label{sec221}
The next proposition gives the assumption we need on the quasi-modes
$(  \tilde \psi_i)_{i=1,\ldots,n}$ whose span approximates $\range
\pi_h^{(1)}$, and $\tilde u$ whose span approximates $\range
\pi_h^{(0)}$, in order to prove Theorem~\ref{TBIG0}.
 \begin{proposition} \label{ESTIME} \label{page.pre}
Assume \textbf{[H1]}, \textbf{[H2]} and \textbf{[H3]}. As in the statement of Theorem~\ref{TBIG0}, for
all $i\in \left\{1,\ldots,n\right\}$, $\Sigma_i $ denotes an
open set included in $\partial \Omega$ containing $z_i$ and such that
$\overline\Sigma_i \subset B_{z_i}$.

Let us assume in addition that there exist $n$
quasi-modes $(  \tilde \psi_i)_{i=1,\ldots,n}$ and a family of quasi-modes $(\tilde u=\tilde u_{\delta})_{\delta>0}$ satisfying the following conditions:
\begin{enumerate}
\item For all $ i\in \left\{1,\ldots,n\right\}$, $\tilde \psi_i\in
  \Lambda^1 H^1_{w,T}(\Omega)$ and $\tilde u \in \Lambda^0
  H^1_{w,T}(\Omega)$. The function $\tilde u$ is non negative in $\Omega$.  Moreover, one assumes the following
  normalization: for all $ i\in \left\{1,\ldots,n\right\}$,  $$\left\|\tilde
    \psi_i\right\| _{L^2_w}= \left\|\tilde u \right\| _{L^2_w} = 1.$$
\item
\begin{itemize}
\item[(a)]       There exists $\ve_1>0$  such that for all $ i\in \left\{1,\ldots,n\right\}$, in the limit $h\to 0$:
\begin{equation}\label{eq:assump_2a_psi}
  \left \|    \left(1-\pi_h^{(1)}\right) \tilde \psi_i \right\|_{H^1_w}^2   \   =  \    O \left(e^{-\frac{2}{h}( \max[f(z_n)-f(z_i), f(z_i)-f(z_1)] +\ve_1)}\right).
     \end{equation} 
     %\max[f(z_n)-f(z_i), f(z_i)-f(z_1)] 
     % \inf_{z\in B_{z_i}^c} d_a(z,z_i)
\item[(b)] For any $\delta>0$, in the limit $h\to 0$: $   \left\|   \nabla  \tilde u\right\|_{L^2_w}^2       =    O \left(e^{-\frac{2}{h}(f(z_1)-f(x_0) - \delta)}\right)$. 
\end{itemize}
\item   There exists $\ve_0>0$   such that  $\forall (i,j)
  \in \left\{1,\ldots,n\right\}^2$ with $i<j$, in the limit $h\to 0$: 
$$\left\lp \tilde \psi_i, \tilde \psi_j\right\rp_{L^2_w}=O\left( e^{-\frac{1}{h}(f(z_j)- f(z_i)+\varepsilon_0)}  \right).$$
\item
\begin{itemize}

\item[(a)] There exist constants
  $(C_i)_{i=1,\ldots,n}$ and  $p$ which do not depend on~$h$ such that for all $i\in \left\{1,\ldots,n\right\}$, in the limit $h\to 0$:
\begin{equation*}
  \left \lp       \nabla \tilde u  ,    \tilde \psi_i\right\rp_{L^2_w}  =        C_i \ h^p  e^{-\frac{1}{h}(f(z_i)- f(x_0))}   \    (  1  +     O(h )   ) .
\end{equation*}
\item[(b)] There exist  constants $(B_i)_{i=1,\dots,n}$ and $m$ which do not depend on~$h$ such that for all $(i,j) \in
  \left\{1,\ldots,n\right\}^2$, in the limit $h\to 0$:
\begin{equation*}
  \int_{\Sigma_i}   \  \tilde \psi_j \cdot n  \   e^{- \frac{2}{h} f}  d\sigma =\begin{cases}   B_i \ h^m   \     e^{-\frac{1}{h} f(z_i)}  \    (  1  +     O(h )    )   &  \text{ if } i=j   \\
 0   &   \text{ if } i\neq j.
  \end{cases} 
  \end{equation*}

  \end{itemize}
   \end{enumerate}
 Then,  for all $i\in \left\{1,\ldots,n\right\}$, in the limit $h\to 0$:
\begin{equation*} 
 \int_{\Sigma_i}  (\partial_{n}u_h)\,   e^{-\frac{2}{h}f} d\sigma =    C_i B_i  \ h^{p+m}  \,  e^{-\frac{1}{h}(2f(z_i)- f(x_0))}    \   (1+  O(h) ),  
  \end{equation*} 
  where $u_h$ is the eigenfunction associated with the smallest eigenvalue of $-L^{D,(0)}_{f,h}(\Omega)$ (see Proposition~\ref{UU}) which satisfies \eqref{eq.u_norma0}.
\end{proposition}
\noindent
Let us comment on the assumptions made on the quasi-modes. Assumption~1 gives the proper functional setting and the
normalization. Assumption~2 will be used to show that ${\rm
  span}(\tilde{\psi}_i, \, i=1, \ldots, n)$ (respectively ${\rm
  span}(\tilde{u})$) is included in $\range (\pi^{(1)}_h)$ (respectively
in $\range (\pi^{(0)}_h)={\rm
  span}(u_h)$) up to exponentially small terms.  Assumption~3 states the
quasi-orthonormality of the quasi-modes $(\tilde{\psi}_i)_{i=1,
  \ldots, n}$. Finally, Assumption~4(a) gives the components of the
decomposition of $\nabla\tilde{u}$ over ${\rm
  span}(\tilde{\psi}_i, \, i=1, \ldots, n)$, and Assumption~4(b) is
then used to find the asymptotic behavior of $\int_{\Sigma_i} 
(\partial_{n}u_h)\,   e^{-\frac{2}{h}f} d\sigma$, knowing those of $\int_{\Sigma_i}    \tilde \psi_j \cdot n  \,   e^{- \frac{2}{h}f}     d\sigma$.

\medskip

\noindent
Theorem~\ref{TBIG0} is a consequence of the existence of quasi-modes satisfying this proposition. The
construction of such quasi-modes $\tilde u$ and $(\tilde
  \psi_i)_{i=1,\ldots,n}$ satisfying the requirements of Proposition~\ref{ESTIME} will be the focus of
Section~\ref{sec:quasi-modes}, where explicit values for the constants
$B_i$, $C_i$, $p$ and $m$ will be given in~\eqref{mBi} and \eqref{Cip}. Let us mention that, from~\eqref{mBi} and \eqref{Cip}, it holds  for all $i\in \{1,\ldots,n\}$, ~$C_iB_i<0$, which is  consistent   with the fact that  $\pa_nu_h< 0$ on $\pa \Omega$ (due to the first statement in~\eqref{eq.u_norma0} and the Hopf  Lemma, see~\cite[Section 6.4.2]{evans-10}).

%%%%%
\subsubsection{Assumptions on quasi-modes for $\Delta_{f,h}^{D,(p)}$, $p\in \{0,1\}$} \label{flat}

The quasi-modes $(\tilde{\psi}_i)_{i \in
  \{1,\ldots,n\}}$ will be built using eigenforms of some Witten
Laplacians. It will thus be more convenient to work in non weighted
Sobolev spaces, and to actually consider the $1$-forms (see~\eqref{eq:unitary_p0}): for $i \in
\{1, \ldots ,n\}$,
\begin{equation}\label{eq:phi_psi}
\tilde \phi_i :=e^{-\frac{1}{h} f} \tilde \psi_i \in
\Lambda^1H^1_T(\Omega).
\end{equation}\label{page.tildephi0}
For further references, let us rewrite the hypotheses on the $1$-forms
$(\tilde \psi_j)_{j=1,\ldots,n}$ stated in Proposition~\ref{ESTIME} in
terms of the $1$-forms $(\tilde \phi_i)_{i=1,\ldots,n}$ defined by~\eqref{eq:phi_psi}:
\begin{enumerate}
\item  For all $ i\in \left\{1,\ldots,n\right\}$, $\tilde \phi_i\in
  \Lambda^1 H^1_{T}(\Omega)$ and $\left\|\tilde
    \phi_i\right\| _{L^2} = 1.$

\item   There exist $\ve_1>0$  such that for all $ i\in \left\{1,\ldots,n\right\}$, in the limit $h\to 0$:
\begin{equation}\label{eq:assump_1_phi}
   \left\|   \left ( \  1-\pi_{[0,h^{\frac32} )} \left( \Delta^{D,(1)}_{f,h}(\Omega) \right) \right) \tilde \phi_i\right\|_{H^1}^2   =O \left(e^{-\frac{2}{h}(\max[f(z_n)-f(z_i), f(z_i)-f(z_1)] +\ve_1  )}\right).
\end{equation}
\item   There exists $\ve_0>0$   such that $\forall (i,j)
  \in \left\{1,\ldots,n\right\}^2$, $i<j$, in the limit $h\to 0$:
\begin{equation}\label{eq:assump_2_phi}
\int_{\Omega} \tilde \phi_i(x)\cdot  \tilde \phi_j(x) \,d x= O \left( e^{-\frac{1}{h}[f(z_j)- f(z_i)+\varepsilon_0]} \right).
\end{equation}
\item
\begin{itemize}
\item[(a)] There exist constants $(C_i)_{i = 1,\ldots,n}$ and  $p$ which do not depend on $h$ such that 
  for all $i \in \{1,\ldots,n\}$, in the limit $h\to 0$:
\begin{equation}\label{eq:assump_3_phi}
  \int_{\Omega} \nabla  \tilde u  \cdot \tilde \phi_i \ e^{-\frac{1}{h} f}  =        C_i \ h^p  e^{-\frac{1}{h}(f(z_i)- f(x_0))}   \    (  1  +     O(h )   ) .
\end{equation}

\item[(b)] There exist  constants $(B_i)_{i = 1,\ldots,n}$ and $m$ which do not depend on $h$ such that for all $(i,j) \in
  \{1,\ldots,n\}^2$, in the limit $h\to 0$:
\begin{equation}\label{eq:assump_4_phi}
  \int_{\Sigma_i}   \  \tilde \phi_j \cdot n  \   e^{- \frac{1}{h} f} d\sigma =\begin{cases}   B_i \ h^m   \     e^{-\frac{1}{h} f(z_i)}  \    (  1  +     O(h )    )   &  \text{ if } i=j,   \\
 0   &   \text{ if } i\neq j.
  \end{cases} 
  \end{equation}

  \end{itemize}
  \end{enumerate}
As mentioned above, the construction of quasi-modes $\tilde{u}$ and
$(\tilde{\phi}_i)_{i=1, \ldots,n}$ satisfying those estimates will be the
purpose of Section~\ref{sec:quasi-modes}.

%\noindent
Let us comment on the equivalence between the first assumption here (namely~\eqref{eq:assump_1_phi}) and
assumption 2(a) in Proposition~\ref{ESTIME}
(namely~\eqref{eq:assump_2a_psi}). 
% If there exist $C_i>0$,  $\ve_1>0$ and $h_0>0$ such that for all $ i\in \left\{1,\ldots,n\right\}$ and $\forall h\in (0,h_0)$ 
%  $$\left\Vert \left(1-\pi_h^{(1)}\right)\tilde \phi_i\right\|_{H^1_w}^2 =    O \left(e^{-\frac{2}{h}\left( C_i  +     \ve_1\right)}\right),$$
% then there exist $\ve_1>0$ and $h_0>0$ such that for all $ i\in \left\{1,\ldots,n\right\}$ and $\forall h\in (0,h_0)$ 
% $$\left\|  \left ( \  1-\pi_{[0,h^{\frac32} )} \left( \Delta^{(1)}_{f,h}(\Omega) \right) \right)\tilde \phi_i\right\|_{H^1}^2   =  O \left(e^{-\frac{2}{h}\left( C_i  +     \ve_1\right)}\right),$$
%  and the reciprocal also holds.
This equivalence is a consequence of the following relation between the projectors which comes from (\ref{eq:unitary}):
$$ \pi_{[0,h^{\frac32} )} \left( \Delta^{D,(1)}_{f,h}(\Omega) \right)  = e^{-\frac{1}{h}f} \pi_h^{(1)} e^{\frac{1}{h}f}.$$
Indeed, using this relation, one has: 
$$\|e^{-\frac{1}{h}f}(1-\pi^{(1)}_h) \tilde \psi_i\|_{H^1}= \left\|  \left ( \  1-\pi_{[0,h^{\frac32} )} \left( \Delta^{D,(1)}_{f,h}(\Omega) \right) \right) \tilde \phi_i \right\|_{H^1}.$$
Moreover, one easily checks that there exists $C>0$ such that, for all $h\in (0,1)$ and for all $u \in \Lambda^p H^1(\Omega)$,
$$\|u\|_{H^1_w} \le \frac{C}{h} \left\|u \, e^{-\frac{1}{h}f}\right\|_{H^1}.$$
Therefore
$$\|(1-\pi^{(1)}_h) \tilde \psi_i\|_{H^1_w} \le \frac{C}{h} \left\|  \left ( \  1-\pi_{[0,h^{\frac32} )} \left( \Delta^{D,(1)}_{f,h}(\Omega) \right) \right) \tilde \phi_i \right\|_{H^1}$$
which shows that~\eqref{eq:assump_1_phi} (with $\ve_1$) implies~\eqref{eq:assump_2a_psi} (with $\ve_1/2$). A similar reasoning shows
that~\eqref{eq:assump_2a_psi} also implies~\eqref{eq:assump_1_phi}, but
this will not be used in the following.
% Now the strategy consists in constructing the quasi-modes $\tilde u$ and $\{\tilde \phi_j, j=1,\ldots,n\}$ satisfying the estimates on Proposition~\ref{ESTIME} in the flat space and it is the purpose of the following Section. 

\subsection{Proof of Proposition~\ref{ESTIME} } \label{estime}
In all this section, we assume   that the hypotheses \textbf{[H1]}, \textbf{[H2]} and \textbf{[H3]} hold and we assume the existence of $n+1$ quasi-modes $(\tilde u, ( \tilde
\psi_i)_{i=1,\ldots,n})$ satisfying the conditions of
Proposition~\ref{ESTIME}. In the following, $\varepsilon$ denotes a
positive constant independent of $h$, smaller than $\min (\varepsilon_1,
\varepsilon_0)$, and whose precise value may vary (a finite number of times) from one occurrence
to the other.
\medskip

\noindent
Let us start the proof with two preliminary lemmas relating
$\tilde{u}$ with $u_h$ on the one hand, and $ {\rm
   span}\left(\tilde \psi_j,j=1,\ldots,n\right) $ with $ \range   \pi_h^{(1)}$ on the
 other hand.
\begin{lemma}\label{pi0u} Let us assume that the assumptions of Proposition~\ref{ESTIME} hold. Then there exist $c>0$ and $h_0>0$ such that for $h \in (0,h_0)$,
$$\left\|\pi_h^{(0)}  \tilde u\right\|_{L^2_w}=1+O \left( e^{-\frac{c}{h}}\right).$$
\end{lemma}
\begin{proof}
Since $\tilde u\in \Lambda^0H^1_{w,T}$, $\left\Vert
  (1-\pi_h^{(0)} ) \tilde u\right\Vert _{L^2_w}$ is bounded from above by
$ h^{1/4} \left\Vert \nabla \tilde u \right\Vert_{L^2_w}$
thanks to Lemma~\ref{quadra}. In addition since  $\left\|   \nabla
  \tilde u\right\|_{L^2_w}^2   =
O\left(e^{-\frac{2}{h}[f(z_1)-f(x_0) - \delta]}\right)$ (see
assumption 2(b) in Proposition~\ref{ESTIME}), by taking $\delta \in (0,f(z_1)-f(x_0))$, one gets that  $ \left\|\pi_h^{(0)}  \tilde u \right\|_{L^2_w}  =1+  O \left( e^{-\frac{c}{h}}\right)$.
\end{proof}
\noindent
As a direct consequence of Lemma~\ref{pi0u}, one has that for $h$
small enough $\pi_h^{(0)} \tilde u\neq 0$ and therefore (since moreover $\tilde u$ is non negative in $\Omega$: $\lp u_h, \pi_h\tilde
  u  \rp_{L^2_w}= \lp u_h, \tilde
  u  \rp_{L^2_w}\ge 0$),
\begin{equation}\label{eq:uh}
 u_h  =      \frac{\pi_h^{(0)} \tilde u}{\left\|\pi_h^{(0)}  \tilde
     u\right\|_{L^2_w}}.
\end{equation}
Additionally, one has the following lemma concerning the $1$-forms.
\begin{lemma}\label{indep}
 Let us assume that the assumptions of Proposition~\ref{ESTIME} hold.
Then there exists $h_0$ such that for all $h \in (0,h_0)$,
$$ {\rm span}\left( \pi_h^{(1)}\tilde \psi_i, \,
  i=1, \ldots,n\right)=\range   \pi_h^{(1)}.$$
\end{lemma}
\begin{proof}
The determinant of the Gram matrix of the $1$-forms  $\left(\pi_h^{(1)}\tilde \psi_i\right)_{i=1,\ldots,n}$ is $1+O(e^{-c/h})$ thanks to the following identity 
\begin{equation}\label{fg}
 \left\lp \pi_h^{(1)}  \tilde \psi_i , \pi_h^{(1)} \tilde \psi_j
 \right\rp_{L^2_w} =-\left\lp \left(\pi_h^{(1)}  -1\right) \tilde \psi_i ,  \left(\pi_h^{(1)}   -1\right) \tilde \psi_j \right\rp_{L^2_w} + \left\lp \tilde \psi_i,  \tilde \psi_j \right\rp_{L^2_w}
 \end{equation}
and the fact that, from assumptions 1, 2(a) and 3 in Proposition~\ref{ESTIME}, there exist $h_0>0$, $c>0$ such that for $h\in (0,h_0)$,
$$\left\lp \tilde \psi_i, \tilde
  \psi_j\right\rp_{L^2_w}=(1-\delta_{i,j})O(e^{-\frac ch}) +
\delta_{i,j} \text{ and } \left\|    \left(1-\pi_h^{(1)}\right) \tilde \psi_i\right\|_{H^1_w}^2  =    O \left(e^{-\frac{c}{h}}\right).
 $$
 Moreover, from Proposition~\ref{ran}, $\dim \range   \pi_h^{(1)}=n$. This proves Lemma~\ref{indep}.
\end{proof}
\noindent
Thanks to Lemma~\ref{indep}, one can build on orthonormal basis $(\psi_i)_{i=1,\ldots,n}$ of ${\rm
  Ran} \left(\pi_h^{(1)}\right)$ using a Gram-Schmidt orthonormalization procedure on
$(\pi_h^{(1)}  (\tilde \psi_i)) _{i=1,\ldots,n}$. This will be done in
Section~\ref{sec:gram} below. Then, since $$\nabla u_h \in {\rm Ran} \left(\pi_h^{(1)}\right)=  {\rm
   span}\left(\psi_j,j=1,\ldots,n\right) $$ (see~\eqref{eq:uh_in_span})
 and  $\Vert \psi_j\Vert_{L^2_w}=1$,  one has
\begin{equation}\label{eq:dnuh_decomp}
\int_{\Sigma_k}   (\partial_{n}u_h) \,  e^{-\frac{2}{h} f}d\sigma=
\sum_{j=1}^n \langle \nabla u_h , \psi_j\rp_{L^2_w}  \int_{\Sigma_k}
\psi_j \cdot n \,  e^{- \frac{2}{h}  f}d\sigma .
\end{equation}
\begin{sloppypar}
\noindent
The proof of Proposition~\ref{ESTIME} then consists in replacing, in the right-hand side of~\eqref{eq:dnuh_decomp}, the function~$u_h$ by its expression~\eqref{eq:uh}
in terms of $\tilde{u}$, and the $(\psi_i)_{i=1,\ldots,n}$ by the
$(\tilde{\psi}_i)_{i=1,\ldots,n}$, and to use the assumptions of
Proposition~\ref{ESTIME} to get an asymptotic equivalent of $\int_{\Sigma_k}
(\partial_{n}u_h)  \, e^{-\frac{2}{h} f}d\sigma$ when $h\to 0$. In Section~\ref{sec:estim_inter}, one 
provides asymptotic equivalents on $\langle \nabla u_h ,
\psi_j\rp_{L^2_w}$ for $j\in \{1,\ldots,n\}$. In Section~\ref{sec:estim_bound}, one gives  asymptotic equivalents 
on $\int_{\Sigma_k}
\psi_j \cdot n \,  e^{- \frac{2}{h}  f}$ for $(k,j)\in \{1,\ldots,n\}^2$. All these results are then
gathered to conclude the proof of Proposition~\ref{ESTIME} in Section~\ref{sec:conc_prop}.
\end{sloppypar}

\subsubsection{Gram-Schmidt orthonormalization}\label{sec:gram}

Using a Gram-Schmidt procedure, one obtains the following result.
\begin{lemma} \label{gram}  Let us assume that the assumptions of Proposition~\ref{ESTIME} hold. Then for all $j\in \left\{1,\ldots,n\right\}$, there exist $(\kappa_{ji})_{i=1,\ldots,j-1}\in \mathbb R^{j-1}$ such that the $1$-forms
\begin{equation}\label{eq:gram}
 v_j:=\pi_h^{(1)}  \left[ \tilde \psi_j + \sum_{i=1}^{j-1} \kappa_{ji}
   \, \tilde \psi_i \  \right],
\end{equation}\label{page.kappaji}
(with the convention $\sum_{i=1}^0=0$) satisfy
\begin{itemize}
\item for all $k\in \left\{1,\ldots,n\right\}$, ${\rm
    span}\left(  v_i,i=1,\ldots,k \right)={\rm span}\left( \pi_h^{(1)}\tilde \psi_i,i=1,\ldots,k\right)$,
\item for all $i\neq j$, $\lp v_i , v_j \rp_{L^2_w}=0$.
\end{itemize}
\end{lemma}
\noindent
In the following, we denote by \label{page.zjpsij}
\begin{equation}\label{eq:psij}
Z_j:= \Vert v_j \Vert_{L^2_w} \text{ and } \psi_j:=\frac{v_j}{Z_j}
 \end{equation}
the normalized $1$-forms.
\medskip

\noindent
We are first interested in estimating $\kappa_{ji}$ and $Z_j$.
\begin{lemma} \label{scalarproduct}  Let us assume that the assumptions of Proposition~\ref{ESTIME} hold.
There exist $\ve>0$ and $h_0>0$ such that for all $(i,j) \in
\left\{1,\ldots,n\right\}^2$ with $i<j$ and all $ h\in (0,h_0)$
$$
\left\lp \pi_h^{(1)}  \tilde  \psi_i,  \pi_h^{(1)}  \tilde \psi_j \right\rp_{L^2_w}=O\left(e^{-\frac{1}{h}(f(z_j) - f(z_i)  +     \ve )}\right).
$$
\end{lemma}
\begin{proof}
Using assumption 2(a) in Proposition~\ref{ESTIME}, one gets: for $i<j$,
\begin{align*}
 \left\lp  \left(1- \pi_h^{(1)}  \right)  \tilde \psi_j ,  \left(1-
     \pi_h^{(1)}  \right) \tilde \psi_i\right\rp_{L^2_w} &\le \left \|    \left(1-\pi_h^{(1)}\right) \tilde \psi_k \right\|_{H^1_w}\left \|    \left(1-\pi_h^{(1)}\right) \tilde \psi_k \right\|_{H^1_w}\\
 &= O \left(e^{-\frac{1}{h}(  f(z_n)-f(z_i)+f(z_j)-f(z_1)  +   \ve )}\right)\\
 &= O \left(e^{-\frac{1}{h}(f(z_j)-f(z_i)  +   \ve )}\right).
 \end{align*}
The result is then a consequence of assumption 3 in
Proposition~\ref{ESTIME} and the identity~\eqref{fg}.\end{proof}
\noindent
Notice that since $\pi_h^{(1)}$ is an $L^2_w$-projection, $\left\lp
  \pi_h^{(1)}  \tilde  \psi_i,  \pi_h^{(1)}  \tilde \psi_j
\right\rp_{L^2_w}=\left\lp \pi_h^{(1)}  \tilde  \psi_i,  \tilde \psi_j
\right\rp_{L^2_w}$. This will be used extensively in the following.

\begin{lemma} \label{estimate1}
 Let us assume that the assumptions of Proposition~\ref{ESTIME} hold. 
Then there exist $h_0>0$, $\ve>0$ and $c>0$ such that for all  $j\in \left\{1,\ldots,n\right\}$, $i\in \{1,\ldots,j-1\}$ and $h\in (0,h_0)$
$$
\kappa_{ji}= O \left(e^{-\frac{1}{h}(f(z_j)-f(z_i)+\ve)}\right)
\text{ and } Z_j=1+O \left( e^{-\frac{c}{h}}\right).
$$
\end{lemma}
%\noindent
\begin{proof}
Let us introduce the notation: for all $i \in \{1, \ldots ,n\}$,
$$ r_i:=\left \|    \left(1-\pi_h^{(1)}\right) \tilde \psi_i \right\|_{H^1_w}^2.$$
Let us prove this lemma by induction. Concerning $\psi_1$, one has from Lemma~\ref{gram}
$$\psi_1 = \frac{v_1}{Z_1} \text{ with } v_1=\pi_h^{(1)} \tilde\psi_1.$$
Since $\Vert  \tilde \psi_1 \Vert_{L^2_w}=1$, one has $Z_1=\Vert
\pi_h^{(1)} \tilde \psi_1 \Vert_{L^2_w} \leq 1$. In addition, by
Pythagorean  Theorem and by assumption 2(a) in
Proposition~\ref{ESTIME} on $r_1$, there exists $c>0$ such that for $h$ small enough
$$
Z_1^2 \geq 1-\left\Vert \left(1- \pi_h^{(1)} \right)\tilde \psi_1 \right\Vert^2_{L^2_w} \geq 1-r_1
\geq 1-e^{-\frac{c}{h}}.
$$
Thus  $Z_1    =   1+O \left( e^{-\frac{c}{h}}\right)$. Now, concerning $\psi_2$, one has
$$\psi_2=\frac{v_2}{Z_2} \text{ with } v_2=\pi_h^{(1)} \tilde
  \psi_2 -\left\lp \pi_h^{(1)} \tilde  \psi_2 ,
    \psi_1\right\rp_{L^2_w} \ \psi_1,$$
% $$v_2=\pi_h^{(1)} \left[ \tilde \psi_2 -\left\lp \pi_h^{(1)} \tilde
%     \psi_2 , \tilde \psi_1 \right\rp_{L^2_w} \ \tilde \psi_1  Z_1^{-2}\right],$$
and thus
$\kappa_{21}=-\frac{1}{Z_1^2}\left\lp \pi_h^{(1)} \tilde \psi_1,
  \tilde \psi_2\right\rp_{L^2_w}=O\left(
  e^{-\frac{1}{h}(f(z_2)-f(z_1)+\varepsilon)}\right)$ (by Lemma
\ref{scalarproduct}). Then one obtains that $Z_2=1+O \left(
  e^{-\frac{c}{h}}\right)$ by a similar reasoning as the one we used
above for~$Z_1$.
\medskip

\noindent
In order to prove Lemma~\ref{estimate1} by induction, let us now assume that for $k\in \left\{1,\ldots,n\right\}$ and for all  $j\in \{1,\ldots,k\}$, $i\in \{1,\ldots,j-1\}$,
$$
\kappa_{ji}= O \left(e^{-\frac{1}{h}[f(z_j)-f(z_i)+\ve]}\right) \ {\rm and} \ 
Z_j=1+O \left( e^{-\frac{c}{h}}\right).
$$
One gets by the Gram-Schmidt procedure which defines the
$(\psi_i)_{i=1,\ldots,n}$,
$$\psi_{k+1}=\frac{v_{k+1}}{Z_{k+1}}$$
where, using the notation $\kappa_{ii}=1$,
\begin{align*}
v_{k+1}&= \pi_h^{(1)} \tilde \psi_{k+1} -\sum_{j=1}^k  \lp \pi_h^{(1)} \tilde \psi_{k+1}  ,     \psi_j  \rp_{L^2_w} \psi_j \\
&= \pi_h^{(1)} \tilde \psi_{k+1} -\sum_{j=1}^k \sum_{l,q=1}^j \frac{1}{Z_{j}^2}  \lp \pi_h^{(1)} \tilde \psi_{k+1}  ,   \pi_h^{(1)} \tilde \psi_l  \rp_{L^2_w} \   \kappa_{jl}\kappa_{jq}    \  \pi_h^{(1)} \tilde \psi_q \\
&= \pi_h^{(1)} \tilde \psi_{k+1} -\sum_{q=1}^k  \pi_h^{(1)} \tilde \psi_q  \sum_{j=q}^k \sum_{l=1}^j \frac{1}{Z_{j}^2}  \lp \pi_h^{(1)} \tilde \psi_{k+1}  ,   \pi_h^{(1)} \tilde \psi_l  \rp_{L^2_w} \   \kappa_{jl}\kappa_{jq}  .
\end{align*}
Then for $q\in \{1, \ldots ,k\}$,
\begin{equation}\label{eq.kappaqk}
\kappa_{(k+1) q} =- \sum_{j=q}^k \sum_{l=1}^j  \frac{1}{Z_{j}^2}  \lp \pi_h^{(1)} \tilde \psi_{k+1}  ,   \pi_h^{(1)} \tilde \psi_l  \rp_{L^2_w} \   \kappa_{jl}\kappa_{jq}.    
\end{equation}
Since $Z_j=1+O \left( e^{-\frac{c}{h}} \right)$, one gets $Z_j^{-1}=1+O \left( e^{-\frac{c}{h}}\right)$. 
Using Lemma~\ref{scalarproduct}, one obtains $\left\lp \pi_h^{(1)} \tilde \psi_{k+1}  ,   \pi_h^{(1)} \tilde \psi_l \right\rp_{L^2_w} = O\left(e^{-\frac{1}{h}[f(z_{k+1}) - f(z_l)  +     \ve ]}\right)$, since $l<k+1$. Moreover,  since $l\leq j$ and $q\leq j$, it holds:  
$$\kappa_{jl}=O \left(e^{-\frac{1}{h}[f(z_j)-f(z_l)]}\right) \text{ and } \kappa_{jq}=O \left(e^{-\frac{1}{h}[f(z_j)-f(z_q)]}\right).$$
Consequently, one obtains that 
$$ \frac{1}{Z_{j}^2}  \lp \pi_h^{(1)} \tilde \psi_{k+1}  ,   \pi_h^{(1)} \tilde \psi_l  \rp_{L^2_w} \   \kappa_{jl}\kappa_{jq}=O \left(e^{-\frac{1}{h}[f(z_{k+1}) - f(z_q) + 2(f(z_j)-f(z_l))+\ve]}\right). $$
Thus, since $f(z_j)\ge f(z_l)$, one obtains from~\eqref{eq.kappaqk}, that   for $q\in \{1, \ldots,k\}$, there exists $\ve >0$ such that for $h$ small enough:
%If $q<j$, $\kappa_{jl}\kappa_{jq}  =  O \left(e^{-\frac{1}{h}[f(z_j)-f(z_q)+\ve]}\right)$,
%and thus 
%\begin{align*}
%\left\lp \pi_h^{(1)} \tilde \psi_{k+1}  ,   \pi_h^{(1)} \tilde \psi_l \right\rp_{L^2_w} \kappa_{jl}\kappa_{jq} &=
%O \left(e^{-\frac{1}{h}[f(z_{k+1}) - f(z_l)+f(z_j)-f(z_q)+\ve]}\right)\\
% &=O \left(e^{-\frac{1}{h}[f(z_{k+1}) - f(z_q)+\ve]}\right).
%\end{align*}
% If $l<j$ and if $q=j$,  $$\kappa_{jl}\kappa_{jq}  =  O \left(e^{-\frac{1}{h}[f(z_j)-f(z_l)+\ve]}\right).$$
%Since, if $l<j$ and $q=j$, $f(z_q)=f(z_j)\ge f(z_l)$ one obtains that 
%\begin{align*}
%\left\lp \pi_h^{(1)} \tilde \psi_{k+1}  ,   \pi_h^{(1)} \tilde \psi_l \right\rp_{L^2_w} \kappa_{jl}\kappa_{jq} &=
%O \left(e^{-\frac{1}{h}[f(z_{k+1}) - f(z_l)+f(z_j)-f(z_l)+\ve]}\right)\\
% &=O \left(e^{-\frac{1}{h}[f(z_{k+1}) - f(z_l)+\ve]}\right)=O \left(e^{-\frac{1}{h}[f(z_{k+1}) - f(z_q)+\ve]}\right).
%\end{align*}
%If $l=q=j$, $\kappa_{jl}\kappa_{jq} = 1$.
$$\kappa_{(k+1) q} =O\left(e^{-\frac{1}{h}[f(z_{k+1}) - f(z_q)  +     \ve ]}\right).$$
The estimate $Z_{k+1}=1+O \left( e^{-\frac{c}{h}}\right)$ is a consequence of the fact that $(\kappa_{(k+1) q})_{q \in \{1,\ldots ,k\}}$ are exponentially small together with  the estimate 
$$\left\Vert \pi_h^{(1)} \tilde \psi_{k+1}\right\Vert_{L^2_w} = 1+O \left( e^{-\frac{c}{h}}\right).$$
This concludes the proof of Lemma~\ref{estimate1} by induction. 
\end{proof}

\subsubsection{Estimates on the interaction terms $\left(\lp    \nabla
    u_h ,   \psi_j \rp_{L^2_w}\right)_{j \in \{1,\ldots,n\}}$}\label{sec:estim_inter}
\begin{lemma} \label{interaction} Let us assume that the assumptions of Proposition~\ref{ESTIME} hold. Then for $j\in \{1,\ldots,n\}$, one has
  \begin{equation*}
 \lp    \nabla  u_h ,   \psi_j \rp_{L^2_w}   =C_j\ h^p   e^{-\frac{1}{h}(f(z_j)- f(x_0))}   \    ( 1  +     O(h)).
  \end{equation*} 
  \end{lemma}
%\noindent
\begin{proof} From~\eqref{commutationL}, for any $\phi \in H^1_{w,T}(\Omega)$ and $v\in L^2_w(\Omega)$, it holds,
\begin{equation}\label{eq.comut_nabla_pi}
\lp \nabla \pi^{(0)}_h\phi , \pi^{(1)}_h v\rp_{L^2_w}=\lp
    \nabla  \phi , \pi^{(1)}_h v \rp_{L^2_w}.
    \end{equation}
Using \eqref{eq:uh}--\eqref{eq:gram}--\eqref{eq:psij}--\eqref{eq.comut_nabla_pi},
  for all $j\in\left\{1,\ldots,n\right\}$, one has
      \begin{align}  \lp    \nabla u_h ,   \psi_j \rp_{L^2_w}    &=\frac{Z^{-1}_j}{\left\|\pi_h^{(0)} \tilde u\right\|_{L^2_w}} \left[   \left\lp       \nabla \tilde u  ,     \pi_h^{(1)} \tilde \psi_j  \right\rp_{L^2_w}   +\sum_{i=1}^{j-1}\kappa_{ji} \  \left\lp       \nabla \tilde u  ,    \pi_h^{(1)} \tilde \psi_i  \right\rp_{L^2_w} \  \right]\nonumber\\
 &=\frac{Z^{-1}_j}{\left\|\pi_h^{(0)} \tilde u \right\|_{L^2_w}} \left[   \left\lp       \nabla \tilde u  ,    \tilde \psi_j  \right\rp_{L^2_w} +  \left\lp   \nabla \tilde u, \left( \pi_h^{(1)}-1   \right) \tilde \psi_j  \right\rp_{L^2_w}   \right]  \nonumber \\
 &\quad +\frac{Z^{-1}_j}{\left\|\pi_h^{(0)} \tilde u \right\|_{L^2_w}}\left[\sum_{i=1}^{j-1}\kappa_{ji} \ \left(  \left\lp       \nabla \tilde u  ,    \tilde \psi_i  \right\rp_{L^2_w} + \left\lp   \nabla \tilde u, \left( \pi_h^{(1)}-1   \right) \tilde \psi_i  \right\rp_{L^2_w}  \right)\  \right].\label{eq:nablauh_psij}
  \end{align} 
From Lemmata \ref{pi0u} and \ref{estimate1}, one has  $\frac{Z_{j}^{-1}}{
  \left\|\pi_h^{(0)} \tilde u\right\|_{L^2_w} }=1+O \left(
  e^{-\frac{c}{h}}\right)$. Using assumptions~2 and~4(a) in
Proposition~\ref{ESTIME} and Lemma~\ref{estimate1}, there exists
$\delta_0>0$ such that for all $\delta\in (0,\delta_0)$,
\begin{align*} 
 \lp    \nabla u_h ,   \psi_j \rp_{L^2_w} &=C_j h^p  e^{-\frac{1}{h}[f(z_j)- f(x_0)]}     (     1 + O(h ) )\\
   &\quad + O\left(  e^{-\frac{1}{h}[f(z_1)- f(x_0) -\delta+ f(z_j)-f(z_1) + \ve]}  \right) \\
   &\quad +\sum_{i=1}^{j-1}O\left(  e^{-\frac{1}{h}[f(z_j)-f(z_i)+\ve +f(z_i)- f(x_0) -\delta]}  \right) \\
    &\quad +\sum_{i=1}^{j-1}O\left(  e^{-\frac{1}{h}[f(z_j)-f(z_i)+\ve +f(z_1)- f(x_0) -\delta+f(z_i)-f(z_1)+\ve]}  \right).
\end{align*}
Therefore choosing $\delta<\varepsilon$, there exists $\ve'>0$ such that
\begin{align*} 
 \lp    \nabla u_h ,   \psi_j \rp_{L^2_w}  
&=C_j h^p  e^{-\frac{1}{h}[f(z_j)- f(x_0)]}     (     1  +     O(h ) )
+ O \left(  e^{-\frac{1}{h}[f(z_j)- f(x_0) + \ve']} \right).
 \end{align*} 
This concludes the proof of Lemma~\ref{interaction}.
\end{proof}
\subsubsection{Estimates on the boundary terms $\left(\int_{\Sigma_k}  \psi_j \cdot n \,  e^{- \frac{2}{h} f}d\sigma\right)_{(j,k)\in \left\{1,\ldots,n\right\}^2}$}\label{sec:estim_bound}
One denotes in this section,  for $k\in \left\{1,\ldots,n\right\}$, $K_k:=\max(f(z_n)-f(z_k), f(z_k)-f(z_1))\ge0$. 
\begin{lemma} \label{gamma}
 Let us assume that the assumptions of Proposition~\ref{ESTIME} hold. 
Then for all $(j,k)\in \left\{1,\ldots,n\right\}^2$, there exists $\ve>0$ such that it holds      \begin{equation*}
   \int_{\Sigma_k}  \psi_j \cdot n  \ e^{- \frac{2}{h} f}  d\sigma   =  \
  \begin{cases}
  O\left( e^{-\frac{1}{h}[f(z_j)+\ve]} \right)  &  \text{ if } k<j,  \\
   B_j h^m  \     e^{-\frac{1}{h} f(z_j)}    \left( 1 +     O(h)   \right)  &       \text{ if } k=j, \\
    O\left(  e^{-\frac{1}{h}[K_j+f(z_k)+\ve]}  \right) +  \sum_{i=1}^{j-1} O\left(  e^{-\frac{1}{h}[f(z_j)-f(z_i)+K_i+f(z_k)+\ve]}  \right) &  \text{ if } k>j. \\
  \end{cases} 
  \end{equation*}
  \end{lemma}
%\noindent
\begin{proof} 
Using \eqref{eq:gram}--\eqref{eq:psij} and writing $\pi_h^{(1)}  \tilde
\psi_i=  \tilde \psi_i +\left( \pi_h^{(1)}-1   \right)  \tilde
\psi_i$, one obtains that
  $$\int_{\Sigma_k}  \psi_j \cdot n \, e^{- \frac{2}{h} f}
d\sigma=a_{jk}+b_{jk}+\sum_{i=1}^{j-1}(c_{jki}+d_{jki})$$
with for $(j,k) \in \{1, \ldots ,n\}^2$ and $i \in \{1, \ldots, j-1\}$,
$$a_{jk}=Z_j^{-1}\int_{\Sigma_k}  \tilde \psi_j \cdot n \, e^{-
  \frac{2}{h}f} d\sigma \text{ , }
b_{jk}=Z_j^{-1}\int_{\Sigma_k}  \left( \pi_h^{(1)}-1   \right) \tilde \psi_j \cdot n \, e^{-
  \frac{2}{h}f} d\sigma $$
$$c_{jki}=Z_j^{-1}\kappa_{ji} 
    \int_{\Sigma_k} \tilde  \psi_i \cdot n \, e^{- \frac{2}{h} f}d\sigma 
    \text{ and }
d_{jki}=Z_j^{-1}\kappa_{ji} 
    \int_{\Sigma_k} \left( \pi_h^{(1)}-1   \right)\tilde  \psi_i \cdot n \, e^{- \frac{2}{h} f}d \sigma  .$$
Using the trace theorem and assumption
2(a) in Proposition~\ref{ESTIME}, one has, for some universal constant $C$,
\begin{align}
\int_{\Sigma_k}  \left( \pi_h^{(1)}-1   \right) \tilde \psi_j \cdot n \, e^{-
  \frac{2}{h}f}d\sigma &\le \frac Ch \left \|    \left(\pi_h^{(1)}-1\right) \tilde \psi_j
\right\|_{H^1_w} \sqrt{\int_{\Sigma_k} e^{-\frac{2}{h}f}}\nonumber\\
&=O \left( e^{-\frac{1}{h}[K_j+f(z_k)+\ve]}  \right).
\label{eq:estim_Sk}
\end{align}
% Let us again use the notation: for $i \in \{1, \ldots n\}$,
% $$ r_i:=\left \|    \left(1-\pi_h^{(1)}\right) \tilde \psi_i
% \right\|_{H^1_w}^2.$$
% and using the trace theorem, one gets
%   \begin{align*}
%  \int_{\Sigma_k}  \psi_j \cdot n  \ e^{-\frac{2}{h}f}   &=Z_j^{-1}\left[ \int_{\Sigma_k}  \tilde \psi_j \cdot n \, e^{- \frac{2}{h}f}  +O \left(\sqrt{r_j}\sqrt{\int_{\Sigma_k} e^{-\frac{2}{h}f}} \right)\right] \\
% &\quad +Z_j^{-1}\left[\sum_{i=1}^{j-1}\kappa_{ji} \ \left[ \int_{\Sigma_k} \tilde  \psi_i \cdot n \, e^{- \frac{2}{h} f}      + O \left(\sqrt{r_i}\sqrt{\int_{\Sigma_k} e^{-\frac{2}{h}f}}\right)\ \right] \right].
%   \end{align*} 
If $k=j$ and $i \in \{1, \ldots, j-1\}$, one gets,
using~\eqref{eq:estim_Sk}, Lemma~\ref{estimate1} and assumption 4(b) in Proposition~\ref{ESTIME}:
$$  a_{jk}=B_j h^m  \     e^{-\frac{1}{h} f(z_j)}     ( 1  +
O(h)  ) \text{ , }
  b_{jk}=O \left( e^{-\frac{1}{h}[K_j+f(z_j)+\ve]}  \right), $$
$$ c_{jki}= 0 \text{ and }
  d_{jki}=O\left(  e^{-\frac{1}{h}[f(z_j)-f(z_i)+K_i+f(z_j)+\ve]} \right) .
 $$
  If $k\neq j$ and $i \in \{1, \ldots, j-1\}$,
% . One notices that $c_{jki}=0$ for all $i<j$, if  $j< k$;
   % and if $j>k$ there exists only one $i$ such that $i<j$, $i=k$. Therefore
   one gets using again~\eqref{eq:estim_Sk}, Lemma~\ref{estimate1} and assumption~4(b) in Proposition~\ref{ESTIME}:
$$
  a_{jk}=0 \text{ , }
    b_{jk}=O\left(  e^{-\frac{1}{h}[K_j+f(z_k)+\ve]}  \right),
$$
$$    c_{jki}= \begin{cases}
  O \left( e^{-\frac{1}{h}[f(z_j)+\ve]}  \right)  &  \text{ if } k=i ,  \\
  0 &        \text{ if } k\neq i, \\
  \end{cases} \text{ and }
  d_{jki}=O\left( e^{-\frac{1}{h}[f(z_j)-f(z_i)+K_i+f(z_k)+\ve]}  \right).
 $$
Notice that $c_{jki}=0$ if  $j< k$ and that if $j>k$ there exists only
one $i$ such that $c_{jki}\neq0$, namely $i=k$.
This concludes the proof of the Lemma~\ref{gamma}.
\end{proof}

\subsubsection{Estimates on $\left(\int_{\Sigma_k}  ( \partial_{n}u_h)  \,
    e^{- \frac{2}{h} f} d\sigma\right)_{k\in \{1,\ldots,n \}}$}\label{sec:conc_prop}

We are now in position to conclude the proof of
Proposition~\ref{ESTIME}, by proving that for $k\in\left\{1,\ldots,n\right\}$, one has
 $$\int_{\Sigma_k}   (\partial_{n}u_h)\,   e^{-\frac{2}{h}f}d\sigma =C_kB_k \ h^{p+m} e^{- \frac{1}{h} \left(2f(z_k)-f(x_0)\right)} \left(  1+     O(h )  \right ).$$

\begin{proof}  Let us assume that the assumptions of Proposition~\ref{ESTIME} hold. Let us recall the decomposition~\eqref{eq:dnuh_decomp}:
\begin{equation*}
\int_{\Sigma_k}   (\partial_{n}u_h)  \, e^{-\frac{2}{h} f}d\sigma=
\sum_{j=1}^n \langle \nabla u_h , \psi_j\rp_{L^2_w} \ \int_{\Sigma_k}
\psi_j \cdot n \,  e^{- \frac{2}{h}  f}d\sigma.
\end{equation*}
Using Lemmas~\ref{interaction} and~\ref{gamma} 
, we can now estimate the terms in the sum in the right-hand side.
%
%
% From Lemmas~\ref{interaction} and~\ref{gamma}, it is simple to check that for a fixed $k\in \{1,\ldots,n\}$, the
% dominant term in the sum in the right-hand side is the one with index
% $j=k$. Let us now estimate the term $\int_{\Sigma_k}   (\partial_{n}u)
% e^{-\frac{2}{h} f}d\sigma$, using again Lemmas~\ref{interaction} and~\ref{gamma}.
If $j>k$, there exist $\ve>0$ and $h_0>0$ such that for all $ h\in (0,h_0)$
\begin{align*}
\langle \nabla u_h , \psi_j\rp_{L^2_w} \ \int_{\Sigma_k}   \psi_j \cdot n  \ e^{- \frac{2}{h} f}d\sigma&=C_jh^p O\left( e^{-\frac{1}{h}[f(z_j)-f(x_0)]}e^{-\frac{1}{h} [f(z_j)+\ve]}\right)\\
&=C_jh^p O\left(e^{-\frac{1}{h}[2f(z_j)-f(x_0)+\ve]}\right)\\
&=C_jh^p O \left(e^{-\frac{1}{h}[2f(z_k)-f(x_0)+\ve]}\right).
\end{align*}
If $j<k$, there exist $\ve>0$ and $h_0>0$ such that for all $ h\in (0,h_0)$
\begin{align*}
&\langle \nabla u_h , \psi_j\rp_{L^2_w} \ \int_{\Sigma_k}   \psi_j
  \cdot  n   \ e^{- \frac{2}{h}f}d\sigma\\
&=O\left(  e^{-\frac{1}{h}[f(z_j)-f(x_0)+K_j+f(z_k)+\ve]}  \right) +  \sum_{i=1}^{j-1} O\left( e^{-\frac{1}{h}[f(z_j)-f(x_0)+ f(z_j)-f(z_i)+K_i+f(z_k)+\ve]}  \right )\\
&=O\left(  e^{-\frac{1}{h}[f(z_j)-f(x_0)+f(z_n)-f(z_j)+f(z_k)+\ve]}  \right) +  \sum_{i=1}^{j-1} O\left( e^{-\frac{1}{h}[f(z_j)-f(x_0)+ f(z_j)-f(z_i)+f(z_n)-f(z_i)+f(z_k)+\ve]}  \right    )\\
&=O\left(  e^{-\frac{1}{h}[f(z_n)+f(z_k)-f(x_0)+\ve]}  \right) +  \sum_{i=1}^{j-1} O\left( e^{-\frac{1}{h}[f(z_n)+f(z_k)-f(x_0)+2(f(z_j)-f(z_i)) +\ve]}  \right   )\\
&=O\left(  e^{-\frac{1}{h}[2f(z_k)-f(x_0)+\ve]}  \right).
\end{align*}
Finally if $j=k$, $\exists \ve>0$ and $\exists h_0>0$ such that for all $ h\in (0,h_0)$
\begin{equation}\label{eq:dnuh_dominant}
\langle \nabla u_h , \psi_k\rp_{L^2_w} \ \int_{\Sigma_k}   \psi_k
 \cdot  n   \ e^{- \frac{2}{h}f}d\sigma=C_kB_k\ h^{p+m} e^{- \frac{1}{h}
   (2f(z_k)-f(x_0))} (  1+     O(h )   ).
\end{equation}

 From these estimates, for a fixed $k\in \{1,\ldots,n\}$, the
 dominant term in the sum in the right-hand side
 of~\eqref{eq:dnuh_decomp} is the one with index 
 $j=k$, namely~\eqref{eq:dnuh_dominant}.
 This concludes the proof of  Proposition~\ref{ESTIME}.
 % To end the proof of Proposition~\ref{ESTIME} it remains to prove the estimate of $\lambda_h$.
 \end{proof}

\section{On the Agmon distance}\label{agmonproperty}

In this section, we present the main properties of the
Agmon distance introduced in Definition~\ref{def.agmon-intro}. The Agmon distance is
useful to quantify the decay of eigenforms of
  Witten Laplacians away from critical points of $f$ and
$f|_{\partial \Omega}$. The Agmon distance
on a domain without boundary has been extensively
analyzed in many previous works (see in
particular~\cite{helffer-sjostrand-84,helffer-sjostrand-85a,helffer-sjostrand-85b,helffer-sjostrand-86a}). The
aim of this section is to generalize well-known results in the case
without boundary to our context,
namely for bounded domains. Indeed, to the best of our knowledge, this has
not been done in the literature before in a comprehensive way.

%In all this section, we assume that the function $f : \overline
%\Omega\to \mathbb R$ and the restriction of $f$ to the boundary of
%$\Omega$ (denoted by $f|_{\partial \Omega}$) have a finite number of
%critical points.  

For simplicity, all the proofs in this section are made for a bounded
connected open connected $C^{\infty}$ domain $\Omega\subset \mathbb
R^d$ (equipped with the geodesic Euclidean distance \eqref{eq:geodesic_euclidean}) and for a $C^{\infty}$  function  $f:\overline \Omega\to \mathbb R$. The generalization to a $C^{\infty}$ compact connected
Riemannian manifold of dimension $d$ with boundary is
straightforward. The notation $\vert x-y \vert$  denotes the Euclidean distance between $x$ and $y$ in $\mathbb R^d$. If one deals with a compact connected Riemannian manifold of dimension $d$ with boundary, this distance has to be replaced by the  geodesic distance on $\overline \Omega$ for the initial metric and the scalar product between two vectors of $\mathbb R^d$ has to be replaced by the one induced by the initial metric on the tangent space of $\overline \Omega$.   
\medskip

\noindent
This section is organized as follows. Section~\ref{motivationA} is
devoted to an equivalent definition of the Agmon distance, which will
be crucial in the following. In Section~\ref{properties}, we
then give a few useful properties of the Agmon distance. As already
mentioned in Section~\ref{sec:def_agmon_intro}, there is a
link between the Agmon distance and the eikonal equation. This is
explained in Sections~\ref{eiko_agmon} and~\ref{eiko_agmon_basin}. This
link is useful in order to build explicit curves realizing the Agmon
distance, as explained in Section~\ref{minim}.

%\subsection{The set $A(x,y)$ and proof of \eqref{daA}}
\subsection{The set $A(x,y)$ and an equivalent definition of the Agmon
distance}
\label{motivationA} 
In order to study the Agmon distance, it will be more convenient for
technical reasons  to restrict the class of curves appearing in the
definition of the Agmon distance (see Definition~\ref{def.agmon-intro}).
\begin{definition} \label{defA}
 For $x,y\in \overline \Omega$, 
we denote by $A\left(x,y\right)$ the set of  curves $\gamma :[0,1] \to \overline \Omega$ which are Lipschitz with $\gamma(0)=x$, $\gamma(1)=y$ and such that the set $\partial \{t\in [0,1]\big | \gamma(t)\in \partial \Omega\}$ is finite.
\end{definition}
\noindent
Here, \label{page.axy} $\partial \{t\in [0,1], \, \gamma(t)\in \partial \Omega\}$
denotes the boundary of the set $\{t\in [0,1], \, \gamma(t)\in \partial \Omega\}$. The main result of this
section is that, under assumption \textbf{[H3]}, the Agmon distance $d_a$
satisfies (compare with~\eqref{eq:definition}):
\begin{equation} \label{daA}
\forall \left(x,y\right)\in \overline{\Omega} \times \overline{\Omega} , \  d_a(x,y)=\inf_{\gamma\in A(x,y)} L(\gamma,(0,1)).
\end{equation}
See Corollary~\ref{equalityagmon} below.

The following lemma will be needed several times throughout this section. It motivates the use of the set $A\left(x,y\right)$ appearing in Definition~\ref{defA}.
\begin{lemma} \label{h}
Let $x,y \in \overline \Omega$ and $\gamma \in A\left(x,y\right)$. Then for any $h: \overline{\Omega}\to \mathbb R$ which is $C^1$, one gets
\begin{equation} \label{equality2}
h(y)-h(x)=\int_{\left\{ t \in [0,1], \,  \gamma(t)\in
    \Omega\right\}}  (\nabla h)(\gamma)\cdot \gamma ' +  \int_{{\rm
    int} \left\{ t \in [0,1], \, \gamma (t)\in \partial \Omega\right\} }  (\nabla_T h)(\gamma )\cdot \gamma'.
\end{equation}
Here, the notation ${\rm int}$ stands for the interior of a domain.
\end{lemma}
\begin{proof}
Since $\gamma$ is Lipschitz, $h\circ \gamma$ is Lipchitz and thus absolutely continuous. Therefore, one has: 
\begin{align*}
h(y)-h(x) &=\int_0^1 \frac{d}{dt} (h\circ \gamma )\\
&= \int_{\left\{ t \in [0,1], \, \gamma(t)\in \Omega\right\}}
\frac{d}{dt} (h\circ \gamma )+  \int_{{\rm int} \left\{ t \in [0,1], \, \ \gamma(t)\in \partial\Omega\right\} }  \frac{d}{dt} (h\circ \gamma )\\
&\quad + \int_{\partial \left\{ t \in [0,1], \, \gamma(t)\in \partial\Omega\right\} }  \frac{d}{dt} (h\circ \gamma ).
\end{align*}
By definition of the set $A\left(x,y\right)$ (see Definition~\ref{defA})
the set $\partial \left\{ t \in [0,1], \,  \
  \gamma(t)\in \partial\Omega\right\} $ has Lebesgue measure zero, and thus
$\int_{\partial \left\{ t \in [0,1], \,  \gamma(t)\in \partial\Omega\right\} }  \frac{d}{dt} (h\circ \gamma )=0$.
The curve $\gamma$ is continuous and  thus the set $\left\{ t \in
  [0,1], \,  \gamma(t)\in \Omega\right\}$ is open in $[0,1]$. As a consequence, using in addition that since $\gamma$ is Lipschitz, it is almost everywhere differentiable (by the Rademacher Theorem), one has for almost every $t\in [0,1]$:
$$% \label{hh}
\frac{d}{dt}h(\gamma)(t)=\left\{
\begin{aligned}
  (\nabla h)\left(\gamma(t)\right)\cdot \frac{d}{dt}\gamma(t)  \ {\rm \ a.e.\
    on \ } \left\{ t \in [0,1],  \ \gamma(t)\in \Omega\right\}  \\ 
 (\nabla_T h)\left(\gamma(t)\right)\cdot  \frac{d}{dt}\gamma(t)  \ {\rm  \ a.e.\ on \ }  {\rm int} \left\{ t \in [0,1], \ \gamma(t)\in \partial\Omega\right\}.
\end{aligned}
\right.
$$
This proves \eqref{equality2}.
\end{proof}
\begin{remark} Notice that there exist Lipschitz curves $\gamma$ such
  that  $\partial \left\{ t \in [0,1], \ \gamma(t)\in \partial
    \Omega\right\}$ has a positive Lebesgue measure. Let us give an
  example. Consider $\Omega=(0,1)\times (0,2)$ and the curve 
$$\gamma: t\in [0,1]\mapsto\left(t,\inf_{y\in K}\vert t-y\vert\right)\in [0,1]^2,$$
where $K$ is the Smith-Volterra-Cantor set in $[0,1]$, such that $K$
is closed, has a positive Lebesgue
measure and has an empty interior (see \cite[Section 2.5]{pugh2002real}). Notice that the
distance $\inf_{y\in K}\vert t-y\vert$ to $K$ is a Lipschitz function
of $t \in (0,1)$, so that $\gamma$ is a Lipschitz function. The set
$K$ is closed and thus  
$$\left\{ t \in [0,1], \,  \gamma(t)\in \partial
  \Omega\right\}=\left\{ t \in [0,1], \, 
  \gamma(t)=0\right\}=K.$$ 
Therefore $\partial \left\{ t \in [0,1], \,  \ \gamma(t)\in \partial\Omega\right\}=\overline{K}\setminus 
({\rm int} K)=K$. 
\end{remark}
\noindent
The following results will be useful to prove the equality~\eqref{daA}
and to prove the existence of curves realizing the Agmon distance in  Section~\ref{minim}.
\begin{proposition} \label{four}
Assume that \textbf{[H3]} holds. Let $
\gamma : [0,1]\to \overline \Omega$ be a Lipschitz curve and assume
that there exists a time $t^*\in [0,1]$ such that $\gamma(t^*)\in \partial \Omega$. Then there exists $(a,b)\in [0,1]^2$, with $a\leq t^*\leq b$ and $a<b$ such that for all $ (t_1,t_2)\in [0,1]^2$, with $a\leq t_1<t_2\leq b$, there exists a Lipschitz curve $\gamma_{12}: [t_1,t_2]\to \overline \Omega$ satisfying 
$$\gamma_{12}(t_1)=\gamma(t_1) \text{ and }\gamma_{12}(t_2)=\gamma(t_2),$$ 
\begin{equation} \label{oo}
 L(\gamma, (t_1,t_2))\geq  L(\gamma_{12}, (t_1,t_2))
\end{equation} 
and, either
  $\left\{ t \in [t_1,t_2], \ \gamma_{12}(t)\in \partial
  \Omega\right\}$ is empty, or its boundary $\partial \left\{ t \in [t_1,t_2], \
  \gamma_{12}(t)\in \partial \Omega\right\}$ consists of isolated
points in $\left\{ t \in [t_1,t_2], \ \gamma_{12}(t)\in \partial
  \Omega\right\}$. Moreover, if the following is satisfied:
\begin{equation}\label{eq:s1s2s3}
\exists (s_1,s_2,s_3)\in
[t_1,t_2]^3, s_1<s_2<s_3,\,
\gamma(s_1)\in \partial \Omega, \gamma(s_2)\in \Omega \text{ and }
\gamma(s_3)\in \partial \Omega,
\end{equation}
then the inequality $\eqref{oo}$ is strict.
\end{proposition}
\begin{remark}
Notice that if $t^*\in \partial \left\{ t \in [0,1], \
  \gamma(t)\in \partial \Omega\right\}$ is not isolated in \linebreak $\partial
\left\{ t \in [0,1], \ \gamma(t)\in \partial \Omega\right\}$, then there exists a neighborhood $[t_1,t_2]$ of $t^*$ in $[0,1]$ such that~\eqref{eq:s1s2s3} is satisfied and thus the
inequality $\eqref{oo}$ is strict. Therefore if a Lipschitz curve $\gamma$ realizes the infimum of $L$ on ${\rm Lip}(x,y)$, then $\partial \left\{ t \in [0,1], \
  \gamma(t)\in \partial \Omega\right\}$ is finite. This motivates the introduction of the set $A(x,y)$.
\end{remark}
\begin{proof} 

Let $t^*\in [0,1]$ be such that $\gamma(t^*)\in \partial \Omega$. The proof is divided into three steps.\\

\noindent
\underline{Step 1: Introduction of a normal system of coordinates and
  definition of $a$ and $b$.} 
  \medskip

\noindent
Let us consider a neighborhood $V_{\partial
  \Omega}$ of $\gamma(t^*)$ in $\partial \Omega$, and a  smooth
local system of coordinates in $V_{\partial
  \Omega} \subset \partial \Omega$, denoted by $x_T:V_{\partial
  \Omega} \to \R^{d-1}$. Let us now extend it to a tangential and normal system
of coordinates around $\gamma(t^*)$ in $\overline \Omega$, denoted by
$\phi(x)=(x_T,x_N)$. The function $\phi$ is defined from a
neighborhood of  $\gamma(t^*)$ in $\overline \Omega$ to
$\R^d$. Moreover, one has $x_N(x)\geq 0$ and for all $x$,
$x_N(x)=0$ if and only if $x\in \partial \Omega$. One can assume without loss of generality that
$\phi$ is defined on a neighborhood $V_\alpha$ of  $\gamma(t^*)$ in
$\overline \Omega$ such that $\phi(V_\alpha)=U \times [0,\alpha]$ for
$\alpha >0$, and $U \subset \R^{d-1}$. For this normal system of
coordinates, the metric tensor $G$ is such that:
% % introduces a smooth
% change in the initial metric on $\overline \Omega$ around
% $\gamma(t^*)$ in $\overline \Omega$ and the new metric $G$ (seen as a
% symmetric square matrix of size $d$) is such that $x_T\perp x_N$ in
% the sense that:
$\forall (x_T,x_N) \in U\times[0,\alpha]$, 
$$G(x_T,x_N)=\begin{pmatrix}
\tilde G (x_T,x_N) & 0\\
 0 & G_{NN}(x_T,x_N)
\end{pmatrix},$$
where  $\tilde G (x_T,x_N) \in \R^{(d-1)\times(d-1)}$ and $G_{NN}(x_T,x_N) \in \R$ are smooth
functions of $(x_T,x_N)$.  The existence of such a
system of coordinates is a consequence of the collar theorem,
see~\cite{GSchw}.  
\medskip

\noindent
Under assumption~\textbf{[H3]} (namely $\partial_nf>0$ on $\partial
\Omega$), there exist constants $A>1$ and $\varepsilon_1>0$ such that
for all $ x_N \in (0,\varepsilon_1]$ and  for all $ x_T \in
U$,
(see~\eqref{eq:def_g} for the definition of~$g$)
\begin{equation} \label{gg} 
g(\phi^{-1}(x_T,x_N))\geq A\,g(\phi^{-1}(x_T,0)).
\end{equation}
Since the local change of coordinates is smooth, for all $\varepsilon
\in (0,1)$, there exists $\eta>0$ such that for all  $x_N \in [0,\eta]$ and for all $x_T \in U$, one has
$$\tilde G\left(x_T,x_N\right)\geq (1-\varepsilon)^2\tilde G(x_T,0).$$
Let us choose $\varepsilon>0$ such that $(1-\varepsilon)\,A>1$. One
can reduce $V_{\alpha}$ such that $0\leq x_N(x) \leq \min
(\eta,\varepsilon_1)$ for all $x\in V_{\alpha}$. By continuity of
$\gamma$, there exist $(a,b)\in [0,1]^2$, with $a\leq t^*\leq b$ and
$a<b$ such that for all $ t\in [a,b]$, $\gamma(t)\in V_{\alpha}$. Let
us introduce the components of $\gamma$ in the normal system of
coordinates: $(\gamma_T(t),\gamma_N(t))=\phi(\gamma(t))$. Let us now
define: for $t\in [a,b]$,
$$\tilde \gamma(t):= \phi^{-1}\left(\gamma_T(t),0\right)\in \partial \Omega.$$
For almost every $t \in (a,b)$, if $\gamma(t)\in \partial \Omega$, $\gamma(t)=\phi^{-1} (\gamma_T(t),0)=\tilde \gamma(t)$, $g(\gamma(t))=g(\tilde \gamma(t))$ and 
\begin{align*}
\vert \gamma'(t)\vert^2 &= \left [\left(\gamma_T,\gamma_N\right)'\right]^{Tr} \,G\left(\gamma_T(t),0\right)\, \left(\gamma_T,\gamma_N\right)'\\
&= \left [\gamma_T'(t)\right]^{Tr}\, \tilde G(\gamma_T(t),0)\, \gamma_T'(t) + G_{NN}\left(\gamma_T(t),0\right)\, \gamma_N'(t)^2\\
&\geq \left [\gamma_T'(t)\right]^{Tr} \,\tilde  G\left(\gamma_T(t),0\right) \,\gamma_T'(t) =\left\vert \tilde \gamma'(t)\right\vert^2,
\end{align*}
where the supersript $^{Tr}$ stands for the transposition operator. 
For almost every  $t \in (a,b)$, if $\gamma(t)\in \Omega$,
\begin{align*}
\vert \gamma'(t)\vert^2 &= \left [\left(\gamma_T,\gamma_N\right)'\right]^{Tr}\, G\left((\gamma_T,\gamma_N)\right) \,\left(\gamma_T,\gamma_N\right)'\\
&= \left [\gamma_T'(t) \right]^{Tr}\,\tilde G\left((\gamma_T,\gamma_N)\right) \,\gamma_T'(t) + G_{NN}\left((\gamma_T,\gamma_N)\right)\, \gamma_N'(t)^2\\
&\geq  (1-\varepsilon)^2 \left [\gamma_T'(t) \right]^{Tr}\,\tilde G\left(\gamma_T(t),0\right) \gamma_T'(t) =(1-\varepsilon)^2\,\left\vert \tilde \gamma'(t)\right\vert^2.
\end{align*}

\begin{sloppypar}
\noindent
\underline{Step 2: Definition of $\gamma_{12}$.} Let $(t_1,t_2)\in [0,1]^2$, with $a\leq t_1<t_2\leq b$. 
\medskip

\noindent
Let us distinguish between two cases.
\begin{itemize}
\item If the set $\left\{t \in [t_1,t_2], \, \gamma(t)\in \partial
    \Omega\right\}$ is  non empty, let us consider $t_1^+:=\inf \left\{t
    \in [t_1,t_2], \, \gamma(t)\in \partial \Omega\right\}$ and
  $t_2^-:=\sup \left\{t \in [t_1,t_2], \, \gamma(t)\in \partial
    \Omega\right\}$. The curve $\gamma_{12}: [t_1,t_2]\to \overline
  \Omega$ is then defined by 
$$
\gamma_{12}(t)= \left\{
    \begin{array}{ll}
        \gamma(t) & \mbox{if } t\in (t_1,t_1^+),  \\
        \tilde \gamma(t) & \mbox{if } t\in (t_1^+,t_2^-),  \\
        \gamma(t) & \mbox{if } t\in (t_2^-,t_2).
    \end{array}
\right.
$$
Observe that for all $t\in (t_1,t_1^+)\cup (t_2^-,t_2)$,
$\gamma(t)= \gamma_{12}(t)$, which implies $g(\gamma(t))\vert
\gamma'(t)\vert=g(\gamma_{12}(t))\vert \tilde
\gamma_{12}'(t)\vert$ almost everywhere in $ (t_1,t_1^+)\cup (t_2^-,t_2)$. Moreover, using
the fact that $A(1-\varepsilon)>1$, for almost every  $t\in (t_1^+,t_2^-)$,
\begin{equation} \label{strictt}
g(\gamma(t))\vert \gamma'(t)\vert \geq A(1-\varepsilon) g(\phi^{-1}(\gamma_T(t),0)) \vert \tilde  \gamma'(t)\vert>g(\tilde \gamma)\vert \tilde  \gamma'(t)\vert=g(\gamma_{12})\vert \gamma_{12}'(t)\vert.
\end{equation}
Therefore \eqref{oo} is satisfied.
\item If the set $\left\{t \in [t_1,t_2], \, \gamma(t)\in \partial
    \Omega\right\}$ is  empty,   which means that $\forall t\in
  [t_1,t_2]$, $\gamma(t)\in \Omega$, then one simply defines the curve $\gamma_{12}: [t_1,t_2]\to \overline \Omega$ by $\gamma_{12}=\gamma$. 
\end{itemize}
In both cases, the curve $\gamma_{12}$ is Lipschitz,
$\gamma_{12}(t_j)=\gamma(t_j)$ for $j\in \{1,2\}$ and \eqref{oo} is
satisfied. Moreover by construction of $\gamma_{12}$, the set
$\partial \left\{ t \in [t_1,t_2], \,  \gamma_{12}(t)\in \partial
  \Omega\right\}$ consists of isolated points in $\left\{ t \in
  [t_1,t_2], \,  \gamma_{12}(t)\in \partial \Omega\right\}$, or is
empty.

\medskip
\noindent
\underline{Step 3: On the strict inequality in~\eqref{oo}.} 
\medskip

\noindent
Assume
that~\eqref{eq:s1s2s3} holds and let us show that the inequality~\eqref{oo} is
strict. Indeed, in that case $t_1^+\le s_1 < s_3 \le t_2^-$ and by continuity of~$\gamma$, there exists $(u_1,u_2)\in (s_1,s_3)^2$ such that
$ u_1<s_2<u_2 $ and   $\gamma([u_1,u_2]) \subset
\Omega$. Thus, the inequality~\eqref{strictt} holds almost everywhere on the open nonempty interval
$(u_1,u_2)$ which implies that $L(\gamma, (u_1,u_2))>
L(\gamma_{12}, (u_1,u_2))$.
This concludes the proof of Proposition~\ref{four}.
\end{sloppypar}
\end{proof}
\noindent
A consequence of the previous proposition is the following result.
\begin{proposition} \label{front}
 Let $x,y \in \overline \Omega$ and assume that \textbf{[H3]} holds. For any Lipschitz curve $\gamma:  [0,1]\to \overline \Omega$ with $\gamma(0)=x$ and $\gamma(0)=y$, there exists $\gamma_1\in A\left(x,y\right)$ such that 
 $L(\gamma,(0,1))\geq L(\gamma_1,(0,1))$. 
\end{proposition}
\begin{proof}
The set $\partial\left \{ t \in [0,1], \, \gamma(t)\in \partial
  \Omega\right\}$ is closed, so its limit points are its non
isolated points. Let us define ${\rm Ad}(\gamma)$ as the set of
limit points of $\partial \left\{ t \in [0,1], \,
  \gamma(t)\in \partial \Omega\right\}$. If ${\rm Ad}(\gamma)$ is
empty, then $\partial \left\{ t \in [0,1], \,  \gamma(t)\in \partial
  \Omega\right\}$ is empty or consists of isolated points in  
$\partial\left \{ t \in [0,1], \,  \gamma(t)\in \partial
  \Omega\right\}$ and since $\partial \left\{ t \in [0,1], \, \gamma(t)\in \partial \Omega\right\}$ is compact, this implies that 
$\gamma \in A\left(x,y\right)$ and Proposition~\ref{front} is thus
proved by simply taking $\gamma_1=\gamma$. 
\medskip

\noindent
If ${\rm Ad}(\gamma)$ is non empty, we will construct a curve $\gamma_1\in A\left(x,y\right)$ such that 
$$L(\gamma,(0,1))\ge L(\gamma_1,(0,1)).$$
Without loss of generality, one can assume that $0$ and $1$ are not in
${\rm Ad}(\gamma)$. Otherwise one could modify $\gamma$ in neighborhoods of $0$ and $1$ 
without increasing $L(\gamma,(0,1))$ and without changing the end
points using
Proposition~\ref{four}. To prove the result, we will show by induction
on $N \geq 1$ the following property $\mathcal P_N$: for all
$\{t_1,\ldots,t_N\} \subset  {\rm Ad}(\gamma)$, denote by
$(a_j,b_j)_{j=1,\ldots,N}$ the open intervals given  by
Proposition~\ref{four} for each $t_i$; then, it is possible to change
$\gamma$  to construct a Lipschitz curve $\gamma_1: [0,1]\to \overline
\Omega$ with $\gamma_1(0)=x$ and $\gamma_1(0)=y$, such that $\gamma_1=\gamma$ on $[0,1]\setminus \left (\bigcup_{j=1}^N(a_j,b_j)\right )^c$, ${\rm
  Ad}(\gamma_1) \cap \bigcup_{j=1}^N(a_j,b_j)=\emptyset $ and
$L\left(\gamma,(0,1)\right)\ge L(\gamma_1,(0,1))$. 
\medskip

\noindent
The first step to prove $\mathcal P_1$ is just a straightforward
application of Proposition~\ref{four} (choosing $t_1=a_1$ and $t_2=b_2$). Now, let us prove 
$\mathcal P_{N+1}$ assuming $\mathcal P_N$. Let us consider
$\{t_1,\ldots,t_N,t_{N+1}\} \subset  {\rm Ad}(\gamma)$ and denote by
$(a_j,b_j)_{j=1,\ldots,N,N+1}$ the open intervals given by Proposition
\ref{four} for each~$t_i$.  Applying $\mathcal P_N$, it is possible to change $\gamma$ to construct a Lipschitz curve $\gamma_1: [0,1]\to
\overline \Omega$  with $\gamma_1(0)=x$ and $\gamma_1(0)=y$, such that
${\rm Ad}(\gamma_1) \cap \bigcup_{j=1}^N(a_j,b_j) =\emptyset$ and
$L\left(\gamma,(0,1)\right)\ge L\left(\gamma_1,(0,1)\right)$. If
$(a_{N+1},b_{N+1})\subset  \bigcup_{j=1}^N(a_j,b_j)$, then $\mathcal P_{N+1}$ holds taking $\gamma_1$. Otherwise, there exist $K\in \mathbb N^*$ and $(q_1,\ldots,q_K,d_1,\ldots,d_K)\in [0,1]^{2K}$ such that
 $$\left(a_{N+1},b_{N+1}\right)\cap \left[  \bigcup_{j=1}^N(a_j,b_j)\right]^c =\bigcup_{i=1}^K\left[(q_i,d_i)\right],$$
with $0<q_1<d_1<q_2<d_2<\ldots<q_K<d_K<1$; the notation $[($ and $)]$
mean that the extremities can or not belong to the interval. In
addition, for $i \in \{1, \ldots,K\}$, $q_i\in \{a_{N+1}\}\cup \{ b_1,\ldots,b_N\}$ and $d_i\in
\{b_{N+1}\}\cup \{ a_1,\ldots,a_N\}$. Then applying  Proposition~\ref{four} to $\gamma_1$ on each interval $(q_i,d_i)\subset
(a_{N+1},b_{N+1})$, one gets that it is possible to construct a Lipschitz curve $\gamma_2$ with $\gamma_2(0)=x$ and $\gamma_2(0)=y$, such that
 $$L(\gamma_1,(0,1))\geq L(\gamma_2,(0,1)), \ {\rm Ad}(\gamma_2) \cap \bigcup_{i=1}^K(q_i,d_i)=\emptyset.$$
If for some $k\in \{1,\ldots, K\}$ and $j\in \left\{1,\ldots,N\right\}$
$q_k=b_j$, then
$q_k$ is isolated from the left in $\left\{t\in [0,1], \,
  \gamma_1(t)\in \partial \Omega\right\}$ from the construction of
$\gamma_1$ (see Proposition~\ref{four}) and,  by construction of
$\gamma_2$, $q_k$ is also isolated from the right in $\left\{t\in [0,1], \,
  \gamma_2(t)\in \partial \Omega\right\}$. Thus, since there exists $s\in [0,q_k)$ such that  $\gamma_2=\gamma_1$ on $[s,q_k]$, one has $q_k\notin {\rm
  Ad}(\gamma_2)$. A similar reasoning holds for the points $d_k$. Thus 
 $${\rm Ad}(\gamma_2) \cap
  \left(a_{N+1},b_{N+1}\right)\cap \left[
    \bigcup_{j=1}^N(a_j,b_j)\right]^c =\emptyset.$$
 This proves that ${\rm Ad}(\gamma_2) \cap \bigcup_{j=1}^{N+1}(a_j,b_j)
 =\emptyset$ and thus $\mathcal P_{N+1}$.
\medskip

\noindent
By induction, we thus have proven $\mathcal P_N$ for all $N\geq
1$. Now, notice that by a compactness argument, 
 there exist $N\geq 0$ and $\{t_1,\ldots,t_N\} \subset  {\rm Ad}(\gamma)$ such that if one denotes by  $(a_j,b_j)_{j=1,\ldots,N}$ the open intervals given by Proposition~\ref{four} for each $t_i$, then
$$ {\rm Ad}(\gamma) \subset \bigcup_{i=1}^N(a_i,b_i).$$
Applying $\mathcal P_N$ yields the desired result.
\end{proof}
\noindent
A direct consequence of Proposition~\ref{front} is the following.
\begin{corollary}\label{equalityagmon}
 Assume that \textbf{[H3]} holds. Then the Agmon distance $d_a$ introduced in Definition~\ref{def.agmon-intro} satisfies \eqref{daA}: for all $\left(x,y\right)\in \overline{\Omega}^2$
 $$d_a\left(x,y\right)=\inf _{\gamma\in A\left(x,y\right)}L\left(\gamma,(0,1)\right).$$
 \end{corollary}
%In all what follows, we will often use the formula~\eqref{daA} to study
%the property of the Agmon distance.
% Corollary~\ref{equalityagmon} thus ensures that, if \textbf{[H3]} holds, one can work with curves in $A\left(x,y\right)$ (see Definition~\ref{defA}) or in $\rm{Lip}\left(x,y\right)$ (see Definition~\ref{def.agmon-intro}) to show the properties of $d_a$, which is the purpose of next Sections. In the following one only works with curves which belong to the set $A\left(x,y\right)$. It is actually easier to work with curves belonging to the set $A\left(x,y\right)$ in order to prove properties of the Agmon distance $d_a$ (see the following sections). However to compute the Agmon distance $d_a$, it is more convenient to be allowed to work with Lipschitz curves which are easier to identify.

\subsection{First properties of the  Agmon distance} \label{properties}
In this section, we aim at giving the basic properties of the Agmon distance. 
%In all this section, \textbf{[H3]} is assumed to hold. We make this assumption to use Corollary~\ref{equalityagmon} even if some results below hold without this assumption. 

\subsubsection{Upper bounds on $d_a$ and topology of $(\overline{\Omega},d_a)$}
\begin{proposition} \label{propo5}
The Agmon distance $(x,y)\in \overline \Omega \times\overline \Omega \mapsto d_a(x,y)$ is symmetric and satisfies the triangular inequality. Moroever,  
 there exists a constant $C$ such that for all $x,y \in \overline \Omega$,
\begin{equation} \label{eq:ineq2}
 d_a\left(x,y\right) \leq C \vert x-y \vert .
\end{equation} 
For any fixed $y \in \overline \Omega$,  $x \in \overline \Omega\mapsto
d_a\left(x,y\right)$ is Lipschitz. Its gradient is well
defined almost everywhere and satisfies for $y \in \overline \Omega$ and for almost every $x \in \Omega$,
\begin{equation} \label{eq:eqq}
\vert \nabla_x d_a\left(x,y\right) \vert  \leq \vert \nabla f (x) \vert.
\end{equation} 
Moreover, if  \textbf{[H3]}  holds, for all $x,y \in \overline \Omega$, we have
\begin{equation} \label{eq:ineq}
\vert f(x)-f(y)\vert \leq d_a\left(x,y\right).
\end{equation} 

\end{proposition} 
\begin{proof}
The  inequality (\ref{eq:ineq2}) is proved below in
Lemma~\ref{lemma2}.
 The proof of (\ref{eq:eqq}) is standard, see for instance~\cite[p. 53]{dimassi-sjostrand-99}. For the ease of the reader, let us recall the proof. 
% Let us conclude the proof.
Let $y \in \overline \Omega$ and
$x \in \Omega$. Since $\Omega$ is open, there exists an open ball $B \subset
\Omega$ with center $x$. Let $z \in B$. For $t\in [0,1]$, the path $\gamma(t)=tz+(1-t)x$ is included in $B$. Then, one obtains
\begin{align*}
\vert d_a\left(x,y\right)-d_a(z,y) \vert \leq d_a(x,z) &\leq \vert x-z\vert \int_0^1 g(tz+(1-t)x) \, dt\\ 
&=\vert x-z\vert \int_0^1 \vert \nabla f \vert (x+t(z-x)) \, dt.
\end{align*}
Since $f$ is smooth, up to considering a smaller ball $B$ centered at
$x$, there exists a constant $c>0$ such that for all $z\in B$, for all
$t \in [0,1]$, $\vert \nabla f \vert (x+t(z-x)) \le \vert  \nabla f \vert (x)+c\left\vert x-z\right\vert$. 
Thus, for all $z \in B$, it holds $\vert d_a\left(x,y\right)-d_a(z,y)\vert \leq \left\vert x-z\right\vert \left( \vert  \nabla f \vert(x)+c\left\vert x-z\right\vert  \right)$. 
As a consequence, for any fixed $y \in \overline \Omega$ and for almost $x \in \Omega$ one
gets (\ref{eq:eqq}), by considering the limit $z \to x$ in the
previous inequality. 
Let us now prove the inequality~(\ref{eq:ineq}). For any $\gamma \in
A\left(x,y\right)$, using Lemma~\ref{h}, one has:
\begin{align*}
\vert f(x)-f(y) \vert &=\left\vert \int_0^1 \frac{d}{dt} (f\circ \gamma )\, dt \right\vert \\
&=\left\vert\, \int_{\left\{ t \in [0,1], \ \gamma (t)\in \Omega \right\}}  (\nabla f)(\gamma )\cdot \gamma ' \, dt+  \int_{{\rm int} \left\{ t \in [0,1], \ \gamma (t)\in \partial \Omega\right\} }  (\nabla_T f)(\gamma )\cdot \gamma '\, dt \right\vert \\
&\leq \int_{0}^1  g(\gamma(t) ) \, \left\vert \gamma '(t)\right\vert \, dt.
\end{align*}
Therefore (\ref{eq:ineq}) is proved by taking the infimum over
$\gamma \in A\left(x,y\right)$ in the right-hand side (see Corollary~\ref{equalityagmon}). 
\end{proof} 
\noindent
Let us now give the proof of~\eqref{eq:ineq2}.
\begin{lemma} \label{lemma2}
 The function $\left(x,y\right) \in \overline \Omega \times \overline
 \Omega \mapsto  d_a\left(x,y\right)$ is bounded and satisfies 
 \begin{equation} \label{eq:Li}
\sup_{(x,y) \in \overline \Omega\times \overline \Omega, \, x\neq y} \frac{ d_a\left(x,y\right)}{ \ \vert x-y\vert }<\infty.
\end{equation} 
 \end{lemma} 
\begin{proof}
 Let us first prove by contradiction that $\left(x,y\right) \in
 \overline \Omega \times \overline \Omega \mapsto
 d_a\left(x,y\right)$ is bounded. Let us assume that there exists a
 sequence $(x_k,y_k)_{k\geq1}\in \overline \Omega \times \overline
 \Omega$ such that for all $k\geq 1$, 
\begin{equation}\label{eq:contradict_1}
d_a(x_k,y_k)\geq k.
\end{equation} Up to the
 extraction of a subsequence, it can be assumed that $\lim_{k \to
   \infty} (x_k,y_k)= \left(x,y\right) \in \overline \Omega \times
 \overline \Omega$ (the convergence being for the Euclidean metric).  Notice that 
\begin{equation}\label{eq:triang_ineq}
d_a(x_k,y_k)\leq d_a(x_k,x)+d_a\left(x,y\right)+d_a(y_k,y).
\end{equation}
Let us consider $d_a(x_k,x)$. If $x\in \Omega$ then there exist an open  ball
$B\subset \Omega$ centered on $x$ and an integer $N$ such that for all
$ k\geq N$, $x_k \in B$ and therefore $\gamma(t)=tx_k+(1-t)x \in B$
for all $t\in (0,1)$. Then, by definition of the Agmon distance, for all $k\geq 1$,
$$d_a(x_k,x) \leq \Vert g\Vert_{L^{\infty}(B)}\,  \vert x-x_k \vert.$$
If $x\in \partial \Omega$ then there exist  $r>0$ and a $C^{\infty}$
bijective map $\phi \ : \ B(x,r) \to B(0,1)$ such that
$\phi(x)=0$, $\phi ( B(x,r) \cap \partial \Omega)=Q_0$ and $\phi(
B(x,r) \cap \overline \Omega)=Q_-$, where $Q_0:=\{y=(y_1,\ldots,y_d),
\ \vert y \vert \leq 1, \ y_d=0\}$ and $Q_-:=\{y=(y_1,\ldots,y_d), \
\vert y \vert \leq 1,\ y_d\leq0\}$. Moreover, there exists $N$ such
that for all $ k\geq N$, $x_k \in B(x,r) \cap \overline \Omega$. Now,
for any $k \ge N$, let us consider  $\gamma(t)=\phi^{-1}((1-t)\phi(x_k)+t\phi(x))$. 
Notice that $\gamma \in A(x_k,x)$ and satisfies $\gamma([0,1]) \subset
B(x,r)$. Moreover $c(t)=\phi\left(\gamma(t)\right)=(1-t)\phi(x_k)+ t \phi(x)
\in  Q_- $ for $ t\in [0,1]$. Then, one has:
\begin{align*}
d_a(x_k,x) \leq \int_0^1  g(\gamma) \vert \gamma' \vert  & \leq \|g\|_{L^\infty(B(x,r))}  \int_0^1 \vert \gamma' \vert \\
& = \|g\|_{L^\infty(B(x,r))}  \int_0^1 \vert  {\rm Jac}(\phi^{-1})(c) c'
  \vert \\
& \leq \|g\|_{L^\infty(B(x,r))}  \|  {\rm Jac}(\phi^{-1})\|_{L^\infty(B(0,1))} |\phi(x_k)-\phi(x)|,
\end{align*}
and therefore, since $\phi$ is Lipschitz,
\begin{equation}\label{eq:d_xk_x}
d_a(x_k,x) \leq C \vert x_k-x\vert,
\end{equation}
where $C$ is a constant independent of $k \ge N$.
% Then,
% for all $k\geq 1$, \comment{check. la
%   def de ton ensemble etait delirante}
% \begin{align*}
% d_a(x_k,x) &\leq \inf\left\{  \int_0^1  g(\gamma) \vert \gamma' \vert,
%   \ \gamma \in A(x_k,x), \gamma([0,1]) \subset  B(x,r) \right\}\\
% & \leq C\inf \left\{  \int_0^1 \vert \gamma' \vert, \ \gamma \in
%   A(x_k,x), , \gamma([0,1]) \subset  B(x,r)  \right\} \\
% & \leq C\inf \left\{  \int_0^1 \vert  {\rm Jac}(\phi^{-1})(c) c'
%   \vert, \  c=\phi(\gamma), \gamma \in A(x_k,x) , \gamma([0,1]) \subset  B(x,r) \right\} \\
% & \leq C\inf \left\{  \int_0^1 \vert  c' \vert, \ c=\phi(\gamma), \gamma \in A(x_k,x) , \gamma([0,1]) \subset  B(x,r) \right\},
% \end{align*}
% where $C$ is a constant which only depends on some $L^\infty$ norm of $g$ and the Jacobian
% matrix of $\phi^{-1}$. Now, one can take $\gamma$ defined by
% $$\gamma(t)=\phi^{-1}((1-t)\phi(x_k)+t\phi(x)).$$
% \comment{tu avais inverse les bornes je crois}
% Notice that $\gamma \in A(x_k,x)$ and satisfies $\gamma([0,1]) \subset
% B(x,r)$, $c(t)=\phi\left(\gamma(t)\right)=(1-t)\phi(x_k)+ t \phi(x) \in  Q_- $ for $ t\in [0,1]$ and therefore, since $\phi$ is Lipschitz,
% \begin{align*} 
% d_a(x_k,x) &\leq C \vert \phi(x_k)-\phi(x)\vert \leq C \vert x_k-x\vert.
% \end{align*}
This shows that $d_a(x_k,x)$ is bounded. The same reasoning shows that
$d_a(y,y_k)$ is bounded. This yields a contradiction, considering the
inequality~\eqref{eq:triang_ineq} and~\eqref{eq:contradict_1}. 
\medskip

\noindent
 To show (\ref{eq:Li}), one proceeds in the same way. Assume that there
 exists a sequence $(x_k,y_k)\in \overline \Omega \times \overline
 \Omega$ such that 
\begin{equation}\label{eq:contradict_2}
d_a(x_k,y_k)\geq k \left\vert
   x_k-y_k\right\vert.
\end{equation}
Up to the extraction of a subsequence, it can
 be assumed that  $\lim_{k \to \infty} (x_k,y_k)= (x,y) \in \overline \Omega \times \overline \Omega$. If $x\neq y$, then, for $k$ sufficiently large
$$\frac{d_a(x_k,y_k)}{ \vert x_k-y_k\vert }\leq 2 \, \,
\frac{\sup_{(x,y)\in \overline \Omega}d_a(x,y)}{ \vert x-y\vert
}<\infty$$
which contradicts~\eqref{eq:contradict_2}.
If $x=y\in \Omega$, then, for all $k$ sufficiently large, the curve
$\gamma(t)=tx_k+(1-t)y_k$ is with values in $\Omega$ and therefore for
all $k$ sufficiently large, $d_a(x_k,y_k) \leq \Vert g \Vert_{L^{\infty}(\overline \Omega)}
\vert y_k-x_k \vert$. 
This again leads to a contradiction when $k \to \infty$.
Finally, if $x=y\in \partial \Omega$, using the same reasoning as
above to prove~\eqref{eq:d_xk_x}, one has the existence of a constant
$C$ such that for all $k$ sufficiently large, $
d_a(x_k,y_k) \leq C \vert x_k-y_k\vert$, 
which again contradicts~\eqref{eq:contradict_2}. This concludes
the proof of Lemma~\ref{lemma2}.
\end{proof}
\noindent
A consequence of the previous lemma is the following proposition.
\begin{proposition} \label{opoo} Assume that \textbf{[H1]} holds. 
The space $\left(\overline \Omega, d_a\right)$ is a compact separated
metric space. Moreover the topology of the metric space
$\left(\overline \Omega, d_a\right)$ is equivalent to the topology
induced by the Euclidean metric on $\overline{\Omega}$.
\end{proposition}
\begin{proof} 
%It is clear that $d_a\geq0$, and for all $(x,y,z) \in \overline\Omega
%^3$, $d_a\left(x,y\right)\leq
%d_a\left(x,z\right)+d_a\left(z,y\right)$. 
Let us show that for any
$(x,y) \in \overline\Omega \times \overline \Omega$, if $x\neq y$ then
$d_a\left(x,y\right)>0$.  Let us denote by $d_e$ the geodesic
distance on $\overline \Omega$ for 
  the  Euclidean metric: for all $x,y \in \overline{\Omega}$,
\begin{equation}\label{eq:geodesic_euclidean}
d_e\left(x,y\right):=\inf _{\gamma}\int_0^1  \left\vert \gamma'(t) \right\vert \, dt,
\end{equation}
where the infimum is taken over all the paths $\gamma :[0,1] \to \overline \Omega $ which are Lipschitz  
with $\gamma(0)=x$ and $\gamma(1)=y$. Since \textbf{[H1]} holds, the functions $f$
and $f|_{ \partial \Omega}$ have a finite number of critical points, and thus,
there exist $0<r_1<r_2<d_e(x,y)$ such that the infimum of $g$ on the
set $C(r_1,r_2):= \left\{ z\in \overline \Omega, \,   r_1<d_e(x,z)<r_2 \right\}$ is positive i.e. $c(r_1,r_2):=\inf_{C(r_1,r_2)} g>0$. For any path $\gamma\in A\left(x,y\right)$, one has
$$   \int_0^1 \vert \gamma'(t)\vert \,g(\gamma(t)) \, dt \geq c(r_1,r_2) r(C(r_1,r_2)),$$
where $r(C(r_1,r_2)):=\inf_{z\in C(r_1,r_2)}\sup_{y\in
  C(r_1,r_2)}d_e(z,y)>0$. Then $d_a\left(x,y\right)>0$. If $x=y$, it
is clear that $d_a\left(x,y\right)=0$ since
$L\left(\gamma,(0,1)\right)=0$ where $\gamma(t)=x$ for all $t\in
[0,1]$. This shows that $\left(\overline \Omega, d_a\right)$
  is separated. 

The fact that $\left(\overline \Omega, d_a\right)$ is compact comes
from the inequality~\eqref{eq:ineq2} proved in Lemma~\ref{lemma2}.
Indeed, since  $\left(\overline \Omega, d_a\right)$ is a metric space,
it is sufficient to prove the sequential compactness. Let $(x_n)_{n
  \ge0}$ be a sequence in $\overline \Omega$. Since $\overline{\Omega}$ is
compact for the Euclidean metric, one can extract a converging
subsequence $(x'_n)_{n \ge 0}$ for the Euclidean  metric. From~\eqref{eq:ineq2}, this
subsequence is also converging for $d_a$, which ends the proof.
\medskip

\noindent
Let us finally prove the equivalence of the topologies on
$\overline{\Omega}$. From Lemma~\ref{lemma2}, it is obvious that if a
sequence $(x_n)_{n \ge 0}$ converges to $x$ in $\overline{\Omega}$ for the Euclidean metric, then
$d_a(x_n,x)$ converges to $0$. Conversely, let us assume that
$(x_n)_{n \ge0}$ is a sequence of  $\overline{\Omega}$ such that
$d_a(x_n,x)$ converges to $0$, for a point $x \in
\overline{\Omega}$. Since $\overline{\Omega}$ is compact for the
Euclidean metric, it is enough to show that $x$ is the only limit
point of the sequence to show that $(x_n)_{n \ge0}$ converges to $x$
for the Euclidean metric. From Lemma~\ref{lemma2}, any limit point $y$ for
the Euclidean metric is also an limit point for the Agmon distance,
and thus, since $\left(\overline \Omega, d_a\right)$ is a separated
space, $y=x$. This concludes the proof.
\end{proof}
\noindent
Notice that from the proof, it is obvious that the topology is separated as soon as $f$ and $f|_{\partial \Omega}$ have a finite number of critical points, which is a weaker assumption than~\textbf{[H1]}.
\medskip

\noindent
Finally, the following lemma will be useful in the following. 
\begin{lemma} \label{inca} Assume that  \textbf{[H3]}  holds.
Let $I\subset \mathbb R$ be an interval and $\gamma: I \to
\overline{\Omega}$ a Lipschitz curve  such that $\partial \{t\in I,
\ \gamma(t)\in \partial \Omega\}$ is finite and such that $x:=\lim_{t\to \left(\inf I\right)^+} \gamma(t)$ and $y:=\lim_{t\to \left(\sup I\right)^-} \gamma(t)$ exist. Then one has
$$d_a\left(x,y\right)\leq L\left(\gamma, I\right).$$
\end{lemma}
\begin{proof} Let $(a,b)\in I^2$ with $a<b$ and define for $u\in
  [0,1]$, $\gamma_{ab}(u)=\gamma(a+u(b-a))$. Then $\gamma_{ab}\in
  A(\gamma(a),\gamma(b))$.  By definition of the Agmon distance (see
  Definition~\ref{def.agmon-intro}), $d_a(a,b)\leq
  L\left(\gamma_{ab},(0,1)\right)=L\left(\gamma,(a,b)\right)$. Taking
  the limits $a\to \left(\inf I\right)^+$,  $b\to \left(\sup
    I\right)^-$ and using the continuity of the Agmon distance, one
  obtains that $d_a\left(x,y\right)\leq L\left(\gamma,I\right)$. Lemma
  \ref{inca} is proved.\end{proof}
  \noindent
As a simple consequence of this lemma, we have the following simple
remark. 
\begin{remark}\label{rem:agmon_distance_to_min}
Let $x^*$ be a local minimum of $f$ (from \textbf{[H3]}, $x^*\in \Omega$). Then, for all $x$ in the basin of
attraction of $x^*$ for the dynamics 
\begin{equation}\label{eq:agmon_distance_to_min_dyn}
\gamma'=\left\{
\begin{aligned}
  - \nabla f(\gamma)  & \text{ in  }  \Omega  \\ 
 - \nabla_T f(\gamma) & \text{ on  } \partial \Omega 
\end{aligned}
\right.
\end{equation} it holds $x \in \Omega$ and
\begin{equation}\label{eq:agmon_distance_to_min}
d_a(x^*,x)=f(x)-f(x^*).
\end{equation}
Indeed, from~\eqref{eq:ineq}, we already have $d_a(x^*,x)\ge
f(x)-f(x^*)$. To prove the reverse inequality, from Lemma~\ref{inca},
it is enough to exhibit a Lipschitz curve  $\tilde \gamma: I \to
\overline{\Omega}$ such that $\partial \{t\in I,
\ \tilde \gamma(t)\in \partial \Omega\}$ is finite, $L(\tilde \gamma,I)=f(x)-f(x^*)$ and $\lim_{t\to \left(\inf
    I\right)^+}\tilde  \gamma(t)=x^*$ and $\lim_{t\to \left(\sup I\right)^-}
\tilde \gamma(t)=x$. Such a curve is given on the interval $I=(-\infty,0]$ by   $\tilde \gamma: t\in I\mapsto \gamma(-t)$ where $\gamma$ is the solution to~\eqref{eq:agmon_distance_to_min_dyn}
with initial condition $\gamma(0)=x$. Notice that if $\exists t_0$ such that $\tilde \gamma(t_0) \in \partial \Omega$, then $\forall t \le t_0$, $\tilde \gamma(t) \in \partial \Omega$. Thus, since $\lim_{t\to -\infty}
\tilde  \gamma(t)=x^*\in \Omega$, one has: for all $t\leq 0,\tilde   \gamma(t)\in \Omega$. Therefore,
\begin{align*} 
f(x)-f(x^*)&=\int_{-\infty }^0\frac{d}{dt} f \circ \tilde \gamma(t) \, dt  =\int_{-\infty }^0 \!\!  \nabla f (\tilde \gamma(t)) \cdot \tilde \gamma\,'(t) \, d t=\int_{-\infty }^0\!\!  \vert \nabla f (\tilde \gamma(t)) \vert^2 dt\\
&=\int_{-\infty }^0 \vert \nabla f (\tilde \gamma(t))\vert \left\vert\tilde \gamma\, '(t)\right\vert \, d t=\int_{-\infty }^0 g\left(\tilde \gamma(t)\right) \left\vert \tilde \gamma\, '(t) \right\vert \, d  t\\
&= \lim_{t\to -\infty} L(\tilde \gamma, (t,0)).
\end{align*}
This concludes the proof of~\eqref{eq:agmon_distance_to_min}.
Notice that for $\ve>0$ small enough, the set $f^{-1}\left (  [f(x^*),f(x^*)+\ve)   \right)\cap B(x^*,\ve)\subset \Omega$ is a neighborhood of $x^*$ which is included  in the basin of
attraction of $x^*$ for the dynamics~\eqref{eq:agmon_distance_to_min_dyn}. Therefore~\eqref{eq:agmon_distance_to_min} holds in a neighborhood of $x^*$.
\end{remark}

\subsubsection{A lower bound on the Agmon distance} \label{lower bound}
In this section, easily computable lower bounds on the Agmon distance are provided. This is in particular useful to check if the hypothesis~\eqref{hypo1} appearing in Theorem~\ref{TBIG0} is satisfied, see for example
Section~\ref{sec:numeric}. 

\begin{proposition} \label{gradient}
Let $z\in \overline \Omega$ and denote by $W$ and $W'$ two  closed neighborhoods of $z$ in $\overline \Omega$ with $W\subset W'$. Define  
\begin{equation}\label{eq.alpha_agmon}
\alpha:=\inf \{d_e\left(x,y\right),  \, x\in \overline \Omega\setminus W' , \, y \in W \},
\end{equation}
where $d_e$ denotes the geodesic distance for the Euclidean metric, see~\eqref{eq:geodesic_euclidean}.
Assume that $\alpha>0$ and that there exists $K>0$ such that
\begin{equation*} \label{eq:HVV}
\inf_{x\in \overline{W'\setminus  W}} g(x) > \frac{ K }{ \alpha },
\end{equation*}
where $g$ has been introduced in Definition~\ref{L}.
 Then, for all set $B \subset \overline{\Omega}$ such that $B\cap \overline{W'}=\emptyset$,  $$\inf_{y\in B} d_a(z,y) >K,$$ where $d_a$ is the Agmon distance (see Definition~\ref{def.agmon-intro}).
 \end{proposition}
 \begin{proof}
By assumption, there exists $\varepsilon>0$ such that
\begin{equation*} 
\inf_{x\in \overline{W'\setminus  W}} g(x) \geq \frac{ K }{ \alpha }+ \varepsilon.
\end{equation*}
 Let $y \in B$ and $\gamma\in {\rm Lip}(z,y)$. Let us define 
$$t_2=\inf\{ t\in [0,1], \, \gamma(t)\notin W'\}, \ t_1=\sup\{ t\in [0,1],\, t<t_2,\, \gamma(t) \in W\}.$$ 
Since $\gamma$ is continuous and $\alpha>0$, it holds $0<t_1<t_2<1$ and one has $\gamma(t)\in \overline{ W'\setminus  W}$ for all $t\in [t_1,t_2]$, $\gamma(t_1)\in  \overline W=W$ and  $\gamma(t_2)\in \overline{\overline \Omega\setminus W'}$. Then, one has:
 \begin{align*}
 \int_0^1  g\left(\gamma(t)\right) \left\vert \gamma'(t) \right\vert \, dt &\geq  \int_{t_1}^{t_2}  g\left(\gamma(t)\right) \left\vert \gamma'(t) \right\vert \, dt\\
 &\geq \left(\frac{ K }{ \alpha }+\varepsilon\right) \int_{t_1}^{t_2}  \left\vert \gamma'(t) \right\vert \, dt\\
  &\geq \left(\frac{ K }{ \alpha }+\varepsilon\right) \alpha
  =K+\varepsilon \alpha.
 \end{align*}
Since $\varepsilon \alpha$ is independent of $y\in  B$ and since
$\varepsilon \alpha$ is also independent of the curve $\gamma$, one can take the infimum over $\gamma$ and $y\in B$. Thus $\inf_{y\in B} d_a(z,y) >K$. 
 \end{proof}

 \noindent
We now give a simple sufficient condition for the
hypotheses~\eqref{hypo1} to hold,  in the case where $f|_{\partial
  \Omega}$ has only two local minima. This result is based on
Proposition~\ref{PoPo} that will be proven in Section~\ref{sec:equality_agmon} below and which shows that
$$d_a(z_1,z_2)>f(z_2)-f(z_1).$$
In particular, the condition stated in the following proposition has been used in Section~\ref{sec:numeric} in order to check 
  hypothesis~\eqref{hypo1}.
\begin{proposition} \label{agmonz1}
Assume that \textbf{[H1]} and \textbf{[H3]} hold and assume in
addition that $f|_{\partial \Omega}$ has only two local minima $z_1$
and $z_2$ (with $f(z_1)\leq f(z_2)$) on $\partial \Omega$. Then, if $z_2$ is the only global minimum of
$f|_{\partial \Omega}$ on $B_{z_1}^c$, one has
$$\inf_{z\in B_{z_1}^c} d_a(z_1,z) > f(z_2)-f(z_1).$$
 \end{proposition}
\begin{proof}
Proposition~\ref{PoPo}  and the continuity of the Agmon distance
ensure that there exist an open ball $B_2\subset B_{z_1}^c$ centered at $z_2$, and $\ve>0$ such that for all $z\in B_2$
$$ d_a(z_1,z) \geq  f(z_2)-f(z_1) +\ve.$$
Since $z_2$ is the only global minimum of
$f|_{\partial \Omega}$ on $B_{z_1}^c$, there exists $\ve'>0$, such that for all $z\in B_{z_1}^c \setminus B_2$, $f(z)\geq f(z_2) +\ve'$. In addition, from the  inequality~(\ref{eq:ineq}), for all $z\in B_{z_1}^c \setminus B_2$, it holds
$$ d_a(z_1,z) \geq  f(z)-f(z_1) \geq f(z_2)-f(z_1) +\ve'.$$
Consequently $\inf_{z\in B_{z_1}^c} d_a(z_1,z) > f(z_2)-f(z_1)$. 
\end{proof}

\subsection{Agmon distance near critical points of $f$ or
  $f|_{ \partial \Omega}$ and eikonal equation} \label{eiko_agmon}

We will show that the Agmon distance $d_a$ locally solves the eikonal equation in a
neighborhood of any critical point of
$f|_{\partial \Omega}$ or $f$ (or equivalently, any point $x$ such that $g(x)=0$, see~\eqref{eq:def_g}).

\subsubsection{The Agmon distance near critical points of $f$}

\begin{proposition} \label{existe}Let us assume that \textbf{[H1]} holds.
Let $x^*\in \Omega$ be such that $\nabla f(x^*)=0$. Let us denote by
$(\mu_1,\ldots,\mu_d)\in (\mathbb R^*)^d$ the eigenvalues of the
Hessian of $f$ at $x^*$. Then there exist a neighborhood $V^*$ of
$x^*$ in $ \Omega$ and a $C^{\infty}$ function
$\Phi:V^* \to \R$ such that
\begin{equation} \label{eq:eikonalll}
\left\{
\begin{aligned}
\vert \nabla \Phi \vert ^2&=\vert \nabla f \vert ^2,\\
\Phi(x_1,\ldots,x_d)&=\frac 12\sum_{i=1}^d \vert \mu_i \vert \, \left(x_i-x_i^*\right)^2 +O\left(\vert x-x^*\vert^3\right).
\end{aligned}
\right.
\end{equation} 
Moreover, one has the following uniqueness result: if $\widetilde \Phi$ is
a $C^\infty$ real valued function defined on a neighborhood
$\widetilde{V^*}$ of $x^*$ satisfying~\eqref{eq:eikonalll}, then $\widetilde
\Phi=\Phi$ on $V^* \cap\widetilde{V^*}$. 
\end{proposition}
\noindent
Let us notice that $\Phi(x^*)=0$. In addition, up to choosing a smaller
neighborhood $V^*$ of $x^*$, one can assume that $\Phi$ is  positive on $V^*\setminus \{x^*\}$. The
point $x^*$ is then a  non degenerate minimum of $\Phi$. 
\begin{proof}
The proof is made in \cite[Proposition 2.3.6]{helffer2006semi} in the more general setting
where $\vert \nabla f\vert^2$ is replaced in (\ref{eq:eikonalll}) by a
smooth positive function $W$ around a non degenerate minimum $y^*$ of
$W$ such that $W(y^*)=0$. Here $W= \vert \nabla f \vert^2$ and
$y^*=x^*$. This leads to $\nabla W= 2{\rm Hess }f \left(\nabla
  f\right)$ and thus ${\rm Hess  }W(x^*)=2 \left({\rm Hess }
  f\right)^2 (x^*)$ is a non degenerate matrix. Therefore $x^*$ is indeed a non degenerate minimum of $W= \vert \nabla f \vert^2$. 
\end{proof}

\begin{proposition}\label{da1} Let us assume that \textbf{[H1]} and \textbf{[H3]}  hold.
Let $x^*\in \Omega$ be such that $\nabla f(x^*)=0$. Then there exists a neighborhood $U^*$ of $x^*$ in $\Omega$ such that for all $ x\in U^*$
\begin{equation} \label{e2} 
d_a (x^*,x)=\Phi(x),
\end{equation}
where $\Phi$ is the smooth solution to (\ref{eq:eikonalll}) and $d_a$ is the Agmon distance.
\end{proposition}
\noindent
 For the ease of the reader, let us give the proof of Proposition~\ref{da1} which is similar to the proof of~\cite[Proposition A.1]{dimassi-sjostrand-99}.
\begin{proof}
\begin{sloppypar}
Notice that hypothesis \textbf{[H3]} allows us to use Corollary~\ref{equalityagmon}. Let $\Phi$ be a smooth solution to (\ref{eq:eikonalll}) on a neighborhood
 $V^*$ of $x^*$, as defined in Proposition~\ref{existe} and such that
 $\Phi$ is positive on $V^*\setminus \{x^*\}$. There exists $\varepsilon>0$ such that
 $U^*:=\Phi^{-1}([0,\varepsilon))\subset V^*$ is a neighborhood of
 $x^*$ (consider for example $\varepsilon=\inf
 \left\{ \Phi(x), x\in V^* \setminus B(x^*,r)  \right\}>0$ where $r>0$
 is such that $B(x^*,2r) \subset  V^*$). 
\end{sloppypar}
% . This is possible by taking for instance,
%  for $r>0$ such that $\partial B(x^*,r) \subset  V^*$, $\ve=\inf
%  \left\{ \Phi(x), x\in \partial B(x^*,r)  \right\}>0$ ($\Phi$ is
%  positive on $\Omega  \cap V^*\setminus \{x^*\}$), where $B(x^*,r)$ is
%  the ball centered on $x^*$ and of radius $r$ for the geodesic
%  distance.
Let us first prove that for $x\in U^*$, $\Phi(x)\le d_a(x,x^*)$.
For $x\in U^*$, one has $\Phi(x)< \ve$ and thus
$\Phi^{-1}([0,\Phi(x)))\subset U^*$. Let $\gamma$ belong to
$A\left(x^*,x\right)$. Let us define the time $t_0:=\inf \left\{t\in
  [0,1], \ \gamma(t)\notin \Phi^{-1}([0,\Phi(x)))\right\}$. By
continuity of the curve $\gamma$, one has $t_0>0$,
$\Phi(\gamma(t_0))=\Phi(x)$ and for all $t\in [0,t_0)$, $\gamma(t)\in
\Phi^{-1}([0,\Phi(x)))\subset U^*$. Thus, since the curve $\gamma$ is Lipschitz and since for all $t\in [0,t_0)$, $\gamma(t)\in \Omega$, one has
\begin{align*}
\Phi(x)&= \int_0^{t_0} \frac{ d   }{d  t   } \Phi (\gamma)(t) \, d  t=\int_0^{t_0} \nabla \Phi (\gamma(t)) \cdot \gamma'(t) \, dt \\
&\leq \int_0^{t_0} \vert \nabla \Phi (\gamma(t)) \vert \left\vert \gamma'(t) \right\vert \  dt \\
&\leq \int_0^{t_0} \vert \nabla f (\gamma(t)) \vert \left\vert \gamma'(t) \right\vert \  dt \le \int_0^{1} g\left(\gamma(t)\right) \left\vert  \gamma'(t) \right\vert \, dt  = L\left(\gamma,(0,1)\right).
\end{align*}
Taking the infimum on the right-hand side over $\gamma\in
A\left(x^*,x\right)$, one gets $\Phi(x)\leq d_a\left(x^*,x\right)$,
for all $x\in U^*$. Let us now prove the reverse inequality: for $x\in U^*$, $\Phi(x)\ge d_a(x,x^*)$. For $x\in
U^*$, let us define a curve $\gamma:\R_+ \to U^*$ by
$$
\forall t \ge 0, \, \gamma'(t)=- \nabla \Phi \left(\gamma(t)\right)
\text{ and } \gamma(0)=x.
$$
Since the function $t\mapsto \Phi\left(\gamma(t)\right)$ is decreasing, the curve $\gamma$ always belongs to $U^*$ and is defined on $\mathbb R_+$.  Moreover $\gamma$ is $C^{\infty}$ and satisfies $\lim_{t\to +\infty}\gamma(t)=x^*$. 
Since $\gamma$ is with values in $U^* \subset \Omega$, one has
\begin{align*} 
-\Phi(x)&=\int_0^{+\infty }\frac{d}{dt} \Phi \circ \gamma(t) \, dt  =\int_0^{+\infty }  \nabla \Phi (\gamma(t)) \cdot \gamma'(t) \, d t= -\int_0^{+\infty}  \vert \nabla \Phi (\gamma(t)) \vert^2 dt\\
&=-\int_0^{+\infty } \vert \nabla \Phi (\gamma(t))\vert \left\vert\gamma'(t) \right\vert \, d t=- \int_0^{+\infty } g\left(\gamma(t)\right) \left\vert  \gamma'(t) \right\vert \, d  t=- \lim_{t\to +\infty} L(\gamma, (0,t)).
\end{align*}
Then, thanks to Lemma~\ref{inca}, it holds: $d_a\left(x,x^*\right)\leq L(\gamma, (0,+\infty))=\Phi(x)$. 
Therefore $\Phi(x)=d_a\left(x^*,x\right)$ for all $x\in U^*$.
\end{proof}
\begin{remark}
\noindent 
Let us \label{curve1} mention a simple consequence of the previous proof that will be
useful in the following. If $x^*\in \Omega$ is such that $\nabla
f(x^*)=0$, there exists a neighborhood $U^*$ of $x^*$ in $\Omega$ such
that for all $x\in U^*$, there exists a $C^{\infty}$ curve $\gamma: \mathbb R_+\to  \Omega$ such that 
$$d_a (x^*,x)=\int_0^{+\infty} \vert \nabla f\left(\gamma(t)\right) \vert \left\vert \gamma'(t) \right\vert \, dt,$$
with $\gamma(0)=x$ and $\lim_{t\to +\infty}\gamma(t)=x^*$. The curve $\gamma$ is defined by 
\begin{equation}\label{eq.int-gamma'-eik}
\gamma'(t)=-\nabla\Phi\left(\gamma(t)\right), \quad \gamma(0)=x,
\end{equation}
where $\Phi$ solves (\ref{eq:eikonalll}).  In addition $\{ t\in [0,\infty), \ \gamma(t)\in \partial \Omega\}$ is
empty.
\end{remark}

\subsubsection{The Agmon distance near
  critical points of $f|_{ \partial \Omega}$}

Let us first define the Agmon distance in the boundary $\partial \Omega$.
\begin{definition} \label{defff}
The Agmon distance between $x\in \partial \Omega$ and $y \in \partial \Omega$ in the boundary $\partial \Omega$ is defined by
\begin{equation} \label{eq:definition2}
d_a^{\partial \Omega}\left(x,y\right)=\inf _{\gamma}\int_0^1  \left\vert \nabla_Tf \left(\gamma(t)\right)\right\vert \left\vert \gamma'(t) \right\vert \, dt,
\end{equation} 
\label{page.dapaomega} 
where the infimum is taken over all the paths $\gamma :[0,1]
\to \partial \Omega$ which are Lipschitz with
$\gamma(0)=x$ and $\gamma(1)=y$. 
\end{definition}
Similarly to Remark~\ref{rem:agmon_distance_to_min}, one has:
\begin{remark}\label{rem:dadomega}
If $x^*$ is a local minimum of $f|_{\partial \Omega}$, one has
$d_a^{\partial \Omega}(x^*,x)=f(x)-f(x^*)$ for all $x \in \partial
\Omega$ which is in the
basin of attraction of $x^*$ in $\partial \Omega$ for the gradient
dynamics $\gamma'=-\nabla_T f(\gamma)$.
\end{remark}
The next proposition
is the equivalent of Proposition~\ref{da1} for that Agmon distance in
$\partial \Omega$. Since $\partial \Omega$ is a smooth manifold
without boundary, the next result is a direct consequence of well known results from  \cite{helffer-88},
\cite{dimassi-sjostrand-99} and \cite{evans-10}.
\begin{proposition}\label{daaa} Let us assume that \textbf{[H1]} holds.
Let $x^*\in \partial \Omega$ be such that $\nabla_Tf(x^*)=0$. Then there exists a neighborhood $U^*$ of $x^*$ in $\partial \Omega$ such that $y\mapsto d_a^{\partial \Omega} (x^*,y)$ is smooth on $U^*$ and  $\forall x\in U^*$,
\begin{equation} \label{eee} 
\left\vert \nabla_T d_a^{\partial \Omega} (x^*,x)\right\vert^2=\vert \nabla_T f (x)\vert^2.
\end{equation}
\end{proposition}
\begin{proof}
 The boundary $\partial \Omega$ is a $C^{\infty}$ compact manifold and
 $x^*$ is a non degenerate minimum of $\vert \nabla_Tf\vert^2$. The
 proof is made in \cite[Proposition 2.3.6]{helffer2006semi} in the more general setting where
 $\vert \nabla_Tf\vert^2$ is replaced in (\ref{eee}) and in
 (\ref{eq:definition2}) by a smooth non negative function $W$ around a
 non degenerate minimum $y^*$ of $W$ such that $W(y^*)=0$. Here $W= \vert \nabla_T
 f \vert^2$ and $y^*=x^*$. This leads to $\nabla W= 2\,{\rm
   Hess}\left(f|_{\partial \Omega}\right) \left(\nabla_T f\right)$ and
 therefore $x^*$ is a critical point of $W= \vert \nabla_T f \vert^2$
 (which turns out to be a minimum). In addition, since $\nabla_T
 f(x^*)=0$, one gets that ${\rm Hess }\, W(x^*)=2\,\left( {\rm Hess}
   \left(f|_{\partial \Omega}\right)\right)^2(x^*)$ which is a non
 degenerate matrix. 
\end{proof}
\begin{proposition} \label{eikonal} Let us assume that  \textbf{[H1]}  and \textbf{[H3]}  hold.
Let $x^*\in \partial \Omega$ be such that $\nabla_Tf(x^*)=0$. Then, 
there exist a neighborhood $V^*$ of $x^*$ in $\overline \Omega$ and a
$C^{\infty}$ function $\Phi:V^* \to \R$ such that 
\begin{equation} \label{eikonalequation}
\left\{
\begin{aligned}
\vert  \nabla \Phi \vert^2 &=  \vert  \nabla f \vert^2   \text{ in }  \Omega \cap V^*,  \\ 
\Phi &= d_a^{\partial \Omega}(x^*,.) \text{ on } \partial \Omega  \cap V^*, \\
\partial_n \Phi  &< 0 \text{ on } \partial \Omega  \cap V^*.
\end{aligned}
\right.
\end{equation}
Moreover, one has the following uniqueness results: if $\widetilde
\Phi$ is a $C^\infty$ real valued function defined on a neighborhood
$\widetilde{V^*}$ of $x^*$ satisfying~\eqref{eikonalequation}, then $\widetilde \Phi=\Phi$ on $\widetilde{V^*}
\cap V^*$.  
Finally, up to choosing a smaller neighborhood $V^*$ of $x^*$, one can
assume that:
 \begin{itemize}
\item The function $\Phi$ is positive on $ V^*\setminus \{x^*\}$, so that
$x^*$ is a non degenerate minimum of $\Phi$ on $V^*$. 
%\item  There is no critical point of $\Phi$ in $\overline{V^*}$. 
\item According to~\eqref{eee} and~\eqref{eikonalequation}, it holds on $V^*\cap \pa \Omega$,  $\vert \nabla_T \Phi\vert =\vert \nabla_T f \vert$.
%,  and $x^*$ is the only critical point of $\Phi: \overline{V^*\cap \pa \Omega}\to \mathbb R^+$.
 \end{itemize}
\end{proposition}

%%%%%
\begin{proof} From Proposition~\ref{daaa}, the function $x\in \partial \Omega\mapsto d_a^{\partial \Omega}(x^*,x)$ is smooth near $x^*$. 
Then, the result stated  can be proven using the method of characteristics, see \cite[Theorem 1.5]{dimassi-sjostrand-99} or \cite[Section 3.2]{evans-10}. Let us mention that the proof crucially relies on the assumption
$\partial_nf(x^*)>0$. The fact that one can reduce $V^*$ such that
$\Phi$ is positive on $   V^*\setminus \{x^*\}$ is a consequence of $\partial_n \Phi<0$ on $\partial \Omega  \cap V^*$ together with the fact that $x^*$ is the only minimum of $d_a^{\partial \Omega}(x^*,.)$ (which is positive on $\partial \Omega\setminus \{x^*\}$).
\end{proof}

Let us state a simple corollary of Proposition~\ref{eikonal} and
Remark~\ref{rem:dadomega}.
\begin{corollary}\label{cor:eikonal_min}
 Let us assume that \textbf{[H1]} and \textbf{[H3]}  hold.
Let $x^*\in \partial \Omega$ be a local minimum of $f|_{\partial \Omega}$. Then
there exist a neighborhood $V^*$ of $x^*$ in $\overline \Omega$ and a
$C^{\infty}$ function $\Phi:V^* \to \R$ such that 
\begin{equation} \label{eikonalequation_min}
\left\{
\begin{aligned}
\vert  \nabla \Phi \vert^2 &=  \vert  \nabla f \vert^2   \text{ in }  \Omega \cap V^*,  \\ 
\Phi &= f-f(x^*) \text{ on } \partial \Omega  \cap V^*, \\
\partial_n \Phi  &< 0 \text{ on } \partial \Omega  \cap V^*.
\end{aligned}
\right.
\end{equation}
Moreover, one has the following uniqueness results: if $\widetilde
\Phi$ is a $C^\infty$ real valued function defined on a neighborhood
$\widetilde{V^*}$ of $x^*$ satisfying~\eqref{eikonalequation}, then $\widetilde \Phi=\Phi$ on $\widetilde{V^*}
\cap V^*$.  
Finally, up to choosing a smaller neighborhood $V^*$ of $x^*$, one can
assume that $\Phi$ is positive on $ V^*\setminus \{x^*\}$, and that 
$\Phi-f > -f(x^*)$ in $V^* \cap (\partial \Omega)^c$. As a consequence,
\begin{equation}\label{eq:carac_boundary}
\{x \in V^*,\, \Phi(x)=f(x)-f(x^*)\} \subset \partial \Omega.
\end{equation}
\end{corollary}
\begin{proof}
All the statements but~\eqref{eq:carac_boundary} are direct
consequences of Proposition~\ref{eikonal} and the fact that
$d_a^{\partial \Omega}(x^*,x)=f(x)-f(x^*)$, thanks to
Remark~\ref{rem:dadomega}. Now, notice that on $\partial \Omega \cap
V^*$, $\Phi-f=-f(x^*)$ and $\partial_n(\Phi -f)<0$ so that, up to
choosing a smaller neighborhood~$V^*$ of~$x^*$, one can assume that
$\Phi-f > -f(x^*)$ in $V^* \cap (\partial \Omega)^c$. This concludes the proof of~\eqref{eq:carac_boundary}.
\end{proof}
We are now in position to state the main result of this
section.
\begin{proposition} \label{equaa} Let us assume that \textbf{[H1]} and \textbf{[H3]}  hold.
Let $x^*\in \partial \Omega$ be such that $\nabla_Tf(x^*)=0$. Then,  there exists a neighborhood $U^*$ of $x^*$ in $\overline \Omega$ such that for all $x\in U^*$
$$d_a\left(x,x^*\right)=\Phi(x),$$
where $\Phi$ solves (\ref{eikonalequation}) and $d_a$ is the Agmon distance.
\end{proposition}
\begin{proof}
Notice that hypothesis \textbf{[H3]} allows us to use  Corollary~\ref{equalityagmon}. The proof follows the same lines of  the proof of  Proposition~\ref{da1}.
Let $x^*\in \partial \Omega$ be such that $\nabla_Tf(x^*)=0$. Let
$\Phi$ be the smooth solution to (\ref{eikonalequation}) on a neighborhood $V^*$ of $x^*$ and
such that $\Phi$ is positive on $V^*\setminus \{x^*\}$ and   it holds $\vert \nabla_T \Phi\vert =\vert \nabla_T f \vert$ on $V^*\cap \pa \Omega$, as
defined in Proposition~\ref{eikonal}. One chooses
$\varepsilon>0$ sufficiently small such that
$U^*:=\Phi^{-1}([0,\varepsilon))\subset V^*$. Notice that $U^*$ is a neighborhood of
$x^*$ in $\overline \Omega$. % This is possible by taking for instance,
% for $r>0$ such that $\partial B(x^*,r) \cap \overline \Omega\subset
% V^*$, $\ve=\inf \left\{ \Phi(x), x\in \partial B(x^*,r)\cap \overline
%   \Omega  \right\}>0$ ($\Phi$ is positive on $ V^*\setminus \{x^*\}$),
% where $B(x^*,r)$ is the ball centered at $x^*$ and of radius $r$ for
% the geodesic distance.

\medskip
\noindent
\underline{Step 1.} Let us first prove that for all $x \in U^*$, $\Phi(x) \le d_a(x^*,x)$.
\medskip

\noindent
For $x\in U^*$, one has $\Phi(x)< \ve$ and thus $\Phi^{-1}([0,\Phi(x)))\subset U^*$. Let $\gamma$ belong to $A\left(x^*,x\right)$. Let us define the time $t_0:=\inf \left\{t\in [0,1],\ \gamma(t)\notin \Phi^{-1}([0,\Phi(x)))\right\}$. By continuity of the curve $\gamma$, one has $t_0>0$, $\Phi(\gamma(t_0))=\Phi(x)$ and for all $t\in [0,t_0)$, $\gamma(t)\in \Phi^{-1}(\left[0,\Phi(x)) \right)\subset U^*$. Thus, using  Lemma~\ref{h}, one obtains
\begin{align*}
\Phi(x)&= \int_0^{t_0} \frac{ d   }{d  t} \Phi \circ \gamma(t) \, d  t = \int_0^{t_0} \nabla \Phi (\gamma(t)) \cdot \gamma'(t) \, dt \\
&=\int_{ {\rm int} \left\{  t\in (0,t_0), \ \gamma(t) \in \partial
    \Omega \right\}  } \nabla_T \Phi (\gamma(t)) \cdot  \gamma'(t) \, dt+\int_{ \left\{  t\in (0,t_0), \ \gamma(t) \in \Omega \right\}  } \nabla \Phi (\gamma(t)) \cdot \gamma'(t) \, dt.
\end{align*}
On the one hand, 
\begin{align*}
 \int_{ {\rm int} \left\{  t\in (0,t_0), \ \gamma(t) \in \partial
     \Omega \right\}  }  \nabla_T \Phi (\gamma(t)) \cdot  \gamma'(t) \, dt  &\leq \int_{ {\rm int} \left\{  t\in (0,t_0), \ \gamma(t) \in \partial \Omega \right\}  }  \left\vert \nabla_{T} \Phi( \gamma(t) ) \right\vert  \left\vert  \gamma'(t) \right\vert \, dt \\
&\leq \int_{ {\rm int} \left\{  t\in (0,t_0), \ \gamma(t) \in \partial \Omega \right\}  }  \vert \nabla_T f (\gamma(t))\vert  \left\vert  \gamma'(t) \right\vert \, dt,
\end{align*}
where one used the last statement in  Proposition~\ref{eikonal}. On the other hand, using \eqref{eikonalequation}, one obtains
\begin{align*}
\int_{ \left\{  t\in (0,t_0), \ \gamma(t) \in \Omega \right\}  }
\nabla \Phi (\gamma(t)) \cdot \gamma'(t) \, dt &\leq
\int_{ {\rm int} \left\{  t\in (0,t_0), \ \gamma(t) \in \Omega \right\}  }  \vert \nabla \Phi (\gamma(t)) \vert  \left\vert\gamma'(t) \right\vert \, dt \\
&\leq \int_{ {\rm int} \left\{  t\in (0,t_0), \ \gamma(t) \in \Omega \right\}  }  \vert \nabla f (\gamma(t))\vert  \left\vert\gamma'(t) \right\vert \, dt.
\end{align*}
Thus one gets
\begin{align*} 
\Phi(x)&\leq \int_{ {\rm int} \left\{  t\in (0,t_0), \ \gamma(t)
    \in \partial \Omega \right\}  }  \vert \nabla_T f (\gamma(t))\vert
 \left\vert\gamma'(t) \right\vert \, dt+\int_{
  \left\{  t\in (0,t_0), \  \gamma(t) \in \Omega \right\}  }  \vert \nabla  f (\gamma(t))\vert \left\vert \gamma'(t) \right\vert \, dt\\
&= \int_0^{t_0} g\left(\gamma(t)\right) \left\vert \gamma'(t) \right\vert \, d  t \leq \int_0^{1} g\left(\gamma(t)\right) \left\vert \gamma'(t) \right\vert \, d  t =L\left(\gamma,(0,1)\right). 
\end{align*}
Taking the infimum on the right-hand side over $\gamma\in A\left(x^*,x\right)$, one gets $\Phi(x)\leq d_a\left(x^*,x\right)$, for all $x\in U^*$. 

\medskip
\noindent
\underline{Step 2.} Let us now prove the reverse inequality: $\forall x\in U^*$, $d_a(x,x^*)\le \Phi(x)$.
\medskip

\noindent
 Let us define the following vector field on $U^*$,
\begin{equation} \label{eq:X}
X:=\left\{
\begin{aligned}
  - \nabla \Phi   & \text{ in  }  \Omega \cap U^*, \\ 
 - \nabla_T \Phi & \text{ on  } \partial \Omega \cap U^*.
\end{aligned}
\right.
\end{equation}
For $x\in U^*$, let us define the curve $\gamma$ by
\begin{equation} \label{eq:c}
\forall t \ge 0, \,  \gamma'(t)=X\left(\gamma(t)\right) \text{ and } \gamma(0)=x.
\end{equation}
This curve is well defined for all positive time using the
Cauchy-Lipschitz theorem and the fact that $(\gamma(t))_{t \ge 0}$
remains in $U^*$ for all positive time.

Indeed, if $x \in
\partial \Omega \cap U^*$, then $\gamma$ solves  $ \gamma'(t)=-
\nabla_T \Phi (\gamma(t))$ (with $\gamma(0)=x$). Since the function $t\mapsto \Phi\left(\gamma(t)\right)$ is
decreasing, the curve $\gamma$ remains in $\partial \Omega \cap
U^*=\partial \Omega \cap \Phi^{-1}([0,\varepsilon))$ and is defined on
$\mathbb R_+$.  Moreover, $\lim_{t\to +\infty}\gamma(t)=x^*$. 

Besides, if $x \in \Omega \cap U^*$, let us introduce the first time
$t_{\partial \Omega}$ such that $\gamma(t_{\partial \Omega}) \not \in \Omega \cap U^*$ for the curve
solution to  $ \gamma'(t)=-\nabla \Phi (\gamma(t))$  (with
$\gamma(0)=x$). If $t_{\partial \Omega}=\infty$,
then $\gamma$ belongs to $U^* \cap \Omega$ for all time, and since $t\mapsto \Phi\left(\gamma(t)\right)$ is
decreasing, necessarily, $\lim_{t\to +\infty}\gamma(t)=x^*$. If
$t_{\partial \Omega}<
\infty$, then, since $t\mapsto \Phi\left(\gamma(t)\right)$ is
decreasing and $U^*=\Phi^{-1}([0,\varepsilon))$, necessarily,
$\gamma(t_{\partial \Omega}) \in \partial \Omega \cap U^*$. The curve $\gamma$ is then defined on
$[t_{\partial \Omega},\infty)$ as above, for an initial condition in $\partial \Omega \cap U^*$.

We have thus shown that the function  $\gamma$ is globally defined,
piecewise $C^{\infty}$,  continuous, remains in $\partial \Omega$ if
it enters $\partial \Omega$, and satisfies 
$$\lim_{t\to +\infty}\gamma(t)=x^*.$$
Recall that $t_{\partial \Omega}=\inf \left\{t\in [0,+\infty), \ \gamma(t)\in \partial \Omega\right\}\in [0,\infty]$. One has
\begin{align*} 
-\Phi(x)&=\int_0^{+\infty} \frac{d}{dt} \Phi \circ
\gamma(t) \, dt\\
&=\int_0^{t_{\partial \Omega} }  \nabla \Phi (\gamma(t)) \cdot  \gamma'(t) \, d t+ \int_{t_{\partial \Omega}}^{+\infty } \nabla_T \Phi (\gamma(t)) \cdot \gamma'(t)  \, dt \\
&=-\left(\int_0^{t_{\partial \Omega} }  \vert \nabla \Phi (\gamma(t)) \vert^2 dt + \int_{t_{\partial \Omega}}^{+\infty } \vert \nabla_T \Phi (\gamma(t)) \vert^2 dt\right)\\
&=-\left(\int_0^{t_{\partial \Omega} } \vert \nabla \Phi (\gamma(t))\vert \left\vert  \gamma'(t) \right\vert \, d t+ \int_{t_{\partial \Omega}}^{+\infty } \vert \nabla_T \Phi (\gamma(t))\vert \left\vert  \gamma'(t) \right\vert \, dt\right)\\
&=- \int_0^{+\infty } g\left(\gamma(t)\right) \left\vert  \gamma'(t) \right\vert \, d  t=-\lim_{t\to +\infty} L(\gamma, (0,t)).
\end{align*}
Thanks to Lemma  \ref{inca},
$$d_a\left(x,x^*\right)\leq L(\gamma, (0,+\infty))=\Phi(x).$$
In conclusion, $\Phi(x)=d_a\left(x^*,x\right)$ for all $x\in U^*$.
\end{proof}
\begin{remark}\label{curve2}
Let us mention a simple consequence of the previous proof that will be
useful in the following. If $x^*\in \partial \Omega$ is such that
$\nabla_Tf(x^*)=0$, there exists a neighborhood $U^*$ of~$x^*$ such
that for all $x\in U^*$, there exists a piecewise $C^{\infty}$ and continuous curve $\gamma: \mathbb R_+\to \overline \Omega$ such that 
$$d_a (x^*,x)=\int_0^{+\infty} g \left(\gamma(t)\right) \left\vert \gamma'(t) \right\vert \, dt,$$
with $\gamma(0)=x$ and $\lim_{t\to +\infty}\gamma(t)=x^*$. The curve $\gamma$ is solution to~\eqref{eq:X}--\eqref{eq:c}.
In addition, the set 
$\partial \{ t\in [0,\infty), \ \gamma(t)\in \partial \Omega\}$ either consists of one point or is empty.
\end{remark}

\subsection{Curves realizing the Agmon distance} \label{minim}
In this section, it is proven that for any two points $x\in
\overline \Omega$ and $y\in \overline \Omega$, their exists a finite
number of curves $(\gamma_i)_{i=1,\ldots,N}$ defined on the intervals $(I_i)_{i=1,\ldots,N}$ such that the sum of their lengths $L(\gamma_i,I_i)$ equals
the Agmon distance $d_a\left(x,y\right)$. The precise statement is
given in the following theorem.
\begin{theorem} \label{exis}
Assume that \textbf{[H1]} and \textbf{[H3]}  hold. Let $x,y \in \overline \Omega$. Then
there exists a finite number of Lipschitz curves
$(\gamma_j)_{j=1,\ldots,N}$ which are defined on possibly unbounded
intervals $I_j\subset \mathbb R$, with values in $\overline{\Omega}$, such that for all $ j\in
\left\{1,\ldots,N\right\}$,  the sets $\partial \{ t\in I_j, \ \gamma_j(t)\in \partial \Omega\}$ are finite and 
$$d_a\left(x,y\right)=\sum_{j=1}^{N} L\left(\gamma_j,I_j\right).$$
Additionally, by construction, the intervals
$(I_j)_{j \in \{1, \ldots , N\}}$ are either
$[0+\infty)$, $(-\infty,0]$ or $[0,1]$. Moreover, if $I_j=[0,+\infty)$ or $I_j=(-\infty,0]$, then $\gamma_j$ is continuous and
piecewise $C^{\infty}$ (see Lemma~\ref{agmoncurve} below  for a more precise definition of the curves $\gamma_j$ in this case). If $I_j=[0,1]$, then $\gamma_j \in
A(\gamma_j(0),\gamma_j(1))$. Finally the curves $((\gamma_1,I_1),\ldots,(\gamma_N,I_N))$ are ordered such that 
$$\lim_{t\to (\inf I_1)^+} \gamma_1(t)=x, \quad \lim_{t\to (\sup I_N)^-} \gamma_N(t)=y,$$
and for all $k\in \{1,\ldots,N-1\}$,
$$\lim_{t\to \left(\sup I_k\right)^-} \gamma_k(t)=\lim_{t\to \left(\inf I_{k+1}\right)^+} \gamma_{k+1}(t).$$
\end{theorem}
\noindent
This section is entirely  dedicated to the proof of Theorem
\ref{exis}. In the following, one denotes by 
$$\{x_1,\ldots,x_m\}=\{x\in \overline \Omega, \, g(x)=0\},$$ 
\label{page.x1xm}
where $g$ is defined by~\eqref{eq:def_g} (there is a finite number of zeros of the function $g$ thanks to \textbf{[H1]}). 

\subsubsection{Preliminary results}
Let us first consider the simple case when the curve realizing the Agmon distance does not meet zeros of the function $g$.

\begin{lemma} \label{gpositive} Assume that \textbf{[H1]} and \textbf{[H3]}  hold. 
Let $\left(x,y\right)\in \overline \Omega\times \overline \Omega$. Let
$(\gamma_n)_{n \ge0} \in A\left(x,y\right)^{\mathbb N}$ be a
minimizing sequence of curves for $d_a\left(x,y\right)$: $\lim_{n\to
  \infty}L(\gamma_n,(0,1))= d_a\left(x,y\right)$. In addition, assume
that for each $k \in \{1,\ldots,m\}$, there exists a neighborhood
$V_k$ of $x_k$ in $\overline \Omega$, such that:
 $$\forall n \in \mathbb N, \, \forall k \in \{1,\ldots,m\},\, \Ran(\gamma_n)\cap V_k=\emptyset.$$ Then, there exists $\gamma \in A\left(x,y\right)$ such that 
$$L\left(\gamma,(0,1)\right)=d_a\left(x,y\right).$$
\end{lemma}
\begin{proof}
Let $M$ be such that for all $ n$, $L(\gamma_n,(0,1))\leq M$ and let us define
$$c:=\inf_{\overline \Omega \setminus (V_1\cup\ldots \cup V_m)} g>0.$$
One defines for $t\in [0,1]$, $\phi_n(t)=\frac{L(\gamma_n,(0,t))+t}{L(\gamma_n,(0,1))+1}$. The map $\phi_n$ is strictly increasing and continuous from $[0,1]$ to $[0,1]$. Therefore it admits an inverse. Setting $\tilde \gamma_n(u):=\gamma_n \circ \phi_n^{-1}(u)$, one gets $L(\gamma_n,(0,1))=L(\tilde \gamma_n,(0,1))$ and 
\begin{align*}
\left\vert \tilde \gamma_n' \right\vert ( \phi_n(t)) &=\frac{ \left\vert  \gamma_n'(t) \right\vert}{ g( \gamma_n(t)) \,  \left\vert \gamma_n'(t)\right\vert+1   } \, \left(L\left(\gamma_n,(0,1)\right)+1\right)\\
&\leq\frac{  \left\vert \gamma_n'(t) \right\vert}{ c \, \left\vert \gamma_n'(t)\right\vert+1   } \, \left(L\left(\gamma_n,(0,1)\right)+1\right)\\
&\leq \frac{1}{c}\left(L\left(\gamma_n,(0,1)\right)+1\right)\\
&\leq  \frac{1}{c}\left(M+1\right).
\end{align*}
 Thus, up to replacing $\gamma_n$ by $\tilde \gamma_n$, one may assume
 that the Lipchitz constants of $\gamma_n$ are bounded uniformly in
 $n$. In addition since for all $t\in [0,1]$, $\gamma_n(t)\in
 \overline \Omega$,  the sequence $(\gamma_n)_{n \ge 0}$ is relatively
 compact in $C^0([0,1], \overline \Omega)$. Thus, up to the extraction
 of  a subsequence, there exists a Lipschitz curve $\gamma$ such that $\lim_{n\to \infty}\gamma_n = \gamma$ uniformly on $[0,1]$. Moreover since  $(\gamma_n)_{n\ge 0}$ is bounded in $H^1([0,1],\overline \Omega)$, up to the extraction of a subsequence, $(\gamma_n)_{n\ge 0}$ converges weakly to $\gamma$ in $H^1([0,1],\overline \Omega)$. It is not difficult to see that for all $t\in \mathbb [0,1]$,
$$\liminf_{n\to \infty} g\left(\gamma_n(t)\right)\geq g\left(\gamma(t)\right).$$
Indeed, for $t\in [0,1]$, there are two cases:
\begin{itemize}
\item If $\gamma(t)\in \Omega$, then for $n$ large enough, all the
  points $\gamma_n(t)$ are in $\Omega$ and thus $\liminf_{n\to \infty} g\left(\gamma_n(t)\right) =\lim_{n\to \infty} g\left(\gamma_n(t)\right)= \vert \nabla f\left(\gamma(t)\right)\vert= g\left(\gamma(t)\right)$, 
\item If $\gamma(t)\in \partial \Omega$, since $\mathbb N=\{n, \
  \gamma_n(t)\in \partial \Omega\}\cup \{n, \ \gamma_n(t)\in\Omega\}$,
  one obtains that the set of limit points of $(\vert \nabla f(\gamma_n(t)   ) \vert )_{n \ge
    0}$ is included in $\{\vert \nabla f\left(\gamma(t)\right)\vert, \vert \nabla_T f\left(\gamma(t)\right)\vert\}$. Therefore, from \textbf{[H3]}, one has: $\liminf_{n\to \infty} g\left(\gamma_n(t)\right)\geq \vert \nabla_T f\left(\gamma(t)\right)\vert=g\left(\gamma(t)\right)$.
\end{itemize}
Then, one obtains
\begin{align*}
d_a\left(x,y\right)=\lim_{l\to \infty} \int_0^1 g(\gamma_l(t))\vert \gamma_l'(t)\vert dt &\geq \liminf_{n\to \infty} \liminf_{p\to \infty} \int_0^1 g(\gamma_p(t))\vert \gamma_n'(t)\vert dt\\
&\geq \liminf_{n\to \infty}  \int_0^1 \liminf_{p\to \infty} g(\gamma_p(t))\vert \gamma_n'(t)\vert dt\\
&\geq \liminf_{n\to \infty}  \int_0^1 g\left(\gamma(t)\right)\vert \gamma_n'(t)\vert dt\\
&\geq   \int_0^1 g\left(\gamma(t)\right)\vert \gamma ' (t)\vert dt.
\end{align*}
In the previous computation, one used Fatou Lemma and the lower semi continuity (for the weak convergence) of the convex functional
$$h\in H^1([0,1],\overline \Omega) \mapsto \int_0^1 g\left(\gamma(t)\right)\left\vert h' (t)\right\vert dt.$$
Since \textbf{[H3]} holds, using Proposition~\ref{front}, there exists  a curve $\tilde \gamma \in A\left(x,y\right)$ such that $L\left(\gamma,(0,1)\right)\geq L(\tilde \gamma,(0,1))$ and thus $d_a\left(x,y\right)=L(\tilde \gamma,(0,1))$.
\end{proof}
\noindent
Let us now introduce a  sufficient condition so that a minimizing sequence of  curves realizing the
Agmon distance avoids a neighborhood of a zero of the function $g$. For $x\in \overline \Omega$, one introduces the following sets:
\begin{equation} \label{Ak}
\forall  k \in \left\{1,\ldots,m\right\} , \, A_k(x):=\left\{ z\in \overline \Omega,  \  d_a(x,z)=d_a(x,x_k)+d_a(x_k,z)      \right\}.
\end{equation}
One notices that $z\in A_k(x)$ if and only if $x\in A_k(z)$.
\begin{proposition} \label{popo1} Assume that \textbf{[H1]} and \textbf{[H3]}  hold.
Let $(x,y) \in \overline \Omega^2$ and assume that there exists $k \in
\{1,\ldots,m\}$ such that $y\notin A_k(x)$. If $(\gamma_n)_{n \ge 0} \in A\left(x,y\right)^{\mathbb N}$ is a minimizing sequence of curves for $d_a\left(x,y\right)$, then there exists a neighborhood $V_k$ of $x_k$ in $\overline \Omega$ and $n_0\in \mathbb N$, such that for all $ n\geq n_0$,
$$\Ran(\gamma_n)\cap V_k=\emptyset.$$
\end{proposition}
\begin{proof}
If $y\notin A_k(x)$, for a $k\in \{1,\ldots,m\}$, then
$d_a(x,y)<d_a(x,x_k)+d_a(x_k,y)$ and thus $y\neq x_k$ and $x
  \neq x_k$. Let us define 
$$\varepsilon:=d_a(x,x_k)+d_a(x_k,y)-d_a(x,y)>0,$$
and $V_k:=B_a\left(x_k, \min \left(\frac{\varepsilon}{3},
    \frac{d_a(x_k,y)}{2}\right)\right)$ where 
\begin{equation}\label{eq:def_Ba}
\forall z\in \overline \Omega, \, \forall r>0, \, B_a(z,r):=\left\{ u \in
  \overline \Omega,  \, d_a(z,u) <  r\right\}.
\end{equation}
 Notice that $y\notin V_k$.
We now prove Proposition~\ref{popo1} by contradiction. We assume that, up to the extraction of a subsequence, for all $n\in \mathbb N$, $\Ran(\gamma_n)\cap V_k\neq\emptyset$. 
Let us  define, for all $n\in \mathbb N$, 
$$t_0^n:=\inf \{ t\in [0,1], \ \gamma_n(t)\in V_k\}\,  \text{ and }  \, t_1^n:=\sup
\{ t\in [0,1], \ \gamma_n(t)\in V_k\}.$$
We have for all $n\in \mathbb N$, owing to the triangular inequality,
$$L(\gamma_n,(0,t_0^n))\geq d_a(x,x_k)-\frac{\varepsilon}{3}, \quad L(\gamma_n,(t_1^n,1))\geq d_a(x_k,y)-\frac{\varepsilon}{3}.$$
Thus, for all $n\in \mathbb N$, it holds:
\begin{align*}
L(\gamma_n(0,1))&\geq L(\gamma_n,(0,t_0^n))+L(\gamma_n,(t_1^n,1))\\
&\geq d_a(x,x_k) +d_a(x_k,y)-\frac{2 \varepsilon}{3}= d_a(x,y)+\frac{\varepsilon}{3}.
\end{align*}
This contradicts the fact that $\lim_{n \to \infty} L(\gamma_n,(0,1))= d_a(x,y)$.
\end{proof}
A direct corollary of Proposition~\ref{popo1} and  Lemma~\ref{gpositive} is the following result:
\begin{corollary} \label{coco1} Assume that \textbf{[H1]} and \textbf{[H3]}  hold. 
Let $y \in \overline \Omega$ and assume that $y\notin A_j(x)$ for all $j\in\{1,\ldots,m\}$. Then,  there exists a curve $\gamma \in A\left(x,y\right)$  such that 
$$d_a\left(x,y\right)=L\left(\gamma,(0,1)\right).$$
\end{corollary}
Notice that $y\notin A_j(x)$ for all $j\in\{1,\ldots,m\}$ implies in
particular that $x$ and $y$ are not zeros of the function $g$.
This corollary will be used below to build the curves $\gamma_j$
associated with intervals $I_j=[0,1]$ in Theorem~\ref{exis}.
The curves $\gamma_j$
associated with intervals $I_j=[0,+\infty)$ or $I_j=(-\infty,0]$ will 
be built using the following lemma, which is a direct consequence of Remarks~\ref{curve1} and~\ref{curve2}.
\begin{lemma}\label{agmoncurve} Assume that  \textbf{[H1]}  and \textbf{[H3]}  hold.
Let $k\in \{1,\ldots,m\}$. There exists a neighborhood $V_k$ of $x_k$ in $\overline{\Omega}$, such that for all $y\in V_k$, there exists a continuous and piecewise $C^{\infty}$ curve $\gamma: (-\infty,0]\to V_k$ satisfying
$$d_a(y,x_k)=L\left(\gamma,(-\infty,0]\right),\ \lim_{t\to -\infty}
\gamma(t)=x_k,  \ \gamma(0)=y.$$
If $x_k\in  \Omega$, $\gamma$  is with values in $\Omega$ and satisfies~\eqref{eq.int-gamma'-eik}. If $x_k\in \pa \Omega$,  $\gamma$ satisfies~\eqref{eq:X}--\eqref{eq:c} and is such that $\partial \{ t\in (-\infty,0], \
  \gamma(t)\in \partial \Omega\}$ is either empty or a single point.
 \end{lemma}
Before proving Theorem~\ref{exis}, we finally need two additional preliminary lemmas. 
\begin{lemma} \label{popo2}Assume that \textbf{[H1]} and \textbf{[H3]}  hold. 
Let $u\in \overline \Omega$ and $w\in \overline \Omega$. For any
$\delta>0$ small enough, there exists $z_{\delta}$ such that
$d_a(u,z_\delta)=\delta$ and $d_a(w,u)=d_a(w,z_{\delta})+d_a(z_{\delta},u)$. 
\end{lemma}
\begin{proof}
Notice that $d_a(u,z_\delta)=\delta$ is equivalent to  $z_\delta \in \partial
B_a(u,\delta)$, where $B_a$ is defined by~\eqref{eq:def_Ba}.
We prove Lemma~\ref{popo2} by contradiction. Assume that there exists
$\delta\in \left (0,\frac{d_a(u,w)}{2}\right)$ such that for all $z
\in \partial B_a(u,\delta)$, $d_a\left(w,u\right)<d_a\left(w,z\right)+d_a\left(z,u\right)$. 
 By compactness of $\partial B_a(u,\delta)$, there exists $a_{\delta}>0$ such that for all $z \in \partial B_a(u,\delta)$,
 $$d_a(w,u) +a_{\delta}\leq d_a\left(w,z\right)+d_a\left(z,u\right).$$
Thus if $\gamma\in A(u,w)$, since there exists a time $t_{\delta}$ such that $\gamma(t_{\delta})\in \partial B_a(u,\delta)$, one has
\begin{align*}
L\left(\gamma,(0,1)\right)=  L\left(\gamma,\left(0,\gamma(t_{\delta})\right)\right)+L\left(\gamma,\left(\gamma(t_{\delta}),1\right)\right) &\geq d_a\left(u,\gamma(t_{\delta})\right)+d_a\left(\gamma(t_{\delta}),w\right)\\
&\geq d_a\left(u,w\right)+a_{\delta}.
\end{align*}
 This is impossible since by definition $d_a\left(u,w\right)=\inf_{\gamma \in A(u,w)}L\left(\gamma,(0,1)\right)$.
\end{proof}

  \begin{lemma}\label{lem1} Assume that \textbf{[H1]} holds. 
   Let $(x,y)\in \overline\Omega^2$ with $x\neq y$.  Then, there exist $N\in \mathbb N$ and a sequence $(b_i)_{i\in \{0,\ldots,N+1\}}\in \overline \Omega^{N+2}$, $b_0=x$, $b_{N+1}=y$, $(b_i)_{i\in \{1,\ldots,N\}} \in\left\{x_1,\ldots,x_m \right\}^{N}$ (with the convention $\left\{x_1,\ldots,x_m \right\}^{0}=\emptyset$) such that the following holds:
 \begin{enumerate}
 \item For all $i\in \{0,\ldots,N\}$, $b_i\neq b_{i+1}$ and 
\begin{equation}\label{decom}
  d_a\left(x,y\right)=\sum_{i=0}^{N} d_a(b_i,b_{i+1}).
\end{equation} 
 \item For all $i\in \{0,\ldots,N\}$ and for all $z\in
   \{x,y,x_1,\ldots,x_m\}\setminus
   \{b_i,b_{i+1}\}$, $$d_a(b_i,b_{i+1})<d_a(b_i,z) + d_a(z,b_{i+1}).$$
 \end{enumerate}
\end{lemma}
 \begin{proof} 
Since $x\neq y$, the following set 
\begin{align*}
 E:=  \{  &(N,b), \, N\in \mathbb N,\, b=(b_i)_{i\in \{0,\ldots,N+1\}} \in \overline \Omega^{N+2},\, b_0=x, \, b_{N+1}=y,\,  \\
 &\forall i\in \{0,\ldots,N\}, \, b_i\neq b_{i+1}, \, (b_i)_{i\in \{1,\ldots,N\}} \in \{x_1,\ldots,x_m\}^N,  \eqref{decom} \text{ holds} \},
\end{align*}
 is not empty since by assumption it contains $(0,\{x,y\})$. For $(N,b)\in E$, one defines the cardinal of
 $(N,b)$ by the number of different critical points $b$ contains. The cardinal of
 an element of $E$ belongs to $\{0,\ldots,m\}$. 
 
 Let us now consider an
 element $(N,b)\in E$  which is maximal for
 the cardinal. By construction, this element satisfies point $1$ in Lemma~\ref{lem1}.
Let us now show that it also satisfies point $2$ in Lemma
\ref{lem1}. Notice that $\{b_0, \ldots ,b_{N+1}\} \subset \{x,y,x_1,\ldots,x_m\}$.
Let $i\in
   \{0,\ldots,N\}$ and $z\in \{x,y,x_1,\ldots,x_m\}\setminus \{b_i,b_{i+1}\}$.
% On ne peut pas rajouter de point critique
 If $z\in \{x,y,x_1,\ldots,x_m\} \setminus \{b_0, \ldots ,b_{N+1}\}$, the equality $d_a(b_i,b_{i+1})=d_a(b_i,z) + d_a(z,b_{i+1})$ cannot hold since $b$ has been chosen maximal in $E$ for the cardinal. Thus, by the triangular inequality $d_a(b_i,b_{i+1})<d_a(b_i,z) + d_a(z,b_{i+1})$. 
% On ne peut pas repasser par un point deja visite
 If $z\in \{b_0, \ldots ,b_{N+1}\}\setminus \{b_i,b_{i+1}\}$, let us
 prove that $d_a(b_i,b_{i+1})<d_a(b_i,z) + d_a(z,b_{i+1})$ by
 contradiction. By the triangular inequality, if the previous
 inequality does not hold, one has $d_a(b_i,b_{i+1})=d_a(b_i,z) +
 d_a(z,b_{i+1})$ for some $z\in \{b_0, \ldots ,b_{N+1}\}\setminus \{b_i,b_{i+1}\}$. Let us denote by $j_0 \in \{0, \ldots,i-1,i+2,\ldots, N+1\}$ the
 index such that $z=b_{j_0}$. One has $d_a(b_i,b_{i+1})=d_a(b_i,b_{j_0}) +
 d_a(b_{j_0},b_{i+1})$. Let us assume without loss of generality that
 $j_0<i$ (the case $j_0>i+1$ is treated similarly). In this case, one
 has, using the triangular inequality:
\begin{align*}
d_a(x,y)&=\sum_{j=0}^N d_a(b_j,b_{j+1})\\
&=\sum_{j=0}^{i-1} d_a(b_j,b_{j+1})+d_a(b_i,b_{j_0})+d_a(b_{j_0},b_{i+1})+\sum_{j=i+1}^{N} d_a(b_j,b_{j+1})\\
&=\sum_{j=0}^{j_0-1}
d_a(b_j,b_{j+1})+d_a(b_{j_0},b_{i+1})+\sum_{j=i+1}^{N}
d_a(b_j,b_{j+1}) +\sum_{j=j_0}^{i-1} d_a(b_j,b_{j+1})
+d_a(b_i,b_{j_0})\\
&\ge d_a(x,y) +\sum_{j=j_0}^{i-1} d_a(b_j,b_{j+1})
+d_a(b_i,b_{j_0}).
\end{align*}
Thus, $\sum_{j=j_0}^{i-1} d_a(b_j,b_{j+1})
+d_a(b_i,b_{j_0})=0$ and $b_{j_0}=b_i$ which is in contradiction
with $z\not\in \{b_i, b_{i+1}\}$.
% Therefore, one obtains that
% $$\left\{
% \begin{aligned}
%  & d_a\left(b_{j_0},b_{j_0+1}\right)+\ldots+d_a\left(b_{i-1},b_{i}\right)+d_a\left(b_{i},b_{j_0}\right)\leq 0 \ {\rm if \ } j_0<i-1,  \\ 
% & 2\,d_a\left(b_{i-1},b_{i}\right)\leq 0  \ {\rm if \ } j_0=i-1,
% \end{aligned}
% \right.$$
% which is impossible since $d_a$ is a distance and $b_{i-1}\neq b_{i}$.
%
% We know that $z\in \left\{b_0,\ldots,b_{i-1},b_{i+2},\ldots, b_{N+1} \right\}$. Let us assume without loss of generality that $i\geq1 $ and $z=b_{j_0}$ with $j_0\in \left\{ 0,\ldots,i-1 \right\}$. 
% One has from the triangular inequality applied between the points $\left\{b_0,\ldots,b_{j_0}, b_{i+1}, \ldots,b_{N+1}\right\}$,
%  $$d_a\left(b_0,b_{N+1}\right)\leq d_a\left(b_0,b_1\right)+\ldots+d_a\left(b_{j_0-1},b_{j_0}\right)+d_a\left(b_{j_0},b_{i+1}\right)+\ldots+d_a\left(b_{N},b_{N+1}\right)+d_a\left(b_{N+1},y\right).$$
% In addition from~\eqref{decom} together with the equality $d_a\left(b_i,b_{i+1}\right)=d_a\left(b_i,b_{j_0}\right) + d_a\left(b_{j_0},b_{i+1}\right)$, one obtains that 
% $$\left\{
% \begin{aligned}
%  & d_a\left(b_{j_0},b_{j_0+1}\right)+\ldots+d_a\left(b_{i-1},b_{i}\right)+d_a\left(b_{i},b_{j_0}\right)\leq 0 \ {\rm if \ } j_0<i-1,  \\ 
% & 2\,d_a\left(b_{i-1},b_{i}\right)\leq 0  \ {\rm if \ } j_0=i-1,
% \end{aligned}
% \right.$$
% which is impossible since $d_a$ is a distance and $b_{i-1}\neq b_{i}$.
Therefore  $d_a\left(b_i,b_{i+1}\right)<d_a\left(b_i,b_{j_0}\right) +
d_a\left(b_{j_0},b_{i+1}\right)$. This concludes the proof of  Lemma~\ref{lem1}. 
 \end{proof}
 
\subsubsection{Proof of Theorem~\ref{exis}}
Let us now prove Theorem~\ref{exis}.  Recall that by assumption, the hypotheses  \textbf{[H1]}  and \textbf{[H3]}  hold.
 \begin{proof} Let $x,y \in \overline{\Omega}$. If $x=y$, then
   Theorem~\ref{exis} is proved by taking the constant curve
   $\gamma(t)=x$ for all $t\in [0,1]$. Let us deal with the case
   $x\neq y$. From Lemma~\ref{lem1}, there exist $N\in \mathbb N$ and a sequence $(b_j)_{j\in \{0,\ldots,N+1\}}\subset \overline \Omega^{N+2}$ such that $b_0=x$, $b_{N+1}=y$, $(b_j)_{j\in \{1,\ldots,N\}} \subset\left\{x_1,\ldots,x_m \right\}^{N}$ (with the convention $\left\{x_1,\ldots,x_m \right\}^{0}=\emptyset$)  and for all $k\in \{0,\ldots,N\}$, $b_k\neq b_{k+1}$ and 
 \begin{equation}\label{eq:decomp_da}  
d_a\left(x,y\right)=\sum_{k=0}^{N} d_a \left(b_k,b_{k+1}\right).
\end{equation}
If $N=0$, then for all $k\in \{1,\ldots,m\}$ $y\notin
   A_k(x)$ and Theorem~\ref{exis} is then a consequence of Corollary~\ref{coco1}. 
   
   Let us now assume that $N \ge 1$, namely that there exists
   $k\in \left\{1,\ldots,m\right\}$ such that $y\in A_k(x)$. Let us actually consider the case $N\geq 2$ (the case $N=1$ is treated similarly). Let $k\in \{1,\ldots,N-1\}$ and let us consider the
term $d_a\left(b_k,b_{k+1}\right)$ in~\eqref{eq:decomp_da} (the first term
$d_a\left(x,b_1\right)$ and the last term $d_a\left(b_N,y\right)$ in the sum are treated in a similar way). One can label the points
$\{x_1,\ldots,x_m\}$ such that $b_k=x_1$ and $b_{k+1}=x_2$. Point~$2$
in Lemma~\ref{lem1} implies that $x_2\notin A_j(x_1)$ for all $j\in
\{3,\ldots,m\}$. From Lemma~\ref{popo2}, for any $\delta>0$ there
exists $z_1\in \partial B_a(x_1,\delta)$ such that
$d_a\left(x_1,x_2\right)=d_a\left(x_1,z_1\right) +
d_a\left(z_1,x_2\right)$ (where $B_a$ is defined
by~\eqref{eq:def_Ba}). By taking $\delta$ small enough, this implies
that $z_1\notin A_1(x_2)$ and $z_1\notin
\{x_1,\ldots,x_m\}$. Likewise, from Lemma~\ref{popo2}, for any
$\delta>0$ there exists $z_2\in \partial B_a(x_2,\delta)$ such that
$d_a\left(z_1,x_2\right)=d_a\left(z_1,z_2\right) +
d_a\left(z_2,x_2\right)$ and by taking $\delta$ small enough, this implies that $z_2\notin A_2(z_1)$ and $z_2\notin \{x_1,\ldots,x_m\}$. Therefore one gets
$$d_a(b_k,b_{k+1})=d_a\left(x_1,x_2\right)=d_a\left(x_1,z_1\right)+d_a\left(z_1,z_2\right) + d_a\left(z_2,x_2\right).$$
Taking $\delta$ small enough and using Lemma~\ref{agmoncurve},  there exists a continuous and piecewise $C^{\infty}$ curve $\gamma_1$  defined on $(-\infty,0]$ such that
$d_a\left(x_1,z_1\right)=L\left(\gamma_1,(-\infty,0]\right)$,
$\lim_{t\to -\infty} \gamma_1(t)=x_1$, $\gamma_1(0)=z_1$, and $\partial
\{ t\in (-\infty,0], \ \gamma_1(t)\in \partial \Omega\}$
is either empty or a single point. Similarly, there exists a continuous and piecewise $C^{\infty}$ curve $\gamma_2$  defined on $[0,+\infty)$ such that
$d_a\left(z_2,x_2\right)=L\left(\gamma_2,[0,+\infty)\right)$,
$\gamma_2(0)=z_2$, $\lim_{t\to +\infty} \gamma_2(t)=x_2$ and $\partial
\{ t\in [0,+\infty), \ \gamma_2(t)\in \partial \Omega\}$
is either empty or a single point. 
Let us show by contradiction that $z_2\notin A_j(z_1)$ for all $j\in \{3,\ldots,m\}$. On the one hand, if  $z_2\in A_j(z_1)$ for some $j\in \{3,\ldots,m\}$, one has 
$$
d_a\left(x_1,x_2\right)= d_a\left(x_1,z_1\right)+d_a\left(z_1,x_j\right)+ d_a\left(x_j,z_2\right)+d_a\left(z_2,x_2\right).$$
 On the other hand, $x_1\notin A_j(x_2)$, and thus
 $$
d_a\left(x_1,x_2\right)<d_a\left(x_1,x_j\right)+d_a\left(x_j,x_2\right)\leq d_a\left(x_1,z_1\right)+d_a\left(z_1,x_j\right)+ d_a\left(x_j,z_2\right)+d_a\left(z_2,x_2\right).$$
This leads to a contradiction. Therefore $z_2\notin A_j(z_1)$ for all
$j\in \{3,\ldots,m\}$. One also has by a similar reasoning that
$z_2\notin A_1(z_1)$. Indeed, If $z_2\in A_1(z_1)$, then one has on
the one hand
$$
d_a\left(z_1,x_2\right)=d_a\left(z_1,z_2\right)+d_a\left(z_2,x_2\right)=d_a\left(z_1,x_1\right)+d_a\left(x_1,z_2\right)+ d_a\left(z_2,x_2\right).$$
On the other hand, since $z_1\notin A_1(x_2)$, one has  
$$
d_a\left(z_1,x_2\right)<d_a\left(z_1,x_1\right)+d_a\left(x_1,x_2\right)\leq d_a\left(z_1,x_1\right)+d_a\left(x_1,z_2\right)+ d_a\left(z_2,x_2\right).$$
This leads to a contradiction. In conclusion $z_2\notin A_j(z_1)$ for all $j\in \{1,\ldots,m\}$. Therefore, from Corollary~\ref{coco1}, there exists a curve $\gamma \in A\left(z_1,z_2\right)$  such that $d_a\left(z_1,z_2\right)=L\left(\gamma,(0,1)\right)$. 
In conclusion, we have built three curves $\gamma$, $\gamma_1$ and $\gamma_2$ such that
$$d_a(b_k,b_{k+1})=L(\gamma_1,(-\infty,0])+L(\gamma,(0,1))+L(\gamma_2,[0,+\infty)).$$

A similar reasoning for all the terms in the sum in~\eqref{eq:decomp_da} concludes the proof of Theorem~\ref{exis}.
\end{proof}

 A consequence of Theorem~\ref{exis} is the following.
\begin{lemma} \label{LLL}Let us assume that \textbf{[H1]} and \textbf{[H3]}  hold.
Let $\left(x,y\right) \in \overline{\Omega}$. Let us denote by $((\gamma_1,I_1),\ldots,(\gamma_N,I_N))$ the curves given by Theorem~\ref{exis} ordered such that 
$$\lim_{t\to (\inf I_1)^+} \gamma_1(t)=x, \quad \lim_{t\to (\sup I_N)^-} \gamma_N(t)=y,$$
and which realize the Agmon distance between $x$ and $y$. 
Let $k_1\leq k_2$ with $(k_1,k_2)\in \left\{1,\ldots,N\right\}^2$ and let $t_1\in \overline{I_{k_1}}$ and $t_2\in \overline{I_{k_2}}$. If $k_1=k_2$, $t_1$ and $t_2$ are chosen such that $t_1\leq t_2$. Then,  one has:
\begin{itemize}
\item If $k_1<k_2$, 
$$
d_a\left(\gamma_{k_1}(t_1),\gamma_{k_2}(t_2)\right)= L\left(\gamma_{k_1}, ( t_1,\sup I_{k_1}  )\right)+ \!\!\! \sum_{k=k_1+1}^{k_2-1} \!\!\! L\left(\gamma_{k}, I_k\right)+L\left(\gamma_{k_2}, ( \inf I_{k_2},t_2  )\right) ,
$$
where by convention, if $k_2=k_1+1$, $\sum_{k=k_1+1}^{k_2-1} L(\gamma_{k}, I_k)=0$.
\item  If $k_1=k_2$, $d_a\left(\gamma_{k_1}(t_1),\gamma_{k_2}(t_2)\right)= L\left(\gamma_{k_1}, ( t_1,t_2  ) \right)$. 
\end{itemize}
 In addition, the following equality holds
$$d_a\left(x,y\right)=d_a\left(x,\gamma_{k_1}(t_1)\right)+d_a\left(\gamma_{k_1}(t_1),\gamma_{k_2}(t_2)\right)+d_a\left(\gamma_{k_2}(t_2),y\right).$$
\end{lemma}
\noindent
The proof of  Lemma~\ref{LLL} is done easily reasoning by contradiction and using the triangular inequality on the Agmon distance. 
%\begin{proof} Let us make the proof for $k_1<k_2$ (the case $k_1=k_2$
%  is treated in a similar way). One has using the triangular inequality for the Agmon distance and Lemma~\ref{inca},
%$$d_a\left(\gamma_{k_1}(t_1),\gamma_{k_2}(t_2)\right)\leq L\left(\gamma_{k_1}, ( t_1,\sup I_{k_1}  )\right)+ \sum_{k=k_1+1}^{k_2-1} L\left(\gamma_{k}, I_k\right)+L\left(\gamma_{k_2}, ( \inf I_{k_2},t_2  )\right).$$
%Let us now conclude the proof by contradiction.  Assume that
%$$d_a(\gamma_{k_1}(t_1),\gamma_{k_2}(t_2))<L(\gamma_{k_1}, ( t_1,\sup I_{k_1}  ))+ \sum_{k=k_1+1}^{k_2-1} L(\gamma_{k}, I_k)+L(\gamma_{k_2}, ( \inf I_{k_2},t_2  )).$$
% Using the triangular inequality, one has
%\begin{align*}
%d_a\left(x,y\right) &\leq  d_a(x,\gamma_{k_1}(t_1))+d_a(\gamma_{k_1}(t_1),\gamma_{k_2}(t_2))+d_a(\gamma_{k_2}(t_2),y) \\
%&< \sum_{k=1}^{k_1-1}  L( \gamma_k, I_k) + L(\gamma_{k_1},(\inf I_{k_1},t_1) )+ L(\gamma_{k_1}, ( t_1,\sup I_{k_1}  ))\\
%&\quad+\sum_{k=k_1+1}^{k_2-1} L(\gamma_{k}, I_k)+L(\gamma_{k_2}, ( \inf I_{k_2},t_2  )) + L(\gamma_{k_2}, (t_2,\sup I_{k_2} )) \\
%&\quad + \sum_{k=k_2+1}^{N} L(\gamma_{k}, I_k)=d_a\left(x,y\right),
%\end{align*}
%where by convention, if $k_1=1$, $\sum_{k=1}^{k_1-1}  L( \gamma_k, I_k)=0$ and if $k_2=N$, $\sum_{k=k_2+1}^{N} L(\gamma_{k}, I_k)=0$. The last inequality is impossible and all the previous inequalities have to be equalities. This proves Lemma~\ref{LLL}.
%\end{proof}

\subsubsection{On the equality
  in~(\ref{eq:ineq})}\label{sec:equality_agmon}
We end up this section with some results in case of equality in the
inequality (\ref{eq:ineq}). We will prove in particular
Proposition~\ref{PoPo} which has been used in Section~\ref{lower
  bound} above to give lower bounds on the Agmon distance.
\begin{corollary} \label{ccc}  Let us assume that \textbf{[H1]} and \textbf{[H3]}  hold. 
Let $x,y \in \overline \Omega$ with $f(x)\leq f(y)$. Let us denote by $((\gamma_1,I_1),\ldots,(\gamma_N,I_N))$ the  curves given by Theorem~\ref{exis} ordered such that 
$$\lim_{t\to (\inf I_1)^+} \gamma_1(t)=x, \quad \lim_{t\to (\sup I_N)^-} \gamma_N(t)=y,$$
and which realize the Agmon distance between $x$ and $y$. If it holds: 
$$d_a\left(x,y\right)= f(y)-f(x) ,$$
then for all $i\in \left\{1,\ldots,N\right\}$, there exist measurable
functions $\lambda_i : I_i\to \mathbb R_+$ such that for almost every
$t$ in  $\left\{t\in I_i, \ \gamma_i(t)\in \Omega \right\}$ 
\begin{equation}\label{eq.lambdai_i}
\gamma_i'(t)=\lambda_i(t) \nabla f \left(\gamma_i(t) \right),
\end{equation}
and such that for almost every $t$ in ${\rm int}\left\{t\in I_i, \ \gamma_i(t)\in \partial \Omega \right\}$
\begin{equation}\label{eq.lambdai_i2}
\gamma_i'(t)=\lambda_i(t) \nabla_T f \left(\gamma_i(t) \right).
\end{equation}
Moreover, for all $i\in\{1,\ldots,N\}$, $\lambda_i\in L^{\infty}(I_i,\mathbb R_+)$. Finally, if $I_i$ is unbounded (i.e. $I_i=[0,+\infty)$ or $I_i=(-\infty,0]$),  it holds   for almost every $t\in I_i$,  $\lambda_i=1$.
\end{corollary}
\begin{proof}  
Using Lemma~\ref{h}, one gets using first the triangular inequality and then the Cauchy-Schwarz inequality 
\begin{align*}
f(y)-f(x)&= \sum_{k=1}^N \left( \int_{\{ t \in I_k, \  \gamma_k(t)\in \Omega\}}  (\nabla f)( \gamma_k)\cdot  \gamma_k' +  \int_{{\rm int} \{ t \in   I_k, \  \gamma_k(t)\in \partial \Omega\} }  (\nabla_T f)( \gamma_k)\cdot  \gamma_k'\right)\\
&\leq \sum_{k=1}^N \left( \int_{\{ t \in   I_k, \  \gamma_k(t)\in \Omega\}}  \vert \nabla f( \gamma_k)\vert \vert   \gamma_k' \vert +  \int_{{\rm int} \{ t \in   I_k, \  \gamma_k(t)\in \partial \Omega\} }  \vert \nabla_T f ( \gamma_k)\vert\vert  \gamma_k'\vert  \right)\\
&=\sum_{k=1}^N  L\left(\gamma_k,  I_k\right)=d_a\left(x,y \right).
\end{align*}
If $d_a\left(x,y\right)=f(y)-f(x)$, then the previous inequality is
necessarily an equality. Using the cases
  of equality in both the triangular inequality and the
Cauchy-Schwarz inequalities, for $k\in
\left\{1,\ldots,N\right\}$, there exists a nonnegative function 
$\lambda_k:I_k\to \mathbb R_+$ such that~\eqref{eq.lambdai_i} and~\eqref{eq.lambdai_i2} hold.

Let $i\in \{1,\ldots,N\}$. 
Let us first consider the case when 
$I_i=[0,1]$. Then, by construction, the curve $\gamma_i$ does not meet any
critical points of the functions~$f$ and~$f|_{\partial \Omega}$. This
implies that $\inf_{I_i} |\nabla f(\gamma_i)|>0$ and $\inf_{I_i} |\nabla_T f(\gamma_i)|>0$, and thus, since
$\|\gamma_i'\|_{L^\infty}<\infty$, one concludes that $\lambda_i\in L^{\infty}([0,1],\mathbb R_+)$.\\
Let us now consider the case when $I_i$ is not bounded. Using the construction of the
curves $(\gamma_k)_{k=1,\ldots,N}$,  this implies that $I_i$ is either
$(-\infty,0]$ or $[0,+\infty)$ and $\gamma_i$ is constructed using the
gradient flow of the eikonal solution near a critical point $x^*$ of
$f$ or of $f|_{\partial \Omega}$ (see Lemma~\ref{agmoncurve}). Let us assume that $x^*$ is a critical point of
$f|_{\partial \Omega}$ and $I_i=[0,+\infty)$ (the other cases are
treated similarly). Let $\Phi$ be the solution to~\eqref{eikonalequation} on the neighborhood  $V^*$ of $x^*$ in $\overline \Omega$ (see Proposition~\ref{eikonal}). The curve $\gamma_i$ satisfies by construction, Ran$(\gamma_i)\subset V^*$,  $
\lim_{t\to \infty} \gamma_i(t)=x^*$, 
 and on  $ (-\infty, 0]$,
\begin{equation}\label{eq.gammai'-1}
\gamma_i'=\left\{
\begin{aligned}
 &-\nabla \Phi (\gamma_i)  \ {\rm in \ }  \Omega,  \\ 
& -\nabla_T \Phi (\gamma_i) \ {\rm on \ } \partial \Omega.
\end{aligned} 
\right.
\end{equation} 
In addition, by the
previous reasoning, one also has on  $ (-\infty, 0]$,
\begin{equation}\label{eq.gammai'-2}
\gamma_i'=\left\{
\begin{aligned}
 & \lambda_i \nabla f (\gamma_i)  \ {\rm in \ }  \Omega,  \\ 
& \lambda_i\nabla_T f (\gamma_i) \ {\rm on \ } \partial \Omega,
\end{aligned} 
\right.
\end{equation}
for some measurable function $\lambda_i:[0,+\infty)\to \mathbb R_+$. 
Furthermore, from the last point in Proposition~\ref{eikonal},  it holds on $V^*\cap \pa \Omega$, 
\begin{equation}\label{eq.cong-eiko}
\vert \nabla_T \Phi\vert =\vert \nabla_T f\vert.
\end{equation}
Taking the norm in~\eqref{eq.gammai'-1} and~\eqref{eq.gammai'-2},    and using the fact
that~$\Phi$ solves \eqref{eikonalequation}  together with the
equality~\eqref{eq.cong-eiko}, one obtains that $\lambda_i(t)=1$ for almost every  $t\in I_i$. 
This concludes the proof of Corollary~\ref{ccc}. 
\end{proof}

%%%%%
\begin{sloppypar}
\noindent
Let us define the notion of generalized integral curves. 
\begin{definition}\label{generalized_curve}
Let $D\subset \overline \Omega$ be a $C^{\infty}$ domain and $X\in C^{\infty}(D,\mathbb R)$. Let $N\in \mathbb N^*$ and for $i\in \{1,\dots,N\}$, let  $I_i\subset \mathbb R$ be an interval and $\gamma_i: I_i\to D$ be Lipschitz and such that 
 $$\lim_{t\to (\inf I_1)^+} \gamma_1(t) \ \text{ and }  \ \lim_{t\to
   (\sup I_N)^-} \gamma_N(t) \text{ exist}$$
and for all $k\in \{1,\ldots,N-1\}$,
$$\lim_{t\to \left(\sup I_k\right)^-} \gamma_k(t)=\lim_{t\to \left(\inf I_{k+1}\right)^+} \gamma_{k+1}(t).$$
The set of curves $\{\gamma_1,\dots,\gamma_N\}$ is a generalized integral curve of the vector field 
$$\left\{
\begin{aligned}
  \nabla X   & \text{ in  }  D\cap \Omega, \\ 
 \nabla_T X & \text{ on  } D\cap \partial \Omega,
\end{aligned}
\right.
$$
 if for all $i\in \left\{1,\ldots,N\right\}$, there exist measurable
functions $\lambda_i : I_i\to \mathbb R_+$ such that for almost every
$t$ in  $\left\{t\in I_i, \ \gamma_i(t)\in D\cap \Omega \right\}$:
 $\gamma_i'(t)=\lambda_i(t) \nabla X \left(\gamma_i(t) \right)$, 
and such that for almost every $t$ in ${\rm int}\left\{t\in I_i, \ \gamma_i(t)\in \partial \Omega\cap D \right\}$: $\gamma_i'(t)=\lambda_i(t) \nabla_T X \left(\gamma_i(t) \right)$.
\end{definition}
The notion of generalized integral curve has been introduced in the case of manifolds without boundary in  \cite{helffer-sjostrand-85}. As introduced in Definition~\ref{generalized_curve}, the set of curves $\{\gamma_1,\dots,\gamma_N\}$ given by Corollary~\ref{ccc} is a generalized integral curve of the vector field $$\left\{
\begin{aligned}
  \nabla f   & \text{ in  }  \Omega, \\ 
 \nabla_T f & \text{ on  } \partial \Omega.
\end{aligned}
\right.
$$ Let us mention that in the case when $\Omega$ is a manifold without boundary, Corollary~\ref{ccc} is exactly \cite[Lemma A2.2]{helffer-sjostrand-85}.
\end{sloppypar}
\medskip

\noindent
Let us end this section with the following proposition which which has been used in Section~\ref{lower bound}. 
\begin{proposition} \label{PoPo} Let us assume that \textbf{[H1]} and
  \textbf{[H3]} hold.
Let us denote by $\{z_1,\ldots,z_n\}$  the local minima of $f|_{\partial \Omega}$ ordered such that $f(z_1)\leq f(z_2)\leq\ldots\leq f(z_n)$. Then, for all $i<j$, $(i,j)\in \{1,\ldots,n\}^2$, one has
$$d_a\left(z_i,z_j\right)>f(z_j)-f(z_i).$$
\end{proposition}
\begin{proof} From the inequality \eqref{eq:ineq}, one has
  $d_a\left(z_i,z_j\right)\geq f(z_j)-f(z_i)$ for all $i<j$. Let us prove
  Proposition~\ref{PoPo} by contradiction. Assume that
  $d_a\left(z_i,z_j\right)= f(z_j)-f(z_i)$ for some $i<j$. Denote by $((\gamma_1,I_1),\ldots,(\gamma_m,I_m))$ the curves given by Theorem~\ref{exis} ordered such that 
$$\lim_{t\to (\inf I_1)^+} \gamma_1(t)=z_i, \quad \lim_{t\to (\sup I_m)^-} \gamma_m(t)=z_j,$$
and which realize the Agmon distance between $z_i$ and $z_j$. Since
$d_a\left(z_i,z_j\right)= f(z_j)-f(z_i)$, from Corollary~\ref{ccc},
for all $i\in \left\{1,\ldots,m\right\}$, there exist measurable
functions $\lambda_i : I_i\to \mathbb R_+$ such that for almost every
$t$ in  $\left\{t\in I_i, \ \gamma_i(t)\in \Omega \right\}$,
$\gamma_i'(t)=\lambda_i(t) \nabla f \left(\gamma_i(t) \right)$, and
such that for almost every $t$ in ${\rm int}\left\{t\in I_i, \
  \gamma_i(t)\in \partial \Omega \right\}$, $\gamma_i'(t)=\lambda_i(t)
\nabla_T f \left(\gamma_i(t) \right)$. 
% (the set of curves $\{\gamma_1,\dots,\gamma_m\}$ is a generalized integral curve of the vector field $$\left\{
%\begin{aligned}
%  \nabla f   & \text{ in  }  \Omega, \\ 
% \nabla_T f & \text{ on  } \partial \Omega,
%\end{aligned}
%\right.
%$$
% according to Definition~\ref{generalized_curve}). 
 Let us recall that from Remark
\ref{curve2}, $I_1=(-\infty, 0]$ and $I_m=[0,+\infty)$ since $z_i$ and
$z_j$ are critical points of~$f|_{\partial \Omega}$. 

\medskip
\noindent
\underline{Step 1.} Let us show that  for  all $ t\in (-\infty, 0]$, $\gamma_1(t)\in \partial \Omega$. 
\medskip

\noindent
On the one hand, from Remark~\ref{curve2}, $\lim_{t\to -\infty} \gamma_1(t)=z_i$ and on  $ (-\infty, 0]$,
$$\gamma_1'=\left\{
\begin{aligned}
 &\nabla \Phi (\gamma_1)  \ {\rm in \ }  \Omega  \\ 
& \nabla_T \Phi (\gamma_1) \ {\rm on \ } \partial \Omega,
\end{aligned} 
\right.
$$
where $\Phi$ solves \eqref{eikonalequation}. On the other hand, from Corollary~\ref{ccc}, one has   on  $ (-\infty, 0]$,
$$
\gamma_1'(t)=\left\{
\begin{aligned}
 & \nabla f (\gamma_1)  \ {\rm in \ }  \Omega  \\ 
& \nabla_T f (\gamma_1) \ {\rm on \ } \partial \Omega.
\end{aligned} 
\right.
$$
Then, for all $t\leq 0$, one has $\frac{d}{dt} \left( f
  (\gamma_1)(t)-\Phi (\gamma_1(t))\right) =0$. Therefore there exists
$C>0$ such that   for  all $ t\in (-\infty, 0]$, $\gamma_1(t)\in
\left\{x, \, f(x)-\Phi(x)=C\right\}$. Since $\lim_{t\to -\infty}
\gamma_1(t)=z_i$ and   $(f-\Phi)(z_i)=f(z_i)$, one gets that
$C=f(z_i)$ and thus for  all $ t\in (-\infty, 0]$, $\gamma_1(t)\in
\left\{x,\, f(x)-\Phi(x)=f(z_i)\right\}$. From
Corollary~\ref{cor:eikonal_min} and Proposition~\ref{equaa}, $\gamma_1$ lives in a neighborhood
$U^*$ of $z_i$ such that (see Equation~\eqref{eq:carac_boundary}):
$$\left\{x,\, f(x)-\Phi(x)=f(z_i)\right\} \subset \partial \Omega.$$
We thus get that for all $t\le0$, $\gamma_1(t) \in \partial \Omega$,
and then $\gamma_1'(t)=\nabla_T f (\gamma_1(t))=\nabla_T \Phi (\gamma_1(t))$.   Step 1 is proved.

\medskip
\noindent
\underline{Step 2.}  We are going to show that for  all $ t\in I_k$,
$\gamma_k(t)\in \partial \Omega$.
\medskip

\noindent
 If it is not the case, from Step 1,
there exist $k\in \{2,\ldots,m\}$ and $t_k\in I_k$ such that
$\gamma_k(t_k) \in \Omega$. Let us define the first time, denoted by
$t^*$, for which the curves $((\gamma_2,I_2),\ldots,(\gamma_m,I_m))$
leave~$\partial \Omega$.  By construction of the curves $\gamma_1,\dots,\gamma_m$, there are two cases: either $t^*$ is finite (and thus belongs to ${\rm int}(I_k)$ for $k\in \{1,\dots,m\}$) or, there exist $j\in \{2,\dots,m-1\}$, $s<0$ and $z\in \pa \Omega$ such that $g(z)=0$, $\lim_{t\to +\infty} \gamma_j(t)=z$, $\lim_{t\to -\infty} \gamma_{j+1}(t)=z$ and $\gamma_{j+1}(-\infty,s)\subset \Omega$ in which case   $t^*=-\infty$. 
 Let us assume
that $t^*$ is finite and belongs to ${\rm int}(I_k)$ (the other case is treated similarly).

As in Step 1 of the proof of Proposition~\ref{four}, let us now
introduce a smooth tangential and
normal system of coordinates around $\gamma_k(t^*)$ in $\overline
\Omega$, denoted by $\phi(x)=(x_T,x_N)$. The function $\phi$ is defined
from a neighborhood of  $\gamma_k(t^*)$ in $\overline \Omega$ to
$\R^d$. Moreover, one has $x_N\geq 0$ and $x_N(x)=0$ if and only if
$x\in \partial \Omega$. We may assume that the neighborhood $V_\alpha \subset \R^d$ on which
$\phi$ is defined is such that
$\phi(V_\alpha)=U\times[0,\alpha]$ for $\alpha >0$ and $U \subset
\R^{d-1}$. Since $\partial_nf>0$ on $\partial \Omega$, $\alpha>0$ can be
chosen small enough such that $\nabla f(x) \cdot n(x)>0$ for all $x\in
V_{\alpha}$ where $n(x)=-\frac{\nabla x_N(x)}{\vert
\nabla x_N(x)\vert}$. Indeed, for $x\in \partial \Omega$, $n(x)$ is nothing but the
unit outward normal to $\partial \Omega$.

Now, by continuity of the curve $\gamma_k$, there exists $\ve>0$ such that $[t^*, t^*+\ve]\subset I_k$ and for all $t\in [t^*, t^*+\ve]$, $\gamma_k(t)\in V_{\alpha}$. The mapping $t\in [t^*, t^*+\ve]\mapsto x_N\left(\gamma_k(t)\right)$ is Lipschitz and satisfies: for almost every $s\in (t^*, t^*+\ve)$, 
$$\frac{d}{ds}x_N\left(\gamma_k(s)\right)=-\vert \nabla x_N(\gamma_k(s))\vert \, \gamma_k'(s)\cdot n\left(\gamma_k(s)\right).$$
Then, for all $t\in  [t^*, t^*+\ve]$, one has:
%-\vert \nabla x_d(\gamma_k(t))\vert \times \lambda_k(t)\,\partial_{n\left(\gamma_k(t)\right)}f\left(\gamma_k(t)\right)\leq 0
$$
\frac{d}{ds}x_N\left(\gamma_k(s)\right)=\left\{
\begin{aligned}
 &0 \text{ for  a.e.  }  s \in {\rm int}\, \left\{ u\in (t^*,t), \ \gamma_k(u)\in \partial \Omega\right\}  \\ 
&-\vert \nabla x_N(\gamma_k(s))\vert \,
\lambda_k(s)\,\nabla f(\gamma_k(s))\cdot n(\gamma_k(s)) \text{ for  a.e.  }  s \in  \left\{ u\in (t^*,t), \ \gamma_k(u)\in \Omega\right\}.
\end{aligned} 
\right.
$$
Since $\partial \left\{ u\in (t^*,t), \ \gamma_k(u)\in \partial
  \Omega\right\}$ is of Lebesgue measure zero (see Theorem~\ref{exis})
and since $\nabla f \cdot n >0$ in $V_\alpha$, one has from Lemma~\ref{h}, for all $t\in  [t^*, t^*+\ve]$
$$x_N\left(\gamma_k(t)\right)=x_N\left(\gamma_k(t)\right)-x_N\left(\gamma_k(t^*)\right)=\int_{t^*}^t\frac{d}{ds}x_N\left(\gamma_k(s)\right)\, ds\leq 0.$$
This implies that for all $t\in  [t^*, t^*+\ve]$, $x_N\left(\gamma_k(t)\right)=0$ and thus  $\gamma_k(t)\in \partial \Omega$ for all $t\in  [t^*, t^*+\ve]$. This contradicts the definition of $t^*$. Step 2 is proved.

\medskip
\noindent
\underline{Step 3.} End of the proof of Proposition~\ref{PoPo}. 
\medskip

\noindent
From the last two steps, for all $t\in
[0,+\infty)$, $\gamma_m(t)\in \partial \Omega$. From
Corollary~\ref{ccc}, one has
$\gamma_m'(t)=\nabla_Tf\left(\gamma_m(t)\right)$ for all $t\in
[0,+\infty)$ and therefore, the map $t\in [0,+\infty)\mapsto
f\left(\gamma_m(t)\right)$ is increasing (indeed, one has
$\frac{d}{dt}f\left(\gamma_m(t)\right)=\left\vert \nabla_T
  f\left(\gamma_m(t)\right)\right\vert^2$). This is impossible since
$z_j$ is a local minimum of $f|_{\partial \Omega}$. This concludes the
proof of Proposition~\ref{PoPo} by contradiction. 
\end{proof}

\subsection{Agmon distance in a neighborhood of the basin of
  attraction of a local minimum of $f|_{\partial \Omega}$ and eikonal equation}
  \label{eiko_agmon_basin}
The aim of this section is to generalize the results of Section~\ref{eiko_agmon} to relate the Agmon distance and the solution to an eikonal equation on a neighborhood of a basin of attraction $B_z$ (see Definition~\ref{Bz}) of a local minimum $z$ of $f|_{\pa \Omega}$. 
Let first introduce a  solution to the eikonal equation $\vert  \nabla \phi \vert^2 =  \vert  \nabla f \vert^2$ defined globally on a neighborhood of the boundary~$\partial \Omega$. 

\begin{proposition} \label{eikonalboundary} Let us assume that   \textbf{[H3]} holds.
There exists a neighborhood of~$\partial \Omega$ in~$\overline \Omega$, denoted $V_{\partial \Omega}$, such that there exists $\Phi \in C^{\infty}(V_{\partial \Omega},\mathbb R)$ satisfying 
\begin{equation} \label{eikonalequationboundary}
\left\{
\begin{aligned}
\vert  \nabla \Phi \vert^2 &=  \vert  \nabla f \vert^2   \ {\rm in \ }  \Omega \cap V_{\partial \Omega}  \\ 
\Phi &= f \ {\rm on \ } \partial \Omega  \\
\partial_n \Phi&=-\partial_n f \ {\rm on \ } \partial \Omega .\end{aligned}
\right.
\end{equation}
Moreover\label{page.vpaomega}, one has the following uniqueness results: if $\tilde \Phi$ is
a $C^{\infty}$ real valued function defined on a neighborhood
$\tilde{V}$ of $\partial \Omega$ satisfying~\eqref{eikonalequationboundary}, then $\tilde \Phi =\Phi$ on
$\tilde{V} \cap V_{\partial \Omega}$.
\end{proposition}
\begin{proof} 
Let $z\in \partial \Omega$. Using \cite{dimassi-sjostrand-99} or \cite{evans-10}, thanks to \textbf{[H3]}, there exists a neighborhood of $z$ in $\overline \Omega$, denoted by $\mathcal V_z$, such that there exists $\Phi \in C^{\infty}(\mathcal V_z,\mathbb R)$ satisfying 
$$
\left\{
\begin{aligned}
\vert  \nabla \Phi \vert^2 &=  \vert  \nabla f \vert^2   \ {\rm in \ }  \Omega \cap \mathcal V_z  \\ 
\Phi &= f \ {\rm on \ } \partial \Omega \cap \mathcal V_z  \\
\partial_n \Phi&=-\partial_n f \ {\rm on \ } \partial \Omega \cap \mathcal V_z. \end{aligned}
\right.$$
Moreover, $\mathcal V_z$ can be chosen such that the following
uniqueness result holds: if a function $\tilde \Phi \in C^{\infty}(\mathcal V_z,\mathbb R)$
satisfies the previous equalities, then $\tilde \Phi =\Phi$ on $\mathcal V_z$. Now, one concludes using the fact that $\partial \Omega$ is compact and can thus be covered by a finite number of these neighborhoods $(\mathcal V_z)_{z\in \partial \Omega}$.
\end{proof} 
\begin{remark} Let us mention another standard approach to prove
Proposition~\ref{eikonalboundary}, using the notion of viscosity
solutions. Let us recall some results  from~\cite[Theorem 5.1]{lions-82}. For $\left(x,y\right)\in \overline{\Omega}^2$, one defines 
$$\tilde d\left(x,y\right):=\inf_{T>0,\gamma}  \int_0^T \left\vert \nabla f\left(\gamma(t)\right)\right\vert dt  ,$$
where the infimum is taken over $T>0$ and over Lipschitz curves
$\gamma: [0,T]\to \overline \Omega$ which satisfy $\gamma(0)=x$,
$\gamma(T)=y$, $\vert \gamma'\vert\leq 1$. Then,  $v(x):=\inf \left\{
  f(y)+\tilde d\left(x,y\right), \ y\in \partial \Omega   \right\}$ is Lipschitz and is a viscosity solution to
$$
\left\{
\begin{aligned}
\vert  \nabla v \vert &=  \vert  \nabla f \vert   \ {\rm in \ }  \Omega   \\ 
v &= f \ {\rm on \ } \partial \Omega.
 \end{aligned}
\right.$$
Let us notice that this implies $\vert \partial_nv \vert=\vert \partial_nf \vert$ on $\pa \Omega$. 
To prove Proposition~\ref{eikonalboundary} using this result, one has
to show that $v$ is $C^{\infty}$ near $\partial \Omega$ and
$ \partial_nv =-\partial_nf$. This is a consequence of the
characteristic method, see~\cite[Section 1.2]{lions-82}. \end{remark}
\begin{remark}
Let $x^*$ be a local minimum of $f|_{\partial \Omega}$ and let us
denote by $\tilde{\Phi}$ the solution to the eikonal
equation~\eqref{eikonalequation_min} introduced in
Corollary~\ref{cor:eikonal_min}, defined on a neighborhood $V^*$ of~$x^*$. Then, one has on $V^* \cap V_{\partial \Omega}$:
$$\tilde{\Phi}=\Phi-f(x^*)$$
where $\Phi$ is the solution
to~\eqref{eikonalequationboundary}. 
\end{remark}
\noindent
Let us now introduce the function $f_-$ which will be used in the sequel. 
%This function $f^-$ will be used again in Section~\ref{geometry} when defining suitable coordinates near the boundary.
\begin{proposition} \label{ffmoins} Assume that   \textbf{[H3]}  holds.
Let $\Phi \in C^\infty(V_{\partial \Omega},\R)$ be the function introduced in Proposition~\ref{eikonalboundary}.
Let us define the function $f_- \in C^\infty(V_{\partial \Omega},\R)$ by
\begin{equation} \label{fmoins}
f_-=\frac{\Phi-f}{2}  .
\end{equation}
\label{page.fmoins}
Then, $f_-=0$ on $\partial \Omega$, and up to choosing a smaller neighborhood $V_{\partial \Omega}$ of
$\partial \Omega$, the function $ f_-$ is positive in $V_{\partial \Omega}\setminus \partial \Omega$ and $\vert \nabla f_-\vert >0$ on $V_{\partial \Omega}$.
\end{proposition}
\begin{proof}
Since $\partial_n(\Phi-f)=-2\partial_n f<0$ and $\Phi=f$ on $\partial
\Omega$,  then, up to choosing a smaller neighborhood $V_{\partial \Omega}$ of
$\partial \Omega$, one has $\Phi> f $ on $V_{\partial \Omega}\setminus \partial \Omega$ and $\vert \nabla (\Phi-f)\vert >0$ on $V_{\partial \Omega}$.
\end{proof}
%============================================%

%Notice that the vector field $\nabla f_-$ on $\partial \Omega$ as the
%following components:
%$$
%\forall x \in \partial \Omega, \, \nabla_T  f_-(x)=0, \  \pa_{n}  f_-(x) = -\pa_{n} f<0.
%$$
%In order to build a local system of coordinates in a neighborhood of
%$z \in \partial \Omega$, one
%then proceeds as follows, following \cite{helffer-nier-06}. Let $x'$
%be a system of coordinates in $\pa\Omega$ around some point
%$z\in \partial \Omega$. One first extends them as constants along the integral curves of the vector field $\nabla  f_-$, and therefore one obtains  $d-1$
%smooth functions on $V_{\partial \Omega}$. Secondly, one defines
%\begin{equation}\label{xd}
%x_{d}:=- f_-.
%\end{equation} 
%One has then constructed
%a smooth system of local coordinates $(x',x_N)$ o
%On $V(\partial\Omega)$ in a neighborhood of $z$. 
%The boundary
%$\pa\Omega$ is locally define in these coordinates by $\{x_{d}=0\}$
%and the interior $\Omega$ by $\{x_{d}<0\}$. 
\noindent
We are now in position to prove the main result of this
section. 
   
\begin{proposition} \label{TTH} Let us assume that \textbf{[H1]} and \textbf{[H3]}  hold. Let $\Phi$ be the function given
by Proposition~\ref{eikonalboundary}. 
Denote by $z$ a local minimum of $f|_{\partial \Omega}$ and denote by
$B_{z} \subset \partial \Omega$ the associated basin of attraction
(see Definition~\ref{Bz}). Besides, let $\Gamma_{z}
\subset \partial \Omega$ be an open domain such that $\overline
{\Gamma_{z}}\subset B_{z}$ and $z\in \Gamma_{z}$. Then there exists a
neighborhood  of $\overline{\Gamma_{z}}$ in $\overline \Omega$, denoted by $V_{\Gamma_{z}}$, such that $\overline{\partial V_{\Gamma_{z}} \cap \partial \Omega}\subset  B_{z}$ and for all $x\in V_{\Gamma_{z}}$,
$$d_a(x,z)=\Phi(x)-f(z).$$
\end{proposition}
\noindent
Notice that in this proposition, $\Gamma_{z}$ can be chosen  as large as needed  in $B_{z}$.
\begin{proof} Let $\Phi$ be the function given
by Proposition~\ref{eikonalboundary}. The proof is divided into three steps.
 
\medskip
\noindent
\underline{Step 1.} Let us first define $V_{\Gamma_{z}}$. 
\medskip

\noindent
To this end let us denote by $ f_-$ and $V_{\partial \Omega}$ respectively the function and the neighborhood of $\partial \Omega$  given by Proposition~\ref{ffmoins}.  For $\varepsilon>0$ small enough one defines 
\begin{equation}\label{eq.Vepsilon}
V_{\varepsilon}= \left\{ y\in \Omega, \ 0\leq  f_- (y)\leq  \varepsilon \right\}\subset V_{\partial \Omega}.
\end{equation}
The parameter $\ve>0$ can be chosen such that there is no critical point of $f$ on $\partial V_{\ve} \cap \Omega=\{ y\in \Omega, \  f_-(y)=
 \varepsilon\}$. The set $V_{\ve}$ is a neighborhood of $\partial \Omega$ in $\overline{\Omega}$ (see Figure~\ref{fig:ve} for a schematic representation). Let us now fix such a $\ve>0$. Assumption \textbf{[H3]} together with the fact that  $\partial_n\Phi<0$ on $\partial \Omega$, imply that there exists a neighborhood $V_{\Gamma_{z}}$
of $\overline{\Gamma_{z}}$ in $\overline \Omega$, such that $V_{\Gamma_{z}}\subset V_{\varepsilon}$, $\overline{\partial
  V_{\Gamma_{z}} \cap \partial \Omega}\subset  B_{z}$ and
$$\partial_n\Phi>0, \ {\rm on}\ \partial V_{\Gamma_{z}}\cap \Omega.$$
 The set $V_{\Gamma_z}$ is 
schematically represented on Figure~\ref{fig:ve_z}.
\begin{figure}[!h]
\begin{center}
\begin{tikzpicture}
\draw (0,0) ellipse (3 and 2)  ;
\draw (0,0) ellipse (2.6 and 1.5)  ;
\tikzstyle{vertex}=[draw,circle,fill=black,minimum size=6pt,inner sep=0pt]
\draw (-3,0) node[vertex,label=above left:{$z$}] (v) {};
\draw[dashed,<->] (0,-2)--(0,-1.5) node[above=0.1cm] {$\ve$};
\draw (0,-2) node[midway,below=2cm] {Domain $\Omega$};
\draw (1,-1) node[above=3cm] {$V_{\ve}$};
\draw (1,-1) node[right=1.7cm] {$\partial \Omega$};
\draw[dashed,->] (0.7,2.2)--(0,1.7);
\end{tikzpicture}
\caption{The set $V_{\ve}$.}
 \label{fig:ve}
 \end{center}
\end{figure}

\begin{figure}
\begin{center}
\begin{tikzpicture}

\tikzstyle{vertex}=[draw,circle,fill=black,minimum size=6pt,inner sep=0pt]
\draw[-] (-3.6,7)--(3.6,7);
\draw[thick, ->] (-4,6)--(4,6);
\draw[->] (0,4)--(0,7.6);
\draw[<->] (-3,6.1)--(-3,6.9) node[below left] {$V_{\ve}$} ;
\draw (0.3,7.7) node[below] { $ f_-$};

\draw (-2.7,6) ..controls (0,6.8) .. (2.7,6) node[midway,right=0.9cm]{$V_{\Gamma_z}$} ;

\draw (0 ,6) node[vertex,label=north east: {$z$}](v){};
\draw[<->](-2.1,5.5)--(2.1,5.5)  node[midway,below right] {\Ga{z}} ;
\draw[<->](-3.5,4.5)--(3.5,4.5)  node[midway,below right] {$B_{z}$} ;

\draw[dashed] (-2.1,5.5)--(-2.1,6) ;
\draw[dashed] (2.1,5.5)--(2.1,6) node[midway,right=1cm] {$\partial \Omega$};
\draw[dashed] (-3.5,4.5)--(-3.5,6);
\draw[dashed] (3.5,4.5)--(3.5,6);

\end{tikzpicture} 
\caption{The set $V_{\Gamma_z}$.}
 \label{fig:ve_z}
\end{center}
\end{figure}
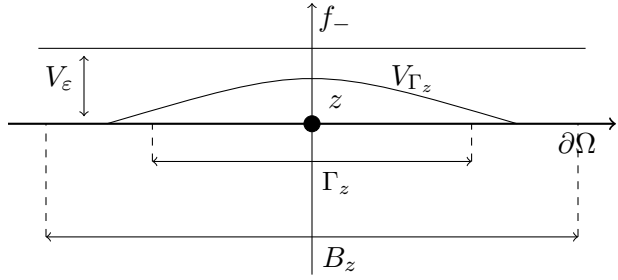
\medskip
\noindent
\underline{Step 2.} Let us first prove that for all $x\in V_{\Gamma_{z}}$, $d_a(x,z)\geq \Phi(x)-f(z)$. 
\medskip

\noindent
For $x\in V_{\Gamma_{z}}$,  denote by $((\gamma_1,I_1),\ldots,(\gamma_N,I_N))$ the curves given by Theorem~\ref{exis} ordered such that 
$$\lim_{t\to (\inf I_1)^+} \gamma_1(t)=z, \quad \lim_{t\to (\sup I_N)^-} \gamma_N(t)=x,$$
and which realize the Agmon distance between $x$ and $z$. %Since $-\varepsilon<x_N(x)\leq0$ and $x_N(z)=0$ and since for all $k\in \left\{1,\ldots,N-1\right\}$, $\lim_{t\to \left(\sup I_k\right)^-} \gamma_k(t)=\lim_{t\to \left(\inf I_{k+1}\right)^-} \gamma_{k+1}(t)\in \overline{\Omega}$,
 One has to deal with the two following cases: 
 \begin{enumerate}
 \item  either $\forall$ $k\in \left\{1,\ldots,N\right\}$, $\forall t\in I_k$, $ \gamma_k(t)\in V_{\ve},$
 \item or $\exists$ $k\in \left\{1,\ldots,N\right\}$ and  $\exists$ $t\in I_k$,  $\gamma_k(t)\in \overline{\Omega}\setminus V_{\ve}$.
 \end{enumerate}
In the first case, since $\Phi$ is defined on $V_{\ve}$, it holds 
%$d_a(z,x)\geq \Phi(x)-\Phi(z)$
$$\Phi(x)-f(z)=\Phi(x)-\Phi(z)= \sum_{j=1}^N\int_{I_j} \frac{ d   }{d  t} \Phi \circ \gamma_j(t) \, dt.$$
 Using Lemma~\ref{h} and the fact that $|\nabla \Phi|=g$ on $\Omega\cap V_{\partial \Omega}$ and $|\nabla_T \Phi|=g$ on $\partial \Omega$, it holds, for all $j \in \{1, \ldots ,N\}$,
$$\int_{I_j} \frac{ d   }{d  t} \Phi \circ \gamma_j(t) \, dt \le L(\gamma_j,I_j)$$
and thus
 $$\Phi(x)-f(z)\leq \sum_{j=1}^N L(\gamma_j, I_j)=d_a(x,z).$$
Let us now consider
the second case. Let us introduce $k_1 \in \{1,\ldots,N\}$ and $t_1\in I_{k_1}$  such that for all $t<t_1$, $\gamma_{k_1}(t) \in
V_\ve$, for all $k \in \{1, \ldots, k_1-1\}$, for all $t\in I_k$,
$\gamma_k(t) \in V_\ve$ and such that there exists $\beta>0$ such that for all $t\in (t_1,t_1+\beta]$, $\gamma_{k_1}(t) \notin
V_\ve$. The couple $(k_1,t_1)$ thus represents the
``first time'' the curves $\gamma_1,\ldots,\gamma_N$ leave
$V_{\ve}$. Likewise, let us introduce $k_2 \in \{1,\ldots,N\}$ and $t_2\in I_{k_2}$  such that for all $t>t_2$, $\gamma_{k_2}(t) \in
V_\ve$, for all $k \in \{k_2+1, \ldots, N\}$, for all $t\in I_k$,
$\gamma_k(t) \in V_\ve$ and such that there exists $\beta>0$ such that for all $t\in [t_2-\beta,t_2)$, $\gamma_{k_2}(t) \notin
V_\ve$. The couple $(k_2,t_2)$ thus represents the
``last time'' the curves $\gamma_1,\ldots,\gamma_N$ leave
$\overline \Omega\setminus	V_{\ve}$. From Step $1$, there is no critical point of $f$ on $\partial V_{\ve}\cap \Omega=\{ y\in \Omega, \  f_-(y)=
 \varepsilon\}$. Therefore, by construction of the curves
 $(\gamma_k)_{k=1,\ldots,N}$, the times $t_1$ and $t_2$ are
 finite and belong respectively to ${\rm int}\,I_{k_1}$ and ${\rm int}\,I_{k_2}$. One has by continuity of $\gamma_{k_1}$ and $\gamma_{k_2}$, $ f_-(\gamma_{k_1}(t_1))= f_-(\gamma_{k_2}(t_2))=\ve$.
Since $\Phi$ is defined on $V_{\varepsilon}$, using again Lemma~\ref{h} and the fact that $|\nabla \Phi|=g$ on $\Omega$ and $|\nabla_T \Phi|=g$, one has
\begin{align*}
\left\vert \Phi(\gamma_{k_1}(t_1))-\Phi(z) \right\vert&\leq \sum_{j=1}^{k_1-1}	  L\left(\gamma_j,I_j\right)+L(\gamma_{k_1},(\inf I_{k_1},t_1)).
\end{align*}
In addition, using Lemma~\ref{LLL},  
$$\sum_{j=1}^{k_1-1}	 L\left(\gamma_j,I_j\right)+L(\gamma_{k_1},(\inf I_{k_1},t_1))=d_a(z,\gamma_{k_1}(t_1)).$$
Thus $\left\vert \Phi(\gamma_{k_1}(t_1))-\Phi(z) \right\vert\leq
d_a(z,\gamma_{k_1}(t_1))$. By similar arguments, one obtains $\vert
\Phi(x)-\Phi(\gamma_{k_2}(t_2))\vert \leq d_a(\gamma_{k_2}(t_2),x)$.  Thanks to
the definition~\eqref{fmoins} of $ f_-$ and using the fact that $ f_-(\gamma_{k_1}(t_1))= f_-(\gamma_{k_2}(t_2))=\ve$, one has $\vert f(\gamma_{k_2}(t_2))-f(\gamma_{k_1}(t_1))\vert=\vert \Phi(\gamma_{k_2}(t_2))-\Phi(\gamma_{k_1}(t_1))\vert$. In addition, using  (\ref{eq:ineq}) one obtains  $d_a(\gamma_{k_1}(t_1),\gamma_{k_2}(t_2))\geq \vert f(\gamma_{k_2}(t_2))-f(\gamma_{k_1}(t_1))\vert=\vert \Phi(\gamma_{k_2}(t_2))-\Phi(\gamma_{k_1}(t_1))\vert$. Using Lemma~\ref{LLL} and gathering these three last inequalities, one gets
\begin{align*}
d_a(x,z)&=d_a(z,\gamma_{k_1}(t_1))+d_a(\gamma_{k_1}(t_1),\gamma_{k_2}(t_2))+d_a(\gamma_{k_2}(t_2),x)\\
&\geq \vert \Phi(z)-\Phi(\gamma_{k_1}(t_1))\vert + \vert\Phi(\gamma_{k_2}(t_2))-\Phi(\gamma_{k_1}(t_1))\vert +\vert\Phi(x)-\Phi(\gamma_{k_2}(t_2))\vert\\
&\geq \vert \Phi(z)-\Phi(x)\vert\geq \Phi(x)-\Phi(z)=\Phi(x)-f(z).
\end{align*}

\medskip
\noindent
\underline{Step 3.} Let us now show that for all $x\in V_{\Gamma_{z}}$,
$d_a(x,z)\leq \Phi(x)-f(z)$. 
\medskip

\noindent
The proof of this inequality is very
similar to the second step in the proof of Proposition~\ref{equaa}. For
$x\in V_{\Gamma_{z}}$, let $\gamma$ be defined by
\eqref{eq:X}--\eqref{eq:c} (where $\Phi$ is defined by~\eqref{eikonalequationboundary}), with $\gamma(0)=x$. The function $\gamma$
is with values in $V_{\Gamma_{z}}$   since
$\partial_n\Phi>0$ on $\partial V_{\Gamma_{z}}\cap \Omega$ and $\partial V_{\Gamma_{z}}\cap \partial \Omega \subset B_z$. Thus $\gamma$ is defined on $\mathbb
R_+$. Thanks to the definition~\eqref{eq:X} of the vector field $X$, if
there exists a time $t_{\partial \Omega}$ such that $\gamma(t_{\partial \Omega})$ is in $\partial
\Omega$, then, for all $t \ge t_{\partial \Omega}$, $\gamma(t) \in \partial \Omega$. The function $t\in \mathbb R_+\mapsto \gamma(t)$ is continuous, piecewise $C^{\infty}$ and satisfies 
$$\lim_{t\to +\infty}\gamma(t)=z.$$
Then, as in the second step of the proof of Proposition~\ref{equaa}, one has 
$$\Phi(x)-f(z)=L\left(\gamma,(0,\infty)\right).$$
Using Lemma~\ref{inca}, one obtains that $d_a(x,z)\leq  L(\gamma,(0,\infty))=\Phi(x)-f(z)$. This proves the inequality: for all $x\in V_{\Gamma_z}$, $d_a(x,z)\leq \Phi(x)-f(z)$.
This concludes the proof of Proposition~\ref{TTH}.
\end{proof} 
\noindent
The following corollary is similar to Corollary~\ref{ccc} in the sense that is deals with the case of equality between the Agmon distance and the function $\Phi$ introduced in Proposition~\ref{eikonalboundary}. Corollary~\ref{re.phi_curve} will be needed in  the proof of Proposition~\ref{pr.WKB-comparison}. 

\begin{corollary}\label{re.phi_curve} Let us assume that \textbf{[H1]} and \textbf{[H3]}  hold.  Let $\Phi$ be the function introduced in Proposition~\ref{eikonalboundary} and, let $ f_-$ and $V_{\partial \Omega}$ be respectively the function and the neighborhood of $\partial \Omega$  given by Proposition~\ref{ffmoins}.  Let $V_\alpha$ be defined by~\eqref{eq.Vepsilon},  the parameter $\alpha >0$ is chosen such that:
\begin{enumerate}
\item[(i)] $V_{\alpha}\subset V_{\pa \Omega}$,  
\item[(ii)]  there is no critical point of $f$ on $\partial V_{\alpha} \cap \Omega=\{ w\in \Omega, \  f_-(w)=
\alpha\}$, 
\item[(iii)]  $\pa_nf>0$ on $\partial V_{\alpha} \cap \Omega$, 
\item[(iv)]  $\pa_nf^-<0$ on $\partial V_{\alpha} \cap \Omega$, and
\item[(v)]   $\vert \nabla \Phi\vert  \neq 0$ in $V_{\alpha}$.
\end{enumerate}
Notice that it is possible to choose such an $\alpha>0$ since $\pa_nf^-=\partial_n\Phi=-\partial_nf<0$ on $\pa \Omega=V_0$. 
 Let $x,y\in V_{\alpha}$ and denote by $((\gamma_1,I_1),\ldots,(\gamma_N,I_N))$ the curves given by Theorem~\ref{exis} ordered such that $\lim_{t\to (\inf I_1)^+} \gamma_1(t)=x$, $\lim_{t\to (\sup I_N)^-} \gamma_N(t)=y$
and which realize the Agmon distance between $x$ and $y$. 
% As in the proof of Proposition~\ref{TTH},  there are two cases: 
% \begin{enumerate}
% \item  either $\forall$ $k\in \left\{1,\ldots,N\right\}$, $\forall t\in I_k$, $ \gamma_k(t)\in V_{\alpha},$
% \item or $\exists$ $k\in \left\{1,\ldots,N\right\}$ and  $\exists$ $t\in I_k$,  $\gamma_k(t)\in \overline{\Omega}\setminus V_{\alpha}$.
% \end{enumerate}
%In the second case, the couple $(t_1,k_1)$ is defined  as in Step 2 of the proof of  Proposition~\ref{TTH}. 
Let us assume that $$\Phi(x)-\Phi(y)=d_a(x,y).$$
Then, for all $i\in \left\{1,\ldots,N\right\}$, ${\rm Im}\, \gamma_i\subset V_{\alpha}$ and there exist measurable
functions $\lambda_i : I_i\to \mathbb R_+$ such that for almost every
$t$ in  $\left\{t\in I_i, \ \gamma_i(t)\in \Omega \right\}$, one has $\gamma_i'(t)=-\lambda_i(t) \nabla \Phi \left(\gamma_i(t) \right)$,
and such that for almost every $t$ in ${\rm int}\left\{t\in I_i, \ \gamma_i(t)\in \partial \Omega \right\}$, one has $\gamma_i'(t)=-\lambda_i(t) \nabla_T \Phi \left(\gamma_i(t) \right)$. Moreover, if $I_i$ is not bounded (namely $I_i=(-\infty,0]$ or $I_i=[0,+\infty)$), $ \lambda_i(t)$=1 for almost every $t\in I_i$, and if $I_i=[0,1]$, $\lambda_i\in L^{\infty}([0,1],\mathbb R_+)$.
 \end{corollary}
 \noindent
According to Definition~\ref{generalized_curve}, the set of curves $\{\gamma_1,\dots,\gamma_N\}$ introduced in Corollary~\ref{re.phi_curve} is a generalized integral curve of the vector field $$\left\{
\begin{aligned}
  -\nabla \Phi   & \text{ in  }  V_{\alpha}\cap \Omega, \\ 
 -\nabla_T \Phi & \text{ on  } \partial \Omega.
\end{aligned}
\right.
$$
\begin{proof}
The proof of this statement is similar to the proof of Corollary~\ref{ccc}. 
Let us first prove that for all $i\in \left\{1,\ldots,N\right\}$, ${\rm Im}\, \gamma_i\subset V_{\alpha}$. If it is not the case, then there exist $k\in \left\{1,\ldots,N\right\}$ and  $t\in I_k$ such that  $\gamma_k(t)\in \overline{\Omega}\setminus V_{\alpha}$. Let the couples $(t_1,k_1)$ and $(t_2,k_2)$ be defined  as in Step 2 of the proof of  Proposition~\ref{TTH}. Then, one has (see the second step of the proof of Proposition~\ref{TTH}), $d_a(x,\gamma_{k_1}(t_1))\ge \Phi(x)-\Phi(\gamma_{k_1}(t_1))$, $d_a(\gamma_{k_1}(t_1),\gamma_{k_2}(t_2))\ge \Phi(\gamma_{k_1}(t_1))-\Phi(\gamma_{k_2}(t_2))$ and $d_a(\gamma_{k_2}(t_2),y)\ge \Phi(\gamma_{k_2}(t_2))-\Phi(y)$. Since one has  by assumption and  from Lemma~\ref{LLL}:
 $$\Phi(x)-\Phi(y)=d_a(x,y)=d_a(x,\gamma_{k_1}(t_1))+d_a(\gamma_{k_1}(t_1),\gamma_{k_2}(t_2))+d_a(\gamma_{k_2}(t_2),y),$$
all the previous inequalities are equalities and in particular, it holds:
$$d_a(\gamma_{k_1}(t_1),\gamma_{k_2}(t_2))= \Phi(\gamma_{k_1}(t_1))-\Phi(\gamma_{k_2}(t_2)) = f(\gamma_{k_2}(t_2)) - f(\gamma_{k_1}(t_1))\ge 0.$$
Using Corollary~\ref{ccc}, this implies that when restricting $\gamma_{k_1}$ to $I_{k_1}\cap [t_1,\infty)$ and $\gamma_{k_2}$ to $I_{k_2}\cap (-\infty,t_2]$ the set of curves $\{ \gamma_{k_1},\dots,\gamma_{k_2}\}$ is a generalized integral curve of 
$$\left\{
\begin{aligned}
  \nabla f   & \text{ in  }  \Omega , \\ 
 \nabla_T f & \text{ on  } \pa  \Omega
\end{aligned}
\right.
$$ 
 see Definition~\ref{generalized_curve}. Let $D=\Omega\setminus V_\alpha$ ($\pa D=\Omega\cap V_\alpha=\{w\in \Omega, f_-(w)=\alpha\}$ is $C^{\infty}$  since $f^-$ is $C^{\infty}$ and $\pa_nf^-<0$ on $\partial V_{\alpha} \cap \Omega=\pa D$ which implies that there is no critical point of $f_-$ on $\pa D$). Then, from  Corollary~\ref{ccc} and  by definition of $(t_1,k_1)$ (see the second step of the proof of Proposition~\ref{TTH}), there exists $\ve>0$ and a measurable function $\lambda$: $[t_1, t_1+\ve]\to \mathbb R^+$  such that for all $t \in [t_1, t_1+\ve]$:
$$ \gamma_{k_1} '(t)=\lambda(t) \nabla f\left(  \gamma_{k_1} (t) \right) $$
and for all $t \in (t_1, t_1+\ve]$: 
\begin{equation}\label{eq.cont}
 \gamma_{k_1}(t) \in D.
 \end{equation}
 As in Step 2 of the proof of Proposition~\ref{PoPo}, let us 
introduce a smooth tangential and
normal system of coordinates around $\gamma_{k_1}(t_1)\in \pa D$ in $\overline D$, denoted by $\phi(x)=(x_T,x_N)$. The function $\phi$ is defined
from a neighborhood of  $\gamma_{k_1}(t_1)$ in $\overline D$ to
$\R^d$. Moreover, one has $x_N\geq 0$ and $x_N(x)=0$ if and only if
$x\in  \pa D $. We may assume that the neighborhood $U_\beta \subset \overline{D}$ on which
$\phi$ is defined is such that
$\phi(U_\beta)=U\times[0,\beta]$ for $\beta >0$ and $U \subset
\R^{d-1}$. Since $\partial_nf>0$ on $\partial D$, $\beta>0$ can be
chosen small enough such that $\nabla f(x) \cdot n(x)>0$ for all $x\in
U_{\beta}$ where $n(x)=-\frac{\nabla x_N(x)}{\vert
\nabla x_N(x)\vert}$. Indeed, for $x\in \partial  D$, $n(x)$ is nothing but the
unit outward normal to $\partial D$. 
Now, by continuity of the curve $\gamma_{k_1}$, there exists $\mu>0$ such that for all $t\in (t_1,t_1+\mu]$, $\gamma_{k_1}(t)\in U_{\beta}$. The same considerations as in Step 2 of the proof of Proposition~\ref{PoPo} can then be used to show that:
 $$x_N(\gamma(t))\le 0,$$
 for all $t\in [t_1,t_1+\mu]$ and thus  $\gamma_{k_1}(t)\notin  D$ for all $t\in  [t_1,t_1+\mu]$. This contradicts~\eqref{eq.cont}. Thus, for all $i\in \left\{1,\ldots,N\right\}$, ${\rm Im}\, \gamma_i\subset V_{\alpha}$.

Then, the announced result follows by the same arguments as those used in the proof of Corollary~\ref{ccc} with $f$ replaced by $\Phi$ together with the fact that $\Phi$  satisfies~\eqref{eikonalequationboundary} on $V_{\alpha}$ and  for all $i\in \left\{1,\ldots,N\right\}$, ${\rm Im}\, \gamma_i\subset V_{\alpha}$.\end{proof}

%%%%%%%%%%%%%%%

%%%%%%%

%%%

\section{Construction of the quasi-modes and proof of Theorem~\ref{TBIG0}}\label{sec:quasi-modes}
The aim of this section is to build the quasi-modes $\tilde{u}$ and $(\tilde{\phi}_i)_{i=1, \ldots n}$ satisfying the conditions stated in Section~\ref{flat}. Let us recall that ${\rm span}(\tilde{u})$ (resp. ${\rm span}(\tilde \phi_i, i=1,\ldots,n)$) is intended to be a good approximation (in the sense made precise in items 1 and 2 in  Proposition~\ref{ESTIME}) of ${\rm Ran} \left(\pi_{[0,\sqrt{h})}\left(L^{D,(0)}_{f,h}(\Omega)\right)\right)$ (resp. ${\rm Ran} \left(\pi_{[0,h^{\frac32})}\left(\Delta^{D,(1)}_{f,h}(\Omega)\right)\right)$).
\medskip

\begin{sloppypar}
\noindent
As recalled in Proposition~\ref{ran}, it is known that the dimension of ${\rm Ran} \left(\pi_{[0,h^{\frac32})}\left(\Delta^{D,(1)}_{f,h}(\Omega)\right)\right)$ is equal to the number of generalized critical points of index $1$ (see~\cite[Section 3]{helffer-nier-06}) which are in our setting, thanks to assumptions \textbf{[H1]}, \textbf{[H2]} and \textbf{[H3]}, the local minima $(z_i)_{i=1,\ldots,n}$ of $f|_{\partial \Omega}$. In addition, it is known that the $1$-forms in ${\rm Ran} \left(\pi_{[0,h^{\frac32})}\left(\Delta^{D,(1)}_{f,h}(\Omega)\right)\right)$ are localized in the limit $h \to 0$ in small neighborhoods of the local minima $(z_i)_{i=1,\ldots,n}$.
\medskip

\noindent
For each local minimum $z_i$, we construct an associated quasi-mode $\tilde \phi_i$, using an auxiliary Witten Laplacian on $1$-forms with mixed tangential-normal boundary conditions. This Witten Laplacian is defined on a domain $\dot{\Omega}_i \subset \Omega$ with suitable boundary conditions, so that its only small eigenvalue (namely in the interval $[0,h^{\frac32})$) is $0$, thanks to a complex property (see~\cite{helffer-nier-06,le-peutrec-10}). The associated eigenform is localized near $z_i$, which can be proven thanks to Agmon estimates. Moreover, a precise estimate of this eigenform can be obtained thanks to a WKB expansion. The quasi-mode $\tilde\phi_i$ is then this eigenform multiplied by a suitable cut-off function.
\medskip

\noindent
This section is organized as follows. In Section~\ref{geometry}, we define a Witten Laplacian with mixed boundary conditions on a open domain $\dot\Omega_i\subset \Omega$ associated to each $z_i$, $i\in \{1,\dots,n\}$, and we study its spectrum.  Section~\ref{sec:contruct_quasimode} is dedicated to the construction of the quasi-modes $((\tilde\phi_i)_{i=1,\dots,n}, \tilde u)$. In Section~\ref{sec:agmonestimate}, we prove Agmon estimates on the eigenform associated with the smallest eigenvalue of the Witten Laplacian with mixed boundary conditions on $\dot\Omega_i$ and in Section~\ref{sec:wkb} we compare this eigenform with a WKB approximation. We finally use this construction and these estimates to prove Theorem~\ref{TBIG0} in Section~\ref{sec:goodquasimodes}. 
\end{sloppypar}
\subsection{Geometric setting and definition of the Witten Laplacians
  with mixed boundary conditions} \label{geometry}

This section is organized as follows. In Section~\ref{sec:trace_diff_form}, we discuss some general results on traces of differential
forms and we introduce the Witten Laplacians with mixed
tangential Dirichlet   and normal
Dirichlet boundary conditions on manifolds with boundary. In Section~\ref{sec:dotomega}, the domain $\dot \Omega_i\subset \Omega$ associated with each $z_i$, $i\in \{1,\dots,n\}$, is defined. Finally, Section~\ref{sec:def_mixed} is dedicated to the study of the spectrum of the Witten Laplacian with mixed tangential Dirichlet boundary conditions and normal Dirichlet boundary conditions on $\dot \Omega_i$.

\subsubsection{Trace estimates for differential forms and Witten
  Laplacians on Lipschitz domain with mixed boundary conditions}\label{sec:trace_diff_form}

In this section, we first discuss some general results on traces of differential
forms. This is crucial to then build the Witten Laplacians with mixed
boundary conditions. In the following, $\dot{\Omega}$ refers to any  submanifold $\dot{\Omega}$ of $\Omega$ with Lipschitz boundary\label{page.omegapoint}. We will call such a submanifold a Lipschitz domain.

%\comment{ je ne suis pas tres a l'aise avec la notion de
%  ``domaine''... est-ce qu'on veut dire sous-variete ? quel est
%  exactement l'objet qu'on considere ?}

We first recall that for any Lipschitz domain $\dot{\Omega}$, the trace application
$$\left\{
\begin{aligned}
\Lambda^{p}H^{1} ( \dot \Omega ) &\to \Lambda^{p}H^{\frac12} (\pa \dot \Omega ) \\
G &\mapsto G|_{\partial \dot \Omega}
\end{aligned}
\right.
$$
is a linear continuous and surjective application. We would like to present
extensions of this result to less regular forms.

\medskip
\noindent
\underline{Weak definition of traces}
\medskip

\noindent
For a Lipschitz domain $\dot{\Omega}$, let us introduce the functional spaces 
\begin{equation}
\label{eq.H-d}
\Lambda^p H_{d} ( \dot \Omega ):= \left\{u\in \Lambda^p L^2 ( \dot \Omega ),\ du\in\Lambda^{p+1} L^2 ( \dot \Omega )\right\}
\end{equation}
and
\begin{equation}
\label{eq.H-d*}
\Lambda^p H_{d^*} ( \dot \Omega )\ :=\  \left\{u\in \Lambda^p L^2 ( \dot \Omega ),\ d^*u\in\Lambda^{p-1} L^2 ( \dot \Omega )\right\}
\end{equation}
\label{page.Hd}
equipped with their natural graph norms.
% and
%$$
%\Lambda H_{d} ( \dot \Omega ):=\oplus_{p=0}^d \Lambda^p H_{d} ( \dot \Omega )\quad ,\quad \Lambda H_{d^*} ( \dot \Omega ):=\oplus_{p=0}^d \Lambda^p H_{d^*} ( \dot \Omega ).
%$$
One recalls that for a differential form $f$ in $L^2 (\pa\dot\Omega )$, the tangential and normal components are defined as follows:
\begin{equation}
\label{eq.decomp-bdy}
f= \mathbf t f+\mathbf n f\quad\text{with}\quad
\mathbf t f= \mathbf{i}_{n}(n^{\flat} \wedge f)\quad\text{and}\quad 
 \mathbf n f= n^{\flat}\wedge(\mathbf{i}_{n} f),
\end{equation}
where the superscript $\flat$ stands for the usual musical
isomorphism: $n^{\flat}$ is the 1-form associated with the outgoing unit normal vector $n$. Moreover,
$$
\|f\|^2_{ L^2 (\pa\dot\Omega )}= 
\|\mathbf t f\|^2_{ L^2 (\pa\dot\Omega )}+
\|\mathbf n f\|^2_{ L^2 (\pa\dot\Omega )}
= 
\|n^{\flat} \wedge f\|^2_{ L^2 (\pa\dot\Omega )}+
\|\mathbf{i}_{n} f\|^2_{ L^2 (\pa\dot\Omega )}.
$$ 
The  Green formula for differential forms $(u,v)\in \Lambda^{p
}H^1 ( \dot \Omega ) \times  \Lambda^{p
+1}H^1 ( \dot \Omega ) $ writes 
\begin{equation}
\label{eq:usual_Green}
\begin{aligned}
\langle du, v\rangle_{ L^2 ( \dot \Omega )}- \langle u, d^*v\rangle_{ L^2 ( \dot \Omega )}&= \int_{\partial \dot\Omega} \langle  n^{\flat} \wedge u,v\rangle_{ T^*_{\sigma}\dot\Omega} d\sigma= \int_{\partial \dot\Omega} \langle  n^{\flat} \wedge u, \mathbf n v\rangle_{ T^*_{\sigma}\dot\Omega} d\sigma\\
&=\int_{\partial \dot\Omega} \langle  u,\mathbf{i}_{n} v\rangle_{ T^*_{\sigma}\dot\Omega} d\sigma=\int_{\partial \dot\Omega} \langle \mathbf t  u,\mathbf{i}_{n} v\rangle_{ T^*_{\sigma}\dot\Omega} d\sigma,
\end{aligned}
\end{equation}
where we used the standard relation $(n^\flat\wedge)^*=\mathbf i_n$.

Using this Green formula, the tangential (resp. normal) traces can be
defined for forms in $\Lambda H_d(\dot\Omega)$  (resp. $\Lambda
H_{d^*}(\dot\Omega)$) by duality. Indeed, for any  $u \in \Lambda^p
H_{d} ( \dot \Omega )$, $n^{\flat} \wedge u\in \Lambda^{p+1} H^{-\frac12} (\pa\dot\Omega )$ is defined by
\begin{equation}
\label{eq.n-wedge-u}
\forall g\in\Lambda^{p+1}H^{\frac12} (\pa\dot\Omega ),\ 
\langle n^{\flat} \wedge u , g\rangle_{H^{-\frac12} (\pa\dot\Omega ),H^{\frac12} (\pa\dot\Omega )} = 
\langle du, G\rangle_{ L^2 ( \dot \Omega )}-\langle u, d^*G\rangle_{ L^2 ( \dot \Omega )},
\end{equation}
where $G$ is any form in $\Lambda^{p+1}H^{1} ( \dot \Omega )$ whose trace
in $\Lambda^{p+1}H^{\frac12} (\pa\dot\Omega )$ is $g$. This definition is independent of the  chosen extension $G$ (this indeed follows from the Green formula~\eqref{eq:usual_Green}  together with the density of $\Lambda^p\mathcal C^\infty\big (\, \overline{\dot\Omega}\, \big )$
 in   $\Lambda^p H_{d} ( \dot \Omega )$, see  for example \cite[Proposition~3.1]{jakab-mitrea-mitrea-09}).
Similarly, for any $u \in \Lambda^p H_{d^*} ( \dot \Omega )$, $\mathbf{i}_{n} u\in \Lambda^{p-1}H^{-\frac12} (\pa\dot\Omega )$ is defined by
\begin{equation}
\label{eq.i-n-u}
\forall g\in \Lambda^{p-1} H^{\frac12} (\pa\dot\Omega ),\ 
\langle \mathbf{i}_{n} u , g\rangle_{H^{-\frac12} (\pa\dot\Omega ),H^{\frac12} (\pa\dot\Omega )} =
\langle u, dG\rangle_{ L^2 ( \dot \Omega )}-\langle d^*u, G\rangle_{ L^2 ( \dot \Omega )},
\end{equation}
where $G$ is any extension of $g$ in $\Lambda^{p-1}H^{1} ( \dot \Omega
)$. 
%We would like now to exhibit two particular settings where the
%traces $n^{\flat} \wedge u$ and $\mathbf{i}_{n} u$ can actually be
%defined in $L^2(\partial \dot\Omega)$.
\medskip

\noindent
Let $\Gamma$ be any subset of $\partial \dot \Omega$. For $u\in \Lambda^p H_{d} ( \dot \Omega )$, we will
write $\mathbf t u |_{\Gamma}=0$ if $n^{\flat} \wedge
u|_{\Gamma}=0$. If  $u\in \Lambda^p H_{d} ( \dot \Omega )$ and
$n^{\flat} \wedge u|_{\Gamma}\in \Lambda^{p+1} L^2(\Gamma)$, the
tangential trace on $\Gamma$ is defined by
\begin{equation}
\label{eq.tan-gen}\mathbf t u|_{\Gamma}  :=  \mathbf{i}_{n}
(n^{\flat} \wedge u) \in \Lambda^{p} L^2(\Gamma)\ ,\text{ so that } \|\mathbf t u\|_{ L^2(\Gamma)} = \|n^{\flat} \wedge u\|_{ L^2(\Gamma)}.
\end{equation} 
Similarly, for $u\in \Lambda^p H_{d^*} ( \dot \Omega )$, we will write $\mathbf n u |_{\Gamma}=0$
if $\mathbf{i}_{n}u|_{\Gamma}=0$. If $u\in \Lambda^p H_{d^*} (
\dot \Omega )$  and $\mathbf{i}_{n} u|_{\Gamma}\in \Lambda^{p-1}
L^2(\Gamma)$, the normal trace  on $\Gamma$ is defined by
\begin{equation}
\label{eq.norm-gen}\mathbf n u|_{\Gamma}  :=  n^{\flat} \wedge
(\mathbf{i}_{n}u) \in \Lambda^{p} L^2(\Gamma)\ , \text{
 so that } \|\mathbf n u\|_{ L^2(\Gamma)}
=\|\mathbf{i}_{n}u\|_{ L^2(\Gamma)}.
\end{equation}
%satisfies all the  conditions leading to both \eqref{eq.tan-gen}
%and \eqref{eq.norm-gen}
Lastly, if $u\in \Lambda^pH_d(\dot \Omega)\cap \Lambda^pH_{d^*}(\dot \Omega)$ is such that $n^{\flat}\wedge u|_{\Gamma}\in \Lambda^{p+1}L^2(\Gamma)$ and $\mathbf{i}_{n}u\in \Lambda^{p-1}L^2(\Gamma)$ then $u$ admits a trace $u|_{\Gamma}$  in $L^2(\Gamma)$  defined by
\begin{equation}
\label{eq.trace-gen}
u|_{\Gamma} := \mathbf t u|_{\Gamma}+
\mathbf n u|_{\Gamma}.
\end{equation}
This definition is
 compatible
with \eqref{eq.decomp-bdy} and such a differential form satisfies
$$
\|u|_{\Gamma}\|^2_{ L^2(\Gamma)} = 
\| \mathbf t u|_{\Gamma}\|^2_{ L^2(\Gamma)}
+
\|\mathbf n u|_{\Gamma}\|^2_{ L^2(\Gamma)}
=
\|n^{\flat} \wedge u\|_{ L^2(\Gamma)}^2
+
\|\mathbf{i}_{n}u\|_{ L^2(\Gamma)}^2
.
$$ 
All the above definitions coincide moreover
with the usual ones when $u$ belongs to $\Lambda^{\ell} H^1 ( \dot \Omega )$.
\medskip

\noindent
Let us finally note for further references that if traces are in $L^2(\partial \dot\Omega)$, a direct
consequence of the   Green formula~\eqref{eq:usual_Green} is the following:
for every $u,v \in \Lambda L^{2} ( \dot \Omega ) $ such that
 $du, d^{*}u,d^{*}du, dd^{*}u, dv, d^{*}v\in  \Lambda L^{2} ( \dot \Omega )$
 and $n^{\flat} \wedge d^{*}_{f,h}u$, $\mathbf{i}_{n} d_{f,h}u$,
 $n^{\flat} \wedge v$, $\mathbf{i}_{n} v\in \Lambda L^{2}
 (\pa\dot\Omega )$,
\begin{equation}\label{eq.Green}
\begin{aligned}
&\langle (d_{f,h}d^{*}_{f,h}+d^{*}_{f,h}d_{f,h})u,v\rangle_{L^{2}(\dot
\Omega)}
=
 \langle d_{f,h}u,d_{f,h}v\rangle_{L^{2}(\dot
\Omega)}+
\langle d_{f,h}^{*}u,d_{f,h}^{*}v\rangle_{L^{2}(\dot
\Omega)}\\
&\quad +h \int_{\pa\dot\Omega}\langle
 n^{\flat} \wedge d^{*}_{f,h} u,n^{\flat} \wedge(\mathbf{i}_{n} v)
\rangle_{T^*_{\sigma}\dot\Omega}d\sigma
-h \int_{\pa\dot\Omega}\langle
 n^{\flat} \wedge v,n^{\flat} \wedge(\mathbf{i}_{n} d_{f,h}u)
\rangle_{ T^*_{\sigma}\dot\Omega}d\sigma
% \\ + h
%\int_{\pa\Omega}\!(\mathbf{t}\overline{v})\wedge
%(\star\mathbf{n}d_{f,h}u) 
%- h \int_{\pa\Omega}\!(\mathbf{t}d_{f,h}^{*}u)\wedge(\star \mathbf{n}
%\overline{v})
.
\end{aligned}
\end{equation}

\medskip
\noindent
\underline{The Gaffney's inequality}
\medskip

\noindent
The following extension of Gaffney's inequality (see~\cite{GSchw})
will be useful in the sequel (we refer to Section~\ref{notation} for the definitions of the Hilbert space $\Lambda^{p} H^{1}_{T}(\dot \Omega)$ and $\Lambda^{p} H^{1}_{N}(\dot \Omega)$). Notice that in the following result  we use  that $\dot \Omega$ is smooth (there are actually counterexamples for Lipschitz domains, see for example~\cite{mitrea2001dirichlet}). 
\begin{lemma}
\label{le.Gaffney}
Let $\dot \Omega$ be a smooth domain. The equality
$$\left\{u\in \Lambda^{p} L^2(\dot \Omega)\text{ s.t. }  du, d^*u\in L^2(\dot \Omega)
\text{ and } \mathbf t u=0 \text{ on } \pa \dot \Omega
\right\}  = \Lambda^{p} H^{1}_{T}(\dot \Omega)$$
holds algebraically and topologically, the functional space in the
left-hand side being equipped with the norm
associated with the scalar product
$$
 Q(u,v) :=\langle u, v\rangle_{ L^2(\dot \Omega)}
+\langle du,dv \rangle_{ L^{2}(\dot \Omega)}+\langle d^*u, d^{*}v\rangle_{ L^{2}(\dot \Omega)}.
 $$
In a similar way, the following equality holds algebraically and topologically:
 $$\left\{u\in \Lambda^{p} L^2(\dot \Omega)\text{ s.t. }  du, d^*u\in L^2(\dot \Omega)
\text{ and } \mathbf n u=0 \text{ on } \pa \dot \Omega
\right\} = \Lambda^{p} H^{1}_{N}(\dot \Omega).$$
\end{lemma}

\noindent
Notice that in the defintion of the functional spaces above, the
equality $ \mathbf t u=0$ and $ \mathbf n u=0$ hold in the weak sense
defined above (see~\eqref{eq.tan-gen} and~\eqref{eq.norm-gen}).
A direct consequence of this lemma is that a differential form in $\Lambda H_d(\dot\Omega) \cap \Lambda H_{d^*}(\dot\Omega)$ such
that $\mathbf t  u=0$ or $\mathbf n u=0$ on $\partial \dot \Omega$ admits a trace in $\Lambda L^2(\partial \dot \Omega)$.

\begin{remark}
\label{re.Gaffney}
The statement of Gaffney's inequality in \cite{GSchw} reads as follows
(see indeed Corollary~2.1.6 and Theorem~2.1.7 there):
\begin{equation}
\label{eq.Gaffney}
\exists  C>0,\ \forall  u\in\Lambda^{p}
H^{1}_{T}(\dot \Omega)\cup\Lambda^{p} H^{1}_{N}(\dot \Omega), \   \|u\|^{2}_{  H^{1}(\dot \Omega)}  \leq   C Q(u,u).
\end{equation}
Since it also holds that, for some $C'>0$ and any $u\in \Lambda^{p} H^{1}(\dot \Omega)$,  $Q(u,u)\leq C' \|u\|^{2}_{  H^{1}(\dot \Omega)}$, the scalar products $\langle \cdot,\cdot \rangle_{H^{1}}$ 
and $Q(\cdot,\cdot)$
are then equivalent  on both  $\Lambda^{p} H^{1}_{T}(\dot \Omega)$
and $\Lambda^{p} H^{1}_{N}(\dot \Omega)$. The above lemma can be seen
as a generalization of this result to the spaces $ \{u\in \Lambda^{p} L^2(\dot \Omega)\text{ s.t. }  du, d^*u\in L^2(\dot \Omega)
\text{ and } \mathbf t u=0 \text{ on } \pa \dot \Omega
 \} $ and $ \{u\in \Lambda^{p} L^2(\dot \Omega)\text{ s.t. }  du, d^*u\in L^2(\dot \Omega)
\text{ and } \mathbf n u=0 \text{ on } \pa \dot \Omega
 \}$.
% \\
% The above lemma states in addition that any $u\in \Lambda^{p} L^2(\dot \Omega)$
% satisfying  $du, d^*u\in L^2(\dot \Omega)$
% and $\mathbf t u|_{\pa \dot \Omega}=0$ (resp. $\mathbf n u|_{\pa \dot \Omega}=0$), where $\mathbf t$ (resp. $\mathbf n$) is defined in the above weaker sense
% \eqref{eq.tan-gen} (resp. \eqref{eq.norm-gen}),
% actually belongs to $\Lambda^{p} H^{1}_{T}(\dot \Omega)$ (resp. $\Lambda^{p} H^{1}_{N}(\dot \Omega)$). 
\end{remark}
\begin{proof}
We only prove the first equality in Lemma~\ref{le.Gaffney}, the second
one being similar. Let us define
$$
H := \left\{u\in \Lambda^{p} L^2 ( \dot \Omega )\text{ s.t. }  du, d^*u\in \Lambda L^{2} (  \dot \Omega )
\quad \text{and}\quad n^\flat \wedge  u=0
\text{ on } \pa \dot \Omega
\right\}
 $$
 which is a Hilbert space once equipped with the 
 scalar product $Q$. From
 Gaffney's inequality~\eqref{eq.Gaffney}, 
  $\Lambda^{p} H^{1}_{T}(\dot \Omega)$ is a closed
subset of $H$ and to conclude, we just have to show that $\big(\Lambda^{p} H^{1}_{T}(\dot \Omega)\big)^{\perp}=\{0\}$,
the orthogonal complement of $H$ being taken with respect to the norm inherited from $Q$.
Consider then $u\in H$ such that for any $v\in \Lambda^{p}H^{1}_{T}(\dot \Omega)$,
$$
0 =Q(u,v) =
\langle u, v\rangle_{ L^2(\dot \Omega)}
+\langle du,dv \rangle_{ L^{2}(\dot \Omega)}+\langle d^*u, d^{*}v\rangle_{ L^{2}(\dot \Omega)}.
$$
The above equality holds in particular for every $v\in D$ where
\begin{equation}
\label{eq.domD}
D=\{v \in\Lambda^{p}H^{2}(\dot \Omega) ,\ \mathbf t v|_{\pa\dot \Omega}=\mathbf t d^{*} v|_{\pa\dot \Omega}=0\}.
\end{equation}
 Fix such a $v$.
Since $n^\flat \wedge u=0$ on $\pa \dot \Omega$,
applying \eqref{eq.n-wedge-u} to $u$ and $d v\in \Lambda^{p+1}H^{1}(\dot \Omega)$ 
then leads to 
$$
\langle du,dv \rangle_{ L^{2}(\dot \Omega)} = \langle u,d^{*}dv \rangle_{ L^{2}(\dot \Omega)}.
$$
Applying also \eqref{eq.i-n-u} to $u$ and $d^{*}v\in \Lambda^{p-1}H^{1}(\dot \Omega)$
gives 
$$
\langle d^{*}u,d^{*}v \rangle_{ L^{2}(\dot \Omega)}= \langle u,dd^{*}v \rangle_{ L^{2}(\dot \Omega)}
-\langle \mathbf{i}_{n} u , d^{*}v|_{\pa\dot \Omega}\rangle_{H^{-\frac12} (\pa\dot\Omega ),H^{\frac12} (\pa\dot\Omega )}.
$$
Since $\Lambda^p\mathcal C^\infty\left (\overline{\dot\Omega}\right )$
is densely embedded in both $\Lambda^p H_{d} ( \dot \Omega )$ and $
\Lambda^p H_{d^*} ( \dot \Omega )$ (see for example \cite[Proposition~3.1]{jakab-mitrea-mitrea-09}),
we have moreover  for some sequence $(u_{k})_{k\in\mathbb N}$ of $\Lambda^p\mathcal C^{\infty}\Big(\overline{\dot \Omega}\Big)$ forms:
\begin{align*}
\langle \mathbf{i}_{n} u,d^{*}v|_{\pa\dot \Omega}\rangle_{H^{-\frac12} (\pa\dot\Omega ),H^{\frac12} (\pa\dot\Omega )}
 &= \lim\limits_{k\to+\infty}\int_{\pa\dot \Omega}\langle
\mathbf{i}_{n} u_{k},d^{*}v
\rangle_{T^*_{\sigma} \dot \Omega}d\sigma\\
 &= \lim\limits_{k\to+\infty}\int_{\pa\dot \Omega}\langle
 \mathbf{i}_{n} u_{k},n^{\flat}\wedge(\mathbf{i}_{n} d^{*}v)
\rangle_{T^*_{\sigma} \dot \Omega}d\sigma= 0,
\end{align*}
where the second equality is a consequence of $\mathbf t d^{*} v|_{\pa\dot \Omega}=0$.
It consequently follows 
\begin{equation}
\label{eq.ortho}
0=  
\langle u, v\rangle_{ L^2(\dot \Omega)}
+\langle u,d^{*}dv \rangle_{ L^{2}(\dot \Omega)}+\langle u, dd^{*}v\rangle_{ L^{2}(\dot \Omega)}
=
\langle u, (I+\Delta_{H}^{(p)})v\rangle_{ L^2(\dot \Omega)}
,
\end{equation}
where $\Delta_{H}^{(p)}$ denotes the Hodge Laplacian on $\dot \Omega$ with domain $D$ defined by~\eqref{eq.domD}. 
Since the unbounded operator $(\Delta_{H}^{(p)}, D)$ is  selfadjoint and nonnegative
on $\Lambda^{p} L^2(\dot \Omega)$, we have in particular  $\Ran (I+\Delta_{H}^{(p)})= \Lambda^{p} L^2(\dot \Omega) $
and we deduce from \eqref{eq.ortho} that $u=0$, which completes the proof. 
\end{proof}

\medskip
\noindent
\underline{The case of mixed normal-tangential Dirichlet boundary conditions}
\medskip

\noindent
Let $\Gamma_T$ and $\Gamma_N$ be two disjoint open subsets of $\partial \dot
\Omega$ such that $\overline{\Gamma_T} \cup \overline{\Gamma_N} = \pa \dot \Omega$\label{page.gammatn}. The objective of this section is to consider differential forms such
that $\mathbf t  u=0$ on $\Gamma_T$ and $\mathbf n u=0$ on $\Gamma_N$, and to state results on the existence of a
trace in $L^2(\partial \dot \Omega)$ for such differential forms, as well
as subelliptic estimates. 

\medskip

\noindent
In general,  a trace in $L^2(\partial \dot \Omega)$ does not exist in such a setting~\cite{brown-94,jakab-mitrea-mitrea-09}: one needs a geometric
assumption, namley that $\Gamma_T$ and $\Gamma_N$ meet at an angle strictly
smaller than~$\pi$. This means that the angle between $\Gamma_T$ and $\Gamma_N$
measured in $\dot\Omega$ is smaller than $\pi$. More precisely,
see~\cite{brown-94,jakab-mitrea-mitrea-09}, locally around any point
$x_0 \in \overline{\Gamma_T} \cap \overline{\Gamma_N}$, one requires
that there exists a local system of coordinates $(x_1,x'',x_n) \in \R
\times \R^{d-2} \times \R$ on a
neighborhood $V_0$ of $x_0$, and two Lipschitz functions
$\tilde \varphi:\R^{n-1} \to \R$ and $\tilde \psi:\R^{n-2} \to \R$ such
that $\dot \Omega \cap V_0=\{x_n > \tilde \varphi(x_1,x'')\}$, $\Gamma_T \cap
V_0=\{x_n=\tilde \varphi(x_1,x'') \text{ and } x_1 > \tilde \psi(x'')\}$ and $\Gamma_N \cap
V_0=\{x_n=\tilde \varphi(x_1,x'') \text{ and } x_1 < \tilde  \psi(x'')\}$ and
\begin{equation}\label{eq:angle}
\begin{aligned}
\partial_{x_1} \tilde \varphi (x_1,x'') \ge \kappa &\text{ on }
x_1>\tilde \psi(x'')\\
\partial_{x_1}\tilde \varphi (x_1,x'')  \le -\kappa &\text{ on }
x_1<\tilde \psi(x'')
\end{aligned}
\end{equation}
for some positive $\kappa$. This is equivalent to the existence of a
smooth vector field $\theta$ on $\partial \dot \Omega$ such that
$\langle \theta , n \rangle < 0$ on $\Gamma_T$ and $\langle \theta,
n\rangle > 0$ on
$\Gamma_N$, which is one of the key ingredient of the proofs used in~\cite{brown-94,jakab-mitrea-mitrea-09}.
\medskip

\noindent
Let $\Gamma$ be any open Lipschitz subset of $\partial \dot \Omega$\label{page.gammam}. According to \cite[Proposition~3.1]{jakab-mitrea-mitrea-09},
the space
$$
\left\{u\in\Lambda^p\mathcal C^\infty\left (\overline{\dot\Omega}\right ),\ 
u\equiv 0 \text{ in a neighborhood of } \pa\dot\Omega\setminus \Gamma\right\}
$$
is densely embedded in both
$$
\Lambda^p H_{d,\Gamma} ( \dot \Omega ) := \left\{u\in\Lambda^p H_{d} ( \dot \Omega ),\ 
\supp(n^{\flat} \wedge u)\subset \overline\Gamma\right\}
$$
and
$$
\Lambda^p H_{d^*,\Gamma} ( \dot \Omega ) := \left\{u\in\Lambda^p H_{d^*} ( \dot \Omega ),\ 
\supp(\mathbf{i}_{n}  u)\subset \overline\Gamma\right\}.
$$
In \label{page.Hdgamma} addition, according to \cite[Theorem~3.4]{jakab-mitrea-mitrea-09}, for
 $(u,v)\in\Lambda^p H_{d} ( \dot \Omega )\times\Lambda^{p+1} H_{d^*} ( \dot \Omega )$
satisfying 
the trace conditions 
$$\mathbf{i}_{n} v \in\Lambda^{p} L^2 (\pa\dot\Omega ),\   \supp \mathbf{i}_{n} v\subset \Gamma
\quad\text{and}\quad n^{\flat} \wedge u \in \Lambda^{p+1} L^2(\Gamma),$$
or
$$n^{\flat} \wedge u\in \Lambda^{p+1} L^2 (\pa\dot\Omega ),\ 
\supp(n^{\flat} \wedge u)\subset\Gamma\quad\text{and}\quad 
\mathbf{i}_{n} v \in\Lambda^{p} L^2(\Gamma), $$
one has the following Green formula (compare
with~\eqref{eq:usual_Green}):
\begin{equation}
\label{eq.gen-Green}
\begin{aligned}
\langle du, v\rangle_{ L^2 ( \dot \Omega )}-\langle u, d^*v\rangle_{
  L^2 ( \dot \Omega )}  &=  \int_{\Gamma}\langle
 n^{\flat} \wedge u,n^{\flat} \wedge(\mathbf{i}_{n} v)
\rangle_{T^*_{\sigma}\dot\Omega}d\sigma\\
&  = \int_{\Gamma}\langle
\mathbf{i}_{n} (n^{\flat} \wedge u),\mathbf{i}_{n} v
\rangle_{T^*_{\sigma}\dot\Omega} d \sigma
.
\end{aligned}
\end{equation}
One is now ready to state the following proposition implied
by Theorems~1.1 and 1.2 of \cite{jakab-mitrea-mitrea-09} (see also Theorems~4.1
and~4.2  of \cite{goldshtein-mitrea-mitrea-11}).
\begin{proposition}
\label{pr.QTN}
Let us assume that $\dot \Omega$ is a Lipschitz domain. Let $\Gamma_T$
and $\Gamma_N$ be two disjoint Lipschitz open subsets of $\partial \dot
\Omega$ such that $\overline{\Gamma_T} \cup \overline{\Gamma_N} = \pa
\dot \Omega$ and such that $\Gamma_T$ and $\Gamma_N$ meet at an angle
strictly smaller than $\pi$. Then, the following results hold:

\noindent
{(i)} Let $u$ be a differential form such that
 $$
u\in \Lambda^{p} L^{2} ( \dot \Omega ), \, du \in  L^{2} ( \dot \Omega ), d^*u\in  L^{2} ( \dot \Omega ),\, 
\mathbf t u|_{\Gamma_T}=0 \text{ and } 
\mathbf n u|_{\Gamma_N}=0.$$
Then $u$ satisfies 
$$u\in\Lambda^{p} H^{\frac12} ( \dot \Omega )\quad\text{and}\quad
\mathbf{i}_{n}u,\ n^{\flat} \wedge u \in
\Lambda^{p}L^2 (\pa\dot\Omega )
$$
as well as the subelliptic estimate:
\begin{equation}\label{eq:subelliptic}
\|u\|_{ H^{\frac12} ( \dot \Omega )}+
\|u|_{\pa\dot\Omega}\|_{ L^{2} (\pa\dot\Omega )}\
\leq\ C\left(\|u\|_{ L^{2} ( \dot \Omega )}
+\|du\|_{ L^{2} ( \dot \Omega )}+\|d^*u\|_{ L^{2} ( \dot \Omega )} \right),
\end{equation}
where $u|_{\pa\dot\Omega}$ is defined by \eqref{eq.trace-gen}.

\medskip
\noindent
(ii) The unbounded operators $d_{T}^{(p)} ( \dot \Omega )$ and $\delta_{N}^{(p)} ( \dot \Omega )$  on $\Lambda^{p} L^{2} ( \dot \Omega )$
defined by
$$
d_{T}^{(p)} ( \dot \Omega ) = d_{f,h}^{(p)}$$
with domain 
$$D\left (d_{T}^{(p)} ( \dot \Omega )\right)=\left\{u\in \Lambda^{p}L^2 ( \dot \Omega ),\ d_{f,h}u \in \Lambda^{p+1}L^2(\dot \Omega),\ \mathbf{t}u|_{\Gamma_T}=0 \right\},
$$
and
$$
\delta_{N}^{(p)} ( \dot \Omega ) = \left ( d_{f,h}^{(p)}\right)^*$$
with domain 
$$ D\left(\delta_{N}^{(p)} ( \dot \Omega )\right)=\left\{u\in \Lambda^{p} L^2 ( \dot \Omega ),\ d^*_{f,h}u\in \Lambda^{p-1}L^2( \dot \Omega ),\ \mathbf{n}u|_{\Gamma_N}=0 \right\} ,
$$
are closed, densely defined, and adjoint one of each other\label{page.dtp}.
\end{proposition}

\noindent
Note that in the point~{\em (i)} of Proposition~\ref{pr.QTN}, $d$ and $d^*$
can be replaced by $d_{f,h}$ and $d^*_{f,h}$ owing to the relations
$d_{f,h}=hd+df\wedge$ and $d^*_{f,h}=hd^*+\mathbf i_{\nabla f}$.
Moreover, the point~{\em (ii}) is actually proven in \cite{jakab-mitrea-mitrea-09,goldshtein-mitrea-mitrea-11}
for $d$ and $d^*$ but remains true for $d_{f,h}$ and $d^*_{f,h}$
since $(df\wedge)^*=\mathbf i_{\nabla f}$ on $L^2 ( \dot \Omega )$.

\medskip
\noindent
\underline{The mixed Witten Laplacian $\Delta^M_{f,h}(\dot \Omega)$}
\medskip

\noindent
We are now in position to define the mixed Witten Laplacian $\Delta^M_{f,h}(\dot \Omega)$ (the upperscript $M$ stands for mixed boundary conditions) with
tangential Dirichlet boundary conditions on $\Gamma_T$ and normal
Dirichlet boundary conditions on $\Gamma_N$
(see~\cite{jakab-mitrea-mitrea-09,goldshtein-mitrea-mitrea-11} for more
results on such operators). The operator $\Delta_{f,h}^{M,(p)} ( \dot \Omega )$  on $L^2 (
 \dot \Omega )$ is defined by
\begin{equation}
\label{eq.DeltaTN}
\Delta_{f,h}^{M,(p)} ( \dot \Omega ) := d_{T}^{(p-1)} ( \dot \Omega )\circ \delta_{N}^{(p)} ( \dot \Omega )+\delta_{N}^{(p+1)} ( \dot \Omega )\circ d_{T}^{(p)} ( \dot \Omega ),
\end{equation}
in \label{page.deltam} the sense of composition of unbounded operators, where $d_T$ and
$\delta_N$ have been introduced in Proposition~\ref{pr.QTN}. Notice that 
for any $u\in\Lambda^{p} H_{d} ( \dot \Omega )$ such that $\mathbf t u|_{\Gamma_{T}}=0$,   
one has  $du\in \Lambda^{p+1}H_{d}( \dot \Omega )$  and 
$\mathbf t du|_{\Gamma_{T}}=0$. The latter is easy to check when  $u\in \Lambda^{p} H^{2} ( \dot \Omega )$
and can be proved here using~\eqref{eq.n-wedge-u} together with the density of $
\big\{u\in\Lambda^p\mathcal C^\infty\big (\overline{\dot\Omega}\big ),\ 
u\equiv 0 \text{ in a neighborhood of }  \pa \dot \Omega \setminus \overline{\Gamma_T}\big\}$ into  $  \Lambda^p H_{d,\pa \dot \Omega \setminus \overline{\Gamma_T}} ( \dot \Omega ) $. Likewise, one has 
 $d_{f,h}d_{f,h}=0$ in the distributional sense and 
$\mathbf t d_{f,h}u|_{\Gamma_{T}}=0$
for $u\in\Lambda^{p} H_{d} ( \dot \Omega )$
such that $\mathbf t u|_{\Gamma_{T}}=0$.
This implies in particular 
\begin{equation}\label{eq.dom-dT}
\left\{
\begin{aligned}
&\overline{ \Im  d_{T}}\subset \Ker d_{T} \quad\text{and}\quad d_{T}^{2}=0,\\
&\overline{ \Im  \delta_{N}}\subset \Ker \delta_{N} \quad\text{and}\quad
\delta_{N}^{2}=0.
\end{aligned}
\right.
\end{equation} 
 Owing to this last relation and
to Proposition~\ref{pr.QTN}, a result due to Gaffney 
(see e.g. the proof of \cite[Propositions 2.3 and 2.4]{goldshtein-mitrea-mitrea-11})
states that
$\Delta_{f,h}^{M,(p)} ( \dot \Omega )$
is a densely defined nonnegative selfadjoint  operator on~$ L^{2} (
\dot \Omega )$ (with domain defined below in~\eqref{eq.domDeltaTN}). 
\medskip

\noindent
The domain $D\left( \mathcal Q_{f,h}^{M,(p)} ( \dot \Omega
    ) \right)$ of the closed quadratic form $\mathcal Q_{f,h}^{M,(p)} ( \dot \Omega )$
associated with  $\Delta_{f,h}^{M,(p)} ( \dot \Omega )$
is given by
\begin{align*}
 D\left ( \mathcal
  Q^{M,(p)}_{f,h} ( \dot \Omega )\right )&= D\left (d_{T}^{(p)} ( \dot \Omega )\right )\cap
D\left (\delta_{N}^{(p)} ( \dot \Omega )\right)\\
&=\left\{v\in \Lambda^{p} L^{2} ( \dot \Omega ), \, dv \in  L^{2} ( \dot \Omega ), d^*v\in  L^{2} ( \dot \Omega ),\, 
\mathbf t v|_{\Gamma_T}=0 \text{ and } 
\mathbf n v|_{\Gamma_N}=0\right\}
\end{align*}
\label{page.qm}
and for any $u,v\in D\left( \mathcal Q^{M,(p)}_{f,h} ( \dot \Omega )\right)$, 
$$\mathcal Q_{f,h}^{M,(p)} ( \dot \Omega )(u,v)=\langle d_{T} u,d_{T}v \rangle_{L^2} +\langle \delta_{N}u,\delta_{N}v \rangle_{L^2}.$$
This is proven in \cite[Theorem 2.8]{goldshtein-mitrea-mitrea-11}.
\medskip

\noindent
The domain $D\left (\Delta_{f,h}^{M,(p)} ( \dot \Omega )\right)$ is
explicitly given by: 
\begin{equation}
\label{eq.domDeltaTN}
\begin{aligned}
D\left (\Delta_{f,h}^{M,(p)} ( \dot \Omega )\right) =\Big\{&u\in L^{2}
( \dot \Omega )\text{ s.t. }
d_{f,h} u,\  d^*_{f,h}u,\  d^*_{f,h}d_{f,h} u,\   d_{f,h}d^*_{f,h} u
\in  L^{2} ( \dot \Omega ), \\
& \mathbf{t}u|_{\Gamma_T}=0,\ 
\mathbf{t}d^*_{f,h}u|_{\Gamma_T}=0,\ 
\mathbf{n}u|_{\Gamma_{N}}=0,\ 
\mathbf{n}d_{f,h}u|_{\Gamma_{N}}=0
\Big\}.
\end{aligned}
\end{equation}
The traces $\mathbf{t}d^*_{f,h}u$ and $\mathbf{n}d_{f,h} u$ are a
priori defined in  $ H^{-\frac12} ( \pa \dot \Omega )$ but actually belong
to~$ L^{2} (\pa  \dot \Omega )$. Indeed,
we have $\mathbf{n}d_{f,h}u|_{\Gamma_{N}}=0$ by definition of $D\left (\Delta_{f,h}^{M,(p)} ( \dot \Omega )\right )$ and $\mathbf t d_{f,h}u|_{\Gamma_{T}}=0$ by~\eqref{eq.dom-dT},
so  $d_{f,h}u$ is in $D\left ( \mathcal Q^{M,(p+1)}_{f,h} ( \dot
  \Omega )\right )$ and therefore has
a trace in $ L^{2} ( \dot \Omega )$ according to Proposition~\ref{pr.QTN}. This argument also holds for $d^{*}_{f,h}u\in D\left ( \mathcal Q^{M,(p-1)}_{f,h} ( \dot \Omega )\right )$. 
\medskip

\noindent
We end up this section with the following lemma which will be
frequently used in the sequel.
\begin{lemma}
\label{le.GreenWeak}
Let us assume that the assumptions of Proposition~\ref{pr.QTN} are satisfied. 
Let us moreover  assume that $\Gamma_T$ is $C^\infty$ and that there exist two disjoint $C^\infty$ open subsets $\Gamma_{N,1}$ and $\Gamma_{N,2}$  of $\pa \dot \Omega$ such that 
\begin{equation}
\label{eq.decGammaN}  
\overline{\Gamma_N}=\overline{\Gamma_{N,1}}\cup \overline{\Gamma_{N,2}}.
\end{equation}
Then, the following formula holds: for any $u\in D\left ( \mathcal Q^{M,(p)}_{f,h} ( \dot \Omega )\right)$,
\begin{align}
\mathcal Q^{M,(p)}_{f,h} ( \dot \Omega )(u,u)&=\left\| d_{f,h}u\right\|^{2}_{L^{2} ( \dot \Omega )}+
\left\|d_{f,h}^{*}u\right\|^{2}_{L^{2} ( \dot \Omega )}\nonumber\\
&=
\nonumber
 h^{2}\left\| du\right\|^{2}_{L^{2} ( \dot \Omega )}+
h^{2}\left\| d^{*}u\right\|^{2}_{L^{2} ( \dot \Omega )}\\
\nonumber
&\quad+ \left\|\left|\nabla f\right|u\right\|^{2}_{L^{2} ( \dot \Omega )}
+h\langle(\mathcal{L}_{\nabla f}+\mathcal{L}_{\nabla
  f}^{*})u,u\rangle_{L^{2} ( \dot \Omega )}
\label{eq.compfor}\\
&\quad-h\left(\int_{\Gamma_T}-\int_{\Gamma_{N}}\right) \langle u,u
\rangle_{T^{*}_{\sigma}\dot\Omega} \pa_{n}
f \,  d\sigma,   
\end{align}
where $\mathcal{L}$ stands for the Lie derivative. 
\end{lemma}
 \noindent
Notice that the boundary integral terms are well defined since
$u|_{\partial \dot \Omega} \in L^2(\dot \Omega)$ thanks to point {\em
  (i)} in Proposition~\ref{pr.QTN}.
From the proof of Lemma~\ref{le.GreenWeak}, it will be clear that~\eqref{eq.compfor} actually holds if  $\Gamma_T$ and $\Gamma_N$ are  only piecewise smooth. In the following, we will only need the result for $\Gamma_T$  smooth and $\Gamma_N$ the union of two smooth pieces and this is why we present this result in this setting. 
\begin{proof}
For $u\in  D\left ( \mathcal Q^{M,(p)}_{f,h} ( \dot \Omega )\right)$,
one first gets by straightforward computations,
\begin{align*}
\left\| d_{f,h}u\right\|^{2}_{L^{2} ( \dot \Omega )}+
\left\|d_{f,h}^{*}u\right\|^{2}_{L^{2} ( \dot \Omega )}&=
h^{2}\left\| du\right\|^{2}_{L^{2} ( \dot \Omega )}+
h^{2}\left\| d^{*}u\right\|^{2}_{L^{2} ( \dot \Omega )}
+\left\| d f\wedge u\right\|^{2}_{L^{2} ( \dot \Omega )}
\\
&\quad+ \left\| \mathbf i_{\nabla f} u\right\|^{2}_{L^{2} ( \dot \Omega )} + \  h\langle d f\wedge u,
du \rangle_{L^{2} ( \dot \Omega )} + h \langle d u,
d f\wedge u \rangle_{L^{2} ( \dot \Omega )} \\
&\quad +h\langle d^{*} u,
\mathbf i_{\nabla f} u \rangle_{L^{2} ( \dot \Omega )} + h \langle \mathbf i_{\nabla f} u,
d^{*}u \rangle_{L^{2} ( \dot \Omega )}\\
&=
h^{2}\left\| du\right\|^{2}_{L^{2} ( \dot \Omega )}+
h^{2}\left\| d^{*}u\right\|^{2}_{L^{2} ( \dot \Omega )}+
\left\|\left|\nabla f\right|u\right\|^{2}_{L^{2} ( \dot \Omega )}\\
 &\quad+h\langle (\mathcal L_{\nabla f}+  \mathcal L^{*}_{\nabla f}) u, u \rangle_{L^{2} ( \dot \Omega )}
 +   h\langle d f\wedge u,
du \rangle_{L^{2} ( \dot \Omega )} \\
&\quad -h
\langle d^{*}(d f\wedge u),
u \rangle_{L^{2} ( \dot \Omega )}
-h\langle d \mathbf i_{\nabla f} u , u \rangle_{L^{2} ( \dot \Omega )}
 + h  \langle \mathbf i_{\nabla f} u,
d^{*}u \rangle_{L^{2} ( \dot \Omega )} ,
\end{align*}
where the last equality holds thanks to the  relations
\begin{equation}\label{eq.relationII}
(df\wedge)^*=\mathbf i_{\nabla f}, \ \mathcal L_{\nabla f}=d\circ \mathbf i_{\nabla f}+\mathbf i_{\nabla f}\circ d\quad\text{and}\quad
\mathbf{i}_{\nabla f}(df\wedge u)
+
df\wedge\left(\mathbf{i}_{\nabla f}u\right)
= 
\left|\nabla f\right|^{2}u.
\end{equation}
To get the boundary integral terms in~\eqref{eq.compfor}
one uses \eqref{eq.gen-Green}, which gives here, since $u\in  D\left ( \mathcal Q^{M,(p)}_{f,h} ( \dot \Omega )\right)$
and  $df\wedge u$, $\mathbf i_{\nabla f} u\in \Lambda H_{d} ( \dot \Omega )\cap \Lambda H_{d^{*}} ( \dot \Omega )$:
\begin{align}
\nonumber
&\langle d f\wedge u,
du \rangle_{L^{2} ( \dot \Omega )} -
\langle d^{*}(d f\wedge u),
u \rangle_{L^{2} ( \dot \Omega )} = \!\!\int_{\Gamma_{N}}\!\!\langle {n}^{\flat}\wedge u ,  {n}^{\flat}\wedge \mathbf i_{n} (df\wedge u)\rangle_{ T^*_{\sigma} \dot\Omega}d\sigma\\
\label{eq.trace-norm}
&= \int_{\Gamma_{N,1}}\!\!\langle {n}^{\flat}\wedge u ,  {n}^{\flat}\wedge \mathbf i_{n} (df\wedge u)\rangle_{ T^*_{\sigma} \dot\Omega}d\sigma   +\int_{\Gamma_{N,2}}\!\!\langle {n}^{\flat}\wedge u ,  {n}^{\flat}\wedge \mathbf i_{n} (df\wedge u)\rangle_{ T^*_{\sigma} \dot\Omega}d\sigma,
\end{align}
where the last equality follows from~\eqref{eq.decGammaN}. Likewise, one has:
\begin{equation}
\label{eq.trace-tang}  
 \langle \mathbf i_{\nabla f} u,
d^{*}u \rangle_{L^{2} ( \dot \Omega )}
-\langle d \mathbf i_{\nabla f} u , u \rangle_{L^{2} ( \dot \Omega )}
 = 
-\!\!\int_{\Gamma_{T}}\!\!\!\!\langle {n}^{\flat}\wedge \mathbf i_{\nabla f} u ,  {n}^{\flat}\wedge \mathbf i_{n} u \rangle_{ T^*_{\sigma}\dot\Omega}d\sigma.
\end{equation}
Since $u\in  D\big ( \mathcal Q^{M,(p)}_{f,h} ( \dot \Omega )\big)$, and $\Gamma_T$, $\Gamma_{N,1}$, and $\Gamma_{N,2}$ are smooth open subsets of $\pa \dot \Omega$,  Lemma~\ref{le.Gaffney} implies  that  $u$ is in $\Lambda ^{p} H^{1}$   outside $(\overline{\Gamma_{T}}\cap \overline{\Gamma_{N}}) \cup (\overline{\Gamma_{N,1}}\cap \overline{\Gamma_{N,2}})$, by  a localization argument. Therefore, $u$   admits a boundary trace defined a.e. on $\pa\dot\Omega$ and belonging to
$ L^{2}_{\text{loc}}(\pa\dot\Omega\setminus(\overline{\Gamma_{T}}\cap \overline{\Gamma_{N}}) \cup (\overline{\Gamma_{N,1}}\cap \overline{\Gamma_{N,2}}) )$.
But this trace has to be $u|_{\pa\dot\Omega}$ as defined by \eqref{eq.trace-gen}
and is hence in $ L^{2} (\pa\dot\Omega )$ owing to   item $(i)$ in 
  Proposition~\ref{pr.QTN}. Let us now conclude the proof of Lemma~\ref{le.GreenWeak}. Let us  consider~\eqref{eq.trace-norm}. For $j\in \{1,2\}$ and $\ve >0$, one defines  
$\Gamma_{N,j}^{\varepsilon}:= \{x\in\Gamma_{N,j}, \
d^{\pa\dot\Omega}(x,\pa\Gamma_{N,j})>\varepsilon\}$. One then has for $j\in \{1,2\}$: 
\begin{align*}
\int_{\Gamma_{N,j}}\langle {n}^{\flat}\wedge u ,  {n}^{\flat}\wedge \mathbf i_{n} (df\wedge u)\rangle_{ T^*_{\sigma}\dot\Omega}d\sigma
&= \lim_{\varepsilon\to 0^{+}}\int_{\Gamma_{N,j}^{\varepsilon}}\langle {n}^{\flat}\wedge u ,  {n}^{\flat}\wedge \mathbf i_{n} (df\wedge u)\rangle_{ T^{*}_{\sigma}\dot\Omega}d\sigma\\\
 &=\lim_{\varepsilon\to 0^{+}}\int_{\Gamma_{N,j}^{\varepsilon}}\langle  u ,\mathbf i_{n}  ({n}^{\flat}\wedge \mathbf i_{n} (df\wedge u))\rangle_{ T^{*}_{\sigma}\dot\Omega}d\sigma\\\
  &=\lim_{\varepsilon\to 0^{+}}\int_{\Gamma_{N,j}^{\varepsilon}}\langle  u , \mathbf i_{n} (df\wedge u)\rangle_{ T^{*}_{\sigma}\dot\Omega}d\sigma\\
  &=\lim_{\varepsilon\to 0^{+}}\int_{\Gamma_{N,j}^{\varepsilon}}(\pa_{n}f\, \langle  u , u\rangle_{ T^{*}_{\sigma}\dot\Omega}- \langle  u , df\wedge \mathbf i_{n}u\rangle _{ T^{*}_{\sigma}\dot\Omega}   )d\sigma\\
  &=\int_{\Gamma_{N,j}}\pa_{n}f\, \langle  u , u\rangle_{ T^{*}_{\sigma}\dot\Omega}d\sigma,
\end{align*}
where we used the usual trace properties for $H^{1}$ forms on $\Gamma_{N,j}^{\varepsilon}$,
the fact that $\mathbf i_{n}u =0$ at the second to last line 
and the Lebesgue dominated convergence theorem at the last line. From~\eqref{eq.trace-norm}, one then has:
$$ \int_{\Gamma_{N}} \langle {n}^{\flat}\wedge u ,  {n}^{\flat}\wedge \mathbf i_{n} (df\wedge u)\rangle_{ T^*_{\sigma} \dot\Omega}d\sigma= \int_{\Gamma_{N}}\pa_{n}f\, \langle  u , u\rangle_{ T^{*}_{\sigma}\dot\Omega}d\sigma.$$
The   fact that $\int_{\Gamma_{T}}\langle {n}^{\flat}\wedge \mathbf i_{\nabla f} u ,  {n}^{\flat}\wedge \mathbf i_{n} u \rangle_{ T^*_{\sigma}\dot\Omega}d\sigma=  \int_{\Gamma_{T}}\pa_{n}f\, \langle  u , u\rangle_{ T^{*}_{\sigma}\dot\Omega}d\sigma$ is proved  similarly.  This concludes the proof of Lemma~\ref{le.GreenWeak}. 
\end{proof}

\subsubsection{Construction of the domain $\dot\Omega_i$}\label{sec:dotomega}
In this section, we assume \textbf{[H1]}, \textbf{[H2]} and \textbf{[H3]}. Let us consider $z_i\in \{z_1,\ldots,z_n\}$ a local minimum of $f|_{\pa \Omega}$. The objective of this
section is to build the domain $\dot\Omega_i$ on which the 
Witten Laplacian with mixed tangential-normal Dirichlet  boundary conditions will
be defined.
This auxiliary operator  is  such that $z_i$
remains  the only generalized critical point. 
\medskip

\noindent
Let us recall that $x_0 \in \Omega$ is
the minimum of $f$ on $\overline{\Omega}$. Let  $\Omega_{0}$ be a small smooth
open neighborhood  of $x_0$  such that the 
$\partial_n f < 0$ on
$\Gamma_{0}=\pa\Omega_{0}$,  $n$ being the outward normal derivative
to $\Omega \setminus \Omega_0$. Let $\Gamma_{1,i}$
denote a subset of $B_{z_i}$, as large as we want in $B_{z_i}$, and such that $z_i\in \Gamma_{1,i}$. The basic idea is to define
$\dot\Omega=\Omega\setminus\overline{\Omega_{0}}$ and to consider a
Witten Laplacian on $\dot \Omega$, with tangential Dirichlet boundary conditions on $\Gamma_{0}\cup \Gamma_{1,i}$
and with normal Dirichlet boundary conditions on
$\pa\Omega\setminus\overline{\Gamma_{1,i}}$. This would indeed yield an
operator on a domain $\dot{\Omega}$ with a single generalized critical
point, namely $z_i$. 
\medskip

\noindent
There is however a technical difficulty in this approach, related to
the fact that differential forms with mixed normal and tangential
Dirichlet boundary conditions are singular at the boundary between the
domains where tangential and normal boundary conditions are applied,
as explained in Section~\ref{sec:trace_diff_form}. With the previous construction,
$\Gamma_{1,i}$ and $\pa\Omega\setminus\overline{\Gamma_{1,i}}$ meet at an
angle~$\pi$. We therefore need to define a domain $\dot{\Omega}_i$
stricly included in $\Omega \setminus\overline{\Omega_{0}}$, with
boundary $\pa \dot \Omega_i=\overline{\Gamma_0} \cup \overline{\Gamma_{1,i}} \cup
\overline{\Gamma_{2,i}}$ where $\Gamma_0=\pa\Omega_{0}$ as defined above,
$\Gamma_{1,i} \cap \Gamma_{2,i}=\emptyset$, $\Gamma_{1,i} \subset B_{z_i}$ is
as large as we want in $B_{z_i}$ and $\Gamma_{2,i}$ meets
$\Gamma_{1,i}$ at an angle strictly smaller than $\pi$
(see~\eqref{eq:angle} above for a proper definition). We will then
consider a Witten Laplacian with tangential Dirichlet boundary conditions on
$\Gamma_{0}\cap \Gamma_{1,i}$ and normal Dirichlet boundary conditions on
$\Gamma_{2,i}$. Moreover, in order not to introduce new generalized
critical point on $\Gamma_{2,i}$, we would like to keep the property
$\partial_n f >0$ on $\Gamma_{2,i}$ (where $n$ denotes the outward normal
derivative to $\dot \Omega_i$). The aim of this
section is indeed to define such a domain $\dot{\Omega}_i$.

\medskip

\noindent
\underline{A system of coordinates on a neighborhood of $\partial \Omega$.} 

\medskip

\noindent
Let us consider the function $f_-$ defined on a
neighborhood $V_{\partial \Omega}$ of $\partial \Omega$, as introduced in
Proposition~\ref{ffmoins}. Recall that $f_-(x)=0$ for $x \in \partial
\Omega$ and that $V_{\partial \Omega}$
can be chosen such that $f_->0$ on $V_{\partial
  \Omega}\setminus \partial \Omega$ and $|\nabla f_-|\neq 0$ on
$V_{\partial \Omega}$. Let us now consider $\varepsilon >0$ such that 
$$V_\varepsilon=\{y \in \Omega, \, 0 \le f_-(y) \le \varepsilon \}
\subset V_{\partial \Omega}.$$
For any $x \in V_\varepsilon$, the dynamics
\begin{equation}\label{eq:flow_x'}
\left\{
\begin{aligned}
\gamma_{x}'(t)&=-\frac{\nabla f_-}{|\nabla f_-|^2}(\gamma_{x}(t))\\
\gamma_x(0)&=x
\end{aligned}
\right.
\end{equation}
is such that $\gamma_x(t_x)\in \partial \Omega$, where
$$t_x=\inf\{t, \, \gamma_x(t) \not\in {\rm int}\, V_\varepsilon\}.$$
This is indeed a consequence of the fact that $\frac{d}{dt}
f_-(\gamma_x(t)) =-1 < 0$ on $[0,t_x)$. Notice that this also implies that  $t_x$ describes $[0,\ve]$ when $x$ describes $V_\ve$. 

The application
$$\Gamma:\left\{ 
\begin{aligned}
V_\varepsilon & \to  \partial \Omega \times [-\varepsilon,0]\\
x & \mapsto (\gamma_x(t_x),-t_x)
\end{aligned}
\right.
$$
defines a $C^\infty$ diffeomorphism.  The inverse application of $\Gamma$ is
$(x',x_d) \in \partial \Omega \times [-\varepsilon,0] \mapsto \gamma_{x'}(x_d)$.
\label{page.xd}

\begin{definition}\label{def:x'xd} Let us assume that the hypothesis \textbf{[H3]} holds. Let us define the following system of coordinates for $x \in V_{\varepsilon}$:
\begin{equation}\label{eq:x'xd}
\forall x \in V_{\varepsilon}, \, (x'(x),x_d(x))= (\gamma_x(t_x),-t_x) \in \partial \Omega \times [-\varepsilon,0].
\end{equation}
\end{definition}
 \noindent
Notice that, by construction (since $\frac{d}{dt}
f_-(\gamma_x(t)) =-1$), 
$$x_d(x)=-f_-(x).$$
Thus, in this system of coordinates, $\{x_{d}=0\} = \partial
\Omega$ and $\{x_{d}<0\} = \Omega \cap V_\varepsilon$.
We will sometimes need to use a local system of coordinates in
$\partial \Omega$, that we will then denote by the same notation
$x'$. By using the same procedure as above, $(x',x_d)$ then defines a
local system of coordinates. Let us make this precise. For $y\in \partial
  \Omega$, let us consider $x':V_y \to \R^{d-1}$ a smooth local
system of coordinates in $\partial \Omega$, in a neighborhood $V_y
\subset \partial \Omega$ of $y$. These coordinates are then extended
in a neighborhood of $V_y$ in $\Omega$, as constant along the integral
curves of $\gamma'(t)=\frac{\nabla f_-}{|\nabla f_-|^2}(\gamma(t))$, for
$t \in [0,\varepsilon]$. 
The function $x\mapsto (x',x_d)$ (where, we recall, $x_d(x)=-f_-(x)$) thus defines a smooth system
of coordinates in a neighborhood $W_y$ of~$y$ in $\overline\Omega$. In this system of coordinates,
the   metric tensor $G$ writes:
\begin{align*}
G(x',x_d)=G_{dd}(x',x_d)\, dx_{d}^2+\sum_{i,j=1}^{d-1}G_{ij}(x',x_d)\, dx_{i}dx_{j}, 
\end{align*}
where $x'=(x_1, \ldots, x_{d-1})$. In particular if $\psi: V_y\to \mathbb R$ is a Lipschitz function which only depends on $x'$, it holds a.e. on $V_y$:
\begin{equation}\label{eq.psi(x')_0}
|\nabla \psi(x',x_d)|=|\nabla (\psi|_{\Sigma_{x_d}})(x')|,
\end{equation}
where $\forall a>0$, $\Sigma_a:=\{x \in V_\ve, \, x_d(x)=a\}$ is endowed
with the Riemannian structure induced by the Riemannian structure in
$\Omega$.\label{page.sigmaa}

\medskip
\noindent
\underline{Definitions of the functions $\Psi_i$, $f_{+,i}$ and $f_{-,i}$.}

\begin{definition} \label{fplus}  Let us assume that the hypotheses \textbf{[H1]} and \textbf{[H3]} hold. Let us consider~$z_i$ a local minimum of $f|_{\partial \Omega}$ as introduced in hypothesis \textbf{[H2]}. Let us define on  $\Omega$ the following Lipschitz functions 
$$\Psi_i(x):=d_a(x,z_i), \quad
f_{+,i}:=\frac{\Psi_i+ f-f(z_i)}{2} \text{ and } f_{-,i}:=\frac{\Psi_i-( f-f(z_i))}{2}.
$$
\end{definition}
\noindent
Owing 
\label{page.fplus}
  to $\Psi_i(x)=d_a(x,z_i)\geq|f(x)-f(z_i)|$ for all $x\in\Omega$, the functions $f_{\pm,i}$ are non negative and 
$$
f=f(z_i)+f_{+,i}-f_{-,i} \text{ and } \Psi_i=f_{+,i}+f_{-,i}
\text{ on } \Omega.
$$ 
 Let $\Gamma_{1,i}\subset B_{z_i}$ be an open smooth $d-1$ dimensional manifold with boundary such that $z_i\in  \Gamma_{1,i}$ and $\overline {\Gamma_{1,i}}\subset B_{z_i}$. \label{page.gamma1i}
From Proposition~\ref{TTH}, there exists a neighborhood  of $\overline{\Gamma_{1,i}}$ in $\overline \Omega$, denoted $V_{\Gamma_{1,i}}$\label{page.vgamma1i}, such that $\overline{\partial V_{\Gamma_{1,i}} \cap \partial \Omega}\subset  B_{z_i}$ and for all $x\in V_{\Gamma_{1,i}}$,
$$
\Psi_i(x)=\Phi(x)-f(z_i)
$$
where $\Phi$ is the solution to the eikonal equation in a neighborhood of the boundary (see Proposition~\ref{eikonalboundary}). 
Notice that it implies that on $V_{\Gamma_{1,i}}$ the function
$f_{-,i}$ coincides with the function $f_-$ defined in Proposition
\ref{ffmoins} on $V_{\partial \Omega} \cap V_{\Gamma_{1,i}}$.
Moreover, it implies that the functions $f_{\pm,i}$ are $C^{\infty}$  on $V_{\Gamma_{1,i}}$ and one has:
$$
\text{ on } V_{\Gamma_{1,i}}\cap \partial \Omega, \quad  f_{+,i}= f -f(z_i)
 ,
\quad  f_{-,i}= 0 
 ,
\quad \pa_{n} f_{+,i} = 0 , 
\text{ and }
\pa_{n} f_{-,i} = -\pa_{n} f,
$$
where $n$ is the unit outward normal to $\Omega$. Therefore, as in
Proposition~\ref{ffmoins}, up to choosing a smaller neighborhood
$V_{\Gamma_{1,i}}$ of $\overline{\Gamma_{1,i}}$ in $\overline \Omega$,
the function $ f_{-,i}$ is positive on
$V_{\Gamma_{1,i}}\setminus \partial \Omega$ and such that
\begin{equation}\label{eq:nablaf-}
|\nabla f_{-,i}| \neq  0 \text{ in } V_{\Gamma_{1,i}}.
\end{equation}
 Besides, since $\vert \nabla \Psi_i\vert =\vert \nabla f \vert$ in $V_{\Gamma_{1,i}}$, one has
\begin{equation}\label{eq:f+f-}
\nabla f_{+,i}\cdot\nabla f_{-,i}=0 \text{ in }
V_{\Gamma_{1,i}},
\end{equation}
and thus
$$\vert \nabla \Psi_i\vert^2 =\vert \nabla f\vert^2= \vert \nabla
f_{+,i} \vert^2 + \vert \nabla f_{-,i}\vert^2\text{ in }
V_{\Gamma_{1,i}}.$$
In the following, we will assume in addition that $V_{\Gamma_{1,i}}$
is sufficiently small so that the system of
coordinates  $(x',x_d)$ introduced in Definition~\ref{def:x'xd} is
well defined on $V_{\Gamma_{1,i}}$. A consequence of~\eqref{eq:f+f-} is that $\frac{d}{dt}
f_{+,i}(\gamma_x(t))=0$, where $\gamma_x$
satisfies~\eqref{eq:flow_x'}. Thus, in the system of coordinates $(x',x_d)$,
the functions $f_{+,i}$, $\Psi_i$ and~$f$ write:
\begin{align*}
&f_{+,i}(x',x_d)=f_{+,i}(x',0), \,\Psi_i(x',x_d)=f_{+,i}(x',0) -x_{d}\\
&\text{and } f(x',x_d)=f(z_i)+f_{+,i}(x',0)+x_d.
\end{align*}
Notice that by construction 
\begin{equation}\label{eq:nablaf+}
\forall x \in V_{\Gamma_{1,i}},\,  |\nabla f_{+,i}|(x)=0 \iff x'(x)=x'(z_i).
\end{equation}
Indeed, $f_{+,i}(x',x_d)=f_{+,i}(x',0)$ and
$x' \mapsto f_{+,i}(x',0)=f(x',0)-f(z_i)$ has a single critical point at $x'(z_i)$.

\medskip
\noindent
\underline{Strongly stable domain in $B_{z_i}$.}
\medskip

\noindent
In order to build an appropriate domain $\dot \Omega_i$, we will need to
define $\Gamma_{1,i} \subset B_{z_i}$ as a strongly stable domain, as
defined now.
\begin{definition} \label{stronglystable}
A smooth open set $ A\subset \pa\Omega$ is called strongly stable if 
$$\forall \sigma\in \partial A, \quad
\langle\nabla \big (f|_{\pa\Omega}\big)(\sigma), n_{\sigma}( A )\rangle_{T_{\sigma}\pa\Omega} >  0,
$$
where $n_{\sigma}( A)\in T_{\sigma}\pa \Omega$ denotes the outward normal to $A$ at $\sigma\in \partial A$.
\end{definition} 
\noindent
Notice that $\nabla (f|_{\pa\Omega} )=\nabla_T f=\nabla f_{+,i}$ (this is due to the fact that on $B_{z_i}$, one has $f-f(z_i)=\Psi_i$ and thus $\nabla_Tf=\nabla_T\Psi_i$). Thus, the
strong stability condition appearing in Definition~\ref{stronglystable} is equivalent to 
\begin{equation}\label{eq.panf>0}
\forall \sigma\in \partial A, \quad
\partial_{n_{\sigma}( A )} f_{+,i}(\sigma) >  0.
\end{equation}
The name "stable" is justified by the following: if $  A\subset \pa\Omega$ is strongly stable, then for any curve satisfying for all $t> 0$, $\gamma'(t)=-\nabla  (f|_{\pa\Omega}\big)\left(\gamma(t)\right)$ with $\gamma(0)\in \overline A$, one has for all $t\geq 0$, $\gamma(t)\in \overline A$.

The following proposition will be needed to get the existence of an arbitrary large and strongly stable domain in $B_{z_i}$. 

\begin{proposition} \label{stronglystableexis} Let us assume that the hypotheses \textbf{[H1]} and \textbf{[H2]} hold.  For all compact sets  $K\subset B_{z_i}$ there exists a $C^{\infty}$ open domain $A$ which is strongly stable in the sense of Definition~\ref{stronglystable}, simply connected and such that $K\subset A$ and $\overline{A}\subset B_{z_i}$.
\end{proposition} 
\begin{proof}
For the ease of notation, we drop the subscript $i$ in the proof.
One will first construct the set $A$. Then it will be proven that $A$ has the  stated properties.
For $a>0$, let us define
$$L_a:=f|_{\partial \Omega}^{-1}\big ( [f(z),f(z)+a) \big) \cap B_{z}.$$
For a fixed $a>0$ small enough $L_a$ is a $C^{\infty}$ simply connected open set
(which contains~$z$) with boundary the level set $f|_{\partial
  \Omega}^{-1}(\{f(z)+a\})$. The domain $L_a$ is $C^{\infty}$ since $f$ is $C^\infty$.
  
   Let us define for $x\in B_{z}$ the curves $\gamma_x$ by 
$$
\gamma_x'(t)=\nabla f|_{\partial \Omega}(\gamma_x(t)), \quad \gamma_x(0)=x.
$$
For any $x\in \partial L_a$, for all $t>0$, $\gamma_x(-t)\in L_a$ since $t\geq 0\mapsto f|_{\partial \Omega}(\gamma_x(-t))$ is decreasing ($\frac{d}{dt}f|_{\partial \Omega}(\gamma_x(-t))=-\vert \nabla f|_{\partial \Omega}(\gamma_x(-t))\vert^2$ and $f|_{\partial \Omega}(\gamma_x(0))=a$). Let us now define for $T> 0$
$$
A_T:=\{  \gamma_x(t), \  x\in \partial L_a,  t\in [0,T)    \} \cup L_a \subset B_{z}.
$$
One clearly has $A_T\subset A_{T'}$ if $T<T'$.
One claims that $A_T$ is a $C^{\infty}$  simply connected open set which satisfies 
$$\forall \sigma\in \partial A_T, \quad
\partial_{n_{\sigma}( A_T )} f|_{\pa\Omega}(\sigma) > \ 0.
$$
 Let us first prove that $A_T$ is $C^{\infty}$. One has $\partial A_T=\{  \gamma_x(T), \ x\in \partial L_a  \}$. The boundary of $A_T$ is thus a $C^{\infty}$ homotopy of $\partial L_a$ where the homotopy function is 
$$H(t,x)=\gamma_x(t).$$ 
Additionally since this homotopy is with values in $B_{z}$ and since $L_a$ is simply connected (because $L_a$ can be asymptotically retracted on $z$ in the sense that for all $x\in L_a$, $\lim_{t\to -\infty}H(t,x)=z$), $A_T$ is simply connected. Let us prove that  $A_T$ is open. Let us denote by $d_{\partial \Omega}$ the geodesic distance in $\partial \Omega$. Let $x_0\in  A_T\setminus \overline L_a$. There exists a time $t_0\in (0,T)$ such that $\gamma_{x_0}(-t_0)\in L_a$. Let us define $\ve_0=d_{\partial \Omega}(\gamma_{x_0}(-t_0),\partial L_a)/2>0$. Since the mapping $y\mapsto \gamma_y(-t_0)$ is $C^\infty$, there exists $\ve_1>0$ such that if $d_{\partial \Omega}\left(x,y\right)\leq \ve_1$ then $d_{\partial \Omega}(\gamma_y(-t_0),\gamma_{x_0}(-t_0))\leq \ve_0/2$ and thus $\gamma_y(-t_0)\in L_a$. Moreover, since $B_{z}\setminus \overline L_a$ is open, it can be assumed, taking maybe $\ve_1>0$ smaller, that $B_{\partial \Omega}(x_0,\ve_1)\subset B_{z}\setminus \overline L_a$. Then, by continuity, for all  $  y\in B_{\partial \Omega}(x_0,\ve_1)$, there exists $t_0(y)\in (0,t_0)\subset (0,T)$ such that $\gamma_y(-t_0(y))\in \pa L_a$, which implies that $y\in A_T\setminus \overline L_a$. Thus $A_T\setminus \overline L_a$ is open. In addition, since $L_a$ is open and since $\overline L_a\subset A_T$, one has that ${\rm int \ }(A_T)= {\rm int \ }(A_T\setminus \overline L_a)\cup  \overline L_a=(A_T\setminus \overline L_a)\cup  \overline L_a=A_T$. Therefore the set $A_T$ is open. \\Let us now prove that $A_T$ is strongly stable (see Definition~\ref{stronglystable}). By construction, $A_T$ is stable for the dynamics $\gamma'=-\nabla f|_{\pa \Omega}(\gamma)$ and thus  one has 
$$\forall \sigma\in \partial A_T, \quad
\partial_{n_{\sigma}( A_T )} f|_{\pa\Omega}(\sigma) \geq \ 0.
$$
Let us defined now the function
$$\Upsilon :x\in B_{z}\setminus \overline L_a \mapsto (x',t)\in \pa L_a\times \mathbb R_+\ {\rm s.t} \ \gamma_{x'}(t)=x.$$
Notice that $\Upsilon$ is a $C^{\infty}$ diffeomorphism from $B_z$ onto its range, and let us denote $F:=\Upsilon^{-1}$ its inverse function ($F(x',t)=\gamma_{x'}(t)$). Assume that there exists $x\in  A_T$ such that $\partial_{n_{x}( A_T )} f|_{\pa\Omega}(x)=0$ and let $(x',T)=\Upsilon(x)$. This implies that $\nabla f|_{\pa\Omega}(x)\in T_x\partial A_T$ and thus $\partial _t F(x',T)\in  T_x\partial A_T$. Furthermore $\Ran \left(d_{x'}F(.,T)\right)=  T_x\partial A_T$ and thus $d_{(x',T)}F$ is not invertible which contradicts the fact that $F$ is a diffeomorphism. \\
It remains to prove that for any compact set $K\subset B_{z}$, there exists 
$T>0$ such that $K\subset A_T$. One has 
$$ B_{z}=\bigcup_{T>0} A_T.$$
Indeed, if $x\in \overline L_a$, $x\in A_T$ for all $T>0$ and if $x\in B_{z} \setminus \overline L_a$, $\lim_{t\to \infty} \gamma_x(-t)= z$ and thus there exists $s>0$ such that $ \gamma_x(-s)\in \partial L_a$ which implies that $x\in A_s$. Let $K\subset B_{z}$ be a compact set. Then $K\subset \bigcup_{T>0} A_T$ and thus by compactness there exists a sequence $(T_j)_{j=1,\ldots,N}\subset \mathbb R^N$, with  $0<T_1<\ldots<T_m$ such that $K \subset \bigcup_{j=1}^m A_{T_j}=A_{T_m}$. This concludes the proof.
\end{proof}

\medskip
\noindent
\underline{Construction of the domain $\dot{\Omega}_i$.}
\medskip

\noindent
 In this
section, we introduce the domain $\dot{\Omega}_i$ (associated with $z_i$) on which  the
auxiliary Witten Laplacian with mixed tangential-normal Dirichlet
boundary conditions is constructed. 
\label{page.omegaipoint}

\begin{proposition}\label{prop:dotomega}  Let us assume that the hypotheses \textbf{[H1]}, \textbf{[H2]} and \textbf{[H3]} hold.
Let us fix a neighborhood $\Omega_0$ of $x_0$ (the global
minimum of $f$ in $\Omega$) such that 
$$\partial_n f < 0 \text{ on } \Gamma_0 := \partial \Omega_0 \label{page.gamma0}$$
% Prendre une ligne de niveau de $f$
where $n$ denotes the outward normal to $\Omega \setminus \Omega_0$ on $\Gamma_0$. Let us consider a critical point $z_i$ of $f|_{\partial
  \Omega}$. Then there exists a smooth open subset $\Gamma_{1,i}$ of
$B_{z_i}$ containing $z_i$ and arbitrarily large
 in $B_{z_i}$, a neighborhood
$V_{\Gamma_{1,i}}$ of $\overline{\Gamma_{1,i}}$ in $\overline{\Omega}$
such that $\overline{V_{\Gamma_{1,i}} \cap \partial \Omega} \subset
B_{z_i}$ % utile pour le $V_{\Gamma_{1,i}}'$ ensuite
 and a Lipschitz subset $\dot
\Omega_i$ of $\Omega \setminus \Omega_0$ which are such that the following properties are
satisfied:
\begin{enumerate}
\item Following Proposition~\ref{TTH},
$$\forall x \in V_{\Gamma_{1,i}},\, d_a(x,z_i)=\Phi(x) - f(z_i)$$
where $\Phi$ is the solution to the eikonal
equation~\eqref{eikonalequationboundary} ;
\item The system of coordinates $(x',x_d)$ is defined on
  $V_{\Gamma_{1,i}}$, see Definition~\ref{def:x'xd} ;
\item $\partial \dot \Omega_i$ is composed of two connected components:
  $\Gamma_0$ and  $\overline{\Gamma_{1,i}} \cup
  \overline{\Gamma_{2,i}}$, where $\Gamma_{2,i}$ is an open subset of $\pa \dot \Omega_{i}$, $\Gamma_{1,i}\cap \Gamma_{2,i}=\emptyset$, and  $\Gamma_{2,i}$    is such that there exist two disjoint~$C^\infty$ open subsets $\Gamma_{2,i}^1$ and $\Gamma_{2,i}^2$  of $\pa \dot \Omega_{i}$ such that    \label{page.gamma2i}
$$ 
\overline{\Gamma_{2,i}}=\overline{\Gamma_{2,i}^1}\cup \overline{\Gamma_{2,i}^2}
 \, ;
 $$
\item $\Gamma_{1,i}$ and $\Gamma_{2,i}$ meet  at an angle smaller than
  $\pi$, see~\eqref{eq:angle} for a precise definition ;
\item It holds,
\begin{equation}\label{eq:dnfpositive}
\forall x \in \partial \dot \Omega_i \setminus \Big( \big (\overline{\Gamma_{2,i}^1}\cap \overline{\Gamma_{2,i}^2}
 \, \big )\cup \big (\overline{\Gamma_{1,i}}\cap \overline{\Gamma_{2,i}}
 \, \big ) \cup \Gamma_0\Big ), \ \partial_n
f(x) > 0
\end{equation}
where $n$ is the outward normal to $\dot \Omega_i$ ;
\item It holds,
\begin{equation}\label{eq:dnf+positive}
\forall x \in \Gamma_{2,i} \cap V_{\Gamma_{1,i}}, \, \partial_n
f_{+,i}(x) > 0;
\end{equation}
\item Moreover, for all $\delta > 0$,  $\dot\Omega_i$ (and $\Omega_0$) can be chosen such that
\begin{equation}\label{eq:gamma2_close_to_Bzc}
\sup \{d_e(x,y), \, x \in \Gamma_{2,i}, \, y \in B_{z_i}^c \}\le
\delta
\end{equation}
and
\begin{equation}\label{eq:gamma0_close_to_x0}
\sup\{d_e(x_0,x), \, x \in \Gamma_0 \} \le \delta
\end{equation}
where, we recall, $d_e$ denotes the geodesic distance for the
Euclidean metric on $\overline{\Omega}$.
% $$d_a(z_i,\Gamma_{2,i})\ge d_a(z_i, B_{z_i}^c) - \delta,$$
% and
% $$d_a(z_i,\Gamma_{0})\ge d_a(z_i, x_0) - \delta.$$
\end{enumerate}
\end{proposition}
We refer to 
Figure \ref{fig:Sets for mixed Witten  Laplacian} for a schematic
representation of $\dot \Omega_i$.
\begin{proof} The domain $\dot{\Omega}_i \subset \Omega$ is built as follows. First,
let us fix a neighborhood $\Omega_0$ of $x_0$ such
that~\eqref{eq:gamma0_close_to_x0} is satisfied and 
$$\partial_n f < 0 \text{ on } \Gamma_0 = \partial \Omega_0$$
% Prendre une ligne de niveau de $f$
where $n$ denotes the outward normal to $\Omega \setminus \Omega_0$ on
$\Gamma_0$ (this can be done for example by considering $\Omega_0=\{x, \,
f(x) < f(x_0) + \eta\}$ for some positive $\eta$).
Second, let us consider a smooth subset $\Gamma_{1,i}$ of $B_{z_i}$ which can 
be as large as needed in $B_{z_i}$, and which is strongly stable (see
Proposition~\ref{stronglystableexis} for the existence of such a set):
\begin{equation}\label{eq:derivee_normale_positive}
\langle \nabla f|_{\partial \Omega} , n(\Gamma_{1,i}) \rangle_{T \pa \Omega} > 0 \text{ on
} \partial \Gamma_{1,i} \text{ and thus } \pa_{n(\Gamma_{1,i})}f_{+,i}>0 \text{ on
} \partial \Gamma_{1,i},
\end{equation}
where $n(\Gamma_{1,i})$ denotes the outward normal derivative of
$\Gamma_{1,i}$ (see Definition~\ref{stronglystable} and~\eqref{eq.panf>0}).
%{eq.panf>0}

Once $\Gamma_{1,i}$ is fixed, the existence of a neighborhood
$V_{\Gamma_{1,i}}$ of $\overline{\Gamma_{1,i}}$ in $\overline{\Omega}$
such that $\overline{V_{\Gamma_{1,i}} \cap \partial \Omega} \subset
B_{z_i}$ and such that items 1 and 2 are fulfilled is a direct consequence of Proposition~\ref{TTH}.

Let us now consider the system of coordinates $(x',x_d)$ introduced in
Definition~\ref{def:x'xd}. Let $ V_{\partial
  \Gamma_{1,i}} \subset \partial \Omega$ denotes a neighborhood of
$\partial \Gamma_{1,i}$ in $\partial \Omega$ and
$$ V^+_{\partial\Gamma_{1,i}} =\overline{V_{\partial
  \Gamma_{1,i}} \cap \Gamma_{1,i}^c}.$$
The domain $\dot \Omega_i$ is then defined as follows:
$$\dot{\Omega}_i = \Omega \setminus (\overline{\Omega_0} \cup \{x=(x',x_d), \,x'  \in
V^+_{\partial\Gamma_{1,i}} \text{ such that } x_d(x) \in [-\varphi(x'),0] \})$$
where $\varphi:V^+_{\partial\Gamma_{1,i}}  \to \R_+$ is a smooth
function such that
$$\exists \varepsilon > 0, \, \forall x' \in \partial
  \Gamma_{1,i}, \, \varphi(x') \ge \varepsilon,$$
see Figure~\ref{fig:zoom-angle} for a schematic representation.
Notice that by construction,   $\dot{\Omega}_i$ is a connected Lipschitz subset of $\Omega$ and, denoting by $\Gamma_{2,i}=\partial \dot \Omega_i
\setminus (\overline{\Gamma_{1,i}} \cup \Gamma_0)$, item 3 is satified. For each point $z \in \partial \Gamma_{1,i}$,
there is a small neighborhood $\mathcal V$ of $z$ such that 
\begin{equation}\label{eq:gamma_2}
\mathcal V \cap \Gamma_{2,i} \subset \{x=(x',x_d), x' \in \partial \Gamma_{1,i} \text{ and } x_d(x) \in (-\eta,0] \},
\end{equation}
for some $\eta \in (0,\ve)$.
By choosing $\Gamma_{1,i}$ sufficiently large in $B_{z_i}$, and
$\varphi$ such that $\max \varphi$ is sufficiently small,~\eqref{eq:gamma2_close_to_Bzc} is satisfied. This concludes the
verification of item 7.

For each point  $y \in \partial \Gamma_{1,i}$, it is possible to
construct locally a normal system of coodinate $x'=x_T=(x_{T,1},x_{T,2},
\ldots, x_{T,d-1})$ in a neighborhood $V_y$ of $y$ in $\partial \Omega$,
such that $\Gamma_{1,i} \cap V_y = \{x \in V_y, x_{T,1}(x) \le 0\}$, 
 $V^+_{\partial\Gamma_{1,i}}\cap V_y = \{x \in V_y, x_{T,1}(x) \ge 0\}$    and $\partial \Gamma_{1,i} \cap V_y = \{x \in V_y, x_{T,1}(x) =
0\}$. As explained after Definition~\ref{def:x'xd}, by extending this
system of coordinate inside $\Omega$ as constant along the curve
associated with the vector field $\frac{\nabla f_{-,i}}{|\nabla f_{-,i}|^2}$, $x\mapsto (x'(x),x_d(x))$ then defines a local system of
coordinates in a neighborhood $W_y$ of $y$ in $\overline{\Omega}$. For all $x
\in \partial \Gamma_{1,i}$, the vector $n_z(\Gamma_1)=\frac{\nabla x_{T,1}(x)}{|\nabla x_{T,1}(x)|}$ is the
outward normal vector to $\Gamma_{1,i}$ on $\partial
\Gamma_{1,i}$. By a compactness
argument, one gets that $\partial \Gamma_{1,i} \subset
\cup_{k=1}^K \mathcal V_{y_k}$ for a  finite number of points $y_k
\in \partial \Gamma_{1,i}$. See
Figure~\ref{fig:phik} for a schematic representation of the function $\varphi$
in this system of coordinates.

Let us now look at the boundary of $\dot{\Omega}_i$ in a neighborhood of
$\partial \Gamma_{1,i}$ (see Figure~\ref{fig:zoom-angle}). For $\sigma \in \partial \dot{\Omega}_i$, let us
denote by $n_\sigma(\dot \Omega_i)$ the unit outward normal to $\dot
\Omega_i$. Let us show that for all $z
\in \partial \Gamma_{1,i}$,
\begin{equation}\label{eq:limit_nsigma}
\lim_{\sigma \to z} n_\sigma(\dot{\Omega}_i)=n_z(\Gamma_{1,i})
\end{equation}
where the limit is taken for $\sigma \in \Gamma_{2,i}$. Let us prove~\eqref{eq:limit_nsigma}. For any point $z \in \partial \Gamma_{1,i}$,
there is a small neighborhood $\mathcal V$ of $z$ in
$\overline{\Omega}$ such that the system of coordinates $(x_T,x_d)$
introduced above is well defined. In this system of coordinates,
$$\partial \dot{\Omega}_i \cap \mathcal V \cap \Gamma_{2,i} \subset \{x \in \mathcal
V, \, x_{T,1}(x)=0 \text{ and } x_d(x) \in [-\varphi(x'(x_T(x))),0] \}.
$$
Moreover, the outward normal to
$\dot{\Omega}_i$ on this subset is $n(\dot{\Omega}_i)=\frac{\nabla
  x_{T,1}}{|\nabla x_{T,1}|}$ and thus, by construction, for all $z
\in \partial \Gamma_{1,i}$,~\eqref{eq:limit_nsigma} holds.

As a consequence of~\eqref{eq:limit_nsigma}, the two submanifolds $\Gamma_{1,i}$ and $\Gamma_{2,i}$ meet at an angle smaller than $\pi$
(see~\eqref{eq:angle} and Figure~\ref{fig:phi-x1}). This shows that
item 4 is satisfied.
Moreover, using~\eqref{eq:derivee_normale_positive},  one has: for all $z
\in \partial \Gamma_{1,i}$,
$$\lim_{\sigma \to z} \nabla f_{+,i} (\sigma) \cdot n_\sigma(\dot{\Omega}_i) 
= \nabla f_{+,i} (z) \cdot n_z(\Gamma_{1,i})>0,$$
and 
$$\lim_{\sigma \to z} \nabla f (\sigma) \cdot n_\sigma(\dot{\Omega}_i) 
= \nabla f (z) \cdot n_z(\Gamma_{1,i})>0,$$
where the limits are taken for $\sigma\in   \Gamma_{2,i}$. Thus, up
to choosing $\varphi$ with $\max \varphi$ sufficiently small, it is
possible to build $\dot \Omega_i$ such that (see Figure~\ref{fig:zoom-angle})
\begin{equation}\label{eq:dnf>0}
\forall x \in \Gamma_{2,i} \text{ such that } x'(x) \in \partial \Gamma_{1,i}, \, \partial_n
f_{+,i}(x) > 0 \text{ and } \partial_n
f(x) > 0,
\end{equation}
where $n$ here denotes the outward normal to $\dot \Omega_i$. 
%This
%implies item 6. It also implies (since
%$\forall x \in \Gamma_{2,i} \text{ such that } x'(x)\in \partial \Gamma_{1,i},  \partial_n f (x)
%=\partial_n f_{+,i}(x)- \partial_n f_{-,i}(x)=\partial_n f_{+,i}(x)$ which follows from the fact that $\partial_n f_{-,i}(x)=0$):
%%\begin{equation}\label{eq:nablaf-oubli}
%\begin{equation}\label{eq:dnf>0}
%\forall x \in \Gamma_{2,i} \text{ such that } x'(x) \in \partial \Gamma_{1,i}, \, \partial_n
%f(x) > 0.
%\end{equation}
Finally, by using~\eqref{eq:dnf>0} and since $\partial_n f > 0$ on
$\partial \Omega$, up to choosing $\varphi$ with $\max \varphi$ sufficiently small, it is
possible to build $\dot \Omega_i$ such that (see Figure~\ref{fig:zoom-angle})
$$\forall x \in \partial \dot \Omega_i \setminus \Big( \big (\overline{\Gamma_{2,i}^1}\cap \overline{\Gamma_{2,i}^2}
 \, \big )\cup \big (\overline{\Gamma_{1,i}}\cap \overline{\Gamma_{2,i}}
 \, \big ) \cup \Gamma_0\Big ),  \ \partial_n f(x) > 0$$
where $n$ again denotes the outward normal to $\dot \Omega_i$. This is
item 5, and this concludes the proof of Proposition~\ref{prop:dotomega}.

\end{proof}

\begin{definition} \label{Sets}  Let us assume that the hypotheses \textbf{[H1]}, \textbf{[H2]} and \textbf{[H3]} hold.
Let us consider a critical point $z_i$ of $f|_{\partial \Omega}$. In the following, we denote by $$\mathcal S_{M,i}:=\{\dot{\Omega}_i,
\Gamma_0,\Gamma_{1,i},\Gamma_{2,i},V_{\Gamma_{1,i}}\}$$ an ensemble of
sets satisfying the
requirements of Proposition~\ref{prop:dotomega}.\label{page.SMi}
\end{definition}
% In Section~\ref{sec:contruct_quasimode}, when dealing with the index $i\in \{1,\dots,n\}$ to construct the quasi-modes associated to $z_i$, the sets $\Gamma_{2}$ and $S_{M}$ will be respectively denoted by $\Gamma_{2,i}$ and $S_{M,i}$. \comment{ok?}
\noindent
In the following, in order to reduce the amount of notation, the index
$i$ will sometimes be omitted. Thus, we will denote
$$z=z_i, \ \Gamma_{1}=\Gamma_{1,i}, \ \Gamma_{2}=\Gamma_{2,i}, \, \dot
\Omega=\dot \Omega_i, \ V_{\Gamma_{1}}=V_{\Gamma_{1,i}},\ \Psi=\Psi_i,
\ f_{+}=f_{+,i} \ {\rm and} \ f_{-} =f_{-,i}.  \label{page.nota0}$$
We shall warn the reader whenever
the index $i$ is dropped.

\begin{figure}
\begin{center}

\begin{tikzpicture}

\draw[very thick]   (0,3)  ..controls (0,2.6) and (0.3,2.9).. (1,2.9);      %(0,3) arc (220:310:0.7cm) ;
\draw[very thick] (0,-3)  ..controls (0,-2.6) and (0.3,-2.9).. (1,-2.9); %(0.6,-2.95) arc (45:140:0.4cm);

\draw[very thick] (1,2.9) ..controls (5,2) and (5.1,-2) .. (1,-2.9)node[midway,left=0.03cm]{$\Gamma_{2,i}$} ;

\draw (0,0) ellipse (4 and 3)  ;

\tikzstyle{vertex}=[draw,circle,fill=black,minimum size=6pt,inner sep=0pt]

\draw (-4,0) node[vertex,label=above left:{$z_i$}] (v) {};
\draw (0,3) node[vertex,label=above left:{$  \ \overline \Gamma_{1,i} \cap \overline \Gamma_{2,i} $}] (v) {};
\draw (0,-3) node[vertex,label=below:{$  \ \overline \Gamma_{1,i} \cap \overline \Gamma_{2,i} $}] (v) {};
\draw (0,0) node[vertex,label=below:{$x_0$}] (v) {};
\draw (0,0) circle (0.7cm)  node[midway,right=0.7cm]{$\Gamma_0$} ;
\draw (-3.3,1.5) node[label=below:{$\Gamma_{1,i}$}] (v) {};
%\draw (0,-3.5) node[midway,below=4.2cm] {Domain $\Omega$};
\draw (-2,0) node {$\dot \Omega_i$};

\draw[<->] (-0.5,2.95) arc (220:287:0.6cm) node[label=below:{$ \pi/2$}] (v) {};
\draw[<->] (0.03,-2.85) arc (45:160:0.3cm) node[above=0.2cm] {$ \pi/2$};
%%%%
\draw (0,8.5) node[label=above:{$\Gamma_{1,i}$ {\rm in}  $\partial \Omega$}] (v) {};
\draw[->] (-4,6)--(4,6);
\draw[->] (0,4)--(0,8);
\draw (0.3,8) node[below] { $x_N$};

\draw (-3,6) ..controls (0,7) .. (3,6) node[midway,right=0.9cm]{$V_{\Gamma_{1,i}}$} ;

\draw (0 ,6) node[vertex,label=north east: {$z_i$}](v){};
\draw[<->](-2.5,5.5)--(2.5,5.5)  node[midway,below right] {\Ga{1,i}} ;
\draw[<->](-3.5,4.5)--(3.5,4.5)  node[midway,below right] {$B_{z_i}$} ;

\draw[dashed] (-2.5,5.5)--(-2.5,6) ;
\draw[dashed] (2.5,5.5)--(2.5,6) node[midway,right=1cm] {$\partial \Omega$};
\draw[dashed] (-3.5,4.5)--(-3.5,6);
\draw[dashed] (3.5,4.5)--(3.5,6);

\end{tikzpicture}
\caption{The ensemble of sets ${\mathcal S}_{M,i}=\{\dot{\Omega}_i,
\Gamma_0,\Gamma_{1,i},\Gamma_{2,i},V_{\Gamma_{1,i}}\}$ associated with a
critical point~$z_i$ of $f|_{\partial \Omega}$.}
 \label{fig:Sets for mixed Witten  Laplacian}
 \end{center}
\end{figure}
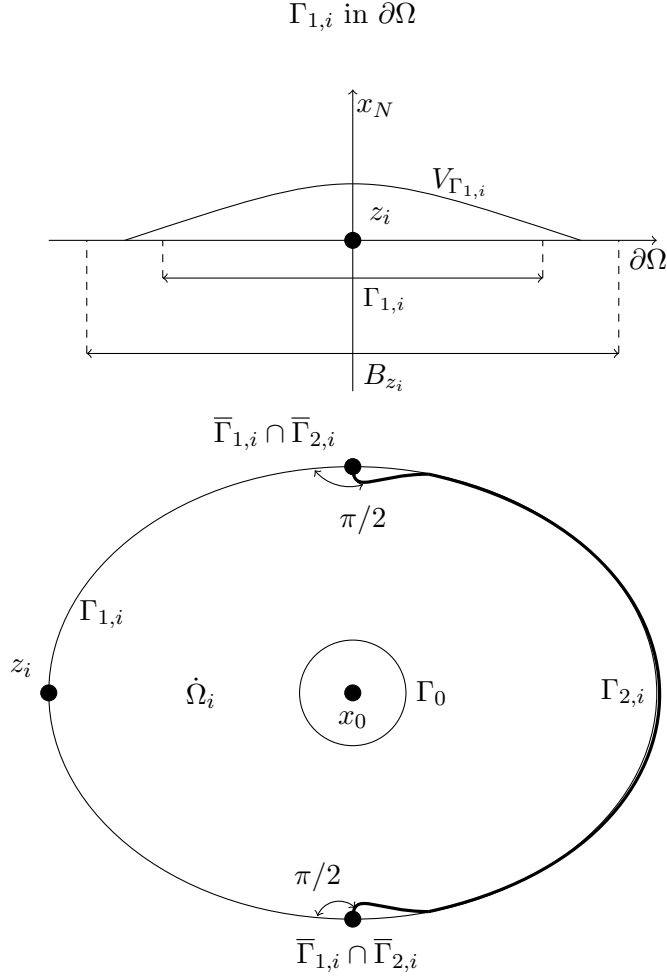

\begin{figure}[!h]
  \begin{center}
 \begin{tikzpicture}
\draw (1.4,1.4) arc (130:138.5:33cm);
\draw (1.8,1.7) node  {$\partial \Omega$};
\draw[->, dashed] (-1.5,-1.35)-- (0.5,-2.94);
\draw[-,  thick] (-1.5,-1.35)-- (-0.3,-2.3);
\draw [ thick] (-0.3,-2.3)  ..controls (1,-1)  and (-1.1,-0.7) .. (0.72 ,0.81);
\draw[-,  thick] (-1.5,-1.35)-- (-0.3,-2.3);
\draw[-,  thick]  (-2.28,-2.21)-- (-1.5,-1.35);
\tikzstyle{vertex}=[draw,circle,fill=black,minimum size=6pt,inner sep=0pt]
\draw (-1.5,-1.35) node[vertex](v){};
\draw (-2 ,-1.1) node  {$\partial\Gamma_{1,i}$};
\draw[->] (-0.98,-1.75)-- (-0.3,-1);
\draw (-2.6 ,-2.4) node  {$\Gamma_{1,i}$};
\draw (-0.8 ,-1.1)  node  {\tiny{$n(\dot \Omega_i)$}};
\draw (2 ,-1) node  {$\Gamma_{2,i}$};
\draw (2 ,-3) node  {$\dot \Omega_i$};
\draw[->, densely dashdotted]  (1.6 ,-1) -- (-0.8,-1.8);
\draw[->, densely dashdotted]  (1.6 ,-1) -- (0,-0.6);
\draw (0.56,-3.2) node  {$-x_d$};
\draw [<->,dashed]  (-1.1,-1.7)  arc (0:-140:0.5cm);
\draw (-1.3 ,-2.4) node  {$ \frac{\pi}{2}$};
\end{tikzpicture}
\caption{$\dot \Omega_{i}$ near $\partial \Gamma_{1,i}$. }
 \label{fig:zoom-angle}
  \end{center}
\end{figure}
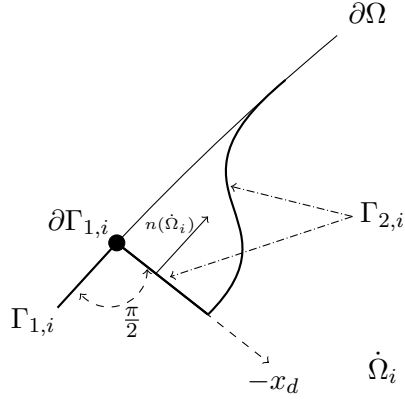

 \begin{figure}[!h]
  \begin{center}
 \begin{tikzpicture}
\draw[->] (0,0)--(4,0) node[right] {$x_T=(x_{T,1},\dots, x_{T,d-1})$} ;
\draw[->] (0,0)--(0,2);
\draw (0 ,0) node [anchor=north] {$(0,x_{T,2},\dots, x_{T,d-1})\in \pa \Gamma_{1,i}$};
\draw (-3.2 ,1) node [anchor=south] {$\varphi(0,x_{T,2}, \dots, x_{T,d-1})\ge \varepsilon$};
\draw[->, densely dashdotted] (-0.9,1.35) -- (-0.14,1.45);
\draw (1.5,1.5)  ..controls (2.5,1.55)  and  (2.5,0) .. (3.5,0);
\draw (0,1.5) -- (1.5,1.5) node[midway,above=0.4cm] {$\varphi$};
\tikzstyle{vertex}=[draw,circle,fill=black,minimum size=6pt,inner sep=0pt]
%\tikzstyle{ball}=[circle, dashed, minimum size=1cm, draw]
\tikzstyle{point}=[circle, fill, minimum size=.01cm, draw]
\draw (0,0) node[vertex](v){};
\end{tikzpicture}
\caption{The function $x_T\in \mathbb R^{d-1}\mapsto \varphi(x_T)$.}
 \label{fig:phik}
  \end{center}
\end{figure}
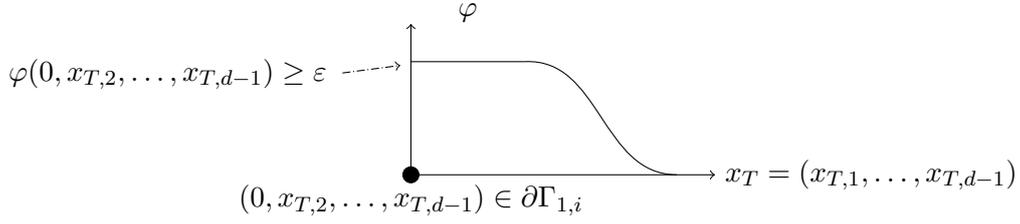

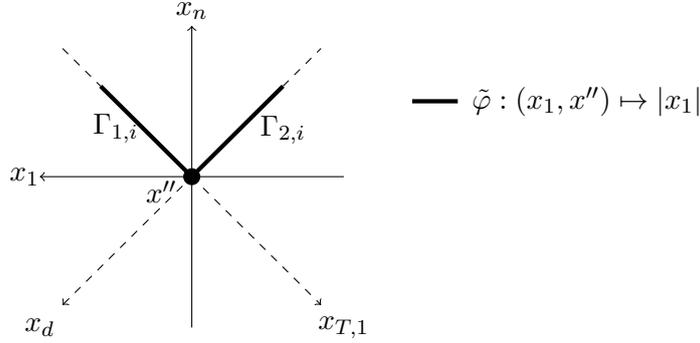
\begin{figure}[!h]
  \begin{center}
 \begin{tikzpicture}
\draw[<-] (-2,0)--(2,0);
\draw[->] (0,-2)--(0,2);
\draw[  dashed] (-1.7,1.7)--(-1.2,1.2);
\draw[-, ultra thick] (-1.2,1.2)--(0,0);
\draw [->,  dashed] (0,0)--(1.7,-1.7);
\draw (-0.4 ,-0.2) node {$x''$};
\draw[<-,  dashed] (-1.7,-1.7)--(0,0);
\draw[-, ultra thick] (0,0)--(1.2,1.2);
\draw[  dashed] (1.2,1.2) --(1.7,1.7);
\draw (2 ,-2) node  {$x_{T,1}$};
\draw (0 ,2.2) node  {$x_{n}$};
\draw (-2.2 ,0) node  {$x_{1}$};
\draw (-2 ,-2) node  {$x_{d}$};
\draw (-1 ,0.64) node  {$\Gamma_{1,i}$};
\draw (1.2 ,0.62) node  {$\Gamma_{2,i}$};
%\draw[->] (-0.9,1.35) -- (-0.14,1.45);
%\draw (1.5,1.5)  ..controls (2.5,1.55)  and  (2.5,0) .. (3.5,0);
%\draw (0,1.5) -- (1.5,1.5) node[midway,above=0.4cm] {$\varphi_k$};
\draw[-, ultra thick] (2.9,1)-- (3.5,1);
\draw (5.2,1) node  {$\tilde \varphi: (x_1,x'')\mapsto \vert x_1\vert$};
\tikzstyle{vertex}=[draw,circle,fill=black,minimum size=6pt,inner sep=0pt]
%\tikzstyle{ball}=[circle, dashed, minimum size=1cm, draw]
\tikzstyle{point}=[circle, fill, minimum size=.01cm, draw]
\draw (0,0) node[vertex](v){};
\end{tikzpicture}
\caption{The sets $\Gamma_{1,i}$ and $\Gamma_{2,i}$ meet at an angle
  smaller than $\pi$, according to \eqref{eq:angle}: take
  $x_1=\frac{x_d-x_{T,1}}{2}$, $x_n=-\frac{x_d+x_{T,1}}{2}$,
  $x''=(x_{T,2},\dots, x_{T,d-1})$ and $\tilde \varphi(x_1,x'')=|x_1|$. }
 \label{fig:phi-x1}
  \end{center}
\end{figure}

\subsubsection{On the spectrum of the Witten Laplacian $\Delta^M_{f,h}(\dot{\Omega}_i)$}\label{sec:def_mixed}
Troughout this section, one assumes \textbf{[H1]}, \textbf{[H2]} and \textbf{[H3]}. In this section, we introduce a Witten Laplacian with mixed
tangential and normal Dirichlet boundary conditions,
associated with the critical point $z_i$. Let $\mathcal S_{M,i}:=\{\dot{\Omega}_i,\Gamma_0,
\Gamma_{1,i},\Gamma_{2,i},V_{\Gamma_{1,i}}\}$ be an ensemble of sets associated
with $z_i$, see Definition~\ref{Sets}.
\medskip

\noindent
Let us now consider the Witten Laplacian $\Delta^M_{f,h}$
on $\dot\Omega_i$ with homogeneous Dirichlet  tangential boundary
conditions on $$\Gamma_T=\Gamma_{0}\cup\Gamma_{1,i}$$
\label{page.gammatbis} and homogeneous Dirichlet  normal boundary
conditions on $$\Gamma_N=\Gamma_{2,i}$$ 
as defined at the end of Section~\ref{sec:trace_diff_form} (see~\eqref{eq.DeltaTN}--\eqref{eq.domDeltaTN}).
%
% In
% particular, this requires some results on traces of forms, taken
% from~\cite{jakab-mitrea-mitrea-09}. The following analyis is independent  of the fact that
% $\Gamma_{1}$ and $\Gamma_{2}$ meet at an angle strictly smaller
% than $\pi$ until the statement of the further Proposition~\ref{pr.QTN}, where this
% fact will be clearly specified.
%
%
%
The main result of this section concerns the spectrum of the operator $\Delta_{f,h}^{M} ( \dot \Omega_i )$. 
\begin{proposition}
\label{pr.DeltaTN} Let us assume that the hypotheses \textbf{[H1]}, \textbf{[H2]} and \textbf{[H3]} hold. 
Let $\Delta_{f,h}^{M,(p)} ( \dot \Omega_i )$ be the unbounded
 nonnegative selfadjoint  operator   on 
$ L^{2} ( \dot \Omega_i )$ defined by \eqref{eq.DeltaTN}
and with domain~\eqref{eq.domDeltaTN} with $\Gamma_T=\Gamma_{1,i}\cup \Gamma_0$ and $\Gamma_N=\Gamma_{2,i}$. One has:
\begin{enumerate}
\item[(i)] The operator $\Delta_{f,h}^{M,(p)} ( \dot
  \Omega_i )$ has compact resolvent.
\item[(ii)] For any eigenvalue $\lambda_{p}$ of $\Delta_{f,h}^{M,(p)} ( \dot \Omega_i )$
and associated eigenform $u_{p}\in D\left (\Delta_{f,h}^{M,(p)} ( \dot \Omega_i )\right)$, one has
$
d_{f,h}u_{p}\in D\left (\Delta_{f,h}^{M,(p+1)} \right)
$ and
$
d^{*}_{f,h}u_{p}\in D\left (\Delta_{f,h}^{M,(p-1)} \right)
$,
with
$$
d_{f,h}\Delta_{f,h}^{M,(p)}u_{p}= 
\Delta_{f,h}^{M,(p+1)}d_{f,h}u_{p}
= \lambda_{p}d_{f,h}u_{p}$$
and
$$
d^*_{f,h}\Delta_{f,h}^{M,(p)} ( \dot \Omega_i )u_{p}= 
\Delta_{f,h}^{M,(p-1)} ( \dot \Omega_i )d^*_{f,h}u_{p}
= \lambda_{p}d^*_{f,h}u_{p}.
$$
If in addition $\lambda_{p}\neq 0$, either $d_{f,h}u_{p}$ or $d^*_{f,h}u_{p}$ is nonzero.
\item[(iii)] There exist $c>0$ and $h_0>0$ such that for any $p\in\{0,\dots,n\}$ and $h\in (0,h_0)$,
$$
\dim\Ran \ \pi_{[0,ch^\frac32)}\left(\Delta_{f,h}^{M,(p)} ( \dot \Omega_i )\right)
=\delta_{1,p}\quad\text{and}\quad
\Sp\!\!\left(\Delta_{f,h}^{M,(1)} ( \dot \Omega_i )\right)\cap [0,ch^\frac32)
=
\{0\},
$$
where $\delta$ is the Kronecker delta:  $\delta_{1,p}=1$ iff $p=1$. 
\end{enumerate}
\end{proposition}
\begin{proof}
 Since the criticial point $z_i$
is fixed, for the ease of notation, we
drop the subscript~$i$ in the proof. 

The point $(i)$ follows from the compactness
of the embedding $ H^{\frac12} ( \dot \Omega )\hookrightarrow
 L^{2} ( \dot \Omega )$ (since additionally $D\left ( \mathcal Q^{M}_{f,h} ( \dot \Omega )\right )\hookrightarrow H^{\frac12} ( \dot \Omega )$ is continuous
according to Proposition~\ref{pr.QTN}). 
The point $(ii)$ is then a straightforward  consequence
of the characterization of the domain of $\Delta_{f,h}^{M,(p)} ( \dot \Omega )$
together with \eqref{eq.dom-dT}. The statement in the case
$\lambda_{p}\neq 0$  follows from
$$
0\ \neq\  \lambda_{p}\langle u_{p},u_{p}\rangle_{L^{2} ( \dot \Omega )}= \langle \Delta_{f,h}^{M,(p)} ( \dot \Omega )u_{p},u_{p}\rangle_{L^{2} ( \dot \Omega )}
= \langle d_{f,h} u,d_{f,h}u_{p} \rangle_{L^{2} ( \dot \Omega )} +\langle d^{*}_{f,h}u_{p},d^{*}_{f,h}u_{p} \rangle_{L^{2} ( \dot \Omega )} .
$$
Let us now give the proof of the last point $(iii)$, which is a consequence of $(ii)$ together with ideas from~\cite{helffer-nier-06,le-peutrec-10},
since $z_{i}$ is the only generalized critical point of $f$ in
$\overline{\dot\Omega}$ for $\Delta_{f,h}^{D,(1)}(\dot\Omega)$. Let us first prove  that for some $c>0$, one has for any $p\in\{0,\dots,n\}$ and $h$ small enough,
$$
\dim\Ran \ \pi_{[0,ch^\frac32)}\left(\Delta_{f,h}^{M,(p)} ( \dot \Omega )\right)
=\delta_{1,p}.
$$
Pick up $u\in D\left ( \mathcal Q^{M,(p)}_{f,h} ( \dot \Omega )\right)$.
From the Green formula~\eqref{eq.compfor} and from the fact that $\mathcal{L}_{\nabla
  f}+\mathcal{L}_{\nabla f}^{*}$
  is a $0^{\text{th}}$ order differential operator,   there exists $C_0>0$ such that  
 for all $u\in  D\left ( \mathcal Q^{M,(p)}_{f,h} ( \dot \Omega )\right)$ and all smooth cut-off function $\chi$  supported in $\dot\Omega$ (whose support avoids $\pa\dot\Omega$):

 $$\left\| d_{f,h}\chi u\right\|^{2}_{L^{2} ( \dot \Omega )}+
\left\|d_{f,h}^{*}\chi u\right\|^{2}_{L^{2} ( \dot \Omega )} \geq \Big(\inf_{{\rm supp}\chi} \vert \nabla f \vert^2 -hC_0\Big)\left\|\chi u\right\|^{2}.$$
 Thus, since $f$ has no critical point in $\dot\Omega$,  there exists some $C>0$ independent of $u\in  D\left ( \mathcal Q^{M,(p)}_{f,h} ( \dot \Omega )\right)$ such that
 for any smooth cut-off function $\chi$  supported in $\dot\Omega$
 (whose support avoids $\pa\dot\Omega$)  and $h$ small enough,
\begin{equation*}
\label{eq-iv)-1}
\mathcal{Q}^{M,(p)}_{f,h} ( \dot \Omega )(\chi u) \ \geq\  C\|\chi u\|^{2}.
\end{equation*} 
Note in addition that owing to $\pa_{n}f> 0$ on $\Gamma_{2}\setminus \big (\overline{\Gamma_{2,i}^1}\cap \overline{\Gamma_{2,i}^2}
 \, \big )$
and $\pa_{n}f<0$ on $\Gamma_{0}$,  the boundary terms in the Green formula~\eqref{eq.compfor} are non negative, for any smooth cut-off function
$\chi$ supported in a neighborhood of any point in $\Gamma_{2}\cup\Gamma_{0}$ (whose
support avoids some neighborhood of~$\Gamma_{1}$). Thus, the
 same considerations show that 
for $h$ small enough, for such functions $\chi$, 
taking maybe $C$ smaller, $\forall u \in  D\left ( \mathcal Q^{M,(p)}_{f,h} ( \dot \Omega )\right)$
\begin{equation*}
\label{eq-iv)-2}
\mathcal{Q}^{M,(p)}_{f,h}  ( \dot \Omega )(\chi u) \ \geq\  C\|\chi u\|^{2}.
\end{equation*} 
 According to the analysis done in \cite[Section~3.4]{helffer-nier-06}, the same estimate also holds 
for $\chi$ supported in a sufficiently small neighborhood of some point $x\neq z$, $x\in \Gamma_{1} $
(whose support avoids a neighborhood of $\{z\}\cup\pa\Gamma_{1}$). This  is related to the fact
that $\Gamma_{1}$ does not contain any generalized critical point of $f$ in the tangential sense
except $z$. Let us now show that such an estimate is also valid near $\pa\Gamma_{1}$.
In order to prove it, one recalls
that 
$$
f=f(z)+f_{+}-f_{-} \text{ a.e  on } \Omega \text{ and }
|\nabla f|^2=|\nabla f_{-}|^2
+|\nabla f_{+}|^2  \text{ a.e \ near } \pa\Gamma_{1},
$$
where $f_{\pm} $ are smooth and satisfy the following relations
on $B_{z}$:
$$
f_{+}= f -f(z), \
f_{-}= 0, \
\pa_{n} f_{+} = 0
\text{ and }
\pa_{n} f_{-} = -\pa_{n} f.
$$
Hence, for any $\chi$ supported in a sufficiently
small neighborhood of $\pa\Gamma_{1}$, 
one deduces from 
  the relation
$\mathcal{Q}^{M,(p)}_{-f_{-},h} ( \dot \Omega )(\chi u)\geq 0$,
the Green formula~\eqref{eq.compfor}, and the fact that
$\mathcal{L}_{-\nabla
  f_{-}}+\mathcal{L}_{-\nabla f_{-}}^{*}$
  is a $0^{\text{th}}$ order differential operator,
  that there exists $C_{1}>0$ independent of $u \in  D\left ( \mathcal Q^{M,(p)}_{f,h} ( \dot \Omega )\right)$ such that:
\begin{align*}
  h\left(\int_{\Gamma_{1}}-\int_{\Gamma_{2}}\right)\langle\chi u ,
  \chi u
\rangle_{ T^{*}_{\sigma}\dot\Omega}\pa_{n} f_{-}\, d \sigma 
 &\geq -h^{2}\left\| d\chi u\right\|^{2}_{L^{2} ( \dot \Omega )}-
h^{2}\left\| d^{*}\chi u\right\|^{2}_{L^{2} ( \dot \Omega )}\\
 &\quad -\left\|\left|\nabla f_{-}\right|\chi u\right\|^{2}_{L^{2} ( \dot \Omega )}
-C_{1}h\|\chi u\|^{2}_{L^{2} ( \dot \Omega )}.
\end{align*}
Using again the Green formula \eqref{eq.compfor}, 
the relation $f-f(z)=f_{+}-f_{-}$ with $\pa_{n}f_{+}=0$ on $\Gamma_{1}$, 
 and the fact that
$\mathcal{L}_{\nabla
  f}+\mathcal{L}_{\nabla f}^{*}$
  is a $0^{\text{th}}$ order differential operator then leads to the existence of
  $C_{2}>0$ independent of $u \in  D\left ( \mathcal Q^{M,(p)}_{f,h} ( \dot \Omega )\right)$ s.t.
$$\mathcal{Q}^{M,(p)}_{f,h} ( \dot \Omega )(\chi u) \geq 
\left\| \left|\nabla f_{+}\right|\chi u\right\|^{2}_{L^{2} ( \dot \Omega )}
-C_{2}h\|\chi u\|^{2}_{L^{2} ( \dot \Omega )}+h\int_{\Gamma_{2}}\langle\chi u|\chi u
\rangle_{ T^{*}_{\sigma}\dot\Omega}\pa_{n} f_{+}\, d \sigma.
$$
Since $f_{+}$ has no critical point around $\pa\Gamma_{1}$ (see~\eqref{eq:nablaf+}), one has then for $h$ small enough, taking maybe $C$ smaller:
\begin{equation} \label{eq.pa-n-f+}
\mathcal{Q}^{M,(p)}_{f,h} ( \dot \Omega )(\chi u)
\ge C\|\chi u\|^{2}_{L^{2} ( \dot \Omega )}+h\int_{\Gamma_{2}}\langle\chi u|\chi u
\rangle_{ T^{*}_{\sigma}\dot\Omega}\pa_{n} f_{+}\, d \sigma.
\end{equation}
Let us recall that due to our construction of $\Gamma_{2}$ near $\partial \Gamma_{1}$,
one has (see~\eqref{eq:dnf+positive} in Proposition~\ref{prop:dotomega}): 
$$
\pa_{n} f_{+}(\sigma)> 0\text{ for } \sigma\in\Gamma_{2}
\text{ sufficiently close to } \pa\Gamma_{1}.$$
% This is due to the fact that $\Gamm% a_{1}$ has been chosen strongly stable and $\partial_{n(\partial \Gamma_{1})}\partial f_{+}=\partial_{n(\partial \Gamma_{1})} f$ on $\partial \Gamma_{1}$ (where $n(\partial \Gamma_{1})$ is the outward normal of $\partial\Gamma_{1}$), together with a continuity argument when one constructed $\Gamma_{2}$ near $\partial \Gamma_{1}$ (while keeping an angle strictly smaller than $\pi$ between $\Gamma_{1}$ and $\Gamma_{2}$).
 This implies that for $\chi$ supported near $\pa\Gamma_{1}$ with sufficiently small support
and $h$ small enough:
$$\mathcal{Q}^{M,(p)}_{f,h} ( \dot \Omega )(\chi u)
\geq C\|\chi u\|^{2}_{L^{2} ( \dot \Omega )}.
$$
Lastly, since $z$ is a generalized critical point with index $1$ in the tangential sense, 
it follows from \cite[Proposition~4.3.2]{helffer-nier-06} that 
for $\chi$ supported in  a neighborhood of  $z$  and $h$ small enough, the spectrum of the Friedrichs extension associated with the quadratic form
$$
\left\{v\in \Lambda^{p} H^{1}(\supp\chi);\
  \mathbf{t}v|_{\Gamma_{1}}=v|_{\pa\supp\chi\setminus\overline\Gamma_{1}}=0\right\}\ni v \mapsto
\left\|d_{f,h} v\right\|_{ L^2}^{2}
 + \left\|d_{f,h}^{*} v\right\|_{ L^2}^{2}
$$
does not meet  $[0,h^{\frac32})$
if $p\neq 1$, and consists of exactly one eigenvalue in  $[0,h^{\frac32})$ which is actually exponentially small
-- i.e. of the size $O(e^{-\frac{C_{3}}{h}})$ --
 if $p=1$.
Denote by $\psi_{1}\in \Lambda^{1} H^{1}(\supp\chi)$ some normalized eigenvalue associated with this exponentially
small eigenvalue.
\medskip

\noindent
Using the IMS localization formula  (see for
example~\cite[Theorem 3.2]{cycon-froese-kirsch-simon-87})
\begin{align}
\nonumber
&\forall(\chi_{k})_{k\in\{1,\dots,K\}} \in  \left(C^{\infty}\left
  (\, \overline{\dot\Omega}\, \right)\right)^K 
\text{ s.t. } \sum_{k=1}^{K}\chi_{k}^{2}=1 \\
\label{eq.IMS}
&\mathcal{Q}^{M,(p)}_{f,h}\left (\dot \Omega\right) (u)
= \sum_{k=1}^{K}\left(\mathcal{Q}^{M,(p)}_{f,h} ( \dot \Omega )(\chi_{k}u)
-h^2 \left\| |\nabla \chi_k | u\right\|_{L^2 ( \dot \Omega )}^2\right),
\end{align}
 the previous analysis then shows that choosing $\chi_{1}\in C^{\infty}\left (\overline{\dot\Omega}\right)$ 
supported in  a neighborhood of~$z$ with $\chi_{1}=1$ near~$z$, one gets 
for some $C,C'>0$ independent of $u\in D\left ( \mathcal Q^{M,(p)}_{f,h} ( \dot \Omega )\right)$
and $h$ small enough:
\begin{equation}
\label{eq.IMS2}
\mathcal{Q}^{M,(p)}_{f,h} ( \dot \Omega )(u)
 \geq  \mathcal{Q}^{M,(p)}_{f,h} ( \dot \Omega )(\chi_{1}u) +
C\| (1-\chi_{1}^{2})^{\frac12} u\|_{L^{2} ( \dot \Omega )}^2
-C'h^2 \left\| u\right\|_{L^{2} ( \dot \Omega )}^2.
\end{equation}
If $p\neq 1$, one deduces immediately from \eqref{eq.IMS2} that for $h$ small enough,
$$
\mathcal{Q}^{M,(p)}_{f,h} ( \dot \Omega )(u)
 \geq  h^{\frac32} \left\| \chi_{1}u\right\|_{L^{2} ( \dot \Omega )}^2+
C\| (1-\chi_{1}^{2})^{\frac12} u\|_{L^{2} ( \dot \Omega )}^2
-C'h^2 \left\| u\right\|_{L^{2} ( \dot \Omega )}^2,
$$
and then that for some $c>0$ and $h$ small enough: 
$$
\mathcal{Q}^{M,(p)}_{f,h} ( \dot \Omega )(u)
\ \geq\  c h^{\frac32} \left\| u\right\|_{L^{2} ( \dot \Omega )}^2.
$$
If $p=1$, one obtains from \eqref{eq.IMS2} the same conclusion
for any $u$ such that $\int_{\dot \Omega} u\chi_{1}\psi_{1}=0$ and therefore
$\Delta_{f,h}^{M,(p)} ( \dot \Omega )$ has no eigenvalue in $[0,ch^{\frac32})$
if $p\neq1$ and at most one if $p=1$.
To end up the proof, it is sufficient to remark that $\tilde\psi_{1}$,
the extension of $\psi_{1}$ to $ \dot \Omega $ by $0$
outside $\supp\chi_{1}$, belongs to $ D\left (  \mathcal Q^{M,(1)}_{f,h} ( \dot \Omega )  \right)$
and satisfies for $h$ small:
$$
\mathcal{Q}^{M,(1)}_{f,h} ( \dot \Omega )(\tilde \psi_{1})= 
\|d_{f,h} \psi_{1}\|_{L^2 ( \dot \Omega )}
 + \|d_{f,h}^{*}  \psi_{1}\|_{L^2 ( \dot \Omega )} = O(e^{-\frac {C_{3}}{h}})\ <\ ch^{\frac32}.
$$
This proves that  there exists $c>0$ and $h_0>0$ such that for all  $p\in\{0,\dots,n\}$ and $h\in (0,h_0)$,
\begin{equation}\label{eq.firstii}
\dim\Ran \ \pi_{[0,ch^\frac32)}\left(\Delta_{f,h}^{M,(p)} ( \dot \Omega_i )\right)
=\delta_{1,p}.
\end{equation}
Then, the fact that there exists $c>0$ and $h_0>0$ such that  for all $h\in (0,h_0)$,
$$
\Sp\!\!\left(\Delta_{f,h}^{M,(1)} ( \dot \Omega_i )\right)\cap [0,ch^\frac32)
=
\{0\},
$$
is a direct consequence of item (ii) in Proposition~\ref{pr.DeltaTN} together with~\eqref{eq.firstii}. This concludes the proof of  Proposition~\ref{pr.DeltaTN}. 

\end{proof}
%\noindent We are also going to use in the sequel 
%the local positive Witten Laplacian  $\Delta_{f,h}^{TD}$
%defined as the 
%Friedrichs extension associated with the quadratic form
%$$
%\Lambda H^{1}_{TD}(V_{\Gamma_{1}})
%\ni \omega \mapsto
%\mathcal{Q}^{TD}_{f,h}(\Omega)=\left\|d_{f,h}\omega\right\|^{2}
% + \left\|d_{f,h}^{*}\omega\right\|^{2},
%$$
%where
%$$
%\Lambda H^{1}_{TD}(V_{\Gamma_{1}}):=\left\{u\in \Lambda H^{1}(V_{\Gamma_{1}});\
%  \mathbf{t}u|_{ \Gamma_{1}}=0,\quad u|_{\paV_{\Gamma_{1}}\setminus \Gamma_{1}}=0\right\}.
%$$
%We will actually just use the positivity of this operator.

Following Proposition~\ref{pr.DeltaTN}, let us introduce an
$L^2$-normalized eigenform $u^{(1)}_{h,i}$ of
$\Delta_{f,h}^{M,(1)}(\dot\Omega_i)$ associated with the eigenvalue
$0$:
\begin{equation} \label{eq.uh=0}
\Delta_{f,h}^{M,(1)} ( \dot \Omega_i ) \, u^{(1)}_{h,i}=0  \text{ in }
\dot\Omega_i \text{ and } \left\|u^{(1)}_{h,i} \right\|_{  L^2 ( \dot \Omega_i )}=1.
\end{equation}
Using standard elliptic
regularity results, one can check that
$u^{(1)}_{h,i}$ is actually in $C^\infty( \dot \Omega_i )$ and is smooth in a neighborhood of any regular point of $\pa \dot \Omega_i  $.  Notice that thanks to item $(iii)$ in
Proposition~\ref{pr.DeltaTN}, $u^{(1)}_{h,i}$ is unique up to a
multiplication by $\pm1$: this multiplicative constant will be fixed
in Proposition~\ref{apriori} below. The quasi-mode $\tilde{\phi}_i$ will be built using a suitable
truncation of $u^{(1)}_{h,i}$.

\subsection{Definition of the quasi-modes}\label{sec:contruct_quasimode}
Troughout this section, one assumes \textbf{[H1]}, \textbf{[H2]} and \textbf{[H3]}. In this section, we  construct the function $\tilde u$ and a family of
$1$-forms $(\tilde \phi_i)_{i=1,\ldots,n}$ which will satisfy the
estimates stated in Section~\ref{flat}. For each $i\in
\left\{1,\ldots,n\right\}$, the $1$-form $\tilde \phi_i$ will be
constructed by a suitable truncation of an eigenfunction $u_{h,i}$ associated with
the eigenvalue $0$ of the mixed Witten  Laplacian attached with $z_i\in
\{z_1,\ldots,z_n\}$, as defined in Section~\ref{sec:def_mixed}. \label{page.sigmai1}

% Let us recall that thhis construction is motivated by the works \cite{helffer-nier-06} and \cite{le-peutrec-10} where for $h$ small enough, the vector space $\Ran \ \pi_{[0,h^{\frac32} )} \left( \Delta^{(1)}_{f,h}(\Omega)\right)$ is the $n$ dimensional and can be well approximated by projecting 1 forms suitably constructed near each local minimum $z_i\in \partial \Omega$.

%In this section we assume that \textbf{[H1]}, \textbf{[H2]} and \textbf{[H3]} hold. 
 We recall that  $\Sigma_i $ is an open set included in $\partial \Omega$ containing $z_i$ which is such that $\overline\Sigma_i \subset B_{z_i}$ (see Definition~\ref{Bz}). 

\subsubsection{Definition of the quasi-mode $\tilde u$}\label{sec:utilde}

\begin{definition} \label{tildeu}  Let us consider the global minimum $x_0$ introduced in the hypothesis  \textbf{[H2]}.
Let $\chi \in C^{\infty}_c(\Omega)$ such that $\{x \in \Omega | \chi(x)=1\}$ is a neighborhood of $x_0$ and such that  $0\leq \chi\leq 1$ (in particular $\chi(x_0)=1$). The quasi-mode $\tilde u$ is defined by
\begin{equation*}\label{page.tildeu}
\tilde u := \frac{\chi}{ \sqrt{\displaystyle \int_{\Omega} \chi^2 e^{- \frac{2f}{h}}}}.
\end{equation*}
\end{definition}
\noindent
The function $\tilde u$ belongs to $C^{\infty}_c(\Omega)$ and therefore $\tilde u \in  H^1_0\left(e^{-\frac{2}{h} f(x)} dx\right)$.
The function $\chi$ will be chosen such that ${\rm supp}(|\nabla
\chi|)$ is as close as needed to $\partial \Omega$, as will be made
precise in~Section~\ref{sec:goodquasimodes}.

Let us first prove that $\tilde{u}$ satisfies item $2(b)$ in Proposition~\ref{ESTIME}.
\begin{lemma}\label{nablauu} 
 Let us assume that the hypotheses \textbf{[H1]} and \textbf{[H2]}  hold.
Then for any  $\delta>0$, there exist $h_0>0$, $C>0$, and  there exists $\chi\in C^{\infty}_c(\Omega)$ such that the set $\{x \in \Omega|\chi(x)=1\}$ is a neighborhood of $x_0$,  $0\leq \chi\leq 1$ and for all $h\in (0,h_0)$ 
$$\int_{\Omega} \left\vert  \nabla  \tilde u  \right\vert^2e^{-\frac{2f}{h}}     \leq C h^{-\frac d2}e^{-2\frac{(f(z_1)-f(x_0))-\delta }{h}},$$
where $\tilde u$ is defined in Definition~\ref{tildeu}.
\end{lemma}
\begin{proof}
There exists a constant $C$ such that
$$\int_{\Omega} \left\vert  \nabla  \tilde u(x)  \right\vert^2e^{-\frac{2f(x)}{h}}  dx   \leq     C \frac{ \int_{{\rm supp} \nabla \chi } e^{-2\frac{f(x)}{h}} dx}{\int_{\Omega}  \chi^2(y)  e^{-2\frac{f(y)}{h}}dy}.$$
Since $\rm{supp} \nabla \chi$ can be chosen arbitrarly close to $\partial \Omega$ and since $z_1$ is the minimum of $ V$ on $\partial \Omega$, by continuity of $f$, for any  $\delta>0$ there exists $\chi\in C^{\infty}(\Omega)$ such that $\{x \in \Omega|\chi(x)=1\}$ is a neighborhood of $x_0$,  $0\leq \chi\leq 1$ and 
$$\int_{{\rm supp} \nabla \chi } e^{-2\frac{f(x)}{h}} dx \leq C  e^{-2\frac{f(z_1)+2\delta }{h}}.$$
Moreover, since $x_0$ is the global minimum of $f$ in $\Omega$, one gets, using Laplace's method
$$\int_{\Omega}  \chi^2(y)  e^{-2\frac{f(y)}{h}} dy = \frac{  (\pi \ h)^{\frac{d}{2} } }{ \sqrt{  {\rm det \ Hess } f   (x_0)   }   }e^{-2\frac{f(x_0) }{h}}(1+O(h)).$$
%In conclusion, for any  $\delta>0$ there exists $C>0$ and $\chi\in
%C^{\infty}(\Omega)$ such that $\{x \in \Omega, \, \chi(x)=1\}$ is a neighborhood of $x_0$,  $0\leq \chi\leq 1$ and 
%$$\int_{\Omega} \left\vert  \nabla  \tilde u(x)  \right\vert^2e^{-\frac{2f}{h}}  dx   \leq C e^{-2\frac{(f(z_1)-f(x_0))-\delta }{h}},$$
%which is
This yields the desired estimate.
\end{proof}
\noindent
Notice that item 2(b) in Proposition~\ref{ESTIME} is a direct consequence of  Lemma~\ref{nablauu}.

\subsubsection{Definition of the quasi-mode $\tilde{\phi}_i$ attached to
  $z_i$}  \label{construction}

 Let $z_i$ be a local minimum of $f|_{\pa \Omega}$. Let us recall that $\Sigma_i$ is an open subset of $\pa \Omega$ such that $z_i\in \Sigma_i$ and $\overline{\Sigma_i  } \subset B_{z_i}$. 
Let $\mathcal S_{M,i}:=\{\dot\Omega_i,\Gamma_0,
\Gamma_{1,i},\Gamma_{2,i},V_{\Gamma_{1,i}}\}$ be an ensemble of sets
associated with $z_i$, see Definition~\ref{Sets}. Thanks to Propositions
\ref{stronglystableexis} and~\ref{prop:dotomega},  the set $\Gamma_{1,i}$ can be taken such
that 
\begin{equation}\label{eq.sigmai-gammai}
\overline\Sigma_i \subset\Gamma_{1,i}.
\end{equation} We recall that Section
\ref{geometry} was dedicated to the construction of a domain $\dot
\Omega_i\subset \Omega$ and a mixed Witten  Laplacian
$\Delta^{M,(1)}_{f,h} ( \dot \Omega_i )$ (see~\eqref{eq.DeltaTN})
associated with this ensemble of sets~$\mathcal S_{M,i}$. Proposition 
\ref{pr.DeltaTN} gives the spectral properties of the operator $\Delta^{M}_{f,h} ( \dot \Omega_i )$. In the following, we consider a normalized eigenform $u^{(1)}_{h,i}\in D\left(\Delta_{f,h}^{M,(1)} ( \dot \Omega_i )\right)$ associated with the first eigenvalue $0$, i.e. such that \eqref{eq.uh=0} holds

The quasi-mode $\tilde \phi_i$ is defined as the following truncation of $u^{(1)}_{h,i}$.
  \begin{definition} \label{tildephii}
   Let us assume that the hypotheses \textbf{[H1]}, \textbf{[H2]} and \textbf{[H3]} hold.
 Let $\chi_i \in C^{\infty}\left(\overline{ \Omega} \right)$ be such that: 
 \begin{enumerate}
 \item $\chi_i \in C^{\infty}_c\left( \dot \Omega_i \cup \Gamma_{1,i}\right)$ (and thus $\chi_i=0$ on a neighborhood of $\Gamma_{2,i}\cup\Gamma_0$ and  on a neighborhood of $\partial \Omega \setminus \Gamma_{1,i}$),
\item $\chi_i=1$ on a neighborhood of $\overline \Sigma_i$ in $\overline{\dot \Omega_i}$,
\item $0\leq\chi_i\leq 1$.
\end{enumerate}
One defines $\mathcal V_i:=\left\{x\in \Omega | \chi_i(x)=1
\right\} $. The   quasi-mode $\tilde \phi_i$ is defined on $\Omega$ by:
\begin{equation}
\label{eq:phi1} 
\tilde \phi_i :=  \frac{\chi_i u^{(1)}_{h,i}}{\sqrt{\displaystyle \int_{\Omega} \left\vert \chi_i(x) u^{(1)}_{h,i}(x) \right\vert^2 dx }}.
\end{equation}
\label{page.tildephii}
\end{definition}

 \begin{figure}
  \begin{center}
 \begin{tikzpicture}
\draw[->] (-6,0)--(6,0) node[right] {$\partial \Omega$} ;
\draw[->] (0,0)--(0,2);
\draw (0 ,0) node [anchor=north east] {$z_i$};
\draw (0 ,1) node [anchor=north east] {1};
\draw[dashed] (-2,-1)--(-2,1);
\draw[dashed] (2,-1)--(2,1);
%\draw[dashed] (-2.5,-2)--(-2.5,1);
%\draw[dashed] (2.5,-2)--(2.5,1);
\draw[dashed] (-4.5,-1.7)--(-4.5,0);
\draw[dashed] (4.5,-1.7)--(4.5,0);
\draw[dashed] (-5.5,-2.7)--(-5.5,0);
\draw[dashed] (5.5,-2.7)--(5.5,0);
\draw[<->](-2,-1)--(2,-1)  node[midway,below] {$\Sigma_i$} ;
%\draw[<->](-2.5,-2)--(2.5,-2)  node[midway,below] {$\Gamma_{St}$} ;
\draw[<->](-4.5,-1.7)--(4.5,-1.7)  node[midway,below] {$\Gamma_{1,i}$} ;
\draw[<->](-5.5,-2.7)--(5.5,-2.7)  node[midway,below] {$B_{z_i}$} ;
\draw (-4.1,0) ..controls (-3.5,0) and  (-3.5,1) .. (-2.5,1);
\draw (2.5,1)  ..controls (3.5,1)  and  (3.5,0) .. (4.1,0);
\draw (-2.5,1) -- (2.5,1) node[midway,above=1cm] {$\chi_i$};
\end{tikzpicture}
\caption{The support of $\chi_i$ on $\partial \Omega$.}
 \label{fig:chii}
  \end{center}
\end{figure}
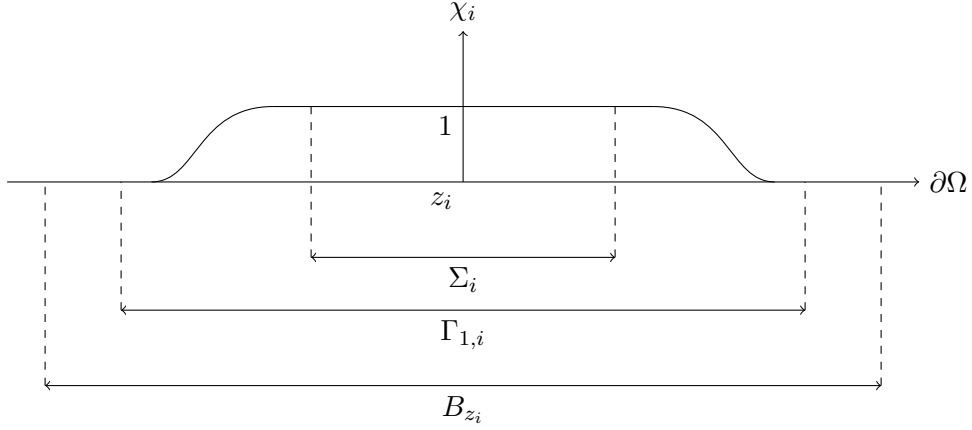

\begin{figure}[!h]
\begin{center}
\begin{tikzpicture}
\draw[very thick,red] (0,3) arc (220:310:0.7cm) ;
\draw[very thick,red] (0.6,-2.95) arc (45:140:0.4cm);
\draw[very thick, red] (1,2.9) ..controls (5,2.15) and (5.1,-2.36) .. (0.6,-2.95) node[midway,right=0.03cm]{$\Gamma_{2,i}$} ;
\draw (0,0) ellipse (4 and 3) ;
\draw[very thick] (4,2) node[midway,above=1cm]{$\{\chi_i=1\}=\mathcal V_i$} ;
\tikzstyle{vertex}=[draw,circle,fill=black,minimum size=6pt,inner sep=0pt]
\draw (0,3) node[vertex,label=above :{$  \ \overline \Gamma_{1,i} \cap \overline \Gamma_{2,i} $}] (v) {};
\draw (0,-3) node[vertex,label=below:{$  \ \overline \Gamma_{1,i}  \cap \overline \Gamma_{2,i} $}] (v) {};
\draw (0,0) node[vertex,label=below:{$x_0$}] (v) {};
\draw (0,0) circle (0.5cm)  node[midway,right=0.4cm]{$\Gamma_0$} ;
\draw[ultra thick] (-2.4,6) node[below=3.7cm]{$\Gamma_{1,i} $};
%\draw (0,-3.5) node[midway,below=4.2cm] {Domain $\Omega$};
\draw[double=gray,double distance=3pt] (-1.26, 2.8) ..controls (4.5,2.15) and (4.5,-2.15) .. (-1.25, -2.8);
\draw[ultra thick] (0,3) ..controls (-5.3,2.71) and (-5.3,-2.71) .. (0,-3) ;
\draw[|-|,ultra thick,blue] (-3.29,1.7) ..controls (-4.2,0.8) and (-4.2,-0.8) .. (-3.29,-1.7) node[midway,above right=0.4cm]{$\Sigma_i$} ;
\draw (-4,0) node[vertex,label=above left:{$z_i$}] (v) {};
\draw[double=gray,double distance=3pt] (0,0) circle (1cm) ;
\draw[<-](1.9,1.75)--(3,3) ;
\draw[<-](0.7,0.7)--(3,3);
\draw (3.3,3.1) node[above] {${\rm supp} \ \nabla \chi_i$} ;
\end{tikzpicture}
\caption{The set $\mathcal V_i=\left\{x\in \Omega \big |\chi_i(x)=1
  \right\}$ and, in gray, the support of $\nabla
  \chi_i$.}
 \label{fig:support}
 \end{center}
\end{figure}
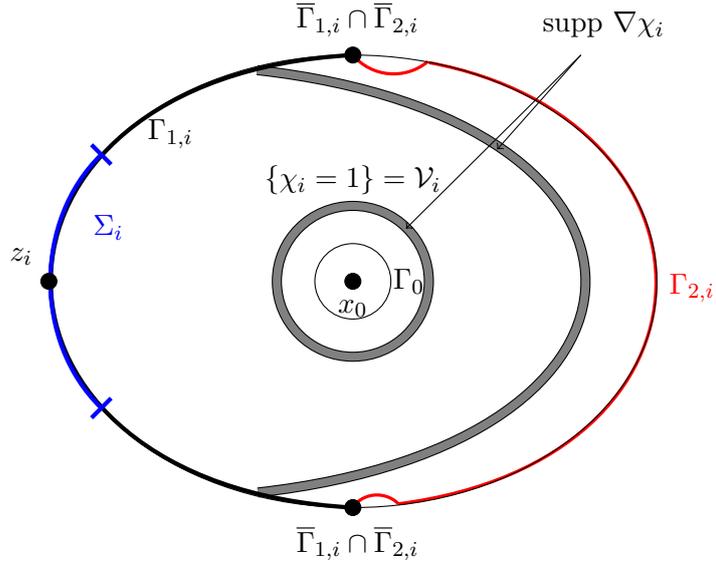

\noindent
The support of $\chi_i$ on $\partial \Omega$ is represented on Figure \ref{fig:chii} and the support of $\chi_i$ in $\Omega$ is represented in Figure \ref{fig:support}. 
The sets $\Gamma_{1,i}$, $\Gamma_{2,i}$ and, the function $\chi_i$ will be chosen such that   ${\rm supp}(|\nabla
\chi_i|)$ is as close as needed from $B_{z_i}^c \subset \partial
\Omega $ or from $x_0$, as will be made precise at the end of
Section~\ref{sec:goodquasimodes}. This is possible thanks to item 7 in Proposition~\ref{prop:dotomega}.
Using  Lemma~\ref{le.Gaffney} and the fact that ${\bf t}
\chi_i u_{h,i}^{(1)}=0$ on $\partial \Omega$, one easily shows that
$$\tilde{\phi}_i \in \Lambda^1 H^1_T(\Omega).$$ 
Using now the regularity of $u_{h,i}^{(1)}$, one can check that
$\tilde{\phi}_i$ is actually a $C^\infty_c(\Omega\cup \Gamma_{1,i})$ function.
\medskip

\noindent
We will show in Section~\ref{sec:goodquasimodes} that the family of
forms $(\tilde{u},\tilde{\phi}_1, \ldots ,\tilde{\phi}_n)$ satisfy the
estimates stated in Sections~\ref{sec221} and~\ref{flat}. This requires some preliminary
results on the eigenforms $(u^{(1)}_{h,i})_{i \in \{1, \ldots n\}}$ that are
provided in Section~\ref{sec:agmonestimate} and~\ref{sec:wkb}.

\subsection{Agmon estimates on $u^{(1)}_{h,i}$} \label{sec:agmonestimate}
Throughout this section, one assumes \textbf{[H1]}, \textbf{[H2]} and \textbf{[H3]}. In all this section, we consider, for a fixed critical point $z_i$, an
ensemble of sets~$\mathcal S_{M,i}$ associated with $z_i$ (see Definition~\ref{Sets}) and
an $L^2$-normalized eigenform 
$u^{(1)}_{h,i}$ of $\Delta^{M,(1)}_{f,h}(\dot \Omega_i)$ associated
with the eingevalue $0$, as introduced at the end of Section~\ref{sec:def_mixed}. 

 The aim of this section is to prove the
following proposition.
\begin{proposition}
\label{pr.Agmon}
 Let us assume that the hypotheses \textbf{[H1]}, \textbf{[H2]} and \textbf{[H3]} hold.
Any $L^2$-normalized
eigenform $u^{(1)}_{h,i}$ of 
$ \Delta_{f,h}^{M,(1)} ( \dot \Omega_i )$ 
associated with the eigenvalue $0$ satisfies:
\begin{equation} \label{eq.Agmon}
\exists  N \in\mathbb N, \,
\Big\|e^{\frac{\Psi_i}{ h}} u^{(1)}_{h,i}\Big\|_{  L^{2} ( \dot \Omega_i )}
+\Big\|d \Big (e^{\frac{ \Psi_i}{ h}} u^{(1)}_{h,i}\Big) \Big\|_{  L^{2} ( \dot \Omega_i )}+\Big\|d^*\Big (e^{\frac{ \Psi_i}{ h}} u^{(1)}_{h,i}\Big) \Big\|_{  L^{2} ( \dot \Omega_i )} 
=O(h^{-N})
\end{equation}
where, we recall, $\Psi_i(x)=d_a(x,z_i)$ (see Definition~\ref{fplus}).
\end{proposition} 
For the ease of notation, we drop the
subscript $i$ in the remaining of this section.

The proof is inspired by the first part of the proof of 
 \cite[Proposition~4.3.2]{helffer-nier-06}
where the authors consider a Witten Laplacian
with mixed tangential -- full Dirichlet boundary conditions
in a local system of coordinates in a neighborhood of $z$. The proof actually
only requires that  $u^{(1)}_{h}$ is an eigenform of 
$ \Delta_{f,h}^{M,(1)} ( \dot \Omega )$ 
associated with an eigenvalue $\lambda_{h}=O (h)$.
It crucially relies on the following Agmon-type energy equality.
\begin{lemma}
 \label{le.intbypart}
  Let us assume that the hypotheses \textbf{[H1]}, \textbf{[H2]} and \textbf{[H3]} hold.
Let $\varphi$ be a real-valued Lipschitz
function on $\overline{\dot\Omega}$. Then, for any $ u\in
D\left ( \mathcal Q_{f,h}^{M,(1)} ( \dot \Omega )\right)$, one has: 
\begin{equation}
  \label{eq.intbypartphi}
\begin{aligned}
\mathcal{Q}^{M,(1)}_{f,h} ( \dot \Omega )( u,e^{2\frac{\varphi}{h}} u)&=
 h^{2}\left\| d e^{\frac{\varphi}{h}} u\right\|^{2}_{ L^{2} ( \dot \Omega )}+
h^{2}\left\| d^{*} e^{\frac{\varphi}{h}} u\right\|^{2}_{ L^{2} ( \dot \Omega )}\\
&\quad + \langle
(|\nabla f|^{2}-|\nabla \varphi|^{2}+h\mathcal{L}_{\nabla
  f}+h\mathcal{L}_{\nabla
  f}^{*})e^{\frac{\varphi}{h}} u, e^{\frac{\varphi}{h}}
 u\rangle_{ L^{2} ( \dot \Omega )}\\
&\quad +h\left(-\int_{\Gamma_{0}\cup \Gamma_{1}}
+\int_{\Gamma_{2}}\right)
\langle u, u
\rangle_{ T_{\sigma}^{*}\dot\Omega}\;e^{\frac{2 }{h}\varphi }\pa_{n}
f\, d\sigma.
\end{aligned}
\end{equation}
Moreover, when $ u\in D\left (\Delta_{f,h}^{M,(1)} ( \dot \Omega )\right)$, the left-hand
 side equals $\langle
 e^{2\frac{\varphi}{h}}\Delta_{f,h}^{M,(1)} ( \dot \Omega ) u, u\rangle_{L^{2} ( \dot \Omega )}$.   
\end{lemma}

\begin{proof}
This result is standard for manifolds without boundary or for bounded
manifolds and quadratic forms
with full normal or tangential boundary conditions (see e.g.
\cite{dimassi-sjostrand-99,helffer-nier-06,le-peutrec-10}). We extend
it here to our setting.

 Note first that $ u\in
  D\left (  {\mathcal Q}^{M,(1)}_{f,h} ( \dot \Omega )  \right)$ implies $e^{2\frac{\varphi}{h}} u\in
  D\left (  {\mathcal Q}^{M,(1)}_{f,h} ( \dot \Omega )  \right)$, since  for $u\in D\left ( \mathcal Q^{M,(1)}_{f,h} ( \dot \Omega )\right )$,
 $n^{\flat}\wedge e^{2\frac \varphi h} u=e^{2\frac \varphi h} n^{\flat}\wedge u$ and 
 $\mathbf{i}_{n} e^{2\frac \varphi h}u=e^{2\frac \varphi h}
\mathbf{i}_{n} u$. One then gets by straightforward computations:
\begin{align*}
 \mathcal Q^{M,(1)}_{f,h} ( \dot \Omega ) ( u,e^{2\frac{\varphi}{h}} u)
&=
\langle d_{f,h} u,
d_{f,h}(e^{2\frac{\varphi}{h}} u)\rangle
+
\langle d_{f,h}^{*} u,
d_{f,h}^{*}(e^{2\frac{\varphi}{h}} u)\rangle
\\
&=
\langle e^{\frac{\varphi}{h}}d_{f,h} u,
d_{f,h}(e^{\frac{\varphi}{h}} u)\rangle
+ 
 \langle e^{\frac{\varphi}{h}}d_{f,h} u,
d\varphi\wedge(e^{\frac{\varphi}{h}} u)\rangle
\\
&\quad +
\langle e^{\frac{\varphi}{h}}d_{f,h}^{*} u,
d_{f,h}^{*}(e^{\frac{\varphi}{h}} u)\rangle
-
 \langle e^{\frac{\varphi}{h}}d_{f,h}^{*} u,
\mathbf{i}_{\nabla \varphi}(e^{\frac{\varphi}{h}} u)\rangle 
\\
&= \langle d_{f,h}(e^{\frac{\varphi}{h}} u),
d\varphi\wedge(e^{\frac{\varphi}{h}} u)\rangle
- \langle d\varphi\wedge(e^{\frac{\varphi}{h}} u),
d_{f,h}(e^{\frac{\varphi}{h}} u)\rangle
\\
&\quad + 
\| d_{f,h}(e^{\frac{\varphi}{h}} u)\|^{2}
-\| d \varphi\wedge (e^{\frac{\varphi}{h}} u)\|^{2}
 +
\| d_{f,h}^{*}(e^{\frac{\varphi}{h}} u)\|^{2}
-  \| \mathbf i_{\nabla \varphi} e^{\frac{\varphi}{h}} u\|^{2}
\\
&\quad +
\langle\mathbf{i}_{\nabla \varphi}(e^{\frac{\varphi}{h}} u),
d_{f,h}^{*}(e^{\frac{\varphi}{h}} u)\rangle
 -
\langle d_{f,h}^{*}(e^{\frac{\varphi}{h}} u),
\mathbf{i}_{\nabla \varphi}(e^{\frac{\varphi}{h}} u)\rangle 
\; .
\end{align*}
Let us set $\tilde{ u}:=e^{\frac{\varphi}{h}} u\in
 D\left (  Q^{M,(1)}_{f,h} ( \dot \Omega )  \right)$. The formulas stated in~\eqref{eq.relationII} lead to:
\begin{align*}
 \mathcal Q^{M,(1)}_{f,h} ( \dot \Omega ) ( u,e^{2\frac{\varphi}{h}} u)
&=
 \mathcal Q^{M,(1)}_{f,h} ( \dot \Omega ) (\tilde u) -
\langle
\left|\nabla\varphi\right|^{2}\tilde u,\tilde u\rangle -
\langle d\varphi\wedge \tilde u, d_{f,h}\tilde u\rangle
\\
&\quad+\langle d_{f,h}\tilde u, d\varphi\wedge \tilde u\rangle + 
\langle \mathbf{i}_{\nabla
\varphi}\tilde u,d_{f,h}^{*}\tilde u\rangle
-
\langle d_{f,h}^{*}\tilde u, \mathbf{i}_{\nabla \varphi}
\tilde  u\rangle
\end{align*}
and hence to
$$
 \mathcal Q^{M,(1)}_{f,h} ( \dot \Omega ) ( u,e^{2\frac{\varphi}{h}} u)
=
 \mathcal Q^{M,(1)}_{f,h} ( \dot \Omega ) (\tilde u) -
\langle
\left|\nabla\varphi\right|^{2}\tilde u,\tilde u\rangle\;.
$$
One concludes by applying Lemma~\ref{le.GreenWeak}. 
\end{proof}
\noindent
We are now in position to prove Proposition~\ref{pr.Agmon}. 
%\begin{remark}
%\label{re.intbypart}
%Note that one gets a similar formula for $\Delta_{f,h}^{TD}$
%on $V_{\Gamma_{1}}$ with, owing to the full Dirichlet boundary condition on
%$\paV_{\Gamma_{1}}\setminus\Gamma_{1}$, only a trace term on $\Gamma_{1}$. 
%\end{remark}
\begin{proof} (of Proposition~\ref{pr.Agmon})\\
Following the proof of \cite[Proposition~4.3.2]{helffer-nier-06}, one
proves the result in two steps. First, the
Agmon estimate along $\Gamma_{1}\subset\pa\Omega$ is proven by applying
Lemma~\ref{le.intbypart} with a function $\varphi$ close
 to $f_{+}$ (recall that on $\Gamma_1$, $\Psi=f_+$). The Agmon estimate in $\dot\Omega$ is
then obtained using again Lemma~\ref{le.intbypart} with $\varphi$
close to $\Psi$, and the Agmon estimate along $\Gamma_{1}$.

\begin{figure}[!h]

\begin{center}

\begin{tikzpicture}

\tikzstyle{vertex}=[draw,circle,fill=black,minimum size=6pt,inner sep=0pt]

\draw[->] (0,-5.5)--(0,6);

\draw (0,6.3) node {$\partial \Omega$};

\draw[->] (-0,4.8)--(4,4.8);

\draw (4.6,4.8) node {$-x_d$};

\draw[very thick] (0,-3.5)--(0,3.5);

\draw[very thick] (0,3.5)--(3.5,3.5);

\draw[very thick] (0,-3.5)--(3.5,-3.5);

\draw (3.9,3.5) node {$\Gamma_{2}$};

\draw (3.9,-3.5) node {$\Gamma_{2}$};

\draw[<->,dashed] (-3.5,-3.5)--(-3.5,3.5);

\draw[dashed] (-3.5,-3.5)--(0,-3.5);

\draw[dashed] (-3.5,3.5)--(0,3.5);

\draw (-3.8,0) node {$\Gamma_{1}$};

\draw[<->,dashed] (-2,-2.5)--(-2,2.5);

\draw[dashed] (-2,-2.5)--(0,-2.5);

\draw[dashed] (-2,2.5)--(0,2.5);

\draw (-2.4,0) node {$\Gamma_{St}'$};

\draw (1.4,-2.5)--(1.4,2.5);

\draw (1.4,2.5)--(0,2.5);

\draw (1.4,-2.5)--(0,-2.5);

\draw (1,2) node {$V_{\Gamma_{St}'}$};

\draw (2,-2.8)--(2,2.8);

\draw (2,2.8)--(0,2.8);

\draw (2,-2.8)--(0,-2.8);

\draw (1.7,-2.2) node {$V_{\eta}$};

\draw[<->,dashed] (-1,-1.5)--(-1,1.5);

\draw[dashed] (-1,-1.5)--(0,-1.5);

\draw[dashed] (-1,1.5)--(0,1.5);

\draw (-1.4,0) node {$\Gamma_{St}$};

\draw (0.8,-1.5)--(0.8,1.5);

\draw (0.8,1.5)--(0,1.5);

\draw (0.8,-1.5)--(0,-1.5);

\draw (0.4,1) node {$V_{\Gamma_{St}}$};

\draw (0,0) node[vertex](v){};

\draw (0.3,0) node {$z$};

\draw [] (0,-4) ..controls (4.2,-4) and (4.2,4) .. (0,4);

\draw (2.7,0) node {$V_{\Gamma_1}$};

\draw [] (0,-4.5) ..controls (5.1,-4.8) and (5.1,4.8) .. (0,4.5);

\draw (3.5,0) node {$V_{\Gamma_1}'$};

\end{tikzpicture}

\caption{Neighborhoods of $z$.}

\label{fig:representation_des_gamma}

\end{center}

\end{figure}
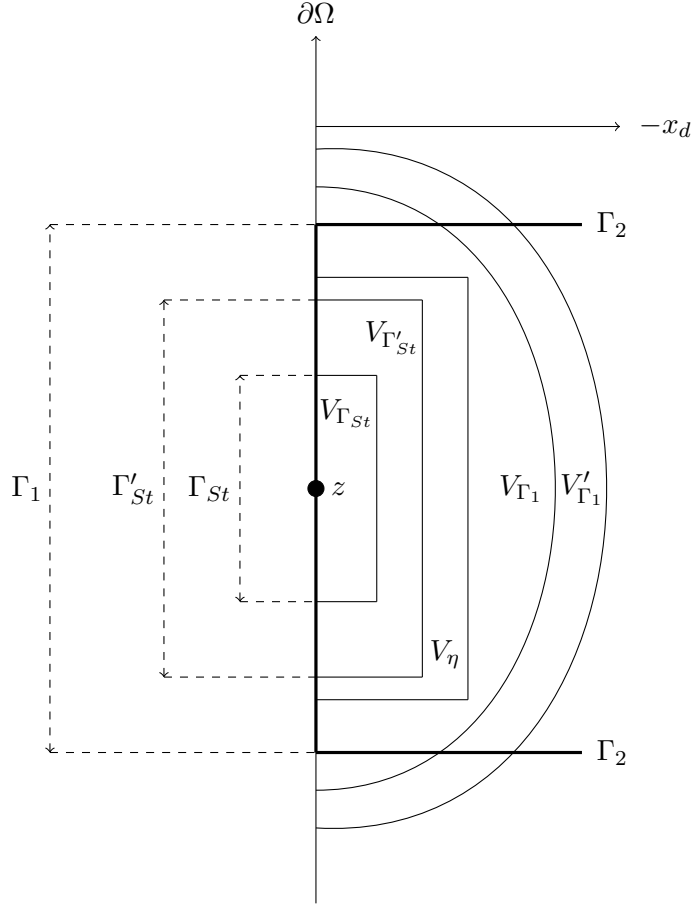

\noindent
In order to separate the analysis along $\Gamma_{1}$ and elsewhere,
one introduces two smooth cut-off functions
$\chi_{0}$ and $\chi_{1}$ on $\overline{\Omega}$
such that: 
$$\chi_1:= \sqrt{1-\chi_{0}^2}\ , \quad \chi_{0}=1
\text{ on } \overline{V_{\Gamma_{1}}} \text{ with }
\supp \chi_{0}\subset V_{\Gamma_{1}}',$$
for a set $V_{\Gamma_{1}}'\subset \overline{\Omega}$ such that for some
$\varepsilon>0$, (see Figure~\ref{fig:representation_des_gamma})
\begin{itemize}
\item[(i)] $(V_{\Gamma_{1}}+B(0,\varepsilon))\cap \overline{\Omega} \subset V_{\Gamma_{1}}'$,
\item[(ii)] $\Gamma_{1}':=V_{\Gamma_1}'\cap \pa\Omega$ is smooth
and $(\Gamma_{1} + B(0,\varepsilon)) \cap \partial \Omega
\subset \Gamma_{1}'$ and $(\Gamma_{1}' + B(0,\varepsilon)) \cap \partial \Omega  \subset B_{z}$,
\item[(iii)] $ \Psi=d_a(z,\cdot)$ is a smooth solution to the following  eikonal equation in $V_{\Gamma_{1}}'$ (see Proposition~\ref{TTH}):
$$\left\{\begin{aligned}
|\nabla\Psi |^2= |\nabla f |^2&\text{ in } V_{\Gamma_{1}}'\\
 \Psi = f-f(z) &\text{ on } \Gamma_{1}'\\
\pa_{n} \Psi = -\pa_{n}f &\text{ on } \Gamma_{1}'\ .
\end{aligned}\right.
$$
\end{itemize}
It is possible to choose $V_{\Gamma_{1}}'$ such that all the
properties stated previously on $V_{\Gamma_1}$ also hold on
$V_{\Gamma_{1}}'$ (in particular~\eqref{eq:nablaf-}, \eqref{eq:f+f-}  and the properties stated in Proposition~\ref{prop:dotomega}).
We recall that one has by~\eqref{eq:eqq}:
$$
 |\nabla\Psi|^2
\leq |\nabla f|^2\quad \text{a.e. in } \Omega.$$
Thus 
$$|\nabla f_{\pm} \vert
= \left|\nabla \left(\frac{ \Psi\pm (f-f(z))}2\right)\right|
\leq |\nabla f| \text{ a.e. in } \Omega.$$
Thanks to the relations
$f-f(z)=f_{+}-f_{-}$ and
 $ \Psi=f_{+}+f_{-}$, together with the equality  $|\nabla\Psi |^2=
 |\nabla f |^2$ a.e in  $V_{\Gamma_{1}}'$, one has
$$ \nabla f_{-} \cdot\nabla f_{+} =0 \quad\text{a.e \ in}\ V_{\Gamma_{1}}', 
\  |\nabla  \Psi|^2 = |\nabla f|^2
= |\nabla f_{+} |^2+ |\nabla f_{-} |^2
\quad\text{a.e \ in}\ V_{\Gamma_{1}}'.
$$
Let now $u^{(1)}_{h}\in D\left (\Delta_{f,h}^{M,(1)} ( \dot \Omega )\right )$
satisfy
$$
\Delta_{f,h}^{M,(1)} ( \dot \Omega )\, u^{(1)}_{h}=0 \text{ and }  \left\|u^{(1)}_{h}\right\|_{L^{2} ( \dot \Omega )}=1.
$$

\noindent
\underline{Step 1:} Agmon estimate in $\Gamma_1$.
\medskip

\noindent
In this part, we are going to prove the estimate~\eqref{eq.Agmon} 
with $ \Psi$ replaced by $f_{+}$ namely:
$$  \| e^{\frac{f_{+}}{h}}u^{(1)}_{h} \|_{  L^{2} ( \dot \Omega )}
+\|d(e^{\frac{f_{+}}{h}}u^{(1)}_{h})\|_{  L^{2} ( \dot \Omega
  )}+\|d^*(e^{\frac{f_{+}}{h}}u^{(1)}_{h})\|_{  L^{2} ( \dot \Omega )}
= O(h^{-N_{0}})$$
for some integer $N_0$. By the trace result~\eqref{eq:subelliptic},
this will give the estimate $$  \left\|e^{\frac{f-f(z)}{h}}u^{(1)}_{h}
  \right\|_{ L^{2}(\Gamma_{1})}= O(h^{-N_{0}}),$$ which is the first
  step to prove~\eqref{eq.Agmon}.

To get these results, the idea is to apply
Lemma~\ref{le.intbypart} to a convenient $\varphi$ comparable with $f_{+}$
and such that $|\nabla \varphi|\leq |\nabla f_{+}|$.
This kind of estimate is classic and the ideas behind the computations presented below,
which follow \cite{helffer-nier-06,le-peutrec-10},
originate from the article \cite{helffer-sjostrand-84} where similar estimates were obtained in the case of manifolds without boundary.  The presence of a boundary leads here to some technicalities
which can somehow hide the reasoning and  we refer for example to \cite{helffer-88,dimassi-sjostrand-99} for 
a presentation of semi-classical Agmon estimates
in manifolds without boundary.  We recall from the work~\cite{helffer-sjostrand-84} that
if one just wants to get an error of the size 
$O(e^{\frac{\varepsilon}{h}})$ 
with $\varepsilon>0$ arbitrarily small, the choice
$\varphi=(1-\varepsilon)f_{+}$ is sufficient, but it does not yield an
error of the size $O(h^{-N})$. To get such an error term, a good choice for $\varphi$ is~the following.
\medskip

\noindent
Let $\varphi:\dot \Omega \to \R$ be the following Lipschitz function: 
\begin{equation}\label{eq:choicephi}
\varphi=
\left\{
  \begin{aligned}[c]
f_{+}-Ch\ln \frac{f_{+}}{h} \  &\text{if}\ f_{+} > Ch,\\
f_{+}-Ch\ln C\ &\text{if}\ f_{+} \leq Ch ,
  \end{aligned}
\right.
\end{equation}
for some constant $C> 1$ that will be fixed at the end of this step.
Define the   level sets:
\begin{eqnarray*}
 && \Omega_{-}=\left\{x\in \dot\Omega\ \text{s.t.}\  f_{+}(x) \leq Ch\right\}
 \quad
\text{and}\quad
\Omega_{+}=\dot\Omega\setminus \Omega_{-}.
\end{eqnarray*}
Then $\nabla  \varphi = \nabla f_{+} $
a.e. in $\Omega_{-}$ and
$$
\nabla \varphi=\nabla f_{+}\left(1-\frac{Ch}{f_{+}}\right)
\quad \text{a.e. in } \Omega_{+}.
$$
This implies in particular the two following inequalities
valid a.e. on $\Omega_{+}$ and that will be used in the sequel:
\begin{align}
|\nabla f_{+}|^2-
|\nabla \varphi |^{2}&=
|\nabla f_{+}|^2\left(\frac{2Ch}{f_{+}}-\frac{C^{2}h^{2}}{f_{+}^{2}}\right)
\ \geq\  Ch \frac{|\nabla f_{+}|^{2}}{f_{+}}
\ \text{on $\Omega_{+}$}\label{eq.diff.f+/Ag}
\\
 |\nabla f|^2-|\nabla \varphi |^{2}
&\geq
|\nabla f|^2-|\nabla f_{+}|^2\left(1-\frac{Ch}{f_{+}}\right)
\geq Ch \frac{|\nabla f |^{2}}{f_{+}}
\ \text{on $\Omega_{+}$.}\label{eq.diff.f/Ag}
\end{align}
The last inequality in~\eqref{eq.diff.f/Ag} is a consequence of the
inequality $ |\nabla f_{+}|^2 \leq |\nabla f|^2$. This implies in particular that
$|\nabla \varphi| \le |\nabla f|$ a.e. in $\dot\Omega$.

Note lastly that there exists a constant $K>0$ depending on $f_{+}$ and $f$, 
such that
\begin{equation}
\label{eq.lowboundnablaf}
\frac{|\nabla f |^{2}}{f_{+}}\geq
K
\text{ in } \dot\Omega \text{  and  }
 \frac{|\nabla f_{+}|^{2}}{f_{+}}
\geq
K
\text{ in }V_{\Gamma_{1}}'\, ,
\end{equation}
the last inequality being a consequence of the facts that
$f_{+}(x',x_{d})=f_{+}(x',0)$ and 
$x' \mapsto f(x',0)=f(z)+f_{+}(x',0)$ is a Morse function with $z$ as only critical point.
% On a que f_+(x',x_d)=f|_{\partial \Omega} (x')
% Or $f|_{\partial \Omega}$ est Morse, donc $|\nabla_T f|^2/f \ge K$.

Now, using the fact that $\Delta^{M,(1)}_{f,h}(\dot
\Omega) \, u^{(1)}_h=0$ and the IMS localisation formula~\eqref{eq.IMS},
one gets 
$$
\begin{aligned}
0&=\mathcal{Q}^{M,(1)}_{f,h} ( \dot \Omega )(u^{(1)}_{h},e^{2\frac{\varphi}{h}}u^{(1)}_{h})\\
&= \sum_{k\in\{0,1\}} \left[
\mathcal{Q}^{M,(1)}_{f,h} ( \dot \Omega )(\chi_{k}u^{(1)}_{h},e^{2\frac{\varphi}{h}}\chi_{k}u^{(1)}_{h})
-h^2 \left\| |\nabla \chi_k |
  e^{\frac{\varphi}{h}}u^{(1)}_{h}\right\|_{ L^2 ( \dot \Omega )}^2
\right].
\end{aligned}
$$
Setting $$\tilde{u}^{(1)}_{h}:=e^{\frac{\varphi}{h}}u^{(1)}_{h},$$ and applying \eqref{eq.intbypartphi}
to $\chi_k u^{(1)}_{h}$, $k\in \{0,1\}$, one obtains:
\begin{align*}
 &C_{1}h^2\sum_{k\in\{0,1\}}\left\| \chi_{k}\tilde{u}^{(1)}_{h} \right\|_{ L^2 ( \dot \Omega )}^{2}=C_{1}h^2\left\| \tilde{u}^{(1)}_{h} \right\|_{ L^2 ( \dot \Omega )}^{2}\\
&\geq \sum_{k\in\{0,1\}} \Big[
 \left\| hd\chi_{k}\tilde{u}^{(1)}_{h} \right\|^{2}_{L^2 ( \dot \Omega )}
+\left\|hd^{*}\chi_{k}\tilde{u}^{(1)}_{h}\right\|^{2}_{ L^2 ( \dot \Omega )}
+\ 
\langle(|\nabla f|^{2}-|\nabla \varphi|^{2})\chi_{k} \tilde{u}^{(1)}_{h},\chi_{k} \tilde{u}^{(1)}_{h}
  \rangle _{ L^2 ( \dot \Omega )} \\
&\quad+\  h\langle(\mathcal{L}_{\nabla
  f}+\mathcal{L}_{\nabla f}^{*})\chi_{k} \tilde{u}^{(1)}_{h},\chi_{k} \tilde{u}^{(1)}_{h}
  \rangle_{ L^2 ( \dot \Omega )} \Big]
  +\ h\left(\int_{\Gamma_{2}}-\int_{\Gamma_{1}}\right)\langle\chi_{0}\tilde{u}^{(1)}_{h},\chi_{0}\tilde{u}^{(1)}_{h}
\rangle\pa_{n} f\, d\sigma,
\end{align*}
where $C_{1}=\max(\|\nabla \chi_{0}\|^2_{\infty},\|\nabla \chi_{1}\|^2_{\infty})$.  Note that one has used that  $\chi_0=0$ on $\Gamma_{0}$, $\chi_1=0$ on $\Gamma_{1}$ and 
$$\left(-\int_{\Gamma_{0}}+\int_{\Gamma_{2}}\right)\langle\chi_{1}\tilde{u}^{(1)}_{h},\chi_{1}\tilde{u}^{(1)}_{h} \rangle\pa_{n} f\, d\sigma\geq 0,$$
which follows from $\pa_{n}f>0$
on $\Gamma_{2}$
and from $\pa_{n}f<0$ on $\Gamma_{0} $.

Now, since
$\mathcal{L}_{\nabla
  f}+\mathcal{L}_{\nabla f}^{*}$
  is a $0^{\text{th}}$ order differential operator,  
and $\left|\tilde{u}^{(1)}_{h}(x)\right|\leq e^C |u^{(1)}_{h}(x)|$ a.e. on $ \Omega_{-}$,
one obtains that for some constants $C_{2}$ (independent of
$C$) and $C_{3}(C)$ depending 
on $C$,
\begin{equation}\label{eq.sum}
\begin{aligned}
  C_{3}(C)h&\geq 
  \sum_{k\in\{0,1\}}\Big[
  \left\|hd\chi_{k}\tilde{u}^{(1)}_{h}\right\|^{2}_{ L^2 ( \dot \Omega )}
+\left\|hd^{*}\chi_{k}\tilde{u}^{(1)}_{h}\right\|^{2}_{ L^2 ( \dot \Omega )}+ \langle(|\nabla f|^{2}-|\nabla \varphi|^{2} )\chi_{k} \tilde{u}^{(1)}_{h},\chi_{k} \tilde{u}^{(1)}_{h}
  \rangle _{ L^2 ( \dot \Omega )}\\
&\quad -C_{2}h\left\|\chi_{k} \tilde{u}^{(1)}_{h}\right\|^2_{ L^{2}(\Omega_{+})} \Big]
+h\left(\int_{\Gamma_{2}}-\int_{\Gamma_{1}}\right)\langle\chi_{0}\tilde{u}^{(1)}_{h},\chi_{0}\tilde{u}^{(1)}_{h}
\rangle\pa_{n} f \, d\sigma.
\end{aligned}
\end{equation}
\noindent
Let us first consider the case $k=1$. Using $|\nabla \varphi|\leq |\nabla f|$ and  
\eqref{eq.diff.f/Ag}--\eqref{eq.lowboundnablaf}, one gets:
\begin{align}
&\langle(|\nabla f|^{2}-|\nabla \varphi|^{2} )\chi_{1} \tilde{u}^{(1)}_{h},\chi_{1} \tilde{u}^{(1)}_{h}
  \rangle_{ L^{2} ( \dot \Omega )}
-C_{2}h\left\|\chi_{1} \tilde{u}^{(1)}_{h}\right\|^2_{ L^{2}(\Omega_{+})}
\nonumber
\\
&\ge
\langle(|\nabla f|^{2}-|\nabla \varphi|^{2} )\chi_{1} \tilde{u}^{(1)}_{h},\chi_{1} \tilde{u}^{(1)}_{h}
  \rangle_{ L^{2}(\Omega_+)}
-C_{2}h\left\|\chi_{1} \tilde{u}^{(1)}_{h}\right\|^2_{ L^{2}(\Omega_{+})}
\nonumber\\
&\geq
\left\langle \left(Ch \frac{|\nabla f|^2}{f_{+}} -C_{2}h\right) \chi_{1} \tilde{u}^{(1)}_{h},\chi_{1} \tilde{u}^{(1)}_{h}
  \right\rangle_{ L^{2}(\Omega_{+})}
\nonumber\\
& \geq 
(KC-C_{2})h \left\|\chi_{1} \tilde{u}^{(1)}_{h}
  \right\|^2_{ L^{2}(\Omega_{+})}
.    \label{eq.k=n}
\end{align}
\noindent
Let us then consider the case $k=0$. In this case, one deduces from $\supp \chi_{0}\subset V_{\Gamma_{1}}'$ where $|\nabla f|^2=|\nabla f_{+}|^2
+|\nabla f_{-}|^2$, from $|\nabla \varphi|^2=|\nabla f_{+}|^2$ on~$\Omega_-$, and from \eqref{eq.diff.f+/Ag}--\eqref{eq.lowboundnablaf} the inequality:
\begin{align}
&\langle(|\nabla f|^{2}-|\nabla \varphi|^{2} )\chi_{0} \tilde{u}^{(1)}_{h},\chi_{0} \tilde{u}^{(1)}_{h}
  \rangle_{ L^{2} ( \dot \Omega )}
-C_{2}h\left\|\chi_{0} \tilde{u}^{(1)}_{h}\right\|^2_{ L^{2}(\Omega_{+})}\nonumber\\
&=
\left\| |\nabla f_{-}|\chi_{0} \tilde{u}^{(1)}_{h}   \right\|^2_{ L^{2} ( \dot \Omega )}
+\langle(\underbrace{|\nabla f_{+}|^2-|\nabla \varphi|^{2} }_{=0})\chi_{0} \tilde{u}^{(1)}_{h},\chi_{0} \tilde{u}^{(1)}_{h}
  \rangle_{ L^{2}(\Omega_{-})}\nonumber\\
&\quad + \langle(|\nabla f_{+}|^2-|\nabla \varphi|^{2} -C_{2}h )\chi_{0} \tilde{u}^{(1)}_{h},\chi_{0} \tilde{u}^{(1)}_{h}
  \rangle_{ L^{2}(\Omega_{+})}
\nonumber\\
&\geq
\left\| |\nabla f_{-}|\chi_{0} \tilde{u}^{(1)}_{h}   \right\|^2_{ L^{2} ( \dot \Omega )}
+
(KC-C_{2})h \left\|\chi_{0} \tilde{u}^{(1)}_{h}
  \right\|^2_{ L^{2}(\Omega_{+})}\nonumber\\
&  \geq
(1+ 2C_{4}(C)h)  \left\| |\nabla f_{-}|\chi_{0} \tilde{u}^{(1)}_{h}   \right\|^2_{ L^{2} ( \dot \Omega )}
- (KC-C_{2})h
 \left\|\chi_{0} \tilde{u}^{(1)}_{h}
  \right\|^2_{ L^{2}(\Omega_{-})}
,\label{eq.k<n1}
\end{align}
where $C_{4}(C):= \frac{KC-C_{2}}{2\left\|\nabla
    f_{-}\right\|^2_{L^\infty(V_{\Gamma_1}')}}$  (see~\eqref{eq:nablaf-})
and $C$ has been chosen large enough to ensure
that $KC-C_{2}>0$.

In order to get a lower bound for the boundary term in~\eqref{eq.sum},
one uses the fact that the mixed Witten Laplacian $\Delta^{M,(1)}_{\hat f, \hat h} ( \dot \Omega )$
associated with
$\hat f=-\tilde \chi_{0}f_{-}$ where $\tilde \chi_{0}\in C^{\infty}(\overline \Omega,[0,1])$, $\tilde \chi_{0}=1$ on ${\rm supp}\, \chi_{0}$, ${\rm supp} \, \tilde \chi_{0} \subset (V_{\Gamma_1}'+B(0,\alpha))\cap \overline \Omega$     for $\alpha>0$ such that $f_-$ is smooth on ${\rm supp}\, \tilde \chi_{0}$ and $\hat h= \frac{h}{1+C_{4}(C)h}$, is nonnegative.
Starting from the inequality $(1+C_{4}(C)h) \mathcal Q_{\hat f,\hat h}^{M,(1)} ( \dot \Omega )(\chi_{0}\tilde u^{(1)}_{h},\chi_{0}\tilde u^{(1)}_{h})\geq 0$
 and then applying Lemma~\ref{le.GreenWeak} to $\chi_{0}\tilde
 u^{(1)}_{h}\in D\left(\mathcal Q^{M,(1)}_{\hat f,\hat h} ( \dot \Omega
   )\right)$  lead to (since $\chi_0=0$ on $\Gamma_0$):
\begin{multline}
\label{eq.k<n2}
 h \left(\int_{\Gamma_{1}}-\int_{\Gamma_{2}}\right)
\langle\chi_{0}\tilde u^{(1)}_{h},\chi_{0}\tilde u^{(1)}_{h}
\rangle
\pa_{n}f_{-}\, d\sigma \geq - 
(1+C_{4}(C)h)\left\| |\nabla f_{-}|\chi_{0}\tilde
  u^{(1)}_{h}\right\|_{L^2(\dot \Omega)}^2
\\
 -\frac{h^2}{1+C_{4}(C)h} \left(\left\|d\chi_{0}\tilde u^{(1)}_{h}\right\|^{2}_{L^2(\dot \Omega)}
+\left\|d^{*}\chi_{0}\tilde u^{(1)}_{h}\right\|_{L^2(\dot \Omega)}^{2}
 \right)
- h C_{5} \| \chi_{0}\tilde u^{(1)}_{h}\|_{L^2(\dot \Omega)}^{2},
\end{multline}
where $C_{5}$ is some positive constant independent of $C$ (it only
depends on $f_{-}$).

%Note
%also that we have used here the relations
%$\chi_{k}\tilde u^{(1)}_{h}|_{ \paV_{\Gamma_{1}}\setminus\overline{\Gamma_{1}}}=0$ and 
% $\pa_{n}f_{+}=\pa_{n}f$
%on $\Gamma_{1}$.\\

Injecting  the estimates \eqref{eq.k=n}--\eqref{eq.k<n2}
in~\eqref{eq.sum}
and 
using $f=-f_{-}+f_++f(z)$ on $V_{\Gamma_{1}}'$ with $\pa_{n}f_{+}=0$ on $\Gamma_{1}$,
then leads to:
\begin{align*}
C_{3}h&\geq 
  \sum_{k\in\{0,1\}}
  \frac{C_{4}h^3}{1+C_{4}h}\left(
  \left\|d\chi_{k}\tilde{u}^{(1)}_{h}\right\|^{2}_{ L^{2}(\dot \Omega)}
+\left\|d^{*}\chi_{k}\tilde{u}^{(1)}_{h}\right\|^{2}_{ L^{2}(\dot \Omega)}\right)
+(KC-C_2)h \left\|\chi_{1} \tilde{u}^{(1)}_{h}
  \right\|^2_{ L^{2}(\Omega_{+})}\\
&\quad+
C_{4}h\left\| |\nabla f_{-}|\chi_{0}\tilde u^{(1)}_{h}\right\|^2_{ L^2 (\dot \Omega)}
- (KC-C_{2})h
 \left\|\chi_{0} \tilde{u}^{(1)}_{h}
  \right\|^2_{ L^{2}(\Omega_{-})}\\
&\quad - h C_{5} \| \chi_{0}\tilde u^{(1)}_{h}\|^{2}_{ L^2 (\dot \Omega)}
  +h\int_{\Gamma_{2}}\langle\chi_{0} u|\chi_{0} u
\rangle_{ T^{*}_{\sigma}\dot\Omega}\pa_{n} f_{+}\,d \sigma.
\end{align*}
In the last computation, one has used that $1 \geq
\frac{C_{4}h}{1+C_{4}h}$,
$1-\frac{1}{1+C_{4}h}=\frac{C_{4}h}{1+C_{4}h}$.
It follows moreover from~\eqref{eq:dnf+positive} that  $\pa_{n} f_{+} > 0$ on $\supp\chi_{0}\cap \Gamma_{2}$.
Then, 
since $
 |\tilde{u}^{(1)}_{h}(x) |\leq e^C |u^{(1)}_{h}(x)|$ a.e. on $ \Omega_{-}$,
there exists $C_{6}(C,C_{2},K)$ such that
\begin{align*}
C_{6}h&\geq 
  \sum_{k\in\{0,1\}}
  \frac{C_{4}h^3}{1+C_{4}}\left(
  \left\|d\chi_{k}\tilde{u}^{(1)}_{h}\right\|^{2}_{ L^{2}(\dot \Omega)}
+\left\|d^{*}\chi_{k}\tilde{u}^{(1)}_{h}\right\|^{2}_{ L^{2}(\dot \Omega)}\right)
+2C_{4}\left\|\nabla f_{-}\right\|^2_{L^\infty(V'_{\Gamma_1})}h \left\|\chi_{1} \tilde{u}^{(1)}_{h}
  \right\|^2_{ L^{2}(\dot \Omega)}\\
&\quad+
C_{4}h\left\| |\nabla f_{-}|\chi_{0}\tilde u^{(1)}_{h}\right\|^2_{ L^2 (\dot \Omega)}
- h C_{5} \| \chi_{0}\tilde u^{(1)}_{h}\|^{2}_{ L^2 (\dot \Omega)}.
\end{align*}
One has used that $1+C_{4}\geq 1+C_{4}h$, for $h\leq 1$ and $KC-C_2 =2C_4\left\|\nabla
  f_{-}\right\|^2_{L^\infty(V'_{\Gamma_1})}$. 
Additionally, since $|\nabla f_{-}|\geq c>0$ on
$\overline{V_{\Gamma_{1}}'}$ (see~\eqref{eq:nablaf-}), $\lim_{C \to
  \infty} C_{4}(C)= +\infty$ and $C_{5}$ is independent of $C$,
one can then choose $C$ such that $c^2C_{4}-C_{5}>0$. This implies the existence of a constant $C_{7}>0$ such that for
$h_0>0$ small enough and for all $h\in (0,h_0]$,
\begin{equation}
\| \tilde{u}^{(1)}_{h} \|_{  L^{2} ( \dot \Omega )}
+\|d\tilde{u}^{(1)}_{h}\|_{  L^{2} ( \dot \Omega )}+\|d^*\tilde{u}^{(1)}_{h}\|_{  L^{2} ( \dot \Omega )}
\ \leq
\ \frac{C_{7}}{h}.
\end{equation}
Since $\varphi-f_{+}\ge - C_8 h \ln \frac1h$ for some constant $C_8$,
there exists $N_0>0$ such that:
\begin{equation}
  \label{eq.step1}
  \| e^{\frac{f_{+}}{h}}u^{(1)}_{h} \|_{  L^{2} ( \dot \Omega )}
+\|d(e^{\frac{f_{+}}{h}}u^{(1)}_{h})\|_{  L^{2} ( \dot \Omega
  )}+\|d^*(e^{\frac{f_{+}}{h}}u^{(1)}_{h})\|_{  L^{2} ( \dot \Omega )}
= O(h^{-N_{0}}).
\end{equation}
One has in particular, owing to 
the trace result~\eqref{eq:subelliptic} stated in  Proposition~\ref{pr.QTN} and since $f_{+}=f-f(z)$
on $\Gamma_{1}$,
\begin{equation}
  \label{eq.step1trace}
  \left\|e^{\frac{f-f(z)}{h}}u^{(1)}_{h}
  \right\|_{ L^{2}(\Gamma_{1})}= O(h^{-N_{0}}).
\end{equation}

\noindent\underline{Step 2:} Agmon estimate in $\dot\Omega$.
\medskip

\noindent
One follows the same approach as in step 1 but with the function
$$\varphi=
\left\{
  \begin{aligned}[c]
& \Psi-Ch\ln \frac{ \Psi}{h},  \ \text{if}\ \Psi > Ch,\\
&\Psi-Ch\ln C, \ \text{if}\ \Psi \leq Ch,
  \end{aligned}
\right.
$$
where the constant $C> 1$ will be fixed later on, and with the  
 level sets:
\begin{eqnarray*}
 && \Omega_{-}=\left\{x\in \dot\Omega\ \text{s.t.}\   \Psi(x) \leq Ch\right\}
 \quad
\text{and}\quad
\Omega_{+}=\dot\Omega\setminus \Omega_{-}.
\end{eqnarray*}
 Applying formula \eqref{eq.intbypartphi} then leads to
(note that $\vert \tilde{u}^{(1)}_{h}\vert \leq e^C
\vert {u}_{h}\vert$ on $\Omega_-$, $\partial_nf <0$ on $\Gamma_{0}$ and $\partial_nf >0$ on $\Gamma_{2}$):
\begin{align}
 C_2(C) h\left(1+
 \int_{\Gamma_{1}}
\langle\tilde u^{(1)}_{h},\tilde u^{(1)}_{h}
\rangle_{ T_{\sigma}^{*}\Omega}~d\sigma
\right)
 &\ge
\left\|hd\tilde{u}^{(1)}_{h}\right\|^{2}_{ L^{2}(\dot \Omega)}
+\left\|hd^{*}\tilde{u}^{(1)}_{h}\right\|^{2}_{ L^{2}(\dot \Omega)}
\label{eq.trace0}
\\
&\quad +\langle (|\nabla
f|^{2}-|\nabla\varphi|^{2})\tilde{u}^{(1)}_{h},\tilde{u}^{(1)}_{h}\rangle_{L^2(\dot
  \Omega)}-C_{1}h\left\|\tilde{u}^{(1)}_{h}\right\|_{ L^{2}(\Omega_{+})}^2,
\nonumber\end{align}
where  $\tilde{u}^{(1)}_{h}:=e^{\frac{\varphi}{h}} u^{(1)}_h$, the constant $C_{1}$ is independent of $C$, whereas $C_2$ is a
constant depending on $C$.
Besides, due to the relations
$$
 \Psi= f-f(z) \text{ on } \Gamma_{1}
\text{ and } e^{\frac{\varphi}{h}}\ \leq\ 
e^{\frac{ \Psi}{h}} \text{ on $\dot \Omega$},
$$
the trace estimate obtained in~\eqref{eq.step1trace} implies
\begin{equation}
\label{eq.trace1}
\int_{\Gamma_{1}}
\langle\tilde u^{(1)}_{h},\tilde u^{(1)}_{h}
\rangle_{ T_{\sigma}^{*}\Omega}~d\sigma =O
(h^{-2N_{0}}).
\end{equation}
Injecting \eqref{eq.trace1} in~\eqref{eq.trace0} then gives
\begin{align}
&\left\|hd\tilde{u}^{(1)}_{h}\right\|^{2}_{ L^{2}(\dot \Omega)}
+\left\|hd^{*}\tilde{u}^{(1)}_{h}\right\|^{2}_{ L^{2}(\dot \Omega)}\nonumber\\
&+\langle (|\nabla
f|^{2}-|\nabla\varphi|^{2})\tilde{u}^{(1)}_{h},\tilde{u}^{(1)}_{h}\rangle
_{ L^{2}(\dot \Omega)}-C_{1}h\left\|\tilde{u}^{(1)}_{h}\right\|_{
  L^{2}(\Omega_{+})}^2 = O( h^{1-2N_{0}}). 
\label{eq.trace2}
\end{align}
Since moreover $|\nabla  \Psi|^2\leq |\nabla f|^2$ (see~\eqref{eq:eqq}) and $f$ has no critival point in $\overline{\dot\Omega}$, one gets:
$$ 
|\nabla f|^2-|\nabla \varphi |^{2}
 \geq 
|\nabla f|^2-|\nabla  \Psi|^2\left(1-\frac{Ch}{ \Psi}\right)
 \geq Ch \frac{|\nabla f |^{2}}{ \Psi}
 \geq  CC_{3}h
 \text{ on $\Omega_{+}$}
$$
%\begin{eqnarray*}
%|\nabla f|^2-|\nabla \psi^h|^2&=& 1_{\Omega_{+}^{h}}(x)\left(2Ch\frac{\left|\nabla  \Psi\right|^2}{ \Psi}-C^2h^2\frac{\left|\nabla  \Psi\right|^2}{ \Psi^2}\right) \\
%&\geq&  \frac{Ch|\nabla  \Psi|^{2}}{ \Psi} \geq C_2^{-1}Ch,
%\end{eqnarray*}
where $C_3>0$ is independent of $C$.
Since $|\nabla  f|^2\geq |\nabla \Psi|^2=|\nabla \varphi|^2$ a.e. on $\Omega_-$, adding the term $(CC_{3}-C_{1})h\left\| \tilde{u}^{(1)}_{h}\right\|^2_{ L^{2}(\Omega_{-})}$
to \eqref{eq.trace2} leads to
$$ \left\|hd\tilde{u}^{(1)}_{h}\right\|^{2}_{ L^{2}(\dot \Omega)}
+\left\|hd^{*}\tilde{u}^{(1)}_{h}\right\|^{2}_{ L^{2}(\dot \Omega)}
+ (CC_3-C_{1})h\left\|\tilde{u}^{(1)}_{h}\right\|^2_{ L^{2} ( \dot
  \Omega )} = O(h^{1-2N_0})
$$
Now,   since $\varphi - \Psi \ge -C_4 h \ln
\frac{1}{h}$, taking
$C>  \frac{C_{1}}{C_{3}}$,  there exists $N_1>0$ such that:
\begin{equation}
  \label{eq.step2}
  \| e^{\frac{ \Psi}{h}}u^{(1)}_{h} \|_{ L^{2} ( \dot \Omega )}
+\|d(e^{\frac{ \Psi}{h}}u^{(1)}_{h})\|_{ L^{2} ( \dot \Omega )}+\|d^*(e^{\frac{ \Psi}{h}}u^{(1)}_{h})\|_{ L^{2} ( \dot \Omega )}
  = O(h^{-N_{1}}).
\end{equation}
This concludes the proof of~\eqref{eq.Agmon}.
%\\
%\noindent\textbf{Step 3}~: Elliptic regularity.\\
% We now set 
%$\tilde{u}^{(1)}_{h}=e^{\frac{\Phi}{h}}u^{h}$. For $\rho'< \rho$,
%we take a cut-off $\chi\in \mathcal{C}^{\infty}(\Omega_{U_{0},\rho})$
%with compact support in $\Omega_{U_{0},\rho}\cup \Gamma_{TD}$ and
%such that $\chi= 1$ on a neighborhood of $\Omega_{U_{0},\rho'}$.
%The form $v^{h}=\chi \tilde{u}^{(1)}_{h}$ satisfies the boundary value
%problem 
%\begin{eqnarray*}
%  \left\{
%\begin{aligned}[c]{ll}
%v^{h}-\Delta v^{h}=r_{0}^{h} &\text{ in } \rz^{n}_{-}\;,\\
%\mathbf{t}v^{h}=0\mbox{ and } \mathbf{t}d^{*}v^{h}=r_{1}^{h} &
%\text{on} \left\{x_{d}=0\right\}\;,
%\end{aligned}
%\right.\\
%\text{ with }
%\left\|r_{0}^{h}\right\|_{ L^{2}(\rz^{n}_{-})}=O(h^{-N_{1}})\;
%\text{ and }
%\left\|r_{1}^{h}\right\|_{ H^{1/2}(\rz^{d-1})}=O(h^{-N_{1}})\; .
%\end{eqnarray*}
%This implies the existence of $N_1>0$ such that~:
%$$\left\|v^{h}\right\|_{ H^{2}}=O(h^{-N_{1}})\;.$$
%We conclude by induction for any finite decreasing
% sequence $(\rho_{k})_{0\leq k \leq K}$ and
%associated cut-offs $\chi_{k},$, with $\chi_{k}= 1$ in a neighborhood
%of $\Omega_{U_{0},\rho_{k}}$ and $\supp \chi_{k}\subset
%\left\{\chi_{k-1}= 1\right\},$.
\end{proof}

\subsection{Comparison of the eigenform $u^{(1)}_{h,i}$ and its WKB approximation}\label{sec:wkb}
Throughout this section, one assumes \textbf{[H1]}, \textbf{[H2]} and \textbf{[H3]}. In all this section, we consider, for a fixed critical point $z_i$, an
ensemble of sets~$\mathcal S_{M,i}$ associated with $z_i$ (see Definition~\ref{Sets}) and
an $L^2$-normalized eigenform 
$u^{(1)}_{h,i}$ of $\Delta^{M,(1)}_{f,h}(\dot \Omega_i)$ associated
with the eigenvalue $0$, as introduced at the end of Section~\ref{sec:def_mixed}. 
For the ease of notation, we drop the
subscript $i$ in all this section.

\subsubsection{Construction of the WKB expansion of $u^{(1)}_{h}$} \label{sectionwkb}
Let $z$ be a local minimum of $f|_{\pa \Omega}$. 
Before going through a rigorous construction of a WKB expansion  $u^{(1)}_{z,wkb}$ of $u^{(1)}_{h}$ in a neighborhood of $z$, let us explain formally how  we  proceed. Let us recall that the $1$-form $u^{(1)}_{h}$ satisfies:
\begin{equation}\label{eq.uhh1}
\left\{
\begin{aligned}
&\Delta_{f,h}^{(1)}u^{(1)}_{h}=0 \text{ in }
\dot \Omega,\\
&\mbf{t} u^{(1)}_{h}=0 \text{ and } \mbf{t} d^*_{f,h}u^{(1)}_h =0 \text{ on }
\Gamma_{1},
\end{aligned}\right.
\end{equation}
plus additional boundary conditions on $\Gamma_0\cup\Gamma_2$ that we do not recall since the objective is to approximate $u^{(1)}_h$ in a neighborhood of $\Sigma$ in $\overline{\dot{\Omega}}$ (where we recall $\Sigma$ is an open subset of $\pa \Omega$ containing $z$ and such that $\overline\Sigma\subset \Gamma_1$, see~\eqref{eq.sigmai-gammai}). The behavior of $u^{(1)}_{h}$ in a neighborhood of $\Gamma_1$ exhibited in Proposition~\ref{pr.Agmon} suggests to take $u^{(1)}_{z,wkb}$ of the form $u^{(1)}_{z,wkb}(x,h)=a^{(1)}(x,h)\, e^{-\frac{d_a(x ,z) }{h}}$ where $a^{(1)}$ is expanded in powers of $h$: $a^{(1)}(x,h)=\sum_{k\ge 0} a_k^{(1)}(x)h^k$ and to look for $1$-forms $(a_k^{(1)})_{k\ge 0}$ so that $u^{(1)}_{z,wkb}$ is a nontrivial $1$-form satisfying (compare with~\eqref{eq.uhh1}):
\begin{equation}\label{eq.formally_wkb1}
\left\{
\begin{aligned}
&\Delta_{f,h}^{(1)}u^{(1)}_{z,wkb}=O(h^{\infty})\, e^{-\frac{d_a(\cdot ,z) }{h}} \text{ in }
\dot \Omega,\\
&\mbf{t} u^{(1)}_{z,wkb} =0 \text{ and } \mbf{t} d^*_{f,h}u^{(1)}_{z,wkb} =O(h^{\infty})\, e^{-\frac{d_a(\cdot ,z) }{h}} \text{ on }
\Gamma_{1},
\end{aligned}\right.
\end{equation}
where the meaning of $O(h^{\infty})$ is formally $h^sO(h^{\infty})=o_h(1)$ for any $s\in \mathbb R$.  
The boundary conditions in~\eqref{eq.formally_wkb1} ensures that when cutting suitably a solution to~\eqref{eq.formally_wkb1} near $\Gamma_1$, the resulting $1$-form belongs to the form domain of $\Delta^{M,(1)}_{f,h}(\dot \Omega)$ (this is needed if one wants to approximate $u^{(1)}_{h} $ on $\pa \dot \Omega$).  Instead of directly trying to solve~\eqref{eq.formally_wkb1}, the construction of $u^{(1)}_{z,wkb}$ can be simply done as follows (see \cite[Section 4.2]{helffer-nier-06}). Using the complex property, one considers  $u^{(1)}_{z,wkb}=d_{f,h}u^{(0)}_{z,wkb}$ where the function $u^{(0)}_{z,wkb}=a^{(0)}(\cdot,h) \, e^{-\frac{d_a(. ,z) }{h}}$ where $a^{(0)}(x,h)=\sum_{k\ge 0} a_k^{(0)}(x)h^k$ for a non trivial family of functions $(a_k)_{k\ge 0}$ such that: 
\begin{equation}\label{eq.formally_wkb0}
\left\{
\begin{aligned}
&\Delta_{f,h}^{(0)}u^{(0)}_{z,wkb}=O(h^{\infty})\, e^{-\frac{d_a(\cdot ,z) }{h} } \text{ in }
\dot \Omega,\\
& u^{(0)}_{z,wkb} =e^{-\frac 1h f} \text{ on }
\Gamma_{1}.
\end{aligned}\right.
\end{equation}
\begin{sloppypar}
\noindent
This implies the boundary condition: $a^{(0)}=1$ on $\Gamma_1$.  Then, if $u^{(0)}_{z,wkb}=a^{(0)} \, e^{-\frac{d_a(. ,z) }{h}}$ is a solution to~\eqref{eq.formally_wkb0},  we set: $$u^{(1)}_{z,wkb}=d_{f,h}u^{(0)}_{z,wkb}.$$
One can easily check that the $1$-form $u^{(1)}_{z,wkb}$ then satisfies~\eqref{eq.formally_wkb1} and the extra boundary condition $\mbf{t} d^{*}_{f,h}u^{(1)}_{z,wkb} = O(h^{\infty})\, e^{-\frac{f-f(z) }{h}}$ on $\Gamma_{1}$. Indeed, it holds:
$$d_{f,h} u^{(0)}_{z,wkb}=e^{-\frac{d_a(\cdot ,z) }{h} } \left( d(f-d_a(\cdot,z))\, a^{(0)}+ h\,d a^{(0)} \right),$$
which implies $\mbf{t} u^{(1)}_{z,wkb} =0$ since $a^{(0)}=1$ and \begin{equation}\label{eq:cond_limit_phase}
f-d_a(\cdot,z)=f(z) \text{ on } \Gamma_1.
\end{equation} In addition, one has
$$d_{f,h}^* d_{f,h} u^{(0)}_{z,wkb}=\Delta_{f,h}^{(0)}u^{(0)}_{z,wkb}=O(h^{\infty})\, e^{-\frac{d_a(\cdot ,z) }{h}}$$
 which implies  $\mbf{t} d^*_{f,h}u^{(1)}_{z,wkb} =O(h^{\infty})\, e^{-\frac{d_a(\cdot ,z) }{h}}$ and
$$
\Delta_{f,h}^{(1)} d_{f,h} u^{(0)}_{z,wkb}=d_{f,h}\Delta_{f,h}^{(0)}u^{(0)}_{z,wkb}=O(h^{\infty})\, e^{-\frac{d_a(\cdot ,z) }{h}}.
$$
Thus, the $1$-form $u^{(1)}_{z,wkb}$ satisfies~\eqref{eq.formally_wkb1}. \\
 Expanding in powers of $h$ the function $e^{\frac{d_a(x ,z) }{h} }\, \Delta_{f,h}^{(0)} \Big(\big (\sum_{k\ge 0} a_k^{(0)}(x)h^k\big ) e^{-\frac{d_a(x ,z) }{h} }\Big)$, $u^{(0)}_{z,wkb}$ is a solution to~\eqref{eq.formally_wkb0} if it holds:
\begin{equation}\label{eq.ju1}
\vert \nabla d_a(x, z)\vert =\vert \nabla f(x)\vert, \ {\rm for }\ x\in \dot \Omega,
\end{equation}
which is satisfied at least in a neighborhood of $z$ (see Proposition~\ref{TTH}) 
 and if $(a_k^{(0)})_{k\ge 0}$ satisfies the following transport equations,
defined recursively by:
\begin{equation}\label{eq.ju2}
 \left\{
    \begin{array}{ll}
        (\Delta \Phi-\Delta f +2\nabla \Phi \cdot \nabla)a_0^{(0)}=0 & \mbox{in }  \dot \Omega \\
      (\Delta \Phi-\Delta f +2\nabla \Phi \cdot \nabla)a_{k+1}^{(0)}=\Delta a_k^{(0)}  & \mbox{in } \dot \Omega , \  \forall k\geq 0,
    \end{array}
\right.
\end{equation}
with boundary conditions 
$
 \left\{
    \begin{array}{ll}
        a_0^{(0)}=1 & \mbox{on } \Gamma_1 \\
       a_k^{(0)} =0 & \mbox{on } \Gamma_1, \ \forall k\geq 1
    \end{array}
\right.$. The equation \eqref{eq.ju1} together with the boundary condition~\eqref{eq:cond_limit_phase} justify {\em a posteriori} the choice of the function $d_a(\cdot,z)$ in the exponential for the ansatz on $u^{(1)}_{z,wkb}$. 
Let us mention that  $\partial_n f > 0$ on $\Gamma_1$ implies that there exists a non trivial solution $u^{(0)}_{z,wkb}$ to~\eqref{eq.formally_wkb0} in a neighborhood of $\Gamma_1$ since in that case the transport equations~\eqref{eq.ju2} are non degenerate. Let us now justify rigorously the construction of the WKB expansion $u^{(1)}_{z,wkb}$ of $u^{(1)}_{h}$, which is, in view of~\eqref{eq.ju1} and~\eqref{eq.ju2},  possible near $\Gamma_1$. 
\end{sloppypar}

\medskip
\noindent
\underline{A preliminary construction.}
\medskip

\noindent
Let $\Phi$ be the solution to the eikonal equation (\ref{eikonalequationboundary}) on a neighborhood $V_{\partial \Omega}$ of the boundary $\partial \Omega$. Let us introduce the formal transport operator 
$$T:=\Delta \Phi-\Delta f +2\nabla \Phi \cdot \nabla.$$
Let us consider the solutions to the following transport equations,
defined recursively by
\begin{equation}\label{eq:transport}
 \left\{
    \begin{array}{ll}
        Ta_0=0 & \mbox{in } V_{\partial \Omega} \\
       Ta_{k+1}=\Delta a_k  & \mbox{in } V_{\partial \Omega}, \quad \forall k\geq 0,
    \end{array}
\right.
\end{equation}
with boundary conditions 
\begin{equation}\label{eq:transport_bdd}
 \left\{
    \begin{array}{ll}
        a_0=1 & \mbox{on } \partial \Omega \\
       a_k =0 & \mbox{on } \partial \Omega, \quad \forall k\geq 1.
    \end{array}
\right.
\end{equation}
For a fixed $k$, the transport equation can be solved locally around each $z\in \partial \Omega$
thanks to the condition $\partial_n\Phi=-\partial_nf<0$ on $\partial
\Omega$ and thus on a neighborhood of $\partial \Omega$ (independent
of $k$) using a compactness argument. Therefore, up to choosing a
smaller neighborhood $V_{\partial \Omega}$ of $\partial \Omega$ in
$\overline{\Omega}$, there exists a unique sequence of 
$C^{\infty}(V_{\partial \Omega})$ functions $(a_k)_{k\geq 0}$ solution to~\eqref{eq:transport}-\eqref{eq:transport_bdd}.

There exists a function $a=a(x,h)$  (called a resummation of the
formal symbol  $\sum_{k=0}^{+\infty} a_kh^k$) $C^\infty$ and uniformly
bounded together with all its derivatives such that 
$$a(x,h)=1 \ {\rm on} \ \partial \Omega  \quad {\rm and}  \quad
a(x,h)\sim \sum_{k=0}^{+\infty} a_k(x)h^k.$$
This means that $a-\sum_{k=0}^{+\infty} a_kh^k$ is $O(h^\infty)$
in the following sense: for all compact $K$ in $V_{\partial \Omega}$, for all
$\alpha \in \mathbb{N}^d$, for all $N \in \mathbb{N}$,
\begin{equation}\label{eq.hinfty}\left\|\partial^\alpha_x \left( a- \sum_{k=0}^N a_k(x) h^k
  \right)\right\|_{L^\infty(K)} \le C_{K,\alpha,N} h^{N+1}.
  \end{equation}
Such a construction is standard and can be found in
\cite{dimassi-sjostrand-99} or in\cite{helffer-nier-06}, where it is
done using a Borel summation. Moreover $a$ is unique up to a term of
 order $O(h^{\infty})$. Let us now define on $V_{\partial \Omega}$: 
$$u_{wkb}^{(0)}(x,h):=a(x,h) \,e^{-\frac{\Phi}{h}}.$$
\label{page.uowkb}
By construction of the sequence $(a_k)_{k\geq 0}$, the function $u_{wkb}^{(0)}$ solves 
$$\left\{
\begin{aligned}
\Delta_{f,h}^{(0)}u^{(0)}_{wkb} =O(h^{\infty})\,e^{-\frac{\Phi}{h}} \text{ in }
V_{\partial \Omega},\\
u^{(0)}_{wkb}=e^{-\frac{\Phi}{h}}=e^{-\frac{f}{h}} \text{ on }
\partial \Omega,
\end{aligned}\right.
$$
where $O(h^{\infty})$ is defined in~\eqref{eq.hinfty}. 
Indeed, using~\eqref{eq:laplwitt0}, $\vert \nabla f \vert^2=\vert
\nabla \Phi \vert^2$ on $V_{\partial \Omega}$, and the
equations~\eqref{eq:transport} satisfied by $(a_k)_{k \ge 0}$,
\begin{align*}
e^{\frac{\Phi}{h}}\Delta_{f,h}^{(0)}u^{(0)}_{wkb}&=-h^2\Delta a(x,h)+h[a(x,h)\Delta \Phi +2 \nabla \Phi\cdot \nabla a(x,h)] -a(x,h) \vert \nabla \Phi\vert^2\\
&\quad + a(x,h) \vert \nabla f \vert^2 -h a(x,h)\Delta f\\
&\sim hTa_0 + h^2
\sum_{k=0}^{+\infty} h^k(Ta_{k+1}-\Delta a_{k})\\
& =O(h^\infty).
\end{align*}
In addition, it holds $u^{(0)}_{wkb}=e^{-\frac{\Phi}{h}}$  on $\partial \Omega$ since $a(x,h)=1$ on $\partial \Omega$. Let us now define on $V_{\partial \Omega}$:
$$u^{(1)}_{wkb}:=d_{f,h} u^{(0)}_{wkb} .$$
\label{page.u1wkb}
The $1$-form $u^{(1)}_{wkb}$  satisfies:
\begin{equation}\label{lem:wkb1}
\left\{
\begin{aligned}
\Delta_{f,h}^{(1)} u^{(1)}_{wkb} =O(h^{\infty})e^{-\frac{\Phi}{h}}
\text{ in } 
V_{\partial \Omega} ,\\
\mathbf{t}u^{(1)}_{wkb}=0\text{ on }
\partial \Omega ,\\
\mathbf{t}d_{f,h}^{*}u^{(1)}_{wkb}= 
O (h^{\infty})e^{-\frac{\Phi}{h}}\text{ on }
\partial \Omega,
\end{aligned}\right.
\end{equation}
\noindent 
\begin{sloppypar}
\noindent
where $O(h^{\infty})$ is defined in~\eqref{eq.hinfty}.
Indeed, one has $\mathbf{t}u^{(1)}_{wkb}=\mathbf{t}d_{f,h} u^{(0)}_{wkb}=d_{f,h}
\mathbf{t} u^{(0)}_{wkb}=d_{f,h} \mathbf{t} \left(a(x,h)
e^{-\frac{\Phi}{h}}\right)=d_{f,h} \mathbf{t} e^{-\frac{f}{h}} = d_{f,h} e^{-\frac{f}{h}} = 0$
 since $a(x,h)=1$ and $\Phi=f$ on 
$\partial \Omega $. Moreover $\mathbf{t}d_{f,h}^{*}u^{(1)}_{wkb}=\mathbf{t}d_{f,h}^{*}d_{f,h}u^{(0)}_{wkb}= \mathbf{t} \Delta_{f,h}^{(0)}u^{(0)}_{wkb}= O (h^{\infty}) e^{-\frac{\Phi}{h}}$. Finally,
$\Delta_{f,h}^{(1)}u^{(1)}_{wkb}= \Delta_{f,h}^{(1)}d_{f,h}u^{(0)}_{wkb}=d_{f,h}\Delta_{f,h}^{(0)}u^{(0)}_{wkb}=(hd+df \wedge) O (h^{\infty}) e^{-\frac{\Phi}{h}}=O (h^{\infty}) e^{-\frac{\Phi}{h}}$. \end{sloppypar}

\medskip
\noindent
\underline{WKB expansion of $u^{(1)}_{h}$.}
\medskip

\noindent
Let $z$ be a local minimum of $f|_{\pa \Omega}$. 
Let us now define the WKB expansion of $u^{(1)}_{h}$ on $V_{\partial \Omega}$ by:
\begin{equation}\label{wkbi}
u^{(1)}_{z,wkb}:=e^{\frac{f(z)}{h}} u^{(1)}_{wkb}= e^{\frac{f(z)}{h}}
d_{f,h}u^{(0)}_{wkb} = d_{f,h} \left(a(\cdot,h) e^{-\frac{\Phi-f(z)}{h}}\right).
\end{equation}\label{page.u1zwkb}
One recalls (see Proposition~\ref{TTH}) that for any smooth open domain
$\Gamma$  such that $\overline{\Gamma} \subset \Gamma_1$ and $z\in \Gamma$, there exists a
neighborhood  of $\Gamma$ in $\overline \Omega$, denoted by
$V_{\Gamma}\subset V_{\partial \Omega} \cap (\Gamma_1\cup \dot{\Omega})$, such that for all $x\in V_{\Gamma}$,
$$\Psi(x)=d_a(x,z)=\Phi(x)-f(z).$$
\begin{lemma}\label{wkb1}
 Let us assume that the hypotheses \textbf{[H1]} and \textbf{[H3]} hold. Let us consider $z$ a local minimum of $f|_{\partial \Omega}$ as introduced in hypothesis \textbf{[H2]}.
The $1$-form $u^{(1)}_{z,wkb}$  satisfies 
\begin{equation}\label{WW}
\left\{
\begin{aligned}
\Delta_{f,h}^{(1)}   u^{(1)}_{z,wkb} =O(h^{\infty})e^{-\frac{\Phi-f(z)}{h}}
& \text{ in } 
V_{\partial \Omega} ,\\
\mathbf{t}u^{(1)}_{z,wkb}=0 &\text{ on }
\partial \Omega ,\\
\mathbf{t}d_{f,h}^{*}u^{(1)}_{z,wkb} = 
O (h^{\infty})e^{-\frac{\Phi-f(z)}{h}} &\text{ on }
\partial \Omega,
\end{aligned}\right.
\end{equation}
where $O(h^{\infty})$ is defined in~\eqref{eq.hinfty}.
For any $\chi\in C^{\infty}_c(V_{\Gamma})$ such that $\chi=1$ on a
neighborhood of $z$, it
holds: in
the limit $h\to 0$,
\begin{equation}\label{laplacewkb}
\int_{\Omega} \vert \chi(x) u^{(1)}_{z,wkb}(x)\vert^2dx=C_{z,wkb}^2\
h^{\frac{d+1}{2}} \ (1+O(h)),
\end{equation}
 where 
\begin{equation}\label{cwkbi}
C_{z,wkb}:= \pi^{\frac{d-1}{4}}\frac{\sqrt{2\partial_nf(z)}}{({\rm det \ Hess} f|_{\partial \Omega}(z) )^{\frac 14} }.
\end{equation}\label{page.czwkb}
Furthermore, there exists $C>0$ such that for $h$ small enough,
\begin{equation}\label{eq:estim_wkb_H1}
\Vert \chi  u^{(1)}_{z,wkb}\Vert_{H^1(\Omega)} \leq Ch^{-1},
\end{equation}
% Estimee grossiere en fait... Il faudrait faire un Laplace pour avoir
% une estimee plus precise.
and $\mathcal Q_{f,h}^{M,(1)} ( \dot \Omega )(\chi u^{(1)}_{z,wkb})=O(h^{\infty})$.
 \end{lemma}
\begin{proof}  
Equation~\eqref{WW} is easily obtained from~\eqref{lem:wkb1}. Let
us now prove~\eqref{laplacewkb} and~\eqref{cwkbi}. Notice that one can write
(using~\eqref{eq:dfh}) 
\begin{align} \nonumber u^{(1)}_{z,wkb} &=  d_{f,h}\left[e^{-\frac{\Phi-f(z)}{h}}  \sum_{k=0}^{+\infty} a_kh^k\right]\\
\nonumber
&=e^{-\frac fh}(hd)e^{\frac fh}\left[ e^{-\frac{\Phi-f(z)}{h}}  \sum_{k=0}^{+\infty} a_kh^k\right]     \\
\nonumber
&=e^{-\frac{\Phi-f(z)}{h}} e^{-\frac{f-\Phi}{h}}(hd)e^{\frac{f-\Phi}{h}}\left[  \sum_{k=0}^{+\infty} a_kh^k\right]  \\
\nonumber
&= e^{-\frac{\Phi-f(z)}{h}} \left( d(f-\Phi)\wedge a_0 \right) \\
\label{eq.carre}
&\quad +e^{-\frac{\Phi-f(z)}{h}} \left[h d  \sum_{k=0}^{+\infty} a_kh^k   + d(f-\Phi)\wedge  \sum_{k=1}^{+\infty} a_kh^k \right].
\end{align} 
Recall that the function $\chi$ is supported in $V_{\Gamma}$ and the
function $x\mapsto \Phi(x)-f(z)$ has a unique minimum on $V_{\Gamma}$ which
is $z$ since  $\Phi(x)-f(z)=d_a(x,z) \ge 0$ on
$V_{\Gamma}$. Therefore, in the limit $h\to 0$:
$$\sqrt{\int_{\Omega} \left\vert \chi(x)
    u^{(1)}_{z,wkb}(x)\right\vert^2dx}=\sqrt{\int_{\Omega}\left\vert
    e^{-\frac{\Phi-f(z)}{h}} \chi d(f-\Phi)\wedge
    a_0\right\vert^2} \,  (1+O(h)).$$
 Additionaly, since $\chi(z)=1$ and $\vert d(f-\Phi)(z)\vert ^2=\vert
 \nabla (f-\Phi)(z)\vert^2= \vert \nabla_T (f-\Phi)(z)\vert^2+
 (2\partial_nf(z))^2=(2\partial_nf(z))^2$, one gets  using Laplace's method
\begin{equation}\label{eq:laplace}
\int_{\Omega} \left\vert e^{-\frac{\Phi-f(z)}{h}} \chi
  d(f-\Phi)\wedge a_0\right\vert^2=(\pi
h)^{\frac{d+1}{2}}\frac{2\partial_nf(z)}{\pi\sqrt{ {\rm det \ Hess}
    f|_{\partial \Omega}(z)   } }\ (1+O(h)).
\end{equation}

Let us give more details on how to obtain~\eqref{eq:laplace}. Recall that on ${\rm
  supp}(\chi)$, $\Phi-f(z)=f_++f_-$, $f-\Phi=f-\Psi-f(z)=-2f_-$ and,
on ${\rm supp}(\chi) \cap \partial \Omega$, $\partial_nf = -\partial_nf_- =|\nabla  f_-|$.
Thus, using the coordinate set introduced in Definition~\ref{def:x'xd} and the co-area formula
$dx=\frac{d\sigma_{\Sigma_\eta}}{|\nabla f_-|} d\eta$ (see for example~\cite{ambrosio-fusco-pallara-00})
\begin{align*}
&\int_{\Omega} \left\vert e^{-\frac{\Phi(x)-f(z)}{h}} \chi(x)
  d(f-\Phi) (x) \wedge a_0 (x)\right\vert^2dx
=4\int_{-\alpha}^0  e^{-2\frac{\eta}{h}} \int_{\Sigma_\eta} 
  e^{-2\frac{f_+}{h}} \chi^2a_0^2  \vert \nabla f_-\vert
  d\sigma_{\Sigma_\eta} \, d\eta \\
&=4\int_{-\alpha}^0  e^{-2\frac{\eta}{h}} \int_{\partial \Omega} 
  e^{-2\frac{f_+(x',0)}{h}} \chi^2(x',\eta)a_0^2(x',\eta)  |\nabla f_-|(x',\eta)
  j(x',\eta) d\sigma_{\partial \Omega}(x') \, d\eta
\end{align*}
where $\Sigma_\eta=\{x, \, f_-(x)=-\eta\}$, $\sigma_{\Sigma_\eta}$ is
the Lebesgue measure on $\Sigma_\eta$. In the last equality, $j(x',\eta)$ is the Jacobian of
the parametrization of $\Sigma_\eta$ by $x' \in \partial \Omega$. Using the
Laplace formula, for any $\eta \in [-\eta_0,0]$ with $\eta_0>0$
sufficiently small so that $ \chi^2(z,\eta)\neq 0$ for all $\eta \in [-\eta_0,0]$, one has
\begin{align}
&\int_{\partial \Omega} 
  e^{-2\frac{f_+(x',0)}{h}} \chi^2(x',\eta) a_0^2(x',\eta)  |\nabla f_-|(x',\eta)
  j(x',\eta) d\sigma_{\partial \Omega}(x')\nonumber\\
&=(\pi h)^{\frac{d-1}{2}} (\det
  \Hess f_+(z))^{-1/2} \chi^2(z,\eta)  |\nabla f_-|(z,\eta)
  j(z,\eta) a_0^2(z,\eta)  (1+O(h))\label{eq:wkb_bord}
\end{align}
where $O(h)$ is a function of $\eta$ and $h$ with $L^\infty$ norm in
$\eta \in [0,\eta_0]$
bounded from above by a constant times $h$ (thanks to the regularity
of the involved terms), for sufficiently small~$h$.
Thus, using again Laplace's method: 
\begin{align*}
&\int_{\Omega} \left\vert e^{-\frac{\Phi-f(z)}{h}} \chi
  d(f-\Phi)\wedge a_0\right\vert^2\\
&=4 \int_{-\alpha}^0  e^{-2\frac{\eta}{h}} (\pi h)^{\frac{d-1}{2}} (\det
  \Hess f_+(z))^{-1/2} \chi^2(z,\eta) a_0^2(z,\eta) |\nabla f_-|(z,\eta)
  j(z,\eta) \, d\eta (1+O(h))\\
&=2h (\pi h)^{\frac{d-1}{2}} (\det
  \Hess f_+(z))^{-1/2} \chi^2(z,0)a_0^2(z,0)  |\nabla f_-|(z,0)
  j(z,0) (1+O(h))
\end{align*}
Since $\chi(z,0)=1$, $a_0(z,0)=1$, $\Hess f_+(z)=\Hess f|_{\partial \Omega}(z)$ and
$j(z,0)=1$,
this concludes the proof of~\eqref{eq:laplace}, and thus of~\eqref{laplacewkb}-\eqref{cwkbi}.

Now, writing  
$$
u^{(1)}_{z,wkb} = e^{-\frac{\Phi-f(z)}{h}} \left[d(f-\Phi)\wedge a(\cdot,h)
+h \, d  a(\cdot,h)\right],$$
and noticing that $\Phi-f(z)\geq 0$ on ${\rm supp} \chi$, one has
$\Vert \chi  u^{(1)}_{z,wkb}\Vert_{H^1( \Omega)} \leq Ch^{-1}$.

It remains to prove the last statement. Using the fact that supp$[\Delta_{f,h}^{(1)},\chi]\subset {\rm supp}\, \chi$, $\Phi-f(z)= \Psi\ge c'>0$ on supp$\nabla \chi$ and~(\ref{WW}), one gets 
 \begin{align*}
\Delta_{f,h}^{(1)}(\chi u^{(1)}_{z,wkb})&=\chi \Delta_{f,h}^{(1)}\left( u^{(1)}_{z,wkb}\right)+[\Delta_{f,h}^{(1)},\chi]\left(u^{(1)}_{z,wkb} \right) \\
&=O(h^{\infty})+ O(e^{-\frac ch})=O\left(h^{\infty}\right).
\end{align*}
Therefore, from the Cauchy-Schwartz inequality, one has $\langle \chi
u^{(1)}_{z,wkb}, \Delta_{f,h}^{(1)}(\chi
u^{(1)}_{z,wkb})\rangle_{L^2(\dot \Omega)}=O(h^{\infty})$. The fact
that $\mathcal Q_{f,h}^{M,(1)} ( \dot \Omega )(\chi
u^{(1)}_{z,wkb})=O(h^{\infty})$ then follows from an integration by
parts and the boundary conditions in~\eqref{WW}.
\end{proof}

\subsubsection{A first estimate of the accuracy of the WKB approximation}\label{sec:WKB_first_estim}

Recall that $z\in \{z_1,\ldots,z_n\}$ is a local minimum of
$f|_{\partial \Omega}$ and that $u^{(1)}_{h}$ is a $L^2$-normalized
eigenform of $\Delta^{M,(1)}_{f,h}(\dot \Omega)$
associated with the eigenvalue $0$. The objective of this section is
to prove that $u^{(1)}_{h}$ is accurately approximated by the function
$u^{(1)}_{z,wkb}$ defined in the previous section. The computations
below are inspired by those made in \cite[Chapter~4]{helffer-nier-06} where the authors were adapting \cite{helffer-sjostrand-84,helffer-88}
to manifolds with boundary. The novelty is that we compare the two
$1$-forms in a neighborhood of $B_{z}$, instead of a neighborhood of $z$.

Take two smooth open sets $\Gamma_{St}\subset
\Gamma_{St}'\subset\Gamma_{1}$ which are
strongly stable (see Definition~\ref{stronglystable} and Proposition~\ref{stronglystableexis}) and such that, for some positive $\varepsilon$, $(\Gamma_{St}+B(0,\varepsilon))
\cap \partial \dot{\Omega} \subset \Gamma_{St}'$ and $(\Gamma_{St}'+B(0,\varepsilon))
\cap \partial \dot{\Omega} \subset \Gamma_{1}$, see Figure~\ref{fig:representation_des_gamma}. The fact that
$\Gamma_{St}$ and $\Gamma_{St}'$ are strongly stable and thus that $V_{\Gamma_{St}}$ and $V_{\Gamma_{St}}'$ are stable under the dynamics \eqref{eq:flow_Phi} -see below-  will actually
be needed only to get refined estimates in Section~\ref{sec:wkb_precis}.
Let us now consider the system of coordinate $(x',x_d)$ 
(see~Definition~\ref{def:x'xd}) which is well defined on
$V_{\Gamma_1}$ by assumption (see item 2 in
Proposition~\ref{prop:dotomega}). Let us introduce  the Lipschitz sets
$V_{\Gamma_{St}}$ and $V_{\Gamma_{St}'}$
\begin{equation}
\label{eq.Omega-1-rectangle}
V_{\Gamma_{St}} = \{(x',x_d)\in\Gamma_{St}\times(-a,0)
\}
\text{ and } V_{\Gamma_{St}'} = \{(x',x_d)\in\Gamma_{St}'\times(-a',0)
\}
\end{equation}\label{page.vgammast}
where $0< a < a'$ are small enough so that $V_{\Gamma_{St}} \subset
V_{\Gamma_{St}'} \subset V_{\Gamma_1}$. By construction,
there exists $\varepsilon > 0$ such that $V_{\Gamma_{St}} +
B(0,\varepsilon) \subset V_{\Gamma_{St}'} $ and $V_{\Gamma_{St}'} +
B(0,\varepsilon)\subset V_{\Gamma_1}  \cap (\dot \Omega \cup
\Gamma_1)$ (see again Figure~\ref{fig:representation_des_gamma} for a
schematic representation of these sets). In addition
$V_{\Gamma_{St}}\cap \Gamma_1=\Gamma_{St}$ and $V_{\Gamma_{St}'}\cap \Gamma_1=\Gamma_{St}'$.
Moreover, $a$ and $a'$ can be
chosen sufficiently small so that the sets $\overline{V_{\Gamma_{St}}}$ and $\overline{V_{\Gamma_{St}'}}$ are stable under the dynamics 
\begin{equation}\label{eq:flow_Phi}
 x'(t)= \left\{
\begin{aligned}
 -\nabla \Phi (x(t))\ &\text{ on } 
\Omega   \\
-\nabla_T \Phi (x(t))  &\text{ on } 
\partial \Omega .
\end{aligned}\right.
\end{equation}
%This means that if $x(0)\in V_{\Gamma_{St}}$ (resp. $V_{\Gamma_{St}'}$), then $\forall t\geq 0$,  $x(t)\in  V_{\Gamma_{St}}$
%(resp. $V_{\Gamma_{St}'}$). 
\begin{sloppypar} 

\noindent
This stability is a consequence of two
facts. First, for $x(t)$ solution to~\eqref{eq:flow_Phi}, $\frac{d}{dt}
f_{-}(x(t))'=-\vert \nabla f_{-}(x(t))\vert^2$ (since $\nabla \Phi
\cdot \nabla f_-=|\nabla f_-|^2$ on $V_{\Gamma_1}$, thanks
to~\eqref{eq:f+f-}, and $\nabla_T \Phi
\cdot \nabla f_-=0$ in $\partial \Omega$) so that $\forall t\geq 0,\, x_d(x(0))\leq x_d(x(t))\leq
0$. Second, by construction, for sufficiently
small $a$ and $a'$,
$$\forall x \in \partial V_{\Gamma_{St}} \text{ such that } x'(x)
\in \partial \Gamma_{St}, \, \nabla \Phi(x) \cdot n_x(V_{\Gamma_{St}})>0$$
(where $n(V_{\Gamma_{St}})$ is the unit ouward normal to
$V_{\Gamma_{St}}$). Indeed for any $z \in \partial \Gamma_{St}$, $\lim_{\sigma
  \to z} n_\sigma(V_{\Gamma_{St}})=n_z(\Gamma_{St})$ (where the limit
is taken for $\sigma \in \partial V_{\Gamma_{St}}$ with $x'(\sigma)
\in \partial \Gamma_{St}$),
see~\eqref{eq:limit_nsigma} for a proof, and since
 $\Gamma_{St}$ is chosen strongly
stable, for $z \in \partial \Gamma_{St}$, $\nabla \Phi(z) \cdot
n_z(V_{\Gamma_{St}})=(\nabla f_+ + \nabla f_-) \cdot
n_z(\Gamma_{St})=\nabla f_+ \cdot n_z(\Gamma_{St})=\nabla f|_{\partial
  \Omega}\cdot n_z(\Gamma_{St}) > 0$. The argument is of course the same for $V_{\Gamma_{St}'}$.
\end{sloppypar}

% and some piecewise smooth open sets $V_{\Gamma_{St}} \subset\subset V_{\Gamma_{St}'}\subset \subset V_{\Gamma_{1}} $
% whose common boundaries with $\Gamma_{1}$ are respectively $\Gamma_{St}$ and 
% $\Gamma_{St}'$.
%\begin{equation}
%\label{eq.Gamma-1-stable}
%\Gamma_{St} \ \subset\subset\ \Gamma_{1}\ 
% \subset\subset\  B_{z}.\end{equation}
%for some arbitrarily small but fixed $\eta>0$ and
%some set $\Gamma_{St}$ stable  under the flow of $-\nabla_{T} f$. 
%We consider moreover some open set $ \widetilde{V_{\Gamma_{1}}}\subset \subset V_{\Gamma_{1}} $
%of $\dot\Omega$ whose boundary along $\pa\Omega$
%is some
%smooth $n-1$ manifold $\widetilde{\Gamma_{1}}$
%satisfying 
%\begin{equation}
%\label{eq.Gamma-tilde-stable}
%\Gamma_{St} \ \subset\ \widetilde{\Gamma_{1}}\  \subset\ 
% \Gamma_{St}^{\eta'},
%\end{equation}
%where $0<\eta'<\frac\eta2$ is chosen small enough so that

Let us now introduce two smooth cut-off functions $0\leq\chi\leq\eta\in
\mathcal{C}^{\infty}_c(\dot \Omega \cup \Gamma_{1})$
satisfying
\begin{align}
\label{eq.cut-off}
&\chi=1  \text{ in a neighborhood of } \overline{V_{\Gamma_{St}}},
\quad
\supp\chi\subset V_{\Gamma_{St}'}
\\
\text{ and } &
\eta=1\text{ in a neighborhood of } \overline{V_{\Gamma_{St}'}},\quad 
\supp \eta\subset V_{\Gamma_{1}} \cap (\dot \Omega \cup \Gamma_1).\label{eq.cut-off-2}
\end{align}
Notice that by construction, $\eta=0$ on $\Gamma_2$.
In the following, we moreover assume that $\chi$ and $\eta$ are tensor
products in the system of coordinates $(x',x_d)$ (this will actually
be needed only to get refined estimates in Section~\ref{sec:wkb_precis}):
\begin{equation}\label{eq.chi-separate}
\chi(x',x_d)=\chi_1(x')\chi_d(x_d) \text{ and } \eta(x',x_d)=\eta_1(x')\eta_d(x_d).
\end{equation}\label{page.chi1eta1}
Let $\kappa\in\{\chi,\eta\}$\label{page.kappa}. Owing to Lemma~\ref{le.Gaffney}, the $1$-form  $\kappa u^{(1)}_h$ belongs to $\Lambda^{1}H_{T}^{1}(\dot \Omega)$. The first a priori estimate on $\kappa(u^{(1)}_h-c(h)u^{(1)}_{z,wkb})$  is the following:
\begin{proposition} \label{apriori}
 Let us assume that the hypotheses \textbf{[H1]}, \textbf{[H2]} and \textbf{[H3]} hold.
For $\kappa\in\{\chi,\eta\}$, one has
\begin{equation}
\label{eq.a-priori-comparison}
\left\|\kappa(u^{(1)}_h-c_z(h)u^{(1)}_{z,wkb})\right\|_{ H^{1}(\dot \Omega)}=O (h^{\infty})
\end{equation}
where 
\begin{equation}\label{eq:czh_def}
 c_z(h)^{-1}=\langle  u^{(1)}_h, \chi u^{(1)}_{z,wkb}
\rangle_{L^2(\dot \Omega)}.
\end{equation}
 The $1$-form $u^{(1)}_h$ can be chosen
such that $c_z(h)>0$.
Additionally, when $h\to 0$
\begin{equation}\label{eq:czh}
 c_z(h) =C_{z,wkb}^{-1}h^{-\frac{d+1}{4}}(1+O(h^{\infty})),
\end{equation}\label{page.czh}
where $C_{z,wkb}$ is defined by~\eqref{cwkbi}.
\end{proposition}
Notice that $|c_z(h)|^{-1}$ is equivalent (in the limit $h \to 0$) to
$\|\kappa u^{(1)}_{z,wkb}\|_{L^2(\dot \Omega)}$ (see~\eqref{laplacewkb}), and can thus be simply
understood as a normalizing factor.
\begin{proof} Let us first consider the case $\kappa=\chi$, the other case is considered at the end of the proof. 
One defines $$k(h):= \langle  u^{(1)}_h, \chi u^{(1)}_{z,wkb}
\rangle_{L^2(\dot \Omega)} \in \mathbb R.$$ If $k(h)<0$, then one
changes $u^{(1)}_h$ to $-u^{(1)}_h$ so that one can suppose without
loss of generality that
$$k(h) \ge 0.$$
 For $h$ small enough, one
has (from Proposition~\ref{pr.DeltaTN}, item $(iii)$)
$$\pi_{[0,ch^{3/2})}( \Delta_{f,h}^{M, (1)}(\dot \Omega) )  ( \chi
u^{(1)}_{z,wkb})= k(h) u^{(1)}_h.$$
Let us define  $$\alpha_h:=\chi(u^{(1)}_{z,wkb}- k(h) u^{(1)}_h).$$ 
Thus, the following identity holds for $h$ small enough
$$\alpha_h=  k(h) \left(1-\chi\right)u^{(1)}_h+\pi_{[h^{3/2}, +\infty]}( \Delta_{f,h}^{M,(1)} ( \dot \Omega ) ) ( \chi u^{(1)}_{z,wkb}).$$
Notice that, using Cauchy-Schwarz inequality and Lemma~\ref{wkb1}, there exist $C>0$ and  $h_0>0$ such that for all $ h\in (0,h_0)$
$$\vert k(h)\vert \leq Ch^{\frac{d+1}{4}}.$$

Therefore, using Lemma~\ref{quadra}, Proposition~\ref{pr.Agmon} and Lemma~\ref{wkb1} we get
\begin{align*}
\Vert  \alpha_h \Vert^2_{L^2(\dot{\Omega})} &\leq 2  k(h)^2 \Vert
\left(1-\chi\right)u^{(1)}_h\Vert^2_{L^2(\dot{\Omega})}  + 2 \left\Vert \pi_{[ch^{3/2}, +\infty]}\left ( \Delta_{f,h}^{M,(1)} ( \dot \Omega ) \right) ( \chi u^{(1)}_{z,wkb}) \right\Vert^2_{L^2(\dot{\Omega})}\\
&\leq Ch^{\frac{d+1}{2}}  \left\Vert \left(1-\chi\right)u^{(1)}_he^{ \frac{ \Psi}{h}}e^{ \frac{ -\Psi}{h}}\right\Vert^2_{L^2(\dot{\Omega})}   +   C h^{-3/2} \mathcal Q_{f,h}^{M,(1)} ( \dot \Omega )(\chi u^{(1)}_{z,wkb})\\
&\leq Ch^{\frac{d+1}{2}} h^{-N_0} e^{-\frac{c}{h} }  +   C h^{-3/2} \mathcal Q_{f,h}^{M,(1)} ( \dot \Omega )(\chi u^{(1)}_{z,wkb})\\
&= O(h^{\infty}),
\end{align*}
 with $c:= \inf_{{\rm supp}( 1-\chi)}  \Psi>0$ (since $\chi=1$ near $z$) and the integer $N_0$ is given by Proposition  \ref{pr.Agmon}. 
Moreover, since $d_{f,h} =hd+df\wedge$ and $d_{f,h}^* =hd^*+\mbf{i}_{\nabla
  f}$, one obtains  using the triangular inequality, the Gaffney
inequality~\eqref{eq.Gaffney} (since $\alpha_h \in H^1_T(\dot
\Omega)$), the fact that $\mathcal Q_{f,h}^{M,(1)} ( \dot \Omega
)(\chi u^{(1)}_h)=O(e^{-\frac{c}{h}})$ (from the Agmon estimate~\eqref{eq.Agmon}) and $\mathcal Q_{f,h}^{M,(1)} ( \dot \Omega
)(\chi u^{(1)}_{z,wkb})=O(h^\infty)$,
\begin{align*}
\Vert \alpha_h\Vert_{H^1(\dot \Omega)}^2
&\leq C( \Vert d\alpha_h\Vert^2_{L^2(\dot \Omega)}+\Vert d^*\alpha_h\Vert^2_{L^2(\dot \Omega)}+\Vert \alpha_h\Vert^2_{L^2(\dot \Omega)} )\\
&\leq Ch^{-2} \left(\mathcal Q_{f,h}^{M,(1)} ( \dot \Omega )(\alpha_h)
  + \Vert \alpha_h \Vert^2_{L^2(\dot \Omega)} \right) \\
&= O(h^{\infty}).
\end{align*}
Moreover since $ \Vert \chi u^{(1)}_h
\Vert_{L^2(\dot \Omega)} =1+O(e^{-\frac ch})$ (from the Agmon estimate~\eqref{eq.Agmon}), by considering $ \Vert \chi (  u^{(1)}_{z,wkb}-  k(h) u^{(1)}_h
)\Vert_{L^2(\dot \Omega)} =O(h^{\infty})$, one gets:  
\begin{align*}
 k(h) ^2  & = \frac{\Vert \chi u^{(1)}_{z,wkb}
  \Vert^2_{L^2(\dot \Omega)} + O(h^{\infty})}{2-\Vert \chi u^{(1)}_h\Vert^2_{L^2(\dot \Omega)}}\\
&= \frac{C^2_{z,wkb} h^{\frac{d+1}{2}} +O(h^{\infty})}{1+O(e^{-\frac ch})},
\end{align*}
with $C_{z,wkb} $  given by~\eqref{cwkbi} in  Lemma~\ref{wkb1}. Therefore, since $k(h)
\ge 0$, 
$k(h)=C_{z,wkb} h^{\frac{d+1}{4}}(1+O(h^{\infty}))$. This concludes
the proof of~\eqref{eq.a-priori-comparison} for $\kappa=\chi$, by choosing
$$c_z(h):=k(h)^{-1}.$$

Let us now deal with the case $\kappa=\eta$. There exists $c>0$ such
that, for $h$ sufficiently small,
 \begin{align*}
 \Vert \eta(u^{(1)}_{z,wkb}- k(h) u^{(1)}_h) \Vert_{H^1(\dot \Omega)} &\leq \Vert \alpha_h \Vert_{H^1(\dot \Omega)}+\Vert (\eta-\chi)(u^{(1)}_{z,wkb}- k(h) u^{(1)}_h) \Vert_{H^1(\dot \Omega)} \\
 &\leq O(h^{\infty}) +\Vert (\eta-\chi)u^{(1)}_{z,wkb}\Vert_{H^1(\dot \Omega)} + \vert k(h)\vert \Vert (\eta-\chi) u^{(1)}_h \Vert_{H^1(\dot \Omega)}\\
 &\leq O(h^{\infty}) + e^{-\frac ch}.
 \end{align*}
The last inequality is the consequence of two facts. First, $\Vert
(\eta-\chi) u^{(1)}_h \Vert_{H^1(\dot \Omega)}=e^{-\frac ch}$ thanks
to Proposition~\ref{pr.Agmon} and~\eqref{eq.Gaffney} together with the fact that $\chi= \eta$ near $z$. Second, a direct computation shows that 
$$\Vert (\eta-\chi)u^{(1)}_{z,wkb}\Vert_{H^1(\dot \Omega)}\leq Ch^{-1}e^{-\frac{\inf_{ {\rm supp}(\eta-\chi)  }  \Psi}{h}}\leq e^{-\frac ch} .$$ 
This concludes the proof of Proposition~\ref{apriori}. 
\end{proof}

The estimate we obtained in Proposition~\ref{apriori} is sufficient to get
the result of Theorem~\ref{TBIG0}. The more precise estimates obtained in 
Section~\ref{sec:wkb_precis} are only needed to prove Theorem~\ref{th.gene_sigma}.

\subsubsection{A more accurate comparison on the WKB approximation}\label{sec:wkb_precis}

The objective of this section is to combine the techniques used to
obtain the Agmon estimates of Proposition~\ref{pr.Agmon} and the first
estimate of the accuracy of the WKB approximation of
Proposition~\ref{apriori} in order to obtain a more precise estimate of the latter.
\medskip

\noindent
Let us start with estimates which are simple consequences of Propositions~\ref{pr.Agmon} and~\ref{apriori}. Notice that,  for
$\kappa\in\{\chi,\eta\}$, one obviously gets from
Proposition~\ref{pr.Agmon}  the following relation in
$\Lambda^1H^1(\dot \Omega)$:
\begin{equation}
\label{eq.a-priori-comparison2}
\exists N_{0}\in\mathbb N ,\,
 e^{\frac{ \Psi}{h}} \kappa(u^{(1)}_h-c_z(h)u^{(1)}_{z,wkb})\ = O (h^{-N_{0}}).
\end{equation}
For the term involving $u^{(1)}_{z,wkb}$, this is due to $\Psi(x)=\Phi(x)-f(z)$ on
$\supp \kappa$ and the estimate~\eqref{eq:czh} on $c_z(h)$.
Let us now set 
\begin{equation}
\label{eq.wh}
w_{h}:=\kappa(u^{(1)}_h-c_z(h)u^{(1)}_{z,wkb}).
\end{equation}\label{page.wh}
The $1$-form $w_{h}$ is in $C^\infty_c(\dot \Omega \cup \Gamma_1)$ and  satisfies in $\dot{\Omega}$:
\begin{align}
\nonumber
\Delta_{f,h}^{(1)}w_{h}&=\kappa \Delta_{f,h}^{(1)}(
u^{(1)}_{h}-c_z(h)u^{(1)}_{z,wkb})
+[\Delta_{f,h}^{(1)},\kappa](u^{(1)}_{h}-c_z(h)u^{(1)}_{z,wkb})\\
\nonumber
&=  -c_z(h)\kappa \Delta_{f,h}^{(1)}u^{(1)}_{z,wkb}
+[\Delta^{(1)}_{f,h},\kappa] (u^{(1)}_{h}-c_z(h)u^{(1)}_{z,wkb})\\
\label{eq.Delta-r1}
& = (r_1+r_{1}') e^{-\frac{ \Psi}{h}}\;,
\end{align}
where, owing to \eqref{WW} and \eqref{eq:czh}:
\begin{equation}
\label{eq.r1}
r_1:= -   e^{\frac{ \Psi}{h}} c_z(h)\kappa
\Delta_{f,h}^{(1)}u^{(1)}_{z,wkb}=O(h^\infty)
\end{equation}
in $\Lambda^{1} L^2(\Omega)$ and, from~\eqref{eq.a-priori-comparison2}:
\begin{align}
\nonumber
&r_1':= e^{\frac{
    \Psi}{h}} \, [\Delta^{(1)}_{f,h},\kappa] (u^{(1)}_{h}-c_z(h)u^{(1)}_{z,wkb}) = O
(h^{-N_{0}})\ {\rm in }\, \Lambda^{1} L^2(\Omega) \text{ and }\\
\label{eq.r1'}
&\supp r_1'\subset \supp \nabla \kappa.
\end{align}
%supp$[\Delta_{f,h}^{(1)},\chi]\subset {\rm supp}\, \chi$
Additionally, one gets similarly on the boundary $\Gamma_{1}$:
$$  
\mathbf{t}w_{h}|_{\Gamma_{1}}=0
\text{ and }
\mathbf{t}d_{f,h}^{*}w_{h}|_{\Gamma_{1}}
= (r_2 +r'_{2})e^{-\frac{ \Psi}{h}}
=(r_2 +r'_{2})e^{-\frac{f-f(z)}{h}},  
$$
where   owing to \eqref{WW} and \eqref{eq:czh}:
\begin{equation}
\label{eq.r2}
 r_2 :=\mathbf{t} e^{\frac{
    \Psi}{h}}\kappa \, d^*_{f,h} (u^{(1)}_{h}-c_z(h)u^{(1)}_{z,wkb}) |_{\Gamma_{1}} =-
 \mathbf{t} e^{\frac{
    \Psi}{h}} \, \kappa\, 
 c_z(h) \, d^*_{f,h} u^{(1)}_{z,wkb}|_{\Gamma_{1}}=O(h^\infty)
\end{equation}
in  $ L^{2}(\pa\Omega)$ and
\begin{align}
\nonumber 
&r_2':= \mathbf{t} e^{\frac{
    \Psi}{h}} h \, \mbf i_{\nabla \kappa } (u^{(1)}_{h}-c_z(h)u^{(1)}_{z,wkb})|_{\Gamma_{1}}= O (h^{-N_{0}})\ {\rm in } \, L^2(\pa \dot \Omega) \, \text{with}\\
    \label{eq.r2'}
 & \supp r_2'\subset \Gamma_{1}\cap\supp \nabla \kappa.
\end{align}\label{page.r}
%The Green formula for the Witten Laplacian
%that will be used in addition of \eqref{eq.intbypartphi} is:
% for every $u\in \Lambda^{p}H^{2}(\Omega)$ and $v\in \Lambda^{p}H^{1}(\Omega)$, one has
%\begin{multline}\label{eq.Green}
% \langle d_{f,h}u,d_{f,h}v\rangle_{\Lambda^{p+1}L^{2}}+
%\langle d_{f,h}^{*}u,d_{f,h}^{*}v\rangle_{\Lambda^{p-1}L^{2}} \\
%=\langle \Delta_{f,h}u,v\rangle_{ L^{2}}
% + h
%\int_{\pa\Omega}\!(\mathbf{t}\overline{v})\wedge
%(\star\mathbf{n}d_{f,h}u) 
%- h \int_{\pa\Omega}\!(\mathbf{t}d_{f,h}^{*}u)\wedge(\star \mathbf{n}
%\overline{v}).
%\end{multline}

\noindent
We are now in position to prove the following proposition. 
%Roughly speaking, having in mind Lemma~\ref{le.intbypart},  this is the term $O(h^{-N_{0}})$ in the last computations which prevents us to get the estimate $e^{\frac{ \Psi}{h}}w_{h}=O(h^{\infty})$ in $\Lambda^1H^1(\Omega)$ with reasonable efforts (compared those made for the proof of the following Proposition). However at the stage we have reached, we don't know how to easily get ride of the term $O(h^{-N_{0}})$. However we expect the estimate $e^{\frac{ \Psi}{h}}w_{h}=O(h^{\infty})$ in $\Lambda^1H^1(\Omega)$ to hold and this is the purpose of the following Proposition. 
\begin{proposition}
\label{pr.WKB-comparison} Let us assume that the hypotheses \textbf{[H1]}, \textbf{[H2]} and \textbf{[H3]} hold.
One has the following estimate in the limit $h\to0$:
\begin{equation}
\label{eq.WKB-comparison}
\left \|e^\frac{ \Psi}{h}(u^{(1)}_{h}-c_z(h)u^{(1)}_{z,wkb})\right \|_{ H^1(V_{\Gamma_{St}})}
= 
O(h^{\infty}),
\end{equation}
where $c_z(h)$ is defined by~\eqref{eq:czh_def} and where, we recall, $\Psi(x)=d_a(x,z)$ and $V_{\Gamma_{St}}$ is defined by~\eqref{eq.Omega-1-rectangle}.
\end{proposition}

\begin{proof}
As for the proof of Proposition~\ref{pr.Agmon}, 
one first proves an estimate along the boundary $\Gamma_1$  before
propagating it
in $V_{\Gamma_{St}}$.

\medskip
\noindent
\underline{Step 1.} Comparison in $\Gamma_{1}$.
\medskip

\noindent
Let us consider $w_h$ defined by~\eqref{eq.wh} and the cut-off function $\kappa=\eta$
defined in~\eqref{eq.cut-off-2}. 
Like in the first step of the proof of 
Proposition~\ref{pr.Agmon}, we are going to prove an estimate
of the form \eqref{eq.WKB-comparison} with $ \Psi$ replaced by $f_{+}$.
More precisely, we want to show that 
\begin{equation}
\label{eq.estimtoprove}
\left\|e^{\frac{f_{+}}{h}}w_{h}\right\|_{H^{1}(V_{\Gamma_{St}'})}=
\left\| e^\frac{f_{+}}{h}(u^{(1)}_{h}-c_z(h)u^{(1)}_{z,wkb})\right\|_{ H^{1}(V_{\Gamma_{St}'})}=
\mathcal O(h^\infty),
\end{equation}
which implies in particular the following estimate along the boundary,
since
$
f_{+}=f-f(z)
$
in $\Gamma_{1}$,
\begin{equation}
\label{eq.step1comparison}
\left\|e^\frac{f-f(z)}{h}(u^{(1)}_{h}-c_z(h)u^{(1)}_{z,wkb})
\right\|_{ H^{1/2}(\Gamma_{St}')}
=\mathcal O(h^{\infty}).
\end{equation}
In the following, we denote (see Figure~\ref{fig:representation_des_gamma} for a
schematic representation of the set $V_\eta$) 
$$V_\eta=\supp \eta.$$
\label{page.Veta}
In the system of coordinates $(x',x_d)$, $x \in V_\eta$ if and
only if $x'(x) \in \supp \eta_1$ and $x_d(x) \in \supp \eta_d$.
We recall that $V_\eta$ is a compact set of $\dot \Omega \cup \Gamma_1$.
As for the proof of 
Proposition~\ref{pr.Agmon}, we  introduce the sets
$$ \Omega_{-}=\left\{x\in V_\eta\ \text{s.t.}\  f_{+}(x) \leq Ch\right\}
\text{ and }
\Omega_{+}=V_\eta\setminus \Omega_{-},
$$
and define the Lipschitz function $\varphi:V_\eta \to \R$ by
$$
\varphi=
\left\{
\begin{aligned}
f_{+}-Ch\ln \frac{f_{+}}{h}   &\text{ if } f_{+} > Ch,\\
f_{+}-Ch\ln C &\text{ if } f_{+} \leq Ch ,
\end{aligned}
\right.
$$
for some constant $C> 1$ that will be fixed at the end of this
step. Notice for further purposes that 
\begin{equation}\label{eq:cv_unif_phi}
\lim_{h \to 0} \|\varphi - f_+\|_{L^\infty(V_\eta)}=0.
\end{equation}
 We recall that in the system of coordinates $(x',x_d)$,
$\varphi$ is independent of $x_d$.

The reasoning below is based on~\cite{helffer-sjostrand-84}, see also~\cite[p.~49--52]{helffer-88}
for a presentation in the case without boundary.
According to~\eqref{eq.estimtoprove}, we want to get an error of the form $\mathcal O(h^{N})$ with $N$ arbitrary.
We are not going to work with  the above phase function
$\varphi$ as we did in the proof of Proposition~\ref{pr.Agmon},
but with a phase function $\varphi_{N}$ also  depending 
on some arbitrary $N\in\nz$.  
%Since moreover
%$\left\|e^{\frac{\varphi}{h}}w_{h}\right\|_{H^{1}(V_{\Gamma_{St}'})}=\mathcal O(h^{-N_{1}})$
%for some $N_{1}\in\mathbb Z$
%according to \eqref{eq.a-priori-comparison2},
%\eqref{eq.estimtoprove} also suggests to choose
%$\varphi_{N}$ such that $\varphi_{N}=\varphi+Nh\ln \frac1h$ in $V_{\Gamma_{St}'}$.
Let us define
$$\tilde{w}_{h} =e^{\frac{\varphi_{N}}{h}}w_{h}= e^{\frac{\varphi_{N}}{h}}\eta(u^{(1)}_h-c_z(h)u^{(1)}_{z,wkb}).$$\label{page.tildewh}
Combining the integration by parts formula~\eqref{eq.intbypartphi}
(with $u=w_h$ and $\varphi=\varphi_N$)
with the Green formula~\eqref{eq.Green} (with $u=w_h$ and $v=e^{2\frac{\varphi_{N}}{h}}w_h$)  leads to the estimate
\begin{align}
 \left\|e^{\frac{\varphi_{N}}{h}} \Delta^{(1)}_{f,h}
 w_{h}\right\|_{ L^{2}(V_\eta)} &\left\|\tilde{w_{h}}\right\|_{ L^{2}(V_\eta)}
+ h
\left\| e^{\frac{\varphi_{N}}{h}}
\mathbf{t}d^*_{f,h}w_{h}
\right\|_{ L^{2}(\Gamma_{1})}\left\|\tilde{w}_{h}\right\|_{ L^{2}(\Gamma_{1})} 
\nonumber
\\
&
\geq 
\left\|hd\tilde{w}_{h}\right\|^{2}_{ L^{2}(V_\eta)}+\left\|hd^{*}\tilde{w}_{h}\right\|^{2}_{ L^{2}(V_\eta)}
-h\int_{\Gamma_{1}}\langle\tilde{w}_{h},\tilde{w}_{h}
\rangle_{ T_{\sigma}^{*}\Omega}\pa_{n} f\, d\sigma
\nonumber
\\
&\quad
+
\langle (|\nabla
f|^{2}-|\nabla\varphi_{N}|^{2}+h\mathcal L_{\nabla f}+
h\mathcal L^*_{\nabla f}) 
\tilde{w}_{h},\tilde{w}_{h}\rangle_{L^2(V_\eta)}  .
\label{eq.first-energy}
\end{align}
\begin{sloppypar}
Let us explain formally how the function $\varphi_N$ is chosen.  Roughly speaking, using similar arguments as in the proof of Proposition~\ref{pr.Agmon}, it is natural to choose $\varphi_{N}=\varphi+Nh\ln \frac1h$  and to try to prove that  the left-hand side of
\eqref{eq.first-energy} is bounded from above by 
$\mathcal O(h^{-N_{1}})\left\|\tilde{w_{h}}\right\|_{ H^{1}(V_\eta)}$ for some $N_{1}\in\mathbb N$ independent of~$N$. This would indeed 
 lead to an estimate of the form $\left\|\tilde{w_{h}}\right\|_{ H^{1}(V_\eta)}=\mathcal O(h^{-N_{1}})$
(for some maybe larger~$N_1$) and finally to the desired estimate on $w_h$ since $\|\tilde w_h\|_{H^1(V_\eta)} \simeq h^{-N} \|e^{\frac{f_+}{h}} w_h\|_{H^1(V_\eta)}$. To get this upper bound,  a trace theorem and  \eqref{eq.Delta-r1}--\eqref{eq.r2'} yield the following estimate from~\eqref{eq.first-energy}:
$$\|e^{\frac{\varphi_N-\Psi}{h}} \|_{L^\infty(V_\eta)} O(h^\infty) + \|e^{\frac{\varphi_N-\Psi}{h}} \|_{L^\infty({\rm supp} (\nabla \eta))} O(h^{-N_1}) \ge \|\tilde w_h\|_{H^1(V_\eta)} \simeq h^{-N} \|e^{\frac{f_+}{h}} w_h\|_{H^1(V_\eta)}$$
for some  $N_{1}\in\mathbb N$ independent of~$N$. It can be checked that $\|e^{\frac{\varphi_N-\Psi}{h}} \|_{L^\infty(V_\eta)}= \|e^{\frac{\varphi_N-\Psi}{h}} \|_{L^\infty({\rm supp} (\nabla \eta))} = O(h^{-N})$ so that the first term is well controlled, but the second one is of order $O(h^{-N-N_1})$. These relations suggest
 a choice of $\varphi_{N}$ satisfying
$\varphi_{N}\leq  f_+ \leq \Psi$ on $\supp\nabla\eta$ so that 
$\|e^{\frac{\varphi_N-\Psi}{h}} \|_{L^\infty({\rm supp} (\nabla \eta))} = O(1)$. This would yield the desired estimate $\|e^{\frac{f_+}{h}} w_h\|_{H^1(V_\eta)}=O(h^{N-N_1})$.
\end{sloppypar}
%. This suggests to choose $\varphi_{N}$
%of the form $\varphi_{N}=\varphi+Nh\ln \frac1h$ on $V_{\Gamma_{St}'}\subset\{\eta=1\}$ in order to get \eqref{eq.estimtoprove}. Note that this
%choice is compatible with a control of the left-hand side of  \eqref{eq.first-energy}
%by $\mathcal O(h^{-N_{1}})\left\|\tilde{w_{h}}\right\|_{ H^{1}(V_\eta)}$ according to \eqref{eq.Delta-r1}--\eqref{eq.r2'}.
%Nevertheless, such a choice of $\varphi_N$ on $\supp \nabla\eta$ would
%lead  to some control of the left-hand side of  \eqref{eq.first-energy}
% by $\mathcal O(h^{-N_{1}-N})\left\|\tilde{w_{h}}\right\|_{ H^{1}}$ for some $N>0$, according again to 
%\eqref{eq.Delta-r1}--\eqref{eq.r2'}. These relations suggest
% a choice of $\varphi_{N}$ satisfying
%$\varphi_{N}\leq  f_+$ on $\supp\nabla\eta$.

Let us now enter the rigorous proof. The above considerations (see also \cite[p.~49--52]{helffer-88})
lead to define, for any $N\in\nz$,
\begin{equation}\label{eq:phi_N}
\varphi_{N}=\min \left\{ \varphi+Nh\ln \frac1h , 
  \psi \right\},
\end{equation}
where the Lipschitz function $\psi:V_\eta \to \R$ is defined by the following
relation, for some $\varepsilon\in (0,1)$
that will be specified below:
\begin{equation}\label{eq:psi_phi}
\psi(x',x_d)=\psi(x',0)=\min\left\{\varphi(y',0)+(1-\varepsilon) d^{\pa\Omega}_{a}(x',y'),\
  y'\in \supp \nabla\eta_{1}\right\} .
\end{equation}
Here, $d^{\pa\Omega}_{a}(x',y')$ denotes the Agmon distance
associated with $f|_{\pa\Omega}$
between $x'$ and $y'$ along the boundary (see Definition~\ref{defff}),
i.e. the distance induced by the metric
$|\nabla(f|_{\pa\Omega})|^{2} \,ds^2$,
where $ds^2$ denotes the restriction of the Euclidean metric to the boundary $\pa\Omega$% (see \cite[p.~53]{dimassi-sjostrand-99} for further details)
.

\medskip
\noindent
\underline{Step 1-a:} Preliminary estimates on $\varphi_N$.
\medskip

\noindent
Let us first show that there exists $\varepsilon \in (0,1)$ such that
for any $h\in (0,h_0(N,\varepsilon))$ with $h_{0}=h_0(N,\varepsilon)$ small enough,
\begin{equation}
\label{eq.psi>phi}
\varphi_{N}=\varphi+Nh\ln \frac1h <
\psi  \text{ in }\ V_\eta\cap\left\{x'\in \overline{\Gamma'_{St}}\right\}.
\end{equation}

% \comment{Utiliser ou ??? Note that $\psi(x',0)\leq\varphi(x',0)$ for $x' \in \supp\nabla\eta_1$.}

The proof of~\eqref{eq.psi>phi} is as follows. From~\eqref{eq:ineq}
applied to $d^{\pa\Omega}_{a}$,
% \comment{Argument precedent plus complique ?
% Since $d^{\pa\Omega}_{a}(\cdot,z) = f|_{\pa\Omega}-f(z)=f_{+}$ in $\Gamma_{1}$ (this is a consequence of Corollary~\ref{ccc} and the fact that $\Gamma_1
%  \subset B_z$,
%  the triangular inequality applied to $d^{\pa\Omega}_{a}$ implies that}
\begin{equation}
\label{eq.triang-ineq}
\forall (x',y')\in \overline{\Gamma_{St}'}\times\supp\nabla\eta_{1}\ ,\quad f_{+}(x',0)
 < f_{+}(y',0)+ d^{\pa\Omega}_{a}(x',y').
\end{equation}
The inequality above is strict since if $f_{+}(x',0)
 = f_{+}(y',0)+ d^{\pa\Omega}_{a}(x',y')$ for some
$(x',y') \in \overline{\Gamma_{St}'}\times\supp\nabla\eta_{1}$, then
there exists a generalized integral curve (in the sense of Definition~\ref{generalized_curve}) of $-\nabla
(f|_{\pa\Omega})=-\nabla f_+$ 
joining $x'\in \overline{\Gamma_{St}'}$ to $y\in\supp\nabla\eta_{1}
$ (this is  a consequence of Corollary~\ref{ccc} applied to
the Agmon distance $d^{\pa\Omega}_{a}$ on $\pa \Omega$ rather than the Agmon distance $d_{a}$ in $\overline \Omega$). But since
$\Gamma'_{St}$ is strongly stable, any integral curve  of $-\nabla
f_+$  remains in $\overline{\Gamma'_{St}}$, and thus cannot reach $y'$ which is
not in $\overline{\Gamma'_{St}}$ (see~\eqref{eq.cut-off-2}).

From the strict inequality (\ref{eq.triang-ineq}), there exists
$\ve_0>0$ such that for all $\ve \in [0,\ve_0)$,
$$
\forall (x',y')\in
\overline{\Gamma_{St}'}\times\supp\nabla\eta_{1}\ ,\quad f_{+}(x',0)
\le f_{+}(y',0)+ (1-\ve)d^{\pa\Omega}_{a}(x',y'),
$$
and thus, considering the limit $h \to 0$ (see~\eqref{eq:cv_unif_phi})
and the infimum over $y'\in \supp\nabla\eta_{1}$ of the right-hand
side, there exists $\ve>0$ used to define $\psi$ (see~\eqref{eq:psi_phi}) and such that, for sufficiently small $h$,
$$
\forall x'\in \overline{\Gamma_{St}'}\ ,\quad f_{+}(x',0)
 < \psi(x',0).
$$
Moreover, since $\lim_{h \to 0} \|\varphi+Nh\ln \frac1h
-f_+\|_{L^\infty(V_\eta)}=0$ (thanks to~\eqref{eq:cv_unif_phi}), one
obtains for $h$ small enough, $\forall x' \in
\overline{\Gamma_{St}'}$, $ \varphi(x',0) + N h \ln \frac1h
 < \psi(x',0)$, and by definition of $\varphi_N$, $\varphi_N(x',0)=
\varphi (x',0)+Nh\ln \frac1h$ which leads to
(\ref{eq.psi>phi}) for $x=(x',0)$, with $x' \in
\overline{\Gamma'_{St}}$. The fact that $\varphi_N$ and $\psi$ do 
not depend on $x_d$ in the system of coordinates $(x',x_d)$ concludes the proof of~\eqref{eq.psi>phi}.

Let us now prove that
\begin{equation} \label{Me}
\exists M< \frac{1}{1-\varepsilon},\,
\forall x \in V_\eta, \left|\nabla \psi(x) \right|\leq
 M(1-\varepsilon)\left|\nabla
f_{+}(x) \right|.
\end{equation}
The triangular inequality applied to $d^{\pa\Omega}_{a}$ leads to the
relation (since $\psi(x',x_d)$ does not
depend on $x_d$)
\begin{equation}
\label{eq.Lip}
\forall x,y \in V_\eta, \, 
|\psi(x',x_d)-\psi(y',y_d)| \leq (1-\varepsilon)d^{\pa\Omega}_{a}(x',y').
\end{equation}
where we denote $(x',x_d)$ (resp. $(y',y_d)$) the coordinates of $x$
(resp. $y$) in the system of coordinates~\eqref{eq:x'xd}. Let us first
show that~\eqref{eq.Lip} implies that
\begin{equation}
\label{eq.psi(x')}
\text{for a.e. $x' \in V_\eta \cap \partial \Omega$,} \, 
\left|\nabla (\psi|_{\pa\Omega}) (x')\right|\leq
(1-\varepsilon)\left|\nabla
(f|_{\pa\Omega}) (x')\right| = (1-\varepsilon)\left|\nabla
(f_+|_{\pa\Omega} )(x') \right| .
\end{equation}
Indeed, let us consider a local parametrization in $\R^{d-1}$ of a
neighborhood in $\partial \Omega$ of a point $x' \in \partial
\Omega$. In this local chart, let us consider $y_\alpha=x'+\alpha
\frac{\displaystyle\nabla (\psi|_{\pa\Omega})}{\displaystyle\left|\nabla
  (\psi|_{\pa\Omega})\right|}(x')$. One has, in the limit $\alpha \to 0$,
$$\psi(y'_\alpha,0)-\psi(x',0)=\alpha |\nabla
(\psi |_{\pa\Omega}) (x')| + o(\alpha)$$
and likewise using the inequality~\eqref{eq:eqq} applied to
$d^{\pa\Omega}_{a}$ (see also \cite[p.~53]{dimassi-sjostrand-99})
$$d^{\pa\Omega}_a(x,'y'_\alpha)-d^{\pa\Omega}_a(x',x')\le\alpha |\nabla
(f|_{\pa\Omega}) (x')| + o(\alpha).$$
By considering the limit $\alpha \to 0$, one thus
deduces~\eqref{eq.psi(x')} from~\eqref{eq.Lip}.

Now, one can check that, uniformly in $x' \in V_\eta \cap \partial
\Omega$,
\begin{equation}\label{eq:limit_xd}
\lim_{x_d \to 0} |\nabla \psi(x',x_d)|=|\nabla (\psi|_{\partial
  \Omega}) (x')| \text{ and } \lim_{x_d \to 0} |\nabla f_+(x',x_d)|=|\nabla ( f_+|_{\partial
  \Omega} )(x')|.
\end{equation}
Indeed, using the fact that $\psi$ does not depend on $x_d$, one first
has almost everywhere (see~\eqref{eq.psi(x')_0}) $|\nabla \psi(x',x_d)|=|\nabla (\psi|_{\Sigma_{x_d}})(x')|$ where
$\forall a>0$, $\Sigma_a=\{x \in V_\eta, \, x_d(x)=a\}$ is endowed
with the Riemannian structure induced by the Riemannian structure in
$\Omega$. Now, let us consider the smooth diffeomorphism $\Gamma_{x_d}:
\Sigma_{x_d} \to \partial \Omega$ such that for all $x=(x',x_d) \in
\Sigma_{x_d}$, $\Gamma_{x_d}(x)=(x',0) \in \partial \Omega$. The
result~\eqref{eq:limit_xd} on $\psi$ is then a consequence of the fact that
$\psi|_{\partial \Omega} \circ \Gamma_{x_d} = \psi|_{\Sigma_{x_d}}$
and $\lim_{x_d \to 0} \| \Gamma_{x_d} - {\rm Id}
\|_{W^{1,\infty}(\Sigma_{x_d})} = 0$ so that the Jacobian associated
to the change of metric from $\Sigma_{x_d}$ to $\partial \Omega$
converges to $ {\rm Id}$, uniformly on $\Sigma_{x_d}$.
The same reasoning show that~\eqref{eq:limit_xd} also holds for $f_+$ since $f_+$ does
not depend on $x_d$.

By combining~\eqref{eq.psi(x')} and~\eqref{eq:limit_xd}, one
obtains~\eqref{Me} for some $M>1$. Moreover, $M$ can be chosen as close to $1$
as needed, up to modifying $\eta$ (and thus $V_{\Gamma_{St}}\subset V_{\Gamma'_{St}}\subset V_\eta$) such that for all $x \in V_\eta$,
$\| \Gamma_{x_d} - {\rm Id} \|_{W^{1,\infty}(\Sigma_{x_d})}$ is as
  close to $0$ as needed.

Let us finally mention the following inequalities, valid for $h\in (0,h_0)$ 
with $h_0=h_{0}(N,\varepsilon)>0$ small enough:
\begin{align}
\label{eq.AgmonWeight1}
  &\varphi_{N}
   \leq  f_{+}+Nh\ln \frac1h
   \leq   \Psi+Nh\ln \frac1h
  \quad \text{ in } V_\eta\\
  \label{eq.AgmonWeight2}
& \varphi_{N}=\psi \leq \varphi \leq  f_{+} \leq   \Psi
  \quad \text{ in } V_\eta \cap \{x' \in \supp\nabla\eta_{1}\}
\end{align}
% Pour le deuxieme inegalite, on commence par montrer que \psi \le
% \varphi en prenant x \in \supp \nabla \chi dans la definition de
% \psi. On en deduit ensuite que \varphi_{N}=\psi 
 and since $  \Psi=f_++f_{-}>f_{+}$ on $\{x_d\in\supp \eta_{d}'\}$, one has 
   \begin{equation}
  \label{eq.AgmonWeight3}
\varphi_{N} \leq  f_{+}+Nh\ln \frac1h \leq   \Psi\quad
  \text{ in } V_\eta \cap \{ x_d\in\supp  \eta_{d}' \}.
\end{equation}

\medskip
\noindent
\underline{Step 1-b:} Proof of~\eqref{eq.estimtoprove}.
\medskip

\noindent
We are now ready 
to prove
 \eqref{eq.estimtoprove}. Controlling the left-hand side of \eqref{eq.first-energy}
using the relations \eqref{eq.Delta-r1}--\eqref{eq.r2'} gives
\begin{align*}
 \left\| (r_{1} +r_{1}')e^{\frac{\varphi_{N}- \Psi}{h}}
\right\|_{ L^{2}(V_\eta)} &\left\|\tilde{w_{h}}\right\|_{ L^{2}(V_\eta)}
+ 
\left\|  (r_{2}+r_{2}')e^{\frac{\varphi_{N}- \Psi}{h}}
\right\|_{ L^{2}(\Gamma_{1})}\left\|\tilde{w}_{h}\right\|_{ L^{2}(\Gamma_{1})} 
\\
&
\geq 
\left\|hd\tilde{w}_{h}\right\|^{2}_{L^{2}(V_\eta)}+\left\|hd^{*}\tilde{w}_{h}\right\|^{2}_{L^{2}(V_\eta)}
-h\int_{\Gamma_{1}}\langle\tilde{w}_{h},\tilde{w}_{h}
\rangle_{ T_{\sigma}^{*}\Omega}\pa_{n} f \, d\sigma
\\
&\quad
+
\langle (|\nabla
f|^{2}-|\nabla\varphi_{N}|^{2}+h\mathcal L_{\nabla f}+
h\mathcal L^*_{\nabla f}) 
\tilde{w}_{h},\tilde{w}_{h}\rangle_{L^2(V_\eta)}\,  ,
\end{align*}
where, since $\varphi_{N}- \Psi\leq Nh\ln\frac1h$ (by
\eqref{eq.AgmonWeight1}) and $r_{i}=\mathcal O(h^\infty)$ for
$i\in\{1,2\}$ (by~\eqref{eq.r1} and~\eqref{eq.r2}),
$$
\left\|r_{1}e^{\frac{\varphi_{N}- \Psi}{h}}\right\|_{ L^2(V_\eta)}
+
\left\|r_{2}e^{\frac{\varphi_{N}- \Psi}{h}}\right\|_{ L^2(\Gamma_{1})}
=\mathcal O(h^\infty),
$$
and, since $\varphi_{N}\leq \Psi$ on $\supp\nabla\eta$ (by
\eqref{eq.AgmonWeight2}--\eqref{eq.AgmonWeight3}) and $\supp
r_i'\subset \supp\nabla\eta$ for $i\in\{1,2\}$ (by~\eqref{eq.r1'} and~\eqref{eq.r2'}),
$$
\left\|r_{1}'e^{\frac{\varphi_{N}- \Psi}{h}}\right\|_{ L^{2}(V_\eta)}
+
\left\|r_{2}'e^{\frac{\varphi_{N}- \Psi}{h}}\right\|_{ L^{2}(\Gamma_{1})} 
=\mathcal O(h^{-N_{0}}).
$$
This leads to the existence of  $C_{1}=C_{1}(N)>0$ such that
for $h$ small enough:
\begin{align*}
C_{1}h^{-N_{0}}\left\|\tilde{w_{h}}\right\|_{ H^1(V_\eta)}
&\geq 
\left\|hd\tilde{w}_{h}\right\|^{2}_{ L^{2}(V_\eta)}+\left\|hd^{*}\tilde{w}_{h}\right\|^{2}_{ L^{2}(V_\eta)}
-h\int_{\Gamma_{1}}\langle\tilde{w}_{h},\tilde{w}_{h}
\rangle_{ T_{\sigma}^{*}\Omega}\pa_{n} f \, d\sigma
\\
&\quad+
\langle (|\nabla
f|^{2}-|\nabla\varphi_{N}|^{2}+h\mathcal L_{\nabla f}+
h\mathcal L^*_{\nabla f}) 
\tilde{w}_{h},\tilde{w}_{h}\rangle_{L^2(V_\eta)}  .
\end{align*}
Since
$\varphi_N\leq \varphi+Nh\ln \frac1h$, $\varphi \leq Ch$ on
$\Omega_{-}$ and $\left\|w_{h}\right\|_{ H^{1}(V_\eta)}=\mathcal O
(h^{\infty})$ (see~\eqref{eq.a-priori-comparison})
\begin{equation}\label{eq.star0}
\left\|\tilde{w}_{h}\right\|_{ L^{2}(\Omega_{-})}\leq
e^C h^{-N} \left\|w_{h}\right\|_{ L^{2}(\Omega_{-})}
\leq C_{2}(C,N).
\end{equation}
Thus, since $\mathcal L_{\nabla f}+
\mathcal L^*_{\nabla f}$ is a $0^{\text{th}}$ order differential operator, we get the existence of  $C_{3}>0$ independent
of $(C,N)$ and of $C_{4}=C_{4}(C,N)$ such that: 
\begin{align*}
C_{4}(h^{-N_{0}}\left\|\tilde{w}_{h}\right\|_{ H^{1}(V_\eta)}+1)
&\geq 
\left\|hd\tilde{w}_{h}\right\|^{2}_{ L^{2}(V_\eta)}+\left\|hd^{*}\tilde{w}_{h}\right\|^{2}_{ L^{2}(V_\eta)}
-h\int_{\Gamma_{1}}\langle\tilde{w}_{h},\tilde{w}_{h}
\rangle_{ T_{\sigma}^{*}\Omega}\pa_{n} f \, d\sigma
\\
&\quad
+ \langle |\nabla f_{-}|^{2}\tilde{w}_{h},\tilde{w}_{h}\rangle_{ L^{2}(V_\eta)}+
\langle (|\nabla
f_{+}|^{2}-|\nabla\varphi_{N}|^{2}-C_{3}h)
\tilde{w}_{h},\tilde{w}_{h}\rangle_{ L^{2}(\Omega_+)}.
\end{align*}
Moreover, by definition of $\varphi_{N}$, a.e. in $\Omega_+$,
$\nabla \varphi_N=\nabla\psi 1_{\{\varphi_{N}=\psi\}} + \nabla
f_{+}(1-\frac{Ch}{f_{+}}) 1_{\{\varphi_{N}<\psi\}}$. Now,
\begin{itemize}
\item On $\{\varphi_{N}=\psi\} $, since by~\eqref{eq.psi>phi} $\{\varphi_{N}=\psi\}$  avoids a neighborhood of
$\{(z,x_d),x_d\in {\rm supp}\, \eta_d\}=\{x\in V_\eta,|\nabla f_{+}(x)|=0\}$ (see~\eqref{eq:nablaf+}), we get
$$
  \left|\nabla
  f_{+}\right|^{2}-\left|\nabla\varphi_{N}\right|^{2}
 \geq  \big(1-M^{2}(1-\varepsilon)^{2}\big)\left|\nabla
  f_{+}\right|^{2} \geq  c_{\varepsilon}>0,
$$
  where~\eqref{Me} have been used;
\item On $\{\varphi_{N}<\psi\}\cap \Omega_+ $, we get like in the proof of 
Proposition~\ref{pr.Agmon} (see~\eqref{eq.diff.f+/Ag} and~\eqref{eq.lowboundnablaf}),
$$
\left|\nabla f_{+}\right|^{2}-\left|\nabla\varphi_{N}\right|^{2}\geq KCh.
$$
\end{itemize}
Choosing $C> \max (1,\frac{C_{3}}{K})$, we 
obtain that for $h$ small enough: 
\begin{align}
\nonumber
C_{4}(h^{-N_{0}}&\left\|\tilde{w}_{h}\right\|_{ H^{1}(V_\eta)}+1)
\geq 
\left\|hd\tilde{w}_{h}\right\|^{2}_{ L^{2}(V_\eta)}+\left\|hd^{*}\tilde{w}_{h}\right\|^{2}_{ L^{2}(V_\eta)}-h\int_{\Gamma_{1}}\langle\tilde{w}_{h},\tilde{w}_{h}
\rangle_{ T_{\sigma}^{*}\Omega}\pa_{n} f\, d\sigma
\\
&+ \langle |\nabla f_{-}|^{2}\tilde{w}_{h},\tilde{w}_{h}\rangle_{ L^{2}(V_\eta)}
+
(KC-C_{3})h\left(\left\|
\tilde{w}_{h}\right\|^2_{{ L^{2}}(V_\eta)}
-\left\|
\tilde{w}_{h}\right\|^2_{{ L^{2}(\Omega_{-})}}
\right). \label{eq.estim-step2}
\end{align}
We can now control from below the r.h.s. of the above estimate
exactly as we did  at the end of the first step of Proposition~\ref{pr.Agmon}:
defining $C_{5}(C):= \frac{KC-C_{3}}{2\left\||\nabla
    f_{-}|^2\right\|_{L^\infty(V_\eta)}}$ (see~\eqref{eq:nablaf-}),
one gets the inequality
\begin{equation}
\label{eq.recall-pr-Agmon}
(KC-C_{3})h\left\|
\tilde{w}_{h}\right\|^2_{ L^{2}(V_\eta)}
+
\langle |\nabla f_{-}|^{2}\tilde{w}_{h},\tilde{w}_{h}\rangle_{ L^{2}(V_\eta)}
\ 
\geq\
(1+2C_{5}h) \langle |\nabla f_{-}|^{2}\tilde{w}_{h},\tilde{w}_{h}\rangle_{ L^{2}(V_\eta)}
\end{equation}
and from Lemma~\ref{le.GreenWeak} applied with 
$u=\tilde{w}_{h}$,
$ f=-\tilde\eta f_{-}$ where $\tilde \eta \in C^{\infty}(\overline \Omega,[0,1])$, $\tilde \eta =1$ on ${\rm supp}\, \eta $, ${\rm supp} \,\tilde\eta \subset ({\rm supp} \,\eta  +B(0,\alpha))\cap \overline \Omega$     for $\alpha>0$ such that $f_-$ is smooth on ${\rm supp}\, \tilde \eta $  and $\frac{h}{1+C_{5}h}$ instead of $h$,
one gets the following lower bound: 
\begin{align}
 -h \int_{\Gamma_{1}}
\langle\tilde{w}_{h},\tilde{w}_{h}
\rangle
\pa_{n}f \, d\sigma &= h \int_{\Gamma_{1}}
\langle\tilde{w}_{h},\tilde{w}_{h}
\rangle
\pa_{n}f_{-} \, d\sigma \nonumber \\
& \geq\  - 
(1+C_{5}h)\left\| |\nabla f_{-}|\tilde{w}_{h}\right\|_{L^2(V_\eta)}^2
\label{eq.lower-int}
\\
& \quad -\frac{h^2}{1+C_{5}h} \left(\left\|d\tilde{w}_{h}\right\|^2_{L^2(V_\eta)}
+\left\|d^{*}\tilde{w}_{h}\right\|_{L^2(V_\eta)}^{2}
\right)
-  C_{6}h \| \tilde{w}_{h}\|_{L^2(V_\eta)}^{2},
\nonumber\end{align}
where $C_{6}$ is some positive constant independent of $C$
(it only depends on $f_-$).
Injecting  the estimates \eqref{eq.recall-pr-Agmon} and \eqref{eq.lower-int}
in~\eqref{eq.estim-step2}
then leads to:
\begin{align*}
C_{4}(h^{-N_{0}}&\left\|\tilde{w}_{h}\right\|_{ H^{1}(V_\eta)}+1)\geq 
  \frac{C_{5}h^3}{1+C_{5}h}\left(
  \left\|d\tilde{w}_{h}\right\|^{2}_{ L^{2}(V_\eta)}
+\left\|d^{*}\tilde{w}_{h}\right\|^{2}_{ L^{2}(V_\eta)}\right)
\\
&\quad +
C_{5}h\left\| |\nabla f_{-}|\tilde{w}_{h}\right\|^2_{ L^2(V_\eta)}
- (KC-C_{3})h
 \left\|\tilde{w}_{h}\right\|^2_{ L^{2}(\Omega_{-})}
 -  C_{6}h \| \tilde{w}_{h}\|^{2}_{ L^2(V_\eta)}.
\end{align*}
Then, since $|\nabla f_{-}|\geq c>0$ on $\overline{V_\eta}$
(see~\eqref{eq:nablaf-}), $\lim_{C\to \infty} C_{5}(C)= +\infty$. Therefore, since $C_{6}$ is independent of $C$,
one can choose $C$ such that $c^2C_{5}-C_{6}>0$, which implies, remembering also $ \left\|\tilde{w}_{h}\right\|_{ L^{2}(\Omega_{-})}
\leq C_{2}(C,N)$ (see~\eqref{eq.star0}),
 the existence of a constant $C_{7}>0$ and a constant
$h_0>0$ such that, for every $h\in (0,h_0)$,
\begin{equation*}
\| \tilde{w}_{h} \|^{2}_{  L^{2}(V_\eta)}
+\|d\tilde{w}_{h}\|^{2}_{  L^{2}(V_\eta)}+\|d^*\tilde{w}_{h}\|^{2}_{  L^{2}(V_\eta)}
 \leq
 \frac{C_{7}}{h^{3}}(h^{-N_{0}}\left\|\tilde{w}_{h}\right\|_{ H^{1}(V_\eta)}+1).
\end{equation*}
According to Gaffney's inequality \eqref{eq.Gaffney}, this finally leads to the existence
of a positive constant
$C_{8}$ such that
$$
\left\|\tilde{w}_{h}\right\|_{ H^{1}(V_\eta)}\ \leq\  C_8
h^{-N_{0}-3}.
$$
Moreover, according to \eqref{eq.psi>phi}, we have $\varphi_{N}=\varphi+ Nh\ln\frac1h$ in $V_\eta\cap\left\{x'\in \overline{\Gamma'_{St}}\right\}$ 
and then $\varphi_{N}-Nh\ln\frac1h-f_{+} \ge  - C_9 h\ln\frac1h$
(with a constant $C_9$ independent of $N$)
in $V_{\Gamma_{St}'}\subset V_\eta\cap\left\{x'\in \overline{\Gamma'_{St}}\right\}$.
Therefore, there exists $N_1$ independent of $N$
such that for $h$ small enough,
$$
\left\|e^{\frac{f_{+}}{h}}w_{h}\right\|_{ H^{1}(V_{\Gamma_{St}'})}
 \leq  C_{N} h^{N-N_{1}},
$$
from which \eqref{eq.estimtoprove} and  \eqref{eq.step1comparison} follow
since $N$ is arbitrary.\\

\medskip
\noindent\underline{Step 2:} Comparison in $V_{\Gamma_{St}}$.
\medskip

\noindent
We work now with the cut-off function $\chi$
defined in~\eqref{eq.cut-off}. Recall that $\eta\chi=\chi$.
Similarly as in the previous step, let us define
the sets 
$$  \Omega_{-}=\left\{x\in V_{\Gamma_{St}'}  \text{ s.t. }  \Psi(x)\leq Ch\right\}
\text{ and }
\Omega_{+}=V_{\Gamma_{St}'}\setminus \Omega_{-},
$$
and the function 
$$\varphi_{N}= \min\left\{\varphi+Nh\ln \frac1h,
   \psi \right\},
$$   
where $\varphi$ and $\psi$
are respectively defined by
\begin{align*}
\varphi=
\left\{
  \begin{aligned}[c]
 \Psi-Ch\ln \frac{ \Psi}{h}\  &\text{if}\  \Psi> Ch\\
 \Psi-Ch\ln C\ &\text{if}\   \Psi \leq Ch,
  \end{aligned}
\right.
\end{align*}
and
$$
\psi(x)=\min\left\{\varphi(y)+(1-\varepsilon) d_{a}\left(x,y\right),\
  y\in  \supp \nabla\chi\right\}.$$
The constant $C>1$ will be chosen at the end of the proof. Following
the proof of~\eqref{eq.psi>phi}, there exists $\varepsilon\in (0,1)$ such that for any $h\in (0,h_0)$
with $h_{0}=h_0(N,\varepsilon)$ small enough,
\begin{equation}
\label{eq.psi>phi2}
\varphi_{N}=\varphi+Nh\ln \frac1h < 
\psi  \mbox{ in }\ V_{\Gamma_{St}}.
\end{equation}
Indeed, using
the fact that $\Psi(x)=d_a(x,z)$ and a triangular inequality,
$$
\forall\left(x,y\right)\in\overline{V_{\Gamma_{St}}}\times\supp\nabla\chi,
\quad
 \Psi(x)< \Psi(y)+ d_{a}\left(x,y\right).
$$
The inequality is strict since if $\Psi(x)= \Psi(y)+ d_{a}\left(x,y\right)$
for some
$\left(x,y\right) \in\overline{V_{\Gamma_{St}}}\times\supp\nabla\chi$,
then $\Phi(x)-\Phi(y)=d_a(x,y)$ and from Corollary~\ref{re.phi_curve}, up to modifying $\eta$ such that $V_{\eta}\subset V_{\alpha}$ (see Corollary~\ref{re.phi_curve} for the definition of $V_{\alpha}$) there exists a generalized integral curve (in the sense of Definition~\ref{generalized_curve}) of 
$\left\{\begin{aligned}
-\nabla\Phi & \text{ on } V_{\alpha}\cap  \Omega\\
-\nabla_{\!T}\Phi &  \text{ on }
\pa\Omega \end{aligned}\right.$ 
joining $x\in\overline{V_{\Gamma_{St}}}$ to  $y\notin\overline{V_{\Gamma_{St}}}$. This
contradicts the fact that $\overline{V_{\Gamma_{St}}}$ is stable
for~\eqref{eq:flow_Phi}. The end of the proof of~\eqref{eq.psi>phi2}
then follows exactly the same lines of the proof of~\eqref{eq.psi>phi}.

Moreover, owing to the properties of $d_{a}$, one has
analogously to \eqref{eq.psi(x')}
the following estimate valid a.e. in $V_{\Gamma_{St}'}$:
\begin{equation}
\label{eq.psi(x)}
\left|\nabla \psi \right|\leq
(1-\varepsilon)\left|\nabla
f \right|= (1-\varepsilon)\left|\nabla
 \Psi \right|.
\end{equation}

Let us finally mention the following inequalities, valid for $h\in (0,h_0)$ 
with $h_0=h_{0}(N,\varepsilon)>0$ small enough:
 \begin{align}
\label{eq.AgmonWeight'1}
  &\varphi_{N} \leq   \Psi+ Nh\ln \frac1h
  \quad \text{ in } V_{\Gamma_{St}'}\\
  \label{eq.AgmonWeight'2}
\text{and } &\varphi_{N}=\psi \leq \varphi \leq   \Psi  \quad \text{ on } \supp\nabla\chi.
\end{align}
% Pour le deuxieme inegalite, on commence par montrer que \psi \le
% \varphi en prenant x \in \supp \nabla \chi dans la definition de
% \psi. On en deduit ensuite que \varphi_{N}=\psi 

We are now in position to prove~\eqref{eq.WKB-comparison}. Let us define
$$\tilde{w_{h}}=
e^{\frac{\varphi_{N}}{h}}w_{h}= e^{\frac{\varphi_{N}}{h}} \chi(u^{(1)}_{h}-c_z(h)u^{(1)}_{z,wkb}).$$ 
Using the
relations \eqref{eq.Delta-r1}--\eqref{eq.r2'}
and the integration by parts formulae
\eqref{eq.intbypartphi} and \eqref{eq.Green}, there exists $C_{1}>0$
(only depending on $f$) such that
\begin{align}
\nonumber
 &\left\| (r_{1}
 +r_{1}')e^{\frac{\varphi_{N}- \Psi}{h}}\right\|_{ L^{2}(V_{\Gamma_{St}'})}\left\|\tilde{w_{h}}\right\|_{ L^{2}(V_{\Gamma_{St}'})}
+ 
\left\|  (r_{2}
 +r_{2}')e^{\frac{\varphi_{N}- \Psi}{h}}
\right\|_{ L^{2}(\Gamma_{1})}\left\|\tilde{w}_{h}\right\|_{ L^{2}(\Gamma_{1})} 
\\
&+  C_{1} h\int_{\Gamma_{1}}\langle\tilde{w}_{h},\tilde{w}_{h}
\rangle_{ T_{\sigma}^{*}\Omega}~d\sigma
\geq 
\left\|hd\tilde{w}_{h}\right\|^{2}_{
  L^{2}(V_{\Gamma_{St}'})}+\left\|hd^{*}\tilde{w}_{h}\right\|^{2}_{
  L^{2}(V_{\Gamma_{St}'})}
\nonumber
\\
\label{eq.big-estim-last-step}
&\qquad\qquad +
\langle (|\nabla
f|^{2}-|\nabla\varphi_{N}|^{2}-C_{1}h) 
\tilde{w}_{h},\tilde{w}_{h}\rangle_{L^2 (\Omega_+)}
-C_{1}h\left\|\tilde{w}_{h}\right\|^2_{ L^{2}(\Omega_{-})},
\end{align}
where we have used the fact that almost everywhere on $\Omega_-$, $|\nabla \varphi_N|$ is
either equal to $|\nabla \varphi|=|\nabla \Psi|=|\nabla f|$ or to
$|\nabla \psi| \le (1-\varepsilon) |\nabla f|$. Moreover, since
$\varphi_{N}- \Psi\leq Nh\ln\frac1h$ (see~\eqref{eq.AgmonWeight'1}), one has from~\eqref{eq.r1} and~\eqref{eq.r2}
$$
\left\|r_{1}e^{\frac{\varphi_{N}- \Psi}{h}}\right\|_{ L^2(V_{\Gamma_{St}'})}
+
\left\|r_{2}e^{\frac{\varphi_{N}- \Psi}{h}}\right\|_{ L^2(\Gamma_{1})}
=\mathcal O(h^\infty)
$$
and, since $\varphi_{N}\leq \Psi$ on $\supp\nabla\chi$ (see~\eqref{eq.AgmonWeight'2}), one gets from~\eqref{eq.r1'} and~\eqref{eq.r2'}
$$
\left\|r_{1}'e^{\frac{\varphi_{N}- \Psi}{h}}\right\|_{ L^{2}(V_{\Gamma_{St}'})}
+
\left\|r_{2}'e^{\frac{\varphi_{N}- \Psi}{h}}\right\|_{ L^{2}(\Gamma_{1})} 
=\mathcal O(h^{-N_{0}}).
$$
Additionally, since $ \Psi=f-f(z)$ on $\Gamma_{1}$ and $\varphi_{N}-
\Psi\leq Nh\ln\frac1h$ (see~\eqref{eq.AgmonWeight'1}), we deduce from the relation
\eqref{eq.step1comparison} obtained in the first step the following estimate:
$$
\left\|\tilde{w}_{h}\right\|_{ L^2(\Gamma_{1})}= 
\left\|e^{\frac{\varphi_{N}}{h}} \chi(u^{(1)}_{h}-c_z(h)u^{(1)}_{z,wkb})\right\|_{ L^2(\Gamma_{St}')}
= \mathcal O(h^\infty) .
$$
Consequently, using in addition the relation 
\begin{equation*}
\left\|\tilde{w}_{h}\right\|_{ L^{2}(\Omega_{-})}\leq
e^C h^{-N} \left\|w_{h}\right\|_{ L^{2}(\Omega_{-})}
\leq C_{2}(C,N),
\end{equation*}
(since $\varphi_N\leq \varphi+Nh \ln \frac1h$, $\varphi \leq Ch$ on
$\Omega_-$ and $\|w_h\|_{L^2(V_{\Gamma_{St}'})}=O(h^\infty)$ from~\eqref{eq.a-priori-comparison}), we deduce from \eqref{eq.big-estim-last-step} the existence of some positive constant $C_{3}=C_{3}(C,C_{1},N)$ such that
\begin{align}
 C_{3}
 \left(h^{-N_{0}}\left\|\tilde{w}_{h}\right\|_{ L^2 (V_{\Gamma_{St}'})}+ 1 \right)
&\geq
\left\|hd\tilde{w}_{h}\right\|^{2}_{ L^{2}(V_{\Gamma_{St}'})}+\left\|hd^{*}\tilde{w}_{h}\right\|^{2}_{ L^{2}(V_{\Gamma_{St}'})}\nonumber\\
&\quad +\langle (|\nabla
f|^{2}-|\nabla\varphi_{N}|^{2}-C_{1}h)\tilde{w}_{h},
\tilde{w}_{h}\rangle_{L^2(\Omega_+)}
. \label{eq.last-lower-bound}
\end{align}
Lastly, one has a.e. in $\Omega_{+}$, $\nabla \varphi_N=\nabla \psi
1_{\{\varphi_N=\psi\}} + \nabla  \Psi(1-\frac{Ch}{ \Psi})
1_{\{\varphi_N<\psi\}}$, and thus
\begin{itemize}
\item on $\{\varphi_N=\psi\}$, from~\eqref{eq.psi(x)},
$$
\left|\nabla f\right|^{2}-\left|\nabla \varphi_{N}\right|^{2}\geq
(2\varepsilon-\varepsilon^{2})\left|\nabla f\right|^{2}\geq
c_{\varepsilon}>0,
$$
\item on $\{\varphi_N<\psi\} \cap \Omega_+$, there exists $C_{4}>0$ independent of 
$C$ such that,
$$
\left|\nabla f\right|^{2}-\left|\nabla\varphi_{N}\right|^{2}
\geq Ch\frac{|\nabla f|^2}{ \Psi}\geq
C_{4}C h.
$$
\end{itemize}
Taking
$C>  \frac{C_{1}}{C_{4}}$ and adding  $(CC_{4}-C_{1})h\left\| \tilde{w}_{h}\right\|^2_{ L^{2}(\Omega_{-})}$
to \eqref{eq.last-lower-bound} then leads to
$$C_5\left( h^{-N_{0}}\left\|\tilde{w}_{h}\right\|_{ L^2
    (V_{\Gamma_{St}'})}+ 1 \right)
 \ge   \left\|hd\tilde{w}_{h}\right\|^{2}_{ L^{2}(V_{\Gamma_{St}'})}
+\left\|hd^{*}\tilde{w}_{h}\right\|^{2}_{ L^{2}(V_{\Gamma_{St}'})}
+ (CC_4-C_{1})h\left\|\tilde{w}_{h}\right\|^2_{ L^{2}(V_{\Gamma_{St}'})},
$$
for a constant $C_5$ depending on $C$ and $N$.
Using Gaffney's inequality \eqref{eq.Gaffney}, we consequently get
the existence of
$C_{6}>0$  such that
$$
\left\|\tilde{w}_{h}\right\|_{ H^{1}(V_{\Gamma_{St}'})}\leq C_6
h^{-N_{0}-3/2}.
$$
Now, since  $\varphi_{N}-Nh\ln\frac1h- \Psi \ge -C_7 h\ln\frac1h$ in $V_{\Gamma_{St}}$
(with a constant $C_7$ independent of~$N$, from the definition of
$\varphi$ and the fact that $\varphi_{N}-Nh\ln\frac1h=\varphi$ in $V_{\Gamma_{St}}$, see~\eqref{eq.psi>phi2}),
we also get the existence of $N_2$ independent of $N$
such that for $h$ small enough,
$$
\left\|e^{\frac{ \Psi}{h}}w_{h}\right\|_{ H^{1}(V_{\Gamma_{St}})}
 \leq  C_{N} h^{N-N_{2}},
$$
for some constant $C_N >0$, which concludes the proof of Proposition~\ref{pr.WKB-comparison}. 
\end{proof}

\subsection{Proof of Theorem~\ref{TBIG0}} \label{sec:goodquasimodes}

%Assumptions \textbf{[H1]}, \textbf{[H2]} and
%\textbf{[H3]} allow us to  construct  the quasi-modes, see Section~\ref{sec:contruct_quasimode}. To conclude the proof of Theorem~\ref{TBIG0}, it only remains to prove that  these quasi-modes satisfy all the estimates stated in Section~\ref{flat} (see Proposition~\ref{ESTIME}).
  
 The aim of this section is to conclude the proof of Theorem~\ref{TBIG0} by checking that the function $\tilde{u}$ and
the family of $1$-forms $ (\tilde \phi_i )_{i=1,\ldots,n}$ introduced in Section~\ref{sec:contruct_quasimode} 
satisfy the estimates appearing in Proposition~\ref{ESTIME} rewritten
in the flat space (see~Section~\ref{flat}).  In all this section, we assume in addition to the  hypotheses \textbf{[H1]}, \textbf{[H2]} and \textbf{[H3]}, that~\eqref{hypo1} and~\eqref{hypo2} hold.

From
Sections~\ref{sec:utilde} and~\ref{construction}, it
only remains to
prove~\eqref{eq:assump_1_phi}, \eqref{eq:assump_2_phi}, \eqref{eq:assump_3_phi}
and~\eqref{eq:assump_4_phi}.
Let us start with a lemma about the normalisation term appearing in~\eqref{eq:phi1}. 
\begin{lemma} \label{N1}  Let us assume that the hypotheses \textbf{[H1]}, \textbf{[H2]} and \textbf{[H3]} hold.
Let us define $\displaystyle \Theta_i:=\sqrt{\int_{\Omega} \left\vert \chi_i(x) u^{(1)}_{h,i}(x) \right\vert^2 dx}$. There exist $c>0$ and $h_0>0$ such that for all $h\in (0,h_0)$,
$$\Theta_i^2=1+O\left(e^{- \frac{c}{h} }  \right).$$\label{page.thetai}
\end{lemma}
\begin{proof}
On the one hand, one has the upper bound
$$\|\chi_i u^{(1)}_{h,i}\|_{L^2(\Omega)}=\|\chi_i
u^{(1)}_{h,i}\|_{L^2(\dot \Omega_i)} \le \|
u^{(1)}_{h,i}\|_{L^2(\dot \Omega_i)} =1.$$
% $$
% \int_{\Omega} \left\vert \chi_i(x) u^{(1)}_{h,i}(x) \right\vert^2 dx =\int_{{\dot \Omega_i}} \left\vert \chi_i(x) u^{(1)}_{h,i}(x) \right\vert^2 dx\leq \int_{{\dot \Omega_i}} \left\vert  u^{(1)}_{h,i}(x) \right\vert^2 dx= 1.
% $$
On the other hand, the triangular inequality yields the lower bound
\begin{align*}
\|\chi_i
u^{(1)}_{h,i}\|_{L^2(\dot \Omega_i)} &\ge
\|u^{(1)}_{h,i}\|_{L^2(\dot \Omega_i)} - \|(1-\chi_i)
u^{(1)}_{h,i}\|_{L^2(\dot \Omega_i)} = 1- \|(1-\chi_i)
u^{(1)}_{h,i}\|_{L^2(\dot \Omega_i)} .
\end{align*}
% \\
% &= 1-\sqrt{\int_{{\dot \Omega_i}} \left\vert\left(1-\chi_i(x)\right) u^{(1)}_{h,i}(x) e^{\pm \frac{1}{h} d_a(x,z_i)} \right\vert^2  dx}.
% \end{align*}
% \begin{align*}
% \sqrt{\int_{\Omega}\left\vert \chi_i(x) u^{(1)}_{h,i}(x) \right\vert^2 dx}&=\sqrt{\int_{{\dot \Omega_i}}\left\vert \chi_i(x) u^{(1)}_{h,i}(x) \right\vert^2 dx}\\
% &=\sqrt{\int_{{\dot \Omega_i}} \left\vert\left[ 1-\left(1-\chi_i(x)\right)\right] u^{(1)}_{h,i}(x) \right\vert^2 dx}\\
% &\geq 1-\sqrt{\int_{{\dot \Omega_i}} \left\vert\left[1-\chi_i(x)\right] u^{(1)}_{h,i}(x) \right\vert^2 dx}\\
% &\geq 1-\sqrt{\int_{{\dot \Omega_i}} \left\vert\left[1-\chi_i(x)\right] u^{(1)}_{h,i}(x) e^{\pm \frac{1}{h} d_a(x,z_i)} \right\vert^2  dx}.
% \end{align*}
Thanks to Proposition~\ref{pr.Agmon}, there exist $N\in \mathbb N$ and $c>0$ independent of $h$ such that
\begin{align*}
\|(1-\chi_i)
u^{(1)}_{h,i}\|_{L^2(\dot \Omega_i)}^2&=
\int_{{\dot \Omega_i}} \left\vert\left(1-\chi_i(x)\right)
  u^{(1)}_{h,i}(x) e^{ \frac{1}{h} d_a(x,z_i)} e^{ -\frac{1}{h} d_a(x,z_i)} \right\vert^2  dx \\
& \le \int_{{\dot \Omega_i} \setminus \mathcal V_i} \left\vert
  u^{(1)}_{h,i}(x) e^{ \frac{1}{h} d_a(x,z_i)} e^{ -\frac{1}{h} d_a(x,z_i)} \right\vert^2  dx \\
&\leq C\,h^{-N} e^{-\inf_{{\dot \Omega_i} \setminus \mathcal V_i}
  \frac{2}{h} d_a(.,z_i)} \leq C\,  e^{- \frac{c}{h} },
\end{align*}
where, we recall ${\mathcal V_i}=\{x \in \Omega, \, \chi_i=1\}$. This concludes the proof of Lemma~\ref{N1}.\end{proof}
\noindent
We are now in position to check the estimates stated in
Section~\ref{flat}.

\medskip
\noindent \underline{{Step 1.}} Study of the term $\left\|   \left ( \  1-\pi_{[0,h^{\frac32} )} \left( \Delta^{D,(1)}_{f,h}(\Omega) \right) \right)\tilde \phi_i\right\|_{H^1(\Omega)}$.
\medskip

\noindent
We recall that from~\eqref{eq.Gaffney}, $\tilde \phi_i$ belongs to $\Lambda^1H^1_T(\Omega)$ and then we get from   Lemma~\ref{quadra} that there exist $c>0$ and $h_0>0$ such that for all $h\in (0,h_0)$,
\begin{align*}
\left\|    \left(1-\pi_{[0,h^{\frac32} )} \left( \Delta^{D,(1)}_{f,h}(\Omega)\right)\right) \tilde \phi_i\right\|_{L^2(\Omega)}^2  &\leq c h^{-3/2} \left( \left\|  d_{f,h} \tilde \phi_i\right\|_{L^2(\Omega)}^2+\left  \| d_{f,h}^* \tilde \phi_i \right\|_{L^2(\Omega)}^2 \right)\\
&= c h^{-3/2}  \left( \left\|  d_{f,h} \tilde \phi_i\right\|_{L^2 ( {\dot \Omega_i} )}^2+ \left \| d_{f,h}^* \tilde \phi_i\right\|_{L^2 ( {\dot \Omega_i} )}^2\right).
\end{align*}
 Moreover, from Proposition~\ref{pr.DeltaTN} (items $(ii)$ et $(iii)$)
\begin{equation}\label{eq:dfhphitilde}
 d_{f,h} \tilde \phi_i =  \Theta_i^{-1} \left( \chi_i \, d_{f,h}
   u^{(1)}_{h,i} + h  \, d\chi_i \wedge  u^{(1)}_{h,i} \right)  =
 \Theta_i^{-1}   h  \,d\chi_i \wedge  u^{(1)}_{h,i},
\end{equation}
and
$$
 d_{f,h}^* \tilde \phi_i =  \Theta_i^{-1} \left( \chi_i\,d_{f,h}^* u^{(1)}_{h,i} - h\, u^{(1)}_{h,i} \cdot \nabla \chi_i\right)=-\Theta_i^{-1}  h\, u^{(1)}_{h,i} \cdot \nabla \chi_i.
$$
As a consequence, using Lemma~\ref{N1} and Proposition~\ref{pr.Agmon}, one gets for some  $N\in
 \mathbb N$ and for some $c>0$ which may change from one occurrence to another, 
 \begin{align}
 \left \|    \left(1-\pi_{[0,h^{\frac32} )} \left( \Delta^{D,(1)}_{f,h}(\Omega)\right)\right) \tilde \phi_i\right\|_{L^2(\Omega)}^2  &\leq c h^{-3/2}  \left( \left\| h  d\chi_i \wedge  u^{(1)}_{h,i}\right\|_{L^2 ( {\dot \Omega_i} )}^2+  \left\| h u^{(1)}_{h,i} \cdot \nabla \chi_i\right\|_{L^2 ( {\dot \Omega_i} )}^2\right) \nonumber\\
   &\leq c\, h^{1/2} \int_{{\rm supp} \nabla \chi_i} \left\vert
    u_{h,i}^{(1)}(x)e^{ \frac{1}{h} d_a(x,z_i)} e^{ -\frac{1}{h} d_a(x,z_i)}\right\vert ^2 dx  \nonumber\\
 &\leq c\,  h^{1/2-N} e^{  -  \inf_{\rm{supp} \nabla \chi_i}
   \frac2h d_a(\cdot ,z_i)  }.\label{eq:majo_inf_supp}
 \end{align}
The function $\chi_i$ can be chosen such that the set $\big\{x \in {\dot \Omega_i}, \, \nabla \chi_i(x)\neq
  0\big\}$ is as  close as needed to $\Gamma_{2,i}$ and
to $\Gamma_0$ (see Figure~\ref{fig:support}). % Additionally, the set $\Gamma_{,i2}$ can be as close as one wants to $B_{z_i}^c$ (which is also possible because $\Gamma_{1,i}$ can be taken as large as one wants to $B_{z_i}$, see Proposition~\ref{stronglystableexis}).
Therefore by continuity of the Agmon distance, % % by growing the
                                % support of $\chi_i$ as close as one
                                % wants to $B_{z_i}^c$ and to $x_0$,
using~\eqref{eq:gamma2_close_to_Bzc}--\eqref{eq:gamma0_close_to_x0},                                for
                                any~$\delta>0$, one can
                                choose $\chi_i$ satisfying the three
                                conditions stated in  Definition
                                \ref{tildephii} and such that 
  \begin{align} \label{eq:da1}
   \inf_{z\in{\rm supp} \nabla \chi_i} d_a(z,z_i) &\geq \min \left(  d_a(x_0,z_i),  \inf_{z\in B_{z_i}^c} d_a(z,z_i) \right)-\delta.
 \end{align}
From \eqref{hypo1}, there exists $r>0$ such that
 $$\inf_{z\in B_{z_i}^c} d_a(z,z_i) \geq \max\left[f(z_n)-f(z_i),f(z_i)-f(z_1)\right] +r.$$
In addition, using~\eqref{hypo2}, there exists $r'>0$ such that
\begin{align*}
d_a(z_i,x_0)&\geq f(z_i)-f(x_0)\geq f(z_1)-f(x_0)\ge f(z_n) - f(z_1) + r'\\
&\geq\max[f(z_n)-f(z_i),f(z_i)-f(z_1)]+r'.
\end{align*}
Therefore, choosing $\chi_i$ such that $\delta < \min(r,r')$,  there exists $\ve'>0$ such that
  \begin{equation} \label{eq:daa1}
   \inf_{z\in{\rm supp} \nabla \chi_i} d_a(z,z_i) \geq\max\left[f(z_n)-f(z_i),f(z_i)-f(z_1)\right] +\ve'.
 \end{equation}
Using the estimate~\eqref{eq:daa1} in~\eqref{eq:majo_inf_supp}, there exist $\ve_1>0$, $c>0$, $N\in \mathbb N$ and $h_0>0$, such that for every $ h\in (0,h_0)$
\begin{align} 
\nonumber
\left\|    \left(1-\pi_{[0,h^{\frac32} )} \left(
      \Delta^{D,(1)}_{f,h}(\Omega)\right)\right) \tilde
  \phi_i\right\|_{L^2(\Omega)}^2 &\leq ch^{-N} e^{- \frac2 h  \left(\max[f(z_n)-f(z_i),f(z_i)-f(z_1)]+ \ve'\right)} \\ \label{dodo}
&\leq e^{ -\frac2h \left(\max[f(z_n)-f(z_i),f(z_i)-f(z_1)]+ \ve_1\right)}.
\end{align}
This last inequality leads to the desired estimate in the $L^2(\Omega)$-norm. In order to get the same upper bound in the $H^1(\Omega)$-norm, notice now that one has
   \begin{align*}
& \left(1-\pi_{[0,h^{\frac32} )} \left( \Delta^{D,(2)}_{f,h}(\Omega) \right)\right) d_{f,h} \tilde \phi_i=d_{f,h}  \left(1-\pi_{[0,h^{\frac32} )} \left( \Delta^{(D,1)}_{f,h}(\Omega)\right)\right) \tilde \phi_i\\
&= h  d \left(1-\pi_{[0,h^{\frac32} )} \left( \Delta^{D,(1)}_{f,h}(\Omega)\right)\right) \tilde \phi_i + df\wedge \left(1-\pi_{[0,h^{\frac32} )} \left( \Delta^{D,(1)}_{f,h}(\Omega)\right)\right) \tilde \phi_i.
 \end{align*}
Therefore it holds
   \begin{align*}
 h d \left(1-\pi_{[0,h^{\frac32} )} \left(
     \Delta^{D,(1)}_{f,h}(\Omega)\right)\right) \tilde
 \phi_i&=\left(1-\pi_{[0,h^{\frac32} )} \left(
     \Delta^{D,(2)}_{f,h}(\Omega) \right)\right) d_{f,h}\, \tilde
 \phi_i\\
&\quad - df\wedge \left(1-\pi_{[0,h^{\frac32} )} \left( \Delta^{D,(1)}_{f,h}(\Omega)\right)\right) \tilde \phi_i.
\end{align*}
Let us introduce $K_i:=\max[f(z_n)-f(z_i),f(z_i)-f(z_1)]$. From \eqref{dodo}, there exist $C>0$ and $h_0>0$, such that for all $ h\in (0,h_0)$
$$\left\Vert df\wedge \left(1-\pi_{[0,h^{\frac32} )} \left(
      \Delta^{D,(1)}_{f,h}(\Omega)\right)\right) \tilde \phi_i
\right\Vert_{L^2(\Omega)}^2\leq C e^{- \frac 2 h (K_i+ \ve_1)}.$$ 
Moreover, using~\eqref{eq:dfhphitilde} and~\eqref{eq:majo_inf_supp} there exist $\ve > 0$, $C>0$ and $h_0>0$, such that for all $ h\in (0,h_0)$,
\begin{align*}
\left \Vert  \left(1-\pi_{[0,h^{\frac32} )} ( \Delta^{D,(2)}_{f,h}(\Omega)
  ) \right) d_{f,h}  \tilde \phi_i \right\Vert_{L^2(\Omega)}^2 &\leq
\left\Vert d_{f,h} \tilde \phi_i \right\Vert_{L^2(\Omega)}^2 = \Theta_i^{-2} \left\| h  d\chi_i \wedge  u^{(1)}_{h,i}\right\|_{L^2 ( {\dot \Omega_i} )}^2\\
&\leq C e^{- \frac 2 h (K_i+ \ve)} .
\end{align*}
Thus one gets: there exist $\ve > 0$, $C>0$ and $h_0>0$, such that for all $ h\in (0,h_0)$,
$$h^2 \left\Vert d \left(1-\pi_{[0,h^{\frac32} )} \left( \Delta^{D,(1)}_{f,h}(\Omega)\right)\right) \tilde \phi_i\right\Vert_{L^2(\Omega)}^2\leq C e^{- \frac 2 h (K_i+ \ve)} .$$
Similarly, there exist $\ve>0$, $C>0$ and $h_0>0$, such that for all $ h\in (0,h_0)$
$$h^2 \left\Vert d^* \left(1-\pi_{[0,h^{\frac32} )} \left( \Delta^{D,(1)}_{f,h}(\Omega)\right)\right) \tilde \phi_i\right\Vert_{L^2(\Omega)}^2\leq C e^{- \frac 2 h (K_i+ \ve)}  .$$
As a consequence, using~\eqref{eq.Gaffney}, there exist $\ve>0$, $C>0$ and $h_0>0$, such that for all $ h\in (0,h_0)$
\begin{align*}
 \left\Vert \left(1-\pi_{[0,h^{\frac32} )} \left( \Delta^{D,(1)}_{f,h}(\Omega)\right)\right)\tilde \phi_i\right\Vert_{H^1(\Omega)}^2
 &\leq C e^{- \frac 2 h (K_i+ \ve)}.
 \end{align*}
This concludes the proof of~\eqref{eq:assump_1_phi}.

\medskip
\noindent
\underline{{Step 2}}. Study of the terms $\displaystyle \int_{\Omega} \tilde \phi_i
\cdot \tilde \phi_j $ for $(i,j)\in
\left\{1,\ldots,n\right\}^2$.
\medskip

\noindent
Let $(i,j)\in \left\{1,\ldots,n\right\}^2$ be such that $i<j$. One then has $f(z_i)\leq f(z_j)$.  From
Proposition~\ref{PoPo}, it holds 
$d_a(z_i,z_j)>f(z_j)-f(z_i)$.
Now, according to Proposition~\ref{pr.Agmon} and Lemma~\ref{N1} and to the triangular inequality for $d_a$,  there exist $\ve>0$, $N\in \mathbb N$ and $h_0>0$ such that for all $ h\in (0,h_0)$,
\begin{align}
\left\vert \int_{\Omega} \tilde \phi_i(x) \cdot \tilde \phi_j(x) \, dx
\right\vert & \le e^{ -\frac{d_a(z_j,z_i)}{h}} \int_{{\rm
    supp}\chi_i\cap {\rm supp}\chi_j} e^{ \frac{d_a(x,z_i)}{h}}
\vert\tilde \phi_i(x)\vert e^{ \frac{d_a(x,z_j)}{h}} \vert\tilde \phi_j(x) \vert \, dx \nonumber \\
&\leq \Theta_i^{-1} \Theta_j^{-1} \left\Vert
  e^{\frac{d_a(.,z_i)}{h} } \chi_i u^{(1)}_{h,i}
  \right\Vert_{L^2(\dot \Omega_i)}  \left\Vert   e^{\frac{d_a(.,z_j)}{h}} \chi_j u^{(1)}_{h,j}
  \right\Vert_{L^2(\dot \Omega_j)} \!\!\! e^{ -\frac{d_a(z_i,z_j)}{h} } \label{eq:majo_phii_phij} \\
&\leq C  h^{-N} e^{ - \frac1h (f(z_j)-f(z_i)+\ve) }. \nonumber
 \end{align}
This concludes the proof of~\eqref{eq:assump_2_phi}.

\medskip
\noindent 
\underline{{Step 3.}} Study of the terms $\displaystyle \int_{\Sigma_i} \tilde
\phi_j \cdot n  \ e^{-\frac fh}$ for $(i,j)\in
\left\{1,\ldots,n\right\}^2$.
\medskip

\noindent
By construction, for all $(i,j)\in \{1,\ldots,n\}^2$ such that $i\neq j$, one has 
 $$\int_{\Sigma_i}   \tilde \phi_j \cdot n  \   e^{- \frac{1}{h}f}=0.$$
Now, let us compute the term $\int_{\Sigma_i}   \tilde \phi_i \cdot n
\   e^{- \frac{1}{h}f} $. Let $u^{(1)}_{z_i,wkb}$ be the WKB expansion
defined by \eqref{wkbi}. Following the beginning of
Section~\ref{sec:WKB_first_estim}, let us consider
\begin{enumerate}
\item a neighborhood $V_{\Gamma_{St,i}}$ of $\Sigma_i$ in $\overline
  \Omega$, which is stable under the dynamics~\eqref{eq:flow_Phi} and
  such that, for some $\ve > 0$, $V_{\Gamma_{St,i}} + B(0,\ve) \subset  V_{\Gamma_{1,i}} \cap (\Gamma_{1,i} \cup \dot
  \Omega_i)$
\item and a cut-off function $\chi_{wkb,i}\in
  C^{\infty}_c(\dot \Omega_i \cup \Gamma_{1,i})$ with $\chi_{wkb,i} \equiv 1$ on a
  neighborhood of $\overline{V_{\Gamma_{St,i}}}$ and such that $\supp
  \chi_{wkb,i} \subset V_{\Gamma_{1,i}} \cap (\dot \Omega_i \cup \Gamma_{1,i})$.
\end{enumerate}
Using Proposition~\ref{apriori}, there exists $c_{z_i}(h)\in
\mathbb R_+^*$ such that
$$
\left\| \chi_{wkb,i} \left(u^{(1)}_{h,i}-c_{z_i}(h)u^{(1)}_{z_i,wkb}\right)\right\|_{ H^{1}(\dot\Omega_i)}\!\!\!=O (h^{\infty}).
$$
Let us now introduce
\begin{equation}\label{page.tildephiwkb}
\tilde \phi_{z_i,wkb}:=c_{z_i}(h) \chi_{wkb,i} \ u_{z_i,wkb}^{(1)}
\end{equation} 
so that
$$\int_{\Sigma_i}   \tilde \phi_i \cdot n  \   e^{- \frac{1}{h}f}  = \int_{\Sigma_i}    \tilde \phi_{z_i,wkb} \cdot n  \ e^{- \frac{1}{h}f}  +\int_{\Sigma_i}  \left  (\tilde \phi_i-\tilde \phi_{z_i,wkb}\right) \cdot n \,  e^{- \frac{1}{h}f} . $$
Let us first deal with the term $\int_{\Sigma_i}    \tilde
\phi_{z_i,wkb} \cdot n  \,  e^{- \frac{1}{h}f} $. Using Laplace's method (the computation is
similar to~\eqref{eq:wkb_bord}), one gets
when $h\to 0$ (since $\Phi=f$ and $\partial_n\Phi=-\partial_nf$ on~$\partial \Omega$, see~\eqref{eikonalequationboundary} and~\eqref{wkbi})
\begin{align*}
\int_{\Sigma_i}  \chi_{wkb,i} \  u_{z_i,wkb}^{(1)} \cdot n \,  e^{-
  \frac{1}{h}f}&=\int_{\Sigma_i}  e^{-\frac{\Phi-f(z_i)}{h}}
\, a_0\, \partial_n(f-\Phi)   \, e^{- \frac{1}{h}f} \  (1+O(h))\\
&=2e^{\frac{f(z_i)}{h}} \int_{\Sigma_i}  e^{-\frac{2f}{h}}  \, a_0\, \partial_nf \  (1+O(h))\\
&= \frac{2\,\partial_nf(z_i)\, \pi ^{\frac{d-1}{2} }} {  \sqrt{ {\rm  det \ Hess } f|_{\partial \Omega }   (z_i)  }}  h ^{\frac{d-1}{2} }\ e^{- \frac{1}{h}f(z_i)}  (1+O(h)). 
\end{align*} 
Then one obtains when $h\to 0$
$$\int_{\Sigma_i}    \tilde \phi_{z_i,wkb} \cdot n  \, e^{-
  \frac{1}{h}f}  =  c_{z_i}(h) \frac{2\,\partial_nf(z_i)\, \pi ^{\frac{d-1}{2} }} {  \sqrt{ {\rm  det \ Hess } f|_{\partial \Omega }   (z_i)  }}  h ^{\frac{d-1}{2} }\ e^{- \frac{1}{h}f(z_i)}
(1+O(h)).$$
We recall from Proposition~\ref{apriori} that in the limit $h\to 0$:
$$c_{z_i}(h)= C_{z_i,wkb}^{-1}h^{-\frac{d+1}{4}}(1+O(h^{\infty})),$$
where the constant $C_{z_i,wkb}$ is defined by
\eqref{cwkbi}. Therefore, in the limit $h\to 0$
$$\int_{\Sigma_i}    \tilde \phi_{z_i,wkb} \cdot n  \, e^{-
  \frac{1}{h}f}  =  \frac{ \pi ^{\frac{d-1}{4}} \sqrt{2 \partial_n
    f(z_i) }   }{ ( {\rm det \ Hess } f|_{\partial \Omega }   (z_i)
  )^{1/4}  } \  h^{\frac{d-3}{4}} \ e^{- \frac{1}{h}f(z_i)}
(1+O(h)).$$

Let us now deal with the term $ \int_{\Sigma_i} \left(\tilde
  \phi_i-\tilde \phi_{z_i,wkb} \right) \cdot n \ e^{- \frac{1}{h}f}
$. One obtains using  Lemmata~\ref{wkb1}~and~\ref{N1}, that there exist $C>0$, $h_0>0$ and $\eta>0$ such that for all $ h\in (0,h_0)$
 \begin{align*}
\left\vert \int_{\Sigma_i} \left(\tilde \phi_i-\tilde \phi_{z_i,wkb} \right) \cdot n \ e^{- \frac{1}{h}f} \right\vert&=\left \vert \int_{\Sigma_i} \left(\frac{u^{(1)}_{h,i}}{\Theta_i}-c_{z_i}(h)u_{z_i,wkb}^{(1)}\right) \cdot n \ e^{- \frac{1}{h}f}  \right\vert\\
&= \frac{1}{\Theta_i}\left\vert \int_{\Sigma_i} \left(u^{(1)}_{h,i}-\Theta_i \ c_{z_i}(h) u_{z_i,wkb}^{(1)}\right) \cdot n \ e^{- \frac{1}{h}f} \right\vert\\
&\leq \frac{e^{- \frac{1}{h}f(z_i)}}{\Theta_i}   \int_{\Sigma_i} \left\vert \left (u^{(1)}_{h,i}- c_{z_i}(h) u_{z_i,wkb}^{(1)}\right) \cdot n \right\vert   \\
&\quad +e^{- \frac{1}{h}f(z_i)}\,\frac{\vert \Theta_i-1\vert
}{\Theta_i} |c_{z_i}(h)| \int_{\Sigma_i}  \left\vert u_{z_i,wkb}^{(1)} \cdot n  \right\vert  \\
 &\leq C e^{- \frac{1}{h}f(z_i)}  \left \Vert \chi_{wkb,i}
   \left(u^{(1)}_{h,i}- c_{z_i}(h)
     u_{z_i,wkb}^{(1)}\right)\right\Vert_{H^1(\dot \Omega_i)}   \\
&\quad +C e^{- \frac{1}{h}f(z_i)}\,e^{-\frac{\eta}{h}}
h^{-\frac{d+1}{4}} \left\Vert \chi_{wkb,i} u_{z_i,wkb}^{(1)}
    \right\Vert_{H^1(\dot \Omega_i)}.
\end{align*}
 % \\
 % &\leq C e^{- \frac{1}{h}f(z_i)}\left(\left \Vert \chi_{wkb,i}
 %   \left(u^{(1)}_{h,i}- c_{z_i}(h)
 %     u_{z_i,wkb}^{(1)}\right)\right\Vert_{H^1(\dot \Omega_i)}+e^{-\frac{\eta}{h}} h^{-\frac{d+5}{4}}\right).
 % \end{align*}
Therefore, one obtains using Proposition~\ref{apriori} and~\eqref{eq:estim_wkb_H1}
$$ \ e^{\frac{1}{h}f(z_i)} \left\vert \int_{\Sigma_i} \left(\tilde
    \phi_i-\tilde \phi_{z_i,wkb} \right) \cdot n \ e^{- \frac{1}{h}f}
\right\vert  = O(h^{\infty}) + C e^{-\frac{\eta}{h}} h^{-\frac{d+5}{4}} = O(h^\infty).$$
In conclusion, we have when $h\to 0$
$$\int_{\Sigma_i}     \tilde \phi_i \cdot n  \   e^{- \frac{1}{h}f}      = \int_{\Sigma_i}   \tilde \phi_{z_i,wkb}\cdot n \,  e^{- \frac{1}{h}f}   \  (1+O(h^{\infty})),$$
which gives the expected estimate 
\begin{equation}\label{eq:result_step3} 
  \int_{\Sigma_i}   \  \tilde \phi_j \cdot n  \,  e^{- \frac{1}{h}f} =\begin{cases}   B_i h^m   \     e^{- \frac{1}{h}f(z_i)}  \    ( 1  +     O(h   )    )   &  \text{ if } i=j,  \\
 0   &   \text{ if } i\neq j,
  \end{cases} 
  \end{equation}
where 
\begin{equation} \label{mBi}
m=\frac{d-3}{4} \quad {\rm and} \quad B_i= \frac{ \pi ^{\frac{d-1}{4}} \sqrt{2\partial_n f(z_i) }   }{ ( {\rm det \ Hess } f|_{\partial \Omega }   (z_i)   )^{1/4}  }
.
\end{equation}\label{page.bim}
This concludes the proof of~\eqref{eq:assump_4_phi}.

\medskip
\noindent 
\underline{{Step 4.}} Study of the term  $\displaystyle{\int_{\Omega} \nabla  \tilde
u  \cdot \tilde \phi_i \,  e^{- \frac{1}{h}f}}$.
\medskip

\noindent
First one has the equality by Definition~\ref{tildeu}, 
$$\int_{\Omega} \nabla  \tilde u  \cdot \tilde \phi_i \,  e^{-
  \frac{1}{h}f} = \frac{ \displaystyle \int_{\Omega} \nabla  \chi \cdot \tilde \phi_i \,  e^{- \frac{1}{h}f}  } { \sqrt{ \displaystyle\int_{\Omega} \chi^2 e^{-\frac{2}{h}  f} }     },$$
 where $\nabla  \chi \cdot \tilde \phi_i=\mbf i_{\nabla  \chi} \tilde \phi_i=  \tilde \phi_i(\nabla  \chi)$. 
The denominator of the right-hand side is  easily computed thanks to Laplace's method:
$$\sqrt{\int_{\Omega} \chi^2 e^{- \frac{2}{h}f}}= \frac{  (\pi  h)^{\frac{d}{4} } }{ ( {\rm det \ Hess } f   (x_0)   )^{1/4}   }e^{- \frac{f(x_0)}{h}} (1+O(h) ).$$
Using an integration by parts and the fact that $d^*(u^{(1)}_{h,i} \
e^{-  f/h} )=0$ in $L^2({\dot \Omega_i})$ (see
Proposition~\ref{pr.DeltaTN} items $(ii)$ and $(iii)$) which is valid for all $h$ small enough, one obtains 
\begin{align}
\nonumber
\int_{\Omega} \nabla  \chi \cdot \tilde \phi_i \,  e^{-  \frac{f}{h}} 
&= - \int_{\Omega} \nabla(1- \chi) \cdot \chi_i \frac{ u^{(1)}_{h,i}}{\Theta_i} \,  e^{-  \frac{f}{h}} \\
\nonumber
&=  \int_{\Omega} \left(1-\chi\right) \ \nabla \chi_i \cdot
\frac{ u^{(1)}_{h,i}}{\Theta_i} \,  e^{-  \frac{f}{h}}  -
\int_{\partial \Omega} \left(1-\chi\right)\
\tilde \phi_i \cdot n \,  e^{-  \frac{f}{h}}  .
\end{align}
Using the fact that $\chi = 0$ on $\partial \Omega$, one then obtains:
\begin{align}
\nonumber
\int_{\partial \Omega} \left(1-\chi\right)\
\tilde \phi_i \cdot n \,  e^{-  \frac{f}{h}}  
&=   \int_{\partial
  \Omega \cap {\rm supp} \chi_i} \tilde \phi_i \cdot n \,  e^{-
  \frac{f}{h}}  \\\nonumber
&= \int_{\Sigma_i}  \ \tilde \phi_i \cdot n  \,  e^{-  \frac{f}{h}} \
 + \int_{(\partial \Omega \cap {\rm supp} \chi_i) \setminus
\Sigma_i} \hspace{-1.6cm} \tilde \phi_i \cdot n \,  e^{-  \frac{f}{h}}   .
\nonumber
\end{align}
 Using~\eqref{eq:result_step3}, in the limit $h\to 0$:
\begin{align}
\nonumber
\int_{\Omega} \nabla  \chi \cdot \tilde \phi_i \,  e^{-  \frac{f}{h}}
&=  - B_i h^m   \     e^{-\frac 1h f(z_i)}  \    (  1  +     O(h  ))\\
&\quad - \int_{(\partial \Omega \cap {\rm supp} \chi_i) \setminus
\Sigma_i} \hspace{-1.6cm}  \tilde \phi_i \cdot n \,  e^{-  \frac{f}{h}}  +\int_{\Omega} \left(1-\chi\right) \ \nabla \chi_i \cdot
\frac{ u^{(1)}_{h,i}}{\Theta_i} \,  e^{-  \frac{f}{h}}.\label{dodoo}
\end{align}
Let us now prove that the two last terms in~\eqref{dodoo} are
negligible compared to the first one.
\medskip

\noindent
  On the compact set $(\partial \Omega \cap {\rm supp}
  \chi_i)\setminus \Sigma_i$ one  has $f(z)>f(z_i)$ since $z_i\in
  \Sigma_i$ is the only global minimum of $f$ on $B_{z_i}$
  and $\supp \chi_i \cap \partial \Omega \subset \Gamma_{1,i} \subset
  B_{z_i} $. Then, using Proposition~\ref{pr.Agmon} and~\eqref{eq.Gaffney}, there exist $\ve>0$, $h_0>0$, $C>0$ and $N\in \mathbb N$ such that for all $ h\in (0,h_0)$
 \begin{align}
\left\vert \, \int_{(\partial \Omega \cap {\rm supp} \chi_i) \setminus
    \Sigma_i} \hspace{-1.6cm} \tilde \phi_i \cdot n \,  e^{-  \frac{f}{h}}  \right \vert  &\leq 
 \,  e^{-  \frac{f(z_i)+\ve}{h}}\left\vert  \, \int_{(\partial \Omega \cap {\rm supp} \chi_i) \setminus \Sigma_i} \hspace{-1.4cm}  \tilde \phi_i  \cdot n   \ \ \ \ \right \vert \nonumber \\
 &\leq C e^{-  \frac{f(z_i)+\ve}{h}}    \left  \vert \,     \int_{(\partial \Omega \cap {\rm supp} \chi_i) \setminus \Sigma_i} \hspace{-1.6cm}  \tilde \phi_i  \cdot n\ e^{ \frac{d_a(.,z_i)}{h}} \right  \vert    \nonumber    \\
  &\leq C \frac{e^{-  \frac{f(z_i)+\ve}{h}}}{\Theta_i} \    \left\Vert
          \chi_i  u^{(1)}_{h,i}\  e^{
            \frac{d_a(.,z_i)}{h}}\right\Vert_{H^1(\dot \Omega_i)}  \nonumber   \\
 &\leq C e^{ -  \frac{f(z_i)+\ve}{h} } h^{-N}\leq C e^{- \frac{f(z_i)+\ve/2}{h}}.\label{dodoo1}
   \end{align} 
Let us now deal with the last term of \eqref{dodoo}. The support of $
\left(1-\chi\right)  \nabla \chi_i$ is included in the support of
$\nabla \chi_i$ and  does not contain $x_0$ since $\chi \equiv 1$
around $x_0$. The function $\chi_i$ can be chosen such that the set $\{x \in \dot\Omega_i, \, |\nabla \chi_i| \neq 0
\text{ and } \chi \neq 1\}$ is as close as one wants from~$\Gamma_{2,i}$ (see Figure~\ref{fig:support}). Therefore, by continuity of the Agmon distance,
using~\eqref{eq:gamma2_close_to_Bzc}, for any $\delta > 0$, one can
choose $\chi_i$ satisfying the three conditions stated in
Definition~\ref{tildephii} and such that  
\begin{align*}
\inf_{ {\rm supp  }\left(1-\chi\right)\nabla \chi_i }\left(
  d_a( \cdot ,z_i)+f\right)
&\ge \inf_{ B_{z_i}^c}\left( d_a(\cdot ,z_i)+f\right)-\delta.
\end{align*}
Thus, using the Cauchy-Schwarz inequality and Proposition~\ref{pr.Agmon}, there exists $N\in \mathbb N$ such that
  \begin{align}
 \int_{{\rm supp    } \left(1-\chi\right)  \nabla \chi_i}
 \left\vert\left(1-\chi\right)   \nabla \chi_i \cdot  u^{(1)}_{h,i}\
   e^{-  \frac{f}{h}} \right\vert &\leq  C\left\Vert u^{(1)}_{h,i} e^{
     \frac{d_a(\cdot,z_i)}{h}}\right\Vert_{L^2(\dot \Omega_i)}
 e^{-\frac1h \inf_{ {\rm supp    }\left(1-\chi\right) \nabla \chi_i  } (d_a(\cdot,z_i)+f)} \nonumber \\
&\leq  C h^{-N}  e^{- \frac1h \inf_{B_{z_i}^c} (d_a(\cdot,z_i)+f-\delta)}. \label{dodoo2}
 \end{align} 
Besides, from assumption \eqref{hypo1}  
$$\inf_{z\in B_{z_i}^c}\left[ d_a(z,z_i)+f(z)\right]> f(z_i).$$
Indeed,  the inequality \eqref{hypo1} implies that there exists $r>0$ such that for all $z\in B_{z_i}^c$, $d_a(z,z_i)\geq f(z_i)-f(z_1)+r$ and therefore for all $z\in B_{z_i}^c$ one obtains
\begin{align*}
d_a(z,z_i)+f(z)&\geq f(z_i)+ (f(z)-f(z_1)) +r \geq f(z_i)+r.
\end{align*}
Therefore, taking $\chi_i$ such that $\delta < r/2$, one has, when $h\to 0$
$$ \int_{{\rm supp    } \left(1-\chi\right)  \nabla \chi_i}
 \left\vert\left(1-\chi\right)   \nabla \chi_i \cdot  u^{(1)}_{h,i}\
   e^{-  \frac{f}{h}} \right\vert = O\left(
 e^{-\frac{f(z_i)+c}{h}} \right)$$
for some constant $c>0$.\\
In conclusion, for all $i\in \{1,\ldots,n\}$, one has in the limit $h \to 0$,
$$\int_{\Omega} \nabla  \tilde u  \cdot \tilde \phi_i \,  e^{- \frac{1}{h}f} dx =        C_i h^p    e^{-\frac{1}{h}(f(z_i)- f(x_0))}   \    (  1  +     O(h  )    ) ,$$
with \label{page.ci}
\begin{equation}\label{Cip}
\left\{
\begin{aligned}
& C_i=-B_i \, \frac{( {\rm det \ Hess } f   (x_0)   )^{1/4}  }{
  \pi^{\frac{d}{4} }}=- \frac{ \pi ^{-\frac{1}{4}} \sqrt{2\partial_n f(z_i) }  ( {\rm det \ Hess } f   (x_0)   )^{1/4}  }{ ( {\rm det \ Hess } f|_{\partial \Omega }   (z_i)   )^{1/4}  }\, ,\\
& p=m-\frac{d}{4}=-\frac34,
\end{aligned}\right.
\end{equation}
where $B_i$ and $m$ have both been defined in~\eqref{mBi}.
% Let us notice that the sign of $C_i$ is relevant since $\partial_nu_h$ is negative and is approached by the following function
% $$z\in \partial \Omega \mapsto  \sum_{i=1}^n\left(\int_{\Omega} \nabla  \tilde u  \cdot \tilde \phi_i \,  e^{- \frac{1}{h}f(z)} dx\right)  \,  e^{-\frac fh}\,   \tilde \phi_i (z)\cdot n. $$
This concludes the proof of~\eqref{eq:assump_3_phi}, and thus the proof of Theorem~\ref{TBIG0}.

\section{Consequences and generalizations of Theorem \ref{TBIG0}}\label{sec:proofs}

\subsection{Proofs  of Proposition~\ref{moyenneu}, Proposition~\ref{lambdah}, Corollary~\ref{co.proba-sigmai} and Corollary~\ref{cc1}} \label{section17}

%In this section we assume that \textbf{[H1]}, \textbf{[H2]} and
%\textbf{[H3]} together with the inequalities
%\eqref{hypo1} and \eqref{hypo2}.

\subsubsection{Proof of Proposition~\ref{moyenneu}} \label{sec:moyeu}
Assume that \textbf{[H1]}, \textbf{[H2]} and \textbf{[H3]} hold.
From Lemma~\ref{pi0u} and since the function $\tilde u$ is non negative in $\Omega$, there exists $h_0>0$ such that for all $h\in (0,h_0)$
$$u_h=\frac{\pi_h^{(0)} \tilde u}{\left\Vert \pi_h^{(0)}\tilde u \right\Vert_{L^2_w}},$$
where $u_h$ is the eigenfunction associated with the smallest eigenvalue of $-L^{(D),(0)}_{f,h}(\Omega)$ (see Proposition~\ref{UU})  and $\tilde u$ is introduced in Definition~\ref{tildeu}.
Then, there exists $h_0>0$ such that for all $h\in (0,h_0)$,
\begin{align*}
\int_{\Omega}u_h \ e^{-\frac{2}{h} f}&=\int_{\Omega}\frac{\pi_h^{(0)} \tilde u}{\left\Vert \pi_h^{(0)} \tilde u \right\Vert_{L^2_w}}e^{-\frac{2}{h} f}=\frac{1}{\left\Vert \pi_h^{(0)} \tilde u \right\Vert_{L^2_w}  }\int_{\Omega} \left[\tilde u + \left(\pi_h^{(0)} -1\right) \tilde u\right] e^{-\frac{2}{h} f} .
\end{align*}
 From the definition of $\chi$ (see Definition~\ref{tildeu}) and using Laplace's method, one obtains (in the limit $h\to0$)
 $$\int_{\Omega} \chi^2 e^{-\frac{2}{h} f } = \frac{ h^{\frac{d}{2}} \pi ^{\frac{d}{2} } }{ \sqrt{  {\rm det \ Hess } f   (x_0)   }   }e^{-\frac{2}{h} f(x_0)} (1+O(h) )$$
 and likewise
 $$\int_{\Omega} \chi e^{- \frac{2}{h}f } = \frac{ h^{\frac{d}{2}} \pi ^{\frac{d}{2} } }{ \sqrt{  {\rm det \ Hess } f   (x_0)   }   }e^{-\frac{2}{h} f(x_0)} (1+O(h) ).$$
In addition, using Lemma~\ref{pi0u}, one has $\left\Vert \pi_h^{(0)} \tilde u\right\Vert_{L^2_w}=1+O\left(e^{-\frac{c}{h}}\right)$. Therefore, it holds when $h\to 0$,
\begin{align*}
\frac{1}{\left\Vert \pi_h^{(0)} \tilde u \right\Vert_{L^2_w}}\int_{\Omega} \tilde u \ e^{- \frac{2}{h} f }&=  \frac{\displaystyle \int_\Omega\chi\ e^{- \frac{2}{h} f }}{ \sqrt{\displaystyle \int_{\Omega} \chi^2 e^{- \frac{2f}{h}}}}\big (1+O\big(e^{-\frac{c}{h}}\big) \big )\\
&= \frac{ h^{\frac{d}{4}}\pi ^{\frac{d}{4} } }{ (  {\rm det \ Hess } f   (x_0)   )^{1/4}  }e^{-\frac{1}{h} f(x_0)} (1+O(h) ).
\end{align*}
Moreover, from Lemma~\ref{pi0u}, there exist $c>0$, $h_0>0$ and $C>0$ such that for $h\in (0,h_0)$
\begin{align*}
\frac{1}{\left\Vert \pi_h^{(0)} \tilde u \right\Vert_{L^2_w}  }\left\vert  \int_{\Omega}  \left(\pi^{(0)}_h -1\right) \tilde u \ e^{-\frac{2}{h} f} \right\vert &\leq  C \left\Vert (1-\pi_h^{(0)} )\tilde u\right\Vert_{L^2_w}  \sqrt{  \int_{\Omega} e^{-\frac{2}{h} f}  }\leq  C   e^{-\frac{c}{h} }  \ e^{-\frac{1}{h} f(x_0)}
 \end{align*}
Thus, one has when $h\to 0$, $$\int_{\Omega} u_h \ e^{- \frac{2}{h} f }= \frac{  \pi ^{\frac{d}{4} } }{  (  {\rm det \ Hess } f   (x_0)   )^{1/4}   } \ h^{\frac{d}{4} } \ e^{-\frac{1}{h}f(x_0)} (1+O(h) ).$$ This proves Proposition~\ref{moyenneu}.

\subsubsection{Proof of Proposition~\ref{lambdah} } 
\label{sec:lambdah} 
The aim of this section is to prove~\eqref{eq:lhh}.  To this end, we first state in Proposition~\ref{pr.lh} some estimates that  the quasi-modes constructed  in Section~\ref{sec:contruct_quasimode} satisfy under hypotheses \textbf{[H1]}, \textbf{[H2]} and \textbf{[H3]}. Let us emphasize that these estimates are weaker than those obtained in Section~\ref{sec:goodquasimodes} where in addition to \textbf{[H1]}-\textbf{[H2]}-\textbf{[H3]},   the hypotheses \eqref{hypo1} and \eqref{hypo2} were also assumed. Then, we prove that the estimates of Proposition~\ref{pr.lh} imply \eqref{eq:lhh}.  
\begin{proposition}
\label{pr.lh}
Let us assume that the hypotheses \textbf{[H1]}, \textbf{[H2]} and \textbf{[H3]} hold. Then there exist $n+1$ quasi-modes $( ( \tilde \psi_i)_{i=1,\ldots,n}, \tilde u)$ which satisfy the following estimates:
\begin{enumerate}
\item $\forall i\in \left\{1,\ldots,n\right\}$, $\tilde \psi_i\in \Lambda^1 H^1_{w,T}(\Omega)$ and $\tilde u \in \Lambda^0 H^1_{w,T}(\Omega)$. The function $\tilde u$ is non negative in $\Omega$. Moreover  $\forall i\in \left\{1,\ldots,n\right\}$,  $\left\|\tilde \psi_i\right\| _{L^2_w}= \left\|\tilde u \right\| _{L^2_w} = 1$.
\item
\begin{itemize}
\item[(a)]       There exists  $\ve_1>0$,  for all $ i\in \{1,\ldots,n\}$, in the limit $h\to 0$:
$$
  \left \|    \left(1-\pi_h^{(1)}\right) \tilde \psi_i \right\|_{H^1_w}^2  =    O\left(e^{-\frac{\ve_1}{h}}  \right).
$$ 
\item[(b)] For any $\delta>0$, $    \|   \nabla  \tilde u\|_{L^2_w}^2     =   O \left(e^{-\frac{2}{h}(f(z_1)-f(x_0) - \delta)} \right)$.
\end{itemize}
\item   There exists $\ve_0>0$ such that $\forall (i,j) \in \left\{1,\ldots,n\right\}^2$, $i<j$, in the limit $h\to 0$: $\left\lp \tilde \psi_i, \tilde \psi_j\right\rp_{L^2_w}=O \left( e^{-\frac{\ve_0}{h}}  \right )$.
\item  There exist $\ve_0>0$, such that for all $i\in \{1,\ldots,n\}$,  in the limit $h\to 0$:
$$
  \left \lp       \nabla \tilde u  ,    \tilde \psi_i\right\rp_{L^2_w}  =      C_i \ h^p  e^{-\frac{1}{h}(f(z_i)- f(x_0))}   \,    (  1  +     O(h )   )+O\left(\,e^{-\frac{1}{h}(f(z_{1} )- f(x_0)+\ve_0)}\right),
$$
where the constants $p$ and  $(C_i)_{i=1,\dots,n}$ are given by \eqref{Cip}. 
  
 \end{enumerate}
\end{proposition}

\begin{proof}  
Thanks to the hypotheses \textbf{[H1]}, \textbf{[H2]} and \textbf{[H3]}, one can introduce the  $n+1$ quasi-modes $( ( \tilde \phi_i)_{i=1,\ldots,n}, \tilde u)$ built in Section~\ref{sec:contruct_quasimode}. Recall that $\tilde \psi_i=e^{\frac 1h f}\tilde \phi_i$ for $i\in \{1,\ldots,n\}$.

Then, one easily obtains that  $( ( \tilde \psi_i)_{i=1,\ldots,n}, \tilde u)$ satisfy the estimates stated in Proposition~\ref{pr.lh}, following exactly the computations made on $( ( \tilde \phi_i)_{i=1,\ldots,n}, \tilde u)$ in Section~\ref{sec:goodquasimodes}: 2(a) follows from~\eqref{eq:majo_inf_supp}, 2(b) is a consequence of Lemma~\ref{nablauu}, 3 follows from~\eqref{eq:majo_phii_phij} and 4 is a consequence of \eqref{dodoo}-\eqref{dodoo1}-\eqref{dodoo2} (in~\eqref{dodoo2}, one uses that for $\delta>0$ small enough, there exists $c>0$ such that  $\inf_{B_{z_i^c}} (d_a(.,z_i)+f-\delta)\ge f(z_1)+c$) . \end{proof}

Let us now prove that the estimates stated in Proposition~\ref{pr.lh} imply~\eqref{eq:lhh}, which will conclude the proof of Proposition~\ref{lambdah}.
\begin{proof}

From~\eqref{eq:u} together with the assumption $\Vert  u_h \Vert_{L^2_w}=1$, it holds 
\begin{equation}\label{eq:lambdah-dec}
\lambda_h =-\langle L^{D,(0)}_{f,h}(\Omega)\, u_h,u_h \rangle_{L^2_w}=\frac{h}{2} \Vert \nabla  u_h \Vert^2_{L^2_w}
\end{equation}
where $u_h$ is the eigenfunction associated with the smallest eigenvalue of $-L^{D,(0)}_{f,h}(\Omega)$ (see Proposition~\ref{UU}). Recall that $\nabla u_h \in \range   \pi_h^{(1)}$ (see~\eqref{eq:uh_in_span}).

In addition, let us recall that from items 1, 2(a) and 3 in Proposition~\ref{pr.lh} and using the proof of Lemma~\ref{indep}, there exists $h_0$ such that for all $h \in (0,h_0)$,
$$
{\rm span}\left( \pi_h^{(1)}\tilde \psi_i, \,
 i=1, \ldots,n\right)=\range   \pi_h^{(1)}.
$$
Let us denote by $(\psi_i)_{
 i=1, \ldots,n }$ the orthornormal basis of $\range   \pi_h^{(1)}$ resulting from the Gram-Schmidt orthonormalisation procedure applied to the set $(\pi_h^{(1)}\tilde \psi_i)_{
 i=1, \ldots,n}$ (see Lemma~\ref{gram}) so that
\begin{equation}\label{eq:dnuh_decomp1}
 \Vert \nabla  u_h \Vert^2_{L^2_w} =
\sum_{j=1}^n \left | \langle \nabla u_h , \psi_j\rp_{L^2_w}\right |^2.
\end{equation}
We now want to estimate the terms $\langle \nabla u_h , \psi_j\rp_{L^2_w}$.
 
 Using  2(b) in Proposition~\ref{pr.lh}  and using the proof of Lemma~\ref{pi0u}, one has that for~$h$
small enough $\pi_h^{(0)} \tilde u\neq 0$ and therefore, since moreover $\tilde u$ is non negative in $\Omega$,  
 $u_h  =      \frac{\pi_h^{(0)} \tilde u}{\left\|\pi_h^{(0)}  \tilde
     u\right\|_{L^2_w}}$. Thus one has (see~\eqref{eq:nablauh_psij}), for $j \in \{1, \ldots, n\}$,
 \begin{align*}  \lp    \nabla u_h ,   \psi_j \rp_{L^2_w}    &=\frac{Z^{-1}_j}{\left\|\pi_h^{(0)} \tilde u \right\|_{L^2_w}} \left[   \left\lp       \nabla \tilde u  ,    \tilde \psi_j  \right\rp_{L^2_w} +  \left\lp   \nabla \tilde u, \left( \pi_h^{(1)}-1   \right) \tilde \psi_j  \right\rp_{L^2_w}   \right]  \nonumber \\
 &\quad +\frac{Z^{-1}_j}{\left\|\pi_h^{(0)} \tilde u \right\|_{L^2_w}}\left[\sum_{i=1}^{j-1}\kappa_{ji} \ \left(  \left\lp       \nabla \tilde u  ,    \tilde \psi_i  \right\rp_{L^2_w} + \left\lp   \nabla \tilde u, \left( \pi_h^{(1)}-1   \right) \tilde \psi_i  \right\rp_{L^2_w}  \right)\  \right]
  \end{align*} 
where $(\kappa_{ji})_{1\le i<j \le n}$ and $(Z_j)_{1\le j \le n}$ are defined in Lemma~\ref{gram}.
 
Now, using the items 1, 2(a) and 3 of Proposition~\ref{pr.lh} and the proof of Lemma~\ref{estimate1}, one  obtains that there exist $\ve_0>0$ and $h_0>0$ such that  $\forall h\in (0,h_0)$, $\forall (i,j) \in \left\{1,\ldots,n\right\}^2$, it holds
\begin{equation}\label{eq.kij_lh}
\kappa_{ji}=O \left(e^{-\frac{\varepsilon_0}{h}} \right ) \text{ and } Z_i=1+O \left( e^{-\frac{\ve_0}{h}}\right),
\end{equation}
Injecting \eqref{eq.kij_lh} and the estimates 2 and 4 of Proposition~\ref{pr.lh} into~\eqref{eq:nablauh_psij} leads to the existence of $\ve'>0$ and $h_0>0$ such that  $\forall h\in (0,h_0)$,
\begin{equation}\label{eq.nablaupsi_lh}
  \left \lp       \nabla u_h  ,\psi_j\right\rp_{L^2_w}  =        C_i \ h^p  e^{-\frac{1}{h}(f(z_j)- f(x_0))}   \,    (  1  +     O(h )   )+O\left(e^{-\frac{1}{h}(f(z_{1} )- f(x_0)+\ve')}\right),
\end{equation}
where  the constants $p$ and  $(C_i)_{i=1,\dots,n}$ are given by~\eqref{Cip}.
Using \eqref{eq.nablaupsi_lh} in~\eqref{eq:lambdah-dec} and~\eqref{eq:dnuh_decomp1}, there exists $\ve'>0$ such that in the limit $h\to 0$
\begin{equation*}
\lambda_h =\frac h2
\sum_{j=1}^n  C_i^2 \ h^{2p}  e^{-\frac{2}{h}(f(z_i)- f(x_0))}   \,    (  1  +     O(h )   ) + O\left(\,e^{-\frac{2}{h}(f(z_{1} )- f(x_0)+\ve')}\right).\end{equation*}
Therefore, the estimate~\eqref{eq:lhh}  holds and Proposition~\ref{lambdah} is proved.
\end{proof}

\subsubsection{Proof of Corollary~\ref{co.proba-sigmai}} \label{proofproba} 
 According to (\ref{eq:dens}) one  has for $i\in\{1,\dots,n\}$:
$$\P_{\nu_h} [  X_{\tau_{\Omega}} \in \Sigma_i]=- \frac{h}{2\lambda_h} \frac{\int_{\Sigma_i} (\partial_n u_h)(z) \, e^{-\frac{2}{h}  f(z)}\sigma(dz)}{\int_\Omega u_h(y) e^{-\frac{2}{h}  f(y)}dy}.$$
Let us assume that \textbf{[H1]}, \textbf{[H2]} and
\textbf{[H3]} together with the inequalities
\eqref{hypo1} and \eqref{hypo2}  hold.
From Propositions~\ref{moyenneu} and~\ref{lambdah}, one obtains when $h\to 0$
\begin{align*}
\lambda_h \frac{2}{h} \int_{\Omega} u_h \ e^{- \frac{2}{h}f }dx&=2 \pi ^{\frac{d-2}{4}} \left ( {\rm det \ Hess } f   (x_0)  \right )^ {\frac{1}{4} }  h^{\frac{d-6}{4}}\\
&\quad \times \sum_{k=1}^{n_0} \frac{  \partial_nf(z_k)    }{ \sqrt{ {\rm det \ Hess } f|_{\partial \Omega}   (z_k) }  } \, e^{-\frac{1}{h}(2f(z_1)- f(x_0))} (1+O(h) ).
\end{align*}
Then, using in addition Theorem~\ref{TBIG0} to estimate $\int_{\Sigma_i} (\partial_n u_h) \, e^{-\frac{2}{h}  f}d\sigma$ ($i\in\{1,\dots,n\}$), one proves Corollary~\ref{co.proba-sigmai}.

\subsubsection{Proof of Corollary~\ref{cc1}}
Before starting the proof of Corollary~\ref{cc1}, let us notice that under the
assumptions stated in  Corollary~\ref{co.proba-sigmai}, for all $i\in \{1,\dots,n\}$ and for any test function
$F\in C^{\infty}(\partial \Omega)$ satisfying ${\rm supp}\, F
\subset B_{z_i}$ and $z_i\in {\rm int} \left ({\rm supp} \, F\right)$,  when $h\to 0$,
\begin{equation}\label{eq:EnuF}
\E_{\nu_h} [ F(X_{\tau_{\Omega}} )]=     \frac{\partial_nf(z_i)  }{ \sqrt{ {\rm det \ Hess } f_{|\partial \Omega }   (z_i) }}  
\left ( \sum_{k=1}^{n_0} \frac{ \partial_nf(z_k) }{\sqrt{ {\rm det \ Hess } f_{|\partial \Omega }   (z_k) } } \right)^{-1}  e^{-\frac{2}{h} (f(z_i)-f(z_1))} (   F(z_i) +   
 O(h) ).
\end{equation}
The strategy for the proof of Corollary~\ref{cc1} is to first extend~\eqref{eq:EnuF} to a deterministic initial condition, and then to deduce the result of Corollary~\ref{cc1}.
To this end, let $i_0\in\{2,\ldots, n\}$ be as in~\eqref{eq:hypo2_bis}, $j\in \{1,\ldots,i_0\}$ and $\alpha \in \mathbb R$ be such that
$$f(x_0)<\alpha<2f(z_1)-f(z_{j}).$$
Notice that we can
assume without loss of generality (up to increasing $\alpha$ if $\alpha$ is
smaller than $f(x_0)+f(z_{j})-f(z_1)$, see~\eqref{eq:hypo2_bis}) that
\begin{equation}\label{eq:hyp_alpha}
f(x_0)+f(z_{j})-f(z_1)<\alpha< 2 f(z_1)-f(z_{j}).
\end{equation}
For such an $\alpha$, let us define
\label{page.kalpha}
$$K_\alpha:=f^{-1}\left( \left(-\infty, \alpha \right]\right) \cap \Omega.$$
For $i\in \{1,\ldots,j\}$, we would like to show that~\eqref{eq:EnuF} holds  when  $X_0= x\in K_\alpha $, for any test function
$F\in C^{\infty}(\partial \Omega)$ satisfying ${\rm supp}\, F
\subset B_{z_i}$ and $z_i\in {\rm int} \left ({\rm supp} \, F\right)$.
\medskip

\noindent
 Let us  introduce the principal eigenfunction $\tilde{u}_h$ of $-L^{D,(0)}_{f,h}(\Omega)$:
 \label{page.tildeuh}
\begin{equation}\label{eq:vh_eig}
\left\{
\begin{aligned}
 -L^{ (0)}_{f,h}  \tilde{u}_h &=  \lambda_h \tilde{u}_h    \ {\rm on \ }  \Omega,  \\ 
\tilde{u}_h&= 0 \ {\rm on \ } \partial \Omega, 
\end{aligned}
\right.
\end{equation}
with
 $\tilde{u}_h>0$ on $\Omega$ and normalized such that
\begin{equation}\label{eq:vh_norm}
 \int_{\Omega} \tilde{u}_h^2\, dx=1.
\end{equation} Notice that $u_h$ solution to~\eqref{eq:u} only differs from $\tilde{u}_h$ by
a multiplicative constant so that, from Proposition~\ref{uniqueQSD}
\begin{equation} \label{eq:expQSD_bis}
\nu_h(dx)=Z_h(\Omega)^{-1}  \tilde{u}_h(x) e^{-\frac{2}{h}  f(x)}dx,
\end{equation}
where, for any set $O\subset \Omega$,
\label{page.zho}
$$Z_h(O):=\int_{O} \tilde{u}_h \ e^{- \frac{2}{h} f }.$$
For $F\in C^{\infty}(\partial \Omega)$, let us define 
\label{page.whh}
$$w_h(x)=\mathbb
E_{x}   [ F (X_{\tau_{\Omega}} )] \text{ for all } x\in \overline
\Omega.$$
 The function $w_h$ is such that: $\forall h >0$ and $x \in
\Omega$,
$$\vert w_h(x)\vert \leq \|F\|_{L^\infty}.$$  
Moreover, a standard Feynman-Kac formula shows that $w_h$ satisfies
\begin{equation}\label{eq:wh}
\left\{
\begin{aligned}
L^{(0)}_{f,h} w_h &=  0    \ {\rm on \ }  \Omega,  \\ 
w_h&= F \ {\rm on \ } \partial \Omega,
\end{aligned}
\right.
\end{equation}
where, we recall, the differential operator $L^{(0)}_{f,h}$ is defined by~\eqref{eq:L0}.
Our objective is to compare $w_h(x)$ with $\E_{\nu_h}  [F(X_{\tau_{\Omega}})]$.
\medskip

\noindent
By the Markov property,
using~\eqref{eq:expQSD_bis}, we have 
\begin{align}
\E_{\nu_h}  [F(X_{\tau_{\Omega}})]&= \left (\int_{\Omega} \tilde{u}_h\,
  e^{-\frac{2}{h}f} \right )^{-1}\left (\int_{\Omega}w_h\, \tilde{u}_h\,
  e^{-\frac{2}{h}f} \right ) \nonumber \\
&= Z_h^{-1}(\Omega)\left (\int_{\Omega \setminus K_{\alpha}}w_h\,
  \tilde{u}_h\, e^{-\frac{2}{h}f} \right) + Z_h^{-1}(\Omega)\left
  (\int_{K_{\alpha}} w_h\, \tilde{u}_h\, e^{-\frac{2}{h}f}  \right ). \label{eq:prob_sep}
\end{align}
\noindent
In order to estimate the first term in~\eqref{eq:prob_sep}, we need a
leveling property for $\tilde{u}_h$, which is stated in~\cite[Theorem 2.4]{devinatz-friedman-78a}.
\begin{lemma}\label{de1}  Let us assume that \textbf{[H1]},
  \textbf{[H2]} and \textbf{[H3]} hold and let us consider $\tilde{u}_h$ the
  principal eigenfunction of $L^{D,(0)}_{f,h}(\Omega)$ (see~\eqref{eq:vh_eig}) with normalization~\eqref{eq:vh_norm}. Then, for any compact set $K\subset \Omega$,
 $$\lim_{h\to 0} \left\|\tilde{u}_h - \left(\int_{\Omega}  dx\right )^{-1/2}\right\|_{L^\infty(K)}=0.$$ %Moreover there exists $\gamma>0$ such that $\Vert\nabla \tilde{u}_h\Vert_{L^{\infty}(\Omega)} \leq h^{-\gamma}$.
\end{lemma}
\noindent
Notice that the reason why we consider a smooth test function $F$ rather than $1_{\Sigma_i}$ is that we would like to apply the results in~\cite{devinatz-friedman-78a}.
\medskip

\noindent
 A direct consequence of Lemma~\ref{de1} is the following limit,
\begin{equation}\label{e3}
\lim \limits_{h\to 0} h\ln  \left(Z_h(\Omega)\right)=-2f(x_0).
\end{equation}
Indeed,  from the normalization of $\tilde{u}_h$, we get $ Z_h(\Omega)\leq  e^{-\frac{2}{h}f(x_0)}$, and from Lemma~\ref{de1} we have,   for $h$ small enough and for $r>0$ such that the open ball $B(x_0,2r)$ is included in~$\Omega$,
$$ Z_h(\Omega) \geq  Z_h(B(x_0,r))\geq \frac{1}{2}\left(\int_{\Omega}  dx\right )^{-1/2} \int_{B(x_0,r)}e^{-\frac{2f}{h}}.$$
Since $\displaystyle{\lim_{h\to 0} h\ln  \left(\int_{B(x_0,r)}
    e^{-\frac{2f}{h}}dx\right)=-2f(x_0)}$, we get \eqref{e3}. 
\medskip

\noindent
Let us now consider $i\in \{1,\ldots,j\}$. For the first term in~\eqref{eq:prob_sep}, using \eqref{e3} and~\eqref{eq:vh_norm}, we have for any $\delta>0$, for $h$ small enough,
\begin{align*}
Z_h(\Omega)^{-1}\left \vert \ \int_{\Omega \setminus K_{\alpha}}w_h\,
  \tilde{u}_h\, e^{-\frac{2}{h}f}\right \vert&\leq  \|F\|_{L^\infty}\,
e^{\frac{\delta}{h} }\, e^{-\frac{2}{h}\left(\inf_{\Omega \setminus
      K_{\alpha}} f-f(x_0)\right) }\\
&=\|F\|_{L^\infty} \, e^{\frac{\delta}{h} }\, e^{-\frac{2}{h}(\alpha
  -f(x_0)) }
\end{align*}
and thus, thanks to~\eqref{eq:hyp_alpha} and the fact that $f(z_i)\le f(z_j)$, by choosing $\delta$ small enough, there exists $c>0$ such
that, for all $h$ small enough,
\begin{equation}\label{eq:estim_1}
Z_h(\Omega)^{-1}\left \vert \ \int_{\Omega \setminus K_{\alpha}}w_h\,
  \tilde{u}_h\, e^{-\frac{2}{h}f} \right \vert \leq \|F\|_{L^\infty} \, e^{-\frac{2}{h}(f(z_i)-f(z_1) + c) }.
\end{equation}
In order to estimate the second term in~\eqref{eq:prob_sep}, we need 
a leveling property for $w_h$.
\begin{lemma}\label{le.eiz}
Let us assume that \textbf{[H1]}, \textbf{[H2]} and \textbf{[H3]}
hold, as well as~\eqref{eq:hyp_alpha}. Let us consider $w_h$ solution
to~\eqref{eq:wh}. Then there exists $C>0$ such
that for any $\delta>0$, for any $h$ small enough, for all $x,y\in K_\alpha$, 
$$\vert w_h(x)-w_h(y)\vert \leq C\, e^{\frac{\delta}{h} }\,  e^{-\frac{2}{h}(f(z_1)-\alpha ) }.$$
\end{lemma}
\begin{proof}
From~\cite[Theorem 1]{eizenberg-90}, it is known that for any $\delta
>0$, for any $h$ small enough and for all $x,y\in K_\alpha$, 
$$\vert w_h(x)-w_h(y)\vert \leq C\, e^{\frac{\delta}{h} }\,  e^{-\frac{2}{h} V_\Omega(K_\alpha) },$$
where $V_\Omega(K_\alpha)$ is defined by,
$$V_\Omega(K_\alpha)=\inf \limits_{x\in K_\alpha}\, \inf \limits_{T> 0} \,  \inf \limits_{ \gamma \in {\rm Abs}(T,x,\partial \Omega)}  \, \frac{1}{4}\int_0^T \left \vert \dot{\gamma}+\nabla f(\gamma)\right \vert^2 dt $$
where ${\rm Abs}(T, x,\partial \Omega)$ is the set of absolutely continuous functions $\gamma:[0,T]\to \overline \Omega$ satisfying $\gamma(0)=x$ and $\gamma(T)\in \partial \Omega$. For any $\gamma \in {\rm Abs}(T, x,\partial \Omega)$, we have 
$$ \int_0^T \vert \dot{\gamma}+\nabla f(\gamma)\vert^2 dt-\int_0^T \vert \dot{\gamma}-\nabla f(\gamma)\vert^2 dt=4\int_0^T  \dot{\gamma}\cdot \nabla f(\gamma) dt=4\left ( f(\gamma(T))-f(x)\right ),$$
and therefore, it holds 
$$ \int_0^T \vert \dot{\gamma}+\nabla f(\gamma)\vert^2 dt\geq 4\left ( f(\gamma(T))-f(x)\right )\geq 4\left ( f(z_1)-f(x) \right ).$$
Finally we obtain 
$$V_\Omega(K_\alpha)\geq f(z_1)-\max \limits_{x\in K_\alpha} f(x)=f(z_1)-\alpha.$$
This concludes the proof of Lemma~\ref{le.eiz}.
\end{proof}
\noindent
We are now in position to estimate the second term
in~\eqref{eq:prob_sep}. Using Lemma~\ref{le.eiz}, we get, for any $\delta
>0$, in the limit $h \to 0$, uniformly in $y_0\in K_\alpha$,

\begin{align*}
Z_h^{-1}(\Omega)\left (\int_{K_{\alpha}} w_h\, \tilde{u}_h\,
  e^{-\frac{2}{h}f}dx\right )&=
w_h(y_0)\frac{Z_h(K_\alpha)}{Z_h(\Omega)}+ O\left (
  e^{\frac{\delta}{h} }\,  e^{-\frac{2}{h}(f(z_1)-\alpha ) }\right) \,
\frac{Z_h(K_\alpha)}{Z_h(\Omega)}.
\end{align*}
Therefore, by choosing $\delta > 0$ small enough, thanks
to~\eqref{eq:hyp_alpha} and the fact that $f(z_i)\le f(z_j)$ (since we recall that $i\in \{1,\ldots,j\}$),  there exists $c>0$ such that, in the limit $h \to 0$,
\begin{equation}\label{eq:estim_2}
Z_h^{-1}(\Omega)\left (\int_{K_{\alpha}} w_h\, \tilde{u}_h\,
  e^{-\frac{2}{h}f}dx\right )
=w_h(y_0)\frac{Z_h(K_\alpha)}{Z_h(\Omega)}+ O\left (  e^{-\frac{2}{h} (f(z_i)-f(z_1)+c )}\right) \, \frac{Z_h(K_\alpha)}{Z_h(\Omega)}.
\end{equation}
In addition we have 
 $\frac{Z_h(K_{\alpha})}{Z_h(\Omega)}=1+O\left (e^{-\frac{c}{h}}\right )$ for some $c>0$ independent of $h$. Indeed $\frac{Z_h(K_{\alpha})}{Z_h(\Omega)}=1-\frac{Z_h(\Omega\setminus K_{\alpha})}{Z_h(\Omega)}$ and using \eqref{e3}, we get  for any $\delta>0$,
\begin{equation}\label{eq:estim_3}
\frac{Z_h(\Omega\setminus K_{\alpha})}{Z_h(\Omega)}\leq
e^{\frac{\delta}{h}}\, e^{-\frac{2}{h} \left(\min_{\Omega\setminus
      K_{\alpha}} f-f(x_0) \right) }=O\left (e^{-\frac{c}{h}}\right ),
\end{equation}
for some $c>0$ independent of $h$ by choosing $\delta$ small
enough. Gathering the results~\eqref{eq:estim_1}--\eqref{eq:estim_2}--\eqref{eq:estim_3} in~\eqref{eq:prob_sep}, for all $i\in \{1,\ldots,j\}$, there exists $c>0$ independent of $h$ such that, in the
limit $h\to 0$, it holds: uniformly in $y_0\in K_\alpha$,
$$\E_{\nu_h}  [ F (X_{\tau_{\Omega}} )]=w_h(y_0) \left( 1+O\left (e^{-\frac{c}{h}}\right )\right ) + O\left (  e^{-\frac{2}{h} (f(z_i)-f(z_1)+c ) }\right).$$
Let $i\in
\{1,\ldots,j\}$ and let us assume  that ${\rm supp}\, F
\subset B_{z_i}$ and $z_i\in {\rm int} \left ({\rm supp} \, F\right)$. Then, combining the last estimate with~\eqref{eq:EnuF} implies that uniformly in $x\in
f^{-1}( (-\infty, \alpha ]) \cap \Omega$, in the limit $h\to 0$: 
\begin{equation}\label{eq:EnuFx}
\E_{x} [ F(X_{\tau_{\Omega}} )]=     \frac{\partial_nf(z_i)  }{ \sqrt{ {\rm det \ Hess } f_{|\partial \Omega }   (z_i) }}  
\left ( \sum_{k=1}^{n_0} \frac{ \partial_nf(z_k) }{\sqrt{ {\rm det \ Hess } f_{|\partial \Omega }   (z_k) } } \right)^{-1}  e^{-\frac{2}{h} (f(z_i)-f(z_1))} (   F(z_i) +   
 O(h) ).
\end{equation}
Let  $\Sigma_i \subset \partial \Omega$ containing $z_i$ and
 such that $\overline\Sigma_i \subset B_{z_i}$. Then, there exit $F,G\in C^{\infty}(\partial \Omega)$ such that ${\rm supp}\, F
\subset B_{z_i}$, $z_i\in {\rm int} \left ({\rm supp} \, F\right)$, ${\rm supp}\, G
\subset B_{z_i}$, $z_i\in {\rm int} \left ({\rm supp} \, G\right)$, $F\le 1_{\Sigma_i}\le G$ and $F(z_i)=G(z_i)=1$. From the inequality 
$$\E_{x} [ F(X_{\tau_{\Omega}} )]\le \P_{x} [ X_{\tau_{\Omega}} \in \Sigma_i]\le \E_{x} [ G(X_{\tau_{\Omega}} )],$$
together with \eqref{eq:EnuFx} applied to $F$ and $G$, one gets, for all $i\in \{1,\ldots,j\}$,  in the limit $h\to 0$:
$$
\P_{x} [ X_{\tau_{\Omega}}\in \Sigma_i]=      \frac{\partial_nf(z_i)  }{ \sqrt{ {\rm det \ Hess } f_{|\partial \Omega }   (z_i) }}  
\left ( \sum_{k=1}^{n_0} \frac{ \partial_nf(z_k) }{\sqrt{ {\rm det \ Hess } f_{|\partial \Omega }   (z_k) } } \right)^{-1}  e^{-\frac{2}{h} (f(z_i)-f(z_1))} (   1 +   
 O(h) ).$$This concludes the proof of Corollary~\ref{cc1}.

%%%%%%%%

\subsection{Proofs of Theorem~\ref{th.gene_sigma} and Corollary~\ref{co.gene_sigma}}\label{sec.section-thm.gene} 
In this section, we prove Theorem~\ref{th.gene_sigma}. The proof is similar to the one made for Theorem~\ref{TBIG0}:   the estimates of Proposition~\ref{ESTIME} and the construction of the quasi-mode associated with $z_{j_0}$ are modified. The proof of Theorem~\ref{th.gene_sigma} is organized as follows. In Section~\ref{sec.required-estimates-THM2}, we give the estimates required for the $n+1$~quasi-modes. Then, in Section~\ref{sec.estime3}, we prove that these estimates imply  Theorem~\ref{th.gene_sigma}. In Section~\ref{sec:const-qm}, the construction of the quasi-modes is given and we check that they satisfy the estimates  stated in Section~\ref{sec.required-estimates-THM2}. 
   
\subsubsection{Statement of the assumptions required for the
  quasi-modes}\label{sec.required-estimates-THM2}
Let su recall that, for $p\in \{0,1\}$, the
orthogonal projector $\pi_{[0,\frac{\sqrt h}{2})}\left(
  -L^{D,(p)}_{f,h}(\Omega)\right)$ is still denoted by $\pi^{(p)}_h$, see~\eqref{eq.pih}.
\medskip

\noindent
The next proposition gives the assumptions   needed on the quasi-modes
$(  \tilde \psi_i)_{i=1,\ldots,n}$ whose span approximates $\range
\pi_h^{(1)}$, and $\tilde u$ whose span approximates $\range
\pi_h^{(0)}$, in order to prove Theorem~\ref{th.gene_sigma}. It is the equivalent of Proposition~\ref{ESTIME} in the more general setting of Theorem~\ref{th.gene_sigma}.

 \begin{proposition} \label{ESTIME3}
 Let us assume that the hypotheses \textbf{[H1]}, \textbf{[H2]} and \textbf{[H3]} hold. Let $\Sigma_i $ denotes an
open set included in $\partial \Omega$ containing $z_i$ ($i\in \{1,\dots,n\}$) and such that
$\overline\Sigma_i \subset B_{z_i}$.  Let $k_0\in\{1,\dots,n\}$ and $f^*$ such that $$f(z_{k_0}) \le f^*\le f(z_{k_0+1}),$$ 
\label{page.ko}
(with the convention $f(z_{k_0+1})=+\infty$ if $k_0=n$). Finally, let $\Sigma\subset \partial \Omega$ be a smooth open set such that $\inf_{\Sigma} f=f^*$ and  $\overline \Sigma\subset B_{z_{j_0}}$,   for some $j_0\in\{1,\dots,k_0\}$.\label{page.jo}

Let us assume  that there exist $n$
quasi-modes $(  \tilde \psi_i)_{i=1,\ldots,n}$ and a family of quasi-modes $(\tilde u=\tilde u_{\delta})_{\delta>0}$ satisfying the following conditions:

\begin{enumerate}
\item $\forall i\in \left\{1,\ldots,n\right\}$, $\tilde \psi_i\in \Lambda^1 H^1_{w,T}(\Omega)$ and $\tilde u \in \Lambda^0 H^1_{w,T}(\Omega)$. The function $\tilde u$ is non negative in $\Omega$. Moreover, $\forall i\in \left\{1,\ldots,n\right\}$,  $\left\|\tilde \psi_i\right\| _{L^2_w}= \left\|\tilde u \right\| _{L^2_w} = 1$.

\item
\begin{itemize}
\item[(a)]       There exists  $\ve_1>0$   such that    for all $ i\in \{1,\ldots,k_0\}$, it holds in the limit $h\to 0$:
$$
  \left \|    \left(1-\pi_h^{(1)}\right) \tilde \psi_i \right\|_{H^1_w}^2  =    O\left(e^{-\frac{2}{h}\left( \max[f^*-f(z_i), f(z_i)-f(z_1)] +    \ve_1 \right)}  \right), 
$$ 
and for all $ i\in \{k_0+1,\ldots,n\}$,  $\left \|    \left(1-\pi_h^{(1)}\right) \tilde \psi_i \right\|_{H^1_w}^2   \   =  \    O \left(e^{-\frac{2}{h}(f^*-f(z_1) +    \ve_1 )}\right)$.
\item[(b)] For any $\delta>0$, in the limit $h\to 0$: $    \|   \nabla  \tilde u\|_{L^2_w}^2     =   O \left(e^{-\frac{2}{h}(f(z_1)-f(x_0) - \delta)} \right)$.

\end{itemize}

\item   There exists $\ve_0>0$, $\forall (i,j) \in \left\{1,\ldots,n\right\}^2$,  in the limit $h\to 0$,  if $i<j\leq k_0$:
$$\left\lp \tilde \psi_i, \tilde \psi_j\right\rp_{L^2_w}=O \left( e^{-\frac{1}{h}(f(z_j)- f(z_i)+\varepsilon_0)}  \right ),$$
and, if  $ k_0<j$, $i<j$:
$$\left\lp \tilde \psi_i, \tilde \psi_j\right\rp_{L^2_w}=O \left( e^{-\frac{1}{h}(f^*- f(z_1)+\varepsilon_0)} \right).$$
\item
\begin{itemize}
\item[(a)]  There exists $\ve'>0$ and there exist constants
  $(C_i)_{i=1,\ldots,n}$ and  $p$ which do not depend on~$h$ such that, in the limit $h\to 0$,  for $i\in \left\{1,\ldots,k_0\right\}$:
$$
  \left \lp       \nabla \tilde u  ,    \tilde \psi_i\right\rp_{L^2_w}  =        C_i \ h^p  e^{-\frac{1}{h}(f(z_i)- f(x_0))}   \,    (  1  +     O(h )   ),
  $$
and for $i\in \{k_0+1,\ldots,n\}$:
$$
  \left \lp       \nabla \tilde u  ,    \tilde \psi_i\right\rp_{L^2_w}  =        C_i \ h^p  e^{-\frac{1}{h}(f(z_i)- f(x_0))}   \,    (  1  +     O(h )   )+O\left(\,e^{-\frac{1}{h}(f^*- f(x_0)+\ve')}\right).
$$
%where the constants $p$ and  $(C_i)_{i=1,\dots,n}$ are given by \eqref{Cip}. 
\item[(b)]  
There exist constants
  $(B_i)_{i=1,\ldots,n}$ and  $m$ which do not depend on~$h$ such that for all $(i,j) \in
  \left\{1,\ldots,n\right\}^2$, in the limit $h\to 0$:

\begin{equation*}
  \int_{\Sigma_i}   \,  \tilde \psi_j \cdot n  \,   e^{- \frac{2}{h} f}d\sigma   =\begin{cases}   B_i \, h^m   \,     e^{-\frac{1}{h} f(z_i)}  \    (  1  +     O(h )    )   &  \text{ if } i=j   \\
 0   &   \text{ if } i\neq j.
  \end{cases} 
  \end{equation*}
%where the constants $(B_i)_{i=1,\dots,n}$ and $m$ are defined by \eqref{mBi}.

 \item[(c)] There exist $C^*$  and $p^*$ independent of $h$ such that for all  $i\in \{1,\ldots,n\}$, in the limit $h\to 0$: 
 $$
  \int_{\Sigma }   \,  \tilde \psi_i \cdot n  \,   e^{- \frac{2}{h} f} d\sigma  =\delta_{j_0,i}\, C^*h^{q^* } e^{-\frac{1}{h}(2f^*-f(z_{j_0}))} ( 1+ O(h)  ).
  $$
  \label{page.cstarqstar}
% where $C^*=\frac{B^* \sqrt{2 }\,({\rm det \ Hess } f|_{\partial \Omega}   (z_{j_0}) )^{\frac 14}  }  { \pi^{\frac{d-1}{4}}\sqrt{\partial_nf(z_{j_0})} } \ h^{p^*-\frac{d+1}{4} }$, $p^*-\frac{d+1}{4}$ 
% and, $B^*$ and $p^*$ are given by \eqref{eq:DL}.
   \end{itemize}
  
\end{enumerate}
Let $u_h$ be the eigenfunction associated with the smallest eigenvalue of $-L^{D,(0)}_{f,h}(\Omega)$ (see Proposition~\ref{UU}) which satisfies \eqref{eq.u_norma0}.  Then, one has:
 \begin{itemize}
  
\item  For all $i\in\{1,\dots,k_0\}$, in the limit $h\to 0$
  \begin{equation*} 
 \int_{\Sigma_i}      (\partial_{n}u_h) \,   e^{-\frac{2}{h}f} d\sigma=C_iB_ih^{p+m} \, e^{-\frac{1}{h}(2f(z_i)- f(x_0))}    \   (1+  O(h) ).
  \end{equation*} 
  %and the limit \eqref{eq:limproba} holds for $\P_{\nu_h} \left[ X_{\tau_{\Omega}} \in \Sigma_i\right]$. 
Moreover, if $f(z_{k_0})<f(z_{k_0+1})$, there exists $\ve>0$ such that for all $i \in\{k_0+1,\dots,n\}$ in the limit $h\to 0$
  \begin{equation*} 
 \int_{\Sigma_i}      (\partial_{n}u_h) \,   e^{-\frac{2}{h}f } d\sigma=\left ( \int_{\Sigma_{k_0}}    (\partial_{n}u_h)  \,   e^{-\frac{2}{h}f } d\sigma\right ) O\left ( e^{-\frac{\ve}{h}}\right).
  \end{equation*} 
  %and $\P_{\nu_h} \left[ X_{\tau_{\Omega}} \in \Sigma_i\right]=\P_{\nu_h} \left[ X_{\tau_{\Omega}} \in \Sigma_{k_0}\right] O \left( e^{-\frac{\ve}{h}}  \right)$.
  \item In the limit $h\to 0$:
\begin{equation*} 
 \int_{\Sigma}       (\partial_{n}u_h) \,   e^{-\frac{2}{h}f} d\sigma =C^*C_{j_0}h^{q^*+p} \, e^{-\frac{1}{h}(2f^*- f(x_0))}    \   (1+  O(h) ).  
  \end{equation*} 
  \end{itemize}
  %\eqref{eq.gene_sig} 
\end{proposition}
\noindent
The asymptotic estimates~\eqref{eq.dnuh_expch} and~\eqref{eq.gene_dnuh} in Theorem~\ref{th.gene_sigma} are  consequences of this proposition and of the construction of some    appropriate quasi-modes  
$((  \tilde \psi_i)_{i=1,\ldots,n}, \tilde u)$, see Section \ref{sec:const-qm}, which will  show that $(B_i)_{i=1,\dots,n}$, $m$, $(C_i)_{i=1,\dots,n}$, $p$ are given by \eqref{mBi}-\eqref{Cip} and $C^*$, $q^*$ will be given in   Lemma~\ref{lemmawkb} below. Moreover, the other asymptotic estimates~\eqref{eq.expch} and~\eqref{eq.gene_sig} in Theorem~\ref{th.gene_sigma} are  consequences  of the asymptotic estimates~\eqref{eq.dnuh_expch} and~\eqref{eq.gene_dnuh} together with Proposition~\ref{moyenneu},  Proposition~\ref{lambdah} and \eqref{eq:dens}. 

\subsubsection{Proof of Proposition~\ref{ESTIME3} } \label{sec.estime3}
The proof of  Proposition~\ref{ESTIME3} follows closely the same steps as the proof of  Proposition~\ref{ESTIME}. We only highlight the main differences.  
In all this section, we assume that the hypotheses \textbf{[H1]}, \textbf{[H2]} and \textbf{[H3]} hold. Let $f^*\in \mathbb R$, $k_0\in \{1,\dots,n\}$,  $j_0\in \{1,\dots,k_0\}$, $(\Sigma_i)_{i\in\{1,\dots,n\}}$ and $\Sigma$ be as stated in Proposition~\ref{ESTIME3}. In addition, let us assume the existence of $n+1$ quasi-modes $(\tilde u, ( \tilde
\psi_i)_{i=1,\ldots,n})$ satisfying all the conditions of
Proposition~\ref{ESTIME3}. In the following, $\varepsilon$ denotes a
positive constant independent of $h$, smaller than $\min (\varepsilon_1,
\varepsilon_0, \ve')$, and whose precise value may vary (a finite number of times) from one occurrence
to the other.
\begin{sloppypar} 

Let us recall a result relating
$\tilde{u}$ with $u_h$ on the one hand, and $ {\rm
   span}\big(\tilde \psi_j,j=1,\ldots,n\big) $ with $ \range   \pi_h^{(1)}$ on the
 other hand.  The following lemma is a direct consequence of Lemma~\ref{pi0u}, Lemma~\ref{indep} and the assumptions 1, 2 and 3 of Proposition~\ref{ESTIME3}. 
 \begin{lemma} \label{le.pre_thm2}
  Let us assume that the assumptions of Proposition~\ref{ESTIME3} hold. Then,  there exist $c>0$ and $h_0>0$ such that for $h \in (0,h_0)$,
$$\left\|\pi_h^{(0)}  \tilde u\right\|_{L^2_w}=1+O \left( e^{-\frac{c}{h}}\right).$$
 In addition, there exists $h_0>0$ such that for all $h \in (0,h_0)$,
$$
{\rm span}\left( \pi_h^{(1)}\tilde \psi_i, \,
 i=1, \ldots,n\right)=\range   \pi_h^{(1)}.
$$
  \end{lemma}
 A direct consequence of Lemma~\ref{le.pre_thm2} and the fact that $\tilde u$ is non negative in $\Omega$ is that it holds for $h$ small enough:
   \begin{equation}\label{eq.uh3}
 u_h  =      \frac{\pi_h^{(0)} \tilde u}{\left\|\pi_h^{(0)}  \tilde
     u\right\|_{L^2_w}}.
\end{equation}
Let us denote by $(\psi_i)_{i=1,\ldots,n}$   the orthonormal basis of $\range   \pi_h^{(1)}$ resulting from the Gram-Schmidt orthonormalization procedure on the set $ (\pi_h^{(1)}\tilde \psi_i)_{
 i=1, \ldots,n }$ (see Lemma~\ref{gram}). 
  Then, since $\nabla u_h \in {\rm Ran} \left(\pi_h^{(1)}\right)=  {\rm
   span}\left(\psi_j,j=1,\ldots,n\right)$ (see~\eqref{eq:uh_in_span})
 and  $\lp \psi_j,\psi_i\rp_{L^2_w}=\delta_{i,j}$,  one has for any $\Gamma\subset \partial \Omega$
\begin{equation}\label{eq:dnuh_decomp3}
\int_{\Gamma}   (\partial_{n}u_h)\,   e^{-\frac{2}{h} f}d\sigma=
\sum_{j=1}^n \langle \nabla u_h , \psi_j\rp_{L^2_w}  \int_{\Gamma}
\psi_j  \cdot n \,  e^{- \frac{2}{h}  f }d\sigma.
\end{equation}
Let $(\kappa_{ji})_{(i,j) \in \left\{1,\ldots,n\right\}^2, \, i<j}$ and $(Z_j)_{j\in\{1,\dots,n\}}$ be the matrix and vector obtained through the Gram-Schimdt othonormalization procedure, see Lemma~\ref{gram}. 
 \medskip

\noindent
The strategy to  prove Proposition~\ref{ESTIME3}  consists in  estimating  precisely the following terms in the limit $h\to 0$: 
$$\kappa_{ji} , \, Z_j, \, \langle \nabla u_h , \psi_j\rp_{L^2_w} \text{ and } \left(\int_{\Gamma} \psi_j \cdot n \,  e^{- \frac{2}{h}  f} d\sigma\right)_{ \Gamma\in \{\Sigma, \Sigma_1,\dots,\Sigma_n\}}$$
 for $(i,j)\in \{1,\ldots,n\}^2$, $i<j$. Then, they will be used to obtain a precise estimate  of the right-hand-side of~\eqref{eq:dnuh_decomp3} when $h\to 0$. \end{sloppypar} 
\medskip

\noindent
\underline{Step 1.} Estimates on the terms $(\kappa_{ji})_{(i,j) \in \left\{1,\ldots,n\right\}^2, \, i<j}$ and $(Z_i)_{i\in\{1,\dots,n\}}$. 

\begin{lemma} \label{scalarproduct3}  Let us assume that the assumptions of Proposition~\ref{ESTIME3} hold. Then, there exist $\ve>0$ and $h_0>0$ such that for all $(i,j) \in
\left\{1,\ldots,n\right\}^2$ with $i<j$ and all $ h\in (0,h_0)$, if $j\le k_0$:
$$
\left\lp \pi_h^{(1)}  \tilde  \psi_i,  \pi_h^{(1)}  \tilde \psi_j \right\rp_{L^2_w}=O\left(e^{-\frac{1}{h}(f(z_j) - f(z_i)  +     \ve )}\right),
$$
and if $j>k_0$:
$$
\left\lp \pi_h^{(1)}  \tilde  \psi_i,  \pi_h^{(1)}  \tilde \psi_j \right\rp_{L^2_w}=O\left(e^{-\frac{1}{h}(f^* - f(z_1)  +     \ve )}\right).
$$
\end{lemma}
\begin{proof}
The proof follows the same lines of the proof of Lemma~\ref{scalarproduct}.
If $i<j$ and $j\le k_0$, from  assumption 2(a) in Proposition~\ref{ESTIME3} and since $f^*\ge f(z_1)$, one gets
\begin{align*}
 \left\lp (1-\pi_h^{(1)})  \tilde  \psi_i,   (1-\pi_h^{(1)}) \tilde \psi_j \right\rp_{L^2_w} &\le \Vert (1-\pi_h^{(1)})  \tilde  \psi_i\Vert_{L^2_w}\Vert (1-\pi_h^{(1)})  \tilde  \psi_j\Vert_{L^2_w} \\
&\le O\left(e^{-\frac{1}{h}(f^*- f(z_i)  + f(z_j)-f(z_1) +   \ve )}\right)= O\left(e^{-\frac{1}{h}(f(z_j) - f(z_i)  +     \ve )}\right).
\end{align*}
If $i<j$ and $k_0<j$, from  assumptions 1 and 2(a) in Proposition~\ref{ESTIME3}, one gets
$$
  \left\lp (1-\pi_h^{(1)})  \tilde  \psi_i,   (1-\pi_h^{(1)}) \pi_h^{(1)}  \tilde \psi_j \right\rp_{L^2_w}\le \Vert   \tilde  \psi_i\Vert_{L^2_w} \Vert (1-\pi_h^{(1)})  \tilde  \psi_j\Vert_{L^2_w} 
\le  O\left(e^{-\frac{1}{h}(f^* - f(z_1)  +     \ve )}\right).
$$
Lemma~\ref{scalarproduct3} is then a consequence of~\eqref{fg} together with assumption 3 in Proposition~\ref{ESTIME3}.\end{proof}
%Notice that since $\pi_h^{(1)}$ is an $L^2_w$-projection, $\left\lp
%  \pi_h^{(1)}  \tilde  \psi_i,  \pi_h^{(1)}  \tilde \psi_j
%\right\rp_{L^2_w}=\left\lp \pi_h^{(1)}  \tilde  \psi_i,  \tilde \psi_j
%\right\rp_{L^2_w}$. This will be used extensively in the following.

\begin{lemma} \label{estimate13}
 Let us assume that the assumptions of Proposition~\ref{ESTIME3} hold.  
Then, there exist $\ve>0$ and $h_0>0$ such that for all $(i,j) \in
\left\{1,\ldots,n\right\}^2$ with $i<j$ and all $ h\in (0,h_0)$, if $j\le k_0$:
$$
\kappa_{ji}=O\left(e^{-\frac{1}{h}(f(z_j) - f(z_i)  +     \ve )}\right),
$$
and if $j>k_0$:
$$
\kappa_{ji}=O\left(e^{-\frac{1}{h}(f^* - f(z_1)  +     \ve )}\right).
$$
In addition, there exist $c>0$ and $h_0>0$ such that for all $j \in
\left\{1,\ldots,n\right\}$ and $ h\in (0,h_0)$,
$$
Z_j=1+O \left( e^{-\frac{c}{h}}\right).
$$
\end{lemma}
\begin{proof}
If $i<j$ and $j\le k_0$, the estimates on $\kappa_{ji}$ and $Z_j$ are proved by induction as in the proof of Lemma~\ref{estimate1}. Let us now deal with the case $i<j$ and $ k_0<j$. For $j=k_0+1$, it follows from~\eqref{eq.kappaqk} that for all $i<k_0+1$,
$$
\kappa_{(k_0+1) i} =- \sum_{k=i}^{k_0} \sum_{l=1}^k  \frac{1}{Z_{k}^2}  \lp \pi_h^{(1)} \tilde \psi_{k_0+1}  ,   \pi_h^{(1)} \tilde \psi_l  \rp_{L^2_w} \   \kappa_{kl}\kappa_{ki},    
$$
where we use the notation $\kappa_{ii}=1$ for every $i\in\{1,\dots,n\}$. Since $1\le k\le k_0$, one has  $Z_k^{-1}=1+O \left( e^{-\frac{c}{h}}\right)$. In addition, since $1\le l\le k\le k_0$ and $1\le i\le k\le k_0$, one has $ \kappa_{kl}\kappa_{ki}=O  ( 1 )$. From Lemma~\ref{scalarproduct3}, one has for $1\le l<k_0+1$, $\lp \pi_h^{(1)} \tilde \psi_{k_0+1}  ,   \pi_h^{(1)} \tilde \psi_l  \rp_{L^2_w}=O\left(e^{-\frac{1}{h}(f^* - f(z_1)  +     \ve )}\right)$. Therefore, one obtains for all $i<k_0+1$, 
$$
\kappa_{(k_0+1) i}=O\left(e^{-\frac{1}{h}(f^* - f(z_1)  +     \ve )}\right).$$
The fact that $Z_{k_0+1}=1+O \left( e^{-\frac{c}{h}}\right)$, comes from
the fact that the terms $(\kappa_{(k_0+1) i})_{i \in \{1,\ldots ,k_0\}}$ are exponentially small and the fact that $\left\Vert \pi_h^{(1)} \tilde \psi_{k_0+1}\right\Vert_{L^2_w} = 1+O \left( e^{-\frac{c}{h}}\right)$. 
In order to prove  by induction the estimates on $\kappa_{ji}$ for $i<j$ and $j>k_0$, let us now assume that for some $k\in \left\{k_0+1,\ldots,n\right\}$ and for all  $j\in \{k_0+1,\ldots,k\}$, $i\in \{1,\ldots,j-1\}$,
$$
\kappa_{ji}= O \left(e^{-\frac{1}{h}(f^*-f(z_1)+\ve)}\right) \ {\rm and} \ 
Z_j=1+O \left( e^{-\frac{c}{h}}\right).
$$
It follows from~\eqref{eq.kappaqk}, for $q\in \{1, \ldots ,k\}$,
$$
\kappa_{(k+1) q} =- \sum_{j=q}^k \sum_{l=1}^j  \frac{1}{Z_{j}^2}  \lp \pi_h^{(1)} \tilde \psi_{k+1}  ,   \pi_h^{(1)} \tilde \psi_l  \rp_{L^2_w} \   \kappa_{jl}\kappa_{jq},    
$$
where we used the notation $\kappa_{ii}=1$. Since $1\le j\le k $, one has $Z_j^{-1}=1+O \left( e^{-\frac{c}{h}}\right)$. In addition, since $1\le l\le j\le k$ and $1\le q\le j\le k$, one has $ \kappa_{jl}\kappa_{jq}=O( 1 )$. From Lemma~\ref{scalarproduct3}, one has for $1\le l<k+1$ and $k> k_0$, $\lp \pi_h^{(1)} \tilde \psi_{k+1}  ,   \pi_h^{(1)} \tilde \psi_l  \rp_{L^2_w}=O\left(e^{-\frac{1}{h}(f^* - f(z_1)  +     \ve )}\right)$. Therefore, one obtains for all $1\le q<k+1$, 
$$
\kappa_{(k+1) q}=O\left(e^{-\frac{1}{h}(f^* - f(z_1)  +     \ve )}\right).$$
The fact that $Z_{k+1}=1+O \left( e^{-\frac{c}{h}}\right)$, comes from
the fact that the $(\kappa_{(k+1) q})_{q \in \{1,\ldots ,k\}}$ are exponentially small and the fact that $\left\Vert \pi_h^{(1)} \tilde \psi_{k+1}\right\Vert_{L^2_w} = 1+O \left( e^{-\frac{c}{h}}\right)$. This concludes the proof by induction. 
\end{proof}

\noindent
\underline{Step 2.} Estimates on the interaction terms $(\langle \nabla u_h , \psi_j\rp_{L^2_w})_{j\in\{1,\dots,n\}}$.
\begin{lemma} \label{interaction3}
 Let us assume that the assumptions of Proposition~\ref{ESTIME3} hold.  Then,  for $j\in \{1,\ldots,k_0\}$, in the limit $h\to 0$:
$$
  \left \lp       \nabla u_h  ,    \psi_j\right\rp_{L^2_w}  =        C_j \ h^p  e^{-\frac{1}{h}(f(z_j)- f(x_0))}   \,    (  1  +     O(h )   ),
  $$
and for $j\in \{k_0+1,\ldots,n\}$,  there exists $\ve'>0$ such that in the limit $h\to 0$:
$$
  \left \lp       \nabla u_h  ,     \psi_j\right\rp_{L^2_w}  =        C_j \ h^p  e^{-\frac{1}{h}(f(z_j)- f(x_0))}   \,    (  1  +     O(h )   )+O\left(\,e^{-\frac{1}{h}(f^*- f(x_0)+\ve')}\right).
$$

\end{lemma}

\begin{proof} For $j\in \{1,\dots,k_0\}$, the proof of the estimate of $\left \lp       \nabla u_h  ,    \psi_j\right\rp_{L^2_w} $ is exactly the same as for Lemma~\ref{interaction}. Let $j\in\{k_0+1,\dots,n\}$. Using \eqref{eq.comut_nabla_pi}--\eqref{eq:uh}--\eqref{eq:gram}--\eqref{eq:psij}, one has
$$  \lp    \nabla u_h ,   \psi_j \rp_{L^2_w}    =a_j+b_j+c_j,$$
  where $a_j$, $b_j$ and $c_j$ are defined by 
      
 $$a_j=\frac{Z^{-1}_j}{\left\|\pi_h^{(0)} \tilde u \right\|_{L^2_w}} \left(   \left\lp       \nabla \tilde u  ,    \tilde \psi_j  \right\rp_{L^2_w} +  \left\lp   \nabla \tilde u, \left( \pi_h^{(1)}-1   \right) \tilde \psi_j  \right\rp_{L^2_w}   \right),  $$
 $$b_j= \frac{Z^{-1}_j}{\left\|\pi_h^{(0)} \tilde u \right\|_{L^2_w}}\sum_{i=1}^{k_0}\kappa_{ji} \ \left(  \left\lp       \nabla \tilde u  ,    \tilde \psi_i  \right\rp_{L^2_w} + \left\lp   \nabla \tilde u, \left( \pi_h^{(1)}-1   \right) \tilde \psi_i  \right\rp_{L^2_w}  \right)\  ,$$
 and
 $$c_j=\frac{Z^{-1}_j}{\left\|\pi_h^{(0)} \tilde u \right\|_{L^2_w}}\sum_{i=k_0+1}^{j-1}\kappa_{ji} \ \left(  \left\lp       \nabla \tilde u  ,    \tilde \psi_i  \right\rp_{L^2_w} + \left\lp   \nabla \tilde u, \left( \pi_h^{(1)}-1   \right) \tilde \psi_i  \right\rp_{L^2_w}  \right)\  ,  
$$
  with the convention $\sum_{i=k_0+1}^{k_0}=0$. From  Lemmata \ref{le.pre_thm2} and \ref{estimate13}, one has  $\frac{Z_{j}^{-1}}{
  \left\|\pi_h^{(0)} \tilde u\right\|_{L^2_w} }=1+O \left(
  e^{-\frac{c}{h}}\right)$. Using assumptions 2 and 4(a) in
Proposition~\ref{ESTIME3} and Lemma~\ref{estimate13}, there exists
$\delta_0>0$ such that for all $\delta\in (0,\delta_0)$,
\begin{align*} 
 a_j&=C_j h^p  e^{-\frac{1}{h}(f(z_j)- f(x_0))}     (     1 + O(h ) )\, + O \left (  e^{-\frac{1}{h}(f^*-f(x_0) + \ve)}  \right ) + O\left(  e^{-\frac{1}{h}(f^*- f(x_0) -\delta + \ve)}  \right),\\
b_j&= \sum_{i=1}^{k_0}O\left (  e^{-\frac{1}{h}(f^*-f(z_1) + \ve  +f(z_i)- f(x_0)-\delta)}    \right )+O\left(  e^{-\frac{1}{h}(f^*-f(z_1) + \ve+f(z_1)- f(x_0) -\delta+ f(z_i)-f(z_1) + \ve)}  \right)\\
   &=O\left(  e^{-\frac{1}{h}( f^*-f(x_0) -\delta+ \ve)}\right), \\
c_j&=\sum_{i=k_0+1}^{j-1}O\left (  e^{-\frac{1}{h}(f^*-f(z_1) + \ve)}\right) \,\left (  O\left(  e^{-\frac{1}{h}(f(z_i)- f(x_0)-\delta)}\right ) + O\left (  e^{-\frac{1}{h}(f^*-f(x_0) + \ve)} \right )\right )\\
   &\quad +O\left (  e^{-\frac{1}{h}(f^*-f(z_1) + \ve+f(z_1)- f(x_0) -\delta+ f^*-f(z_1) + \ve)}  \right )=O \left (  e^{-\frac{1}{h}( f^*-f(x_0) -\delta+ \ve)}\right ).\end{align*}
Therefore, choosing $\delta\in (0,\varepsilon)$, there exists $\ve'>0$ such that
\begin{align*} 
 \lp    \nabla u_h ,   \psi_j \rp_{L^2_w}  
&=C_j h^p  e^{-\frac{1}{h}(f(z_j)- f(x_0))}     (     1  +     O(h ) )
+ O \left(  e^{-\frac{1}{h}(f^*- f(x_0) + \ve')} \right).
 \end{align*} 
This concludes the proof of Lemma~\ref{interaction3}.
\end{proof}

\noindent
\underline{Step 3.} Estimates on the boundary terms $\left (\displaystyle \int_{\Gamma}
\psi_j \cdot n \,  e^{- \frac{2}{h}  f }d\sigma \right )_{j\in\{1,\dots,n\}, \Gamma\in \{\Sigma, \Sigma_1,\dots,\Sigma_n\}}$.

\medskip
\noindent
%One denotes for $i\in \left\{1,\ldots,k_0\right\}$, $K_i:=\max(f^*-f(z_i), f(z_i)-f(z_1))\ge0$. 

   \begin{lemma} \label{gamma32}
 Let us assume that the assumptions of Proposition~\ref{ESTIME3} hold. Then, there exists $\ve>0$ such that in the limit $h\to 0$ one has:
\begin{itemize}
\item  If  $k\in \{1,\dots,k_0\}$ and  $j\in \{1,\dots,k_0\}$,
\begin{align*}  
\int_{\Sigma_k }  \psi_j \cdot n  \, e^{- \frac{2}{h} f} d\sigma&=   
\delta_{j,k}B_k h^m      e^{-\frac{1}{h} f(z_k)} (  1  +     O(h )    ) + O\left(e^{-\frac{1}{h} (2f(z_{k})-f(z_j)+\ve)}\right)\\
&\quad+1_{j>k}O\left(e^{-\frac{1}{h} (f(z_k)+\ve)}\right).
\end{align*}  
\item If  $k\in \{1,\dots,k_0\}$ and  $j\in \{k_0+1,\dots,n\}$,
$\displaystyle\int_{\Sigma_k }  \psi_j \cdot n  \, e^{- \frac{2}{h} f} d\sigma= O\left ( e^{- \frac{1}{h}( f(z_k)+\ve)}  \right)$.
\item If  $k\in \{k_0+1,\dots,n\}$ and  $j\in \{1,\dots,k_0\}$, $\displaystyle \int_{\Sigma_k }  \psi_j \cdot n  \, e^{- \frac{2}{h} f} d\sigma=O\left(e^{-\frac{1}{h} (2f(z_{k_0})-f(z_j)+\ve)}\right).$ 
\item If  $k\in \{k_0+1,\dots,n\}$ and $j\in \{k_0+1,\dots,n\}$, $$\displaystyle \int_{\Sigma_k }  \psi_j \cdot n  \ e^{- \frac{2}{h} f}d\sigma=  \delta_{j,k}O\left( h^m      e^{-\frac{1}{h} f(z_k)}\right)   +O\left ( e^{- \frac{1}{h}( f(z_{k_0+1})+\ve)} \right).$$

\end{itemize}

\end{lemma}
\begin{proof}
For all $(j,k)\in \{1,\dots,n\}^2$, using \eqref{eq:gram}--\eqref{eq:psij} and Lemmata \ref{estimate13} and \ref{le.pre_thm2} together with assumption 4(b) in Proposition~\ref{ESTIME3}, one has in the limit $h\to 0$:
\begin{align}  
\nonumber
\int_{\Sigma_k }  \psi_j \cdot n  \ e^{- \frac{2}{h} f}d\sigma&=  Z^{-1}_j \left[ \int_{\Sigma_k }  \tilde \psi_j \cdot n  \ e^{- \frac{2}{h} f} d\sigma + \int_{\Sigma_k } ( \pi_h^{(1)}-1    ) \tilde \psi_j \cdot n \ e^{- \frac{2}{h} f}  d\sigma \right]  \\
\nonumber
 &\quad +Z^{-1}_j\sum_{i=1}^{j-1}\kappa_{ji} \ \left(  \int_{\Sigma_k }  \tilde \psi_i \cdot n  \ e^{- \frac{2}{h} f} d\sigma  + \int_{\Sigma_k } ( \pi_h^{(1)}-1    ) \tilde \psi_i \cdot n \ e^{- \frac{2}{h} f}  d\sigma \right) \\
 \nonumber
  &= \delta_{j,k}B_k h^m      e^{-\frac{1}{h} f(z_k)} (  1  +     O(h )    )   +\frac 1h\left \|    \left(1-\pi_h^{(1)}\right) \tilde \psi_j\right \|_{H^1_w}O\left ( e^{- \frac{1}{h} f(z_k)} \right)  \\
  \label{eq.dec-sigmak}
 &\quad +\sum_{i=1}^{j-1}\kappa_{ji}  \delta_{i,k}O\left( h^m      e^{-\frac{1}{h} f(z_k)}\right)   + \frac{\kappa_{ji}}{h}  \left \|    \left(1-\pi_h^{(1)}\right) \tilde \psi_i\right \|_{H^1_w} O\left ( e^{- \frac{1}{h} f(z_k)}  \right).
   \end{align} 
   Let us know deal separately with the two cases $k\in \{1,\dots,k_0\}$ and $k\in \{k_0+1,\dots,n\}$. In the following, we use assumption 2(a) in Proposition~\ref{ESTIME3} and Lemma~\ref{estimate13} to  estimate  \eqref{eq.dec-sigmak}.
   \medskip
   
   \noindent
\underline{Case 1}: $k\in \{1,\dots,k_0\}$. 
 \medskip
   
   \noindent
   If $j\in  \{1,\dots,k_0\}$, from~\eqref{eq.dec-sigmak}, one gets in the limit $h\to 0$:
\begin{align*}  
\int_{\Sigma_k }  \psi_j \cdot n  \ e^{- \frac{2}{h} f}d\sigma&=  \delta_{j,k}B_k h^m      e^{-\frac{1}{h} f(z_k)} (  1  +     O(h )    )   +O\left ( e^{- \frac{1}{h}(f^*-f(z_j)+ f(z_k)+\ve)} \right)  \\
 &\quad +\sum_{i=1}^{j-1}  \delta_{i,k}O\left(      e^{-\frac{1}{h} (f(z_k)+\ve)}\right)  +  \sum_{i=1}^{j-1}  O\left ( e^{- \frac{1}{h}(f(z_j)-f(z_i)+f^*-f(z_i)+ f(z_k)+\ve)}  \right).
   \end{align*} 
   Since $f^*\ge f(z_k)$ for all $k\in \{1,\dots,k_0\}$ and since there exists $i<j$ such that $\delta_{i,k}=1$ if and only if $j>k$, there exists $\ve>0$ such that for all $(k,j)\in \{1,\dots,k_0\}^2$ and for all $h$ small enough,
 $$\int_{\Sigma_k }  \psi_j \cdot n\, e^{- \frac{2}{h} f}d\sigma=\delta_{j,k}B_k h^m      e^{-\frac{1}{h} f(z_k)} (  1  +     O(h )    ) + O\left(e^{-\frac{1}{h} (2f(z_{k})-f(z_j)+\ve)}\right)+1_{j>k}O\left(e^{-\frac{1}{h} (f(z_k)+\ve)}\right).$$
 If $j\in \{k_0+1,\dots,n\}$, from  \eqref{eq.dec-sigmak}, one gets 
\begin{align*}  
\int_{\Sigma_k }  \psi_j \cdot n  \ e^{- \frac{2}{h} f}d\sigma&=  O\left ( e^{- \frac{1}{h}(f^*-f(z_1)+ f(z_k)+\ve)} \right)  +\sum_{i=1}^{j-1} O\left(     e^{-\frac{1}{h} ( f(z_k)+\ve)}\right).
   \end{align*}

   \noindent
\underline{Case 2}:  $k\in \{k_0+1,\dots,n\}$. 
  \medskip
   
   \noindent
   If $j\in \{1,\dots,k_0\}$, from~\eqref{eq.dec-sigmak} and  since  $f(z_k)\ge f^*\ge f(z_{k_0})$, one has
\begin{align*}  
\int_{\Sigma_k }  \psi_j \cdot n  \ e^{- \frac{2}{h} f} d\sigma&=  O\left ( e^{- \frac{1}{h} (f^*-f(z_j)+\ve+f(z_k))} \right) +\sum_{i=1}^{j-1} O\left ( e^{- \frac{1}{h} (f^*+f(z_k)+f(z_j)-2f(z_i)+\ve)} \right)\\
&=O\left ( e^{- \frac{1}{h} (2f(z_{k_0})-f(z_j)+\ve)} \right).
   \end{align*} 
 If $j\in \{k_0+1,\dots,n\}$, from  \eqref{eq.dec-sigmak}, one gets     \begin{align*}  
\int_{\Sigma_k }  \psi_j \cdot n  \ e^{- \frac{2}{h} f}d\sigma&=  \delta_{j,k}O\left( h^m      e^{-\frac{1}{h} f(z_k)}\right)   +O\left ( e^{- \frac{1}{h}( f(z_k)+\ve)} \right),
   \end{align*}  
which leads to the desired estimate since $f(z_k)\ge f(z_{k_0+1})$. This concludes the proof of Lemma~\ref{gamma32}.\end{proof}

\begin{lemma} \label{gamma31}
 Let us assume that the assumptions of Proposition~\ref{ESTIME3} hold.  Then, for $j\in \{1,\dots,k_0\}$ one has when $h\to 0$:
$$\int_{\Sigma }  \psi_j \cdot n  \, e^{- \frac{2}{h} f}d\sigma=\delta_{j_0,j}\, C^* h^{q^* }e^{-\frac{1}{h}(2f^*-f(z_{j_0}))}\left( 1+ O(h) \right) +O\left ( e^{- \frac{1}{h} (2f^*-f(z_j)+\ve)} \right)$$
and for $j\in\{k_0+1,\dots, n\}$ one has $\displaystyle \int_{\Sigma }  \psi_j \cdot n  \, e^{- \frac{2}{h} f}d\sigma = 
O\left ( e^{- \frac{1}{h} (2f^*-f(z_1)+\ve)} \right)$.
\end{lemma}
\begin{proof}
Let $j\in\{1,\dots,n\}$. Using \eqref{eq:gram}--\eqref{eq:psij} and Lemmata \ref{estimate13} and \ref{le.pre_thm2} together with assumption 4(c) in Proposition~\ref{ESTIME3}, one has
\begin{align}  
\nonumber\int_{\Sigma }  \psi_j \cdot n  \ e^{- \frac{2}{h} f}  d\sigma  &=Z^{-1}_j\left[ \int_{\Sigma }  \tilde \psi_j \cdot n  \ e^{- \frac{2}{h} f}  d\sigma + \int_{\Sigma } ( \pi_h^{(1)}-1    ) \tilde \psi_j \cdot n \ e^{- \frac{2}{h} f}  d\sigma \right]  \\
\nonumber
 &\quad +Z^{-1}_j\sum_{i=1}^{j-1}\kappa_{ji} \ \left(  \int_{\Sigma }  \tilde \psi_i \cdot n  \ e^{- \frac{2}{h} f} d\sigma + \int_{\Sigma } ( \pi_h^{(1)}-1    ) \tilde \psi_i \cdot n \ e^{- \frac{2}{h} f}  d\sigma \right) \\
 \nonumber
 &=     \delta_{j_0,j}\, C^*h^{q^* } e^{-\frac{1}{h}(2f^*-f(z_{j_0}))} ( 1+ O(h)  )  +\frac 1h\left   \|    \left(1-\pi_h^{(1)}\right) \tilde \psi_j\right \|_{H^1_w}O\left(e^{- \frac{1}{h} f^*} \right)  \\
  \label{eq.decompo-psij}
 &\quad +\sum_{i=1}^{j-1}\delta_{j_0,i} \kappa_{ji} O\left(h^{q^*} e^{-\frac{1}{h}(2f^*-f(z_{j_0}))} \right)+ \frac{\kappa_{ji} }{h}  \left  \|    \left(1-\pi_h^{(1)}\right) \tilde \psi_i\right \|_{H^1_w}O\left( e^{- \frac{1}{h} f^*}  \right).
  \end{align} 
  Let us first deal with the case $j\in \{1,\dots,k_0\}$. Using assumption 2(a)  in Proposition~\ref{ESTIME3} and Lemma~\ref{estimate13}, one gets from~\eqref{eq.decompo-psij}
\begin{align*}  
\int_{\Sigma }  \psi_j \cdot n  \ e^{- \frac{2}{h} f}  d\sigma   &=  \delta_{j_0,j}\, C^* h^{q^* }e^{-\frac{1}{h}(2f^*-f(z_{j_0}))}\left( 1+ O(h) \right) +O\left ( e^{- \frac{1}{h} (f^*-f(z_j)+\ve+f^*)} \right)    \\
 &\quad + \sum_{i=1}^{j-1}O\left (   e^{- \frac{1}{h} ( 2f^*+f(z_j)-2f(z_{i})+\ve)} \right)\\
   &=  \delta_{j_0,j}\, C^* h^{q^* }e^{-\frac{1}{h}(2f^*-f(z_{j_0}))}\left( 1+ O(h) \right) +O\left ( e^{- \frac{1}{h} (2f^*-f(z_j)+\ve)} \right).
   \end{align*} 
Let us now deal with the case $j\in\{k_0+1,\dots,n\}$. In that case, one obtains from~\eqref{eq.decompo-psij}, assumption 2(a)  in Proposition~\ref{ESTIME3}, Lemma~\ref{estimate13} together with the fact that $f^*\ge f(z_{i})$ for all $i\in\{1,\dots,k_0\}$ and $f^*\ge f(z_{k_0})\ge f(z_{j_0})$,
  \begin{align*}  
\int_{\Sigma }  \psi_j \cdot n  \ e^{- \frac{2}{h} f}  d\sigma   &=   O\left ( e^{- \frac{1}{h} (f^*-f(z_1)+\ve+f^*)} \right)  + \sum_{i=1}^{j-1} \delta_{j_0,i} \, O\left ( e^{-\frac{1}{h}(f^*-f(z_1)+ 2f^*-f(z_{j_0} ) +\ve)} \right)\\
  &\quad + \sum_{i=1}^{k_0}O\left ( e^{- \frac{1}{h} (2f^*-f(z_1)+f^*-f(z_i)+\ve)} \right) + \sum_{i=k_0+1}^{j-1}O\left ( e^{- \frac{1}{h} (2f^*-2f(z_1)+\ve+f^*)} \right) \\
  &=O\left ( e^{- \frac{1}{h} (2f^*-f(z_1)+\ve)} \right).
     \end{align*} 
 This concludes the proof of Lemma~\ref{gamma31}.\end{proof}
\noindent
\underline{Step 4.} Estimates on the boundary terms $\left (\displaystyle \int_{\Gamma}
(\partial_nu_h) \,  e^{- \frac{2}{h}  f}d\sigma \right )_{ \Gamma\in \{\Sigma, \Sigma_1,\dots,\Sigma_n\}}$.\\

\noindent
We are now in position to conclude the proof of
Proposition~\ref{ESTIME3}.  
\begin{proof}  Let us assume that the assumptions of Proposition~\ref{ESTIME3} hold.   The proof is divided into two cases.

\medskip
\noindent
\underline{Case 1}: $\Gamma=\Sigma_k$ in~\eqref{eq:dnuh_decomp3} for some $k\in \{1,\dots,n\}$.
\medskip

\noindent
 If $k\in \{1,\dots,k_0\}$, from Lemmata~\ref{interaction3} and~\ref{gamma32} and the fact that $f^*\ge f(z_{k_0})\ge f(z_k)$, one obtains that there exists $\ve >0$ such that for all $j\in \{1,\dots,n\}$, in the limit $h\to 0$
\begin{align*}
\lp    \nabla u_h ,   \psi_j \rp_{L^2_w}   \int_{\Sigma_k }  \psi_j \cdot n  \, e^{- \frac{2}{h} f}d\sigma&=\delta_{jk}B_kC_k h^{m+p}      e^{-\frac{1}{h} (2f(z_k)-f(x_0))} (  1  +     O(h )    )\\
&\quad +O \left(  e^{-\frac{1}{h}(2f(z_k)- f(x_0) + \ve)} \right).
\end{align*}
   Therefore, from~\eqref{eq:dnuh_decomp3}, one gets for all $k\in \{1,\dots,k_0\}$,  in the limit $h\to 0$
$$\int_{\Sigma_k } (\partial_nu_h) \ e^{- \frac{2}{h} f }  d\sigma= B_kC_k h^{m+p}      e^{-\frac{1}{h} (2f(z_k)-f(x_0))} (  1  +     O(h )    ).$$
If $k\in \{k_0+1,\dots,n\}$. From Lemmata \ref{interaction3} and  \ref{gamma32}, one has for $j\in \{1,\dots,k_0\}$:
$$\lp    \nabla u_h ,   \psi_j \rp_{L^2_w}   \int_{\Sigma_k }  \psi_j \cdot n  \, e^{- \frac{2}{h} f}d\sigma=O \left( h^p e^{-\frac{1}{h}(2f(z_{k_0})- f(x_0)+ \ve)} \right).$$
and for $j\in \{k_0+1,\dots,n\}$ (since $f(z_j)\ge f(z_{k_0+1})$ and $f(z_{k_0})\le f^*\le(z_{k_0+1})$):
\begin{align*} 
\lp    \nabla u_h ,   \psi_j \rp_{L^2_w}   \int_{\Sigma_k }  \psi_j \cdot n  \, e^{- \frac{2}{h} f}d\sigma&=\delta_{jk} O\left ( h^{p+m}  e^{-\frac{1}{h}(2f(z_{k})- f(x_0))}     \right) + O \left(  e^{-\frac{1}{h}(2f(z_{k_0})- f(x_0) + \ve)} \right).
     \end{align*} 
Therefore, if one assumes that $f(z_{k_0+1})>f(z_{k_0})$, from~\eqref{eq:dnuh_decomp3}, one gets for all $k\in \{k_0+1,\dots,n\}$ and for all $h$ small enough
$$\int_{\Sigma_k } (\partial_nu_h )  \ e^{- \frac{2}{h} f }  d\sigma= O \left(  e^{-\frac{1}{h}(2f(z_{k_0})- f(x_0) + \ve)} \right).$$
%Finally, the asymptotics of $\P_{\nu_h} \left[ X_{\tau_{\Omega}} \in \Sigma_k\right]$ for $k\in \{1,\dots,n\}$ are a consequence  of the above asymptotics of $\int_{\Sigma_k } \partial_nu_h  \ e^{- \frac{2}{h} f}  \sigma(dz)$  together with \eqref{eq:dens}, Propositions~\ref{moyenneu} and~\ref{lambdah}.

\medskip
\noindent
\underline{Case 2}: $\Gamma=\Sigma$ in~\eqref{eq:dnuh_decomp3}. 
\medskip

\noindent
From~\eqref{eq:dnuh_decomp3} and using Lemmata~\ref{gamma31} and~\ref{interaction3}, one has 
\begin{align*}  
\int_{\Sigma } \partial_nu_h  \ e^{- \frac{2}{h} f}  d\sigma   &=  C_{j_0}C^* h^{q^*+p} e^{-\frac{1}{h}(2f^*-f(x_{0}))}(1+O(h)) +  \sum_{j=1, j\neq j_0}^{k_0}O\left ( h^p  e^{-\frac{1}{h}(- f(x_0)+2f^*+\ve)} \right)  \\
  &\quad  + \sum_{j=k_0+1}^{n}  O\left ( h^p  e^{-\frac{1}{h}(f(z_j)- f(x_0)+2f^*-f(z_1)+\ve)}\right)+   O\left (e^{-\frac{1}{h}(f^*- f(x_0) +2f^*-f(z_1)+\ve)} \right)\\
&= C_{j_0}C^* h^{q^* } e^{-\frac{1}{h}(2f^*-f(x_{0}))}(1+O(h)),
 \end{align*}
 which is the desired result. 
 %The asymptotic \eqref{eq.gene_sig} is a consequence  of the above asymptotic together with \eqref{eq:dens} and Propositions~\ref{moyenneu} and~\ref{lambdah}. 
 Proposition~\ref{ESTIME3} is proved. \end{proof}

\subsubsection{Construction of the quasi-modes which satisfy the estimates of Proposition~\ref{ESTIME3}}
\label{sec:const-qm}
In this section, we first construct the
quasi-modes $(\tilde \psi_i)_{i\in \{1,\ldots,n\}}$ and the family of quasi-modes $(\tilde u=\tilde u_{\delta})_{\delta>0}$. Then, we prove that they satisfy the estimates stated in Proposition~\ref{ESTIME3}. In all this section,  one assumes that the hypotheses \textbf{[H1]}, \textbf{[H2]} and \textbf{[H3]} hold.  Let  $(\Sigma_i)_{i\in\{1,\dots,n\}}$ and $\Sigma$ be as in Proposition~\ref{ESTIME3}.

\medskip
\noindent
\underline{Construction of the quasi-modes.}
\medskip

\noindent
The $n+1$ quasi-modes $( ( \tilde \psi_i)_{i\in \{1,\ldots,n\} }, \tilde u)$ are constructed as in Section~\ref{sec:contruct_quasimode} except that one adds  an extra condition on the set $\Gamma_{1,{j_0}}$ used to define $\tilde \psi_{j_0}$ (where we recall that $j_0\in \{1,\dots,n\}$ is such that $\overline{\Sigma}\subset  B_{z_{j_0}}$). Let us be more precise on this   point. 
 Let us recall that  for all $i\in \{1,\ldots,n\}\setminus \{j_0\}$, the $1$-form~$\tilde \psi_{i}$ is defined as:
$$
\tilde \psi_{i}=e^{\frac 1hf}\tilde \phi_{{i}} \in \Lambda^1H^1_{w,T}(\Omega),
$$
where   the $n-1$ quasi-modes $( \tilde \phi_i)_{i\in \{1,\ldots,n\}\setminus \{j_0\}}$ are built in Section~\ref{sec:contruct_quasimode} (see Definition~\ref{tildephii}). Let us also recall that  the function $\tilde u$ is the one introduced in~Definition~\ref{tildeu}. 
 The construction of  $\tilde \psi_{j_0}$  
  requires to take into account the set $\Sigma$ in addition to the set $\Sigma_{j_0}$ when defining the cut off function $\chi_{j_0}$ in Definition~\ref{tildephii}\label{page.phijo}. More precisely, and thanks to Proposition
\ref{stronglystableexis},  the set $\Gamma_{1,{j_0}}$ can be taken such
that $$\overline\Sigma_{j_0} \cup \overline\Sigma \subset\Gamma_{1,{j_0}}.$$
 Then,  with  $\Gamma_{1,{j_0}}$ satisfying  the previous condition, the cut-off fonction $\chi_{j_0}$ and the $1$-form $\tilde \phi_{j_0}\in  \Lambda^1H^1_{T}(\Omega)$ are defined exactly as in Definition~\ref{tildephii} for $i=j_0$. Finally, one sets: 
\begin{equation} \label{eq.psij0}
\tilde \psi_{j_0}=e^{\frac 1hf}\tilde \phi_{{j_0}} \in \Lambda^1H^1_{w,T}(\Omega).
\end{equation}
\medskip

\noindent
\underline{The quasi-modes satisfy the estimates stated in Proposition~\ref{ESTIME3}.}
\medskip

\noindent
Using in addition to \textbf{[H1]}-\textbf{[H2]}-\textbf{[H3]} the hypotheses \eqref{eq.hyp.gene.1} and \eqref{eq.hyp.gene.2}, one easily obtains that  $( ( \tilde \psi_i)_{i=1,\ldots,n}, \tilde u)$ satisfy the estimates 1, 2, 3, 4(a) and 4(b) stated in Proposition~\ref{ESTIME3}, following exactly the computations made on $( ( \tilde \phi_i)_{i=1,\ldots,n}, \tilde u)$ in Section~\ref{sec:goodquasimodes}: 2(a) follows from~\eqref{eq:majo_inf_supp}-\eqref{eq:da1}-\eqref{eq.hyp.gene.1}-\eqref{eq.hyp.gene.2}, 2(b) is proven in Lemma~\ref{nablauu}, 3 follows from~\eqref{eq:majo_phii_phij}-\eqref{eq.hyp.gene.1}, 4(b) is proven in Step 3 in Section~\ref{sec:goodquasimodes}  and 4(a) is a consequence of \eqref{dodoo}-\eqref{dodoo1}-\eqref{dodoo2}-\eqref{eq.hyp.gene.1}. In particular, one gets that the constants $(B_i)_{i=1,\dots,n}$, $m$, $(C_i)_{i=1,\dots,n}$ and $p$ in Proposition~\ref{ESTIME3} are given by \eqref{mBi}-\eqref{Cip}.

The following lemma  deals with the assumption 4(c) in  Proposition~\ref{ESTIME3} which requires to use Proposition~\ref{pr.WKB-comparison}. 

\begin{lemma}  \label{lemmawkb} Let us assume that the hypotheses \textbf{[H1]}, \textbf{[H2]} and \textbf{[H3]} hold. Let $j\in \{1,\dots,n\}$. Then, when $h\to 0$, one has
 $$\int_{\Sigma} \tilde \psi_{j}\cdot n\,  e^{-\frac{2}{h}f} d\sigma=\delta_{j_0,j}\frac{B^* \sqrt{2 }\,({\rm det \ Hess } f|_{\partial \Omega}   (z_{j_0}) )^{\frac 14}  }  { \pi^{\frac{d-1}{4}}\sqrt{\partial_nf(z_{j_0})} } \, h^{p^*-\frac{d+1}{4} }  e^{-\frac{1}{h}(2f^*-f(z_{j_0}))}\left( 1+ O(h) \right),$$\label{page.pstarr}
 where $B^*$ and $p^*$ are defined by \eqref{eq:DL}. 
\end{lemma}

\begin{proof} By construction, if $j\neq j_0$ then $\tilde \psi_{j}\equiv 0$ on $B_{z_{j_0}}$. Let us deal with the case $j=j_0$. Using \eqref{eq.psij0}, one has
\begin{equation}\label{eq.relaj0}
\int_{\Sigma} \tilde \psi_{j_0}\cdot n\,  e^{-\frac{2}{h}f}d\sigma=\int_{\Sigma}   \tilde \phi_{j_0} \cdot n  \   e^{- \frac{1}{h}f}d\sigma.
\end{equation}
Let $u^{(1)}_{z_{j_0},wkb}$ be the WKB expansion
defined by \eqref{wkbi}. Following the beginning of
Section~\ref{sec:WKB_first_estim}, let us consider
\begin{enumerate}
\item a neighborhood $V_{\Gamma_{St,j_0}}$ of $\Sigma$ in $\overline
  \Omega$, which is stable under the dynamics~\eqref{eq:flow_Phi} and
  such that, for some $\ve > 0$, $V_{\Gamma_{St,j_0}} + B(0,\ve) \subset V_{\Gamma_{1,j_0}} \cap (\dot \Omega_{j_0} \cup \Gamma_{1,j_0})$,
\item and a cut-off function $\chi_{wkb,j_0}\in
  C^{\infty}_c(\dot \Omega_{j_0} \cup \Gamma_{1,j_0})$ with $\chi_{wkb,j_0} \equiv 1$ on a
  neighborhood of $\overline{V_{\Gamma_{St,j_0}}}$ such that $\supp
  \chi_{wkb,j_0} \subset V_{\Gamma_{1,j_0}} \cap (\dot \Omega_{j_0} \cup \Gamma_{1,j_0})$.
\end{enumerate}
Using Proposition~\ref{pr.WKB-comparison}, there exists $c_{z_{j_0}}(h)\in
\mathbb R_+^*$ such that
$$
\left\| e^{{\frac 1h}d_a(\cdot,\,z_{j_0})} \left(u^{(1)}_{h,j_0}-c_{z_{j_0}}(h)u^{(1)}_{z_{j_0},wkb}\right)\right\|_{ H^{1}(V_{\Gamma_{St,j_0}})}\!\!\!=O (h^{\infty}).
$$
Let us now introduce
$$
\tilde \phi_{z_{j_0},wkb}:=c_{z_{j_0}}(h) \chi_{wkb,j_0} \ u_{z_{j_0},wkb}^{(1)}
$$
so that
\begin{equation}\label{numero} \int_{\Sigma}   \tilde \phi_{j_0} \cdot n  \   e^{- \frac{1}{h}f} d\sigma = \int_{\Sigma}    \tilde \phi_{z_{j_0},wkb} \cdot n  \ e^{- \frac{1}{h}f} d\sigma +\int_{\Sigma}  \left  (\tilde \phi_{j_0}-\tilde \phi_{z_{j_0},wkb}\right) \cdot n \,  e^{- \frac{1}{h}f} d\sigma. 
\end{equation}
Let us first deal with the term $\int_{\Sigma}    \tilde
\phi_{z_{j_0},wkb} \cdot n  \,  e^{- \frac{1}{h}f} $ in~\eqref{numero}. Using \eqref{eq:DL}, one has (since $\Phi=f$, $\partial_n\Phi=-\partial_nf$ and $a_0=1$ on $\partial \Omega$, see~\eqref{eq.carre}) when $h\to 0$,
\begin{align*}
\int_{\Sigma}    \tilde \phi_{z_{j_0},wkb} \cdot n \,  e^{- \frac{1}{h}f} d\sigma&= c_{z_{j_0}}(h) \int_{\Sigma} \chi_{wkb,{j_0}} \ u_{z_{j_0},wkb}^{(1)} \cdot n \ e^{- \frac{1}{h}f}\\
 &=2\, c_{z_{j_0}}(h)\int_{\Sigma} \partial_nf \, e^{- \frac{1}{h}(2f-f(z_{j_0}))} \, \left( 1+ O(h) \right)\\
 &=2\,c_{z_{j_0}}(h)\,B^*\, h^{p^* } \, e^{-\frac{1}{h}(2f^*-f(z_{j_0}))} \left( 1+ O(h) \right).
\end{align*}
Then using \eqref{eq:czh}, one obtains in the limit $h\to 0$:
\begin{equation}\label{numero2}
\int_{\Sigma}  \tilde \phi_{z_{j_0},wkb} \cdot n \,  e^{- \frac{1}{h}f} d\sigma=
\frac{B^* \sqrt{2 }\, ({\rm det \ Hess } f|_{\partial \Omega}   (z_{j_0}) )^{\frac 14}  }  { \pi^{\frac{d-1}{4}}\sqrt{\partial_nf(z_{j_0})} }  h^{p^*-\frac{d+1}{4} }  e^{-\frac{1}{h}(2f^*-f(z_{j_0}))}\left( 1+ O(h) \right).
\end{equation}
Let us now estimate the term $\int_{\Sigma} \left(\tilde \phi_{j_0}-\tilde \phi_{z_{j_0},wkb} \right) \cdot n \ e^{- \frac{1}{h}f}$ in~\eqref{numero}. Since $d_a(\cdot,z_{j_0})=f-f(z_{j_0})=\Phi-f(z_{j_0})$ on $\Sigma$, one obtains using  Lemma~\ref{N1}: there exist $C>0$, $h_0>0$ and $\eta>0$ such that for all $ h\in (0,h_0)$,
 \begin{align*}
\left\vert \int_{\Sigma} \left(\tilde \phi_{j_0}-\tilde \phi_{z_{j_0},wkb} \right) \cdot n \ e^{- \frac{1}{h}f} d\sigma\right\vert&=\left\vert \int_{\Sigma}  \left(\frac{u^{(1)}_{h,{j_0}}}{\Theta_{j_0}}-c_{z_{j_0}}(h)u_{z_{j_0},wkb}^{(1)}\right) \cdot n \ e^{- \frac{1}{h}f} d\sigma  \right\vert\\
&\leq \frac{e^{- \frac{1}{h}(2f^*-f(z_{j_0}))}}{  \Theta_{j_0}}   \int_{\Sigma} \left\vert  \left(u^{(1)}_{h,{j_0}}-c_{z_{j_0}}(h) u_{z_{j_0}, wkb}^{(1)}\right)   \, e^{ \frac{d_a(\cdot,z_{j_0})}{h}}d\sigma\right\vert    \\
&\quad +e^{- \frac{1}{h}(2f^*-f(z_{j_0}))}\, \frac{\vert \Theta_{j_0}-1\vert }{\Theta_{j_0}} \vert c_{z_{j_0}}(h)\vert\int_{\Sigma} \left\vert  \,u_{z_{j_0},wkb}^{(1)}  e^{ \frac{\Phi-f(z_{j_0})}{h}} d\sigma\right\vert  \\
 &\leq C e^{- \frac{1}{h}(2f^*-f(z_{j_0}))} \left\| e^{ \frac{d_a(\cdot,z_{j_0})}{h}} \left(u^{(1)}_{h,j_0}-c_{z_{j_0}}(h)u^{(1)}_{z_{j_0},wkb}\right)\right\|_{ H^{1}(V_{\Gamma_{St,j_0}})}  \\
&\quad +C e^{- \frac{1}{h}(2f^*-f(z_{j_0}))} \,\vert c_{z_{j_0}}(h)\vert  e^{-\frac{\eta}{h}}\left\Vert \chi_{wkb,j_0}\, u_{z_{j_0},wkb}^{(1)} e^{ \frac{\Phi-f(z_{j_0})}{h}}\right\Vert_{H^1(\dot \Omega_{j_0})} .
 \end{align*}
 Since it holds  $  u_{z_{j_0},wkb}^{(1)} e^{ \frac{\Phi-f(z_{j_0})}{h}}= d_{f-(\Phi-f(z_{j_0})),h}a(\cdot,h)=hda(\cdot,h)+\nabla(f-\Phi)\wedge a(\cdot,h)$ (see~\eqref{wkbi}), there exists $C>0$ such that for all $h$ small enough,
 $$ \left\Vert \chi_{wkb,j_0}\, u_{z_{j_0},wkb}^{(1)} e^{ \frac{\Phi-f(z_{j_0})}{h}}\right\Vert_{H^1(\dot \Omega_{j_0})} \le C.$$
Then, one obtains using Proposition~\ref{pr.WKB-comparison} and \eqref{eq:czh}:
\begin{equation}\label{numero3}
 \ e^{\frac{1}{h}(2f^*-f(z_{j_0}))} \left\vert \int_{\Sigma} \left(\tilde
    \phi_{j_0}-\tilde \phi_{z_{j_0},wkb} \right) \cdot n \ e^{- \frac{1}{h}f}
\right\vert  = O(h^{\infty}) + C e^{-\frac{\eta}{h}} h^{-\frac{d+1}{4}} = O(h^\infty).
\end{equation}
Injecting the estimates \eqref{numero2}--\eqref{numero3} in  \eqref{numero} and using \eqref{eq.relaj0} imply that in the limit $h\to 0$:
 $$\int_{\Sigma} \tilde \psi_{j_0}\cdot n\,  e^{-\frac{2}{h}f}d\sigma =\frac{B^* \sqrt{2 }\,({\rm det \ Hess } f|_{\partial \Omega}   (z_{j_0}) )^{\frac 14}  }  { \pi^{\frac{d-1}{4}}\sqrt{\partial_nf(z_{j_0})} } \ h^{p^*-\frac{d+1}{4} } \ e^{-\frac{1}{h}(2f^*-f(z_{j_0}))}\left( 1+ O(h) \right).$$
This proves Lemma~\ref{lemmawkb}.\end{proof}
\noindent
In conclusion, the $n$
quasi-modes $(  \tilde \psi_i)_{i=1,\ldots,n}$ and the family of quasi-modes $(\tilde u=\tilde u_{\delta})_{\delta>0}$ satisfy all the conditions of Proposition~\ref{ESTIME3}. This concludes the proof of Theorem~\ref{th.gene_sigma}.

\subsubsection{Proof of Corollary~\ref{co.gene_sigma}}

Let us assume that the hypotheses \textbf{[H1]}-\textbf{[H2]}-\textbf{[H3]} hold and let us assume that
$f|_{\partial \Omega}$ has only two local minima $z_1$ and $z_2$ such that $f(z_1)<f(z_2)$. Let $\Sigma \subset \partial \Omega$ be a smooth open set such that
$\overline{\Sigma} \subset B_{z_1}$ and 
$$f^*:=\inf_{\Sigma} f.$$
In addition, let us assume that \eqref{hypo11_bis} and \eqref{hypo22_bis} hold and let us assume that $f^*=f(z_2)$. Then, the inequalities \eqref{hypo11_bis} and \eqref{hypo22_bis} are exactly \eqref{eq.hyp.gene.1} and~\eqref{eq.hyp.gene.2} (in the case $n=2$ with $j_0=1$ and $k_0=2$). Therefore,~\eqref{eq.gene_sig} holds. It remains to compute the prefactor in~\eqref{eq.gene_sig}. To this end, we need the constants $B^*$ and $p^*$  in~\eqref{eq:DL}. Let us assume that there is only one minimizer $z^*$ of  $f$ on $ \Sigma$. This implies  that  $z^*\in \partial \Sigma$ since $z_1$ is the only critical point of $f|_{\partial \Omega}$ in $B_{z_1}$. Furthermore, we assume  that $z^*$ is a non degenerate minimum  of $f_{|\partial \Sigma }$ with $\partial_{n(\partial \Sigma)} f(z^*)<0$ where $n(\partial \Sigma)$ is the unit outward normal to $\partial \Sigma\subset \partial \Omega$. Then, using Laplace's method, in the limit $h \to 0$:
$$\int_{\Sigma}  \partial_n  f  \, e^{-\frac{2}{h} f} d\sigma=-\frac{ \partial_n f(z^*)(\pi h)^{\frac d2}}{2\pi \partial_{n(\partial \Sigma)} f(z^*) \sqrt{ {\rm det \ Hess } f_{|\partial \Sigma }   (z^*) }  } e^{-\frac{2}{h} f^*} (1+O(h)),$$
with by convention, ${\rm det \ Hess } f_{|\partial \Sigma }   (z^*) =1$ if $d=2$.
This specifies the constants $B^*$ and $p^*$ appearing in~\eqref{eq:DL}. This ends  the proof of Corollary~\ref{co.gene_sigma}.

\section*{Acknowledgements}

The authors would like to thank Guy Barles, Bernard Helffer, Fran\c cois
Laudenbach, Francis Nier and Beno\^it Perthame for fruitful discussions on
preliminary versions of this work. This work is supported by the
European Research Council under the European Union's Seventh Framework
Programme (FP/2007-2013) / ERC Grant Agreement number 614492.

%{bobo}
%{page.LOFH}
 
\section*{Main notation used in this work}

%\textbf{Section~\ref{intro}} 
\begin{multicols}{2}
\begin{itemize}
\item[$\bullet$] $(\tau_\Omega,X_{\tau_\Omega})$, p.~\pageref{page.tau}
\vspace{-0.1cm}

\item[$\bullet$] $\P_x$, p.~\pageref{page.px}
\vspace{-0.1cm}

%\item[$\bullet$] $\E_x$, p.~\pageref{page.ex}
%\vspace{-0.1cm}

\item[$\bullet$] $\nu_h$, p.~\pageref{page.nuh} 

\vspace{-0.1cm}

\item[$\bullet$] $L_{f,h}^{(0)}$ and $L_{f,h}^{D,(0)}(\Omega)$, p.~\pageref{page.LOFH}

\vspace{-0.1cm}

\item[$\bullet$] $\lambda_h$, p.~\pageref{page.lambdah} (see also p.~\pageref{page.uhlambdah})

\vspace{-0.1cm}

\item[$\bullet$] $u_h$, p.~\pageref{page.uh} (see also p.~\pageref{page.uhlambdah})

\vspace{-0.1cm}

\item[$\bullet$] $\{k_{0,1},\ldots,k_{0,n}\}$, p.~\pageref{page.koi}

\vspace{-0.1cm}

\item[$\bullet$]  $g$, p.~\pageref{page.g}

\vspace{-0.1cm}

\item[$\bullet$]  $L (\gamma,I)$, p.~\pageref{page.L}

\vspace{-0.1cm}

\item[$\bullet$]  $d_a$, p.~\pageref{page.dagmon}

\vspace{-0.1cm}

\item[$\bullet$]  $x_0$, p.~\pageref{page.z1zn}

\vspace{-0.1cm}

\item[$\bullet$]  $\{z_1,\ldots,z_n\}$, p.~\pageref{page.z1zn}

\vspace{-0.1cm}

\item[$\bullet$]  Hypotheses \textbf{[H1]}, \textbf{[H2]}, and \textbf{[H3]}, p.~\pageref{page.z1zn}

\vspace{-0.1cm}

\item[$\bullet$]  $n_0$, p.~\pageref{page.n0}

\vspace{-0.1cm}

\item[$\bullet$]   $B_{z_i}$ and $B_{z_i}^c$, p.~\pageref{page.Bzi}

\vspace{-0.1cm}

\item[$\bullet$]  $\Sigma_i$, p.~\pageref{page.Sigmai}

\vspace{-0.1cm}

\item[$\bullet$] $f^*$,  p.~\pageref{page.fstar}

\vspace{-0.1cm}

\item[$\bullet$] $B^*$ and $p^*$,  p.~\pageref{page.fstar}
\vspace{-0.1cm}

\item[$\bullet$]  $z^*$,  p.~\pageref{page.zstar}

%
%\end{itemize}
% \end{multicols}
%
%\noindent
%\textbf{Section~\ref{sec:gram_schmidt}}
%
%\begin{multicols}{2}
%\begin{itemize}

\vspace{-0.1cm}

\item[$\bullet$] $\Lambda^pC^\infty(\overline \Omega)$, p.~\pageref{page.cinfty}

\vspace{-0.1cm}

\item[$\bullet$] $\Lambda^pC^\infty_T(\overline \Omega)$ and $\Lambda^pC^\infty_N(\overline \Omega)$, p.~\pageref{page.cinftyt}

\vspace{-0.1cm}

\item[$\bullet$] $\Lambda^pL^2_w( \Omega)$ and  $\Lambda^pH^q_w(  \Omega)$,   p.~\pageref{page.wsobolevq}

\vspace{-0.1cm}

\item[$\bullet$]    $\Lambda^pH^q_{w,T}( \Omega)$ and $\Lambda^pH^q_{w,N}(  \Omega)$,   p.~\pageref{page.wsobolevqt}

\vspace{-0.1cm}

\item[$\bullet$] $\Lambda^pL^2(  \Omega)$ and $\Lambda^pH^q(  \Omega)$,   p.~\pageref{page.psl2}

\vspace{-0.1cm}

\item[$\bullet$] $\Lambda^pH^q_{T}(  \Omega)$ and $\Lambda^pH^q_{N}(  \Omega)$,   p.~\pageref{page.psl2}

\vspace{-0.1cm}

\item[$\bullet$] $\Vert .\Vert_{H^q_w}$ and $\langle\, ,\, \rangle_{L^2_w}$,  p.~\pageref{page.psl2w}

\vspace{-0.1cm} 

\item[$\bullet$]$ \Vert .\Vert_{H^q}$ and $\langle\, ,\, \rangle_{L^2}$,  p.~\pageref{page.psl2}

\vspace{-0.1cm}

\item[$\bullet$] $\Delta^{(p)}_{f,h}$,  p.~\pageref{page.wlaplacien}

\vspace{-0.1cm}

\item[$\bullet$] $\Delta^{D,(p)}_{f,h}(\Omega)$ and  $\Delta^{N,(p)}_{f,h}(\Omega)$,  p.~\pageref{page.wlaplaciend}

\vspace{-0.1cm}

\item[$\bullet$] $L^{D,(p)}_{f,h}(\Omega)$,  p.~\pageref{page.generatord}

\vspace{-0.1cm}

\item[$\bullet$] $\pi_E(A)$,  p.~\pageref{page.piE}

\vspace{-0.1cm}

\item[$\bullet$]  For $p\in \{0,1\}$,  $\pi^{(p)}_h$, p.~\pageref{page.pihp}

\vspace{-0.1cm}

\item[$\bullet$] $\tilde u $,  p.~\pageref{page.pre} (construction    \pageref{page.tildeu})

\vspace{-0.1cm}
 
\item[$\bullet$] $p$, $m$,   $C_i$, and $B_i$,  p.~\pageref{page.pre} (explicit values  p.~\pageref{page.bim} and \pageref{page.ci})

\vspace{-0.1cm}

\item[$\bullet$]  $\tilde \psi_i$,  p.~\pageref{page.pre}

\vspace{-0.1cm}

\item[$\bullet$] $\tilde \phi_i$,  p.~\pageref{page.tildephi0}
 (construction  p.~\pageref{page.tildephii})

\vspace{-0.1cm}

\item[$\bullet$] $\kappa_{ji}$,  p.~\pageref{page.kappaji}

\vspace{-0.1cm}

\item[$\bullet$] $Z_j$ and  $\psi_j$ p.~\pageref{page.zjpsij}

\vspace{-0.1cm}

%\end{itemize}
% \end{multicols}
%
%
%%%%%%%%%%
%\noindent
%\textbf{Section~\ref{agmonproperty}}
%
% 
%\begin{itemize}

\item[$\bullet$]  $A(x,y)$, p.~\pageref{page.axy}

\vspace{-0.1cm}

\item[$\bullet$]  $d_a^{\pa \Omega}$, p.~\pageref{page.dapaomega}

\vspace{-0.1cm}

\item[$\bullet$]  $\{x_1,\ldots,x_m\}$, p.~\pageref{page.x1xm}

\vspace{-0.1cm}

\item[$\bullet$]  $V_{\pa \Omega}$, p.~\pageref{page.vpaomega}

\vspace{-0.1cm}

\item[$\bullet$]  $f_-$, p.~\pageref{page.fmoins} 

% 
%\end{itemize}
% 
% 
% 
% 
%%%%%%%%%%
%\noindent
%\textbf{Section~\ref{sec:quasi-modes}} 
%
%\begin{multicols}{2}
%\begin{itemize}

\vspace{-0.1cm}

\item[$\bullet$] $\dot \Omega$, p.~\pageref{page.omegapoint} 

\vspace{-0.1cm}

\item[$\bullet$] $\Lambda^pH_{d}(\dot \Omega)$ and $\Lambda^pH_{d^*}(\dot \Omega)$, p.~\pageref{page.omegapoint} 

\vspace{-0.1cm}

\item[$\bullet$]  $\Gamma_T$ and $\Gamma_N$, p.~\pageref{page.gammatn} (see also  p.~\pageref{page.gammatbis})

\vspace{-0.1cm}

\item[$\bullet$]  $\Lambda^pH_{d,\Gamma}(\dot \Omega)$ and $\Lambda^pH_{d^*,\Gamma}(\dot \Omega)$, p.~\pageref{page.Hdgamma} 

\vspace{-0.1cm}

\item[$\bullet$] $d^{(p)}_T(\dot \Omega)$ and $\delta^{(p)}_N(\dot \Omega)$, p.~\pageref{page.dtp} 

\vspace{-0.1cm}

\item[$\bullet$] $\Delta^{M,(p)}_{f,h}(\dot \Omega)$, p.~\pageref{page.deltam} 

\vspace{-0.1cm}

\item[$\bullet$] $\mathcal Q^{M,(p)}_{f,h}(\dot \Omega)$, p.~\pageref{page.qm} 

\vspace{-0.1cm}

\item[$\bullet$] $x_d$, p.~\pageref{page.xd} 

\vspace{-0.1cm}

\item[$\bullet$]  $\Psi_i$, $f_{+,i}$, and $f_{-,i}$, p.~\pageref{page.fplus} 

\vspace{-0.1cm}

\item[$\bullet$]   $\Sigma_a$ ($a>0$), p.~\pageref{page.sigmaa}

\vspace{-0.1cm}

\item[$\bullet$]  $\Gamma_{1,i}$, p.~\pageref{page.gamma1i}

\vspace{-0.1cm}

\item[$\bullet$]  $V_{\Gamma_{1,i}}$, p.~\pageref{page.vgamma1i}

\vspace{-0.1cm}

\item[$\bullet$]  $\dot{ \Omega_i}$, p.~\pageref{page.omegaipoint}

\vspace{-0.1cm}

\item[$\bullet$] $\Omega_0$, $\Gamma_{0}$, p.~\pageref{page.gamma0}

\vspace{-0.1cm}

\item[$\bullet$]  $\Gamma_{2,i}$, p.~\pageref{page.gamma2i}

\vspace{-0.1cm}

\item[$\bullet$]  $\mathcal S_{M,i}$, p.~\pageref{page.SMi}

\vspace{-0.1cm}

\item[$\bullet$] When the index $i\in \{1,\ldots,n\}$ is dropped (p.~\pageref{page.nota0}): $z=z_i$,  $ \Gamma_{1}=\Gamma_{1,i}$, $\Gamma_{2}=\Gamma_{2,i}$, $\dot
\Omega=\dot \Omega_i$, $V_{\Gamma_{1}}=V_{\Gamma_{1,i}}$, $ \Psi=\Psi_i$, $f_{+}=f_{+,i} \ {\rm and} \ f_{-} =f_{-,i}$.

%
%
%\item[$\bullet$]  Constructions of  $\Gamma_T$ and $\Gamma_N$, p.~\pageref{page.gammatbis}

\vspace{-0.1cm}

%{page.sigmai1}

%\item[$\bullet$] For $i\in \{1,\ldots,n\}$, $\Sigma_i$, p.~\pageref{page.sigmai1}
%
%\vspace{-0.1cm}

%\item[$\bullet$] $\tilde u$, p.~\pageref{page.tildeu}

%\item[$\bullet$] Construction of  $\tilde \phi_i$ (for $i\in \{1,\ldots,n\}$), p.~\pageref{page.tildephii}

\vspace{-0.1cm}

\item[$\bullet$]  $u_{wkb}^{(0)}$, p.~\pageref{page.uowkb}

\vspace{-0.1cm}

\item[$\bullet$]  $u_{wkb}^{(1)}$, p.~\pageref{page.u1wkb}

\vspace{-0.1cm}

\item[$\bullet$]  $u_{z,wkb}^{(1)}$, p.~\pageref{page.u1zwkb}

\vspace{-0.1cm}

\item[$\bullet$]  $C_{z,wkb}$, p.~\pageref{page.czwkb}

\vspace{-0.1cm}

\item[$\bullet$] $V_{\Gamma_{st}}$ and $V_{\Gamma_{st}'}$, p.~\pageref{page.vgammast}

\vspace{-0.1cm}

\item[$\bullet$] $\kappa$, p.~\pageref{page.kappa}

\vspace{-0.1cm}

\item[$\bullet$]  $c_{z}(h)$, p.~\pageref{page.czh}

\vspace{-0.1cm}

\item[$\bullet$]  $\Theta_i$, p.~\pageref{page.thetai}

\vspace{-0.1cm}

\item[$\bullet$] $\tilde \phi_{z_i,wkb}$, p.~\pageref{page.tildephiwkb}

%\vspace{-0.1cm}
%
%
%
%\item[$\bullet$] Explicit values of $m$ and $B_i$ ($i\in \{1,\ldots,n\}$), p.~\pageref{page.bim}
%
%
%\vspace{-0.1cm}
%
%
%
%\item[$\bullet$] Explicit values of $p$ and $C_i$ ($i\in \{1,\ldots,n\}$), p.~\pageref{page.ci}
%
%\vspace{-0.1cm}

%
%\end{itemize}
% \end{multicols}
%
%\noindent
%\textbf{Section~\ref{sec:proofs}}
%
%\begin{multicols}{2}
%\begin{itemize}

%{sec:proofs}

\item[$\bullet$]  $K_\alpha$, p.~\pageref{page.kalpha}

\vspace{-0.1cm}

\item[$\bullet$]  $\tilde u_h$, p.~\pageref{page.tildeuh}

\vspace{-0.1cm}

\item[$\bullet$]  $Z_h(O)$, p.~\pageref{page.zho}

\vspace{-0.1cm}

\item[$\bullet$]   $w_h$, p.~\pageref{page.whh}

\vspace{-0.1cm}

\item[$\bullet$] $\Sigma$  and   $j_0$, p.~\pageref{page.jo}

\vspace{-0.1cm}

\item[$\bullet$] $k_0$, p.~\pageref{page.ko}
\vspace{-0.1cm}

\item[$\bullet$] $C^*$  and   $q^*$, p.~\pageref{page.cstarqstar}

\vspace{-0.1cm}

\item[$\bullet$]  $\chi_{j_0}$ and  $\tilde \phi_{j_0}$, p.~\pageref{page.phijo}

%\vspace{-0.1cm}

%
%
%\item[$\bullet$] $p^*$, p.~\pageref{page.pstarr}
%
%\vspace{-0.1cm}
%
%
%\item[$\bullet$] 
%
%\vspace{-0.1cm}
%
%
%
%\item[$\bullet$] 
%
%\vspace{-0.1cm}

\end{itemize}
  \end{multicols}

\bibliography{biblio_sharp} 
\bibliographystyle{plain}

\end{document}